\documentclass[11pt,reqno]{amsart}
\usepackage{amsfonts}
\usepackage{mathtools}
\usepackage{amsthm}
\usepackage{amsmath}
\usepackage{amssymb}
\usepackage{comment}
\usepackage{amscd}
\usepackage{cite}
\usepackage[ngerman,english]{babel}
\usepackage[mathscr]{eucal}
\usepackage{indentfirst}
\usepackage{graphicx}
\usepackage{tikz}
\usepackage{graphics}
\usepackage{pict2e}
\usepackage{epic}
\usepackage[hidelinks]{hyperref}
\numberwithin{equation}{section}
\usepackage[margin= 2.5cm, top = 2cm, bottom = 2cm]{geometry}
\usepackage{epstopdf} 
\usepackage{bbm}
\usepackage{xcolor}
\usepackage{comment}

\let\OLDthebibliography\thebibliography
\renewcommand\thebibliography[1]{
	\OLDthebibliography{#1}
	\setlength{\parskip}{0pt}
	\setlength{\itemsep}{3pt}
}
\newcommand{\tri}[1]{\triangle_{#1}}

\newcommand{\boxalign}[2][.87\textwidth]{
	\par\noindent\tikzstyle{mybox} = [draw=black,inner sep=6pt]
	\begin{center}\begin{tikzpicture}
			\node [mybox] (box){%
				\begin{minipage}{#1}{\vspace{-2mm}#2}\end{minipage}
			};
	\end{tikzpicture}\end{center}
}

\newcommand{\tinysection}[2][]{\;\\[#1pt]
	\emph{\underline{#2}}. \;\\[4pt]
}

\numberwithin{equation}{section}

\usepackage{array}

\newcolumntype{M}[1]{>{\centering\arraybackslash}m{#1}}
\newcolumntype{N}{@{}m{0pt}@{}}

\newcommand{\cell}[2]
{\vtop{\hbox{\strut #1 \vspace{3pt}}\hbox{\strut #2}} 
}

\numberwithin{equation}{section}

\theoremstyle{plain}
\newtheorem{Th}{Theorem}[section]
\newtheorem{Lemma}[Th]{Lemma}
\newtheorem{Cor}[Th]{Corollary}

\DeclareMathOperator{\R}{\mathbb{R}}
\DeclareMathOperator{\Z}{\mathbb{Z}}
\DeclareMathOperator{\C}{\mathbb{C}}
\DeclareMathOperator{\N}{\mathbb{N}}

\DeclareMathOperator{\supp}{\text{supp}}

\DeclareMathOperator{\cN}{\mathcal{N}}

\DeclareMathOperator{\cO}{\mathcal{O}}
\DeclareMathOperator{\tQ}{\tilde{\mathcal{Q}}}

\newcommand{\norm}[2]{\left\lVert #1 \right\rVert_{#2}}

\newcommand{\f}[2]{\frac{#1}{#2}}

\newcommand{\Step}[1]{$\mathbf{{Step}~#1}~~$}
\newcommand{\para}{~\\~~\\}
\newcommand{\Q}[1]{\tQ^{#1}}

\numberwithin{equation}{section}

\newtheorem{thm}{Theorem}[section]
\newtheorem{cor}[thm]{Corollary}
\newtheorem{lem}[thm]{Lemma}
\newtheorem{prop}[thm]{Proposition}

\theoremstyle{remark}
\newtheorem{rem}{Remark}[section]

\theoremstyle{definition}
\newtheorem{Def}[Th]{Definition}
\newtheorem{Rem}[Th]{Remark}

\makeatletter
\newcommand*{\rom}[1]{\expandafter\@slowromancap\romannumeral #1@}
\makeatother

\begin{document}
	


	\title[Approximate solutions for the Zakharov system]{%Blow up dynamics for 
		Finite time blow up for the energy critical\\ Zakharov system I: approximate solutions}
	
	\author[J. Krieger]{Joachim Krieger}
	%	\address{EPFL SB MATH PDE,
		%	Bâtiment MA,
		%	Station 8,
		%	CH-1015 Lausanne}
	%\email{joachim.krieger@epfl.ch}
	\author[T. Schmid]{Tobias Schmid}
	\address{EPFL SB MATH PDE,
		Bâtiment MA,
		Station 8,
		CH-1015 Lausanne}
	\email{joachim.krieger@epfl.ch}
	\email{tobias.schmid@epfl.ch}

	%	\address{EPFL \\ SB MA}

	\begin{abstract} We construct approximate solutions $ (\psi_*, n_*)$ of the critical 4D Zakharov system which collapse in finite time to a singular renormalization of the solitary bulk solutions $ (\lambda e^{i \theta}W, \lambda^2 W^2)$ . To be precise for $ N \in \Z_+, \;N \gg1 $  we obtain a magnetic envelope/ion density pair % just "we obtain functions of the form" instead?
		of the form
		\begin{align*}
			\psi_*(t, x)= e^{i\alpha(t)}\lambda(t) W(\lambda(t)x) + \eta(t, x),\;\;n_*(t,x) = \lambda^2(t) W^2(\lambda(t) x) + \chi(t,x),
		\end{align*}
		where $ W(x) = (1 + \frac{|x|^2}{8})^{-1}$, \;$\alpha(t) = \alpha_0 \log(t),~\lambda(t)= t^{-\frac{1}{2}-\nu}$ with large $\nu > 1 $  and further
		\begin{align*}
			& i \partial_t \psi_* + \triangle \psi_* + n_* \psi_* = \mathcal{O}(t^N),\;\; \square n_* - \triangle  (|\psi_*|^2) = \mathcal{O}(t^N),\\
			& \eta(t) \to \eta_0,\;\; \chi(t) \to \chi_0,
		\end{align*} 
		as $ t \to 0^+$ in a suitable sense. The method of construction  is inspired by matched asymptotic regions and % reminiscent of the 
		approximation procedures %introduced by 
		%reminiscent of the pivotal 
		%uses techniques reminiscent of the approximation procedures 
		in the context of blow up solutions introduced by the first author jointly with W. Schlag and D. Tataru, as well as the subsequently developed  methods in the Schr\"odinger context by G. Perelman et al. % by G. Perelman et al.%, Ortoleva-Perleman and Bahouri-Marachli-Perelman. 
		
	\end{abstract}

	\maketitle
	\tableofcontents
	\section{Introduction}
	In \cite{zakharov} Zakharov derived an effective model for so called Langmuir waves, i.e. rapid oscillations in the electric field of a (weakly) magnetized plasma. A scalar field version of this model evolves the ion density $ n : \R^{d+1} \to \R$ and a complex envelope for the electric field $ \psi : \R^{d+1}\to \C $ as unknown functions. We consider the following classical Zakharov system
	\begin{equation}\label{Zakharov}
		\begin{cases}
			\;\; i\partial_t \psi+ \Delta \psi= -n\psi &\text{in}~  (0,T) \times \R^d \\[3pt]
			\;\; \square n= \Delta(|\psi|^2) &\text{in}~ (0,T) \times \R^d ,
		\end{cases}
	\end{equation}
	with d'Alembertian  $ \square  = -\partial_t^2+ \Delta $ in dimension $ d = 4$. The Zakharov system \eqref{Zakharov} is Hamiltonian, where the energy \\
	\begin{align*}
		E_Z(\psi, n, \partial_t n)_{|_{t}} = \f12 \int_{\R^d} |\nabla \psi(t)|^2 + \f12 ||\nabla |^{-1} \partial_tn(t)|^2 + \f12 |n(t)|^2  + n(t) |\psi(t)|^2~dx,
	\end{align*}
	and mass $ \norm{\psi}{L^2}^2 $ are conserved along solutions $(\psi(t), n(t))$. Diagonalizing the second line of $ \eqref{Zakharov}$ into a half-wave for $ V(t) = ( 1 - i |\nabla|^{-1}\partial_t)n(t) $, we may write 
	$$ E_Z(\psi,V)  = E_S(\psi)+ \f14 \int_{\R^d} | V + |\psi|^2|^2~dx,$$
	with energy
	$$ E_S(\psi) = \f12\int_{\R^d} |\nabla \psi|^2 - \f12 |\psi|^4~dx$$
	depending only on the magnetic envelope. The focusing  cubic nonlinear Schr\"odinger equation (NLS)
	\[
	i \partial_t u  + \Delta u  = - u |u|^2,~~\text{in}~  (0,T) \times \R^d, 
	\]
	has the critical (i.e. scaling invariant) conserved energy $ E_S$ in dimension $ d = 4 $ and arises from \eqref{Zakharov} in the so called subsonic limit. More precisely the scaling $n_c(t,x) = n(c^{-1}t, x)$ introduces the \emph{ion sound speed} $ c > 0$ in \eqref{Zakharov} and letting $ c \to \infty$ we formally observe $ \Delta n = \Delta (|\psi|^2)$ whence $n = |\psi|^2 $ provided $ n,  \psi $ decay as $ |x| \to \infty$. For a systematic consideration of taking this limit, we refer to  \cite{added},\cite{Kenig-Pon-veg-limit},\cite{Mas-Nakan}, \cite{Oz-Tsu-2} and \cite{Schoch-Wein}.\\[4pt]
	In the past and recently, the Cauchy problem for \eqref{Zakharov} has been studied intensely (in all dimensions $d \geq 1$) with \cite{added}, \cite{Kenig-Pon-veg-limit}, \cite{Oz-Tsu}  and \cite{Schoch-Wein} being among the first existence and wellposedness results. The dynamics of solutions to \eqref{Zakharov} is influenced by the connection to the cubic  NLS, for example \eqref{Zakharov} `inherits'  the bound states of the cubis NLS, i.e. if $Q_{\omega } : \R^d \to \R$ is a solution of
	\[
	\Delta Q_{\omega }  - \omega Q_{\omega } = - Q_{\omega }^3,~~ \text{on}~\R^d,
	\]
	then $ (\psi, n ) = ( e^{it \omega} Q_{\omega }, Q_{\omega }^2) $ is a solitary solution of \eqref{Zakharov}. For the \emph{critical} dimension $ d = 4$, a solution exists only if $ \omega = 0$ and  the \emph{ground state} of $E_S$
	\begin{equation}
		W(x):= \left(1+ \frac{|x|^2}{8}\right)^{-1},~ x \in \R^4
	\end{equation}
	is the unique positive solution of 
	$$\Delta W(x) + W^3(x) = 0,~~ x \in \R^4.$$ 
	%This is a coupled Schr\"odinger--Wave equation with $\psi$ (resp. $n$) being the Schrödinger part (resp. Wave part): 
	%	when regarding $n\sim |\psi|^2$ we recover NLS.\\
	This implies the following $2$-parameter family of solitary ground solutions of \eqref{Zakharov} (by scaling and phase rotation)
	%we further notice that $(\lambda W(\lambda x), \lambda^2 W^2(\lambda x))$ is also a static solution, which further yields other static solutions as 
	\begin{equation}\label{gs-family}
		(e^{i \theta}\lambda W(\lambda x),  \lambda^2 W^2(\lambda x)), \;\; \forall \theta \in \R,~ \lambda > 0, 
	\end{equation}
	and \emph{we note general solutions of the system \eqref{Zakharov} on the contrary have no scaling invariance}. Similar to the focusing cubic NLS, the existence of \eqref{gs-family} naturally poses an obstruction to global scattering states. The influence of the wave interaction in the second line of \eqref{Zakharov}, on the other hand, is seen in the local wellposedness, which holds for initial data $ \psi_0 \in H^s$  with regularity \emph{below the critical index} $ s = \f{d}{2} -1$ for the cubic NLS, see e.g. \cite{Bej-Herr1}\\[3pt]
	We denote the energy space of \eqref{Zakharov} in the following by
	\[
	\mathcal{E} : = \dot{H}^1(\R^d) \times L^2(\R^d) \times \dot{H}^{-1}(\R^d),
	\]
	with  initial data $ (u_0, n_0, n_1) \in \mathcal{E}$. %Further we use the same notation 
	%$\mathcal{E} =  \dot{H}^1 \times L^2$
	%	corresponding to the  diagonalized system \eqref{reduced-Zakharov}. 
	Global  wellposedness  of \eqref{Zakharov} in the energy space  for  $d = 1$  was proved by Ginibre-Tsutsumi-Velo \cite{Ginibre-Tsutsumi-Velo} and local wellposedness for $d = 2,3$ by Bourgain-Colliander \cite{Borgain-Coll}. Further local and global results in $ d \leq 3$ dimensions (with low regularity data) have been shown for instance in \cite{Bej-Herr1}, \cite{Bej-Herr-Holmer-Tataru}, \cite{Coll-Holmer-Tzi}, \cite{Fang-Pech}, \cite{Guo-Nakanishi1} and \cite{Guo-Nakanishi-Wang-3D}.\\[3pt] Concerning blow up solutions, we refer to the finite time blow up in $ d = 2$ %(and infinite time blow up in $d =3$)
	proved by Glangetas \& Merle in \cite{GM1} and \cite{GM2} (see also \cite{Merle}). For further results studying the asymptotic behavior, in particular global scattering solutions, we refer to the introduction in \cite{Candy-Herr-Nakanishi-Global} and references therein. %For additional references concerning $ d \leq 3$, we refer for instance to [Bej-Herr-Nakanishi] and henceforth restrict to dimension $ d = 4$.
	We now and henceforth restrict to dimension $ d = 4$.
	\subsection{The energy-critical Zakharov system}
	In order to  establish the wellposedness theory of \eqref{Zakharov}, it is common (see e.g. \cite[Section 2]{Candy-Herr-Nakanishi-sharp}) to use the half-wave reduction  $ V = (1- i |\nabla |^{-1} \partial_t)n$ in the second line, i.e. we may rewrite \eqref{Zakharov} as
	\begin{equation}\label{reduced-Zakharov}
		\begin{cases}
			\;\; i\partial_t \psi+ \Delta \psi= -Re(V) \psi &\text{in}~  (0,T) \times \R^4 \\[3pt]
			\;\; i \partial_t V  + |\nabla | V = |\nabla |(|\psi|^2) &\text{in}~ (0,T) \times \R^4 ,
		\end{cases}
	\end{equation} 
	which is equivalent to \eqref{Zakharov}. The Cauchy problem for \eqref{reduced-Zakharov} is known to be locally wellposed with data $( \psi, V) \in H^s(\R^4) \times H^l(\R^4) $ if $ l \leq s \leq l + 1$ with $ l > 0, 2s > l + 1$ due to Ginibre-Tsutsumi-Velo  \cite{Ginibre-Tsutsumi-Velo}. In case $(s,l)$ are such that
	\[
	l \geq 0,~ s < 4 l +1,~\max\{\frac{l +1}{2}, l -1  \} \leq s \leq \min\{ l + 2, 2 l + \frac{11}{8}  \}, 
	\]
	the following was further established by Bejenaru-Herr-Nakanishi \cite{Bej-Herr-Nakanishi}.
	\begin{itemize}
		\item[$\circ$] The Cauchy problem for \eqref{reduced-Zakharov} is globally wellposed and scatters as $ t \to \pm \infty$ in $ H^s \times H^l$ if the initial data $ (\psi_0, V_0) $ is \emph{small enough} in $H^{\f12}\times L^2$ and $ (s,l) \neq (2,3)$, % or $ (s,l) = (1,0)$,
		\item[$\circ$] The Cauchy problem for \eqref{reduced-Zakharov} is locally wellposed in $ H^s \times H^l$, excluding the endpoint  $ (s,l) =  (2,3)$,
		\item[$\circ$] The Cauchy problem for \eqref{reduced-Zakharov} is ill-posed  in $ H^2 \times H^3$.
	\end{itemize}
	The first result also holds for the energy endpoint $(s,l) =(1,0)$, which however is obtained by a compactness argument since the perturbative estimate breaks down.\\[4pt]
	Large data  in the energy space $ \tilde{\mathcal{E}} : = \dot{H}^1(\R^4) \times L^2(\R^4) $ was subsequently considered in \cite{Guo-Nakanishi2(Threshold)} by Guo-Nakanishi, where a sub-threshold Kenig-Merle dichotomy was derived for the radial system:\\[5pt] In particular if $ (\psi_0, V_0) \in H^1_{rad} \times L^2_{rad}$ with $ E_Z(\psi_0, V_0) < E_Z(W, W^2)$, then a unique local solution $ t \in I \mapsto (\psi(t), V(t))$ of  \eqref{reduced-Zakharov} exists and  satisfies either
	\begin{itemize} \setlength\itemsep{5pt}
		\item[$\circ$] $\| V(t) \|_{L^2} < \| W^2(t)\|_{L^2}$ ~~$\forall t \in I = \R$  with scattering as $ t \to \pm \infty$, or
		\item[$\circ$] $\| V(t) \|_{L^2} > \| W^2(t)\|_{L^2}$ for all $ t\in I$ and if $ \sup(I) = \infty $, then 
		$$\limsup_{t \to \infty} \|(u(t), V(t)) \|_{\tilde{\mathcal{E}}} = \infty.$$
	\end{itemize}
	Improving upon the result in \cite{Bej-Herr-Nakanishi}, the sharp local wellposedness was then obtained by Candy-Herr-Nakanishi \cite{Candy-Herr-Nakanishi-sharp}. That is, the Cauchy problem for \eqref{reduced-Zakharov} is locally wellposed in  $H^s \times H^l$ with analytic flow map if and only if $ (s,l) $ satisfies
	\[
	(\star)~~~~ l \geq 0, \max\{\frac{l +1}{2}, l -1  \} \leq s \leq  l + 2,~ (s,l) \neq (2,0), (2,3),
	\]
	where, in particular,  $(s,l) = (1,0)$ is included in the large data local result. Similarly, the range of $(s,l) $ for the above small data global wellposedness and scattering was improved as well in \cite{Candy-Herr-Nakanishi-sharp} and both their results hold in all dimensions $ d \geq 4$.\\[3pt]
	The radiality assumption in the global wellposedness of \cite{Guo-Nakanishi2(Threshold)} was removed in \cite{Candy-Herr-Nakanishi-Global} by Candy-Herr-Nakanishi, proving the following. If $  (\psi_0, V_0)\in H^s \times H^l$ with $ (s,l) $ satisfying $(\star),~ s \geq 1$ and 
	$$E_Z(\psi_0, V_0) < E_Z(W, W^2),~ \| V_0\| \leq \| W^2\|_{L^2},$$
	then there is a unique global solution continuous in time and  taking values in $ H^s \times H^l $.\\[4pt]
	\textbf{Our result}. We stress that concerning the blow up v.s. scattering  dichotomy below the energy threshold in \cite{Guo-Nakanishi2(Threshold)}, the `blow up' scenario $ \| V \|_{L^2 } > \|W^2 \|_{L^2}$  is only \emph{weak}, i.e. the following possibilities exist.
	\begin{itemize}\setlength\itemsep{4pt}
		\item[(1)] Type I blow up : $ T : =\sup I < \infty $ and $\limsup_{t \to T} \| (\psi,V) \|_{\tilde{\mathcal{E}}}  = \infty$,
		\item[(2)] Type II blow up : $ T : = \sup I < \infty$ and $\limsup_{t \to T} \| (\psi,V) \|_{\tilde{\mathcal{E}}}  < \infty$,
		\item[(3)] Grow up : $ \sup I = \infty$ and $\limsup_{t \to \infty} \| (\psi,V) \|_{\tilde{\mathcal{E}}}  = \infty$.\\
	\end{itemize}
	The optimality of $ \| V_0\|_{L^2} \leq \| W^2\|_{L^2}$ for a global (sub-threshold) result was discussed in  \cite{Guo-Nakanishi2(Threshold)}, \cite{Candy-Herr-Nakanishi-Global} and \cite{Bej-Herr-Nakanishi}.\\[3pt] \emph{Nevertheless, we emphasise} that it is not known until now if the Cauchy problem for \eqref{Zakharov} has a `true' finite time blow up solution in dimension $ d = 4$.  Further, to the authors knowledge, the only blow up solutions known for \eqref{Zakharov} in any dimension are the results of \cite{GM1,GM2}. \emph{This is a first of a series of two articles (c.f. \cite{KSch}) aimed at constructing such finite time blow up solutions} with energy slightly \emph{above} the threshold energy $E_Z(W, W^2)$.\\[5pt] %In particular, we 
	%\textcolor{red}{cite the $d = 4$ threshold and small data theory of Herr, Nakanishi et al.}\\
	%This is the first of a two paper series of articles in which we construct (slow) blow up solutions (in $\psi$) for the Zakharov system \eqref{Zakharov} in the 'critical' dimension $ d = 4 $.
	%We apply a renonormalization procedure for the solitary ground states  \eqref{gs-family}  and find a (sufficiently fast decaying) perturbations of  \emph{collapsing} profiles, i.e. our solutions are of the form
	\emph{In the present article}, we develop an approximation procedure for solving \eqref{Zakharov} near a  \emph{collapsing} rescaled  bulk term up to a fast decaying error.  More precisely, we will iteratively correct a renormalization of the ground state  \eqref{gs-family}, i.e. such that we obtain approximate solutions of the form
	\begin{align}\label{qssol}
		&\psi(t, x)= e^{i\alpha(t)}\lambda(t) W(\lambda(t)x) + \eta_1(t, x) + \eta_2(t, x) + \eta_3(t, x) + \dots ,\\[4pt]
		&n(t,x) = \lambda^2(t) W^2(\lambda(t) x) + \chi_1(t,x) + \chi_2(t,x) + \chi_3(t,x) + \dots,
	\end{align}
	where $\alpha(t) = \alpha_0 \log(t)$ and $\lambda(t)= t^{-\frac{1}{2}-\nu}$ for  $ \alpha_0, \nu \in \R$. In fact here we restrict to taking  $ \nu \gg1 $ irrational and `sufficiently' large.\\[4pt]
	The strategy we pursue is  a combination of techniques developed by the first author jointly with W. Schlag and D.Tataru in \cite{KST1}, \cite{KST-slow}, \cite{KST2-YM} (see also \cite{K-S-full}, \cite{Don-Kr}) for critical nonlinear wave equations and the subsequent adaption of this method for nonlinear Schr\"odinger equations by G. Perelman \cite{Perelman} and Ortoleva-Perelman \cite{OP} (see also \cite{Bahouri-Marachli-Perelman}). The latter approximation method, based on parabolic matched asymptotic regions, has recently been used by the second author \cite{schmid} in order to construct finite time blow up for the focusing 3D quintic NLS. \footnote{This work already contains some relevant details on the approximation of blow up solutions near the ground state of the cubic NLS in $d = 4$ (see \cite[Section 1-2]{schmid})}\\[5pt]
	In the second part \cite{KSch} of the series, we perturbatively complete the approximations\footnote{Under certain spectral assumption to be verified in a forthcoming work \cite{KSch2}} to singular solutions of \eqref{Zakharov} of the form
	\begin{align}\label{qssol2}
		&\psi(t, x)= e^{i\alpha(t)}\lambda(t) W(\lambda(t)x) + \eta(t, x),\\
		&n(t,x) = \lambda^2 W^2(\lambda x) + \chi(t,x),
	\end{align}
	where $ \eta(t) \to \eta_0,\; \chi(t) \to \chi_0$ as $ t \to 0^+$ in a suitable sense. Finally we note, in comparison, blow up constructions in the critical Schr\"odinger context are well-studied, see e.g. \cite{Merle-k-bl}, \cite{Merle-Rapha1}, \cite{Merle-Rapha2}, \cite{Rod-Raph}.\\[5pt]
	We now state the main result of this article, for which we will give the proof in the last Section \ref{sec:finalsection}.
	\begin{Th}\label{main-theorem} Let $ \alpha_0 \in \R$ and  $\nu > 1 $ be  sufficiently large. For any  given $N \in \Z_+$, there exists a radial approximate finite energy and finite time blow up solution $(\psi_*(t), n_*(t))$ for the system \eqref{Zakharov} defined on $(0, \infty)\times \mathbb{R}^4$, which becomes singular at time $t = 0$.\\[4pt]
		The functions $(\psi_*(t,\cdot), n_*(t,\cdot))$ have sharp sharp regularity in $H^{2\nu -c-}({\mathbb{R}^4}) \times H^{2\nu-1-c-}({\mathbb{R}^4}) $ for all $ t > 0$ and a suitable  constant  $ 0 \leq c\leq 1 $ independent of $\nu,\alpha_0, t$.  Further, they admit representations, using $R = \lambda(t)\cdot r$,  $ r = |x|,\; \lambda(t) = t^{-\frac12-\nu},\;\alpha(t) = \alpha_0 \log(t)$,
		\begin{align*}
			&\psi_*(t, R) = e^{i\alpha(t)}\lambda(t) \big(W(R) + g_*^N(t, R)\big) = e^{i\alpha(t)} \lambda(t) W(R) + \eta(t,R),\\[4pt]
			&n_*(t, R) = \lambda^2(t)\big(W^2(R) + h_*^N(t, R)\big) =  \lambda^2(t) W^2(R) + \chi(t, R).
		\end{align*}
		where the functions $g_*^N, h_*^N$ satisfy the bounds 
		\begin{align*}
			\big\| g_*^N(t, \cdot)\big\|_{L^{\infty}_R}\lesssim_{\nu, |\alpha_0|} \lambda^{-1}(t),\,\; \big\|h_*^N(t, \cdot)\big\|_{L^{\infty}_R} \lesssim_{\nu, |\alpha_0|} \lambda^{-2}(t).
		\end{align*}
		In particular $(\psi_*, n_*)$ blows up at the origin $(t, r) = (0, 0)$.
		The solution is approximate of order $N$ in the sense that defining the error functions
		\begin{align*}
			&e^N_{\psi_{*}}= i\partial_t \psi_*+ \Delta \psi_* + n_*\cdot \psi_*,\\[3pt]
			&e^N_{n_*} = \Box n_{*} - \Delta(|\psi_{*}|^2),  
		\end{align*}
		we have the estimates 
		\begin{align*}
			\big\|e^N_{\psi_*}(t, \cdot)\big\|_{L^2_{R^3dR}} + \big\|e^N_{n_*}(t, \cdot)\big\|_{L^2_{R^3dR}}\leq C_N\cdot t^N,\,t>0.
		\end{align*}
		%We further have the representations (using the variable $R = \lambda(t)\cdot r$ and $\lambda(t) = t^{-\frac12-\nu}$)
		%	\begin{align*}
			%	&\psi_*(t, R) = e^{i\alpha(t)}\cdot\lambda(t)\cdot\big(W(R) + g_*^N(t, R)\big) = e^{i\alpha(t)}\cdot\lambda(t) W(R) + \eta(t,R),\\[4pt]
			%	&n_*(t, R) = \lambda^2(t)\big(W^2(R) + h_*^N(t, R)\big) =  \lambda^2(t) \cdot W^2(R) + \chi(t, R),
			%	\end{align*}
		%	where the functions $g_*^N, h_*^N$ satisfy the bounds 
		%	\begin{align*}
			%	\big\| g_*^N(t, \cdot)\big\|_{L^{\infty}_R}\lesssim_{\nu, |\alpha_0|} \lambda^{-1}(t),\,\; \big\|h_*^N(t, \cdot)\big\|_{L^{\infty}_R} \lesssim_{\nu, |\alpha_0|} \lambda^{-2}(t).
			%	\end{align*}
		Finally we  observe non trivial radiation in the sense
		\begin{align*}
			\eta(t) = e^{i\alpha(t)} \lambda(t)\cdot g_*^N(t) = \eta_0 + o(1),\;\; \chi(t) = \lambda^2(t) \cdot h_*^N(t) = \chi_0 + o(1),\;\; t \to 0^+
		\end{align*}
		in $ \dot{H}^1\cap  \dot{H}^2(\R^4)$, respectively in $H^2(\R^4)$, where the limits are explicitly given.
	\end{Th} 
	\begin{Rem}
		In this article we restrict to $\nu >1 $ irrational and we have at least $ \nu > 4$ for technical reasons. From the Lemma \ref{lem:estimates-inner}, Corollary \ref{cor:estimates-in-y} and Corollary \ref{cor:estimates-in-(t,r)} we may derive more estimates than stated in Theorem \ref{main-theorem}. Further we clearly mean $ \psi_*  - e^{i \alpha }\lambda W \in H^{2\nu -c-}({\mathbb{R}^4})$ since $ W \notin L^2(\R^4)$. In fact the sharp regularity is obtained %, since  $ \nu \gg1 $, 
		in light of the singular expansion of the inner region and the iteration procedure in Section \ref{sec:inner}.%Further we clearly mean $ \psi_*  - e^{i \alpha }\lambda W = \eta \in H^{\nu +1-c-}({\mathbb{R}^4})$ since $ W \notin L^2(\R^4)$. In fact, we carefully observe $  c =0$ since  $ \nu \gg1 $ in light of the singular expansion in the inner region and the iteration procedure in Section \ref{sec:inner}.
	\end{Rem}

	\subsection{Outline of the paper and notation}: The general approximation scheme, which will be presented here, is inspired by a matched asymptotic heuristic analysis. In particular, the procedure is reminiscent of Perelman's \cite{Perelman} and Ortoleva-Perelman's \cite{OP} approximation scheme (see also \cite{Bahouri-Marachli-Perelman}),  combined with a wave parametrix introduced in the work \cite{KST1}, \cite{KST2-YM}, \cite{KST-slow} on critical wave equations. We give a short outline of how the present manuscript is organized.\\[4pt]
	\emph{Outline}.  In particular, we start by letting 
	$$  0 < \epsilon_1 \ll1,\;  0 <  \epsilon_2 < 1,~ r = |x| > 0,$$
	where $ \epsilon_1, \epsilon_2 $ are to be determined, and subdivide $ \R^4 \times (0, T)$  into three distinct (radial) regions, starting with  the \emph{inner parabolic region} in Section \ref{sec:inner} which restricts to a domain
	$$ \mathcal{I} : = \big \{ (t,r)~|~0 \leq r \lesssim t^{  \f{1}{2} +\epsilon_1} \big \}.$$
	In this case it has been proven useful, c.f. \cite{Perelman} \cite{OP},  \cite{schmid},  to change variables into $ R = \lambda(t) r $
	such that  $ 0 \leq R \lesssim  t^{  \epsilon_1 -  \nu}$ and which we use for additional temporal decay in `elliptic error corrections' in Section \ref{sec:inner}.  A novel aspect of the present work is the approximation of the wave interaction in \eqref{Zakharov}, for which we need to consider \emph{hyperbolic scaling} $ a =  \f{r}{t}$, i.e.  $  a = \f{R}{\lambda t}, ~ 0 \leq a \ll R $ near $ t \sim 0 $ and in the inner region $ a \lesssim t^{\epsilon_1 - \f12}$.  The  \emph{self-similar parabolic region} is defined to be the set of all $(t,x)$ with
	$$
	\mathcal{S} : = \big \{ (t,r)~|~ t^{ \f{1}{2} + \epsilon_1} \lesssim  r \lesssim  t^{  \f{1}{2} -\epsilon_2 } \big \}.$$
	Here we reduce \eqref{Zakharov} via the variable $ y = r t^{- \f12}$, hence $ a = y t^{- \f12}$,~ $ 0 \leq   y \ll a $ near $ t \sim 0 $ and likewise we notice $ 1 \ll a \lesssim t^{- \f12 - \epsilon_2}$ near $ t \sim 0$.
	We note that in the intersection $\mathcal{I} \cap \mathcal{S}$ of the latter regions $ r  \sim t^{  \f{1}{2} + \epsilon_1}$, thus there holds 
	$  R \sim  t^{\epsilon_1 -  \nu}~~\text{and}~~ y \sim  t^{\epsilon_1 },~\epsilon_1 <  \nu.$ 
	This then  gives a matching condition in Section \ref{sec:self} for the transition of the  $a = \infty,  R = \infty $ expansion in the former region to the $ y = 0 $  asymptotic near the blow up time $ t = 0 $. In particular this expansion serves as a boundary condition for the resulting self-similar elliptic system at $ y = 0$.\\[2pt] We then turn to the iterative system  in $ y \gg1$  regime, where we classify two major asymptotic profiles: A \emph{fast oscillating and strongly decaying } term and a \emph{non-oscillatory slowly decaying} term. Finally the \emph{remote region} in Section \ref{sec:remote} will be a restriction to
	$$
	\mathcal{R} : = \{  (t,r)|~|~ t^{ \f{1}{2} - \epsilon_2}  \lesssim r \}.
	$$
	We thus have $ r  \sim t^{  \f{1}{2} - \epsilon_2}$ in the intersection $ \mathcal{S} \cap \mathcal{R}$, hence $ y \sim  t^{- \epsilon_2 }$.  Thus translating the $ y \gg1$ description into $(t,r)$ at $ r = 0$, we observe the radiation field to be a stationary top level term in the slowly decaying part (where $y \gg1$).  Therefore, in $\mathcal{R}$, we construct an approximation of \eqref{Zakharov} \emph{perturbatively} near the radiation part in the $(t,r)$ framework of the $y \gg1$ asymptotic description.\\[10pt]
	\emph{Notation}. We occasionally use the $O$-convention, which should be clarified. We denote by $ \mathcal{O}(g(R))$ where $ R \gg1 $ (or $ R \ll1$) the class of smooth functions $g(R)$ such that for some $ R _0> 0$ there exists an absolutely convergent expansion
	\begin{align*}
		&f(R)  = g(R) \cdot \sum_{k \geq 0} R^{-2k} c_k,\;\; R \geq R_0,\;\;\text{or},\;\;\;\;f(R)  = g(R) \cdot \sum_{k \geq 0} R^{2k} c_k,\;\; R \leq  R_0,
	\end{align*}
	respectively and likewise we use $\mathcal{O}(g(y))$ for $ y \ll1 $. In  Section \ref{sec:self}, we use the symbol $O(g(y))$ where $ y \gg1 $ for smooth \emph{asymptotic series}, i.e. if there exists $ (c_k)_{k \in \Z_{\geq 0}}$ such that for any $ y_0 > 0, N \in \Z_+$ there exists a smooth $ \mathbf{R}$ with
	\[
	f(R) = g(R) \cdot \sum_{k = 0}^N c_k R^{-2k} + \mathbf{R}^N(y),\;\;\; | \partial^j_y \mathbf{R}^N(y) | \leq C_j(y_0) \cdot y^{-2N -2 -j},\;\;\; y \geq y_0.
	\]
	We will use such expressions at the end of Section \ref{sec:self}.
	\begin{Rem}
		$\circ$\;\;Note that unless specified otherwise, the $ R\gg1 $ asymptotic is \emph{smooth} or \emph{differentiable}, for which we require 
		$ R^j \partial_R^j f = \mathcal{O}(g(R)), j \in \Z_+$.
		with the same $R_0 > 0$. Further $ \mathcal{O}_k(g(R))$ indicates that this holds \emph{for derivatives of order} $j \leq k$.\\
		$\circ$\;\; Also  we \emph{do note allow} the  $c_k$'s to depend on $R$, i.e. through $c_k(a) = c_k(R \slash (t\lambda)),\; a = \frac{r}{t}$. Such  singular (at $a \sim 1$) expansions will be obtained in Section \ref{sec:inner} by the application of the wave parametrix $\Box^{-1}$, however these expansions are only used either explicit or in the S-space notation (see below) in Section \ref{sec:inner}.
	\end{Rem}

	\section{The inner parabolic region $r \lesssim  t^{\f12 + \epsilon_1}$} \label{sec:inner}
	We choose $ 0 < \epsilon_1 < 1$ to be fixed below and $ t_1 > 0 $ small. Then let us consider the \emph{inner parabolic region}, or  \emph{blow up core region}, given by the set
	\[
	\mathcal{I}= \big \{  (t,r) \in (0, t_1) \times [0, \infty)\;|\; r \leq c_1 t^{\f12 + \epsilon_1} \big \},\;\;\; 
	\]
	where $ c_1 > 0 $ is a constant. In the present section we change to the coordinates $ (t, R) = (t, t^{- \f12 - \nu} r)$ and intend to construct an \emph{approximate solution} of \eqref{Zakharov} in $\mathcal{I}$ as \emph{corrections of the  renormalized bulk solution $W(R)$} having a \emph{well described asymptotic expansion} where $ R \gg1 $ is large.\\[10pt]
	The approximate solution will allow good control (over the corrections and error) in the sense of Lemma \ref{lem:estimates-inner} below  if $(t,r) \in \mathcal{I}$ since in particular $ R \lesssim  t^{\epsilon_1 - \nu}$.
	We start by recalling the variables
	\begin{align}
		&\;R = \lambda(t)r,\;\alpha(t) = \alpha_0 \log(t),~\lambda(t) = t^{- \f12 -\nu},
	\end{align}
	for any $  0 < t \lesssim 1,~ r > 0$ (or simply in $(t,r) \in \mathcal{I} $) and writing 
	\begin{align*}
		\psi(t,x) = e^{i \alpha(t)}\lambda(t) u(t,R)
	\end{align*}
	in \eqref{Zakharov} hence leads to the following equation
	\begin{align}
		- \dot{\alpha}(t) \lambda(t) u(t,R) &+ i \dot{\lambda}(t)u(t,R) + i \lambda(t)  \partial_t u(t,R) + i \dot{\lambda}(t)R \partial_Ru(t,R)\\[4pt] \nonumber
		&=\; \lambda^3(t)\Delta_R u(t, R)  -\lambda(t) n(t,R) u(t,R).
	\end{align}
	Here we use the notation $ \Delta_R = \partial_R^2 + \frac{3}{R} \partial_R$.
	Thus dividing by $ \lambda^3(t)$ we infer the system
	\begin{align*}
		&- \frac{\dot{\alpha}(t)}{ \lambda^2(t)} u + i \frac{\dot{\lambda}(t)}{\lambda^3(t)} u+ i \lambda^{-2}(t)  \partial_t u  + i \frac{\dot{\lambda}(t)}{\lambda^3(t)}R \partial_Ru =\; \Delta_R u(t, R)  - \lambda^{-2}(t) n \cdot u,\\[4pt] \nonumber
		&\square n = \lambda^2(t) \Delta (|u|^2),
	\end{align*}
	whence there holds
	\begin{align}\label{Red1}
		\begin{cases}
			~~- \alpha_0 t^{2\nu}u + t^{1 + 2 \nu} i \partial_t u - i( \f12 +\nu) t^{2 \nu}(1 + R\partial_R )u = - \Delta_R u -t^{1 + 2\nu}n \cdot u,&\\[4pt]
			~~\square n = \lambda^2(t) \Delta (|u|^2).&
		\end{cases}
	\end{align}
	\\~~
	We now seek approximate solutions of \eqref{Red1} of the form
	\begin{align}
		& u(t,R) = W(R) + z(t,R),\;\;\;\; %n = \square^{-1}( \lambda^2(t) \Delta (|W + z|^2))
	\end{align}
	% n(t,R) = \lambda^2(t)W^2(R) + m(t,R), 
	with fast decaying error in $\mathcal{I}$ and where $ z$ will be sufficiently small at $ t = 0$. %Here we calculate the inverse $\square^{-1}$ via the parametrix suggested by the two-step procedure from [KST],[KST],$\dots$.\\[4pt]% and  which includes %the necessity of keeping track of the 
	%a singular expansion of $n$ in $ a = r \slash t$ near the boundary of the cone $ a \sim 1 $.
	%In order to construct the approximation , we subsequently add corrections for $ N \in \Z_+$ 
	%\begin{align}
	%	&z_N^*  = \sum_{j =1}^N z_{j},\;\;\;
	%\end{align}
	%to the bulk term $ W(R)$.% where the increments $ z_j$ have the form
	%\begin{align}
	%=  z_1 + z_2  + \dots + z_N,\;\;\;
	%	z_j(t,\cdot) = \sum_{k, l} \frac{t^{2 \nu k}}{(\lambda t)^{2 l}}z_{k, l},~ j = 1, \dots, N,
	%\end{align}
	%with $ k, l \in \Z_+$ in a suitable subset (depending on $j \in \Z_+$). The $ (\lambda t)^{-2 l} = t^{(2\nu - 1)l}$ scale arises from the contribution of $n $ via the $ \square^{-1}$ parametrix in [KST].
	%\tinysection[5]{Deriving the main equation for $z$}\\[5pt]
	To begin with, let us clarify how to eliminate $n$ from \eqref{Red1} for an iterative procedure for $ z$ which simultaneously recovers the unknown ion potential  $n$.  To this end, we like to take the inverse of the wave operator 
	\begin{align}
		\label{that-n}
		n = \square^{-1}( \lambda^2(t) \Delta (|W + z|^2)),\;\;\; \Box = - \partial_t^2 + \partial_r^2 + \frac{3}{r} \partial_r
	\end{align}
	Assume that we have a formal inversion
	in \eqref{that-n}, then we obtain the main equation for this section
	\boxalign{
		\begin{align}\label{main-eq-z}
			- \Delta_R z - W^2(R)z -2 \text{Re}(\bar{z}W)W =\;\;&  \lambda^{-2} \big( \Box^{-1} \partial_t^2 \big(\lambda^2W^2\big) \big) W +  ( \alpha_0 t^{2 \nu}  - i( \f12 + \nu)t^{2\nu}\Lambda)W \nonumber \\[4pt] 
			& +\lambda^{-2} \big(\Box^{-1} \partial_t^2 \big(\lambda^2W^2\big)\big) z\\[4pt] \nonumber
			&  + \lambda^{-2}\Box^{-1} \partial_t^2( 2 \lambda^2  \text{Re}(\bar{z}W)  ) W + \lambda^{-2}\Box^{-1} \partial_{tt} ( 2 \lambda^2  \text{Re}(\bar{z}W)  ) z\\[4pt]  \nonumber
			& +  2 \lambda^2  \text{Re}(\bar{z}W)  ) z+\lambda^{-2}\Box^{-1} \triangle (\lambda^2 |z|^2  ) ( z + W)  + i t^{1 + 2 \nu}\partial_t z\\[4pt]  \nonumber
			& -  ( \alpha_0 t^{2 \nu}  +  i( \f12 + \nu)t^{2\nu}\Lambda)z,
		\end{align}
	}
	where  $\Lambda f = (1 + R\partial_R)f$ and which will define the iteration for $z$ in Section  \ref{subsec:inductive-ite-z} below. We then recover $n$ by setting
	\begin{align*}
		n &= \lambda^2W^2 +  2 \lambda^2  \text{Re}(\bar{z}W) + \Box^{-1} \partial_{t}^2 ( 2 \lambda^2  \text{Re}(\bar{z}W )\\
		&\;\;\; +  \Box^{-1} \partial_t^2 \big(\lambda^2W^2\big) + \Box^{-1} \triangle ( \lambda^2 |z|^2  ).
	\end{align*}
	In fact there holds 
	\begin{align} \label{boxx}
		\Box n &= \Box\big( \lambda^2W^2 \big) +  \partial_t^2 \big(\lambda^2W^2\big) + \Box\big(2 \lambda^2  \text{Re}(\bar{z}W)\big) + \partial_{t}^2 ( 2 \lambda^2  \text{Re}(\bar{z}W) + \triangle (\lambda^2 |z|^2  )\\ \nonumber
		& = \triangle \big(\lambda^2W^2 \big) + \triangle\big(2 \lambda^2  Re(\bar{z}W)\big) +  \triangle (\lambda^2 |z|^2  )\\ \nonumber
		& = \triangle \big(\big|\lambda(W+z)\big|^2\big),
	\end{align}
	and furthermore, from \eqref{main-eq-z} we then have 
	\begin{align}
		&i t^{1 + 2 \nu}\partial_t z -  ( \alpha_0 t^{2 \nu}  +  i( \f12 + \nu)t^{2\nu}(1 + R\partial_R))z +  \Delta_R z\\ \nonumber
		& = \lambda^{-3}(t)\cdot n\cdot \lambda(W+z) - W^3 +  ( \alpha_0 t^{2 \nu}  - i( \f12 + \nu)t^{2\nu}(1 + R\partial_R))W,
	\end{align}
	which together with \eqref{boxx} is of course equivalent to \eqref{Red1}, i.e.
	\begin{align*}
		&\big(i\partial_t + \triangle\big)\Big(\lambda(W(R) + z)\Big) = -n\cdot \lambda(W(R) + z),\\
		&\Box n  = \triangle \big(\big|\lambda(W+z)\big|^2\big).
	\end{align*}
	%for $ z = u - W$
	The first line of \eqref{main-eq-z} is considered to be the main initial source term, determining the first (or $0^{\text{th}}$ order) approximation $ z_1$. The remaining subsequent lines thus contain the linear, quadratic and cubic perturbative terms. In fact we construct the approximate solution of \eqref{main-eq-z} by subsequently adding corrections
	\begin{align}
		&z_N^*  =  \sum_{j =1}^N z_{j} = z_1 + z_2 + \dots + z_N,\;\; N \in \Z_+.\;\;\;
	\end{align}
	In this process we  calculate the inverse $\square^{-1}$ via the parametrix suggested by the procedure from 
	\cite{KST1, KST2-YM, KST-slow} % and  which includes %the necessity of keeping track of the 
	with a singular expansion in $ a = r \slash t$ near the boundary of the light cone $ a \sim 1 $.\\[5pt]
	The initial step  for the approximation of $z$ in \eqref{main-eq-z}, i.e. calculating $ z_1$, is to replace  $ \Box^{-1}(\cdot)$ by the wave parametrix constructed in \cite{KST1} for the main source term
	\[
	n = \Box^{-1} \partial_t^2 \big(\lambda^2(t)W^2(R)\big), 
	\]
	which we proceed to do below. Let us first define suitable spaces (analogous to \cite{KST1}, \cite{KST-slow}), which will describe the iterates $z_j$ and their asymptotics especially as $ R \gg1 $. Thus let us fix 
	$$ \beta = 2 \nu \cdot l_1 - l_2,\;\;\; \nu \geq 2, $$
	where $l_1, l_2 \in \Z$ with $ l_1 \geq 0$.
	\begin{Def}\label{defn:Qbetal} We denote by $\mathcal{Q}^{\beta}$ the vector space of functions $q  : \R_{\geq 0} \to \C,\; q = q(a)$ which are $C^\infty$ restricted to $[0,1)$, $(1, \infty)$ and furthermore satisfy the following properties:\\[4pt]
		\begin{itemize}\setlength\itemsep{4pt}
			\item[$\bullet$] $q$ is analytic near $a = 0$. Thus for $0\leq a\ll 1$ we have an absolutely convergent expansion
			\[
			q(a) = \sum_{ r \geq 0} q^{(0)}_r a^{2r}\,\; \big| q^{(0)}_r \big|\leq C_0^r,\;\; r \in \Z_{\geq 0}
			\]
			where $ q^{(0)} \in \mathbf{C}$ and $C_0\in \R_+$ is a constant.
			\item[$\bullet$] Near $a = 1$ the function $q$ admits an absolutely convergent \emph{Puiseux type expansion}, i.e. if $ \pm  a  \geq   \pm1$ and $ |1 -a| \ll1 $ we have in an absolute sense\footnote{By this we mean 'absolutely convergent'.}
			\begin{align*}
				q(a) = \sum_{\kappa\in\{0,1\}}\sum_{b \geq 0}' \sum_{\tilde{b} \geq b_{\star}(b)} \sum_{0\leq s \leq s_{\star} + b + \tilde{b}}' q^{(1)}_{\pm,  b \tilde{b} s \kappa}\cdot (|1-a|)^{2 \nu\cdot  b + \tilde{b}+ \frac{\kappa}{2}} |\log(|1-a|)|^s,
			\end{align*}
			where $ s_{\star} \in \Z_{\geq 0}, b_{\star} \in \Z$, $ q^{(1)}_{\pm, b \tilde{b} s \kappa} \in \mathbf{C}$ and  the second and fourth sums (with $\Sigma'$) are  finite such that 
			$s = 0$ unless $b > 0$ with (i)\; $2\nu \cdot b + \tilde{b} + \frac{\kappa}{2} \geq \tilde{\beta}(\nu) \geq 2$ if $ b > 0$ or $ \kappa > 0$, where $\tilde{\beta}$ is a function of $ \nu > 1$, as well as
			(ii)\; $ b_{\star}(0) \geq 0$ and (iii)\; $ q^{(1)}_{+, b \tilde{b} s \kappa}  = (-1)^{\tilde{b}}q^{(1)}_{-, b \tilde{b} s \kappa} $. % necessary $b>0$ or $ \kappa > 0$ ?
			\footnote{Also we require there exists $ b > 0$ such that $ q^{(1)}_{\pm, b \tilde{b} s \kappa} \neq 0$ for some parameters $(\pm,  \tilde{b}, s, \kappa)$.}
			Furthermore we have the coefficient bounds 
			\[
			\big|q^{(1)}_{ b \tilde{b} s \kappa} \big|\leq C_1^{b+\tilde{b}+s},\;\;\; b, \tilde{b}, s \in \Z_{\geq 0}
			\]
			for some $ C_1\in \R_+$.
			\item[$\bullet$] \;For $a\gg 1$ the functions admit an absolutely convergent expansion of the form 
			\begin{align*}
				q(a) = \sum_{\tilde{l} = 0}^{l_1}  \sum_{j \geq 0}  \sum_{0 \leq \tilde{k} \leq m_{\star} + j}'q^{(\infty)}_{j \tilde{k} \tilde{l}}\cdot a^{\beta - 2\nu \cdot \tilde{l} + d \cdot \tilde{l} - 2j} (\log(a))^{\tilde{k}},\;\,\big|q^{(\infty)}_{j \tilde{k} \tilde{l}} \big|\leq C_{\infty}^{j + \tilde{k}},\;\; j \in \Z_{\geq 0}
			\end{align*}
			for some $ C_{\infty} \in \R_{+}$, $ q^{(\infty)}_{j \tilde{k} \tilde{l}}\in \mathbf{C}$ and $ 0 \leq d \leq 2 \lfloor \nu \rfloor$.
		\end{itemize}
	\end{Def}
	\;\;\\[4pt]
	\underline{\emph{S-spaces}}. The preceding class of functions appears as coefficients in the spaces
	\[
	f \in S^m\big(R^{k} \log^{l}(R),\,\mathcal{Q}^{\beta}\big)
	\]
	with absolute  expansions for large $ R \gg1 $ similarly as introduced  in \cite{KST1}, \cite{KST-slow}.Therefore if $ R \gg1 $ we seek schematic expressions of the form
	\begin{align*}
		f = \sum_{r \geq 0} \sum_{0\leq j\leq l +  2r } q_{ r j}(a) \cdot R^{ k - 2r} \cdot( \log(R))^j,
	\end{align*}
	where $q_{r j}(a) \in \mathcal{Q}^{\beta}$. However, we require a slightly refined description, which we give in the following definition.\\
	\begin{Def} \label{defn:SQbetal}We let the space
		\begin{align}
			S^m\big(R^{k} \log^{l}(R),\,\mathcal{Q}^{\beta}\big)\\ \nonumber
		\end{align}
		be the vector space of functions $f = f(t,r)$ with  $(t, r) \in \mathcal{I}$ defined as follows:\\
		\begin{itemize} \setlength\itemsep{4pt}
			\item[$\bullet$] There exists $ R_0 > 0$ such that for $R\leq R_0$, we have a series expansion 
			\vspace{6pt}
			\begin{align}
				f = \sum_{p,   j, k' \geq 0} e_{j p k'}(R) \cdot a^{ 2j}  \frac{t^{k'}}{(t\lambda)^{2p}},%\\ \nonumber
			\end{align}
			where the functions $e_{j p k'}(R)$ are smooth and analytic for $R \ll1 $ small and such that
			\begin{align*}
				\big|e_{j p k'}(R)\big|\leq C_0^{j+p + k' },\,\;\;\;\big|\partial_R^{\gamma}e_{j p k'}(R)\big|\leq C_{\gamma}\cdot C_0^{ j +p + k' },\;\;\gamma \geq 0, \;\;R \leq R_0%\\ \nonumber
			\end{align*}
			for  constants $C_0,\; C_{\gamma} \in \R_+$ depending on $R_0 > 0$. In particular, we have absolute convergence for $t>0$ sufficiently small (depending again on $R_0$).
			\item[$\bullet$] The function $f$ vanishes to order $m$ at $R = 0$, to be precise for fixed $t>0$ we require that
			$$R^{-m}\cdot f(t, R \lambda^{-1}(t))$$
			has an even Taylor expansion at $R = 0$.
			\item[$\bullet$]  For $R\gg 1$, $0<a\ll 1$, there is a series expansion 
			\vspace{6pt}
			\begin{align}
				f = \sum_{p, r, k'\geq 0} \sum_{0\leq j\leq l +  2r }  \frac{t^{k'}}{(t\lambda)^{2p}}\cdot q_{p r j k'}(a)\cdot R^{ k - 2r} \cdot( \log(R))^j
			\end{align}
			where the functions $q_{p r j k'}(a) \in Q^{\beta}$ are analytic in $ a \ll1$ and admit even order expansions  	
			\begin{align}
				q_{p r j k'}(a) = \sum_{b \geq n }q_{p r j b k' }\cdot a^{2b},\,\; \big|q_{p r j b k' }\big|\leq \rho_0^{ k' + p+j+ r+ b }
			\end{align}
			for some $\rho_0 \in \R_+$ and $n \in \Z_{ \geq 0}$. In particular the expansion is \emph{absolutely convergent } in all three variables $R^{-1}, t, a$, provided $R$ is sufficiently large.%We observe that this region and the preceding one overlap in $R$ provided $t$ is sufficiently small.\\[4pt]
			\item[$\bullet$] For $R\gg 1$, $a\sim 1 $ with either $ 1-a\gtrsim 1$ or $ a-1 \gtrsim1$, we can write 
			\vspace{6pt}
			\begin{align}
				f = \sum_{p , r, k' \geq 0}\sum_{0\leq j\leq l+ 2r}  \frac{t^{k'}}{(t\lambda)^{2p}} \cdot q_{p r j k' }(a) \cdot R^{ k - 2r} \cdot\log^j(R)
			\end{align}
			where the functions $q_{p j r k'}(a) \in \mathcal{Q}^{\beta}$ are $C^\infty$ and uniformly bounded in the sense
			\vspace{6pt}
			\begin{align*}
				\big|q_{p r j k'}(a)\big|\lesssim \rho_1^{j+r+p + k' },\,\big|\partial_a^\gamma q_{p r j k' }(a)\big|\lesssim C_{\gamma}\cdot \rho_1^{ r +j+p + k' },\;\; \gamma \geq 0,
			\end{align*}
			for some $ \rho_1 \in \R_+$.
			\item[$\bullet$] For $R\gg1$, $|1-a|\ll1 $ and $ \pm a \geq \pm 1$, we can expand
			\begin{align}
				f = \sum_{p, r , j, k' \geq 0}\sum_{0\leq j\leq l+ 2r}  \frac{t^{k'}}{(t\lambda)^{2p}} \cdot q^{\pm}_{ p r j  k'}(a)\cdot R^{k - 2r}\cdot\log^j (R)
			\end{align}
			where the functions $q^{\pm}_{ p r j k' }(a) \in \mathcal{Q}^{\beta}$ admit \emph{Puiseux type} expansions  
			\vspace{6pt}
			\begin{align}
				&q^{\pm}_{p r j k'}(a) = \sum_{ \kappa\in\{0,1\}} \sum_{b \geq 0}' \sum_{b' \geq b_{\star}(b)} \sum_{0\leq  j' \leq d\cdot( r +p)+b+b'}' \times\\[5pt] \nonumber
				&\hspace{5cm} \times q_{\pm p r j k' \kappa}^{ b b' j'} %\times\\ \nonumber
				%& \hspace{6cm}\times
				(|1-a|)^{2\nu \cdot b + b' + \frac{\kappa}{2}} |\log(|1-a|)|^{j'}
			\end{align}
			as  in Definition \ref{defn:SQbetal} where  $d\in \R_{\geq 0}$  and for some $ \rho_2\in \R_+ $ the coefficients satisfy 
			\begin{align*}
				\big|q_{\pm p r j k' \kappa}^{ b b' j'}\big|\leq \rho_2^{ k' +p+r + j + b + b' }.
			\end{align*}
			\item[$\bullet$] $R\gg 1$, $a\gg 1$. In the right most regime, we have an expansion
			\vspace{6pt}
			\begin{align}
				f = \sum_{r , j, p, k' \geq 0}\sum_{0\leq j\leq l+ 2r}   \frac{t^{k'}}{(t\lambda)^{2p}} \cdot q_{ r j p k' }(a)\cdot R^{k - 2r}\cdot\log^j (R)
			\end{align}
			where the functions $q_{ r j p k' }(a) \in Q^{\beta}$ admit absolute expansions 
			\begin{align} 
				&q_{r j p k'}(a) = \sum_{\tilde{l} = 0}^{l_1}  \sum_{i \geq 0}  \sum_{0 \leq \tilde{k} \leq  l + 2r - j}' q^{(\infty)}_{j \tilde{k} \tilde{l} r i p k'} \cdot a^{\beta - 2\nu \cdot \tilde{l} + d \cdot  \tilde{l}- 2i}\; \log^{\tilde{k}}(a),
			\end{align}
			with  $\big|q^{(\infty)}_{j \tilde{k} \tilde{l} r i p k'} \big|\leq \rho_3^{ k' +p + r + j + i + \tilde{k}} $ for some  $ \rho_3 \in \R_+$ and $ 0 \leq d \leq 2\lfloor \nu \rfloor$.
		\end{itemize} 
	\end{Def}
	\;\\
	Some remarks on the preceding definition.
	\begin{rem}
		(i)\;\;The constants $C_0, C_{\gamma}$ and $\rho_j$ for $ j =0, 1,2,3$ in the preceding definition  depend on $f$ and the parameters $k,l,\beta,  m$. Further the domain of convergence of the expansions  is implicitly determined by these constants and thus dependent on $f$.\\[4pt]
		(ii)\;\; For a specific function $ f(a)$ with asymptotic expansions similar \footnote{here we mean expansions as in item one and two of Definition \ref{defn:Qbetal} but arbitrary polynomial growth as $ a \gg1$} to Definition \ref{defn:Qbetal},  we denote by $  f(a)Q^{\beta}$ the space of functions arising from Definition \ref{defn:Qbetal} multiplied with $f(a)$. Further we  may indicate by $ Q^{\beta}_n$ the class $ O(a^{2n})$ as $ a \ll1 $, %regime of the second item of Definition \ref{defn:SQbetal},
		then in particular $ a^{2j} \cdot Q_n^{\beta} \subset  Q_{n +j}^{\beta + 2j} $.\\[4pt]%For $ j  \in \Z_+$  we denote by $ a^{2j} \cdot Q_n^{\beta} \subset  Q_{n +j}^{\beta + 2j}  $ the subspace of functions as in  Definition \ref{defn:Qbetal} for which the expansions are multiplied with $a^{2j}$.\\[4pt]
		(iii)\;\; In  the iteration for $z$ stated below, we replace $k,l$ by linear functions $ c(k), \tilde{c}(l) $ over $l,k$ and $ \beta = \beta_l$. Then the proof of Lemma \ref{lem:Lemma-n-Anfang-elliptic} and Lemma \ref{lem:recoverz1} show that $C_0, C_{\gamma}, \rho_j$ will depend exponentially on $ k,l$, i.e. of the form
		$
		C_0(k,l)\lesssim \tilde{C}_0^{k +l},\;C_{\gamma}(k,l) \lesssim \tilde{C}_{\gamma}^{k +l},\;\rho_j(k,l)\lesssim \tilde{\rho}_j^{k +l},
		$
		and similar for the derivatives in the first and the third region.	Thus iterating finite steps we may restrict to uniform regions of (absolute) convergence. 
	\end{rem}
	\begin{rem} %(i) The last item of Definition \ref{defn:Qbetal} is adapted to $\nu > 0$ irrational, since logarithmic factors otherwise occur, i.e. we would replace the expansion with
		%\begin{align*}
		%	q(a) = \sum_{\tilde{l} = 0}^{l_1}  \sum_{j \geq 0}  \sum_{0 \leq \tilde{k} \leq m_{\star} + j}'q^{(\infty)}_{j \tilde{k} \tilde{l}}\cdot a^{\beta - 2\nu \cdot \tilde{l} - j} (\log(a))^{\tilde{k}},\;\,\big|q^{(\infty)}_{j \tilde{k} \tilde{l}} \big|\leq C_{\infty}^{j + \tilde{k}},\;\; j \in \Z_{\geq 0}
		%	\end{align*}
	%	for some $ C_{\infty} \in \R_{+}$ and $ q^{(\infty)}_{j \tilde{k} \tilde{l}}\in \mathbf{C}$. And similar in the last item of Definition \ref{defn:SQbetal} above.\\[4pt]
	The second item in Definition \ref{defn:Qbetal} could also be described by  the required expansion of $\mathcal{Q}$ in \cite{KST1}, \cite{KST-slow}, \cite{KST2-YM}, where we would write 
	\[
	q_{0}(a) + \sum_{\kappa\in\{0,1\}}\sum_{b \geq 0}' (|1-a|)^{\beta(b) + \frac{\kappa}{2} +1} \sum_{ s \geq 0}'q^{(1)}_{\pm,  \kappa b s}(a)\cdot |\log(|1-a|)|^s,\;\;
	\]
	with $ \beta(b) = \sum_{j \in\Z_{\geq2}}'  (2j-2)\nu -j-1 $ (depending on $b$) and where $ \beta(b) \geq \beta_2 = 2\nu -3 >1 $ if $\nu > 2$. Further $ q_0(a),\;q^{(1)}_{\pm,  \kappa b s}(a) $ are analytic for $ |1-a| \ll1 $ and $ \pm a \geq \pm 1$.
\end{rem}
\;\;\\
Now we additionally need the following.
\begin{Def} \label{defn:DasQprime}We denote by 
	\begin{align}
		S^m\big(R^{k}\log^l(R),\,(\mathcal{Q}^{\beta})'\big)
	\end{align}
	the space of functions defined in analogy to the preceding Definition \ref{defn:SQbetal}, but were in the last item the functions 
	$
	f_{r j}(a) 
	$
	are replaced by 
	\begin{align}
		(a\partial_a) g_{ r j }(a) + (a^{-1}\partial_a) \tilde{g}_{r j }(a)
	\end{align}
	with $g_{ rj }, \tilde{g}_{ r j }\in \mathcal{Q}^{\beta}$. 
	We similarly define 
	\begin{align}
		S^m\big(R^{k}\log^l(R),\,(\mathcal{Q}^{\beta})''\big),
	\end{align}
	by replacing $f_{r j }(a) $ (in the last item as before) with
	\begin{align}
		(a\partial_a)^2g_{ r j}(a) + \partial_a^2\tilde{g}_{ r j }(a) + (a^{-1}\partial_a)^2 \tilde{\tilde{g}}_{r j}(a),
	\end{align}
	where $\,g_{r j }(a), \tilde{g}_{r j}(a), \tilde{\tilde{g}}_{r j }(a)\in \mathcal{Q}^{\beta}$.
	%and the space
	%	\begin{align}
		%		S^m\big(R^{k}\log^l(R),\,(\mathcal{Q}^{\beta})^*\big)
		%	\end{align}
	%	by replacing $f_{r j }(a) $ with
	%	\begin{align}
		%	(a\partial_a)^2g_{ r j}(a) + \partial_a^2\tilde{g}_{ r j }(a) + (a^{-1}\partial_a)^2\tilde{\tilde{g}}_{r j}(a),\,g_{r j }, \tilde{g}_{r j}, \tilde{\tilde{g}}_{r j }\in \mathcal{Q}^{\beta}.
		%	\end{align}
	%	where $g_{r j}(a)$ 
	\;\\
\end{Def}
\begin{rem}\label{rem:truncate} (i)\; We stress that in  Definition \ref{defn:SQbetal} we do not impose analyticity in $R = \lambda(t) r$  everywhere on $ \R_{\geq 0}$ (as e.g. in \cite{KST1}, \cite{KST-slow}) which  allows us to suitably truncate functions in the preceding definition. This will be stated and used in the form of Lemma~\ref{lem:truncate} below.\\[4pt]
	(ii)\; We further use the notation  $S(R^{2k}\log^l(R) )$  for the subspace of 
	\[
	S^m(R^{2k}\log^l(R) ) \subset S^m(R^{2k}\log^l(R),  a^2\mathcal{Q}^{\beta_1} ),\;\; 
	\]
	where all $a$-dependent coefficients are constant and there is no dependence on $t>0$. These spaces are also defined in \cite[Section 2.1]{schmid} where we used the notation $S^m_{2,1}(R^k \log^l(R))$ in $ d = 4$ dimensions. Further we set analogously
	$	S(1) $
	for the subspace of 
	$$ S(1) := S^0\big(R^0\big) \subset  S^0\big( R^0, a^2 \mathcal{Q}^{\beta_1}\big) $$
	Note here the embeddings are realized  by the  leading order terms in the respective expansions for $ a \gg1$.
	%	\[
	%	S^m(R^{2k}\log^l(R) ) \subset S^m(R^{2k}\log^l(R), a^2 \cdot  \mathcal{Q}^{\beta_1} )
	%	\]
	%	where all $a$-dependent coefficients are constant and there is no dependence on $t>0$. These spaces are also defined in ref?, Section? where we used the notation $S^m_{2,1}(R^k \log^l(R))$ in $ d = 4$ dimensions. Further we set analogously
	%%	$	S(1) $
	%	for the subspace of 
	%	$$ S(1) := S^0\big(R^0\big) \subset  S^0\big( R^0, a^2 \cdot \mathcal{Q}^{\beta_1}\big) $$
	%	Note here we set $\beta_1 = -2$ and the embeddings are realized  by the  leading order terms in the respective expansions.
\end{rem}
\;\\
The following statements are a consequence of the above definitions for which we assume $ \beta \in \R$ has the above form, i.e. 
$$ \beta = 2\nu l_1 - l_2,\;\; l _1 \in \Z_{\geq 0}, \; l_2 \in \Z.$$
\begin{lem}\label{lem:truncate}
	Let  $\chi(R)\in C_0^\infty(\R_{\geq 0})$ equals $1$ on $[0, K_1]$ for some $K_1>0$, then we have
	\begin{align}
		\chi(R)\cdot S^m\big(R^{k}\log^bR),\,\mathcal{Q}^{\beta}\big)\subset S^m\big(R^{k}\log^b(R),\,\mathcal{Q}^{\beta}\big)
	\end{align}
	for $ k \in \Z, b \in \Z_{\geq 0}$ and $ \beta$ as above.
\end{lem}
\begin{lem}\label{lem:basicproduct} Let $k_{1,2} \in \Z, m_{1,2}\geq 0, b_{1,2}\geq 0$. Then we have the inclusion %\footnote{Observe the identity $\beta_l + \beta_{l'} = \beta_{l+l'} - 2$.}
	\begin{align}
		S^{m_1}\big(R^{k_1}\log^{b_1}(R),\,&\mathcal{Q}^{\beta_{1}}\big)\cdot S^{m_2}\big(R^{k_2}\log^{b_2}(R),\,\mathcal{Q}^{\beta_{2}}\big)\\[4pt] \nonumber
		&\subset S^{m_1+m_2}\big(R^{2(k_1+k_2)}\log^{b_1+b_2}(R),\,\mathcal{Q}^{\beta_1 + \beta_2}\big),\\[4pt]
		S^{m_1}\big(R^{k_1}\log^{b_1}(R),\,&\mathcal{Q}^{\beta_{1}}\big)\cdot S^{m_2}\big(R^{k_2}\log^{b_2}(R)\big)\\[4pt] \nonumber
		&\subset S^{m_1+m_2}\big(R^{2(k_1+k_2)}\log^{b_1+b_2}(R),\,\mathcal{Q}^{\beta_{1}}\big),\\[4pt]
		S^{m_1}\big(R^{k_1}\log^{b_1}(R)&\big)\cdot S^{m_2}\big(R^{k_2}\log^{b_2}(R)\big) \subset S^{m_1+m_2}\big(R^{2(k_1+k_2)}\log^{b_1+b_2}(R)\big).
	\end{align}
\end{lem}
\begin{Rem}
	Similar inclusions hold if we replace $\mathcal{Q}$ by $ \mathcal{Q}',\mathcal{Q}''$ for which we note schematically 
	$$
	\mathcal{Q} \cdot \mathcal{Q}' \hookrightarrow \mathcal{Q}',\;\;	\mathcal{Q} \cdot \mathcal{Q}'' \hookrightarrow 	\mathcal{Q}''.$$
	The former can be differentiated in the sense of Definition \ref{defn:DasQprime}.
\end{Rem}
\begin{lem}\label{lem:embedding} Let $ k\in \Z,  b \in \Z_{\geq 0}, l \in \Z_+ $ and $ b \geq 2$.  Then we have the inclusion %\footnote{Observe the identity $\beta_l + \beta_{l'} = \beta_{l+l'} - 2$.}
	\begin{align}
		S^{m}\big(R^{k}\log^{b}(R),\;\mathcal{Q}^{\beta - 2\nu \cdot l'- m_1}\big) \subset S^{m}\big(R^{k+2n}\log^{b -2n +m_2}(R),\;\mathcal{Q}^{\beta}\big),
	\end{align}
	for all $ n \in \Z_{+}$ with $2n \leq b$ and $ m_1, m_2 \in \Z_{\geq 0},\; m_1\; \text{even},\; 0 \leq l' \leq l_1$.
\end{lem}
\begin{lem}\label{lem:diff}
	We let $ f \in S^m\big(R^{k}\log^b(R),\,\mathcal{Q}^{\beta}\big)$ for $ k, b \in \Z_{\geq 0},\; l \in \Z$ . Then there holds 
	\begin{align}
		&t^2\partial_t^{\alpha_1}\partial_r^{\alpha_2} f \in   S^m\big(R^{k}\log^b(R),\,(\mathcal{Q}^{\beta})''\big),\;\;\alpha_1 + \alpha_2 = 2,\\
		&t\partial_t^{\alpha_1}\partial_r^{\alpha_2} f \in S^m\big(R^{k}\log^b(R),\,(\mathcal{Q}^{\beta})'\big),\;\; \alpha_1 + \alpha_2 =1.
	\end{align}
	We likewise have, in case of constant $a$ dependence,
	\begin{align*}
		t^{j} \partial_t^{\alpha_1} r^{\alpha_2} \partial_r^{\alpha_2} f \in S^m\big(R^{k}\log^b(R)\big), \;\; \alpha_1 + \alpha_2 = j,
	\end{align*}
	%In case of constant $a$ dependence, i.e. $ f \in  S^m\big(R^{k}\log^b(R)\big)$.
	%	for $\alpha_1 + \alpha_2 = 1$.
	%	\begin{align}
		%	&t^2\partial_t^{\alpha_1}\partial_r^{\alpha_2} S^m\big(R^{k}\log^b(R),\,\mathcal{Q}^{\beta_l}\big)\subset  S^m\big(R^{k}\log^b(R),\,(\mathcal{Q}^{\beta_l})''\big),
		%	\end{align}
	%	as well as 
	%	\begin{align}
		%	&t^2\partial_t^{\alpha_1}\partial_r^{\alpha_2} S^m\big(R^{k}\log^b(R),\,\mathcal{Q}^{\beta_l}\big)\subset  S^m\big(R^{k}\log^b(R),\,(\mathcal{Q}^{\beta_l})'\big),
		%	\end{align}
	%	for $\alpha_1 + \alpha_2 = 1$.
\end{lem}
\;\\
Some last remark on the definition of the S-space and \emph{how we use it in the following}.
\begin{rem} (i) \; We \underline{fix $\tilde{\beta}(\nu) = 2\nu - 3\slash2$} in the second item of Definition \ref{defn:Qbetal}. Further, in the last item of Definition \ref{defn:SQbetal} and Definition \ref{defn:Qbetal},
	\underline{we fix $ d =3 $} in the following use of the S-space with $ \beta_{s} = (2s-2) \nu -s-1 $ \footnote{ this is consistent with the representation \eqref{das-ex-rem}  for $\beta_{s} $ if $ \nu > \frac{3}{2}$} for which the expansion in the latter item reads
	\begin{align} 
		&q_{r j p k'}(a) = \sum_{ s' = 1}^{s}  \sum_{i \geq 0}  \sum_{0 \leq \tilde{k} \leq  l + 2r - j}' q^{(\infty)}_{j \tilde{k} s' r i p k'} \cdot a^{\beta_{s'}+ 2(s-s') - 2i}\; \log^{\tilde{k}}(a),
	\end{align}
	Note in particular, for $ s \geq 2 $ the expansion includes constant $ a $-dependence (e.g. taking $ s' = 1$). We will see that this expansion is sufficiently general to capture the $z$-iteration  (see below in Section \ref{subsec:inductive-ite-z})  but also precise enough to handle the  matching condition of the following Section \ref{sec:self}.\\[5pt]
	(ii)\; The corrections for our inductive approximation of solutions to  \eqref{main-eq-z} lie in versions of  the S-spaces and the general form of the $ a\gg1 $ expansion which is consistent with the interaction terms on the right has the form
	\begin{align}  \label{das-ex-rem}
		q_{r j p k'}(a) =  \sum_{0 \leq \tilde{k} \leq  l + 2r - j}' \;\sum_{\substack{ 2\nu \cdot \tilde{l} + i \;\leq \beta\\ \tilde{l} \geq 0} }  q^{(\infty)}_{j \tilde{k} \tilde{l} r i p k'} \cdot a^{2\nu \cdot \tilde{l} + i } \cdot \log^{\tilde{k}}(a).
	\end{align}
	This is a slightly less strict version of the expansion for the case of maximal $d\geq 0$  in the last item.
\end{rem}
\;\;\\[5pt]
%\tinysection[0]{Spaces for the iteration} 
\underline{\emph{For the iteration of $z$}}. We now set 
$$
\beta_l : = (2l -2) \cdot \nu - l -1,\;\;l \in \Z_+.\;
$$
Clearly we have
$
\beta_{l_1} + \beta_{l_2} = \beta_{l_1 + l_2} - 2\nu - 1 $ if $ l_1, l_2 \in \Z_+ $
and for  the iteration of $z$, the following spaces will then be relevant 
\begin{align}
	&\frac{t^{2\nu k}}{(t \lambda)^{2l}} \cdot S^m(R^{2k}\log^{s_{\ast}(l)}(R), \mathcal{Q}^{\beta_l}),\;\; l \geq 2, \label{wave-space}\\
	&\frac{t^{2\nu k}}{(t \lambda)^{2l}} \cdot S^m(R^{2k -2}\log^{s_{\ast}(l)}(R)),\;\; l = 0,1 \label{pure-space}
\end{align}
where $ k \in \Z_+$ and $s_{\ast}(l) $ is a linear function. In fact we set $ s_{\ast}(l) : = 2l-1$ in the upper line \eqref{wave-space} and $ s_{\ast}(l) : = 2l +1$ in the lower line \eqref{pure-space}.\\[6pt]
\;\;\\
%We set $ \beta_l : = (2l -2) \cdot \nu - l -1$ for $ \l \in \Z_+$, then most relevant for the iteration of $z$ will be the spaces 
%\[
%\frac{t^{2\nu k}}{(t \lambda)^{2l}} \cdot S^m(R^{2k}\log^{s_{\ast}}(R), \mathcal{Q}^{\beta_l}),\;\; \frac{t^{2\nu k}}{(t \lambda)^{2l}} \cdot S^m(R^{2k -2}\log^{s_{\ast}}(R), \mathcal{Q}^{\beta_l}),
%\]
%where $ k \in \Z_+$ and $s_{\ast}(l) $ is a linear function.\\[4pt]
%\;\;\\
Let us now proceed by describing how to approximate $\Box^{-1}$ in the source terms on the right of \eqref{main-eq-z}.  In particular we need to approximate  $\partial_t^2 \big(\lambda^2(t)W^2(R)\big)$ for the first of the initial terms in the first line.
\begin{comment}
	\begin{align}\label{main-eq-z2}
		- \Delta z - W^2(R)z -& 2 Re(\bar{z}W)W\\ \nonumber
		=~&~ ( \alpha_0 t^{2 \nu}  - i( \f12 + \nu)t^{2\nu}(1 + R\partial_R))W +  ( \alpha_0 t^{2 \nu}  - i( \f12 + \nu)t^{2\nu}(1 + R\partial_R))z\\ \nonumber 
		& + i t^{1 + 2 \nu}\partial_t z + \lambda^{-2}( z + W)\square^{-1} \partial_t^2( \lambda^2 W^2 + \lambda^2 |z|^2 + \lambda^2 2 Re(\bar{z}W)  )\\ \nonumber 
		& + z |z|^2 + W |z|^2 + 2 z Re(\bar{z}W)  
	\end{align}
	
	The second line contains only initial source terms for $z$ independent of $n$. The third line allows us to identify 
	$$ \cN(z) = \lambda^2 W^2 + \lambda^2 |z|^2 + \lambda^2 2 Re(\bar{z}W)  $$
	as the hyperbolic source and the last line in \eqref{main-eq-z} consists of nonlinear contributions solely through products of $z,~\bar{z}$. We now proceed as follows.\\
	\Step{1} The source for $z_{1,0}$.
	The source term for $ z_{0,1}$ in the initial step is $ \lambda^{-2} W \square^{-1}\partial_t^2\cN(0)$, which we approximate by  invoking the proceedure from [KST].
\end{comment}
%\tinysection[10]{The wave parametrix $\Box^{-1}$}
%\;\;\\[9pt]
\subsection{The wave parametrix $\Box^{-1}$: General description} \label{sec: the wave para gd} We construct a wave parametrix, following the argument in \cite{KST1}, \cite{KST-slow} (see also \cite{KST2-YM} \cite{K-S-full}) in order to  solve 
\begin{align}
	\begin{cases}
		\;\; n = \Box^{-1}\big(\triangle f(t,r)\big),&\\[3pt]
		\;\;\Box  = - \partial_t^2 + \triangle =  - \partial_t^2 + \partial_r^2 + \frac{3}{r}\partial_r,&
	\end{cases}
\end{align}
%i.e. inverting the wave operator 
%\begin{align*}
%	&~\Box  = - \partial_t^2 + \triangle =  - \partial_t^2 + \partial_r^2 + \frac{3}{r}\partial_r,
%\end{align*}
approximately via corrections
\[
n_k^* =  n_0 + n_1+ \ldots + n_k,\;\; n_0(t,r) = f(t,r)
\]
such that for any given $ N \in \Z$  the error after $k$ iteration steps
$$ e_k = \Box n_k^* - \triangle f = O(t^N) $$
where of course $k = k(N) \in \Z_+$ is sufficiently large. Subtracting $ n_0$, we hence reduce to calculating 
\begin{align*}
	& n = \Box^{-1}\big(\partial_t^2f(t,r)\big),\;\;\\
	& e_k = \Box n_k^{*} - \partial_t^2 f = \Box(n_1 + n_2 + \dots + n_k) - \partial_t^2 f = O(t^N).
\end{align*}
The  function $f(t,r)$ will always be in the spaces of the form \eqref{wave-space} or \eqref{pure-space}, i.e. we can consider generic sources in the space
\begin{align}
	&  \frac{t^{2\nu k}}{(t \lambda)^{2l}} \cdot S^m(R^{2k}\log^{s_{\ast}(l)}(R), \mathcal{Q}^{\beta_l}),
\end{align}
or similar for the  space \eqref{pure-space}., %which is carried out in Lemma ref? and Lemma ref ? below.  
In order to construct the increments  $\{n_k\}_{k}$ of the iteration,  we distinguish odd and even indices  $ n_{2k-1},\; n_{2k}$ and proceed as follows.\\[15pt]
%\boxalign{
	%	\begin{align*}\label{iteration-box1}
		%		&\;\; ( \partial_R^2 + \f{3}{R} \partial_R )n_{2k } = - \lambda^{-2} e_{2k -1},~~ k \geq 1\\[4pt] \label{iteration-box2}
		%		&\;\; ( - \partial_t^2 + \partial_r^2 + \f{3}{r} \partial_r )n_{2k +1} = -  w_{2k}^{0}, ~~ k \geq 1,
		%}
	%where $ n_1 $ is obtained as in \eqref{iteration-box2}.  We now  explain this as follows.\\[4pt]
	\underline{\emph{Elliptic modifier} \emph{Step} $(1)$}: All even increments $n_{2k},  k \in \Z_+$ are obtain via changing to $ R = \lambda(t) r$ and by solving for 
	
	\boxalign[14cm]{\begin{align}
			&\big( \partial_R^2 + \frac{3}{R}\partial_R\big)n_{2k} = - \lambda^{-2}(t) e_{2k -1},
		\end{align}
		where we hence note  the kernel $  G(R,s)  =  R^{-2} -  s^{-2} $ implies
		\begin{align}
			\lambda^2 \cdot n_{2k}(t,R) =     \f14 R^{-2}\int_{0}^R s^{3} \cdot  e_{2k -1}(s)\;ds  - \f14 \int_{0}^{R} s \cdot e_{2k -1}(s)\;ds.
		\end{align}
	}
	The latter integral is replaced by $ \int_R^{\infty}\;dR$  in case  $ e_{2k -1}(s) = O(s^{-2-})$  as $s \to \infty$ \footnote{For the initial source $\lambda^2 W^2$ we will see this  is only the case for the first step.}. The error after this step  has the form
	\boxalign[14cm]{ \begin{align}
			e_{2k} %= \Box n_{2k}^*  - \partial_t^2 f 
			& =  - \partial_t^2 n_{2k} + \lambda^2 \big( \partial_R^2 + \frac{3}{R}\partial_R\big)n_{2k} + e_{2k -1}\\ \nonumber
			&  = - \partial_t^2 n_{2k}.
		\end{align} 
	}
	Note that we will obtain, since $n_{2k},\; n_{2k+1} $ lie in the  above $S-$spaces, that 
	$$ - \lambda^{-2}(t) e_{2k -1} =  \lambda^{-2}(t)\partial_t^2 n_{2k}  = \frac{1}{(t\lambda)^2} \tilde{n}_{2k},\;\; (t \lambda)^{-2} = t^{2\nu -1},$$
	where $ n_{2k}, \tilde{n}_{2k}$ lie in the same $S-$space (have the same asymptotic behaviour). This is an effective gain of temporal decay and the reason for the factor $ (t\lambda)^{-2l}$ in \eqref{wave-space}. This may be compared to the pure Schr\"odinger iterations in \cite{OP}, \cite{Perelman} (see also \cite{schmid}) which resembles the second of the above $S-$spaces \eqref{pure-space} (compare e.g. to \cite[Definition 2.5]{schmid})
	\;\;\\[15pt]
	\underline{\emph{Hyperbolic modifier} \emph{Step} $(2)$}: All odd increments $n_{2k +1}, k \in \Z_+$ are obtained by inverting  the linear wave operator
	\boxalign[14cm]{ 
		\begin{align}
			\big(- \partial_t^2 + \partial_r^2 + \frac{3}{r}\partial_r \big)n_{2k +1} &=  -  w_{2k}^{0},\\
			e_{2k+1} = \Box n_{2k}^*  - \partial_t^2 f & =  \Box n_{2k+1} + e_{2k }\\ \nonumber
			& = e_{2k }^0 -  w_{2k}^{0} + o(1),\;\; R \gg1.  
		\end{align}
	}
	via changing to the (hyperbolic) self-similar variable $ a = \frac{r}{t}$ as we explain now. Here  $ w_{2k}^{0}$ is a certain analytic $R-$modification which matches the top level error, i.e. $ e_{2k}  = e^0_{2k } + o(1)$  \footnote{We pick the minimal set  of terms such that $  e_{2k}  = e^0_{2k } + o(1)$}  as  $ R \to \infty $ (provided $ a, t$ are fixed). %The error after this step reads
	%\boxalign[14cm]{
		%	\begin{align}
			%		e_{2k+1} = \Box n_{2k}^*  - \partial_t^2 f 
			%	& =  \Box n_{2k+1} + e_{2k }\\ \nonumber
			%	& = e_{2k }^0 -  w_{2k}^{0} + o(1),\;\; R \gg1.  
			%	\end{align} 
		%	}
	\begin{Rem}
		(i)	This  is the procedure if $ n_1 $ is obtained from the source $f(t,r)$ via a hyperbolic modifier step of Step (2). In case $n_1 $ is obtained as  in (1), we solve all \emph{even} increments via Step (2) and all odd increments exactly as in Step (1).\\[2pt]
		%(ii) In a sense, we can restrict to the case described above. To be precise, if the initial source term $ f(t,r)$ decays as $R\to \infty$ (for fixed $t > 0$),  we can start the iteration after carrying out finitely many \emph{elliptic modifier steps (1)} until we observe the condition \eqref{cond} and then obtain $n_1 $ as in $(2)$.
		(ii) In a sense, we can restrict to the case described above. If $\partial_t^2 f(t,r) = o(1)$ as $R \to \infty$ for fixed $ a \geq 0,  t> 0$ then we calculate $ \tilde{n}_1, \tilde{n}_2, \dots, \tilde{n}_m$ via
		\begin{align*}
			\big( \partial_R^2 + \frac{3}{R}\partial_R\big)n_{k} =  \lambda^{-2} \partial_t^2 n_{k-1},
		\end{align*}
		until $ e_{m}^0$ is non-vanishing. We then set $ n_1 = \tilde{n}_{m+1}$ and proceed as above starting with Step (2) for $n_1$.
	\end{Rem}
	In order to apply this \emph{Step (2)}, we further  require the form
	\boxalign[14cm]{
		\begin{align}\label{cond}
			t^2 e_{2k} \in  h(t) \cdot S^m(\log^b(R), \mathcal{Q}^{\beta})
		\end{align}
	}
	where $h(t)$ is a smooth function of $ t > 0 $. In fact, for all of the following cases, we set 
	$$h(t) = \frac{t^k \lambda^2}{(t \lambda)^{2l}},\;\beta =  \beta_l + 2k,\; l \geq 1,\;\; k \in \Z_{\geq 0}$$
	and observe (by definition)
	\begin{align*}
		&t^2 e_{2k}  = \frac{t^k \lambda^2}{(t \lambda)^{2l}}\big(\sum_{\tilde{b} = 0}^b \sum_{j \geq 0}\log^{\tilde{b}}(R) R^{-2j}q_{\tilde{b} j}(a) + L.O.T. \big),\;\; q_{\tilde{b} j} \in \mathcal{Q}^{\beta},\;\; R \gg1.\\
		&t^2 e_{2k}^0 =  \frac{t^k \lambda^2}{(t \lambda)^{2l}}\big(\sum_{\tilde{b} = 0}^b \log^{\tilde{b}}(R) q_{\tilde{b} 0}(a)\big).
	\end{align*}
	The ansatz for $ n_{2k+1} $  has then the form
	\boxalign[14cm]{
		\begin{align}
			&n_{2k+1} = \frac{t^k \lambda^2}{(t \lambda)^{2l}}\big(\sum_{\tilde{b} = 0}^b \big( \f12\log( 1 + R^2)\big)^{^{\tilde{b}}}  g_{\tilde{b}}(a)\big),\;\; R \geq 0,\\
			&\Box n_{2k+1} = - t^{-2 }\frac{t^k \lambda^2}{(t \lambda)^{2l}}\big(\sum_{\tilde{b} = 0}^b \big( \f12\log( 1 + R^2)\big)^{^{\tilde{b}}}  q_{ \tilde{b} 0}(a)\big) =: - w_{2k}^0.
		\end{align}
	}
	This is in particular useful if $ R \gg1 $ since the terms $ e_{2k}^0 - w_{2k}^{0}$ involve the expansions
	\begin{align}
		&q_{\tilde{b} 0}(a)\bigg(  \big( \f12\log( 1 + R^2)\big)^{^{\tilde{b}}} -  \big( \log( R)\big)^{\tilde{b}} \bigg)\\ \nonumber
		&= q_{\tilde{b}0}(a)\bigg(  \big( \f12\log( 1 + R^2)\big) -   \log( R) \bigg) \sum_{m = 0}^{\tilde{b} -1} \big( \f12\log( 1 + R^2)\big)^{m} \log^{\tilde{b}-1-m}(R)\\ \nonumber
		&= q_{\tilde{b} 0}(a) \cdot \mathcal{O}(R^{-2}) \sum_{m = 0}^{\tilde{b} -1}\mathcal{O}(1)\cdot  \log^{\tilde{b}-1-m}(R)
	\end{align}
	If all derivatives in the expression $ \Box n_{2k +1} $  fall on $g_{\tilde{b}} $,  we want to generate the factors $q_{\tilde{b} 0} $ in $w_{2k}^{0}$. Therefore we consider the operator
	\boxalign[14cm]{
		\begin{align}\label{hyperbolic-op}
			&Lf : = t^2 \left(\f{t^k \lambda^2}{(t \lambda)^{2l}}\right)^{-1} \square \left(~f(a)~\f{t^k\lambda^2}{(t \lambda)^{2l}} \right),\\[4pt]\nonumber
			& \square = - \partial_t^2 + \partial_r^2 + \f{3}{r}\partial_r,\;\; a = \frac{r}{t}
		\end{align}
	}
	Thus by
	\begin{align}
		\left(\f{t^k\lambda^2}{(t \lambda)^{2l}}\right)^{-1} \partial_t \circ \f{t^k\lambda^2}{(t \lambda)^{2l}} = \partial_t + (2 - 2l)\f{\dot{\lambda}}{\lambda} + \f{k -2l}{t}
	\end{align}
	we have for the temporal derivatives
	\[ -\left( \partial_t + (2 -2l) \f{\dot{\lambda}}{\lambda}  + \f{k -2l}{t}\right)^2 = -\left( \partial_t + \f{(2 -2l)(- \f12 - \nu) - 2l +k}{t}   \right)^2.  \]
	Hence $L$ is of the form
	\boxalign[14cm]{
		\begin{align}
			L_{\beta} &= t^2( - (\partial_t + \f{\beta}{t})^2 + \partial_r^2 + \f{3\partial_r}{r} )\\[3pt] \nonumber
			& = ( 1 - a^2) \partial_a^2 + (2 ( \beta -1) a + 3 a^{-1})\partial_a - \beta^2 + \beta, 
		\end{align}
	}
	where $ a = \f{r}{t}$ and in our case $ \beta = \beta_l +k  = (2l -2) ( \f12 + \nu) - 2l + k$.  Now we write
	\boxalign[14cm]{
		\begin{align}\label{main-eqn-L}
			\Box\big( &\frac{t^k \lambda^2}{(t \lambda)^{2l}}\sum_{\tilde{b} = 0}^b \big( \f12\log( 1 + R^2)\big)^{\tilde{b}}  g_{\tilde{b}}(a) \big)\\ \nonumber
			&=  \sum_{\tilde{b} = 0}^b \big( \f12\log( 1 + R^2)\big)^{\tilde{b}}  t^{-2} \frac{t^k \lambda^2}{(t \lambda)^{2l}} \big(L_{\beta_l + k }g_{\tilde{b}}\big)(a) + \sum_{\tilde{b} = 0}^b \big[ \Box,  \big( \f12\log( 1 + R^2)\big)^{\tilde{b}} \big]   \frac{t^k \lambda^2}{(t \lambda)^{2l}} g_{\tilde{b}}(a),
		\end{align}
	}
	where we note the latter commutator is the sum of the terms
	\begin{align}
		&- 2 \partial_t	 \big( \f12\log( 1 + R^2)\big)^{\tilde{b}}  \cdot \partial_t  \frac{t^k \lambda^2}{(t \lambda)^{2l}} g_{\tilde{b}}(a),\;\;\;2  \partial_r	 \big( \f12\log( 1 + R^2)\big)^{\tilde{b}}  \cdot \partial_r   \frac{t^k \lambda^2}{(t \lambda)^{2l}} g_{\tilde{b}}(a),\\
		&\frac{3}{r} \partial_r \big( \f12\log( 1 + R^2)\big)^{\tilde{b}}  \cdot  \frac{t^k \lambda^2}{(t \lambda)^{2l}} g_{\tilde{b}}(a),\;\;\;\;\partial_r^2 \big( \f12\log( 1 + R^2)\big)^{\tilde{b}}  \cdot  \frac{t^k \lambda^2}{(t \lambda)^{2l}} g_{\tilde{b}}(a),\\
		&- \partial_t^2 \big( \f12\log( 1 + R^2)\big)^{\tilde{b}}  \cdot \frac{t^k \lambda^2}{(t \lambda)^{2l}}g_{\tilde{b}}(a).
	\end{align}
	Thus we obtain
	\begin{align*}
		&t^{2} \bigg( \frac{t^k \lambda^2}{(t \lambda)^{2l}} \bigg)^{-1}	\big[ \Box,  \big( \f12\log( 1 + R^2)\big)^{\tilde{b}} \big]   \frac{t^k \lambda^2}{(t \lambda)^{2l}} g_{\tilde{b}}(a)\\[4pt]
		& = \;( -1-2\nu) ( a g_{\tilde{b}}'(a)  - (2\nu -3)g_{\tilde{b}}(a) ) \tilde{b} \big( \f12\log( 1 + R^2)\big)^{\tilde{b}-1} \frac{R^2}{1 + R^2}\\[4pt]
		&\;\; \;+ \frac{2}{a} g'_{\tilde{b}}(a) \tilde{b} \big( \f12\log( 1 + R^2)\big)^{\tilde{b}-1} \frac{R^2}{1 + R^2}  \;+  \; \frac{3}{a^2}g_{\tilde{b}}(a) \tilde{b} \big( \f12\log( 1 + R^2)\big)^{\tilde{b}-1} \frac{R^2}{1 + R^2}\\[4pt]
		&\;\;\; +  a^{-2}g_{\tilde{b}}(a) \bigg( \tilde{b}(\tilde{b}-1) \big( \f12\log( 1 + R^2)\big)^{\tilde{b}-2}   + 2 \tilde{b}   \big( \f12\log( 1 + R^2)\big)^{\tilde{b}-1}  \bigg)\frac{R^4}{(1 + R^2)^2}\\[4pt]
		& \;\;\;- (\f12 + \nu)^2  g_{\tilde{b}}(a) \bigg( \tilde{b}(\tilde{b}-1) \big( \f12\log( 1 + R^2)\big)^{\tilde{b}-2}   + 2 \tilde{b}   \big( \f12\log( 1 + R^2)\big)^{\tilde{b}-1}  \bigg)\frac{R^4}{(1 + R^2)^2},
	\end{align*}
	where the terms $  \big( \f12\log( 1 + R^2)\big)^{\tilde{b}-2}  $ are absent if $ \tilde{b} = 1$. Considering the leading order asymptotic as $ R \to \infty$, this leads to the following system of equations which finally defines the \emph{hyperbolic step}
	\boxalign[14cm]{
		\begin{align} \label{main-sys-L-1}
			(L_{\beta_l + k }g_{b})(a) &= - q_{0 b}(a),\\[4pt]\label{main-sys-L-2}
			(L_{\beta_l + k }g_{b-1})(a) &= - q_{0 b-1}(a) - b\big[(-1-2\nu) a g_b'(a) - (2\nu -3) g_b(a)\\[4pt] \nonumber\label{main-sys-L-3}
			&\;\;\;\;  + \frac{2}{a} \big( g'_b(a) + \frac{1}{a} g_b(a)\big) + \frac{3}{a^2} g_b(a) - 2(\f12 + \nu)^2 g_b(a)\big] ,\\[4pt]
			(L_{\beta_l + k }g_{\tilde{b}})(a) &= - q_{0 \tilde{b}}(a) - (\tilde{b}+1)\big[(-1-2\nu) a g_{\tilde{b}+1}'(a) - (2\nu -3) g_{\tilde{b}+1}(a)\\[4pt] \nonumber
			&\;\;\;\;  + \frac{2}{a} \big( g'_{\tilde{b}+1}(a) + \frac{1}{a} g_{\tilde{b}+1}(a)\big) + \frac{3}{a^2} g_{\tilde{b}+1}(a) - 2(\f12 + \nu)^2 g_{\tilde{b}+1}(a)\big] ,\\[4pt]\nonumber
			& \;\;\;\; - (\tilde{b}+1)(\tilde{b}+2)\big[ \frac{1}{a^2}  g_{\tilde{b}+2}  - (\f12 + \nu)^2 g_{\tilde{b}+2}(a)\big],\\[4pt]\nonumber
			&\;\;\; 0 \leq \tilde{b} \leq b-2.
		\end{align}
	}
	\;\;\\
	\emph{Back to the error $e_{2k+1}$} For the source terms on the right of \eqref{main-sys-L-1} - \eqref{main-sys-L-3}, we note plugging this into \eqref{main-eqn-L}, the commutator on the right of \eqref{main-eqn-L} does not exactly cancel. However these terms are now controlled as $ R \to \infty$, to be precise $e_{2k +1}$ now reads
	\begin{align} 
		\sum_{\tilde{b}=0}^b q_{\tilde{b} 0}(a) \bigg(  \big( \log( R)\big)^{\tilde{b}} &-  \big( \f12\log( 1 + R^2)\big)^{^{\tilde{b}}}  \bigg)\\ \nonumber
		&+  \sum_{\tilde{b}=0}^b\tilde{q}_{\tilde{b} 0}(a) \bigg(  \big(  \big( \f12\log( 1 + R^2)\big)^{^{\tilde{b}}} \mathcal{O}(R^{-2}) \bigg) + o(1),
	\end{align} 
	where we denote  the coefficients of the commutator by $ \tilde{q}_{\tilde{b} 0}(a) $.
	\begin{Rem} \label{Rem-compension-Box-paramet}
		Here we note $ q_{\tilde{b}0}(a),\; \tilde{q}_{\tilde{b},0}(a)$ grow (essentially of order $O(a^{\beta})$) as $ a\gg1 $, which of course can be understood as growth in $R $. However the temporal factor as in \eqref{wave-space} provides a \emph{compensation in  the region $\mathcal{I}$}, i.e. where $r \lesssim t^{\f12 + \epsilon_1}$. We will make this precise in Lemma \ref{lem:estimates-inner} of the subsequent Section \ref{subsec:inductive-ite-z}. Note further in \cite{KST-slow}, \cite{KST1}, $ a<1$ since the approximation is restricted to the light cone.
	\end{Rem}
	%where the $ \mathcal{O}( R^{-2})$ terms are obtained 
	%$ \frac{R^2}{1 + R^2} -1 =  -(1 + R^2)^{-1}, \frac{R^4}{(1 + R^2)^2} -1 = - \frac{2R^2 +1}{(1 + R^2)^2}$
	\subsection{Fundamental solutions for $L_{\beta}$} We recall the operator\\
	\begin{align}
		L_{\beta} &=   ( 1 - a^2) \partial_a^2 + (2 ( \beta -1) a + 3 a^{-1})\partial_a - \beta^2 + \beta, 
	\end{align}
	which has a first order singularity at $ a = -1,1$.  The Wronskian
	\[
	W(a) = \phi_1' (a)\phi_2(a) - \phi_2' (a)\phi_1 (a) 
	\]
	for any fundamental system $\{ \phi_1, \phi_2\} $ of  $ L_{\beta}$ satisfies
	\[
	W'(a) =\frac{2(\beta-1)a + 3 a^{-1}}{1 - a^2}W(a),~~a \neq 1,\\
	\]
	and hence 
	\[
	W(a)  = c( |1 - a^2| )^{\beta + \f12} a^{-3}.
	\]
	In particular we have the following absolute expansions
	
	\begin{align*}
		&W^{-1}(a)  = c a^3 \sum_{k \geq 0} c^{(0)}_{\beta, k} a^{2k},\;\; \; 0 < a \ll1,\;\;W^{-1}(a)   = c a^{- 2\beta  + 2 }\sum_{k \geq 0} c^{(\infty)}_{\beta, k} a^{-2k},\;\; \;  a \gg1,\\[4pt]
		&W^{-1}(a)  = c (1 -a)^{-\beta - \f12 } \sum_{k \geq 0} c_{\beta, k} (1 -a)^{2k},\;\; \;  | 1 -a| \ll1 ,\;\; a< 1,\\[4pt]
		&W^{-1}(a)  = c  (a-1 )^{-\beta - \f12 } \sum_{k \geq 0} (-1)^k c_{\beta, k} (a-1)^{2k},\;\; \;  | 1 -a| \ll1 ,\;\; a> 1.
	\end{align*}
	Describing the solutions of $ L_{\beta} f = 0$ at $ a= 0,1$, we have  the following. 
	\begin{Lemma}\label{FS-null-eins}For $L_{\beta}$ with $ \beta > -1$ and $ \beta \notin \Z_0^+$ there exist fundamental solutions $ \phi^{(1)}_1(a),~ \phi^{(1)}_2(a)$  on $[0,2] $ with
		\begin{align}
			\phi^{(1)}_1(a) &= \sum_{k \geq 0 } \mu_k ( 1 - a)^k\\
			\phi^{(1)}_2(a) &= ( | 1 - a|)^{\beta + \f{3}{2}} \sum_{k \geq 0 } \tilde{\mu}_k( 1 - a)^k,
		\end{align}
		in an absolute sense  if $ |a-1| \ll 1$ where in particular $ \phi^{(1)}_1$ is entire. The function $\phi^{(1)}_2 \in C^{\lfloor \beta \rfloor }((0,2)) $ is analytic when restricted to $[0, 1),~(1, 2]$ respectively. Further if $ \beta \in \Z_+$ we have to correct $ 	\phi^{(1)}_1(a) $ to
		\[	\phi^{(1)}_1(a) = \sum_{k \geq 0 } \mu_k ( 1 - a)^k + c	\phi^{(1)}_2(a)\log(|1 - a|),\]
		where $ c > 0 $ is a unique constant. Further the equation $ L_{\beta}w = \sum_{ l \geq 0} a^{2l + m }c_l$ for $ a \ll 1 $ satisfying $c_0 \neq 0,~ w(0) = 0, ~w'(0) = 0$ has a unique solution of the form
		$$ w(a) = \big(\sum_{l \geq 0} a^{2l} \tilde{c}_l\big) a^{ m +2},$$
		where $ \tilde{c}_l $ depends on $ c_0, \dots, c_{l},~ \beta $.
	\end{Lemma}
	\begin{Rem}
		Note that we have $\beta_k > -1 $ if $ \nu > 1,~ k \geq 2$.
	\end{Rem}
	\begin{proof}
		The two statements near $ a = 1,~ a = 0 $ are as in \cite{KST-slow} and \cite{K-S-full}. Especially, near $ a= 0 $ we use 
		\[ L_{\beta } a^{ k + 2} = (k + 2)(k + 4) a^{k}  + ( 2 (k + 2)( \beta - 1) - (k + 2 )(k + 1) + \beta - \beta^2 ) a^{ k + 2}, \]
		and hence calculate for $ k \geq 0 $
		\begin{align*}
			&\tilde{c}_ 0(m+2)(m+4) = c_0,\\
			& \tilde{c}_ {k +1}(m+2 + 2k)(m+4 + 2k) = - (2(m + 2k)(\beta -1) - (m + 2k)(m + 2k -1) + \beta - \beta^2)\tilde{c}_{k}+ c_{k + 1},
		\end{align*}
		which implies the convergence near $ a = 0$.
		For the case $ a \sim 1 $, we consider the fundamental solutions  $ \{ 1,   (|1 -a|)^{\beta + \f32}\}$ for 
		$$ L_{\beta}^1 = 2(1 -a)\partial_{a}^2 + 2 (\beta +1 )\partial_a = sign( 1 -a ) 2 (|1-a|)^{\beta + \f{3}{2}} \partial_a((|1-a|)^{-\beta - \f{1}{2}}\partial_a ),$$
		which approximates $ L_{\beta }$ well where $ a \sim 1 $. Hence we distinguish $ \pm (a -1 )  > 0 $ in order to differentiate and argue  (c.f. \cite{KST-slow})
		for the ansatz above.  This leads to the following recursion ( independent of the sign of $( 1 - a)$)
		\begin{align}
			(k+1)(2k -1 - 2 \beta)\mu_{k+1} &= [k(k+4 -2\beta) + \beta^2 - \beta ]\mu_k + 3 \sum_{l = 0}^{k-1}l \mu_l \label{FS-1}\\
			(2\beta + 5 + 2k)(k+1) \tilde{\mu}_{k+1} & = [(\beta + \f32 +k)(k - (\beta - \f12 )+ 5) + \beta^2 - \beta ]\tilde{\mu}_k \label{FS-2}\\ \nonumber
			&~~~~~+  3 \sum_{l = 0}^{k-1}(\beta + \f32 + l)\tilde{\mu}_l
		\end{align}
		and as in \cite{KST-slow} there exists a (unique upon fixing $ \mu_0, \tilde{\mu}_0$) sequence  for \eqref{FS-2} as well as  \eqref{FS-1}  for $ \beta \notin \Z^+_0$. Further the recursion implies that the infinite sums $ \phi^{(1)}_1(a) ,~ \phi^{(1)}_2(a)  $ defined through $ \mu_k,\tilde{\mu}_k$ converge absolutely. In the case where $ \beta \in \Z_0^+$, the suggested correction for \eqref{FS-1} gives additionally
		\begin{align*}
			- 2(1 + a) ( 1 - a)^{k-1} ( k + \beta +\f32)  &\mu_k (|1-a|)^{\beta + \f32} - (1 + a) ( 1 - a)^{ k -1} \mu_k (|1-a|)^{\beta + \f32}\\
			& - [2 ( \beta - 1) a + \f{3}{a}] (1-a)^{k-1} \mu_k (|1-a|)^{\beta + \f32},
		\end{align*}
		which can be checked to lead to a non-degenerate factor for $ \mu_{k+1}$ and hence a solution.
	\end{proof}
	~~\\
	Now for the asymptotics at $ a = \infty $, we infer solutions of order $ O(a^{\beta}), O(a^{\beta -1})$ if $ a \gg1$.
	\begin{Lemma}\label{FS-infty} For $L_{\beta }$ there exist fundamental solutions $ \phi^{(\infty)}_1(a),~ \phi^{(\infty)}_2(a)$  on $ (1, \infty) $ with
		\begin{align}
			\phi^{(\infty)}_1(a) &= a^{\beta}\sum_{k \geq 0 } \gamma_k a^{- 2k }, ~ a \gg 1\\
			\phi^{(\infty)}_2(a) &= a^{\beta -1}\sum_{k \geq 0 } \delta_k a^{- 2k  },~ a \gg1.
		\end{align}
		where both sums converge absolutely. 
	\end{Lemma}
	\begin{proof}
		For the case  $ a \gg 1 $, we consider the fundamental solutions  $ \{ a^{\beta},   a^{\beta -1}\}$ for 
		$$ L_{\beta}^{\infty} = -a^2\partial_{a}^2+ 2(\beta -1) a \partial_{a} - \beta^2 + \beta,$$
		which approximates $ L_{\beta }$ well where $ a \gg 1 $. This motivates the ansatz  above and leads to the following recursion 
		\begin{align*}
			&\gamma_k\cdot [\beta( \beta - 4k + 2) + 4k(k-1)] = \gamma_{k+1}(2(k+1)(2k + 1))\\[3pt]
			& \delta_k\cdot [(\beta-1)( \beta - 4k + 1) + 4k(k-1)] = \delta_{k+1}(2(k+1)(2k + 1)),
		\end{align*}
		which implies
		\[
		\gamma_{k+1}=\gamma_k\cdot(1 + O(k^{-1})), ~~	\delta_{k+1}=\delta_k\cdot(1 + O(k^{-1}))
		\]
		which in turn implies that the above expansions converge absolutely for $a \gg 1$. 
	\end{proof}
	\begin{Rem}
		The operator  $L_{\beta} $ has two real-valued  fundamental solutions on $(0, \infty) $ of the form
		$$ \psi(\beta, a),~~ \tilde{\psi}(\beta, a )(| 1 - a^2|)^{\f32 + \beta},$$
		where $ \psi, \tilde{\psi}$ are entire functions  of $ \beta, 1 - a^2$. Note here that
		\[  
		(| 1 - a^2|)^{\f32 + \beta} = \sum_k c_k (| 1 - a|)^{ k +\f32 + \beta}
		\]
		converges absolutely if $a \sim 1 $. 
		%In fact, $\psi, \tilde{\psi}$ are Gaussian hypergeometric functions of the variable and the asymptotics in  Lemma \ref{FS}, \ref{FS2} are given by their obligatory series %expansions at $ z = 0, \infty $ (as analytic continuations of the Gaussian power series) respectively. The explicit recursion in Lemma \ref{FS}, \ref{FS2} gives an alternative,  more direct,  deduction of the asymptotics. 
		Furthermore,  similar as for the recursion in the $ a\gg 1 $ region in Lemma \ref{FS-infty},  we obtain from the recursion in Lemma \ref{FS-null-eins}
		\begin{align}
			\tilde{c}_{k+1 } = \tilde{c}_{k}( 1 + O(k^{-1})) + c_{k+1}O(k^{-2}),
		\end{align}
		which implies absolute convergence near $ a =1$.
	\end{Rem}
	\;\;\\
	\emph{The inhomogeneous problem} $ L_{\beta}q = f $ has  particular solutions near $  a \sim 1$ of the form
	\begin{align}
		q(a) =&~ \int_{a^*}^a G^{(1)}(a,s) s^3 ( 1 - s^2)^{- \beta - \f32} f(s) ~ds,~~0 < a \leq 1,\label{eq:Green1}\\
		q(a) =&~ \int_a^{2-a^*} G^{(1)}(a,s) s^3 ( s^2 -1)^{- \beta - \f32} f(s) ~ds,~~1 < a \leq 2,\label{eq:Green2}
	\end{align}
	where  $ 0 < a^* <1$ and with the Greens function $ G^{(1)}(a, s) = \phi^{(1)}_1(a) \phi^{(1)}_2(s) - \phi^{(1)}_2(a) \phi^{(1)}_1(s)$. Therefore  if $ |\tilde{a}-1|\ll1 $ for  $ \tilde{a} = a,a^*$ we have the asymptotic expansion
	\begin{align*}
		&q(a) =  \mathcal{O}(1) \int_{a^*}^a s^3 \mathcal{O}(1) f(s)\; ds - (1 - a)^{\beta + \f32} \int_{a^*}^a s^3 \mathcal{O}((1-s)^{- \beta -\f32}) f(s)\;ds\\
		&q(a) =  \mathcal{O}(1) \int_{a}^{2 - a^*} s^3 \mathcal{O}(1) f(s)\; ds - (a-1)^{\beta + \f32} \int_{a^*}^a s^3 \mathcal{O}((s-1)^{- \beta -\f32}) f(s)\;ds
	\end{align*}
	Also if $ f(s) = O((|1-s|)^{-1 +})$  as $ s \to 1^{\pm}$ then we may replace the first integrals on the right by $ \int_a^1\;ds$ and $ \int_1^a\;ds $ in order to obtain decay at $ a =1$.\\[3pt]
	Now, in particular for $ a \gtrsim 1$ the solutions of $ L_{\beta}q = f$ have the form
	\begin{align}\label{sol-large-a}
		q(a) =  c_1(a^*)	\phi^{(\infty)}_1(a)  + c_2(a^*) 	\phi^{(\infty)}_2(a)  + \int_{a^*}^{a}  G^{(\infty)}(a,s) s^3 ( s^2-1)^{- \beta - \f32} f(s) ~ds.
	\end{align}
	Thus  if $a \geq a_*\gg1 $ we again have asymptotic expansions for the particular part
	\begin{align}
		q_{part}(a) =   \mathcal{O}(a^{\beta})\int_{a^*}^{a}  \mathcal{O}(s^{- \beta-1})  f(s) ~ds - \mathcal{O}(a^{\beta-1})\int_{a^*}^{a}  \mathcal{O}(s^{- \beta})  f(s) ~ds.\label{eq:Greenlargea}
	\end{align}
	Therefore if  $ f(s) = O(s^{\beta -1 -})$ as $ s \to \infty$, then we may also write
	\begin{align*}
		q_{part}(a) = \int_{a}^{\infty}  G^{(\infty)}(a,s) s^3 ( s^2-1)^{- \beta - \f32} f(s) ~ds. %= O(a^{\beta -}),\;\; a \to \infty.
	\end{align*}
	in order to calculate the subleading asymptotic in \eqref{sol-large-a}.
	%\emph{The correction $n_3$}. Considering the error $ e^2_0$  and \eqref{hyperbolic-op}, in 

	\begin{comment}

		The iteration is then defined by
		
		for a certain part of the error 
		Similar to [KST] , [KS] we calculate the sequence of approximations $n^*_k $ for the solution of
		$$ \Box n = \Delta ( \lambda^2 W^2(R))$$
		via the corrections $ n_k $ 
		\[  n^*_k = n_0 + n_1 + \dots + n_k,~~ n_0 = \lambda^2 W^2(R) \]
		and the error $ e_k $ in the $k-$step is defined by
		\[e_k = ( - \partial_t^2 + \partial_r^2 + \frac{3 \partial_r}{r} ) n^*_k.\]
		Thus we take
		\begin{align}\label{odd}
			( \partial_R^2 + \f{3 \partial_R}{R})n_{2k } =&~  - \lambda^{-2} e_{2k -1},~~ k \geq 1\\ \label{even}
			( - \partial_t^2 + \partial_r^2 + \f{3 \partial_r}{r})n_{2k +1} =&~  -  e_{2k}^{0} + e_{2k}^{mod}, ~~ k \geq 1.
		\end{align}
		Here  $ e^0_{2k }$ is the top level error with $ e_{2k}  = e^0_{2k } + o(1)$  \footnote{As before we pick the minimal set  of terms such that $  e_{2k}  = e^0_{2k } + o(1)$}  as  $ R \to \infty $ (provided $ a, t$ are fixed).  Further  $ n_{2k+1} $ are \emph{hyperbolic corrections}, i.e. they have the form of $ n_3$ above, and $e_{2k}^{mod}$ contains all error terms generated by $ \Box'$ as explained above (these are small if $R \gg 1 $).\\[3pt]
	\end{comment}
	\;\;\\
	Now these expansions of  fundamental solutions are sufficient to describe the increments $n_k$ and error functions $ e_k$ via the $S-$spaces above. 
	The heuristic argument for alternating \emph{Step (1)} \and  \emph{Step (2)} as in \cite{KST-slow}, is the following\\[4pt]
	(i)\;\;We expect the \emph{elliptic modifier} to be an effective correction if $ r \ll t $ since we have  \\
	\;\;\hspace{1cm}\;\;$\partial_t \sim  \dot{\lambda(t)} r \ll  \lambda(t) \sim \partial_r $ concerning the variable $R = \lambda r $.\\[4pt]
	%$$e_{2k } = e_{2k-1} - e^{0}_{2k-1},~~e_{2k +1 } = e_{2k} - e^{0}_{2k}  - \partial_t^2 v_{2k +1}. $$
	(ii)\;\;The  \emph{hyperbolic modifier} is expected to be a useful  correction if $ \partial_r \sim \dot{\lambda(t)} r \sim \lambda(t) \sim \partial_r$, hence near $ r \sim t$ and we change to the \emph{self-similar} coordinate. We invert the  Sturm-Liouville operator  $ L_{\beta }$, which is singular at $ r = \pm t$ and thus leads to `finite regularity' at $ a =1$ in the sense of Definition \ref{defn:Qbetal}.\\[5pt]
	\emph{Concerning the error:} The gain of $(t\lambda)^{-2}$ is each of the (i) steps comes at the expense of logarithmic growth in $R$, which is 'removed' in each of the $(ii) $ steps at the expense of growth in $a \gg1$. However in the inner region $\mathcal{I}$, this provides  a net control of the error, c.f. also Remark \ref{Rem-compension-Box-paramet} above.\\[5pt]
	We further need to apply this procedure to the expressions on the right of \eqref{main-eq-z} in the main iteration for $z$, i.e. we consider terms of the form
	\begin{align}\label{initial source}
		&(1)\;\;\lambda^{-2} \big( \Box^{-1} \partial_t^2 \big(\lambda^2W^2\big) \big) W,\;\;\;\;\emph{\text{(initial\;source\;term)}},\\[4pt] \label{linear-source}
		& (2) \;\;\lambda^{-2} \big(\Box^{-1} \partial_t^2 \big(\lambda^2W^2\big)\big) z,\;\;\; \lambda^{-2}\big(\Box^{-1} \partial_t^2( 2 \lambda^2  Re(\bar{z}W)  ) \big)W,\;\;\;\;\emph{\text{(linear\;terms)}},\\[4pt]  \label{quad-cub}
		&(3)\;\; \lambda^{-2}\big(\Box^{-1} \partial_{t}^2 ( 2 \lambda^2  Re(\bar{z}W)  )\big) z,\;\;\; \lambda^{-2}\Box^{-1}\big( \triangle (\lambda^2 |z|^2  )\big) ( z + W),\;\;\;\;\emph{\text{(quadratic/cubic)}}.
	\end{align}
	We start by clarifying the from of the parametrix for the initial source \eqref{initial source}.
	\subsection{The wave parametrix for $ \Box^{-1} \partial_t^2 \big(\lambda^2W^2\big)$}\label{subsec:initial}
	\begin{comment}
		hence we note the Greens function $  G(r,s)  = C R^{-2} \cdot  1 - C s^{-2} \cdot 1$ implies
		\[
		\lambda^2 \cdot n_{2k}(t,R) =   C R^{-2}\int_{0}^R s^{3} \cdot  e_{2k -1}(s)\;ds  - C \int_{0}^{R} s \cdot e_{2k -1}(s)\;ds 
		\]
		where the latter integrals are replaced by $ \int_R^{\infty}\;dR$  in case $e_{2k -1}(s) = O(s^{-4-})$ and $ e_{2k -1}(s) = O(s^{-1-})$  as $s \to \infty$ respectively. \footnote{For the initial source $\lambda^2 W^2$ we will see this  is only the case for the latter one in the very first step.}
		
		first consider the \emph{elliptic step} changing to $ R = \lambda(t) r$ and solve for 
		\begin{align*}
			\big( \partial_R^2 + \frac{3}{R}\partial_R\big)n_{2k}
		\end{align*}
		
	\end{comment}

	We construct the above wave parametrix as in \cite{KST1, KST2-YM, KST-slow},  in order to  approximately solve 
	\begin{align*}
		&\Box n = \partial_t^2 \big(\lambda^2(t)W^2(R)\big),\;\; n_k^* =  n_1+ \ldots + n_k,
	\end{align*}
	i.e. for the case of 
	\[
	n_0 =  \lambda^2(t) W^2(R) = \frac{\lambda^2(t)}{\big(1+\frac{R^2}{8}\big)^2},
	\]
	which generates the error (and hence the source for $n_1$)
	\[
	e_0 = \Box n_0 - \triangle\Big(\frac{\lambda^2(t)}{\big(1+\frac{R^2}{8}\big)^2}\Big) = -\partial_{t}^2\Big(\frac{\lambda^2(t)}{\big(1+\frac{R^2}{8}\big)^2}\Big). 
	\]
	\;\;\\
	The preceding function admits an absolutely convergent expansion
	
	\[
	e_0 = %t^{-2}\lambda^2\cdot R^{-4} \sum_{k \geq 0} R^{-2k} c_{2k},\;\; R >1,
	t^{-2}\lambda^2\cdot \big( R^{-4} c_{4} +  R^{-6} c_{6}+ R^{-8} c_{8} + \ldots\big) \;\; R >1,
	\]
	and an even Taylor expansion at $ R = 0$
	\[
	e_0 = %t^{-2}\lambda^2\cdot  \sum_{k \geq 0} R^{2k} \tilde{c}_{2k},\;\; R < 1.
	t^{-2}\lambda^2\cdot \big(  \tilde{c}_{0} +  R^{2} \tilde{c}_{2}+ R^{4} \tilde{c}_{4} + \ldots\big)\;\; R < 1. 
	\]
	We now choose the correction $n_1$ according to the elliptic step, i.e.
	\begin{align}\label{first-ell}
		(\partial_R^2 + \frac{3}{R}\partial_R) n_1 = - \lambda^{-2}e_0.
	\end{align}
	Thus the solution reads
	\begin{align*}
		n_1 &= R^{-2}\int_0^R \frac{\lambda^2}{(\lambda t)^2}\frac{s^3}{4}\cdot\frac{1}{(1+\frac{s^2}{8})^2}\,ds+ \int_R^\infty \frac{\lambda^2}{(\lambda t)^2} \frac{s}{4}\frac{1}{(1+\frac{s^2}{8})^2}\,ds\\
		&=\frac{\lambda^2}{(\lambda t)^2}\big(c_{2 1}R^{-2}\log(R) + c_{2}R^{-2} + c_{4}R^{-4}+\ldots\big), 
		%&=\frac{\lambda^2}{(\lambda t)^2}\big(c R^{-2}\log(R) + R^{-2}\sum_{k \geq 0} c_{2k}R^{-2k} \big), 
	\end{align*}
	for $R\gg 1$, where the expansion is absolutely convergent. The error 
	$$e_1 = \Box (n_0 + n_1) - \triangle n_0 =  \Box  n_1 - \partial_t^2 n_0$$
	generated by this ansatz is then of the form 
	\[
	e_1 = - \partial_{t}^2 n_1  = t^{-2}\frac{\lambda^2}{(\lambda t)^2}\big(d_{21}R^{-2}\log(R) + d_{2}R^{-2} + d_{4}R^{-4}+\ldots\big),\;\;\; R \gg1
	%\big(d \cdot  R^{-2}\log(R) + R^{-2}\sum_{k \geq 0} d_{2k}R^{-2k} \big)
	\]
	and a repeat of the \emph{elliptic modifier step} via
	\begin{align}\label{sec-ell}
		&(\partial_R^2 + \frac{3}{R}\partial_R) n_2 = - \lambda^{-2}e_1,\\
		&n_2 = R^{-2}\int_0^R \frac{\lambda^2}{(\lambda t)^4}\frac{s^3}{4} n_1(s)\,ds+ \int_0^R \frac{\lambda^2}{(\lambda t)^4} \frac{s}{4} n_1(s)\,ds
	\end{align}
	is easily inferred to lead to the next correction 
	\begin{align*}
		n_2 = \frac{\lambda^2}{(\lambda t)^4}\big(c_{02}\log^2 R + c_{01}\log R + c_{00}+ c_{2}R^{-2}+ c_{4}R^{-4} + \ldots\big)
	\end{align*}
	in the regime $R\gg 1$ and $ n_2 = \mathcal{O}(R^2)$ where $ 0 < R \ll1$, with absolutely convergent expansions. We note that the coefficients are different form the $n_1$ correction, but are simply denoted $c$ (and $d$ in the error). The error $ e_2 =\square( n_1 + n_2) - \partial_t^2n_0$ generated by this first two steps likewise has the form
	\[
	e_{2} = - \partial_{tt}n_2 = t^{-2}\frac{\lambda^2}{(\lambda t)^4}\big(d_{02}\log^2 R + d_{01}\log R + d_{00}+ d_{2}R^{-2} + d_{4}R^{-4} + \ldots\big)
	\] 
	\\
	At this stage we have to change to the \emph{hyperbolic modifier } in \emph{Step (2)} in order to improve the approximation and eliminate the \emph{leading order} error term  $ e^0_k$ generated, namely 
	\[
	e_2^0 = t^{-2}\frac{\lambda^2}{(\lambda t)^4}\big(d_{02}\log^2 R + d_{01}\log R + d_{00} \big).
	\]
	\;\;\\
	\underline{\emph{The correction $n_3$}}. Considering the error $ e^2_0$  and \eqref{hyperbolic-op} we define as explained above
	\begin{align}\label{Ansatz-hyp}
		n_3(t, R, a) &=~ \frac{\lambda^2}{(\lambda t)^4}\bigg(c_{02}(a)\big(\frac12\log(R^2+1)\big)^2 + c_{01}(a)\frac12\log(R^2+1)  + c_{00}(a) \bigg),
	\end{align}
	in order to `control' the non-decaying terms %as explained above, i.e.
	i.e. using this correction, the `leading order' term  of the subsequent error will reduce to
	\[
	\log^2(R) - \big(\frac12\log(R^2+1)\big)^2 = \mathcal{O}(R^{-2})\log(R) + \mathcal{O}(R^{-4}),
	\]
	which is controlled in $ L^{\infty}_R$ as $ R \to \infty$. Now as indicated, we opt for the constant coefficient source terms in $e^0_2$, i.e.
	\begin{align*}
		L_{\beta }c_{0 j}(a) &=\; \Big((1-a^2)\partial_{a}^2 + 2\big((\beta-1)a + 3a^{-1}\big)\partial_a - \beta^2 + \beta\Big)c_{0 j}(a) = -d_{0 j},\;\; j =0,1,2,
	\end{align*}
	where $a = \frac{r}{t}$ is the hyperbolic variable. With the error terms of commutator, this leads to the system \eqref{main-sys-L-1} - \eqref{main-sys-L-3}, which in our case reads
	\begin{align}\label{first-source}
		L_{\beta_2} c_{02}(a) &= \Big((1-a^2)\partial_{a}^2 + 2\big((\beta_2-1)a + 3a^{-1}\big)\partial_a - \beta^2_2 + \beta_2\Big)c_{02}(a) = - d_{02},\\ \label{second-source}
		L_{\beta_2 }c_{01}(a) &= - d_{01} - 2\big[(-1-2\nu) a c_{02}'(a) - (2\nu -3) c_{02}(a)\\[4pt] \nonumber
		&\;\;\;\;  + \frac{2}{a} \big( c_{02}'(a) + \frac{1}{a} c_{02}(a)\big) + \frac{3}{a^2} c_{02}(a) - 2(\f12 + \nu)^2 c_{02}(a)\big] ,\\[4pt] \label{third-source}
		L_{\beta_2}c_{00}(a) &= - d_{00} - \big[(-1-2\nu) a c_{01}'(a) - (2\nu -3) c_{01}(a)\\[4pt] \nonumber
		&\;\;\;\;  + \frac{2}{a} \big( c_{01}'(a) + \frac{1}{a} c_{01}(a)\big) + \frac{3}{a^2} c_{01}(a) - 2(\f12 + \nu)^2 c_{01}(a)\big] ,\\[4pt]\nonumber
		& \;\;\;\; - 2\big[ \frac{1}{a^2}  c_{02}(a) - (\f12 + \nu)^2 c_{02}(a)\big],
	\end{align}
	where we set $\beta_2 = 2\nu - 3$.\\[3pt]
	Then the solutions and their asymptotics are subsequently inferred by the above description of the problem $ L_{\beta_2} v = f$.\\[5pt]
	\textit{The solutions near $a = 0,1$}. We shall solve \eqref{first-source} while imposing vanishing of $c_{02}(a)$ at $a = 0$, in fact $c_{02}(a) = \mathcal{O}(a^2)$ as $a\rightarrow 0$. To be precise, by  Lemma \ref{FS-null-eins} we obtain a solution near $ a = 0$ of the form 
	\[
	c_{02}(a) = \sum_{j \geq 0} c_j a^{2j +2},~~~0 \leq a \ll1.
	\]
	Again by Lemma \ref{FS-null-eins} and the proper choice of a constants $ c_1, c_2$, we extend this to a solution of \eqref{first-source} on the interval $[0,2]$ of the form 
	\begin{align*}
		%c_{0,2}(a)  &= c_1 \phi^{(1)}_1(a)+ c_2\phi^{(1)}_2(a)  + d_{0 2}\int_{a^*}^a G^{(1)}(a,s) s^3 ( 1 - s^2)^{- \beta - \f12} ~ds,~~ 0< a < 1, \\
		%c_{0,2}(a)  &= c_1 \phi^{(1)}_1(a)+ c_2\phi^{(1)}_2(a)  + d_{02}\int_a^{2 -a^*} G^{(1)}(a,s) s^3 ( s^2 -1)^{- \beta - \f12}~ds,~ ~1 < a \leq 2 ,
		c_{0,2}(a)  &=  c_2\phi^{(1)}_2(a)  + d_{0 2} c_1 \phi^{(1)}_1(a) \int_{a}^1  \phi^{(1)}_2(s) s^3 ( 1 - s^2)^{- \beta - \f12} ~ds\\
		& \hspace{3cm} -  d_{0 2} c_1 \phi^{(1)}_2(a) \int_{a^*}^a \phi^{(1)}_1(s) s^3 ( 1 - s^2)^{- \beta - \f12} ~ds~~ 0< a < 1,\\
		c_{0,2}(a)  &= c_2\phi^{(1)}_2(a)  + d_{0 2} c_1  \phi^{(1)}_1(a)\int_{1}^a  \phi^{(1)}_2(s) s^3 (  s^2-1)^{- \beta - \f12} ~ds\\
		& \hspace{3cm}  -  d_{0 2} c_1  \phi^{(1)}_2(a)\int_{a}^{2- a^*} \phi^{(1)}_1(s) s^3 (  s^2-1)^{- \beta - \f12} ~ds, ~1 < a \leq 2 ,
	\end{align*}
	where $ 0 < a^* \ll1$ is fixed. The coefficients $ c_1, c_2 \in \R$ are the same in both lines and fixed by the initial conditions at $ a^* \ll1$ from the above solution near $ a = 0 $.
	The expansion around the singularity $a = 1$  is observed to be of the form
	\[
	\mathcal{O}(|1-a|^{2}) + (|1-a|)^{\beta + \f32} ( \mathcal{O}(1) + \mathcal{O}(1)\cdot |\log(|1-a|)|) ,~~~~|1-a| \ll1,
	\]
	where the logarithmic factor can only contribute if $\nu > 0$ is rational. More precisely,  near $ a =1$, we expand $\phi_1^{(1)}, \phi_2^{(1)}$ as in Lemma \ref{FS-null-eins} and essentially need to  check the particular solution. Here we note that if $  0 < 1 - a^* \ll 1 $ we infer for $ a^* < a < 1, \beta= \beta_2$ 
	\begin{align*}
		&\phi_1^{(1)}(a) \int _{a^*}^a \phi_2^{(1)}(s) ( 1 - s^2)^{-\beta - \f32 } s^3~ds =~ \sum_{k,j}\mu_k \hat{\mu}_{j} ( 1 - a)^{k} \int_{a^*}^a  ( 1 -s)^{j} ds,\\
		&\phi_2^{(1)}(a) \int _{a^*}^a \phi_1^{(1)}(s) ( 1 - s^2)^{-\beta - \f32 } s^3~ds = \sum_{k,j}\mu_k\hat{\mu}_{j}  ( 1 -a)^{ j + \beta + \frac{3}{2}} \int_{a^*}^a ( 1 - s)^{- \beta -\f32 + k} ds,
	\end{align*}
	if $ \nu > 0 $ is irrational and where we used
	$$ s^3 = ( 1 - (1-s))^3, ~ ( 1 + s)^{- \beta -\f32} \sim \sum_k \binom{- \beta -\f32}{k} 2^{-k}( 1 - s)^k$$ 
	as binomials in an absolute sense and infer the expansion required in $\mathcal{Q}^{\beta}$. In case $ \nu > 0 $ is rational, we consider the latter integral with an additional term
	$$ \int_{a^*}^a ( 1 + s)^{- \beta -\f32} ( 1 -s)^{k -  \beta - \f32} s^3 ds,~~ \int_{a^*}^a \log( 1 -s) ( 1 + s)^{- \beta -\f32} ( 1 -s)^{k} s^3 ds. $$
	They lead to $ \log(1-a) $ factors upon integration by parts. Similarly, we then have a solution on $ (1, 2]$ via $\phi_1^{(1)}, \phi_2^{(1)}$, say
	$$ c^{(3)}_{02}(a) \sim \phi_1^{(1)}(a)c_1 + \phi^{(1)}_2(a)c_2 +  \int_{a}^{2 - a_1} G^{(1)}(a, s) ( 1 - s^2)^{-\beta - \f32 } s^3 ~ds,$$
	where $ 1 < a \leq 2$ and $c^{(3)}_{02}(a)$ has a similar expansion near $ a = 1$. The singular term only appears in the product $ ( |1-a|)^{\beta + \f{3}{2}} \cdot \log( |1 -a|)$ and hence is readily  controlled at $ a = 1$. If we isolate the $q_0^{\pm}$ expression in the definition of $ \mathcal{Q}^{\beta}$ from the left and right, we infer $ \lim_{a \to 1^+} q_0^{+}(a) =  \lim_{a \to 1^-} q_0^{-}(a)$ exists.\\[4pt]%\footnote{Here and in the formula we actually mean  $q_1^{\pm}(a), q_2^{\pm}(a)$  as $ a \to 1^{\pm}$  are analytic on $[0,1],~ [1, \infty]$ respectively.}. We give more details of the calculation below in the general induction step.\\[3pt]
	\begin{comment}	\;\;\\[8pt]
		\underline{\emph{Details for the subsequent corrections:  $n_3,\,n_4$}}. We now briefly recall the steps we have taken above in order to verify  \eqref{n_3}   - \eqref{e-n_4}.  As discussed above we choose
		\begin{align*}
			n_3(t,a,R) = \frac{\lambda^2}{(\lambda t)^4}\big[c^{(3)}_{02}(a)\big(\frac12\log(R^2+1)\big)^2 + c^{(3)}_{01}(a)\frac12\log(R^2+1)  + c^{(3)}_{00}(a)\big]
		\end{align*}
		We thus first solve $ (L_{\beta_2}c^{(3)}_{02})(a) = \tilde{c}^2_{0,2}$ and obtain a solution by Lemma \ref{FS} 
		$$ c^{(3)}_{02}(a) = \sum_{k \geq 1} c_k a^{2k}, ~ 0 < a < a_0$$
		with $c^{(3)}_{02}(0) = (c^{(3)}_{02})'(0) =0.$ Further, we extend this solution to $ [0,2]$ via $\phi_1^{(1)}, \phi_2^{(1)}$ in  Lemma \ref{FS}. This is done by first extending $c^{(3)}_{0,2}$ to $[0,1)$
		$$ c^{(3)}_{02}(a) \sim \phi_1^{(1)}(a)c_1 + \phi^{(1)}_2(a)c_2 +  \int_{a_1}^a G^{(1)}(a, s) ( 1 - s^2)^{-\beta - \f32 } s^3 ~ds,$$
		with a proper choice of $ c_1, c_2$ and $ a_1 < a_0$. 
	\end{comment} 
	\textit{The solution near $a = \infty$}. Now for the behavior of $c_{02}(a)$ for $a\gg1$, we use the fundamental system near $a = +\infty$ in Lemma \ref{FS-infty}.
	For \eqref{first-source} with a constant source,  the solution $c_{02}$ as $ a \gg 1$ will be a linear combination of this fundamental systems and the exact constant solution of the inhomogeneous equation, i.e.  on $ (1, \infty)$ the solutions of \eqref{first-source} have the form
	\[
	c_{02}(a) = \tilde{c}_1\phi^{(\infty)}_1(a) + \tilde{c}_2\phi_2^{(\infty)}(a) + \frac{d_{02}}{\beta - \beta^2}, 
	\]
	with a proper choice of  $ \tilde{c}_1, \tilde{c}_2$ such that $c^{(3)}_{02}$ extends from the $ a \sim 1 $ region.
	%Precisely 
	%	$$ c_{02}(a) \sim \phi_1^{(\infty)}(a)\tilde{c}_1 + \phi^{(\infty)}_2(a)\tilde{c}_2 +   \frac{\tilde{c}^{(2)}_{02}}{\beta^2 - \beta}$$
	Now if $ a > a^*\gg 1 $ we may write $ \phi^{(\infty)}_1(a) = O(a^{\beta }),~ \phi^{(\infty)}_2(a) = O(a^{\beta -1}) $ as absolutely converging series expansions in $ a^{-1}$. Hence we find $ c^{(3)}_{0 2} \in \mathcal{Q}^{\beta_2}_2$ and next we repeat these arguments for $ L_{\beta_2} c^{(3)}_{01} = g(a) $ where $g(a)$ combines the constant $ - \tilde{c}^{(2)}_{01}$ with the additional error of the $ \Box$-commutator.\\[4pt]
	Evidently $ g(a) \in (\mathcal{Q}^{\beta_2}_2)' $ and hence $c^{(3)}_{01} \in \mathcal{Q}^{\beta_2}_2$ by  Lemma \ref{lem:Inhom-L} below. We should however remark that compared to above in the $ a \gg1 $ region we now consider 
	$$ O(a^{\beta})\int_{a^*}^a O(s^{2\beta -1})( 1 - s^2)^{-\beta - \f32 } s^3 ~ds,\;\; O(a^{\beta -1})\int_{a^*}^a O(s^{2\beta})( 1 - s^2)^{-\beta - \f32 } s^3 ~ds,$$
	for  which the first integral term gives the factor $ \log(a)$ (in case $ \nu > 0$ is rational) and we infer $ c^{(3)}_{01} = O(a^{\beta}) \log(a)$.\\[4pt]
	For the last coefficient we combine $ - \tilde{c}^{(2)}_{00}$ with the new $ \Box$-commutator error. We denote this source term by $ \tilde{g}(a)$ and solve $L_{\beta_2} c^{(3)}_{00} = \tilde{g}(a) \in (\mathcal{Q}^{\beta_2}_2)' $ as before by means of  Lemma \ref{lem:Inhom-L}. Now for $ a \gg 1 $ we infer $  c^{(3)}_{00}  = O(a^{\beta})\log^2(a)$ by the precise asymptotic of the $ \tilde{g}(a) $ source term, where we note the $\log(a)$ contributions arise only if $ \nu > 0$ is allowed to be rational.\\[4pt]
	Thus the third  error satisfies
	\[
	e_3 = \Box(n_0 + n_1 +n_2 + n_3) = e_2  - e_{2}^0 + e_{2}^{mod}.
	\]
	which for  $ R \gg1$ large has the form
	\[
	t^{-2} \frac{\lambda^2}{(\lambda t)^4}\bigg( \log(R)\cdot \big[ d_{1, 2}(a) R^{-2} + d_{1, 4}(a) R^{-4} + \dots\big] +  d_{ 2}(a) R^{-2} + d_{4}(a) R^{-4} + \dots \bigg).
	\]
	Thus  having asymptotic expressions in the above S-spaces (c.f. \cite{KST1}),we write
	\[
	e_3 \in t^{-2}\frac{\lambda^2}{(\lambda t)^4}S^0((\log R)R^{-2}, \big( \mathcal{Q}^{\beta}\big)'),\,\beta_2 = 2\nu - 3.
	\]
	Now for the correction $n_4$, we intend to use the elliptic modifier step as before. However, here we need to take care of the regularity loss at $ a =1$ by differentiating the $a$-dependent coefficients. The result of this procedure and  in particular finding the asymptotic  \eqref{n_4} is stated in Lemma \ref{lem:Lemma-n-Anfang-elliptic} and requires to use of Lemma \ref{lem:recoverz1} . The error after calculating another the elliptic \emph{Step (1)} then reads
	\[
	e_4 = \Box(n_0 + n_1 +n_2 + n_3 + n_4) = - \partial_{tt}( n_4),
	\]
	which implies the form \eqref{e-n_4}.\\[10pt]
	
	\underline{\emph{Inhomogeneous equation for $L$}}. The following lemma sums up the above arguments (for defining\\[3pt] $n_3$) in generality and is used in Section \ref{sec:box-main-lemma} below.
	\begin{Lemma} \label{lem:Inhom-L}
		Let $ k,j \geq 2 $ and  $ f \in \partial_a^2\big (\mathcal{Q}^{\beta_j}\big )$. Then
		\begin{align}
			(L_{\beta_k}g)(a) = f(a),~~ g(0) = g'(0) = 0, 
		\end{align}
		has a solution $ g \in \big(\mathcal{Q}^{\beta}\big)'$ for $\beta = \beta_{\max\{k, j\}}$, which is unique on $ [0,1)$. 
	\end{Lemma}
	\begin{proof} This follows by substituting the expansions in Definition~\ref{defn:Qbetal} in the variation of constants formula \eqref{eq:Green1}, \eqref{eq:Green2} to cover the region $a\sim 1$ as well as the corresponding analogues for $a$ near $0$ or $\infty$. We note that for $a<1$, and parameters $b,\tilde{b},\kappa, \gamma$ (with the latter in lieu of $s$) as in the second case in Definition~\ref{defn:Qbetal},  
		\begin{align*}
			&\phi^{(1)}_1(a) \int_{a}^1  \big(1-s\big)^{2\nu b + \tilde{b} + \frac{\kappa}{2}-2}\big|\log(1-s)\big|^\gamma\phi^{(1)}_2(s) s^3 ( 1 - s^2)^{- \beta - \f12} ~ds\\
			& =\big(1-a\big)^{2\nu b + \tilde{b} + \frac{\kappa}{2}-1}\sum_{0\leq \tilde{\gamma}\leq\gamma}C_{b\tilde{b}\kappa,\tilde{\gamma}}(a)\cdot \big|\log(1-a)\big|^{\tilde{\gamma}}
		\end{align*}
		where the functions $C_{b\tilde{b}\kappa,\tilde{\gamma}}$ admit absolutely convergent expansions near $a = 1$ like $\phi^{(1)}_1$.
		\\
		An analogous observation applies to the expression 
		\begin{align*}
			\phi^{(1)}_2(a) \int_{a}^1  \big(1-s\big)^{2\nu b + \tilde{b} + \frac{\kappa}{2}-2}\big|\log(1-s)\big|^\gamma\phi^{(1)}_1(s) s^3 ( 1 - s^2)^{- \beta - \f12} ~ds,
		\end{align*}
		except that in the special situation $2\nu b + \tilde{b} + \frac{\kappa}{2}-2 = \beta-\frac12$, we pick up an additional $\log(1-a)$. In this case, we have 
		\[
		\phi^{(1)}_2(a)\sim \big(1-a\big)^{\beta+\frac32} = \big(1-a\big)^{2\nu b + (\tilde{b}+1) + \frac{\kappa}{2} - 1}, 
		\]
		and so an additional logarithm in $(1-a)$ is consistent with Definition~\ref{defn:Qbetal}. 
		\\
		Using the asymptotic formula \eqref{eq:Greenlargea}, we also easily verify that application of the variation of constants formula does not result in additional logarithms in the regime $a\gg 1$. 
		
	\end{proof}
	\;\;\\
	\underline{\emph{Higher order corrections}}.  At this point, we claim  expressions of the  general correction $n_k$ and corresponding error term $e_k$ by the above procedure of \emph{Step (1)} and \emph{Step (2)}. We now use the notation of the S-spaces, i.e. the above calculations show
	\begin{align}
		&e_0 = - \partial_t^2 ( \lambda^2 W^2(R)) \in t^{-2 } \lambda^{2} S^0(R^{-4})\\
		& n_1 \in \frac{\lambda^2}{(t\lambda)^2} S^2(R^{-2}\log(R) ),~~ t^2 e_1 \in \frac{\lambda^2}{(t\lambda)^2} S^2(R^{-2}\log(R) )\\
		&n_{2} \in \frac{\lambda^2}{(t\lambda)^{4}} S^4(\log^2(R) ),~~ t^2 e_{2}\in \frac{\lambda^2}{(t \lambda)^{4}}S^4(\log^2 R),
	\end{align}
	where the asymptotics required at $ R = 0 $ is an inductive consequence of the even Taylor expansion for $ e_0(t,R) \sim t^{-2}\lambda^2(t) W^2(R) =   t^{-2}\lambda^2(t) O(1) $.
	Following the above calculation for the first hyperbolic modifier \emph{Step (2)} we obtain
	\begin{align}\label{n_3}
		& n_3 \in  \frac{\lambda^2}{(t\lambda)^{4}} S^0(\log^2(R) , \mathcal{Q}_2^{\beta_{2}}),\\
		&t^2 e_3 \in \frac{\lambda^2}{(t \lambda)^{4}} S^0(R^{-2}\log R, \big(\mathcal{Q}^{\beta_{2}}\big)')
	\end{align}
	The next elliptic modifier  \emph{Step (1)} requires to integrate through the $a-$dependence in above space for a \emph{regularity gain at $a=1$}, %\footnote{see Remark ref? in Section ref?}, 
	which we will state and prove in Lemma \ref{lem:recoverz1} of the subsequent Section \ref{sec:elliptic-modifier}.  In particular we then infer 
	\begin{align}\label{n_4}
		&n_{4} \in  \frac{\lambda^2}{(t\lambda)^{6}} S^2(\log^2(R) , \mathcal{Q}_0^{\beta_{2}}),\\\label{e-n_4}
		& t^2 e_{4}\in \frac{\lambda^2}{(t \lambda)^{6}} S^0(\log^2 R, \big(\mathcal{Q}^{\beta_{2}}\big)'').
	\end{align}
	\begin{rem}
		It can be checked that we actually get more regularity because of the second order gain, namely  $(\mathcal{Q}_0^{\beta_{2}})^* = (|1-a|)\mathcal{Q}^{\beta_{2}} $. This propagates through the iteration replacing $Q'$ by $Q$ and $Q''$ by $Q'$. However  it will not be relevant for the main $ z-$ iteration, which is why we neglect this fact.
	\end{rem}
	We now state the main Lemma for the higher order corrections. It is actually a special case of the Corollary~\ref{cor:Boxminuesonefirst}, which will be proved later, and constructs wave parametrices for much more general source terms, which partly arise from the Schr\"odinger contribution. 
	\begin{Lemma}\label{lem:general-initial}
		There holds for $ l \geq 3 $
		%	\begin{align}
			%	& n_{2k -1} \in \frac{\lambda^2}{(t\lambda)^{4k -2}} IS^2(R^{-2}\log(R) , \mathcal{Q}^{(\beta_{2k -1})})\\
			%	& t^2e_{2k -1}\in \frac{\lambda^2}{(t \lambda)^{4k-2}}IS^1(R^{-2}\log R, \mathcal{Q}^{\beta_{2k}}),\\
			%	& n_{2k} \in \frac{\lambda^2}{(t\lambda)^{2k}} IS^2(\log^2(R) , \mathcal{Q}^{(\beta_{2k})}),\\
			%	& t^2e_{2k}\in \frac{\lambda^2}{(t \lambda)^{2k}}IS^1(\log^2 R, \big(\mathcal{Q}^{\beta_{2k}}\big)'),
			%	\end{align}
		\begin{align} \label{odd-n}
			& n_{2l -1} \in \frac{\lambda^2}{(t\lambda)^{2l}} S^0(\log^{l-1}(R) , (\mathcal{Q}_2^{\beta_{l}})')\\ 
			& t^2e_{2l -1}\in \frac{\lambda^2}{(t \lambda)^{2l}}S^0(R^{-2}\log^{l-1}( R), \big(\mathcal{Q}^{\beta_{l}}\big)''), \label{odd-e}\\
			& n_{2l} \in \frac{\lambda^2}{(t\lambda)^{2l +2}} S^2(\log^l(R) , \mathcal{Q}_0^{\beta_{l}}), \label{even-n}\\
			& t^2e_{2l}\in \frac{\lambda^2}{(t \lambda)^{2l + 2}}S^0(\log^l (R), \big(\mathcal{Q}^{\beta_{l}}\big)''), \label{even-e}
		\end{align}
		
	\end{Lemma}
	We note that the first  instance for the space $\big (\mathcal{Q}^{\beta}\big )'$ is the corrected error $e_{2k -1}$ for $k = 2$. This is observed since $(\mathcal{Q}^{\beta})' \backslash \mathcal{Q}^{\beta}$ must consist of functions with a loss of regularity at $ a_0 = 1 $ that appear in the commutator coefficients in the $ R \gg 1 $ asymptotics through 
	$$ a \partial_a =  \partial_a - (1 - a ) \partial_a ,~~a^{-1} \partial_a = \sum_{k \geq 0} (1 -a)^k \partial_a,$$
	This is similarly seen in any odd error correction $ e_{2k -1}$ within the second modifier stage. \\[2pt]
	Since the subsequent error $e_{2k}$  then involves two orders of temporal derivatives $\partial_{t}^2$, the coefficients will equally be in $ \big(\mathcal{Q}^{\beta}\big)''$ by the application of
	\[
	a^2\partial_a^2 :  \mathcal{Q}^{\beta}  \to  (\mathcal{Q}^{\beta}\big )''.
	\]
	%The $ (\mathcal{Q}^{\beta}\big )^*$ coefficients are here observed by the gain of two order of regularity in the $ a = 1$ regime.
	This is true for any error  $ e_{2k},~ k \geq 2$ generated in the elliptic modifier stage after we pass the second stage for the first time.\\[13pt]
\subsection{Elliptic modifier Steps}\label{sec:elliptic-modifier}
%\emph{Elliptic modifier for $z$}. 
In this section, we give the main Proposition for the \emph{elliptic modifier Step (1)} in the above $\Box^{-1}$ parametrix and likewise for integrating the Schr\"odinger operator on the left of \eqref{main-eq-z} in the main $z-$iteration defined in Section \ref{subsec:inductive-ite-z} below.\\[3pt]%We now treat the various terms in the equation \eqref{main-eq-z} for $z $. %(on p. 16). 
%First we note, up to passing to a subsequence, we infer  from the last section ref? 
%\[
%\Box(\sum_{j \geq 1}^Nn_j )= O(t^{N}) + \partial_t^2(\lambda^2 W^2),~~ N \in \N.
%\]
In order to treat the latter we split into real and imaginary parts $ z = z^{(+)} +  i z^{(-)}$ and denote the linearized operator in \eqref{main-eq-z} by
\begin{align}\label{Lin-op}
	&\mathcal{L}_{+} = - \triangle - 3 W^2 = - \partial_R^2 - \frac{3}{R}\partial_R - 3 (1 + \frac{R^2}{8})^{-2},\\
	& \mathcal{L}_{-} = - \triangle -  W^2  = - \partial_R^2 - \frac{3}{R}\partial_R - (1 + \frac{R^2}{8})^{-2}.
\end{align}
%By invariance under scaling and phase shift, it is easily observed that
%the  functions
% \[ W_1(R) = (1 + R\partial_R)W(R),~~W(R)\] 
% are resonances at zero level energy for the operator $ \mathcal{L}_{+}, \mathcal{L}_{-}$ respectively.\\
The operators $\mathcal{L}_{\pm}$ clearly have solutions $ \Lambda W(R), W(R)$, to be precise
\begin{align*}
	&\Lambda W =  \frac{d}{d \delta}\big( \delta W(\delta R) \big)_{|_{\delta =1}} =  \big(1 + R \partial_R\big)W(R),\\[3pt]
	&W(R) = \text{Im}(i W(R)),\;\; iW(R) = \frac{d}{d \theta}  \big(e^{i \theta}W(R)\big),
\end{align*}
are solutions due to the scaling and phase translation invariance of the cubic 4D NLS.
These are now complemented to  fundamental systems for $ \mathcal{L}_{\pm}$ given in the following Lemma, c.f. \cite[Lemma 2.8]{schmid}.
\begin{Lemma}\label{lem:FS-schrod-inner} We have the fundamental solutions $\{\Phi_{\pm}, \Theta_{\pm} \}$ for $\mathcal{L}_{\pm}$ given by
	\begin{align*}
		&\Phi_-(R) = W(R),\;\;\;\Theta_-(R) =  W(R) \big(Q_-(R) +  \log(R)),\\
		&\Phi_+(R) =  \Lambda W(R),\;\;\Theta_+(R) = W^{4}(R) Q_+(R)+ \Lambda W(R) \log(R),
	\end{align*}
	where $ P_-(R) = R^2 \cdot Q_-(R), P_+(R) = R^2 \cdot Q_+(R)  $  are  even polynomials with $\deg(P_- ) = 4$, $P_{\pm}(0) \neq 0$ and  $ \deg(P_+) = 6$. % and $P_+(0) \neq 0$.
\end{Lemma}
\begin{proof}
	This is seen from writing $ \Theta_-(R)  =  \zeta(R) W(R) $ with 
	\[
	W(R) \zeta_{RR} + \frac{3 W(R)}{R}\zeta_R  + 2 W_R (R)\zeta_R = 0,
	\]
	whence 
	\[
	\zeta_R = R^{-3} W^{-2}(R) = \frac{(1 + \frac{R^2}{8})^2}{R^3}.
	\]
	Similarly, we inspect the fundamental solutions of $\mathcal{L}_+$ using $ \Theta_+(R) = \gamma(R)\Phi_+(R)$, then
	\[
	\Phi_+(R) \gamma_{RR} + \frac{3 \Phi_+(R)}{R}\gamma_R  + 2 (\Phi_+)_R \gamma_R = 0,
	\]
	whence $
	\gamma_R = R^{-3}\Phi_+(R)^{-2} $. 
\end{proof}
\begin{rem}  It is easily verified that we may choose  constants $ \gamma_{\pm} \in \R $ and 
	\[
	P_-(R) = \gamma_-(R^4 -64),~~~P_+(R) =  \gamma_{+}(- 136 R^2+ \f{R^6}{8} -  P_-(R)). 
	\]
\end{rem}
The Wronskians are  then calculated
\begin{align}
	W_{-}(R) &= \Phi_-'(R) \Theta_-(R) - \Phi_-(R) \Theta_-'(R) = \zeta'(R) W^2(R) = R^{-3},\;\; R \gg1\\[3pt]
	W_{+}(R) &= \Phi_+'(R) \Theta_+(R) - \Phi_+(R) \Theta_+'(R) = \gamma'(R) \Phi_+(R) = R^{-3},\;\; R \gg 1.
\end{align}
The particular solutions of the inhomogeneous equation 
\begin{align*}%\label{Inhom-Lin-eq}
	(\mathcal{L}_{\pm} v_{\pm})(R)= f (R)
\end{align*} 
are thus given by
\begin{align}\label{inhomogen-sol}
	v_{\pm}(R) = \f12 \Phi_{\pm}(R)\int_0^R s^3 \Theta_{\pm}(s) f(s)~ds - \f12\Theta_{\pm}(R) \int_0^R s^3 \Phi_{\pm}(s) f(s)~ds,
\end{align}
provided the latter integral converges, i.e. if $ f = O(s)$ as $ s \to 0$. If this is not the case we correct the initial values and replace $\int_0^R~ds$ by $ \int_R^C~ds $ for $ C \in (0, \infty]$, where $ C = \infty$ may  only be chosen if the indefinite integral converges.\\[4pt]
\emph{Recall}. For the sake of simplicity, we use  $ f(R) = \mathcal{O}(R^m)$ for $m \in \Z$, if $ R^{-m}\cdot f(R)$ has an absolute  expansion
\[
c_0 + c_2 \cdot R^{-2} + c_4 \cdot R^{-4} + c_6 \cdot R^{-6} +  \dots,\;\;\text{at}\; R \gg 1,
\] 
and a similar notation is used in the region $0 \leq R \ll 1$.
\begin{rem} \label{rem-FS-pureR}
	(i)~The definitions of $\Theta_{\pm}(R)$ are such that
	\begin{align*}
		\Theta_{\pm}(R)  = \Phi_{\pm}(R) \log(R) + g_{\pm}(R),
	\end{align*} 
	where $ g_{\pm}$ are rational functions of $R$. Thus, if $ f_{\pm}(R)$ has an even Taylor expansion of order $ m \geq 1$ at $R = 0$, the solution $ v_{\pm}(R)$ has an even Taylor expansion of order $ m + 2$ since the logarithmic terms cancel out. \\[2pt]
	(ii)~Note that we have the asymptotics
	\begin{align*}
		&\Phi_{\pm}(R) = O(R^{-2}), \, \, \Theta_{\pm}(R) = O(1) + O(R^{-2}) \log(R),~R \gg1\\[3pt]
		& \Phi_{\pm}(R) = O(1), \,\,  \Theta_{\pm}(R) = O(R^{-2}) + O(1) \log(R),~R \ll 1,
	\end{align*}
	thus,  in particular, we treat the calculation of $ v_{\pm}$ schematically the same in what follows and we simply write $ \Phi, \Theta$ for the fundamental solutions.
\end{rem}
We state the first small Lemma on the  (pure) $ R-$dependent inhomogeneous equations.
\begin{Lemma}\label{lem:Lemma-n-Anfang-elliptic}
	%	Let  $ f \in S^0( R^{-2} \log^{k-1}(R), (\mathcal{Q}^{\beta_k})') $, then 
	%	$$\partial_R^2v + \f{3}{R}\partial_Rv = f,~~v(0)  = \partial_Rv(0) = 0,$$
	%	has a unique solution $ v \in S^2(\log^k(R), (\mathcal{Q}^{\beta_k}_0)^*)$.
	Let  $ f \in S^m( R^{2k} \log^{l}(R)) $, where $  l \in \Z_{\geq0}$, then we have 
	\begin{align*}
		&\Phi(R)\int_0^R W^{-1}(s) \Theta(s) f(s)\;ds - \Theta(R)\int_0^R W^{-1}(s) \Phi(s)  f(s)\; ds\\[4pt] % \in S^{m+2}(R^{2k+2} \log^l(R))
		&\hspace{3cm} \in  \begin{cases}
			S^{m+2}(R^{2k +2} \log^{l+1}(R)) & k = -1,-2\\[3pt]
			S^{m+2}(R^{2k+2} \log^l(R)) & \;  k \geq 0.
		\end{cases}
	\end{align*}
	Similarly we obtain
	\begin{align*}
		R^{-2}\int_0^R s^3 f(s)\;ds - \int_0^R s  f(s)\; ds  \in
		\begin{cases}
			S^{m+2}(R^{2k +2} \log^{l+1}(R)) & k = -1,-2\\[3pt]
			S^{m+2}(R^{2k+2} \log^l(R)) & \;  k \geq 0.
		\end{cases}
	\end{align*}
	Further, in both cases the variation of constants formula belongs to $ S^{m+2}(R^{-2} \log^l(R)) $ in case $ k < -2$.
\end{Lemma}
%if $k \neq -1,-2$ or else $S^{m+2}(R^{2k +2} \log^{l+1}(R))$. 
\begin{proof} We start by  splitting $ \int_0^R\;ds = \int_0^{R_*}\;ds + \int_{R_*}^R\;ds$ where $ R \geq R_{*}\gg1 $ is large so we can use the expressions in remark \ref{rem-FS-pureR}. The former gives fundamental asymptotics and for the latter  we need to consider
	\begin{align}\label{asympt-int}
		\mathcal{O}(R^{-2})\int_0^R \mathcal{O}(s^3)& (\mathcal{O}(1) + \mathcal{O}(s^{-2}) \log(s)) f(s)\;ds\\ \nonumber
		& -(\mathcal{O}(1) + \mathcal{O}(R^{-2}) \log(R))\int_0^R \mathcal{O}(s) f(s)\; ds.
	\end{align}
	Thus setting $ f(R) =  \sum_{b \geq 0} R^{2s -2b} \log^{\tilde{l}}(s)$ calculating \eqref{asympt-int} with the absolutely converging expansions shows the claim. Clearly in case  $ f(s) = s^{-2} \log(s),\;f(s) = s^{-4} \log(s)$ the logarithmic power needs to be raised in the leading order term. We note if $ f(R) =  R^{2k} \log^{l}(R)$ then the terms involving $R \mapsto \log(R)$ in the fundamental solution, have the form 
	\begin{align*}
		\mathcal{O}(R^{-2})\int_{R_{*}}^R  \mathcal{O}(s) \log(s)) s^{2k} \log^l(s)\;ds =& \mathcal{O}(R^{-2}) \log(R) \int_{R_{*}}^R \mathcal{O}(s) s^{2k} \log^l(s)\; ds\\
		&= \log^{l+1}(R) \mathcal{O}(R^{2k}),
	\end{align*}
	which is consistent with the claim by Definition of the S-space. We actually observe a cancellation \footnote{c.f. the  calculation in the beginning of the subsequent self-similar region} of these terms by writing 
	$
	\Theta(R) = \mathcal{O}(1) + \log(R)\Phi(R)
	$
	and hence
	\begin{align*}
		&\Phi(R)\int_{R_{*}}^R W^{-1}(s)\Phi(s) \log(s)) f(s)\;ds - \Phi(R)\log(R) \int_{R_{*}}^R W^{-1}(s) \Phi(s) f(s)\; ds\\
		&=\; \Phi(R) F(s)\log(s)|_{R_{*}}^R - \Phi(R) \int_{R_{*}}^R F(s)\; s^{-1}\;ds - \log(R) \Phi(R) F(s)|_{R_{*}}^R\\
		&= \;  \Phi(R) \big(C_{R_0}  + \tilde{C}(R_0) \log(R) -   \int_{R_{*}}^R F(s)\; s^{-1}\;ds \big),
	\end{align*}
	where $  F(s)$ is a primitive of $ s \mapsto W^{-1}(s) \Phi(s) f(s)$. The expansion at $ R = 0$ is similar where we use this cancellation integrating $ \int_0^R ds$ which cancels  out the $\log(R) $ factors entirely. Finally the second claim of the theorem is implied by the same arguments.
\end{proof}
\;\;\\
Before we define the $z-$iteration for  approximately solving \eqref{main-eq-z}, we state the main Lemma for the iterative corrections. Note in the following  Lemma the time factors $t^{2\nu k} \cdot (t\lambda)^{-2l}$ can be omitted and the statement holds for any suitable combination of $(m, k, b, l)$.

\begin{Lemma}\label{lem:recoverz1} Let $\Phi(R), \Theta(R)$ be the above fundamental system for $\mathcal{L}_{\pm}u = 0$ and further let%, where $\mathcal{L} = -\triangle - W^2$ or $\mathcal{L} = -\triangle -3W^2$, and we restrict to radial functions on $\R^4$. Precisely for $ -\triangle - W^2$ 
	%	we let 
	%	\begin{align*}
		%	\phi(R) = \frac{1}{1+\frac{R^2}{8}},\,\theta(R) = \frac{\frac{R^2}{128} + \frac14\log R - \frac12 R^{-2}}{1+\frac{R^2}{8}},
		%	\end{align*}
	%	while for $ -\triangle - 3W^2$, we let 
	%	\begin{align*}
		%	&\phi(R) = \frac{1-\frac{R^2}{8}}{\big(1+\frac{R^2}{8}\big)^2},\,\theta(R) = \gamma(R)\cdot \frac{1-\frac{R^2}{8}}{\big(1+\frac{R^2}{8}\big)^2},\\
		%	&\gamma(R) = -\int_R^1s^{-3}\phi^{-2}(s)\,ds
		%	\end{align*}
	\begin{align} \label{das-space1}
		&f \in  \frac{t^{2\nu k}}{(t\lambda)^{2l}}S^m\big(R^{2k-2}\log^{b}(R), (\mathcal{Q}^{\beta_l})'\big),\,k \geq -1,\; b\geq 0,\, l\geq 1,
	\end{align}
	then there holds for $ k \geq 1$
	\begin{align}\label{das-space2}
		&\Phi(R)\cdot\int_0^R s^3 \cdot \Theta(s) f(t,s)\,ds - \Theta(R)\cdot\int_0^R s^3 \cdot \Phi(s) f(t,s)\,ds\\ \nonumber&\hspace{4cm}\in \frac{t^{2\nu k}}{(t\lambda)^{2l}}S^{m+2}\big(R^{2k}\log^{b}(R), \mathcal{Q}^{\beta_l}\big).
	\end{align}
	Likewise if 
	\begin{align} \label{das-space-3}
		&f \in  \frac{t^{2\nu k}}{(t\lambda)^{2l +2}}S^m\big(R^{2k-2}\log^{b}(R), \mathcal{Q}^{\beta_l}\big),\,k \geq -1,\; b\geq 0,\, l\geq 1,
	\end{align}
	then we have
	\begin{align} \label{das-space-33}
		&\Phi(R)\cdot\int_0^R s^3 \cdot \Theta(s) f(t,s)\,ds - \Theta(R)\cdot\int_0^R s^3 \cdot  \Theta(s) f(t,s)\,ds\\ \nonumber&\hspace{4cm}\in \frac{t^{2\nu k}}{(t\lambda)^{2l +2}}S^{m+2}\big(R^{2k}\log^{b}(R), \mathcal{Q}^{\beta_l}\big),\;\; k \geq 1.
	\end{align}
	In both of the cases if $ k = -1,0$ the statement holds if we  replace $\log^b(R)$ by $\log^{b+1}(R)$ in \eqref{das-space2} and \eqref{das-space-33} for the spaces on the right. 
	%	If $k\geq 1$, we can also replace the right hand side by 
	%	\[
	%	\frac{t^{2\nu k}}{(t\lambda)^{2l}}S^{m+2}\big(R^{2k-2}\log^{2l+k}R, \mathcal{Q}^{\beta_l}\big)
	%	\]
	The same result holds if we replace $(\mathcal{Q}^{\beta_l})'$ by $(\mathcal{Q}^{\beta_l})''$ in \eqref{das-space1} or the variation of constants formula in \eqref{das-space2} by 
	\begin{align*}
		R^{-2}\int_0^R \frac{s^3}{4}\cdot \tilde{f}(t, s)\,ds - \int_0^R \frac{s}{4}\tilde{f}(t, s)\,ds,\;\;\; \tilde{f}(t,s) = t^{-2}\lambda^2(t) f(t,s).
	\end{align*}
\end{Lemma}
\begin{proof} We start with  \eqref{das-space2} and  split into a number of cases depending on the size of $a = \frac{r}{t} = \frac{R}{\lambda t}$. Let us further note we omit the time factors of the form $ \lambda^2 \cdot \frac{t^{2\nu k}}{(t \lambda)^{2l}}$ since we do not integrate through $ t > 0$.
	According to the definition of the S-space 
	\[
	S^m\big(R^{2k-2}\log^b(R),\,(\mathcal{Q}^{\beta_l})'\big),
	\]
	we first consider the left region $R = O(1)$.\\[5pt]
	{\it{\underline{Case \rom{1}}: Bounded $R > 0$}}. To be precise we select $R \leq R_0$ as in the first item of Definition \ref{defn:SQbetal} and  therefore we can substitute the expansion in the first item for $f$ and  arrive at the integrals
	\begin{align*}
		\chi\{a \lesssim1, 1-a \gtrsim 1\}
		&\Phi(R)\cdot \int_0^R\big(\sum_{k',l',p\geq 0} W^{-1}(s)\cdot e_{k'l'p}(s)\cdot \frac{t^p}{(t\lambda)^{2l'}}%t^{\nu l' + p} 
		\big(\frac{s}{(t\lambda)}\big)^{2k'}\big)\cdot \Theta(s) \,ds,\\
		\chi\{a \lesssim1,1-a \gtrsim 1\}
		&\Theta(R)\cdot \int_0^R\big(\sum_{k',l',p\geq 0} W^{-1}(s)\cdot e_{k'l'p}(s)\cdot \frac{t^p}{(t\lambda)^{2l'}}\big(\frac{s}{(t\lambda)}\big)^{2k'}\big)\cdot \Phi(s) \,ds,
	\end{align*}
	%where we have omitted the factor $\lambda^2\cdot\frac{t^{2\nu k}}{(t\lambda)^{2l+2}}$. 
	where $ W^{-1}(s) = s^3$ and the sums converge uniformly and absolutely for $t > 0$ sufficiently small (depending on the constants in Definition~\ref{defn:SQbetal}), which justifies exchange of summation and integration. Note that since $ a = R (t \lambda)^{-1} \leq  R_0 t^{\nu - \f12} $ we assume $0<a\lesssim 1, 1-a\gtrsim 1$ if $ t > 0$ is small enough (depending only on $R_0$). If we then set 
	\begin{align*}
		&\tilde{e}_{k'l'p}(R): = R^{-2k'}\Phi(R)\cdot \int_0^R s^3\cdot e_{k'l'p}(s)\cdot  \Theta(s) s^{2k'}\,ds,\\&\tilde{\tilde{e}}_{k'l'p}(R): = R^{-2k'}\Theta(R)\cdot \int_0^R s^3\cdot e_{k'l'p}(s)\cdot  \Phi(s) s^{2k'}\,ds,
	\end{align*}
	we directly  infer the bound
	\begin{align*}
		\big|\tilde{e}_{k'l'p}(R)\big| + \big|\tilde{\tilde{e}}_{k'l'p}(R)\big|\leq (C')^{k'+l'+p}
	\end{align*}
	for some $C'\in \R_+$ only depending on $ R _0 >0, C_1(R_0) $ from Definition \ref{defn:SQbetal}  and we can write  %(for $ a\lesssim 1 , 1-a \gtrsim 1$)
	\begin{align}
		&\Phi(R)\cdot \int_0^R\big(\sum_{k',l',p\geq 0} s^3\cdot e_{k'l'p}(s)\cdot \frac{t^p}{(t\lambda)^{2l'}} \big(\frac{s}{(t\lambda)}\big)^{2k'}\big)\cdot \Theta(s) \,ds\\ \nonumber
		& = \sum_{k',l',p}\tilde{e}_{k'l'p}(R)\cdot a^{2k'}\cdot \frac{t^p}{(t\lambda)^{2l'}},\\
		&\Theta(R)\cdot \int_0^R\big(\sum_{k',l',p\geq 0} s^3\cdot e_{k'l'p}(s)\cdot \frac{t^p}{(t\lambda)^{2l'}} \big(\frac{s}{(t\lambda)}\big)^{2k'}\big)\cdot \Phi(s) \,ds\\ \nonumber
		& = \sum_{k',l',p}\tilde{\tilde{e}}_{k'l'p}(R)\cdot a^{2k'}\cdot \frac{t^p}{(t\lambda)^{2l'}}.
	\end{align}
	The convergence is absolute and uniform for $R  \leq R_0, t\ll 1$ in both integrals. The analyticity in the variable $R$ for $R \ll1 $ small enough follows since  then $e_{k'l'p}(s)$ is analytic.  Further with the Taylor expansion from the second item in Definition \ref{defn:SQbetal}, i.e. in an absolute sense
	\begin{align}\label{Taylor}
		f(t, s) = s^{m}\cdot \sum_{k' \geq 0}  s^{2k'} c_{k'}(t),\;\; 0 \leq  s \ll1,
	\end{align}
	for $ t > 0$ fixed, we observe a converging even Taylor expansion of order $ m+2$ for the  integral from calculating
	\begin{align*}
		&  \sum_{k' \geq 0}c_{k'}(t)\big(\Phi(R)\cdot \int_0^R \Theta(s)s^{2k' + m +3} \,ds - \Theta(R)\cdot \int_0^R \Phi(s)s^{2k' + m +3} \,ds\big),
	\end{align*}
	with $ 0 \leq R \ll1$. Note that here and likewise for $\tilde{e}_{k'l'p}, \tilde{\tilde{e}}_{k'l'p}$ (concerning their analyticity), the $\log(R)$ factor cancels out of the variation of constants formula, since the integrands  vanish fast enough at $R =0$.
	\\[8pt]
	{\it{\underline{Case \rom{2}}: $R \gg 1 $ \&  \;$ 0 < a \ll1$.}} Here we have the series expansion in the third item of Definition \ref{defn:SQbetal}. Inserting this into the variation of constants formula, we split $ \int_0^R\;ds = \int_0^{R_*} \;ds+ \int_{R_*}^R\;ds $ for $ R \geq R_* \gg1 $ and thus we arrive at the following integrals
	\begin{align}\nonumber
		&\Phi(R)\cdot \int_{R_*}^R \Theta(s) s^3\cdot \sum_{p, r, k' \geq 0} \sum_{0\leq j\leq b +  2r }\frac{t^{k'}}{(t \lambda)^{2p}} \cdot q_{p r j k'}(\tilde{a})\cdot s^{ 2k-2 - 2r} \cdot( \log(s))^j\,ds\\  \label{die-exp-first}
		& + \Phi(R)\cdot \int_0^{R_*}\big(\sum_{k',l',p\geq 0} s^3 \cdot e_{k'l'p}(s)\cdot  \frac{t^p}{(t\lambda)^{2l'}} \big(\frac{s}{(t\lambda)}\big)^{2k'}\big)\cdot \Theta(s)\,ds,\\[8pt]\nonumber
		&\Theta(R)\cdot \int_{R_*}^R \Phi(s) s^3\cdot \sum_{p, r, k' \geq 0} \sum_{0\leq j\leq b +  2r }\frac{t^{k'}}{(t \lambda)^{2p}} \cdot q_{p r j k'}(\tilde{a})\cdot s^{ 2k-2 - 2r} \cdot( \log(s))^j\,ds\\  \label{die-exp-second}
		& + \Theta(R)\cdot \int_0^{R_*}\big(\sum_{k',l',p\geq 0} s^3 \cdot e_{k'l'p}(s)\cdot  \frac{t^p}{(t\lambda)^{2l'}} \big(\frac{s}{(t\lambda)}\big)^{2k'}\big)\cdot \Phi(s) \,ds,
	\end{align}
	where we write $\tilde{a}: = \frac{s}{\lambda t} \leq a \ll1$ and we again used the expansion for the bounded part.  We recall as $R \gg1$ 
	\[
	\Phi(R) = \mathcal{O}(R^{-2}),\;\;\;	\Theta(R) = \mathcal{O}(1) + \log(R) \mathcal{O}(R^{-2}),
	\]
	thus the integrals in the second and the fourth lines  of \eqref{die-exp-first} - \eqref{die-exp-second} read
	\begin{align*}
		&\mathcal{O}(R^{-2})\sum_{k',l',p\geq 0} \frac{t^p}{(t\lambda)^{2l' + 2k'}}  \int_0^{R_*} s^{3 + 2k'} (\mathcal{O}(1) + \log(s) \mathcal{O}(s^{-2}))e_{k'l'p}(s) \,ds,\\
		&(\mathcal{O}(1) + \log(R) \mathcal{O}(R^{-2}))\sum_{k',l',p\geq 0} \frac{t^p}{(t\lambda)^{2l' + 2k'}}  \int_0^{R_*} \mathcal{O}(s^{1 + 2k'}) e_{k'l'p}(s) \,ds.
	\end{align*}
	Now we use the bounds on the coefficients $e_{k'l'p}$ in the first item of Definition \ref{defn:SQbetal}, from which we directly verify the representation as asserted in the third point of the definition with coefficients $q_{p r j k'}(a)$ which are $a-$constant , i. e. independent of $a = \frac{r}{t}$. \\[3pt]
	Let us now deal with the first and third  integrals of \eqref{die-exp-first} - \eqref{die-exp-second}, hence we insert the expansion for the coefficients $q_{p r j k'}(\tilde{a}) = q_{p r j k'}\big(\frac{s}{(t \lambda)}\big)$ from the definition of the S-space. Since  the series 
	\begin{align*}
		\sum_{p, r, k' \geq 0} \sum_{0\leq j\leq b +  2r }\frac{t^{k'}}{(t \lambda)^{2p}} \cdot q_{p r j k'}(\tilde{a})\cdot s^{ 2k-2 - 2r} \cdot( \log(s))^j,\,\;\;\tilde{a}: = \frac{s}{\lambda t},
	\end{align*}
	converges absolutely and uniformly provided $s\gg1,\,t\ll 1$ (the constants implicitly dependent on  $\rho_0 > 0$ in the third item of Definition \ref{defn:SQbetal}), we may interchange summation and integration, and infer
	\begin{align}\label{eq:all11}
		\;\;\;&\sum_{p, r, k' \geq 0} \sum_{0\leq j\leq b +  2r } \sum_{\tilde{b} \geq 0}\frac{t^{k'}}{(t \lambda)^{2p + 2\tilde{b}}} q_{p r j \tilde{b} k' }\times\\ \nonumber
		&\hspace{1cm}  \times \mathcal{O}(R^{-2}) \int_{R_*}^R s^3 (\mathcal{O}(1) + \log(s) \mathcal{O}(s^{-2})) s^{2\tilde{b}+2k-2- 2r}\cdot\log^j(s)\,ds,\\[4pt] \label{eq:ll12}
		\;\;\;&\sum_{p, r, k' \geq 0} \sum_{0\leq j\leq b +  2r } \sum_{\tilde{b} \geq 0}\frac{t^{k'}}{(t \lambda)^{2p + 2\tilde{b}}} q_{p r j \tilde{b} k' }\times\\ \nonumber
		&\hspace{1cm}  \times (\mathcal{O}(1) + \log(R) \mathcal{O}(R^{-2})) \int_{R_*}^R \mathcal{O}(s) s^{2\tilde{b}+2k-2- 2r}\cdot\log^j(s)\,ds.
	\end{align}
	These integrals and their difference, the variation of constants formula, are calculated with the arguments in the proof of Lemma \ref{lem:Lemma-n-Anfang-elliptic}. In particular we note the $\log(s), \log(R)$ dependence of the fundamental solutions and the increase of the logarithmic power (to leading order) in case $ k =-1, 0$. To be precise, since $ R^{2 \tilde{b}} (t \lambda)^{- 2\tilde{b}} = a^{2\tilde{b}}$, we may simply consider boundedness of the terms
	\begin{align} \label{das-eins}
		R^{-2\tilde{b}- 2k  + 2r}\log^{-j} (R) \cdot\Phi(R)\cdot\int_{R_*}^R W^{-1}(s)\cdot \Theta(s)\cdot s^{2\tilde{b}+2k-2- 2r}\cdot\log^j(s)\,ds,\\ \label{das-zwei}
		R^{-2\tilde{b}- 2k  + 2r}\log^{-j} (R) \cdot\Theta(R)\cdot\int_{R_*}^R W^{-1}(s)\cdot \Phi(s)\cdot s^{2\tilde{b}+2k-2- 2r}\cdot\log^j(s)\,ds,
	\end{align}
	which have uniformly bounded  asymptotic for $R\gg 1$ if $ 2\tilde{b}+2k-2- 2r >-2 $, and in fact  then 
	\begin{align*}
		&\Phi(R)\cdot\int_{R_*}^R W^{-1}(s)  \Theta(s) \cdot s^{2\tilde{b}+2k-2- 2r} \log^j(s)\,ds \in S^0(R^{2\tilde{b}+2k- 2r}\log^j (R) ),\\
		&\Theta(R)\cdot\int_{R_*}^R W^{-1}(s)  \Phi(s) \cdot s^{2\tilde{b}+2k-2- 2r} \log^j(s)\,ds \in S^0(R^{2\tilde{b}+2k- 2r}\log^j (R) )
	\end{align*}
	for $R\gg 1$ such that \eqref{das-eins} - \eqref{das-zwei} under these assumption and lies in $S^0(1)$. Next, note that likewise
	\begin{align*}
		&R^{- 2\tilde{b} -2k + 2r}\log^{-j} (R) \cdot\phi(R)\cdot\int_{R_*}^R W^{-1}(s) \Theta(s)\cdot s^{2\tilde{b}+2k -2 -2r } \log^j (s)\,ds\\
		&R^{- 2\tilde{b} -2k + 2r}\log^{-j-1} (R) \cdot\Theta (R)\cdot\int_{R_*}^R W^{-1}(s) \Phi(s)\cdot s^{2\tilde{b}+2k -2 -2r } \log^j (s)\,ds
	\end{align*}
	is uniformly bounded for $R\gg 1$, $2\tilde{b}+2k -2 -2r = -2$ and in fact they lie again in $S^0(1)$. This of course increases the leading order logarithmic growth $ j \mapsto j+1$ in case $ k = 0$. in the latter line. Finally, for $2\tilde{b}+2k -2 -2r < -2$ we have that 
	\begin{align} \label{das-drei}
		\int_{R_*}^R W^{-1}(s)\Theta(s)\cdot s^{2\tilde{b}+2k-2 -2r}\cdot\log^j(s) \,ds = \eta_{\tilde{b}, k , r, j}(R) -  \eta^0_{\tilde{b}, k , r, j},\\
		\int_{R_*}^R W^{-1}(s)\Phi(s)\cdot s^{2\tilde{b}+2k-2 -2r}\cdot\log^j(s) \,ds = \zeta_{\tilde{b}, k , r, j}(R) -  \zeta^0_{\tilde{b}, k , r, j}, 
	\end{align}
	where the latter  $\zeta^0_{\tilde{b}, k , r, j}, \eta^0_{\tilde{b}, k , r, j} $  are constants  and bounded uniformly in the parameters  $ \tilde{b}, k , r, j$  with $ j \leq b + 2r$. Further the functions
	\begin{align*}
		&R^{-2\tilde{b} -2k +2r}\log^{-j}(R) \cdot\Phi(R)\cdot \zeta_{\tilde{b}, k , r, j}(R),\;\;\;R^{-2\tilde{b} -2k +2r}\log^{-j}(R) \cdot\Theta(R) \cdot \eta_{\tilde{b}, k , r, j}(R)
	\end{align*}
	are uniformly bounded for $R\gg 1$ and for $ \tilde{b}, k , r, j$  with $ j \leq b + 2r$ if $2\tilde{b}+2k -2 -2r < -2$. In fact 
	\begin{align*}
		&\Phi(R) \zeta_{\tilde{b}, k , r, j}(R) \in S^0(R^{2\tilde{b} +2k -2r}\log^j(R)),\\[3pt]
		&\Theta(R) \cdot  \eta_{\tilde{b}, k , r, j}(R) \in S^0(R^{2\tilde{b} +2k -2r}\log^j(R)). %\Phi(R)\cdot  \zeta_{bkl''j}\in S\big(R^{-2}\big)\subset  S\big(R^{b+2k-l''}\log^j R\big).
	\end{align*}
	There is one exception for \eqref{das-drei}, if $ 2\tilde{b}+2k -2 -2r = -4$ since 
	$$ W^{-1}(s) \Theta(s) = \mathcal{O}(s^3) + \mathcal{O}(s) \log(s),$$
	so we have to replace $ j \mapsto j+1$ again for the logarithmic factor. This affects the leading order logarithmic term (only)  in case $ k = -1$.\\[3pt]
	Now, changing back to  $R^{2\tilde{b}}\cdot (t\lambda)^{-2 \tilde{b}} = a^{2\tilde{b}}$, we deduce that \eqref{eq:all11} - \eqref{eq:ll12}  admit expansions as in the third case of Definition \ref{defn:SQbetal} with the parameter $\rho_0$ replaced by $C\rho_0$ for some universal constant $C > 0$.  
	\\[8pt]
	{\it{\underline{Case \rom{3}}: $R \gg 1 $ \&  $a\gtrsim 1 $ \& $ 1-a\gtrsim 1$.}}
	Note here the requirement $ 1 \lesssim a$ implies $ t^{\f12 -\nu} \lesssim R$ and thus necessarily $R\gg 1$  if $t\ll 1$. Again considering the terms %(we simplify notation and suppress the $t$-dependence of $E$ in the sequel) 
	\[
	\Phi(R)\cdot\int_0^R W^{-1}(s)\Theta(s) f(t,s)\,ds,\;\;\;	\Theta(R)\cdot\int_0^R W^{-1}(s)\Phi(s) f(t,s)\,ds,
	\]
	with $f(R)$ as in the statement of Lemma~\ref{lem:recoverz1}, we first split $ \int_0^R\;ds = \int_0^{R_*} \;ds+ \int_{R_*}^R\;ds $ for $ R \geq R_* \gg1 $ as above into four integrals
	\begin{align}\label{das-vierte}
		\Phi(R)&\int_0^R W^{-1}(s)\Theta(s) f(t,s)\,ds\\ \nonumber
		& = 	\Phi(R)\int_0^{R_*} W^{-1}(s)\Theta(s) f(t,s)\,ds + 	\Phi(R)\int_{R_*}^R W^{-1}(s)\Theta(s) f(t,s)\,ds,\\ \label{das-fuenfte}
		\Theta(R)&\int_0^R W^{-1}(s)\Phi(s) f(t,s)\,ds\\\nonumber
		&= \Theta(R)\int_0^{R_*}W^{-1}(s)\Phi(s) f(t,s)\,ds + \Theta(R) \int_{R_*}^R W^{-1}(s)\Phi(s) f(t,s)\,ds. 
	\end{align}
	Further each of the respective latter integrals on the right of \eqref{das-vierte}, \eqref{das-fuenfte} is moreover split via a cut-off $\chi_1(a),\; $ smoothly localizing where $a\ll 1$ such that the expansion of the third case in Definition \ref{defn:SQbetal} is valid there, $\chi_1(a) = 1   $ on a (relatively) open subset of this region and $\chi_2(a) := 1 - \chi_1(a)$. Thus we consider 
	\begin{align}\label{eq:asim11-agtrsim1}
		&\Phi(R)\int_{R_*}^R W^{-1}(s)\Theta(s) f(t,s)\,ds\\ \nonumber
		& = \Phi(R)\int_{R_*}^R \chi_1(\tilde{a})W^{-1}(s)\Theta(s) f(t,s)\,ds + \Phi(R)\int_{R_*}^R \chi_2(\tilde{a}) W^{-1}(s)\Theta(s) f(t,s)\,ds,\\ \label{eq:asim11-agtrsim2}
		&\Theta(R) \int_{R_*}^R W^{-1}(s)\Phi(s) f(t,s)\,ds\\ \nonumber
		& = \Theta(R) \int_{R_*}^R \chi_1(\tilde{a})W^{-1}(s)\Phi(s) f(t,s)\,ds + \Theta(R) \int_{R_*}^R \chi_2(\tilde{a})W^{-1}(s)\Phi(s) f(t,s)\,ds.
	\end{align}
	We observe that the coefficient functions $q_{p r j k'}(a)$ for $ a \ll1$  in the third item of Definition~\ref{defn:SQbetal}  satisfy the bounds of the (\emph{not necessarily analytic}) functions $q_{p r j k'}(a) $ in the fourth item of the definition. We therefore use the arguments in the above $R \leq R_*$ bounded case for the left integrals on the right of \eqref{das-vierte} - \eqref{das-fuenfte}, as well as the arguments in the above preceding case for the left integrals on the right of \eqref{eq:asim11-agtrsim1} - \eqref{eq:asim11-agtrsim2}, i.e. if $ \tilde{a} \ll1$.\\[3pt]
	For the remaining integrals, we thus use the expansions in the fourth item of Definition \ref{defn:SQbetal} for  $\chi_2(a)f(t,s)$, and write
	\begin{align*}
		&\Phi(R)\int_{R_*}^R \chi_2(\tilde{a}) W^{-1}(s)\Theta(s) f(t,s)\,ds \\
		&= \Phi(R)\int_{R_*}^R \chi_2(\tilde{a}) W^{-1}(s) \Theta(s)\sum_{p , r, k' \geq 0}\sum_{0\leq j\leq b+ 2r}\frac{t^{k'}}{(t \lambda)^{2p}} \cdot q_{p r j k'}(\tilde{a}) \cdot s^{ 2k -2 - 2r} \cdot\log^j(s)\,ds\\
		&\Theta(R) \int_{R_*}^R \chi_2(\tilde{a})W^{-1}(s)\Phi(s) f(t,s)\,ds\\
		&= \Theta(R) \int_{R_*}^R \chi_2(\tilde{a})W^{-1}(s)\Phi(s) \sum_{p , r, k' \geq 0}\sum_{0\leq j\leq b+ 2r}\frac{t^{k'}}{(t \lambda)^{2p}} \cdot q_{p r j k'}(\tilde{a}) \cdot s^{ 2k -2 - 2r} \cdot\log^j(s)\,ds.
	\end{align*}
	We then need to  change the  integration variable to  $\tilde{a} = \frac{s}{t \lambda}$, observing the possibility to  interchange summation and integration due to the absolute and uniform convergence of the sum for $t> 0$ sufficiently small and $R_* \gg1 $ sufficiently large. This results in 
	\begin{align}
		\Phi(R)\int_{R_*}^R &\chi_2(\tilde{a}) W^{-1}(s)\Theta(s) f(t,s)\,ds \\ \nonumber
		&=  \sum_{p , r, k' \geq 0}\sum_{0\leq j\leq b+ 2r}\frac{t^{k'}}{(t \lambda)^{2p + 2 - 2k + 2r -1}} \Phi(R)\; \times\\\nonumber
		& \hspace{2cm} \times\int_{a_1}^a \chi_2(\tilde{a})W^{-1}(t \lambda \tilde{a}) \cdot  \Theta(t \lambda \tilde{a}) q_{p r j k'}(\tilde{a}) \cdot \tilde{a}^{ 2k -2 - 2r} \cdot\log^j(t \lambda \tilde{a})\,d\tilde{a},\\
		\Theta(R)\int_{R_*}^R& \chi_2(\tilde{a}) W^{-1}(s) \Phi(s) f(t,s)\,ds \\\nonumber
		&=  \sum_{p , r, k' \geq 0}\sum_{0\leq j\leq b+ 2r}\frac{t^{k'}}{(t \lambda)^{2p + 2 - 2k + 2r -1}} \Theta(R)\; \times\\\nonumber
		& \hspace{2cm} \times\int_{a_1}^a \chi_2(\tilde{a})  W^{-1}(t \lambda \tilde{a}) \cdot  \Phi(t \lambda \tilde{a}) q_{p r j k'}(\tilde{a}) \cdot \tilde{a}^{ 2k -2 - 2r} \cdot\log^j(t \lambda \tilde{a})\,d\tilde{a},
	\end{align}
	where $a_1 \leq a_* = R_* (t \lambda)$ is chosen such that $\text{supp}(\chi_2)\subset [a_1,\infty)$. If we then expand 
	\begin{align*}
		&\log^j\big(\tilde{a}\lambda t\big) = \sum_{l=0}^j \left(\begin{array}{c}j\\ l\end{array}\right)\log^l(\lambda t)\cdot \log^{j-l}(\tilde{a}),\\
		&\log^l(\lambda t) = (\log(R) - \log(a))^l = \sum_{l_1=0}^l (-1)^{l_1}\left(\begin{array}{c}l\\ l_1\end{array}\right)\log^{l_1}(R) \log^{l-l_1}(a),
	\end{align*}
	we may write 
	\begin{align}
		&\Phi(R)\int_{R_*}^R \chi_2(\tilde{a}) W^{-1}(s)\Theta(s) f(t,s)\,ds\\ \nonumber
		&\hspace{2cm}= \sum_{p , r, k' \geq 0}\sum_{0\leq j\leq b+ 2r}\frac{t^{k'}}{(t \lambda)^{2p + 3}} R^{ 2k - 2r} \log^j(R) \tilde{q}_{p r j k'}(t,a) \big(R^2 \cdot\Phi(R)\big)\\ 
		&\Theta(R)\int_{R_*}^R \chi_2(\tilde{a}) W^{-1}(s)\Phi(s) f(t,s)\,ds\\ \nonumber
		&\hspace{2cm}= \sum_{p , r, k' \geq 0}\sum_{0\leq j\leq b+ 2r}\frac{t^{k'}}{(t \lambda)^{2p + 1}} R^{ 2k - 2r} \log^j(R) \tilde{q}'_{p r j k'}(t,a)  \Theta(R),
	\end{align}
	where we can expand
	\begin{align*}
		& \tilde{q}_{p r j k'}(t,a) = \sum_{j_1+j_2\leq b+ 2r -j} D_{j_1,j_2,j} \;a^{-2k+2r -2} \log^{j_1}(a)\;\times\\
		&\hspace{4cm} \times  \int_{a_1}^{a} W^{-1}(t \lambda \tilde{a}) \Theta(\lambda t \tilde{a}) \chi_2(\tilde{a})\tilde{a}^{2k -2 -2r}\cdot q_{p r j k'}(\tilde{a})\cdot \log^{j_2}(\tilde{a})\,d\tilde{a},\\
		&\tilde{q}'_{p r j k'}(t,a) = \sum_{j_1+j_2\leq b+ 2r -j} D_{j_1,j_2,j} \;a^{-2k+2r } \log^{j_1}(a)\;\times\\
		&\hspace{4cm} \times  \int_{a_1}^{a} W^{-1}(t \lambda \tilde{a}) \Phi(\lambda t \tilde{a}) \chi_2(\tilde{a})\tilde{a}^{2k -2 -2r}\cdot q_{p r j k'}(\tilde{a})\cdot \log^{j_2}(\tilde{a})\,d\tilde{a}.
	\end{align*}
	and where $\big|D_{j_1,j_2,j}\big|\leq D^{b + 2r}$ for a suitable universal constant $D > 0$. We then expand 
	\begin{align}
		&\Phi(\lambda t \tilde{a} ) \cdot W^{-1}(\lambda t  \tilde{a})  = \sum_{\tilde{k} \geq 0} \big( t \lambda \tilde{a}\big)^{2\tilde{k} +1},\\
		&\Theta(\lambda t  \tilde{a})W^{-1}(\lambda t  \tilde{a}) =  \sum_{\tilde{k} \geq 0} \big(t \lambda \tilde{a}\big)^{2\tilde{k} +3} +  \log(t \lambda \tilde{a})\cdot \Phi(\lambda t \tilde{a} ) \cdot W^{-1}(\lambda t  \tilde{a}) ,
	\end{align}
	and rewrite $\log (\lambda t)$ in terms of $\log(R), \log(a)$ as before. Note the $\log(R) $ does not contribute to the leading order expansion as the difference of the above two lines cancels out this term.  We finally arrive at an expansion of the form
	\begin{align*}
		&\Phi(R)\int_{R_*}^R \chi_2(\tilde{a}) W^{-1}(s)\Theta(s) f(t,s)\,ds - \Theta(R)\int_{R_*}^R \chi_2(\tilde{a}) W^{-1}(s)\Phi(s) f(t,s)\,ds\\ \nonumber
		&= \sum_{p , r, k' \geq 0}\sum_{0\leq j\leq b+ 2r}\frac{t^{k'}}{(t \lambda)^{2p }} R^{ 2k - 2r} \log^j(R) \tilde{\tilde{q}}_{p r j k'}(a)
	\end{align*}
	where the coefficient functions $ \tilde{\tilde{q}}_{p r j k'}(a)$ satisfy the bound stated in the fourth item of the definition with $\rho_1$ replaced by $\tilde{D}\rho_1$ for a suitable universal constant $\tilde{D} > 0$. 
	\\[8pt]
	{\it{\underline{Case \rom{4}}: $ R \gg1$ \& $|1-a|\ll 1$.}}
	Here we rely on the expansions in the  fifth item of Definition \ref{defn:SQbetal}. The key is again to change to  $\tilde{a} = \frac{s}{\lambda t}$ as integration variable, as well as the structure of the variation of constants formula which implies a \emph{smoothing effect} in the singular $ | 1-a| \ll1 $ regime. To this end we write Green' s kernel
	\begin{align}\label{das-xi}
		\Xi(s,R)  &:= \Phi(R)\Theta(s) - \Theta(R)\Phi(s)\\[3pt] \nonumber
		&\;= \Phi(R)\cdot [\Theta(s) - \Theta(R)] - \Theta(R)\cdot [\Phi(s) - \Phi(R)]. 
	\end{align}
	We further observe that with 
	\[
	\tilde{\Phi}(s, R) = \frac{\Phi(s) - \Phi(R)}{s^2-R^2},\;\;\tilde{\Theta}(s, R) = \frac{\Theta(s) - \Theta(R)}{s^2-R^2},\;\; s \neq R,
	\]
	we can rewrite 
	\begin{align}\label{eq:diffstructure1}
		\Theta(s) - \Theta(R) &= (s^2- R^2)\cdot \tilde{\Theta}(s, R)\\ \nonumber
		& = (\lambda t)^2\cdot (\tilde{a}^2 - a^2)\cdot \tilde{\Theta}(s, R)\\ \nonumber
		& = (\lambda t)^2\cdot (1-a^2)\cdot \tilde{\Theta}(s, R) -  (\lambda t)^2\cdot (1-\tilde{a}^2)\cdot \tilde{\Theta}(s, R),\\[8pt] \label{eq:diffstructure2}
		\Phi(s) - \Phi(R) &=  (\lambda t)^2\cdot (1-a^2)\cdot \tilde{\Phi}(s, R) -  (\lambda t)^2\cdot (1-\tilde{a}^2)\cdot \tilde{\Phi}(s, R).
	\end{align}
	%and likewise 
	%	\begin{align}\label{eq:diffstructure22}
		%	\Phi(s) - \Phi(R) &= (s- R)\cdot \tilde{\Phi}(s, R) = \lambda t\cdot (\tilde{a} - a)\cdot \tilde{\Phi}(s, R)\\ \nonumber
		%	& = \lambda t\cdot (1-a)\cdot \tilde{\Phi}(s, R) -  \lambda t\cdot (1-\tilde{a})\cdot \tilde{\Phi}(s, R).
		%	\end{align}
	Thus plugging \eqref{eq:diffstructure1} - \eqref{eq:diffstructure2} into \eqref{das-xi}, we obtain 
	\begin{align}\label{das-xi2}
		\Xi(s,R)  &=(\lambda t)^2\cdot (1-a^2)\cdot \tilde{\Xi}(s,R) -  (\lambda t)^2\cdot (1-\tilde{a}^2)\cdot  \tilde{\Xi}(s,R),\\[4pt] \nonumber
		\tilde{\Xi}(s,R) &=	\tilde{\Theta}(s, R)  \Phi(R) - \tilde{\Phi}(s, R) \Theta(R).
	\end{align}
	Here the functions $\tilde{\Theta}(s, R), \tilde{\Phi}(s,R)$ admit absolutely convergent expansion in terms of powers of $s^{-2}, R^{-2}, \log(R), \log(s)$ if  $s, R\gg 1$ is sufficiently large. To be precise we infer in an absolute sense
	\begin{align}
		& \tilde{\Phi}(s,R) = \frac{1}{R^2 s^2} \sum_{N, \tilde{k} \geq 0} R^{-2N} s^{-2\tilde{k}} \cdot \tilde{c}_{N, \tilde{k}},\;\; s, R \gg1,\\
		& \tilde{\Theta}(s,R) = \frac{1}{R^2 s^2} \sum_{l_1 + \l_2 \leq 1}\sum_{N, \tilde{k} \geq 0} R^{-2N} s^{-2\tilde{k}}\cdot  \tilde{d}^{l_1, l_2}_{ N, \tilde{k}} \cdot \log^{l_1}(s) \log^{l_2}(R),
	\end{align}
	Then we have 
	\begin{align} \label{das-xi-expansion}
		&\tilde{\Xi}(s,R) = \frac{1}{R^2 s^2} \sum_{l_1 + \l_2 \leq 1}\sum_{N, \tilde{k} \geq 0} R^{-2N-2(l_1 + l_2)} s^{-2\tilde{k}} \cdot \tilde{e}^{l_1, l_2}_{ N, \tilde{k}} \cdot \log^{l_1}(s) \log^{l_2}(R),
	\end{align}
	We now decompose for $|1-a|\ll 1$, $t\ll 1$  and $ R \geq R_* \gg1$
	\begin{align}\label{eq:a-1smalldecomp}
		&\Phi(R)\cdot\int_0^RW^{-1}(s)\Theta(s)f(t,s)\,ds -  \Theta(R)\cdot\int_0^RW^{-1}(s)\Phi(s)f(t,s) \,ds\\\label{eq:a-1smalldecomp2}
		& = \int_0^{R_*}W^{-1}(s)\Xi(s,R)f(t,s)\,ds +  \int_{R_*}^{R}W^{-1}(s)\Xi(s,R)\chi_1(\tilde{a})f(t,s)\,ds\\\label{eq:a-1smalldecomp3}
		& + \int_{R_*}^{R}W^{-1}(s)\Xi(s,R)\chi_2(\tilde{a})f(t,s) \,ds + \int_{R_*}^{R}W^{-1}(s) \Xi(s,R)\chi_3(\tilde{a})f(t,s)\,ds,
	\end{align}
	where $\chi_1 \in C^{\infty}$ is as in the preceding step, i.e. localizes to $\tilde{a}\ll 1$ where the third item of Definition \ref{defn:SQbetal} applies, while the cutoff $\chi_3 \in C^{\infty}$ localizes  \footnote{We require  $\chi_{j}\equiv 1$ on small  relatively open neighborhoods of $ a =0, a =1$ in $[0, \infty)$}  to $|1-\tilde{a}|\ll 1$ where the fifth item applies and further $ \supp(\chi_1 )\cap \supp(\chi_3) = \emptyset$. Then we  define $\chi_2(a)$ via $ \sum_{j=1}^3 \chi_j(a) = 1$.
	Moreover, %the fourth and second item of the definition imply that on the support of $\chi_2$ 
	we then may assume the fourth item of the definition of the S-space  applies on the support set of $ \chi_2$.\\[3pt] 
	We claim that the integrals in \eqref{eq:a-1smalldecomp2} and the first integral in \eqref{eq:a-1smalldecomp3} are as in the fifth item of Definition \ref{defn:SQbetal}. Let us start with the latter one, which we write as 
	\begin{align*}
		\int_{R_*}^{R} &  s^3 \cdot  \Xi(s,R)\chi_2(\tilde{a})f(t,s)\,ds\\
		&= \Phi(R)\cdot  (t\lambda)\int_{\frac{R_*}{\lambda t}}^{a}\Theta(\lambda t \tilde{a})\cdot \chi_2(\tilde{a})\cdot f(t, \lambda t \tilde{a})\mathcal{O}((t \lambda)^3\tilde{a}^3)\,d\tilde{a}\\
		& \;\;\;-  \Theta(R)\cdot  (t\lambda)\int_{\frac{R_*}{\lambda t}}^{R}\Phi(\lambda t \tilde{a})\cdot \chi_2(\tilde{a})\cdot f(t,\lambda t \tilde{a})\mathcal{O}((t \lambda)^3 \tilde{a}^3)\,d\tilde{a}
	\end{align*}
	Substituting the expansion of the fourth item for the function $f$, the integral expressions become %(after the usual interchange of summation and integration)
	\begin{align*}
		&\Phi(R) (t\lambda)^{1+2k-2}\sum_{p , r, k' \geq 0}\sum_{0\leq j\leq b+ 2r}\frac{t^{k'}}{(t \lambda)^{2p}} \cdot  \int_{\frac{R_*}{\lambda t}}^{a}\mathcal{O}((t \lambda)^3\tilde{a}^3) \Theta(\lambda t \tilde{a})\cdot \chi_2(\tilde{a})\cdot \tilde{a}^{2k-2 }\times\\&\hspace{6cm}\times q_{p r j k'}(\tilde{a})\cdot\big(\lambda t\tilde{a}\big)^{- 2r}\cdot \log^j(\lambda t\tilde{a}) \,d\tilde{a},\\[4pt]
		&\Theta(R) (t\lambda)^{1+2k-2}\sum_{p , r, k' \geq 0}\sum_{0\leq j\leq b+ 2r}\frac{t^{k'}}{(t \lambda)^{2p}} \cdot  \int_{\frac{R_*}{\lambda t}}^{a}\mathcal{O}((t \lambda)^3\tilde{a}^3) \Phi(\lambda t \tilde{a})\cdot \chi_2(\tilde{a})\cdot \tilde{a}^{2k-2 } \times\\&\hspace{6cm} \times q_{p r j k'}(\tilde{a})\cdot\big(\lambda t\tilde{a}\big)^{- 2r}\cdot \log^j(\lambda t\tilde{a}) \,d\tilde{a},\\
	\end{align*}
	Expanding $\log^j(\lambda t\tilde{a}) = [\log(R) - \log(a) + \log\tilde{a}]^j$ and  $\Theta(\lambda t \tilde{a}),\; \Phi(\lambda t \tilde{a})$ in powers of $(\lambda t \tilde{a})^{-2}$, as  well as factors of $\log(\lambda t \tilde{a})$, we find the difference of the expressions has the form
	\begin{align*}
		\sum_{p, r , j , k'\geq 0}\sum_{0\leq j\leq b+ 2r}\frac{t^{k'}}{(t\lambda)^{2p}} \cdot \tilde{q}_{ p r j k' }(a)\cdot R^{2k - 2r}\cdot\log^j (R)
	\end{align*}
	where the coefficient function $\tilde{q}_{ p r j k' }(a)$ is given by sums of the form 
	\begin{align*}
		&\tilde{q}_{ p r j k' }(a)= \sum_{ l\geq 0} \sum_{ \sum j_{\kappa}\leq 2r +b -j} C_{l  j_{1}, j_2}a^{-2k}\cdot \log^{j_1}(a) \\&\hspace{4cm}\cdot\int_{a_*}^a q_{ p r j k' }(\tilde{a})\chi_2(\tilde{a})\log^{j_2}(\tilde{a}) \cdot \tilde{a}^{2k -1 -2r -2l}\,d\tilde{a}.
	\end{align*}
	The coefficients $C_{l j_{1}, j_2}$ can be bounded by $\big|C_{l j_{1}, j_2}\big|\leq E^{r}$ for a universal constant $E > 0$. Moreover, we see from the definition of $\chi_2$ that $\tilde{q}_{ p r j k' }(a)$ is $C^\infty$ and admits a bound 
	\begin{align*}
		\big|\tilde{q}_{ p r j k' }(a) \big|\leq \big(\tilde{E}(\rho_0+\rho_1)\big)^{p + r + j + k'}
	\end{align*}
	for a suitable universal constant $\tilde{E}> 0$. Furthermore, since $\chi_2(\tilde{a})$ is supported away from $\tilde{a} = 1$, we see that in the region $ | a -1 | \ll1 $ , the functions $\tilde{q}_{ p r j k' }(a)$ are analytic near $a = 1$, and thus admit expansions as in the fifth item of the definition where $\rho_1$ is  replaced by $\tilde{\tilde{E}}(\rho+\rho_1)$ for a suitable (universal) constant $\tilde{\tilde{E}} > 0$. 
	\\[3pt]
	\emph{The remaining integrals in \eqref{eq:a-1smalldecomp2}} are handled as in the steps before, i.e. we use item four of Definition \ref{defn:SQbetal} in the left integral and obtain an expansion as in item five since we integrate $ \int_0^{R_*}\;ds$. Moreover we use the expansion of item three in the latter integral and obtain the claimed expansion as above since $\supp(\chi_1) \subset \{ 1-a \gtrsim 1 \}$.\\
	\;\;\\
	\emph{We then turn to the most delicate last term in  \eqref{eq:a-1smalldecomp3}}, where we take advantage of \eqref{das-xi2}. We again use the integration variable $\tilde{a}$ instead of $s$, and further decompose the integral into two new integrals. %via
	%\[
	%\int_{a_*}^ad\tilde{a}  = \int_{a_*}^1 d\tilde{a} - \int_{a}^1d\tilde{a},\;\;\;\text{or}\;\;\; 	\int_{a_*}^a d\tilde{a}  = \int_{a_*}^1 d\tilde{a} + \int_{1}^ad\tilde{a},\;\;\text{if}\; a >1,
	%\]
	%	i.e. 
	%The first  over $[a_*,1]$ and the second  over $[a,1]$ or  $[1,a]$ if $ a > 1$
	Precisely, after changing to $\tilde{a}$,  we split
	\begin{align*}
		\int_{a_*}^a\;d\tilde{a} &= \int_{a_*}^1\;d\tilde{a} - \int_{a}^1\;d\tilde{a},\;\;\text{if}\;\; a < 1,\;\;\;\int_{a_*}^a\;d\tilde{a} = \int_{a_*}^1\;d\tilde{a} + \int_{1}^a\;d\tilde{a},\;\;\text{if}\;\; a > 1,
	\end{align*}
	where we recall that we restrict to $\big|1-a\big|\ll 1$ and assume $ \tilde{a} \leq 1$. The first of these integrals 
	\begin{align*}
		\lambda t\int_{\frac{R_*}{\lambda t}}^{1}\Xi(\lambda t\tilde{a},R)\cdot \chi_3(\tilde{a})f(t,\lambda t\tilde{a}) \mathcal{O}(\tilde{a}^3(\lambda t)^3)\,d\tilde{a}
	\end{align*}
	can be treated just like the left term in \eqref{eq:a-1smalldecomp3}, as it leads to a function which is analytic in $a$. We hence reduce to treating the second type of integrals where we assume $ a < 1 $ and consider $\int_a^1\;d\tilde{a}$.   Hence we seek expansions as in the fifth item of Definition \ref{defn:SQbetal} from the left of $ a =1$ dropping the subscript $\pm$. The case where $ a >1$, i.e. expanding on the right of $ a =1$ follows analogously. Thus we write
	\begin{align}
		&\lambda t\int_{a}^{1}\Xi(\lambda t\tilde{a},R)\cdot \chi_3(\tilde{a})f(t,\lambda t\tilde{a}) \mathcal{O}(\tilde{a}^3(\lambda t)^3)\,d\tilde{a}\\ \nonumber
		&=  (\lambda t)^3\cdot (1-a^2) \int_{a}^{1} \tilde{\Xi}(\lambda t \tilde{a}, R) \cdot \chi_3(\tilde{a})f(t,\lambda t\tilde{a})\mathcal{O}(\tilde{a}^3(\lambda t)^3)\,d\tilde{a}\\\nonumber
		&\;\;\; - (\lambda t)^3\cdot \int_{a}^{1} \tilde{\Xi}(\lambda t \tilde{a}, R) \cdot (1-\tilde{a}^2)  \chi_3(\tilde{a})f(t,\lambda t\tilde{a})\mathcal{O}(\tilde{a}^3(\lambda t)^3)\,d\tilde{a}
		%& = (\lambda t)^3\cdot (1-a^2) \Phi(R)\int_{a}^{1}\tilde{\Theta}(\lambda t\tilde{a},R) \chi_3(\tilde{a})f(t,\lambda t\tilde{a})\mathcal{O}(\tilde{a}^3(\lambda t)^3)\,d\tilde{a}\\
		%&\;\;\;- (\lambda t)^3\cdot (1-a^2)\Theta(R)\int_{a}^{1}\tilde{\Phi}(\lambda t\tilde{a},R) \chi_3(\tilde{a})f(t,\lambda t\tilde{a})\mathcal{O}(\tilde{a}^3(\lambda t)^3)\,d\tilde{a}\\
		%& \;\;\;+(\lambda t)^3\cdot \Theta(R)\int_{a}^{1}\tilde{\Phi}(\lambda t\tilde{a},R) (1-\tilde{a}^2)\chi_3(\tilde{a})f(t,\lambda t\tilde{a})\mathcal{O}(\tilde{a}^3(\lambda t)^3)\,d\tilde{a}\\
		%&\;\;\;-(\lambda t)^3\cdot \Phi(R)\int_{a}^{1}\tilde{\Theta}(\lambda t\tilde{a},R) (1-\tilde{a}^2)\chi_3(\tilde{a})f(t,\lambda t\tilde{a})\mathcal{O}(\tilde{a}^3(\lambda t)^3)\,d\tilde{a}
	\end{align}
	We then use the expansion \eqref{das-xi-expansion}, i.e. expanding $\tilde{\Phi}(\lambda t \tilde{a}, R), \tilde{\Theta}(\lambda t\tilde{a},R)$ into an absolutely convergent power series for $R^{-2}, (\lambda t\tilde{a})^{-2}$ where  $\lambda t\tilde{a} = s\gg 1$ on the support of the integrand. Further observe that 
	\[
	\chi_3(\tilde{a}) = 1
	\]
	for $\tilde{a}\in [a, 1]$, $|1-a|\ll 1$, we obtain the form (in the regime $|1-a|\ll 1$)
	\begin{align}
		& (\lambda t)^3\cdot (1-a^2) \int_{a}^{1} \tilde{\Xi}(\lambda t \tilde{a}, R) \cdot  \chi_3(\tilde{a})f(t,\lambda t\tilde{a})\mathcal{O}(\tilde{a}^3(\lambda t)^3)\,d\tilde{a}\\ \nonumber
		%&\lambda t\cdot (1-a)\int_{a}^{1}\tilde{\theta}(\lambda t\tilde{a},R)\chi_3(\tilde{a})f(\lambda t\tilde{a})\tilde{a}^3(\lambda t)^3\,d\tilde{a}\\
		& = \sum_{\gamma, \gamma'\geq 1}  \sum_{p, r , j , k'\geq 0}\sum_{0\leq j \leq b + 2r}\frac{t^{k'}}{(t\lambda)^{2p -2 + 2r + 2 \gamma'}} \cdot R^{2k - 2r - 2 \gamma} C_{\gamma, \gamma'} \times \\ \nonumber&\hspace{3cm} \times (1-a^2)\cdot   a^{-2k+ 2r}\int_{a}^{1}\tilde{a}^{2k -1 -2\gamma'-2r}\cdot \log^{j}\big(\lambda t \tilde{a}\big)\cdot q_{ p r j k' }(\tilde{a}) \,d\tilde{a},
	\end{align}
	%	\begin{align*}
		%	& (\lambda t)^3\cdot (1-a^2) \int_{a}^{1}\big(\tilde{\Theta}(\lambda t\tilde{a},R)\Phi(R) - \tilde{\Phi}(\lambda t\tilde{a},R) \Theta(R)\big) \chi_3(\tilde{a})E(t,\lambda t\tilde{a})\mathcal{O}(\tilde{a}^3(\lambda t)^3)\,d\tilde{a}\\
		%&\lambda t\cdot (1-a)\int_{a}^{1}\tilde{\theta}(\lambda t\tilde{a},R)\chi_3(\tilde{a})E(\lambda t\tilde{a})\tilde{a}^3(\lambda t)^3\,d\tilde{a}\\
		%	& = \sum_{\gamma, \gamma'\geq 0}  \sum_{p, r , j , k'\geq 0}\sum_{0\leq \sum_{\kappa} j_{\kappa} \leq b + 2r}\frac{t^{k'}}{(t\lambda)^{p -3 + 2r + 2 \gamma'}} \cdot R^{2k - 2r - 2 \gamma}\cdot\log^{j_3} (R) C_{\gamma, \gamma', j_1, j_2,j_3} \times \\&\hspace{3cm} \times (1-a^2)\cdot   \log^{j_2} (a) a^{-2k}\int_{a}^{1}\tilde{a}^{2k -1 -2\gamma'-2r}\cdot \log^{j_1}\big(\tilde{a}\big)\cdot q_{ p r j k' }(\tilde{a}) \,d\tilde{a},
		%\end{align*}
		and analogously for 
		\begin{align} 
			&(\lambda t)^3\cdot  \int_{a}^{1} (1-\tilde{a}^2)\tilde{\Xi}(\lambda t \tilde{a}, R) \cdot \chi_3(\tilde{a})f(t,\lambda t\tilde{a})\mathcal{O}(\tilde{a}^3(\lambda t)^3)\,d\tilde{a}\\ \nonumber
			& = \sum_{\gamma, \gamma'\geq 1}  \sum_{p, r , j , k'\geq 0}\sum_{0\leq j \leq b + 2r}\frac{t^{k'}}{(t\lambda)^{2p -2 + 2r + 2 \gamma'}} \cdot R^{2k - 2r - 2 \gamma} C_{\gamma, \gamma'} \times \\ \nonumber &\hspace{3cm} \times   a^{-2k+ 2r}\int_{a}^{1}\tilde{a}^{2k -1 -2\gamma'-2r} (1-\tilde{a}^2)\cdot \log^{j}\big(\lambda t \tilde{a}\big)\cdot q_{ p r j k' }(\tilde{a}) \,d\tilde{a},
		\end{align}
		Note, that instead of $ \gamma, \gamma' \geq 1 $ we may factor off $ R^{-2}s^{-2} = (\lambda t)^{-4} a^{-2} \tilde{a}^{-2}$ in the expansion \eqref{das-xi-expansion} and use the fact 
		$
		a^{-2} = \sum_{k' \geq 0} (1-a)^{k'} \tilde{c}_{k'},\; |1-a| \ll1
		$  (similar for $\tilde{a}^{-2}$).\\[3pt]
		The coefficients $C_{\gamma, \gamma'}$ arising by the use of \eqref{das-xi-expansion} obey an estimate $\big|C_{\gamma, \gamma'}\big|\leq C_4^{\gamma + \gamma'}$ for some absolute constant $C_4 > 0$.  Since $ (1 - a^2) \sim (1-a)$, the smoothing for the expansion in the region $ |1-a|\ll1$ seen through
		\begin{align*}
			f \mapsto (1-a^2) \int_a^1 f(t, \tilde{a})d\tilde{a},\;\; 	f  \mapsto  \int_a^1  (1-\tilde{a}^2) f(t, \tilde{a})d\tilde{a},
		\end{align*}
		which clearly map $\big(\mathcal{Q}^{\beta_l}\big)'|_{[\frac12,2]} = (a\partial_a)\mathcal{Q}^{\beta_l}|_{[\frac12,2]}  +  (a^{-1}\partial_a)\mathcal{Q}^{\beta_l}|_{[\frac12,2]} $ into $\mathcal{Q}^{\beta_l}|_{[\frac12,2]} $ (in fact, one gains two derivatives here).  It follows that we can write (say for the first of the above integrals)
		\begin{align*}
			& (\lambda t)^3\cdot (1-a^2) \int_{a}^{1} \tilde{\Xi}(t \lambda \tilde{a}, R) \chi_3(\tilde{a})f(t,\lambda t\tilde{a})\mathcal{O}(\tilde{a}^3(\lambda t)^3)\,d\tilde{a}\\
			%&\lambda t\cdot (1-a)\int_{a}^{1}\tilde{\theta}(\lambda t\tilde{a},R)\chi_3(\tilde{a})f(\lambda t\tilde{a})\tilde{a}^3(\lambda t)^3\,d\tilde{a}\\
			& = \sum_{\gamma, \gamma'\geq 1}  \sum_{p, r , j , k'\geq 0}\sum_{0\leq j \leq b + 2r}\frac{t^{k'}}{(t\lambda)^{2p -2 + 2r + 2 \gamma'}} \cdot R^{2k - 2r - 2 \gamma}\cdot\log^{j} (R)  \tilde{q}_{ \gamma' \gamma p r j k' }(a)
		\end{align*}
		where the coefficient functions $ \tilde{q}_{ \gamma' \gamma p r j k' }(a)$ admit a representation of the form 
		\begin{align*}
			&\tilde{q}_{ \gamma' \gamma p r j k' }(a) = (1-a^2)\cdot a^{-2k+ 2r} \sum_{j_1+j_2\leq  j} C_{\gamma, \gamma'} D_{j_1,j_2,j}  \log^{j_2} (a) \times\\ 
			& \hspace{3cm}\times  \int_{a}^{1}\tilde{a}^{2k -1 -2\gamma'-2r}\cdot \log^{j_1}(\tilde{a})\cdot q_{ p r j k' }(\tilde{a}) \,d\tilde{a},
		\end{align*}
		and these expressions are  functions in $\mathcal{Q}^{\beta_l}$, provided we restrict to $|1-a|\lesssim 1$ since we can expand $
		q_{ p r j k' }(\tilde{a})
		$
		absolutely as in Definition  \ref{defn:SQbetal} and $ \chi_3(\tilde{a}) =1$. Furthermore, we have 
		\begin{align*}
			& \sum_{\gamma, \gamma'\geq 1}  \sum_{p, r , j , k'\geq 0}\sum_{0\leq j \leq b + 2r}\frac{t^{k'}}{(t\lambda)^{2p -2 + 2r + 2 \gamma'}} \cdot R^{2k - 2r - 2 \gamma}\cdot\log^{j} (R)  \tilde{q}_{ \gamma' \gamma p r j k' }(a)\\
			&\in S^2\big(R^{2k}\log^b(R), \mathcal{Q}^{\beta_l}\big),
		\end{align*}
		since the required bounds on the coefficients in the expansion of the fifth item of Definition \ref{defn:SQbetal} follow from the bound $\big|D_{j_1,j_2,j}\big|\leq D_3^{j}$ for a universal constant $D_3> 0$, the previous bound on $\big|C_{\gamma, \gamma'} \big|$, and the assumed bounds for the coefficients in the expansion of 
		$
		q_{ p r j k' }(\tilde{a})
		$.\\[4pt]
		Before we turn to the case of $ a \gtrsim 1$, we choose a partition of unity 
		$$ \sum_{j =1}^5 \chi_j(a) = 1,\;\; a \in [0, \infty),$$ 
		such that \textbf{(i)} $ \chi_1, \chi_2, \chi_3 $ are as above, \; \textbf{(ii)} $\supp(\chi_2) \cap \supp(\chi_4) = \emptyset$ as well as \textbf{(iii)} we assume 
		$$ \supp(\chi_4) \subset \{  a \lesssim 1,\;a-1 \gtrsim 1\}$$
		and we may use the expansion of the sixth item of Definition \ref{defn:SQbetal} on  $\supp(\chi_5)$.
		\\[8pt]
		{\it{\underline{Case \rom{5}}: $ R \gg1$ \& $a \lesssim 1$ \& $ a-1 \gtrsim 1$ }}. This case can be handled in analogy to the region where $1-a\gtrsim 1$ and $ a \gtrsim 1$,  taking in particular advantage of the fourth expansion  in Definition \ref{defn:SQbetal}.  To be precise we decompose as in previous cases
		\begin{align}\label{das-vierte-large-a}
			&\int_0^R W^{-1}(s)\Xi(s,R)f(t,s)\,ds\\ \nonumber
			& = 	\int_0^{R_*} W^{-1}(s) \Xi(s,R) f(t,s)\,ds +  \int_{R_*}^R W^{-1}(s)\Xi(s,R) f(t,s)\,ds,
		\end{align}
		where the left integral is treated as before using the first item of Definition \ref{defn:SQbetal} as in the case $ a \gtrsim 1,\; 1- a \gtrsim 1$.  Precisely we expand (with constant $a$ dependence)
		\begin{align}\label{das-vierte-large-a-1}
			&\int_0^{R_*} W^{-1}(s) \Xi(s,R) f(t,s)\,ds = \sum_{k',l',p} q_{k'l'p}(R)\cdot \frac{t^p}{(t\lambda)^{2l' + 2k'}},\;\; R \gg1,\\[3pt] \nonumber
			& q_{k'l'p}(R): = \Phi(R)\cdot \int_0^{R_*} s^3\cdot e_{k'l'p}(s)\cdot  \Theta(s) s^{2k'}\,ds - \Theta(R)\cdot \int_0^{R_*} s^3\cdot e_{k'l'p}(s)\cdot  \Phi(s) s^{2k'}\,ds,
		\end{align}
		where we can estimate for some universal constant $ C'' > 0 $
		\begin{align*}
			\big| q_{k'l'p}(R)\big| \leq (C'')^{k'+l'+p}.
		\end{align*}
		For the latter integral on the right of \eqref{das-vierte-large-a} we write
		\begin{align}\label{das-vierte-large-a-2}
			\int_{R_*}^R W^{-1}(s)\Xi(s,R) f(t,s)\,ds &= \; \sum_{j =1}^4 \Phi(R)\int_{R_*}^R W^{-1}(s) \chi_j(\tilde{a}) \Theta(s)f(t,s)\,ds\\ \nonumber
			&\;\;\; - \sum_{j =1}^4 \Theta(R)\int_{R_*}^R W^{-1}(s) \chi_j(\tilde{a}) \Phi(s)f(t,s)\,ds.
		\end{align}
		The first of these integrals corresponding to $ \chi_1(a)$, using item three of the definition of the S-space, we obtain
		\begin{align*}
			\int_{R_*}^R W^{-1}(s) \chi_j(\tilde{a}) \Xi(s,R) f(t,s)\,ds &=	\sum_{p, r, k' \geq 0} \sum_{0\leq j\leq b +  2r } \sum_{\tilde{b} \geq 0}\frac{t^{k'}}{(t \lambda)^{2p + 2\tilde{b}}} q_{p r j \tilde{b} k' }\times\\ \nonumber
			&\hspace{1cm}  \times  \int_{R_*}^{t \lambda a_1} s^3 \chi_1(\tilde{a}) \Xi(s,R) s^{2\tilde{b}+2k-2- 2r}\cdot\log^j(s)\,ds
		\end{align*}
		where $ a_1 \ll1$ is at the boundary of the support of $ \chi_1$. As in previous cases we may consider boundedness of the term
		\begin{align} 
			R^{- 2k  + 2r}\log^{-j} (R) \cdot \int_{R_*}^{t \lambda a_1} s^3\cdot \Xi(s,R) \chi_1(\tilde{a}) \cdot s^{2\tilde{b}+2k-2- 2r}\cdot\log^j(s)\,ds,
		\end{align}
		where, evaluating the integral and using $ a_1 t \lambda  \ll a t \lambda = R$, the left integral boundary gives the upper bound  $ (a_1 t \lambda)^{2\tilde{b}} R^{2k} (a_1 t \lambda)^{- 2r}$, so that the sums over $ \tilde{b}, r $ converge absolutely if $ t \ll 1 $ is small (depending only on the constant in the third item of Definition \ref{defn:SQbetal}).
		Further the integrals 
		\begin{align*}
			\Phi(R)\int_{R_*}^R W^{-1}(s) \chi_2(\tilde{a}) \Theta(s)f(t,s)\,ds,\;\;  \Theta(R)\int_{R_*}^R W^{-1}(s) \chi_2(\tilde{a}) \Phi(s)f(t,s)\,ds,\\
			\Phi(R)\int_{R_*}^R W^{-1}(s) \chi_3(\tilde{a}) \Theta(s)f(t,s)\,ds,\;\;  \Theta(R)\int_{R_*}^R W^{-1}(s) \chi_3(\tilde{a}) \Phi(s)f(t,s)\,ds,
		\end{align*}
		are handled using item four and five in Definition \ref{defn:SQbetal}  as in the  previous cases. Precisely, since both $ \chi_2$ have support away from $ a = 0$ and $ a =1$, the integral has the form
		\begin{align*}
			&\int_{c t\lambda}^{C t\lambda} s^3 \cdot \Xi(s,R)  \chi_2(\tilde{a}) f(t,s)\;ds = \sum_{p , r, k' \geq 0}\sum_{0\leq j\leq b+ 2r}\frac{t^{k'}}{(t \lambda)^{2p }} \log^j(R) \tilde{q}_{p r j k'}(t,a), \\ 
			&  \tilde{q}_{p r j k'}(t,a) = \sum_{j_1+j_2\leq b+ 2r -j} D_{j_1,j_2,j} \; \log^{j_1}(a)\;\times\\
			&\hspace{4cm} \times  \int_{c}^{C} \tilde{a}^3 \cdot \Xi(\tilde{a} t \lambda,R)  \chi_2(\tilde{a})\tilde{a}^{2k -2 -2r}\cdot q_{p r j k'}(\tilde{a})\cdot \log^{j_2}(\tilde{a})\,d\tilde{a}.
		\end{align*}
		Further for $ \chi_3$ we similarly have 
		\begin{align*}
			&\int_{c t\lambda}^{C t\lambda} s^3 \cdot \Xi(s,R)  \chi_3(\tilde{a}) f(t,s)\;ds = \sum_{p , r, k' \geq 0}\sum_{0\leq j\leq b+ 2r}\frac{t^{k'}}{(t \lambda)^{2p }} R^{2k} \times\\
			&\hspace{4cm} \times a^{-2k} \int_c^C \tilde{a}^{3} \cdot \chi_3(\tilde{a}) \Xi(t\lambda \tilde{a}, R) \tilde{a}^{2k - 2r} q_{p r j k'}(\tilde{a})\;d\tilde{a},
		\end{align*}
		where now $ 0 < C-1 \ll1,\; 0 < 1 -c \ll1$.
		Here we thus  note the integral corresponding to $ \chi_2(a), \chi_3(a)$ delivers of course analyticity in $a$ since $\chi_2, \chi_3 $ are supported away from $ a \gtrsim 1$.  For the remaining integrals in the region where  $ a \lesssim 1,\; a-1 \gtrsim 1$, i.e.
		\[
		\Phi(R)\int_{R_*}^R W^{-1}(s) \chi_4(\tilde{a}) \Theta(s)f(t,s)\,ds,\;\; \Theta(R)\int_{R_*}^R W^{-1}(s) \chi_4(\tilde{a}) \Phi(s)f(t,s)\,ds,
		\]
		we change $s$ to $ (\lambda t) \tilde{a}$, note $a_* = R_*(t \lambda)^{-1}$ is small, and split 
		$
		\int_{a_*}^a\;d \tilde{a} =  \int_{1}^{a_0}\;d \tilde{a} + \int_{a_0}^a\;d\tilde{a},
		$
		where $ a_0 > 1$ is such that $ \chi_4(a) = 1$ if $ a > a_0$. Then we use the expansions in the fifth item of Definition \ref{defn:SQbetal} for the first of these integrals,  which delivers an analytic expansion in $a$,  and  further the fifth item of Definition \ref{defn:SQbetal} for the latter. The required bounds in the fourth item of the definition  follow as before considering the coefficients in the expansion of $ W^{-1}(\lambda t \tilde{a}) \Xi(\lambda t \tilde{a}, R) $ and $f(t,  t \lambda \tilde{a})$ in the region (using the definition of the S-space). Precisely we obtain as before in \emph{Case} \rom{3}
		\begin{align*}
			&\int_{a_0 t \lambda}^{a t \lambda} s^3 \cdot \Xi(s,R)  f(t,s)\;ds = \sum_{p , r, k' \geq 0}\sum_{0\leq j\leq b+ 2r}\frac{t^{k'}}{(t \lambda)^{2p }} R^{ 2k - 2r} \log^j(R) \tilde{q}_{p r j k'}(t,a) \big(R^2 \cdot\Phi(R)\big)\\
			& \;\;\; \hspace{4cm}+  \sum_{p , r, k' \geq 0}\sum_{0\leq j\leq b+ 2r}\frac{t^{k'}}{(t \lambda)^{2p}} R^{ 2k - 2r} \log^j(R) \tilde{q}'_{p r j k'}(t,a)  \Theta(R),
		\end{align*}
		where we expand
		\begin{align*}
			& \tilde{q}_{p r j k'}(t,a) = \sum_{j_1+j_2\leq b+ 2r -j} D_{j_1,j_2,j} \;a^{-2k+2r -2} \log^{j_1}(a)\;\times\\
			&\hspace{4cm} \times  \int_{a_0}^{a} \tilde{a}^3 \cdot  \Theta(\lambda t \tilde{a}) \tilde{a}^{2k -2 -2r}\cdot q_{p r j k'}(\tilde{a})\cdot \log^{j_2}(\tilde{a})\,d\tilde{a},\\
			&\tilde{q}'_{p r j k'}(t,a) = \sum_{j_1+j_2\leq b+ 2r -j} D_{j_1,j_2,j} \;a^{-2k+2r } \log^{j_1}(a)\;\times\\
			&\hspace{4cm} \times  \int_{a_0}^{a} \tilde{a}^3 \cdot \Phi(\lambda t \tilde{a}) \tilde{a}^{2k -2 -2r}\cdot q_{p r j k'}(\tilde{a})\cdot \log^{j_2}(\tilde{a})\,d\tilde{a}.
		\end{align*}
		and $\big|D_{j_1,j_2,j}\big|\leq D^{b + 2r}$ for a suitable universal constant $ D> 0$. Rewriting  $\log (\lambda t)$ as before in terms of $\log(R), \log(a)$,  we arrive at an expansion of the form
		\begin{align*}
			&\Phi(R)\int_{R_*}^R \chi_2(\tilde{a}) W^{-1}(s)\Theta(s) f(t,s)\,ds - \Theta(R)\int_{R_*}^R \chi_2(\tilde{a}) W^{-1}(s)\Phi(s) f(t,s)\,ds\\ \nonumber
			&= \sum_{p , r, k' \geq 0}\sum_{0\leq j\leq b+ 2r}\frac{t^{k'}}{(t \lambda)^{2p }} R^{ 2k - 2r} \log^j(R) \tilde{\tilde{q}}_{p r j k'}(a)
		\end{align*}
		where the coefficient functions $ \tilde{\tilde{q}}_{p r j k'}(a)$ satisfy the bound stated in the fourth item of the definition with $\rho_1$ replaced by $\tilde{\tilde{D}}\rho_1$ for a suitable universal constant $\tilde{\tilde{D}} > 0$.  In fact, considering the Case \rom{3} for the $ a \gtrsim 1,\; 1-a \gtrsim 1$ region, we choose the minimum of $ \tilde{D}, \tilde{\tilde{D}} > 0$ so that both regions are consistent with item four of Definition \ref{defn:SQbetal}.
		\\[8pt]
		{\it{\underline{Case \rom{6}}: $ R \gg1$ \& $ a\gg1$ }}. Here we decompose again
		\begin{align}\label{das-vierte-large-a-3}
			\int_{R_*}^R W^{-1}(s)\Xi(s,R) f(t,s)\,ds &= \; \sum_{j =1}^5 \Phi(R)\int_{R_*}^R W^{-1}(s) \chi_j(\tilde{a}) \Theta(s)f(t,s)\,ds\\ \nonumber
			&\;\;\; - \sum_{j =1}^5 \Theta(R)\int_{R_*}^R W^{-1}(s) \chi_j(\tilde{a}) \Phi(s)f(t,s)\,ds,
		\end{align}
		for which the first four integrals follow the argument in the previous case. We thus  focus on the last integrals
		\begin{align}
			&\Phi(R)\int_{R_*}^R W^{-1}(s) \chi_5(\tilde{a}) \Theta(s)f(t,s)\,ds\\ \nonumber
			&\hspace{3cm} = \lambda t \cdot 	\Phi(R)\int_{a_*}^{a_1}W^{-1}(\lambda t \tilde{a}) \chi_5(\tilde{a}) \Theta(\lambda t \tilde{a})f(t,\lambda t \tilde{a})\,d\tilde{a} \\\nonumber
			& \hspace{3cm} \;\;+ \lambda t \cdot 	\Phi(R)\int_{a_1}^a W^{-1}(\lambda t \tilde{a}) \chi_5(\tilde{a}) \Theta(\lambda t \tilde{a})f(t,\lambda t \tilde{a})\,d\tilde{a},
		\end{align}and further
		\begin{align}
			&\Theta(R)\int_{R_*}^R W^{-1}(s) \chi_5(\tilde{a}) \Phi(s)f(t,s)\,ds\\\nonumber
			&\hspace{3cm}=  \lambda t \cdot 	\Theta(R)\int_{a_*}^{a_1} W^{-1}(\lambda t \tilde{a}) \chi_5(\tilde{a}) \Phi(\lambda t \tilde{a})f(t,\lambda t \tilde{a})\,d\tilde{a}\\ \nonumber
			& \hspace{3cm} \;\;+ \lambda t \cdot 	\Theta(R)\int_{a_1}^a W^{-1}(\lambda t \tilde{a}) \chi_5(\tilde{a}) \Phi(\lambda t \tilde{a})f(t,\lambda t \tilde{a})\,d\tilde{a},
		\end{align}
		for which we choose $ a_1 \gg1 $ large such that in particular $ \chi_5(a) =1$ for all $ a > a_1$. Hence the integrals
		\begin{align}
			&\lambda t \cdot \Phi(R)\int_{a_*}^{a_1}W^{-1}(\lambda t \tilde{a}) \chi_5(\tilde{a}) \Theta(\lambda t \tilde{a})f(t,\lambda t \tilde{a})\,d\tilde{a} ,\\
			&\lambda t \cdot \Theta(R)\int_{a_*}^{a_1}W^{-1}(\lambda t \tilde{a}) \chi_5(\tilde{a}) \Phi(\lambda t \tilde{a})f(t,\lambda t \tilde{a})\,d\tilde{a} 
		\end{align}
		are handled using the expansion in the sixth item. We consider the second of these integrals on the right sides, as the first are expanded analogously (with the only difference being $ \chi_5(\tilde{a}) $ is not constant and the expansion is  independent (thus analytic) in $a \gg1$). Thus in the integrals, we expand as in the definition
		\begin{align}
			&\lambda t \cdot \Phi(R)\int_{a_1}^{a}W^{-1}(\lambda t \tilde{a}) \Theta(\lambda t \tilde{a})f(t,\lambda t \tilde{a})\,d\tilde{a} ,\\ \nonumber
			&=  (\lambda t)^4 \cdot \Phi(R) \int_{a_1}^{a} \tilde{a}^3  \Theta(\lambda t \tilde{a}) \sum_{r , j \geq 0}\sum_{0\leq j\leq b+ 2r} q_{ r j }(\tilde{a})\cdot (t\lambda \tilde{a})^{2k -2 - 2r}\cdot\log^j (t\lambda \tilde{a})  \,d\tilde{a},\\
			&\lambda t \cdot \Theta(R)\int_{a_1}^{a}W^{-1}(\lambda t \tilde{a})  \Phi(\lambda t \tilde{a})f(t,\lambda t \tilde{a})\,d\tilde{a}\\ \nonumber
			&= (\lambda t)^4 \cdot \Theta(R) \int_{a_1}^{a} \tilde{a}^3  \Phi(\lambda t \tilde{a}) \sum_{r , j \geq 0}\sum_{0\leq j\leq b+ 2r} q_{ r j }(\tilde{a})\cdot (t\lambda \tilde{a})^{2k  -2- 2r}\cdot\log^j (t\lambda \tilde{a})  \,d\tilde{a}.
		\end{align} 
		Then write provided $2k-2r> -4$
		\begin{align*}
			&(\lambda t)^4 \cdot \Phi(R) \int_{a_1}^{a} \tilde{a}^3  \Theta(\lambda t \tilde{a}) \sum_{r , j \geq 0}\sum_{0\leq j\leq b+ 2r} q_{ r j }(\tilde{a})\cdot (t\lambda \tilde{a})^{2k -2 - 2r}\cdot\log^j (t\lambda \tilde{a})  \,d\tilde{a}\\
			& = (t\lambda)^2\cdot (t\lambda a)^{2k - 2r}\cdot \Phi(R)\cdot \int_{a_1}^{a} \tilde{a}  \Theta(\lambda t \tilde{a}) \sum_{r , j \geq 0}\sum_{0\leq j\leq b+ 2r} q_{ r j }(\tilde{a})\cdot (\tilde{a}a^{-1})^{2k -2 - 2r}\cdot\log^j (t\lambda \tilde{a})  \,d\tilde{a}
		\end{align*}
		Here we note that $(t\lambda)^2\cdot (t\lambda a)^{2k - 2r}\cdot \Phi(R)\cdot a^2 = R^{2k-2r} + O(R^{2k-2r-2})$.
		Using the  expansions for the coefficient functions $ q_{r j}(\tilde{a})$ in  item six of Definition \ref{defn:SQbetal}, and further expanding $ \Theta(\lambda t \tilde{a})$ in inverse powers of $\lambda t \tilde{a}$, one verifies that the preceding expression is indeed in 
		\[
		S^m\big(R^{2k-2}\log^b(R),\,(\mathcal{Q}^{\beta_l})'\big).
		\] 
		If on the other hand $2k-2r\leq  -4$ we don't need to factor out $(t\lambda a)^{2k - 2r}$. \\
		The term with $\Phi$ and $\Theta$ interchanged is handled analogously. 
		\\
		
		Finally for the particular solution
		$$ R^{-2}\int_0^R \frac{s^3}{4}\cdot f(t, s)\,ds - \int_0^R \frac{s}{4} f(t, s)\,ds,\;\;\;$$
		in the last line of the Lemma, the previous arguments apply (and are of course simpler in terms of the kernel expansion). %We only like to give details for the
		In particular we have a similar  \emph{smoothing} near $ a =1$ in the case of $R \gg1, \; |1-a|\ll1$.
	\end{proof}
	\begin{Rem} The proof of Lemma \ref{lem:recoverz1}  shows in the expansions of the S-space of Definition \ref{defn:SQbetal}, there is no additional contribution for $ \sum_{k' \geq 0} t^{k'}$ when inverting the operators as above. Hence  these sums are entirely as in the source term and for our $z$-iteration, they are inductively supported on the leading order $t^0$ terms. We therefore ignore this dependence in the following and only take into account $ \sum_{p \geq0} (t \lambda)^{- 2p}$.
	\end{Rem}
	\;\\
	The next section gives the details for \emph{the hyperbolic modifier step}, i.e. alongside Lemma \ref{lem:recoverz1}, the second main Lemma for constructing the above $\Box^{-1}$ parametrix.
	\subsection{The wave parametrix $\Box^{-1}$: Main lemmas} \label{sec:box-main-lemma}
	When performing the inductive step, we need to apply the previously described parametrix of $\Box^{-1}$ to varying source functions, depending on the stage of the iteration. In particular, apart from treating the initial source term \eqref{initial source} as in Section \ref{subsec:initial} above, we need to consider the \emph{linear source} \eqref{linear-source} and the \emph{quadratic/cubic interaction terms}  \eqref{quad-cub} for \emph{higher order corrections}. The following lemma takes partly care of this.
	\begin{Lemma}\label{lem:Boxminuesonefirst} Let  $l\geq 1, k\geq 0 $, $ b = b(l) \in \Z_{\geq0}$ and further
		\begin{align*}
			t^2\cdot f\in \lambda^2\frac{ t^{2\nu k }}{(t\lambda)^{2l}}S^m\big(R^{2k}\log^{b}(R), (\mathcal{Q}^{\beta_l})''\big).
		\end{align*}
		Then there is 
		\begin{align*}
			n\in \lambda^2 \frac{t^{2\nu k}}{(t\lambda)^{2l}}S^2\big(R^{2k}\log^{b}(R), (\mathcal{Q}^{\beta_l})'\big) + \lambda^2\frac{t^{2\nu k}}{(t\lambda)^{2l+2}}S^{2}\big(R^{2k}\log^{b+1}(R), \mathcal{Q}^{\beta_l}\big)
		\end{align*}
		satisfying 
		\begin{align*}
			t^2(\Box n - f)\in \lambda^2 \frac{t^{2\nu k}}{(t\lambda)^{2l+2}}S^{\min\{m,2\}}\big(R^{2k}\log^{b+1}(R), (\mathcal{Q}^{\beta_l})''\big).
		\end{align*}
	\end{Lemma}
	\begin{Rem}\label{rem:elliptic-steps-first} If the leading order term $R^{2k} \log^b(R)$ for $ f$ in the S-space is actually `better' in the sense that for $ m_1 \in \Z_{+},\; 2k -2m_1 +4 \geq 0$ we have
		\begin{align*}
			t^2\cdot f\in \lambda^2\frac{ t^{2\nu k }}{(t\lambda)^{2l}}S^m\big(R^{2k-2m_1}\log^{b}(R), (\mathcal{Q}^{\beta_l})''\big),
		\end{align*}
		then we may of course apply $m_1 \geq 1$ \emph{elliptic steps}. Thus there exists for $ j = 1, \dots m_1$
		$$ \tilde{n}_j \in \lambda^2\frac{ t^{2\nu k }}{(t\lambda)^{2l + 2j}}S^m\big(R^{2k-2m_1 + 2j}\log^{b + j}(R), \mathcal{Q}^{\beta_l}\big)$$
		with $ \Box ( \tilde{n}_1 + \dots + \tilde{n}_{m_1}) - f = - \partial_{tt} \tilde{n}_{m_1}$ and 
		$$ t^2 \cdot \partial_{tt} \tilde{n}_{m_1} \in \lambda^2 \frac{t^{2\nu k}}{(t\lambda)^{2l+2m_1}}S^{\min\{m,2\}}\big(R^{2k}\log^{b+m_1}(R), (\mathcal{Q}^{\beta_l})''\big),$$
		which is a direct consequence of the description of \emph{Step (1)} in Section \ref{sec: the wave para gd} and Lemma \ref{lem:recoverz1}.
	\end{Rem}
	\;\\
	Iterating the preceding Lemma \ref{lem:Boxminuesonefirst} a finite number of times, we infer the following version for an improved error: 
	\begin{Cor}\label{cor:Boxminuesonefirst} Let $l\geq 1, k\geq 0$,  $ b = b(l) \in \Z_{\geq0}$,
		\begin{align*}
			t^2\cdot f \in \lambda^2 \frac{t^{2\nu k}}{(t\lambda)^{2l}}S^m\big(R^{2k}\log^{b}(R), (\mathcal{Q}^{\beta_l})''\big),
		\end{align*}
		and further let $N \geq	 1$. 
		Then there is
		\begin{align*}
			&n_N\in \sum_{N+l\geq l'\geq l} \lambda^2 \frac{t^{2\nu k}}{(t\lambda)^{2l'}}S^2\big(R^{2k}\log^{b + l' -l }(R), (\mathcal{Q}^{\beta_{l'}})'\big),\\
			&  \hspace{2cm}+  \sum_{N+l\geq l'\geq l} \lambda^2 \frac{t^{2\nu k}}{(t\lambda)^{2l'+2}}S^{2}\big(R^{2k}\log^{b+l'- l +1}(R), \mathcal{Q}^{\beta_{l'}}\big),
			%	sdgfdfg
			%	\frac{t^{2\nu k}}{(t\lambda)^{2l'}}S^2\big(R^{2k}\log^{}R, (\mathcal{Q}^{\beta_{l'}+2})'\big)
		\end{align*}
		satisfying 
		\begin{align*}
			t^2(\Box n_N - f)\in \lambda^2 \frac{t^{2\nu k}}{(t\lambda)^{2(l+N)}}S^{\min\{m,2\}}\big(R^{2k}\log^{b + N +1}(R),  (\mathcal{Q}^{\beta_{l+N}})''\big)
		\end{align*}
		The same result holds replacing $  (\mathcal{Q}^{\beta_l})''$ by $ (\mathcal{Q}^{\beta_l})'$ and $  (\mathcal{Q}^{\beta_l})' $ by $  \mathcal{Q}^{\beta_l}$.
	\end{Cor}
	\begin{proof}[Proof of Lemma \ref{lem:Boxminuesonefirst}] We follow the two step procedure of \cite{KST1} outlined in Section \ref{sec: the wave para gd}. Using the definition of the space $S^m\big(R^{2k}\log^{b}(R), (\mathcal{Q}^{\beta_l})''\big)$, we may expand for large $ R \gg1 $ 
		\begin{align}\label{das-exp-hier-Lemma}
			t^2 \cdot f =&\;\; \lambda^{2} \frac{t^{2\nu k}}{(t\lambda)^{2l}} \sum_{p, r \geq 0} \sum_{0\leq j\leq b +  2r } (t \lambda)^{-2p} \cdot q_{p r j }(a)\cdot R^{ 2k - 2r} \cdot( \log(R))^j,\\ \nonumber
			=&\; \;\lambda^{2}  \frac{t^{2\nu k}}{(t\lambda)^{2l}} \sum_{ r \geq 0} \sum_{0\leq j\leq b +  2r } q_{0 r j }(a)\cdot R^{ 2k - 2r} \cdot( \log(R))^j\\ \nonumber
			& \;\;\;+ \lambda^{2}   \frac{t^{2\nu k}}{(t\lambda)^{2l + 2}} \sum_{p, r \geq 0} \sum_{0\leq j\leq b +  2r } (t \lambda)^{-2p} \cdot q_{p +1 r j }(a)\cdot R^{ 2k - 2r} \cdot( \log(R))^j.
		\end{align}
		The latter terms on the right are embedded into the space  $ S^m(R^{2k} \log^{b+1}(R), (\mathcal{Q}^{\beta_l})'')$, in particular adding and subtracting the global functions (in $R \geq 0$)
		\begin{align*}
			t ^2 \cdot f^{\text{mod}} = & \; \lambda^{2} \frac{t^{2\nu k}}{(t\lambda)^{2l + 2}}\; \chi_{ \geq R_0}(R) \cdot  \sum_{p, r \geq 0} \sum_{0\leq j\leq b +  2r } (t \lambda)^{-2p} \cdot q_{p +1 r j }(a) R^{ 2k-2r}  \log^j(R), % \hspace{3cm} \times \langle R \rangle ^{ 2k} \big( R\langle R \rangle^{-2}\big)^{2r} \cdot(\f12  \log(1 + R^2))^j
		\end{align*}
		where $ \chi_{\geq R_0}(R) $ is a suitable cut-off to the region $ R \gg R_0$ ( as in Definition \ref{defn:SQbetal}), we decompose  by Lemma \ref{lem:truncate}
		\begin{align*}
			t ^2 \cdot f =  \lambda^{2} \frac{t^{2\nu k}}{(t\lambda)^{2l}}S^m\big(R^{2k}\log^{b}(R), (\mathcal{Q}^{\beta_l})''\big) + \lambda^{2} \frac{t^{2\nu k}}{(t\lambda)^{2l + 2}} S^m\big(R^{2k}\log^{b+1}(R), (\mathcal{Q}^{\beta_l})''\big)
		\end{align*}
		where the part in the left space has in fact an expansion of the type of the first term on the right side of \eqref{das-exp-hier-Lemma} and the part in the right space already has the desired form of the error. We hence focus on the first term on the right of \eqref{das-exp-hier-Lemma}, for which we pick the leading order terms, i.e. %of $t^2\cdot E$, 
		\begin{align} \label{leading-hier}
			\lambda^2	\frac{t^{2\nu k}}{(t\lambda)^{2l}}\cdot \sum_{j=0}^{b}c_{0j}^{l k}(a)\cdot R^{2k}\log^j(R),\,c_{0j}^{l k}(a) = q_{0 0 j}(a) \in (\mathcal{Q}^{\beta_l})''.
		\end{align}
		We then further  write 
		\begin{align*}
			&\frac{t^{2\nu k}}{(t\lambda)^{2l}}\cdot  R^{2k} =  \frac{t^{2\nu k}}{(t\lambda)^{2l}}\cdot  a^{2k}\cdot (\lambda t)^{2k} = \frac{t^{k}}{(t\lambda)^{2l}}\cdot a^{2k},\\
			& a^{2k} \cdot (\mathcal{Q}^{\beta_l})'' \subset  (\mathcal{Q}^{\beta_{l}+ 2k})'',
		\end{align*}
		hence \eqref{leading-hier} reads
		\begin{align} \label{leading-hier-2}
			\lambda^2 	\frac{t^{ k}}{(t\lambda)^{2l}}\cdot \sum_{j=0}^{b}\tilde{c}_{0j}^{l k}(a) \log^j(R),\,\;\tilde{c}_{0j}^{l k}(a) \in (\mathcal{Q}^{\beta_l+ 2k})''.
		\end{align}
		We now apply the \emph{hyperbolic modifier} Step (2) as  explained in Section \ref{sec: the wave para gd}, i.e. we observe for  $\tilde{n} = \lambda^2\frac{t^{k}}{(t\lambda)^{2l}}\cdot g(a) = :t^{\beta}\cdot g(a)$, the operator \eqref{hyperbolic-op}, i.e. we obtain the relation
		\begin{align*}
			& \big(-\partial_{t}^2 + \partial_{r}^2 + \frac{3}{r}\partial_r\big)\tilde{n} = t^{\beta - 2}L_\beta g,\\
			&L_{\beta}g = (1-a^2)\partial_{a}^2v + [2(\beta - 1)a + \frac{3}{a}]\partial_a g + (-\beta^2 + \beta)g,\\
			&\beta = \beta_l + k =  (2l -2)\nu - l -1 + k. 
		\end{align*} 
		The ansatz for  $n = n_1 + n_2$ of the lemma thus has the form in this first step
		\begin{align} \label{leading-ansatz-hier}
			n_1 = \sum_{j=0}^{b}g_{rj}^{l k}(a)\cdot \lambda^2 \frac{t^{k}}{(t\lambda)^{2l}}\cdot \big(\f12\log (1 + R^2) \big)^j, 
		\end{align}
		and hence, considering the appearance of $L$ and the commutator terms in \eqref{main-eqn-L} when applying $\Box$  to \eqref{leading-ansatz-hier},
		this leads to  the system for the coefficients $\{ g_{rj}^{l k}(a) \}$
		\begin{align} \label{main-sys-L-1-hier}
			(L_{\beta_l + k }g_{rb}^{l k})(a) &= - c_{rb}^{l k}(a),\\[4pt]\label{main-sys-L-2-hier}
			(L_{\beta_l + k }g_{rb-1}^{l k})(a) &= - c_{r b-1}^{l k}(a) - b\big[(-1-2\nu) a (g_{rb}^{l k})'(a) - (2\nu -3) g_{rb}^{l k}(a)\\[4pt] \nonumber\label{main-sys-L-3-hier}
			&\;\;\;\;  + \frac{2}{a} \big( (g_{rb}^{l k})'(a) + \frac{1}{a} g_{rb}^{l k}(a)\big) + \frac{3}{a^2} g_{rb}^{l k}(a) - 2(\f12 + \nu)^2 g_{rb}^{l k}(a)\big] ,\\[4pt]
			(L_{\beta_l + k }g_{r j}^{l k})(a) &= - c_{r j}^{l k}(a) - (\tilde{b}+1)\big[(-1-2\nu) a (g_{r j+1}^{l k})'(a) - (2\nu -3) g_{r j+1}^{l k}(a)\\[4pt] \nonumber
			&\;\;\;\;  + \frac{2}{a} \big( (g_{r j+1}^{l k})'(a) + \frac{1}{a} g_{r j+1}^{l k}(a)\big) + \frac{3}{a^2} g_{r j+1}^{l k}(a) - 2(\f12 + \nu)^2 g_{r j+1}^{l k}(a)\big] ,\\[4pt]\nonumber
			& \;\;\;\; - (\tilde{b}+1)(\tilde{b}+2)\big[ \frac{1}{a^2}  g_{r  j+2}^{l k}  - (\f12 + \nu)^2 g_{r j+2}^{l k}(a)\big],\\[4pt]\nonumber
			&\;\;\; 0 \leq j \leq b-2.
		\end{align}
		%inductively by the system 
		%	\begin{align*}
			%	&L_{\beta}n_{rkl(2l+k-1)} = c_{rkl(2l+k-1)},\\
			%	&L_{\beta}n_{rklj} = c_{rklj} - (2\beta - 2a\partial_a)n_{kl(j+1)}\cdot (j+1)(\nu+\frac12)\\
			%	&\hspace{1cm} +n_{rkl(j+1)}\cdot (j+1)(\nu+\frac12) + n_{rkl(j+2)}(j+2)(j+1)(\nu+\frac12)^2\\
			%	&\hspace{1cm} + 2ja^{-1}\partial_a n_{rkl(j+1)} + 3ja^{-2} n_{rkl(j+1)} + (j+2)(j+1)a^{-2} n_{rkl(j+2)}
			%	\end{align*}
		%where $0\leq j<2l+k-1$ and we set $n_{rklj'} = 0$ for $j'>2l+k-1$. This 
		which is subsequently solved by Lemma \ref{lem:Inhom-L} and in particular, factoring off $ a^{2k} = R^{2k} (t\lambda)^{-2k}$, we have 
		\[
		t^2\big(\Box n_1 - f\big)\in \lambda^2 \frac{t^{2\nu k }}{(t\lambda)^{2l}}S^{m}\big(R^{2k-2}\log^{b}(R), (\mathcal{Q}^{\beta_l})''\big) =:t^2 f_1,\;\;\; 
		\]
		Next we apply the \emph{elliptic modifier} Step (1) in Section \ref{sec:elliptic-modifier} to this error and  determine $n_2$ by means of the variation of constants formula 
		\begin{align*}
			n_2 = (\lambda t)^{-2}R^{-2}\int_0^R \frac{s^3}{4}\cdot t^2f_1(t, s)\,ds -(\lambda t)^{-2}\int_0^R \frac{s}{4}t^2f_1(t, s)\,ds,
		\end{align*}
		which implies that we have the relation 
		\begin{align*}
			t^{2}\Big(\Box\big(n_1 + n_2\big) - f_1\Big) = -t^2\partial_{t}^2n_2.
		\end{align*}
		Application of Lemma~\ref{lem:recoverz1} hence implies that 
		\begin{align*}
			t^2\partial_{t}^2 n_2\in  \lambda^2 \frac{t^{2\nu k }}{(t\lambda)^{2l+2}}S^m\big(R^{2k}\log^{b +1}(R), (\mathcal{Q}^{\beta_l})''\big).
		\end{align*}
	\end{proof}
	\begin{proof}[Proof of Corollary \ref{cor:Boxminuesonefirst}] The first step is of course to apply Lemma \ref{lem:Boxminuesonefirst}, then we have accordingly
		\begin{align*}
			&n_1 \in  \lambda^2\frac{t^{2\nu k}}{(t\lambda)^{2l}}S^2\big(R^{2k}\log^{b}(R), (\mathcal{Q}^{\beta_l})'\big) +  \lambda^2 \frac{t^{2\nu k}}{(t\lambda)^{2l+2}}S^{2}\big(R^{2k}\log^{b+1}(R), \mathcal{Q}^{\beta_l}\big),\\
			&t^2(\Box n_1 - f)\in   \lambda^2\frac{t^{2\nu k}}{(t\lambda)^{2l+2}}S^{\min\{m,2\}}\big(R^{2k}\log^{b+1}(R), (\mathcal{Q}^{\beta_l})''\big).
		\end{align*}
		Thus we subsequently apply Lemma \ref{lem:Boxminuesonefirst} to $ f_j : = \Box (n_1 + n_2 + \dots + n_j ) - f$,\; $1 \leq j \leq N $  and note that starting with
		\[
		t^2 f_1 \in  \lambda^2\frac{t^{2\nu k}}{(t\lambda)^{2l+2}}S^{\min\{m,2\}}\big(R^{2k}\log^{b+1}(R), (\mathcal{Q}^{\beta_l})''\big),
		\]
		the  $ L_{\beta}$ operator in the hyperbolic step has $ \beta = \beta_{l+1} + k$  and we replace  $ (\mathcal{Q}^{\beta_l})''$ by $(\mathcal{Q}^{\beta_{l+1}})''$ when integrating the system \eqref{main-sys-L-1-hier}- \eqref{main-sys-L-3-hier}. We henceforth need  to proceed likewise in the higher order steps.
	\end{proof}
	\;\\
	The preceding lemma is not quite enough to deal with the source terms, such  as for instance
	\begin{align*}
		\lambda^{-2}\Box^{-1}\partial_t^2\big(2\lambda^2\text{Re}(\bar{z}W)\big)\cdot W,
	\end{align*}
	and similar terms involving $\Box^{-1}$. In particular for terms independent of $ a $, i.e.
	$$ z  \in \frac{t^{2\nu k}}{(t \lambda)^{2l}}\cdot S^m(R^{2k}\log^{s_{\ast}}(R)),\;\; l = 0,1,$$
	the  expression 
	\[
	\partial_t^2\big(2\lambda^2\text{Re}(\bar{z}W)\big)
	\]
	is not of the type in the preceding Lemma \ref{lem:Boxminuesonefirst}. To deal with this,we use the following simple lemma.
	\begin{Lemma}\label{lem:Boxminuesonesecond} Let  $k\geq 1, b \geq 0, s \in \Z$ and 
		\[
		t^2\cdot f\in \lambda^2\cdot \frac{t^{2\nu k}}{(t \lambda)^{2l}}\cdot S^m(R^{2k-2s}\log^b(R)),\;\; l \geq 0.
		\]
		Then there exists a function $ n$ with 
		\begin{align}
			&n\in \lambda^2\cdot \frac{t^{2\nu k}}{(t\lambda)^{2l+2}}\cdot S^{m+2}(R^{2k-2s +2}\log^{b+1}(R)),\;\;\text{if}\; 2k -2s \geq -2,\\
			&n\in \lambda^2\cdot \frac{t^{2\nu k}}{(t\lambda)^{2l+2}}\cdot S^{m+2}(R^{-2}\log^{b+1}(R)),\;\;\text{if}\; 2k -2s < -2,\
		\end{align}
		such that 
		\begin{align*}
			&t^2(\Box n - f)\in \lambda^2\frac{t^{2\nu k}}{(t\lambda)^2}\cdot S^{m+2}(R^{2k-2}\log^{b+1}(R)),\;\;\;\text{or},\;\;\\
			&t^2(\Box n - f)\in \lambda^2\frac{t^{2\nu k}}{(t\lambda)^2}\cdot S^{m+2}(R^{-2}\log^{b+1}(R)).
		\end{align*}
	\end{Lemma}
	\begin{proof} The second assertion is proved similarly to the first\footnote{One replaces the second integral in the variation of constants formula by $-\int_R^\infty$.}, in fact the Lemma is a direct consequence of Lemma \ref{lem:Lemma-n-Anfang-elliptic}. The first case for instance follows by solving 
		\begin{align*}
			\big(\partial_{rr} + \frac{3}{r}\partial_r\big)n = f
		\end{align*}
		by passing to the variable $R = \lambda r$ and using the variation of constants formula
		\begin{align*}
			n(t, R) = R^{-2}\int_0^R \frac{s^3}{4}f(t,s)\,ds - \int_0^R \frac{s}{4}f(t,s)\,ds\in \lambda^2 \frac{t^{2\nu k}}{(t\lambda)^{2l +2}}\cdot S^m(R^{2k-2s +2}\log^{b+1}(R))
		\end{align*}
		Reverting to the original coordinates $(t, r)$, one verifies that 
		\[
		t^2(\Box n - f) = -t^2\partial_{tt}n \in \lambda^2\frac{t^{2\nu k}}{(t\lambda)^{2l +2}}\cdot S^m(R^{2k-2s +2}\log^{b+1}(R)).
		\]
	\end{proof}
	\begin{rem} In both Lemma \ref{lem:Boxminuesonefirst} and Lemma \ref{lem:Boxminuesonesecond} we actually obtain $\log^b(R)$ instead of $\log^{b+1}(R)$ as a leading logarithmic power (according to Definition \ref{defn:SQbetal}) unless $ k = -1,-2$ in Lemma \ref{lem:Boxminuesonefirst}  or $ 2k -2s = -2,-4$ in Lemma \ref{lem:Boxminuesonesecond} . Similarly the logarithmic power in the S-space need only be raised in Corollary \ref{cor:Boxminuesonefirst} for the case of $ k =-1,-2$.  However, for simplicity, we only do make this distinction if necessary and otherwise use the embedding of Lemma \ref{lem:embedding}. 
	\end{rem}
	%	\begin{rem}\label{rem::Boxminuesonesecond} Observe that for $k\geq 1$ the space 
		%	\[
		%	\lambda^2\frac{t^{2\nu k}}{(t\lambda)^2}\cdot S^m(R^{2k-2}\log^CR) = \frac{t^{2\nu (k-1) - 1}}{(t\lambda)^2}\cdot S^m(R^{2k-2}\log^CR)
		%	\]
		%	is as in Lemma~\ref{lem:Boxminuesonefirst}, Corollary~\ref{cor:Boxminuesonefirst}.
		%	\end{rem}
	\subsection{The inductive $z$-iteration}\label{subsec:inductive-ite-z}
	We now perform the induction step to approximately solve \eqref{main-eq-z}, i.e. the equation
	\begin{align}\label{main-eq-z-2} 
		- \Delta z - W^2(R)z -2 \text{Re}(\bar{z}W)W =&\;  \lambda^{-2} \big( \Box^{-1} \partial_t^2 \big(\lambda^2W^2\big) \big) W +  ( \alpha_0 t^{2 \nu}  - i( \f12 + \nu)t^{2\nu}(1 + R\partial_R))W \nonumber \\[4pt] 
		& +\lambda^{-2} \big(\Box^{-1} \partial_t^2 \big(\lambda^2W^2\big)\big) z\\[4pt] \nonumber
		&  + \lambda^{-2}\Box^{-1} \partial_t^2( 2 \lambda^2  \text{Re}(\bar{z}W)  ) W + \lambda^{-2}\Box^{-1} \partial_{t}^2 ( 2 \lambda^2  \text{Re}(\bar{z}W)  ) z\\[4pt]  \nonumber
		& +  2 \lambda^2  \text{Re}(\bar{z}W)  ) z+\lambda^{-2}\Box^{-1} \triangle (\lambda^2 |z|^2  ) ( z + W)  + i t^{1 + 2 \nu}\partial_t z\\[4pt]  \nonumber
		& -  ( \alpha_0 t^{2 \nu}  +  i( \f12 + \nu)t^{2\nu}(1 + R\partial_R))z.
	\end{align}
	Then, as demonstrated in the beginning of this section, the ion potential $n$ is given by 
	\begin{align*}
		n =&\;\; \lambda^2W^2 +  2 \lambda^2  \text{Re}(\bar{z}W) + \Box^{-1} \partial_{t}^2 ( 2 \lambda^2  \text{Re}(\bar{z}W)\\[2pt]
		& \;\;\;+  \Box^{-1} \partial_t^2 \big(\lambda^2W^2\big) + \Box^{-1} \triangle ( \lambda^2 |z|^2  ).
	\end{align*}
	if the function  $z$ solves  \eqref{main-eq-z-2}.\\[4pt]
	\emph{The corrections}. In order to construct an approximate solution $z$, we add increments, i. e. we set 
	\begin{align}
		z^{\ast}_N =&\;\; \sum_{j=0}^N z_j = z_1 + z_2 + \dots  + z_N
	\end{align}
	and likewise
	\begin{align}
		n^{\ast}_N =&\;\;\lambda^2 W^2 +   \Box^{-1} \partial_t^2 \big(\lambda^2W^2\big)+\sum_{j=0}^N n_j\\ \nonumber
		=&\;\;  \lambda^2 W^2 +   \Box^{-1} \partial_t^2 \big(\lambda^2W^2\big)+ n_1  + n_2+ \dots + n_N,
	\end{align}
	where 
	\begin{align*}
		n_j = 2 \lambda^2  \text{Re}(\overline{z_j}W) + \Box^{-1} \partial_{tt} ( 2 \lambda^2  \text{Re}(\overline{z_j}W))  + \Box^{-1} \triangle ( \lambda^2 \triangle_j|z|^2  ) 
	\end{align*}
	and the last term on the right is defined as follows: 
	\begin{equation}\label{eq:differencing}
		\triangle_j|z|^2 = \big|\sum_{k\leq j}z_k\big|^2 -  \big|\sum_{k<j}z_k\big|^2,\;\; z_0 = 0.
	\end{equation}
	The precise inductive definition of the $z_j$ is then given by the following 
	\boxalign[13cm]{
		\begin{align}\label{eq:trianglez-inductivedef}
			&- \Delta (z_1) - W^2(R)z_1 -2 \text{Re}(\overline{z_1}W)W\\[2pt] \nonumber
			&\hspace{2cm} =  \lambda^{-2} \big( \Box^{-1} \partial_t^2 \big(\lambda^2W^2\big) \big) W +  ( \alpha_0 t^{2 \nu}  - i( \f12 + \nu)t^{2\nu}(1 + R\partial_R))W,\\[8pt] \label{eq:trianglez-inductivedef-2}
			&- \Delta(z_j)- W^2(R)z_j -2 \text{Re}(\overline{z_j}W)W\\[2pt] \nonumber
			&\hspace{2cm} = \lambda^{-2} \big(\Box^{-1} \partial_t^2 \big(\lambda^2W^2\big)\big) z_{j-1} +  \lambda^{-2}\Box^{-1} \partial_t^2( 2 \lambda^2  \text{Re}(\overline{z_{j-1}}W)  ) W\\[2pt] \nonumber
			&\hspace{2cm}\;\;+ i t^{1 + 2 \nu}\partial_t z_{j-1} -  ( \alpha_0 t^{2 \nu}  +  i( \f12 + \nu)t^{2\nu}(1 + R\partial_R))z_{j-1}\\[2pt] \nonumber
			&\hspace{2cm}\;\; + \triangle_{j-1}\Big[\lambda^{-2}\Box^{-1} \partial_{t}^2 ( 2 \lambda^2  \text{Re}(\bar{z}W)  ) z +  2 \lambda^2  \text{Re}(\bar{z}W)  ) z +  \lambda^{-2}\Box^{-1} \Delta (\lambda^2 |z|^2  ) ( z + W) \Big],\\ \nonumber
			&\hspace{2cm}\;\; j \geq 2,
		\end{align}
	}
	where we define, analogous to \eqref{eq:differencing}, for any nonlinear expression $F$
	\[
	\tri{j-1}F(z) = F\big( \sum_{k \leq j-1} z_k\big) - F\big( \sum_{k < j-1} z_k\big),\;\; z_0 = 0. 
	\] 
	and similar for $n$ (or both $(z,n)$) where $n_0 = 0$.
	\begin{Rem} We note that by definition  of \eqref{eq:trianglez-inductivedef} - \eqref{eq:trianglez-inductivedef-2}, the iteration scheme only requires us to iterate $z_1, \dots, z_N$ and we obtain $n^{\ast}_N, n_j$  after the iteration, in fact 
		\begin{align*}
			n_N^{\ast} & = \Box^{-1}(\triangle (\lambda^2 W^2)) + \sum_{j = 1}^N n_j\\
			& =  \Box^{-1}(\triangle (\lambda^2 W^2)) + \Box^{-1}(\triangle (2 \text{Re}(\lambda^2 \overline{z_N^{\ast}} \cdot W)) + \Box^{-1}(\triangle (\lambda^2 |z_N^{\ast}|^2))\\
			& = \Box^{-1}(\triangle(\lambda^2 | W + z_N^{\ast}|^2)).
		\end{align*}
	\end{Rem}
	\;\\
	Further, calculating $ N \in \Z_+$ steps of this iteration scheme, we have $ n_N^{\ast}, z_N^{\ast}$ as above and thus set the \emph{error functions} (or \emph{remainder}) to be
	\begin{align*}
		&e^z_N(t,R) := 	i t^{1 + 2 \nu}\partial_t  u_N^{\ast} + \Delta u_N^{\ast}  +  ( \alpha_0 t^{2 \nu}  - i( \f12 + \nu)t^{2\nu}(1 + R\partial_R))u_N^{\ast}\\ 
		& \hspace{2cm} + \lambda^{-2} n_N^{\ast}\cdot u_N^{\ast},\;\;\ u_N^{\ast} =  W + z^{\ast}_N ,\\[4pt]
		& e^n_N(t,R) : = \Box(n_N^{\ast}) - \Delta( \lambda^2|u_N^{\ast}|^2),\;\;\; n_N^{\ast} = \Box^{-1}  \Delta( \lambda^2|u_N^{\ast}|^2),
	\end{align*}
	where the latter denotes the wave parametrix. Thus in particular, by definition and \eqref{eq:trianglez-inductivedef-2}, we infer
	\begin{align} \label{thatstheerror}
		e^z_N(t,R) =&\; i t^{1 + 2 \nu} \cdot \partial_t  z_N  +   t^{2 \nu} \cdot ( \alpha_0 - i( \f12 + \nu)\Lambda)z_N\\ \nonumber
		& \;\; + \lambda^{-2}\Box^{-1}( \lambda^2 \Delta |u_N^{\ast}|^2) u_N^{\ast} -  \lambda^{-2}\Box^{-1}( \lambda^2\Delta |u_{N-1}^{\ast}|^2) u_{N-1}^{\ast}.
	\end{align}
	Next, we need to define spaces capturing the form of $ z_j, n_j$ and the respective errors $ e^z_{j}, e^n_j$ along the iteration. Therefore we consider the following.
	\begin{Def} \label{defn:X-space} For $ j \in \Z_{+},\; \ell \in \Z_{\geq 0}$ and $M_j > 0$ we let $I_j: = [j, \infty) \cap \Z$ and define the space of linear combinations %\footnote{We use the notation `linear hull'}
		\begin{align*}
			X_j^{\ell}: &= \text{Span}\big \langle  \frac{t^{2\nu k}}{(t \lambda)^{2l}}\cdot S^2\big(R^{2k-2 - \ell}\log^{s_{\ast}(l)}(R)\big)\;|\;\;k + l\in I_j,\; l\in \{0,1\}, 0 \leq k \leq M_j\big \rangle\\[2pt]
			&\;\;\;\;\; +  \text{Span}\big\langle \frac{t^{2\nu k}}{(t\lambda)^{2l}}\cdot S^2\big(R^{2k - \ell}\log^{s_{\ast}(l)}(R), \mathcal{Q}^{\beta_l}\big)\;|\;\;k+l\in I_j,\,l>1, 0 \leq k \leq M_j \big\rangle,
		\end{align*}
		and set for simplicity $ X_j : = X_j^0$. Further we let  $ (X_j^{\ell})'$ and $ (X_j^{\ell})''$ be defined likewise with $ \mathcal{Q}^{\beta_l}$  being replaced by $ (\mathcal{Q}^{\beta_l})'$ and $ (\mathcal{Q}^{\beta_l})''$ respectively.  %Let us also for simplicity set
		%$
		%	X^{\ell}_{\geq j} := \bigcup_{k \geq j} X^{\ell}_{k },\;\; %X^{\ell}_{< j} : = X^{\ell}_{\geq 1}  - X^{\ell}_{\geq j},\; 
		%	$
		%	and similarly  $(X^{\ell}_{\geq j} )',\; (X^{\ell}_{\geq j} )''$.
	\end{Def}
	\begin{Rem} In the above definition we keep the upper bound $ 0 \leq k \leq M_j $ implicit since it is not important. Concerning $ z_j$, factor $ t^{2\nu k}$ can only  grow on the right of \eqref{eq:trianglez-inductivedef-2} multiplying by $ t^{2\nu}$  or in the interaction part. Precisely, if we intend to place $ z_j \in X_j$,  we can take $M_j = 3^{j-1}$ which follows from the  cubic interaction source terms.  We intend to place  $n_j \in \lambda^2 X_j$, where we then restrict accordingly to $0 \leq k \leq  M_j^2$.
	\end{Rem}
	\begin{Rem}\label{rem:nor2}
		In the $ j^{\text{th}}$-step of the $z$ iteration, each $ k \geq 0$ as above is associated to a $t^{2 \nu k}$ factor in the source term of \eqref{eq:trianglez-inductivedef-2} for determining $z_{j+1}$.  We may choose explicit \emph{length parameters} $N_{j k} \gg1 $ in $ \Z_+$ in order to control the error of the $ \Box^{-1}(\cdot)$ parametrix applying to such source terms. Thus with $ N_{jk} \geq \max\{j -k, 0\}$ we can replace $ I_{j k} : = [j -k , N_{j k}]\cap \Z$ in Definition \ref{defn:X-space}. Further we may write the terms in the first span with $ l =1$ into the second space, i.e. we rewrite
		\begin{align}
			%	X_j &= \sum_{k \leq M_j}\sum_{l = 0,1} \frac{t^{2\nu k}}{(t \lambda)^{2l}}\cdot S^2\big(R^{2k-2}\log^{s_{\ast}(l)}(R)\big)\\ \nonumber
			%	&\hspace{2cm}\;+ \sum_{k \leq M_j} \sum_{l  \in I_{j k}}  \frac{t^{2\nu k}}{(t\lambda)^{2l}}\cdot S^2\big(R^{2k}\log^{s_{\ast}(l)}(R), \mathcal{Q}^{\beta_l}\big).
			X_j^{\ell}: &= \text{Span}\big \langle  t^{2\nu k}\cdot S^2\big(R^{2k-2 - \ell}\log(R)\big)\;|\; j \leq k \leq M_j\big \rangle\\[2pt]
			&\;\;\;\;\; +  \text{Span}\big\langle \frac{t^{2\nu k}}{(t\lambda)^{2l}}\cdot S^2\big(R^{2k - \ell}\log^{s_{\ast}(l)}(R), \mathcal{Q}^{\beta_l}\big)\;|\;\;l\in I_{j k},\,l \geq 2, 0 \leq k \leq M_j \big\rangle,
		\end{align}
		which is done via the embedding for $ \mathcal{Q}^{\beta_1}$ of the form (exchanging $ a^2 = R^2 (t \lambda)^{-2}$)
		\begin{align*}
			\frac{ t^{2\nu k}}{(t \lambda)^2}\cdot S^2\big(R^{2k-2 - \ell}\log^3(R)\big) & \subset 	\frac{ t^{2\nu k}}{(t \lambda)^2}\cdot S^2\big(R^{2k-2 - \ell}\log^3(R), a^2 \mathcal{Q}^{\beta_1}\big)\\
			& \subset \frac{ t^{2\nu k}}{(t \lambda)^4}\cdot S^2\big(R^{2k - \ell}\log^3(R), \mathcal{Q}^{\beta_2}\big).
		\end{align*}
		We keep the choice of $ N_{jk} \in \Z_+$  implicit in the $X^{\ell}_j $ space of Definition  \ref{defn:X-space}, however it will be useful in the subsequent section to choose an upper bound, say if $ j = 1,2, \dots, N$ we determine $ \mathcal{N} \gg1$ such that $  \sup_{j,k} N_{j k} \leq \mathcal{N}$.
	\end{Rem}
	\;\\
	\emph{The iteration scheme - The first corrections}. We start by solving \eqref{eq:trianglez-inductivedef} in order to determine $z_1$, where we claim in particular
	\[
	z_1 \in X_1,\; n_1 \in \lambda^2  X_1.
	\]The source terms on the right of \eqref{eq:trianglez-inductivedef} are given by 
	\[
	\lambda^{-2} \big( \Box^{-1} \partial_t^2 \big(\lambda^2W^2\big) \big) W,\;\; t^{2\nu}\big( \alpha_0 W   - i( \f12 + \nu)\Lambda W\big)
	\]
	First we proceed as in Section \ref{subsec:initial} and approximate $ \Box^{-1} \big( \partial_t^2 \big(\lambda^2W^2\big) \big)$. In fact, as explained in Section \ref{subsec:initial}, after invoking two \emph{elliptic modifier} steps and the first \emph{hyperbolic/elliptic modifier} procedure as in Lemma \ref{lem:Boxminuesonefirst}, we arrive at the terms
	\begin{align*}
		&\frac{\lambda^2}{(t\lambda)^2} S^2(R^{-2}\log(R) ) + \frac{\lambda^2}{(t\lambda)^{4}} S^4(\log^2(R) ) + \frac{\lambda^2}{(t\lambda)^{4}} S^2(\log^2(R), \mathcal{Q}^{\beta_{2}})\\
		&\;\;\;\; +   \frac{\lambda^2}{(t\lambda)^{6}} S^2(\log^2(R) , \mathcal{Q}^{\beta_{2}}).
	\end{align*}
	Then using in  particular Lemma \ref{lem:general-initial}, which is a conclusion of Corollary~\ref{cor:Boxminuesonefirst}, we have 
	\begin{align*}
		\tilde{z} \in&  \frac{\lambda^2}{(t\lambda)^2} S^2(R^{-2}\log(R) ) + \frac{\lambda^2}{(t\lambda)^{4}} S^4(\log^2(R) ) \\
		&\;\;+ \frac{\lambda^2}{(t\lambda)^{4}} S^2(\log^2(R), \mathcal{Q}^{\beta_{2}}) +   \frac{\lambda^2}{(t\lambda)^{6}} S^2(\log^2(R) , \mathcal{Q}^{\beta_{2}})\\
		&\;\; + \sum_{\tilde{N} \geq l \geq 3} \bigg[ \frac{\lambda^2}{(t\lambda)^{2l}} S^2(\log^{l-1}(R) , (\mathcal{Q}^{\beta_{l}})') +  \frac{\lambda^2}{(t\lambda)^{2l +2}} S^2(\log^l(R) , \mathcal{Q}^{\beta_{l}})\bigg],
	\end{align*}
	for some $ \tilde{N} \in \Z_+$ and such that
	\[
	\Box \tilde{z} -  \partial_t^2 \big(\lambda^2W^2\big) \in \frac{ t^{-2} \lambda^2}{(t \lambda)^{2N + 2}}S(\log^N(R), \big(\mathcal{Q}^{\beta_{N}}\big)'').
	\]
	Note we then set $N_{1,0} : = \tilde{N}$ as explained in Remark \ref{rem:nor2}. Therefore we need to integrate 
	$$ f(t,R) = \lambda^{-2} \tilde{z}(t,R) \cdot W(R) + t^{2\nu} \alpha_0 W(R) - t^{2\nu}i(\nu + \f12)\Lambda W(R)$$
	using Lemma \ref{lem:recoverz1} and Lemma \ref{lem:Lemma-n-Anfang-elliptic}. Thus we obtain using also the embedding in Lemma \ref{lem:embedding}
	\begin{align}
		z_1 & \in \frac{1}{(t\lambda)^2} S^2(\log(R)) + \frac{1}{(t \lambda)^4} S^2(\log^3(R), \mathcal{Q}^{\beta_2})\\ \nonumber
		&\;\; + \sum_{ N \geq l \geq 3} \bigg[ \frac{\lambda^2}{(t\lambda)^{2l}} S^2(\log^{l}(R) , \mathcal{Q}^{\beta_{l}}) +  \frac{\lambda^2}{(t\lambda)^{2l +2}} S^2(\log^{l+1}(R) , \mathcal{Q}^{\beta_{l}})\bigg] + t^{2\nu} S^2(\log(R)),
	\end{align}
	where we used again Lemma \ref{lem:embedding} for the first term 
	\[
	\frac{1}{(t\lambda)^2} S^2(R^{-2}\log^2(R)) \subset  \frac{1}{(t\lambda)^2} S^2(\log(R)),
	\]
	hence we obtain $z_1 \in X_1$. For verifying $ n_1 \in \lambda^2  X_1$, we essentially only have to calculate the approximations of
	\[
	\square^{-1} \triangle (\lambda^2 |z_1|^2),\;\;\;  \square^{-1} \triangle (\lambda^2 z_1 \cdot W(R)).
	\]
	For the left source terms, we may  first use  Lemma \ref{lem:basicproduct} and on the right we apply one elliptic step as in the Remark \ref{rem:elliptic-steps-first} below Lemma \ref{lem:Boxminuesonefirst}. Then in both cases the claim follows again by the use of Corollary \ref{cor:Boxminuesonefirst} and Lemma \ref{lem:embedding}.\\
	\;\;\\
	\emph{The iteration scheme - The inductive step}.
	In order to complete the induction for the construction of $z_j$, we need the following lemma:
	\begin{Lemma}\label{lem: induct} For any $ j \in \Z_+$ we assume $z_k \in X_k $ for all $ 1 \leq k \leq j-1$, then the equation \eqref{eq:trianglez-inductivedef-2} has a unique solution with $ z_j(t,0) = \partial_R z_j(t,0) = 0$ (for fixed $ t > 0$) and in particular
		\begin{equation}\label{eq:trianglezeroz}
			z_j \in X_j,\;\; e_j^z \in \cup_{k \geq j+1}(X^{2}_{k})'.
		\end{equation}
		Further we have $ n_j \in \lambda^2  (X_j)'$ and $ e_j^n \in \cup_{k \geq j+1} \lambda^2  (X_{k})''$.
	\end{Lemma}
	\begin{proof}[Proof of  Lemma \ref{lem: induct}]
		We need to treat the contributions of the various terms on the right hand side of \eqref{eq:trianglez-inductivedef-2} assuming that $z_k\in X_{k}$ for $0 \leq k \leq j-1$.  Let us start with the first two lines, i.e. all terms depending only on $ z_{j-1} \in X_{j-1}$.\\[4pt]
		\emph{\underline{Case \rom{1}}: The term $\lambda^{-2} \big(\Box^{-1} \partial_t^2 \big(\lambda^2W^2\big)\big) z_{j-1}$}. As before we replace $\Box^{-1} \partial_t^2 \big(\lambda^2W^2\big)$ by a term in the space (where we Lemma \ref{lem:embedding})
		\begin{align*}
			&  \frac{\lambda^2}{(t\lambda)^2} S^2(R^{-2}\log(R) ) + \frac{\lambda^2}{(t\lambda)^{4}} S^2(\log^2(R))\\
			&\;\; + \sum_{\tilde{N} \geq l \geq 2} \bigg[ \frac{\lambda^2}{(t\lambda)^{2l}} S^2(\log^{l}(R) , (\mathcal{Q}^{\beta_{l}})') +  \frac{\lambda^2}{(t\lambda)^{2l +2}} S^2(\log^{l+1}(R) , \mathcal{Q}^{\beta_{l}})\bigg].
		\end{align*}
		Hence using Lemma \ref{lem:basicproduct} and the assumption, we infer that
		\[
		\lambda^{-2} \big(\Box^{-1} \partial_t^2 \big(\lambda^2W^2\big)\big) z_{j-1}
		\]
		is in the span of terms listed as follows where in the absence of $l'$ we have $k\in I_{j-1}$ while in the presence of $l'$ we have $k+l'\in I_{j-1}$ and $ l, l'\geq 2$
		\begin{align*}
			&\frac{t^{2\nu k}}{(\lambda t)^2}\cdot S^2(R^{2k-4}\log^{2}(R)),\; \frac{t^{2\nu k}}{(\lambda t)^{4}}\cdot S^2(R^{2k-2}\log^{3}(R)),\\[3pt]
			&\frac{t^{2\nu k}}{(\lambda t)^{2l' +2}}\cdot S^2(R^{2k-2}\log^{2l'}(R), \mathcal{Q}^{\beta_{l'}}),\,\frac{t^{2\nu k}}{(\lambda t)^{2(l'+2)}}\cdot S^2(R^{2k}\log^{2l'+1}(R), \mathcal{Q}^{\beta_{l'}}), \;\;l' \geq 1,\\[3pt]
			& \frac{t^{2\nu k}}{(t\lambda)^{2l}} S^2(R^{2k-2}\log^{l+1}(R) , (\mathcal{Q}^{\beta_{l}})'),\;  \frac{t^{2\nu k}}{(t\lambda)^{2l +2}} S^2(R^{2k -2}\log^{l+2}(R) , \mathcal{Q}^{\beta_{l}}),\\[3pt]
			& \frac{t^{2\nu k}}{(t\lambda)^{2l + 2l'}} S^2(R^{2k}\log^{l + 2l'-1}(R) , (\mathcal{Q}^{\beta_{l + l'} -2\nu -1})'),\;  \frac{t^{2\nu k}}{(t\lambda)^{2l + 2l'+2}} S^2(R^{2k}\log^{l + 2l'}(R) , \mathcal{Q}^{\beta_{l + l'} -2}),
		\end{align*}
		where for simplicity we have embedded the terms in $X_{j-1}$ corresponding to $l=1$ 
		\[
		\frac{t^{2\nu k}}{(t \lambda)^{2}} S^2(R^{2k -2} \log^3(R)) \subset \frac{t^{2\nu k}}{(t \lambda)^{4}} S^2(R^{2k } \log^3(R), \mathcal{Q}^{\beta_2}),
		\]
		as indicated in the above remark, hence these are contained in the second and fourth line (with $ l' =2$). Thus applying the variations of constants formula as in Lemma~\ref{lem:recoverz1}, we arrive at functions in the span of \footnote{For the last term we use that $a^2\cdot \mathcal{Q}^{\beta-2}\subset \mathcal{Q}^{\beta}$.} 
		\begin{align*}
			&\frac{t^{2\nu k}}{(\lambda t)^2}\cdot S^2(R^{2k-2}\log^{3}(R)),\, \frac{t^{2\nu k}}{(\lambda t)^{4}}\cdot S^2(R^{2k}\log^{3}(R)),\\
			&\frac{t^{2\nu k}}{(\lambda t)^{2l' +2}}\cdot S^2(R^{2k}\log^{2l' +1}(R), \mathcal{Q}^{\beta_{l'}}),\,\frac{t^{2\nu k}}{(\lambda t)^{2(l'+1)}}\cdot S^2(R^{2k}\log^{2l'+1}(R), \mathcal{Q}^{\beta_{l'+1}}),\\
			&\frac{t^{2\nu k}}{(\lambda t)^{2l }}\cdot S^2(R^{2k}\log^{l+1}(R), \mathcal{Q}^{\beta_{l}}),\,\frac{t^{2\nu k}}{(\lambda t)^{2(l+1)}}\cdot S^2(R^{2k}\log^{l+2}(R), \mathcal{Q}^{\beta_{l}}),\\
			&\frac{t^{2\nu k}}{(\lambda t)^{2(l'+l -1)}}\cdot S^2(R^{2k}\log^{2l'+ 2l -3}(R), \mathcal{Q}^{\beta_{l' + l -1}}),\,\frac{t^{2\nu k}}{(\lambda t)^{2(l'+l)}}\cdot S^2(R^{2k}\log^{2l'+2l -1}(R), \mathcal{Q}^{\beta_{l'+l}}),
		\end{align*}
		where for the second line we may then use the embedding into
		\[
		\frac{t^{2\nu k}}{(\lambda t)^{2l' +2}}\cdot S^2(\log^{2l' +1}(R), \mathcal{Q}^{\beta_{l' +1}}).
		\]
		In order to obtain the above representation we used the following. For the second term in the second line we exchanged $ R^2 = (t\lambda)^2 a^2 $ and used $ a^2 \cdot \mathcal{Q}^{\beta_{l'}} \subset \mathcal{Q}^{\beta_{l' +1}}$, to be precise
		\begin{align*}
			\frac{t^{2\nu k}}{(\lambda t)^{2(l'+2)}}\cdot S^2(R^{2k+2}\log^{2l'+1}(R), \mathcal{Q}^{\beta_{l'}}) &\subset 	\frac{t^{2\nu k}}{(\lambda t)^{2(l'+1)}}\cdot S^2(R^{2k}\log^{2l'+1}(R), a^2 \cdot\mathcal{Q}^{\beta_{l'}})\\
			&\subset \frac{t^{2\nu k}}{(\lambda t)^{2(l'+1)}}\cdot S^2(R^{2k}\log^{2l'+1}(R), \mathcal{Q}^{\beta_{l' +1}}).
		\end{align*}
		For the last line we used $ R^2 = (t\lambda)^2 a^2 $ for both terms and for the first also $ \beta_{l + l'} - 2\nu -1 = \beta_{l + l' -1} -2$ (below we make the latter more precise) as well as  
		\begin{align*}
			&2l' + l \leq 2 l' + 2l -1,\;\;\text{since} \;\; l \geq 1,\;\;\text{ and} \\
			& 2l' + l -1 \leq 2 l' + 2l -3\;\;\text{ since} \; \;l \geq 2.
		\end{align*}
		Note for the very first term in the first and second line, we obtain  $\log^{3}(R)$, respectively $ \log^{2l' +1}(R)$, only if  $ k =1$ (or $ k =0$ respectively) and otherwise one logarithmic power less, which is however not important.
		Thus this linear combination lies in  $X_j$.\\[6pt]
		\emph{\underline{Case \rom{2}}: The term $ \lambda^{-2}\Box^{-1} \partial_t^2( 2 \lambda^2  \text{Re}(\overline{z_{j-1}}W)  ) W$.} Using the definition of $X_{j-1}$, we see that the function 
		$$ \partial_t^2( 2 \lambda^2  \text{Re}(\overline{z_{j-1}}W) $$
		is in the linear span of the sets
		\begin{align*}
			&\big \{t^{-2}\cdot\lambda^2\cdot t^{2\nu k}S^2\big(R^{2k-4}\log( R)\big) \;|\; k\in I_{j-1}\big\}\\
			&\big \{t^{-2}\cdot\lambda^2\cdot \frac{t^{2\nu k}}{(\lambda t)^{2l}}S^2\big(R^{2k-2}\log^{2l-1} (R), (\mathcal{Q}^{\beta_l})''\big)\;|\; k+l\in I_{j-1}, l \geq 2\big\}. 
		\end{align*}
		We apply Lemma \ref{lem:Boxminuesonesecond}, respectively Remark \ref{rem:elliptic-steps-first}, and then Corollary~\ref{cor:Boxminuesonefirst}, resulting in an $N$-approximation of 
		\[
		\Box^{-1} \partial_t^2( 2 \lambda^2  \text{Re}(\overline{z_{j-1}}W),
		\]
		which belongs to the linear span of the terms in \eqref{eq:spanrelation1} and \eqref{eq:spanrelation2} with
		\begin{align}\label{eq:spanrelation1}
			&\lambda^2\cdot\frac{t^{2\nu k}}{(\lambda t)^2}\cdot S^4(R^{2k-2}\log^2(R)),\;\; \lambda^2\cdot\frac{t^{2\nu k}}{(\lambda t)^4}\cdot S^4(R^{2k}\log^2(R), \mathcal{Q}^{\beta_2}),\\[2pt] \nonumber
			& \frac{\lambda^2\cdot t^{2\nu k}}{(\lambda t)^{2(l'+2)}}\cdot S^2(R^{2k}\log^{2+ l'} (R), (\mathcal{Q}^{\beta_{2 +l'}})'),\;  \frac{\lambda^2\cdot t^{2\nu k}}{(\lambda t)^{2(l'+3)}}\cdot S^2(R^{2k}\log^{3+ l'} (R), \mathcal{Q}^{\beta_{3 +l'}}),
		\end{align}
		concerning the first type of terms where $ k\in I_{j-1}$ and 
		\begin{align}\label{eq:spanrelation2}
			& \lambda^2\cdot \frac{t^{2\nu k}}{(\lambda t)^{2l +2}}S^2\big(R^{2k}\log^{2l} (R), \mathcal{Q}^{\beta_{l+1}}\big),\\[2pt] \nonumber
			&  \frac{\lambda^2\cdot t^{2\nu k}}{(\lambda t)^{2(l'+ l +1)}}\cdot S^2(R^{2k}\log^{2l + l'} (R), (\mathcal{Q}^{\beta_{l +1 +l'}})'),\;  \frac{\lambda^2\cdot t^{2\nu k}}{(\lambda t)^{2(l'+l +2)}}\cdot S^2(R^{2k}\log^{2l +1+ l'} (R), \mathcal{Q}^{\beta_{l +l' +2}}),
		\end{align}
		for which $l+k\in I_{j-1}$ as well as always $  \tilde{N} \geq l' \geq 0$. 
		Multiplying each of these terms by $\lambda^{-2}W$ and applying Lemma~\ref{lem:recoverz1} to the resulting expressions, we observe that the corresponding contribution to $z_j$ is in the span of 
		\begin{align*}
			&\frac{t^{2\nu k}}{(\lambda t)^2}S^2\big(R^{2k-2}\log^2(R) \big),\,\frac{ t^{2\nu k}}{(\lambda t)^{2(l'+1)}}\cdot S^2(R^{2k}\log^{2l'+1} (R), \mathcal{Q}^{\beta_{l'}}),\,k\in I_{j-1}, l'\geq 1\\
			&\frac{t^{2\nu k}}{(\lambda t)^{2(l'+l)}}S^2\big(R^{2k}\log^{2(l'+l)- 1} (R), \mathcal{Q}^{\beta_{l+l'}}\big),\,k+l\in I_{j-1}, l'\geq 1,
		\end{align*}
		which is in $X_j$ as desired.  We note for the first term of \eqref{eq:spanrelation1} we have to replace $\log^2(R)$ by $\log^3(R)$ if $ k =1$, in fact we may use the simple inclusion
		\[
		\lambda^2\cdot\frac{t^{2\nu k}}{(\lambda t)^2}\cdot S^4(R^{2k-2}\log^2(R)) \subset \lambda^2\cdot\frac{t^{2\nu k}}{(\lambda t)^2}\cdot S^4(R^{2k-2}\log^3(R)),
		\]
		which is of the required form. Further note when integrating the terms in the latter line of \eqref{eq:spanrelation2},  we may always increase the logarithmic power (which is only  necessary if $ k =0$), i.e.   we have leading power $\log^{2l + l ' (+1)}(R)$ after applying Lemma \ref{lem:recoverz1}. However, in any case,  we find the consistent representation in the larger space
		\[
		\frac{t^{2\nu k}}{(\lambda t)^{2(l'+l)}}S^2\big(R^{2k}\log^{2l +  l' } (R), \mathcal{Q}^{\beta_{l+l'}}\big) \subset \frac{t^{2\nu k}}{(\lambda t)^{2(l'+l)}}S^2\big(R^{2k}\log^{2(l'+l)-1} (R), \mathcal{Q}^{\beta_{l+l'}}\big)
		\]
		which is a true  embedding since $ l ' \geq 1$. We now turn to the remaining two linear terms.\\[6pt]
		\emph{\underline{Case \rom{3}}:  The terms $ i t^{1 + 2 \nu}\partial_t z_{j-1} ,\; t^{2 \nu}  ( \alpha_0  +  i( \f12 + \nu)\Lambda)z_{j-1}$}. Clearly if $ z_{j-1} \in X_{j-1}$, then \\
		$$
		t \partial_t z_{j-1},\Lambda z_{j-1} \in X_{j-1},
		$$
		except of course we need to replace $ \mathcal{Q}^{\beta_l}$ in Definition \ref{defn:X-space} by $( \mathcal{Q}^{\beta_l})'$. In fact all source terms (of this case)  are linear combinations (with coefficients depending only on $ \nu , \alpha_0$), of 
		\begin{align*}
			&t^{2 \nu(k+1)} S^2(R^{2k-2} \log(R)),\;\; k \in I_{j-1},\\
			& \frac{t^{2 \nu(k+1)}}{(t \lambda)^{2l}} S^2(R^{2k} \log^{2l -1}(R), (\mathcal{Q}^{\beta_l})'),\;\; k + l \in I_{j-1}.
		\end{align*}
		Thus using Lemma \ref{lem:Lemma-n-Anfang-elliptic} and Lemma \ref{lem:recoverz1}, the resulting function is in the span of
		\begin{align*}
			&t^{2 \nu(k+1)} S^2(R^{2(k+1) -2} \log(R)),\;\; k \in I_{j-1},\\
			& \frac{t^{2 \nu(k+1)}}{(t \lambda)^{2l}} S^2(R^{2(k+1)} \log^{2l -1}(R), \mathcal{Q}^{\beta_l}),\;\; k + l \in I_{j-1},
		\end{align*}
		which in turn is contained in $ X_j$. 	Let us now consider the quadratic and cubic terms on the right side of \eqref{eq:trianglez-inductivedef-2}.\\[4pt]
		\emph{\underline{Case \rom{4}}: The quadratic terms}. We treat the contribution of the term 
		\begin{align} \label{that-is-the-quadra-1}
			\triangle_j\big[\lambda^{-2}\Box^{-1} \partial_{t}^2 ( 2 \lambda^2  \text{Re}(\bar{z}W)  ) z \big] &= \lambda^{-2}\Box^{-1} \partial_{t}^2 ( 2 \lambda^2  \text{Re}(\overline{z_{j-1}}W)  ) z^{\ast}_{j-1}\\[3pt] \nonumber
			& \;\;\;\; +   \lambda^{-2}\Box^{-1} \partial_{t}^2 ( 2 \lambda^2 \text{Re}(\overline{z_{j-2}^{\ast}}W)  ) z_{j-1}, 
		\end{align}
		where we recall by definition
		\[
		z_{k}^{\ast} = \sum_{r = 1}^k z_r,\;\; k \geq 1.
		\]
		Considering the first term on the right of \eqref{that-is-the-quadra-1}, by the inductive assumption the function $z_{j-1}^{\ast}$ is in the linear span of 
		\begin{align*}
			&t^{2\nu k}\cdot S^2\big(R^{2k-2}\log(R)\big),\, k\in \cup_{1 \leq \ell\leq j-1}I_{\ell}\\
			&\frac{t^{2\nu k}}{(\lambda t)^{2l}} \cdot S^2\big(R^{2k}\log^{2l - 1}(R), \mathcal{Q}^{\beta_l}\big),\,k+l\in \cup_{0\leq r \leq j-1}I_{r},\,l>0,
		\end{align*}
		which of course means $ 0 \leq k \leq M_{j-1}$ with $k \geq 1$ in the first line, as well as $l + k \geq 1$ and $ k \geq 0$ in the second line.
		We already saw that 
		$$\Box^{-1} \partial_{t}^2 ( 2 \lambda^2  \text{Re}(\overline{z_{j-1}}W)  )$$
		belongs to  the span of \eqref{eq:spanrelation1} and \eqref{eq:spanrelation2}, which we embed (basically we simplify by embedding $\mathcal{Q}^{\beta} \subset (\mathcal{Q}^{\beta})'$) into the span of the following terms.
		\begin{align*}
			& \frac{t^{2\nu k}}{(t \lambda)^2}S^2(R^{2k -2} \log^2(R)),\;\; k \in I_{j-1},\\
			&\frac{t^{2\nu k}}{(t \lambda)^{2(l'+2)}}S^2(R^{2k } \log^{2 + l'}(R), (\mathcal{Q}^{\beta_{2 + l'}})'),\;\; k \in I_{j-1},\;\; l' \geq 0,\\
			&\frac{t^{2\nu k}}{(t \lambda)^{2( l +l'+1)}}S^2(R^{2k } \log^{2l +l' }(R), (\mathcal{Q}^{\beta_{l + l' +1}})'),\;\; k + l \in I_{j-1},\;\; l' \geq 0.
		\end{align*}
		We then infer from Lemma \ref{lem:basicproduct} that the function 
		\begin{align*}
			\lambda^{-2}\Box^{-1} \partial_{t}^2 ( 2 \lambda^2  \text{Re}(\overline{z_{j-1}}W)  ) z_{j-1}^{\ast}
		\end{align*}
		must be contained in the linear span of  (in all lines we have $ l' \geq 0$)
		\begin{align*}
			&\frac{t^{2\nu(k_1 + k_2)}}{(t\lambda)^2}\cdot S^2\big(R^{2(k_1 + k_2) - 4}\log^{3}(R) \big),\,k_1 \in I_{j-1}, \; k_2 \in  \cup_{ \ell\leq j-1}I_{\ell},\\
			&\frac{t^{2\nu(k_1 + k_2)}}{(t\lambda)^{2(l ' +2)}}\cdot S^2\big(R^{2(k_1 + k_2) - 2}\log^{3 + l'}(R), (\mathcal{Q}^{\beta_{2 + l'}})' \big),\,k_1 \in I_{j-1}, \; k_2 \in  \cup_{ \ell\leq j-1}I_{\ell},\\
			&\frac{t^{2\nu(k_1 + k_2)}}{(t\lambda)^{2( l_1 + l' + 1)}}\cdot S^2\big(R^{2(k_1 + k_2) - 2}\log^{2l_1 + l' +1 }(R),  (\mathcal{Q}^{\beta_{l_1 + l' +1}})' \big),\,k_1 + l_1 \in I_{j-1}, \; k_2  \in  \cup_{ \ell\leq j-1}I_{\ell},
		\end{align*}
		\;\\
		as well as
		\begin{align*}
			&\frac{t^{2\nu(k_1 + k_2)}}{(t\lambda)^{2( 1 + l_2)}}\cdot S^2\big(R^{2(k_1 + k_2) - 2}\log^{2 l_2 +1}(R),  (\mathcal{Q}^{\beta_{l_2}})' \big),\,k_1 \in I_{j-1}, \; k_2  + l_2 \in  \cup_{ \ell\leq j-1}I_{\ell},\\
			&\frac{t^{2\nu(k_1 + k_2)}}{(t\lambda)^{2( l' + l_2 +2)}}\cdot S^2\big(R^{2(k_1 + k_2) }\log^{ l' + 2 l_2 +1}(R),  (\mathcal{Q}^{\beta_{2 + l'} + \beta_{l_2}})' \big),\,k_1 \in I_{j-1}, \; k_2  + l_2 \in  \cup_{ \ell\leq j-1}I_{\ell},\\
			&\frac{t^{2\nu(k_1 + k_2)}}{(t\lambda)^{2( l_1 + l_2 + l'+1)}}\cdot S^2\big(R^{2(k_1 + k_2) }\log^{ l' + 2 (l_1 + l_2) - 1}(R),  (\mathcal{Q}^{\beta_{l_1 + l' +1} + \beta_{l_2}})' \big),\,k_1 + l_1 \in I_{j-1}, \; k_2  + l_2 \in  \cup_{ \ell\leq j-1}I_{\ell}.
		\end{align*}
		Applying Lemma~\ref{lem:recoverz1}, we infer that the corresponding contribution to $z_j$ is in the linear span of 
		\begin{align*}
			&\frac{t^{2\nu(k_1 + k_2)}}{(t\lambda)^2}\cdot S^2\big(R^{2(k_1 + k_2) - 2}\log^{3}(R) \big),\,k_1 \in I_{j-1}, \; k_2 \in  \cup_{ \ell\leq j-1}I_{\ell},\\
			&\frac{t^{2\nu(k_1 + k_2)}}{(t\lambda)^{2(l ' +2)}}\cdot S^2\big(R^{2(k_1 + k_2) }\log^{3 + l'}(R), \mathcal{Q}^{\beta_{2 + l'}} \big),\,k_1 \in I_{j-1}, \; k_2 \in  \cup_{ \ell\leq j-1}I_{\ell},\\
			&\frac{t^{2\nu(k_1 + k_2)}}{(t\lambda)^{2( l_1 + l' + 1)}}\cdot S^2\big(R^{2(k_1 + k_2) }\log^{2l_1 + l' +1}(R),  \mathcal{Q}^{\beta_{l_1 + l' +1}}\big),\,k_1 + l_1 \in I_{j-1}, \; k_2  \in  \cup_{ \ell\leq j-1}I_{\ell},\\
			&\frac{t^{2\nu(k_1 + k_2)}}{(t\lambda)^{2( 1 + l_2)}}\cdot S^2\big(R^{2(k_1 + k_2) }\log^{1 + 2 l_2}(R),  \mathcal{Q}^{\beta_{l_2}}\big),\,k_1 \in I_{j-1}, \; k_2  + l_2 \in  \cup_{ \ell\leq j-1}I_{\ell},\\
			&\frac{t^{2\nu(k_1 + k_2)}}{(t\lambda)^{2( l' + l_2 +1)}}\cdot S^2\big(R^{2(k_1 + k_2) }\log^{1 + l' + 2 l_2}(R),  \mathcal{Q}^{\beta_{1 + l_2+ l'} } \big),\,k_1 \in I_{j-1}, \; k_2  + l_2 \in  \cup_{ \ell\leq j-1}I_{\ell},\\
			&\frac{t^{2\nu(k_1 + k_2)}}{(t\lambda)^{2( l_1 + l_2 + l')}}\cdot S^2\big(R^{2(k_1 + k_2) }\log^{ l' + 2 (l_1 + l_2) -1}(R), \mathcal{Q}^{\beta_{l_1 + l_2 + l' } }\big),\,k_1 + l_1 \in I_{j-1}, \; k_2  + l_2 \in  \cup_{ \ell\leq j-1}I_{\ell}.
		\end{align*}
		For the latter two lines we in particular used $ R^2 = (t \lambda)^2 a^2$ similar to the above \emph{Case \rom{1}} via the embedding
		\begin{align*}
			&\frac{t^{2\nu(k_1 + k_2)}}{(t\lambda)^{2( l' + l_2 +2)}}\cdot S^2\big(R^{2(k_1 + k_2) +2 }\log^{1 + l' + 2 l_2}(R),  \mathcal{Q}^{\beta_{2 + l' + l_2} - 2\nu -1} \big)\\
			& \;\;\;\;\subset \frac{t^{2\nu(k_1 + k_2)}}{(t\lambda)^{2( l' + l_2 +1)}}\cdot S^2\big(R^{2(k_1 + k_2) }\log^{1+ l' + 2 l_2}(R), a^2 \cdot \mathcal{Q}^{\beta_{2 + l' + l_2} - 2\nu -1} \big)\\
			&\;\;\;\;\; \subset \frac{t^{2\nu(k_1 + k_2)}}{(t\lambda)^{2( l' + l_2 +1)}}\cdot S^2\big(R^{2(k_1 + k_2) }\log^{1+ l' + 2 l_2}(R), \mathcal{Q}^{\beta_{2 + l' + l_2} - 2\nu +1} \big)\\
			&\;\;\;\;\; = \frac{t^{2\nu(k_1 + k_2)}}{(t\lambda)^{2( l' + l_2 +1)}}\cdot S^2\big(R^{2(k_1 + k_2) }\log^{1 + l' + 2 l_2}(R), \mathcal{Q}^{\beta_{1 + l' + l_2} } \big),
		\end{align*}
		and likewise
		\begin{align*}
			&\frac{t^{2\nu(k_1 + k_2)}}{(t\lambda)^{2( l_1 + l_2 + l' +1)}}\cdot S^2\big(R^{2(k_1 + k_2) +2 }\log^{ l' + 2 (l_1 + l_2)-1}(R),  \mathcal{Q}^{\beta_{l_1 + l_2 + l' +1} - 2\nu -1 }\big)\\
			& \;\;\;\;\subset \frac{t^{2\nu(k_1 + k_2)}}{(t\lambda)^{2( l_1 + l_2 + l')}}\cdot S^2\big(R^{2(k_1 + k_2)  }\log^{ l' + 2 (l_1 + l_2)-1}(R),  \mathcal{Q}^{\beta_{l_1 + l_2 + l' +1} - 2\nu +1 } \big)\\
			&\;\;\;\;\; =  \frac{t^{2\nu(k_1 + k_2)}}{(t\lambda)^{2( l_1 + l_2 + l')}}\cdot S^2\big(R^{2(k_1 + k_2)  }\log^{ l' + 2 (l_1 + l_2)-1}(R),  \mathcal{Q}^{\beta_{l_1 + l_2 + l'}  } \big).
		\end{align*}
		Further, for the other terms, we note in all cases with a constraint $ k_1 \in I_{j-1}$ or $ k_2 \in \cup_{\leq j-1}$, there holds $k_1 \geq 1 $ or $ k_2 \geq 1$ respectively (hence also the leading logarithmic power does not need to be raised, even if $ k_1 + k_2 $ is minimal).\\[3pt]
		The span of the sets of terms above is then observed, via simple embeddings as before,  to be in  $X_j$. For the first line we use the embedding via $ a^2 \mathcal{Q}^{\beta_1}$. Let us now consider the remaining cases,  which close the inductive step.  Next is the second term on the right of \eqref{that-is-the-quadra-1}, i.e. 
		\;\\
		\[
		\lambda^{-2}\Box^{-1} \partial_{t}^2 ( 2 \lambda^2  \text{Re}(\overline{z_{j-2}^{\ast}}W)) z_{j-1}.
		\]
		First we note, similar as above, the terms $ \partial_{t}^2 ( 2 \lambda^2  \text{Re}(\overline{z_{j-2}^{\ast}}W) $ must be in the span of 
		\begin{align*}
			&t^{-2} \lambda^2 \cdot t^{2\nu k}\cdot S^2\big(R^{2k-4}\log(R)\big),\, k\in \cup_{1 \leq \ell\leq j-2}I_{\ell}\\
			& t^{-2} \lambda^2 \cdot \frac{t^{2\nu k}}{(\lambda t)^{2l}} \cdot S^2\big(R^{2k-2}\log^{2l - 1}(R), \mathcal{Q}^{\beta_l}\big),\,k+l\in \cup_{0\leq r \leq j-2}I_{r},\, l \geq 2,
		\end{align*}
		Hence calculating $\lambda^{-2} \Box^{-1}(\cdot )$ using Remark \ref{rem:elliptic-steps-first}, Lemma \ref{lem:Boxminuesonefirst} and  Corollary \ref{cor:Boxminuesonefirst}, we arrive at terms in the span of 
		\begin{align*}
			&\frac{ t^{2\nu k}}{(t\lambda)^2}\cdot S^2\big(R^{2k-2}\log^2(R)\big),\;\; 	\frac{ t^{2\nu k}}{(t\lambda)^4}\cdot S^2\big(R^{2k}\log^2(R)\big),\;\; k\in \cup_{1 \leq \ell\leq j-2}I_{\ell},\\
			& \frac{ t^{2\nu k}}{(t\lambda)^{2(l' +2)}}\cdot S^2\big(R^{2k}\log^{2 + l'}(R), (\mathcal{Q}^{\beta_{2 + l'}})'\big),\;\;k\in \cup_{1 \leq \ell\leq j-2}I_{\ell},\;\; l' \geq 1,
		\end{align*}
		and further
		\begin{align*}
			&  \frac{t^{2\nu k}}{(\lambda t)^{2l +2}} \cdot S^2\big(R^{2k}\log^{2l}(R), \mathcal{Q}^{\beta_{l+1}}\big),\,k+l\in \cup_{0\leq r \leq j-2}I_{r},\, l \geq 2,\\
			&  \frac{t^{2\nu k}}{(\lambda t)^{2(l + l')}} \cdot S^2\big(R^{2k}\log^{2l + l' -1}(R), (\mathcal{Q}^{\beta_{l+l'}})'\big),\,k+l\in \cup_{0\leq r \leq j-2}I_{r},\, l \geq 2,\; l' \geq 1.
		\end{align*}
		By the inductive assumption we have $ z_{j-1} \in X_{j-1}$ and hence multiplying by these terms and using Lemma \ref{lem:recoverz1}, we have products in the span of
		\begin{align*}
			&\frac{ t^{2\nu (k_1 + k_2)}}{(t\lambda)^2}\cdot S^2\big(R^{2(k_1 + k_2)-2}\log^2(R)\big),\;\; k_1\in \cup_{1 \leq \ell\leq j-2}I_{\ell},\;\; k_2 \in I_{j-1},\\
			& \frac{ t^{2\nu (k_1 + k_2)}}{(t\lambda)^{2(l' +2)}}\cdot S^2\big(R^{2(k_1 + k_2) }\log^{2 + l'}(R), \mathcal{Q}^{\beta_{2 + l'}}\big),\;\; k\in \cup_{1 \leq \ell\leq j-2}I_{\ell},\;\; l' \geq 0,\\
			&\frac{ t^{2\nu (k_1 + k_2)}}{(t\lambda)^{2(1 + l)}}\cdot S^2\big(R^{2(k_1 + k_2)}\log^{1 + 2l}(R), \mathcal{Q}^{\beta_l}\big),\;\; k_1 \in \cup_{1 \leq \ell\leq j-2}I_{\ell}, k_2  + l \in I_{j-1}\\
			& \frac{ t^{2\nu (k_1 + k_2)}}{(t\lambda)^{2(l' +1 + l)}}\cdot S^2\big(R^{2(k_1 + k_2)}\log^{1 + l' + 2l}(R), \mathcal{Q}^{\beta_{1 + l' +l}}\big),\;\;k_1\in \cup_{1 \leq \ell\leq j-2}I_{\ell},\;\; k_2 + l \in I_{j-1},\; l' \geq 0,
		\end{align*}
		as well as
		\begin{align*}
			&  \frac{t^{2\nu (k_1 + k_2)}}{(\lambda t)^{2l +4}} \cdot S^2\big(R^{2(k_1 + k_2)}\log^{2l+1}(R), \mathcal{Q}^{\beta_{l+1}}\big),\,k_1+l\in \cup_{0\leq r \leq j-2}I_{r},\; k_2 \in I_{j-1}, \, l \geq 2,\\
			&  \frac{t^{2\nu (k_1 + k_2)}}{(\lambda t)^{2(l + l' +1)}} \cdot S^2\big(R^{2(k_1 + k_2)}\log^{2l + l' }(R), \mathcal{Q}^{\beta_{l+l' +1}}\big),\,k_1+l\in \cup_{0\leq r \leq j-2}I_{r},\, k_2 \in I_{j-1},\; l \geq 2,\; l' \geq 1,\\
			&  \frac{t^{2\nu (k_1 + k_2)}}{(\lambda t)^{2(l_1 + l_2)}} \cdot S^2\big(R^{2(k_1 + k_2)}\log^{2l_1 + 2l_2 -1}(R), \mathcal{Q}^{\beta_{l_1 + l_2}}\big),\,k_1+l_1\in \cup_{0\leq r \leq j-2}I_{r},\, k_2 + l_2 \in I_{j-1},\; l \geq 2,\\
			&  \frac{t^{2\nu (k_1 + k_2)}}{(\lambda t)^{2(l_1 + l_2 + l' -1)}} \cdot S^2\big(R^{2(k_1 + k_2)}\log^{2l_1 + 2l_2 + l' -2}(R), \mathcal{Q}^{\beta_{l+l'-1} }\big),\,k_1+l_1\in \cup_{0\leq r \leq j-2}I_{r},\, k_2 + l_2 \in I_{j-1},\\
			&\; l \geq 2,\; l' \geq 1,
		\end{align*}
		where in the latter two lines we used the above argument factoring off $ R^2 = a^2 (t\lambda)^2$. In particular we infer the required form of $X_{j}$. The next terms are in
		\begin{align} \label{that-is-the-quadra-2}
			\triangle_j \big( \lambda^{-2}\Box^{-1} \partial_t^2( \lambda^2 |z|^2) W\big) =&   \lambda^{-2}\Box^{-1} \partial_t^2( \lambda^2 \; \overline{z_{j-1}} \cdot  z_{j-1}^{\ast}) W\\ \nonumber
			& +  \lambda^{-2}\Box^{-1} \partial_t^2( \lambda^2 \; \overline{z_{j-2}^{\ast}} \cdot z_{j-1}) W,
		\end{align}
		where the latter terms can be treated in the same way  of course. In fact we consider the first of these and the second follows with an analogous argument.
		Hence, by assumption, the terms $ \partial_t^2( \lambda^2 \; \overline{z_{j-1}} \cdot  z_{j-1}^{\ast}) $ are in the span of
		\begin{align*}
			&t^{-2} \lambda^2 \cdot t^{2\nu(k_1 + k_2)} S^2(R^{2(k_1 + k_2) -4} \log^2(R)),\;\; k_1 \in I_{j-1},\; k_2 \in \cup_{1 \leq\ell \leq j-1 } I_{\ell},\\
			&t^{-2} \lambda^2 \cdot \frac{t^{2\nu(k_1 + k_2)} }{(t \lambda)^{2l_1}}S^2(R^{2(k_1 + k_2) -2} \log^{2l_1}(R), \mathcal{Q}^{\beta_{l_1}}),\;\; k_1 + l_1 \in I_{j-1},\; k_2 \in \cup_{1 \leq\ell \leq j-1 } I_{\ell},\\
			&t^{-2} \lambda^2 \cdot \frac{t^{2\nu(k_1 + k_2)} }{(t \lambda)^{2l_2}}S^2(R^{2(k_1 + k_2) -2} \log^{2l_2}(R), \mathcal{Q}^{\beta_{l_2}}),\;\; k_1  \in I_{j-1},\; k_2 + l_2 \in \cup_{1 \leq\ell \leq j-1 } I_{\ell},\\
			&t^{-2} \lambda^2 \cdot \frac{t^{2\nu(k_1 + k_2)} }{(t \lambda)^{2(l_1 + l_2)}}S^2(R^{2(k_1 + k_2) } \log^{2l_1 + 2l_2 -2}(R), \mathcal{Q}^{\beta_{l_1 +l_2} }),\;\; k_1 + l_1 \in I_{j-1},\; k_2 + l_2 \in \cup_{1 \leq\ell \leq j-1 } I_{\ell}.
		\end{align*}
		%where we note we can rewrite the S-space in the latter line via $ a^{-2} = R^{-2} (t \lambda)^2$ into
		%\begin{align*}
		%	\frac{t^{2\nu(k_1 + k_2)} }{(t \lambda)^{2(l_1 + l_2)}}S^2(R^{2(k_1 + k_2) }& \log^{2l_1 + 2l_2 -2}(R), \mathcal{Q}^{\beta_{l_1} + \beta_{l_2}})\\
		%	 &\subset \frac{t^{2\nu(k_1 + k_2)} }{(t \lambda)^{2(l_1 + l_2)}}S^2(R^{2(k_1 + k_2) } \log^{2l_1 + 2l_2 -2}(R), \mathcal{Q}^{\beta_{l_1 + l_2} -2})\\
		%	&\subset \frac{t^{2\nu(k_1 + k_2)} }{(t \lambda)^{2(l_1 + l_2 -1)}}S^2(R^{2(k_1 + k_2) -2} \log^{2l_1 + 2l_2 -2}(R), \mathcal{Q}^{\beta_{l_1 + l_2} }).
		%\end{align*}
		\;\\
		Hence calculating $\lambda^{-2}\Box^{-1}(\cdot)$ as usual via Lemma \ref{lem:Boxminuesonefirst}, we obtain terms in the span of
		\;\\
		\begin{align}  \label{spann-1}
			& \frac{t^{2\nu(k_1 + k_2)}}{(t \lambda)^2}S^2(R^{2(k_1 + k_2) -2} \log^2(R)),\;\; k_1 \in I_{j-1},\; k_2 \in \cup_{1 \leq\ell \leq j-1 } I_{\ell},\\ \nonumber
			& \frac{t^{2\nu(k_1 + k_2)}}{(t \lambda)^{2(l' +2)}}S^2(R^{2(k_1 + k_2)} \log^{2 + l'}(R), (\mathcal{Q}^{\beta_{2 + l'}})'),\;\; k_1 \in I_{j-1},\; k_2 \in \cup_{1 \leq\ell \leq j-1 } I_{\ell},\; l' \geq 0,
		\end{align}
		and further
		\begin{align}  \label{spann-2}
			& \frac{t^{2\nu(k_1 + k_2)} }{(t \lambda)^{2(l_j + l' )}}S^2(R^{2(k_1 + k_2) } \log^{2l_j +l'}(R), (\mathcal{Q}^{\beta_{l_j + l'}})'),\;\; l ' \geq 1,
		\end{align}
		where $ j = 1,2$ and either 
		\begin{align*}
			&k_1 + l_1 \in I_{j-1},\; k_2 \in \cup_{1 \leq\ell \leq j-1 } I_{\ell},\;\;\text{or}\;\;\\
			&k_1  \in I_{j-1},\; k_2 + l_2 \in \cup_{1 \leq\ell \leq j-1 } I_{\ell} ,
		\end{align*}
		as well as $ \lambda^{-2} \Box^{-1}$ applied to the last line. Here we 
		%use the embedding above before applying the Lemma (respectively Remark \ref{rem:elliptic-steps-first} first) and 
		observe
		\;\\
		\begin{align}  \label{spann-3}
			& \frac{t^{2\nu(k_1 + k_2)} }{(t \lambda)^{2(l_1 + l_2 + l')}}S^2(R^{2(k_1 + k_2) } \log^{2l_1 + 2l_2 -2+ l'}(R), ( \mathcal{Q}^{\beta_{l_1 + l_2  + l'} })' ),\\ \nonumber
			&k_1 + l_1 \in I_{j-1},\; k_2 + l_2 \in \cup_{1 \leq\ell \leq j-1 } I_{\ell},\;\; l' \geq 0.
		\end{align}
		Hence multiplying $W$  to the span of \eqref{spann-1} - \eqref{spann-3} and applying Lemma \ref{lem:recoverz1}, we then obtain a term in the span of
		\begin{align*} 
			& \frac{t^{2\nu(k_1 + k_2)}}{(t \lambda)^2}S^2(R^{2(k_1 + k_2) -2} \log^2(R)),\;\; k_1 \in I_{j-1},\; k_2 \in \cup_{1 \leq\ell \leq j-1 } I_{\ell},\\
			& \frac{t^{2\nu(k_1 + k_2)}}{(t \lambda)^{2(l' +2)}}S^2(R^{2(k_1 + k_2)} \log^{2 + l'}(R), \mathcal{Q}^{\beta_{2 + l'}}),\;\; k_1 \in I_{j-1},\; k_2 \in \cup_{1 \leq\ell \leq j-1 } I_{\ell},\; l' \geq 0,\\
			& \frac{t^{2\nu(k_1 + k_2)} }{(t \lambda)^{2(l_j + l' )}}S^2(R^{2(k_1 + k_2) } \log^{2l_j + l'}(R), \mathcal{Q}^{\beta_{l_j + l'}}),\;\; l ' \geq 1,\\
			& \frac{t^{2\nu(k_1 + k_2)} }{(t \lambda)^{2(l_1 + l_2 + l')}}S^2(R^{2(k_1 + k_2) } \log^{2l_1 + 2l_2 -1 + l'}(R),  \mathcal{Q}^{\beta_{l_1 + l_2  + l'} } ),\\ 
			&k_1 + l_1 \in I_{j-1},\; k_2 + l_2 \in \cup_{1 \leq\ell \leq j-1 } I_{\ell},\;\; l' \geq 0.
		\end{align*}
		Therefore we again obtain a linear combination in the space $X_{j}$. We are left in \emph{Case \rom{4}} with the quadratic `pure Schr\"odinger' interactions, i.e. we have to consider 
		\begin{align} \label{that-is-the-quadra-3}
			\triangle_j \big(2 \text{Re}(\overline{z}W)z + |z|^2 W\big) =&\;  2 \text{Re}(\overline{z_{j-1}} W) z_{j-1}^{\ast} + 2 \text{Re}(\overline{z_{j-2}^{\ast}} W) z_{j-1}\\ \nonumber
			&\; + \overline{z_{j-1}} \cdot z_{j-1}^{\ast} W +  \overline{z_{j-2}^{\ast}} \cdot z_{j-1} W. 
		\end{align}
		These  terms can of course all be treated with the same argument, in fact we find integrating these terms is slightly simpler as for the above terms.  For instance the terms 
		\[
		2 \text{Re}(\overline{z_{j-1}} W) z_{j-1}^{\ast} ,\;\;\;\overline{z_{j-1}} \cdot z_{j-1}^{\ast} W ,
		\]
		are contained in the span of
		\begin{align*}
			&t^{2\nu(k_1 + k_2)} S^2(R^{2(k_1 + k_2) - 6} \log^2(R)),\;\; k_1 \in I_{j-1},\; k_2 \in \cup_{1 \leq\ell \leq j-1 } I_{\ell},\\
			& \frac{t^{2\nu(k_1 + k_2)} }{(t \lambda)^{2l_1}}S^2(R^{2(k_1 + k_2) -4} \log^{2l_1}(R), \mathcal{Q}^{\beta_{l_1}}),\;\; k_1 + l_1 \in I_{j-1},\; k_2 \in \cup_{1 \leq\ell \leq j-1 } I_{\ell},\\
			&\frac{t^{2\nu(k_1 + k_2)} }{(t \lambda)^{2l_2}}S^2(R^{2(k_1 + k_2) -4} \log^{2l_2}(R), \mathcal{Q}^{\beta_{l_2}}),\;\; k_1  \in I_{j-1},\; k_2 + l_2 \in \cup_{1 \leq\ell \leq j-1 } I_{\ell},\\
			&\frac{t^{2\nu(k_1 + k_2)} }{(t \lambda)^{2(l_1 + l_2)}}S^2(R^{2(k_1 + k_2)-2 } \log^{2l_1 + 2l_2 -2}(R), \mathcal{Q}^{\beta_{l_1 +l_2} }),\;\; k_1 + l_1 \in I_{j-1},\; k_2 + l_2 \in \cup_{1 \leq\ell \leq j-1 } I_{\ell}.
		\end{align*}
		Integrating via Lemma \ref{lem:recoverz1} gives
		\begin{align*}
			&t^{2\nu(k_1 + k_2)} S^2(R^{2(k_1 + k_2) - 4} \log^3(R)),\;\; k_1 \in I_{j-1},\; k_2 \in \cup_{1 \leq\ell \leq j-1 } I_{\ell},\\
			& \frac{t^{2\nu(k_1 + k_2)} }{(t \lambda)^{2l_1}}S^2(R^{2(k_1 + k_2) -2} \log^{2l_1 +1}(R), \mathcal{Q}^{\beta_{l_1}}),\;\; k_1 + l_1 \in I_{j-1},\; k_2 \in \cup_{1 \leq\ell \leq j-1 } I_{\ell},\\
			&\frac{t^{2\nu(k_1 + k_2)} }{(t \lambda)^{2l_2}}S^2(R^{2(k_1 + k_2) -2} \log^{2l_2 +1}(R), \mathcal{Q}^{\beta_{l_2}}),\;\; k_1  \in I_{j-1},\; k_2 + l_2 \in \cup_{1 \leq\ell \leq j-1 } I_{\ell},\\
			&\frac{t^{2\nu(k_1 + k_2)} }{(t \lambda)^{2(l_1 + l_2)}}S^2(R^{2(k_1 + k_2) } \log^{2l_1 + 2l_2 -1}(R), \mathcal{Q}^{\beta_{l_1 +l_2} }),\;\; k_1 + l_1 \in I_{j-1},\; k_2 + l_2 \in \cup_{1 \leq\ell \leq j-1 } I_{\ell},
		\end{align*}
		where we then use for the first three lines the embedding (schematically)
		\[
		S^2(R^k \log^{b +2}(R)) \subset S^2(R^{k+2} \log^{b}(R)).
		\]
		We now turn to the cubic interaction part.\\[6pt]
		\emph{\underline{Case \rom{5}}: The cubic term}. Here we consider the cubic source term on right of  \eqref{eq:trianglez-inductivedef-2}, i.e. at first
		\begin{align} \label{that-is-the-quadra-4}
			\triangle_j \big( \lambda^{-2}\Box^{-1} \partial_t^2(\lambda^2 |z|^2)z\big),
		\end{align}
		for which it suffices to consider
		\begin{align} \label{that-is-the-quadra-5}
			\lambda^{-2}\big(\Box^{-1} \partial_t^2(\lambda^2 \overline{z_{j-1}} \cdot z_{j-1}^{\ast})\big) z_{j-1}^{\ast},\;\; &\lambda^{-2}\big(\Box^{-1} \partial_t^2(\lambda^2 \overline{z_{j-1}^{\ast}} \cdot z_{j-1})\big) z_{j-1}^{\ast},\;\;\\ \nonumber
			& \lambda^{-2}\big(\Box^{-1} \partial_t^2(\lambda^2 \overline{z_{j-1}^{\ast}} \cdot z_{j-1}^{\ast})\big) z_{j-1}.
		\end{align}
		The first two terms are treated in the same way, thus let us focus on the first of these, i.e.
		\[
		\lambda^{-2}\big(\Box^{-1} \partial_t^2(\lambda^2 \overline{z_{j-1}} \cdot z_{j-1}^{\ast})\big) z_{j-1}^{\ast}.
		\]
		We already observed that the terms $ \partial_t^2(\lambda^2 \overline{z_{j-1}} \cdot z_{j-1}^{\ast}) $ are in the span of 
		\begin{align*}
			&t^{-2} \lambda^2 \cdot t^{2\nu(k_1 + k_2)} S^2(R^{2(k_1 + k_2) - 4} \log^2(R)),\;\; k_1 \in I_{j-1},\; k_2 \in \cup_{1 \leq\ell \leq j-1 } I_{\ell},\\
			& t^{-2} \lambda^2 \cdot \frac{t^{2\nu(k_1 + k_2)} }{(t \lambda)^{2l_1}}S^2(R^{2(k_1 + k_2) -2} \log^{2l_1}(R), \mathcal{Q}^{\beta_{l_1}}),\;\; k_1 + l_1 \in I_{j-1},\; k_2 \in \cup_{1 \leq\ell \leq j-1 } I_{\ell},\\
			&t^{-2} \lambda^2 \cdot \frac{t^{2\nu(k_1 + k_2)} }{(t \lambda)^{2l_2}}S^2(R^{2(k_1 + k_2) -2} \log^{2l_2}(R), \mathcal{Q}^{\beta_{l_2}}),\;\; k_1  \in I_{j-1},\; k_2 + l_2 \in \cup_{1 \leq\ell \leq j-1 } I_{\ell},\\
			&t^{-2} \lambda^2 \cdot \frac{t^{2\nu(k_1 + k_2)} }{(t \lambda)^{2(l_1 + l_2)}}S^2(R^{2(k_1 + k_2) } \log^{2l_1 + 2l_2 -2}(R), \mathcal{Q}^{\beta_{l_1 +l_2} }),\;\; k_1 + l_1 \in I_{j-1},\; k_2 + l_2 \in \cup_{1 \leq\ell \leq j-1 } I_{\ell}.
		\end{align*}
		Hence applying $\lambda^{-2} \Box^{-1}(\cdot)$ implies linear combinations of functions in
		\begin{align*}
			& \frac{t^{2\nu(k_1 + k_2)}}{(t\lambda)^2} S^2(R^{2(k_1 + k_2) - 2} \log^2(R)),\;\;\;\; k_1 \in I_{j-1},\; k_2 \in \cup_{1 \leq\ell \leq j-1 } I_{\ell},\\
			& \frac{t^{2\nu(k_1 + k_2)}}{(t\lambda)^{2(2 + l')}} S^2(R^{2(k_1 + k_2) } \log^{2 + l'}(R),( \mathcal{Q}^{\beta_{2 + l'}})'),\;\;\;\; k_1 \in I_{j-1},\; k_2 \in \cup_{1 \leq\ell \leq j-1 } I_{\ell},\;\; l' \geq 0\\[8pt]
			& \frac{t^{2\nu(k_1 + k_2)} }{(t \lambda)^{2(l_j + l' +1)}}S^2(R^{2(k_1 + k_2) } \log^{2l_j + l'}(R), (\mathcal{Q}^{\beta_{l_j + l' +1}})'),\;\; j =1,2,\;\; l' \geq 1,\\
			&k_1 + l_1 \in I_{j-1},\; k_2 \in \cup_{1 \leq\ell \leq j-1 } I_{\ell},\;\;\text{or}\;\;k_1 \in I_{j-1},\; k_2 + l_2 \in \cup_{1 \leq\ell \leq j-1 } I_{\ell},\\[8pt]
			&\frac{t^{2\nu(k_1 + k_2)} }{(t \lambda)^{2(l_1 + l_2 + l')}}S^2(R^{2(k_1 + k_2) } \log^{2l_1 + 2l_2 -2 + l'}(R), (\mathcal{Q}^{\beta_{l_1 +l_2 + l'} })'),\;\; k_1 + l_1 \in I_{j-1},\; k_2 + l_2 \in \cup_{1 \leq\ell \leq j-1 } I_{\ell},\;\; l' \geq 0,\\[0pt]
		\end{align*}
		and multiplying with $ z_{j-1}^{\ast}$ gives linear combination in the span of 
		\;\\
		\begin{align*}
			& \frac{t^{2\nu(k_1 + k_2 + k_3)}}{(t\lambda)^2} S^2(R^{2(k_1 + k_2 + k_3) - 4} \log^3(R)),\;\;\;\; k_1 \in I_{j-1},\; k_2 \in \cup_{1 \leq\ell \leq j-1 } I_{\ell},\; k_3 \in \cup_{1 \leq\ell \leq j-1 } I_{\ell},\\[8pt]
			& \frac{t^{2\nu(k_1 + k_2 + k_3)}}{(t\lambda)^{2(1 + l_3)}} S^2(R^{2(k_1 + k_2 + k_3) - 2} \log^{1 + 2l_3 }(R), (\mathcal{Q}^{\beta_{l_3}})'),\;\;\;\; k_1 \in I_{j-1},\; k_2 \in \cup_{1 \leq\ell \leq j-1 } I_{\ell},\; k_3 + l_3 \in \cup_{1 \leq\ell \leq j-1 } I_{\ell},\\[8pt]
			& \frac{t^{2\nu(k_1 + k_2 + k_3)}}{(t\lambda)^{2(2 + l')}} S^2(R^{2(k_1 + k_2 + k_3) -2 } \log^{3 + l'}(R),( \mathcal{Q}^{\beta_{2 + l'}})'),\;\;\;\; k_1 \in I_{j-1},\; k_2 \in \cup_{1 \leq\ell \leq j-1 } I_{\ell}, \; k_3 \in \cup_{1 \leq\ell \leq j-1 } I_{\ell}\;\; l' \geq 0\\[8pt]
			& \frac{t^{2\nu(k_1 + k_2 + k_3)}}{(t\lambda)^{2(2 + l' + l_3)}} S^2(R^{2(k_1 + k_2 + k_3) } \log^{1 + l' + 2l_3}(R),( \mathcal{Q}^{\beta_{2 + l' + l_3}})'),\;\;\;\; k_1 \in I_{j-1},\; k_2 \in \cup_{1 \leq\ell \leq j-1 } I_{\ell},\\
			&  k_3  + l_3 \in \cup_{1 \leq\ell \leq j-1 } I_{\ell}\;\; l' \geq 0,\\[0pt]
		\end{align*}
		as well as 
		\;\\
		\begin{align*}
			& \frac{t^{2\nu(k_1 + k_2 + k_3)} }{(t \lambda)^{2(l_j + l'+1)}}S^2(R^{2(k_1 + k_2 + k_3) -2} \log^{2l_j + l' +1}(R), (\mathcal{Q}^{\beta_{l_j + l'}})'),\;\; j =1,2,\;\; l' \geq 1,\\
			&k_1 + l_1 \in I_{j-1},\; k_2 \in \cup_{1 \leq\ell \leq j-1 } I_{\ell},\;\;\text{or}\;\;k_1 \in I_{j-1},\; k_2 + l_2 \in \cup_{1 \leq\ell \leq j-1 } I_{\ell},\;\;k_3  \in \cup_{1 \leq\ell \leq j-1 } I_{\ell} \\[12pt]
			& \frac{t^{2\nu(k_1 + k_2 + k_3)} }{(t \lambda)^{2(l_j + l' + l_3 +1)}}S^2(R^{2(k_1 + k_2 + k_3) } \log^{2l_j + l' + 2l_3 -1}(R), (\mathcal{Q}^{\beta_{l_j + l' + l_3}})'),\;\; j =1,2,\;\; l' \geq 1,\\
			&k_1 + l_1 \in I_{j-1},\; k_2 \in \cup_{1 \leq\ell \leq j-1 } I_{\ell},\;\;\text{or}\;\;k_1 \in I_{j-1},\; k_2 + l_2 \in \cup_{1 \leq\ell \leq j-1 } I_{\ell},\;\;k_3 + l_3 \in \cup_{1 \leq\ell \leq j-1 } I_{\ell} \\[12pt]
			&\frac{t^{2\nu(k_1 + k_2 + k_3)} }{(t \lambda)^{2(l_1 + l_2 + l')}}S^2(R^{2(k_1 + k_2 + k_3) -2 } \log^{2l_1 + 2l_2 -1 + l'}(R), (\mathcal{Q}^{\beta_{l_1 +l_2 + l'} })'),\\
			&\;\; k_1 + l_1 \in I_{j-1},\; k_2 + l_2 \in \cup_{1 \leq\ell \leq j-1 } I_{\ell},\;\;k_3  \in \cup_{1 \leq\ell \leq j-1 } I_{\ell},\;\; l' \geq 0,\\[12pt]
			&\frac{t^{2\nu(k_1 + k_2 + k_3)} }{(t \lambda)^{2(l_1 + l_2 + l_3 + l')}}S^2(R^{2(k_1 + k_2 + k_3) } \log^{2l_1 + 2l_2 + 2l_3 -3 + l'}(R), (\mathcal{Q}^{\beta_{l_1 +l_2 + l_3+ l'} })'),\\
			&\;\; k_1 + l_1 \in I_{j-1},\; k_2 + l_2 \in \cup_{1 \leq\ell \leq j-1 } I_{\ell},\;\; k_3  + l_3 \in \cup_{1 \leq\ell \leq j-1 } I_{\ell},\;\; l' \geq 0.
		\end{align*}
		Therefore integrating these terms via Lemma \ref{lem:recoverz1}, proves linear combinations in $X_{j}$ as above, where we note in the second and the fourth of the latter four blocks, we need to factor off $ R^{2} = a^2 (t \lambda)^2$ and use the trick from before, namely we observe terms in the span of
		\begin{align*}
			& \frac{t^{2\nu(k_1 + k_2 + k_3)}}{(t\lambda)^2} S^2(R^{2(k_1 + k_2 + k_3) - 2} \log^3(R)),\\[4pt]
			& \frac{t^{2\nu(k_1 + k_2 + k_3)}}{(t\lambda)^{2(1 + l_3)}} S^2(R^{2(k_1 + k_2 + k_3) } \log^{1 + 2l_3 }(R), \mathcal{Q}^{\beta_{l_3}}),\\[4pt]
			& \frac{t^{2\nu(k_1 + k_2 + k_3)}}{(t\lambda)^{2(2 + l')}} S^2(R^{2(k_1 + k_2 + k_3) } \log^{3 + l'}(R), \mathcal{Q}^{\beta_{2 + l'}}),\\[4pt]
			& \frac{t^{2\nu(k_1 + k_2 + k_3)}}{(t\lambda)^{2(1 + l' + l_3 )}} S^2(R^{2(k_1 + k_2 + k_3) } \log^{1 + l' + 2l_3}(R),( \mathcal{Q}^{\beta_{1 + l' + l_3}})'),\\[4pt]
			& \frac{t^{2\nu(k_1 + k_2 + k_3)} }{(t \lambda)^{2(l_j + l'+1)}}S^2(R^{2(k_1 + k_2 + k_3) } \log^{2l_j + l' +1}(R), \mathcal{Q}^{\beta_{l_j + l'}}),\;\; j =1,2,\;\; l' \geq 1,\\[4pt]
			& \frac{t^{2\nu(k_1 + k_2 + k_3)} }{(t \lambda)^{2(l_j + l' + l_3 )}}S^2(R^{2(k_1 + k_2 + k_3) } \log^{2l_j + l' + 2l_3 -1}(R), \mathcal{Q}^{\beta_{l_j + l' + l_3 -1}}),\;\; j =1,2,\;\; l' \geq 1,\\[4pt]
			&\frac{t^{2\nu(k_1 + k_2 + k_3)} }{(t \lambda)^{2(l_1 + l_2 + l')}}S^2(R^{2(k_1 + k_2 + k_3) } \log^{2l_1 + 2l_2 -1 + l'}(R), \mathcal{Q}^{\beta_{l_1 +l_2 + l'} }),\\[4pt]
			&\frac{t^{2\nu(k_1 + k_2 + k_3)} }{(t \lambda)^{2(l_1 + l_2 + l_3 + l' -1)}}S^2(R^{2(k_1 + k_2 + k_3) } \log^{2l_1 + 2l_2 + 2l_3 -3 + l'}(R), \mathcal{Q}^{\beta_{l_1 +l_2 + l_3+ l' -1} }),
		\end{align*}
		which leads exactly to the required asymptotic  in $X_j$ by the usual embeddings. The last term in \eqref{that-is-the-quadra-5}, i.e.
		\[
		\lambda^{-2}\big(\Box^{-1} \partial_t^2(\lambda^2 \overline{z_{j-1}^{\ast}} \cdot z_{j-1}^{\ast})\big) z_{j-1}.
		\]
		is in principle handled with the same argument. In fact the spaces are the same as in the latter case, however we of course need to  exchange the conditions 
		\[
		k_3  + l_3 \in \cup_{1 \leq\ell \leq j-1 } I_{\ell},\;\;  (k_3  \in \cup_{1 \leq\ell \leq j-1 } I_{\ell}),\;\; k_1 \in I_{j-1},\;\; (k_1 + l_1 \in I_{j-1}),
		\]
		for 
		\[
		k_1  + l_1 \in \cup_{1 \leq\ell \leq j-1 } I_{\ell},\;\;  (k_1  \in \cup_{1 \leq\ell \leq j-1 } I_{\ell}),\;\; k_3 \in I_{j-1},\;\; (k_3 + l_3 \in I_{j-1}).
		\]
		In particular the other term we skipped,  $\lambda^{-2}\big(\Box^{-1} \partial_t^2(\lambda^2 \overline{z_{j-1}^{\ast}} \cdot z_{j-1})\big) z_{j-1}^{\ast}$ simply requires the same change of constraints with $k_2, l_2$ replacing $k_3, l_3$.\\[4pt]
		In order to conclude the induction for $ z_j$, we still have to consider the cubic `pure Schr\"odinger' interaction, i.e.
		\begin{align} \label{that-is-the-quadra-6}
			\triangle_j \big(  |z|^2z\big),
		\end{align}
		for which it essentially suffices to consider (note, as for the terms above, we simplify here slightly, which however makes no difference in the argument)
		\begin{align} \label{that-is-the-quadra-7}
			&\overline{z_{j-1}} \cdot z_{j-1}^{\ast} \cdot z_{j-1}^{\ast} = \;\overline{z_{j-1}} \cdot (z_{j-1}^{\ast})^2,\\
			& \overline{z_{j-1}^{\ast}} \cdot z_{j-1} \cdot z_{j-1}^{\ast} = \;| z_{j-1}^{\ast} |^2z_{j-1}.
		\end{align}
		These terms are of course handled the same way, i.e. we already obtained that $\overline{z_{j-1}} \cdot z_{j-1}^{\ast}$ lies in the span of 
		\begin{align*}
			& t^{2\nu(k_1 + k_2)} S^2(R^{2(k_1 + k_2) - 4} \log^2(R)),\;\; k_1 \in I_{j-1},\; k_2 \in \cup_{1 \leq\ell \leq j-1 } I_{\ell},\\
			& \frac{t^{2\nu(k_1 + k_2)} }{(t \lambda)^{2l_1}}S^2(R^{2(k_1 + k_2) -2} \log^{2l_1}(R), \mathcal{Q}^{\beta_{l_1}}),\;\; k_1 + l_1 \in I_{j-1},\; k_2 \in \cup_{1 \leq\ell \leq j-1 } I_{\ell},\\
			&\frac{t^{2\nu(k_1 + k_2)} }{(t \lambda)^{2l_2}}S^2(R^{2(k_1 + k_2) -2} \log^{2l_2}(R), \mathcal{Q}^{\beta_{l_2}}),\;\; k_1  \in I_{j-1},\; k_2 + l_2 \in \cup_{1 \leq\ell \leq j-1 } I_{\ell},\\
			& \frac{t^{2\nu(k_1 + k_2)} }{(t \lambda)^{2(l_1 + l_2)}}S^2(R^{2(k_1 + k_2) } \log^{2l_1 + 2l_2 -2}(R), \mathcal{Q}^{\beta_{l_1 +l_2} }),\;\; k_1 + l_1 \in I_{j-1},\; k_2 + l_2 \in \cup_{1 \leq\ell \leq j-1 } I_{\ell}.\\[2pt]
		\end{align*}
		Directly multiplying by $ z_{j-1}^{\ast}$ gives
		\;\\
		\begin{align*}
			& t^{2\nu(k_1 + k_2 + k_3)} S^2(R^{2(k_1 + k_2 + k_3) - 6} \log^3(R)),\;\; k_1 \in I_{j-1},\; k_2 \in \cup_{1 \leq\ell \leq j-1 } I_{\ell}, k_3 \in \cup_{1 \leq\ell \leq j-1 } I_{\ell}\\
			& \frac{t^{2\nu(k_1 + k_2 + k_3)}}{(t\lambda)^{2l_3}} S^2(R^{2(k_1 + k_2 + k_3) - 4} \log^{1 + 2l_3}(R), \mathcal{Q}^{\beta_{l_3}}),\;\; k_1 \in I_{j-1},\; k_2 \in \cup_{1 \leq\ell \leq j-1 } I_{\ell}, k_3 + l_3 \in \cup_{1 \leq\ell \leq j-1 } I_{\ell}\\
			& \frac{t^{2\nu(k_1 + k_2 + k_3)} }{(t \lambda)^{2l_j}}S^2(R^{2(k_1 + k_2 + k_3) -4} \log^{2l_j +1}(R), \mathcal{Q}^{\beta_{l_j }}),\;\; k_3 \in \cup_{1 \leq\ell \leq j-1 } I_{\ell},\\
			& \frac{t^{2\nu(k_1 + k_2 + k_3)} }{(t \lambda)^{2(l_j + l_3)}}S^2(R^{2(k_1 + k_2 + k_3) -2} \log^{2l_j + 2l_3 -1}(R), \mathcal{Q}^{\beta_{l_j + l_3 }}),\;\; k_3 + l_3 \in \cup_{1 \leq\ell \leq j-1 } I_{\ell},\\
			& \frac{t^{2\nu(k_1 + k_2 + k_3)} }{(t \lambda)^{2(l_1 + l_2)}}S^2(R^{2(k_1 + k_2 + k_3) -2 } \log^{2l_1 + 2l_2 -1}(R), \mathcal{Q}^{\beta_{l_1 +l_2} }),\;\; k_1 + l_1 \in I_{j-1},\; k_2 + l_2 \in \cup_{1 \leq\ell \leq j-1 } I_{\ell}, k_3 \in \cup_{1 \leq\ell \leq j-1 } I_{\ell} \\
			& \frac{t^{2\nu(k_1 + k_2 + k_3)} }{(t \lambda)^{2(l_1 + l_2 + l_3)}}S^2(R^{2(k_1 + k_2 + k_3) } \log^{2l_1 + 2l_2 + 2l_3 -3}(R), \mathcal{Q}^{\beta_{l_1 +l_2 + l_3} }),\;\; k_1 + l_1 \in I_{j-1},\; k_2 + l_2 \in \cup_{1 \leq\ell \leq j-1 } I_{\ell},\\
			& k_3 + l_3 \in \cup_{1 \leq\ell \leq j-1 } I_{\ell}.
		\end{align*}
		Applying Lemma \ref{lem:recoverz1} gives linear combinations of the form $X_j$ as before and thus finally  $ z_j \in X_{j-1}$. For the error we simply note the representation \eqref{thatstheerror} observed above, i.e.
		\begin{align} \label{that-erro}
			e^z_j(t,R) =&\; i t^{1 + 2 \nu} \cdot \partial_t  z_j  +   t^{2 \nu} \cdot ( \alpha_0 - i( \f12 + \nu)\Lambda)z_j\\ \nonumber
			& \;\; + \lambda^{-2}\Box^{-1}( \lambda^2 \Delta |u_j^{\ast}|^2) u_j^{\ast} -  \lambda^{-2}\Box^{-1}( \lambda^2\Delta |u_{j-1}^{\ast}|^2) u_{j-1}^{\ast}.
		\end{align}
		where $ u_j^{\ast} = z_j^{\ast} + W$. Since $ z_j \in X_j$ we have $ \partial_t z_{j}, \partial_R z_j \in (X_j) ' $ and obtain from $ t^{2\nu}\cdot X_j \subset X^{2}_{j+1}$ and  applying $ \lambda^{-2} \Box^{-1}(\cdot) $ to the latter line of \eqref{that-erro} as above, that $e^z_j $  must be a linear combination of terms in $ (X^{2}_{k})'$ where $ k \geq j+1$. Further since $ W + z_{j}^{\ast} \in X_j$, we obtain from
		\begin{align*}
			n_j^{\ast} & = \Box^{-1}(\triangle(\lambda^2 | W + z_j^{\ast}|^2)),
		\end{align*}
		by application of the $ \Box^{-1}(\cdot) $ parametrix, that $n_j^{\ast}\in \bigcup_{k \leq j} \lambda^2 (X_{k})'$ and in fact, since
		\begin{align*}
			n_j = 2 \lambda^2  \text{Re}(\overline{z_j}W) + \Box^{-1} \partial_{tt} ( 2 \lambda^2  \text{Re}(\overline{z_j}W))  + \Box^{-1} \triangle ( \lambda^2 \triangle_j|z|^2  ) ,
		\end{align*}
		we recall  the above argument replacing $ z_{j-1}, z_{j-1}^{\ast}$ by $ z_j, z_j^{\ast}$ in order to obtain $ n_j \in \lambda^2 (X_j)'$. Lastly for the error
		$$ 
		e^{n}_j = \Box n_j^{\ast} - \triangle(\lambda^2 | W + z_j^{\ast}|^2 \in \lambda^2 X_{j+1}
		$$
		by direct control of the $ l'\geq 1$ parameter in the $\Box^{-1}$ approximation.
	\end{proof}
	\;\\
	Now, calculating $ N \in \Z_+$ steps of this iteration scheme, we have $ n_N^{\ast}, z_N^{\ast}$ as above and thus recall the error functions are
	\begin{align*}
		&e^z_N(t,R) = 	i t^{1 + 2 \nu}\partial_t  u_N^{\ast} + \Delta u_N^{\ast}  +  ( \alpha_0 t^{2 \nu}  - i( \f12 + \nu)t^{2\nu}(1 + R\partial_R))u_N^{\ast}\\ 
		& \hspace{2cm} + \lambda^{-2} n_N^{\ast}\cdot u_N^{\ast},\;\;\ u_N^{\ast} =  W + z^{\ast}_N ,\\[4pt]
		& e^n_N(t,R)  = \Box(n_N^{\ast}) - \Delta( \lambda^2|u_N^{\ast}|^2),\;\;\; n_N^{\ast} = \Box^{-1}  \Delta( \lambda^2|u_N^{\ast}|^2),
	\end{align*}
	where the latter denotes the wave parametrix. In conclusion, for  each $N \in \Z_+$, we have the following  $N$-approximate solution for \eqref{main-eq-z-2}.
	\begin{Cor}\label{cor:This -end-cor} The system \eqref{eq:trianglez-inductivedef} - \eqref{eq:trianglez-inductivedef-2} has a unique solution $ \{ z_j \}_{j \geq 1}$ of the form of the S-space provided in Definition \ref{defn:SQbetal} such that (for $ t > 0$ fix)  there holds $z_j(t, 0) = \partial_R z_j(t,0) = 0$ and further\\[4pt]
		$\bullet$\;\;	For all $j \geq 1$ we have $z_j \in X_j, n_j \in \lambda^2 (X_j )'$ and\\[4pt]
		$\bullet$\;\; For all $N \in \Z_+$ we there holds 
		$$ e_N^z \in \cup_{k \geq N+1}(X^{2}_{k})',  \; e_N^n \in  \cup_{k \geq N+1}\lambda^2  (X_{k})''.$$
		In fact for the latter union we can replace $N+1$ by any $ N+1 + l'$ with $ l' \in \Z_+$.
	\end{Cor}
	\;\\
	Now we can finally provide the approximate solution for the inner region $\mathcal{I}$. 
	\begin{Def}[Approximation in $\mathcal{I}$] \label{defn: Final-R}We define for $ N \in \Z_+, N \gg1 $ the interior approximation
		\begin{align}
			&u_{\text{app},\mathcal{I}}^N (t,R, a) : =  W(R) + z_N^{\ast}(t,R, a),\;\; a = \tfrac{r}{t},\; R = \lambda(t)r, \;(t,r) \in \mathcal{I}\\[3pt]
			& n_{\text{app},\mathcal{I}}^N(t,R,a ) : = n_N^{\ast}(t,R, a) = \Box^{-1}  \Delta( \lambda^2|u_{\text{app},\mathcal{I}}^N|^2),\;\;(t,r) \in \mathcal{I},
		\end{align} 
		and further the error
		\begin{align}
			&e^{ N, (1)}_{ \text{app}, \mathcal{I}}(t,R) : = 	i t^{1 + 2 \nu} \partial_tu_{\text{app},\mathcal{I}}^N + \Delta u_{\text{app},\mathcal{I}}^N +  ( \alpha_0 t^{2 \nu}  - i( \f12 + \nu)t^{2\nu}(1 + R\partial_R))u_{\text{app},\mathcal{I}}^N\\  \nonumber
			& \hspace{3cm} + \lambda^{-2} n_{\text{app},\mathcal{I}}^N\cdot u_{\text{app},\mathcal{I}}^N,\;\;\; (t,r) \in \mathcal{I},\\
			& e^{N, (2)}_{ \text{app}, \mathcal{I}}(t,R) : = \Box n_{\text{app},\mathcal{I}}^N -  \lambda^2\Delta(| u_{\text{app},\mathcal{I}}^N   |^2),\;\;\;(t,r) \in \mathcal{I}.
		\end{align} 
	\end{Def}
	\;\\
	The Corollary \ref{cor:This -end-cor} implies straight forward pointwise estimates (using the definition of the S-space). 
	\begin{Lemma}\label{lem:point-error} Let $  0 < t \ll1,N \in \Z_+$, then for all $ (t,R) = (t,r \lambda(t))$ with $ (t,r) \in \mathcal{I}$ there holds
		\begin{align*}
			&	| R^{-j}\partial_R^{i} e^z_N(t,R, R(t\lambda)^{-1}) | \leq C_{N, i,j}\cdot  t^{2\nu(N+1) } \langle R\rangle^{2N - i -j}(1 + \log(1 + R))^{s(N)},\\
			&	| R^{-j} e^n_N(t,R, R(t\lambda)^{-1}) | \leq C_{N, j} \cdot t^{2\nu(N+1)} \langle R\rangle^{2N-j}(1 + \log(1 + R))^{\tilde{s}(N)},
		\end{align*}
		where $ \langle R \rangle  = (1 + R^2)^{\f12}$,\;$0 \leq i + j \leq 2,\; i,j \geq 0,\; i <2 $ and $ s(N),\; \tilde{s}(N) \in \Z_+$. 
	\end{Lemma}
	\;\\
	Thus we obtain the following consequence of Corollary \ref{cor:This -end-cor} and Lemma \ref{lem:point-error}.  Let $ \chi_{\mathcal{I}} : (0, \infty) \times\R^4  \to [0,1] $ be a smooth cut-off function with support in the region where $ (t,|x|) \in \mathcal{I}$. 
	\begin{Lemma}[Estimates in $\mathcal{I}$] \label{lem:estimates-inner} Let $\alpha_0,\nu \in \R $ with %$ \nu > \frac{3}{2}$ 
		$ \nu \geq 2$ and  %$ 0 <  \epsilon_1  \leq \tfrac{1}{4}$
		$ 0 < \epsilon_1 < \f12$ be fix. For any $ N \in \Z_+, N \gg1$ there exists $ t_0 = t_0(\alpha_0, \nu, N) \leq 1 $ such that  $ u_{\text{app},\mathcal{I}}^N, n_{\text{app},\mathcal{I}}^N$ in Definition \ref{defn: Final-R} satisfy for all $ t \in (0, t_0)$ 
		\begin{align*}
			&\| \chi_{\mathcal{I}} \big( \lambda^{-2}n_{\text{app},\mathcal{I}}^N - W^2 \big) \|_{L^{\infty}_R} + \| \chi_{\mathcal{I}} \big( u_{\text{app},\mathcal{I}}^N - W \big) \|_{L^{\infty}_R} \leq C_{\nu, |\alpha_0|} \; t^{2 \nu\;- },\;\\[4pt]
			& \| \chi_{\mathcal{I}} R^{-j} \partial_R^i \big( u_{\text{app},\mathcal{I}}^N - W \big) \|_{L^{\infty}_R} \leq C_{\nu, |\alpha_0|} \;  t^{2\nu },\;\;1 \leq i + j \leq 2,\;\; i,j \geq 0\\[4pt]
			& \| \chi_{\mathcal{I}} R^{-j} \partial_R^i \big( \lambda^{-2}n_{\text{app},\mathcal{I}}^N - W^2 \big) \|_{L^{\infty}_R} \leq C_{\nu, |\alpha_0|} \;  t^{2\nu },\;\;1 \leq i + j \leq 2,\;\;i,j \geq 0,\;\; i < 2,\\[4pt]
			& \| \chi_{\mathcal{I}} R^{-j} \partial_R^i \big( u_{\text{app},\mathcal{I}}^N - W \big) \|_{L^2_{R^3dR}} \leq C_{\nu, |\alpha_0|} \;  t^{ \frac{3}{2}\epsilon_1 } \cdot t^{(\nu - \epsilon_1)(i +j)},\;\;0 \leq i + j \leq 2,\; i,j\geq 0\\[4pt]
			& \| \chi_{\mathcal{I}} R^{-j} \partial_R^i \big( \lambda^{-2}n_{\text{app},\mathcal{I}}^N - W^2 \big) \|_{L^2_{R^3dR}} \leq C_{\nu, |\alpha_0|} \;  t^{\frac{3}{2} \epsilon_1 } \cdot t^{(\nu - \epsilon_1)(i +j)},\;\;0 \leq i + j \leq 2,\; i,j\geq 0, i < 2.
		\end{align*}
		Further we have for the remainders $e^{N,(1)}_{\mathcal{I}},  e^{N,(2)}_{\mathcal{I}}$
		\begin{align*}
			&  \| \chi_{\mathcal{I}} R^{-j} \partial_R^i  e^{N, (1)}_{\text{app}, \mathcal{I}} \|_{L^2_{R^3dR}} \leq C_{\nu, |\alpha_0|} \; t^{\nu(i + j)} \cdot t^{2 \epsilon_1 N},\;\; 0 \leq i + j \leq 2,\; i,j \geq 0,\; i <2,\\[4pt]
			&  \| \chi_{\mathcal{I}} R^{-j} \lambda^{-2}e^{N, (2)}_{\text{app}, \mathcal{I}} \|_{L^2_{R^3dR}} \leq C_{\nu, |\alpha_0|} \; t^{\nu j} \cdot t^{2 \epsilon_1 N},\;\; 0 \leq  j \leq 2,
		\end{align*}
		were for the latter we may replace $ N$ on the right by $  N + l' $ for any $ l' \in \Z_+$.
	\end{Lemma}
	\begin{Rem}  We note in principal a restriction $ 0 \leq j \leq 2$ is necessary for multiplying with $R^{-j}$ when considering the $ R \ll1 $ expansion of order $ \mathcal{O}(R^2)$ in the definition of $ X_j$. It is possible to show, via interchanging $ R^2= a^2 (t \lambda)^{2}$ in the regime  $ 1 \leq a \leq R \ll1$, that the functions in $ X_j$ have $ (R \ll1)$-asymptotic $ \mathcal{O}(R^{2j})$. Further for derivatives $\partial_R^i$ such a restriction in the $R \ll1 $ regime is of course unnecessary.
	\end{Rem}
	\begin{Rem} We remark that estimating derivatives of $e^n_N,\; e^z_N, z_N^{\ast}, n_N^{\ast} $ is problematic because of the singular expansion at $ a =1$, hence the restriction to $ i \leq 1 $. We note that if $ \nu  \gg1 $ is large (in fact it suffices $\nu \geq 3$), then we can estimate in the $R \gg1$ region at least up to $\lfloor 2\nu -3\slash2\rfloor -1 $ derivatives for $e^z_N$ and  $\lfloor 2\nu -3\slash2\rfloor -2$ derivatives for $e^n_N$. Combining this with the remark above, we can prove Lemma \ref{lem:point-error} and Lemma \ref{lem:estimates-inner} for derivatives up to 
		\begin{align*}
			&0 \leq i + j \leq \min\{2N,\lfloor 2\nu -3\slash2\rfloor -1 \},\;\; \text{for}\;\; e_{\text{app},\mathcal{I}}^{N, (1)}, \\
			& 0 \leq i + j \leq \min\{2N,\lfloor 2\nu -3\slash2\rfloor -2 \},\;\; \text{for}\;\; e_{\text{app},\mathcal{I}}^{N, (2)}, 
		\end{align*}
		respectively. Similarly we can estimate up to $0 \leq i  \leq \min\{2N,\lfloor 2\nu -3\slash2\rfloor \}$ derivative of $ u_{\text{app},\mathcal{I}}^N$ and $ 0 \leq i \leq \min\{2N,\lfloor 2\nu -3\slash2 \rfloor -1 \}$ derivatives of $ n_{\text{app},\mathcal{I}}^N$.
	\end{Rem}
	\begin{proof}[Proof of Lemma \ref{lem:point-error}] Let us start with the first line and the case of the second kind of terms, i.e. $ l \geq 2$. The leading order error terms consists of a sum 
		\[
		\sum_{k + l \geq N+1} \frac{t^{2 \nu k}}{(t \lambda)^{2l}} f_{k , l }(R, a),
		\]
		where  $ \partial_R^m f_{k , l }(R, a) = \sum_{l_1 + l_2 =m} (t\lambda)^{- l_2} \partial_R^{l_1} \partial_a^{l_2} f_{l, k}(R,a) $. Thus, using the $ a-$ expansions of Definition \ref{defn:SQbetal}, we may similarly expand this in powers of $ a, a^{-1}$ (as in the Definition) with the factor $R^{-l_1 - l_2} = R^{-m}$. Hence we observe for the leading order term of the $ 1 \ll  R $ expansion (we use $ k > N+1 - l$ )
		\begin{align*}
			&\frac{t^{2 \nu k}}{(t \lambda)^{2l}} R^{-m -2} \big( R^{2k} \log^{s_{\ast}(l)}(R)\big)f_{k , l }(a)\\
			&=  	\frac{t^{2 \nu k}}{(t \lambda)^{2l}} R^{- 2l}\cdot \big( R^{2(k - (N+1 -l))} \cdot R^{2N - i - j} \log^{s_{\ast}(l)}(R)\big)f_{k , l }(a)\\
			&= R^{2(k - (N+1 -l))} t^{2 \nu (k - (N+1 - l))} \frac{t^{2 \nu(N +1)}}{(t \lambda)^{2l}} t^{- 2\nu l}R^{- 2l}\cdot \big(R^{2N - i - j} \log^{s_{\ast}(l)}(R)\big)f_{k , l }(a),
		\end{align*}
		where we bound $ R^{2(k - (N+1 -l))} t^{2 \nu (k - (N+1 - l))}  \lesssim 1$ using  $ (t,r) \in \mathcal{I}$. Now in the regime $ a \lesssim 1 $, we simply use $R^{- 2l} \lesssim1 ,\;\;\frac{t^{- 2\nu l}}{(t \lambda)^{2l}} \lesssim 1$ and in case $ a \gg1 $ we have (to leading order) $ f_{k , l }(a) \sim a^{\beta_l} \log^{\tilde{s}(l)}(a)$. In this case it follows
		\[
		\frac{a^{\beta_{l}}}{(t \lambda)^{2l}} \lesssim  t^{\nu(l+1) -l}\cdot t^{(\f12 - \epsilon_1)(l+1)},\;\text{and} \;\; t^{- 2\nu l} R^{-2l} = t^{-2} a^{-2l} \lesssim t^{-2}.
		\]
		Now we can expand $ \log^{\tilde{s}(l)}(a)  = (\log(R) - \log(t \lambda))^{\tilde{s}(l)} = \sum_{l_1 \leq  \tilde{s}(l) - l_2} \log^{l_1}(R)\log^{l_2}(t\lambda)$, where the latter factors are bounded by  multiplying with the product of the above upper bounds  for instance.  Further the estimate where $ R \lesssim1 $ is obtained similarly, where we note $ f_{k,l}(R,a) = \mathcal{O}(R^2)$ as $ 0 < a \leq R \ll 1$. For the first kind of terms in the $ X_j$ spaces, i.e. without $a-$dependence, the argument is the same (slightly simpler) and we in fact refer to \cite[Section 2.1]{schmid}, \cite[Section 2.2]{Perelman}, \cite{OP}.\\[4pt]
		Considering the wave error $ e^n_N(t, R, a)$, we follow the above argument, except we use (in the $a \gg1 $ regime) for some $ l ' \in \Z$
		\begin{align*}
			&\frac{t^{2 \nu(N +1)}}{(t \lambda)^{2(l + l')}} t^{- 2\nu l}R^{2 - 2l}\cdot \big(R^{2N - i - j} \log^{s_{\ast}(l)}(R)\big)f_{k , l + l'}(a)\\
			&\sim  t^{2 \nu(N +1)}\frac{a^{\beta_{l + l'}}}{(t \lambda)^{2(l + l')}} t^{- 2\nu l}R^{2 - 2l}\cdot \big(R^{2N - i - j} \log^{\tilde{s}_{\ast}(l)}(R) \log^{\tilde{\tilde{s}}(l)}(a)\big),
		\end{align*}
		where we crudely bound (at least in the latter one)
		\[
		\frac{a^{\beta_{l + l'}}}{(t \lambda)^{2(l + l')}} \lesssim   t^{\nu(l + l'+1) -l - l'}\cdot t^{(\f12 - \epsilon_1)(l + l'+1)},\;\text{and} \;\; t^{- 2\nu l} R^{2-2l} \lesssim t^{-2\nu l}.
		\]
		This of course simply requires to take $ l' \gg1 $ (controlling the accuracy of the $\Box^{-1}$ parametrix).
	\end{proof}
	\begin{proof}[Proof of Lemma \ref{lem:estimates-inner}]
		As in the above proof, we need to estimate the leading order terms using $ R \lesssim t^{\epsilon_1 - \nu}$ and $ a \lesssim t^{\epsilon_1 - \f12}$. Thus, in the worst case scenario of the  $ R \geq a \gg1 $ expansion we have ( for $ l \geq 2$ and $ l + k \leq N$)
		\begin{align*}
			\frac{t^{2\nu k}}{(t \lambda)^{2l}} R^{2k} \log^{s_{\ast}(l)}(R) \cdot a^{\beta_l} \log^{\tilde{s}_{\ast}(l)}(a) & \lesssim t^{\nu(l+1) -l}\cdot t^{(\f12 - \epsilon_1)(l+1)} \cdot t^{2k\epsilon_1 } \cdot \log^{\tilde{\tilde{s}}(l)}(t) \lesssim t^{2\nu}, 
		\end{align*}
		otherwise we have similarly for $ k \geq 1$
		\begin{align*}
			t^{2\nu k} R^{2k-2} \log(R) & \lesssim t^{2\nu}\cdot t^{(2k-2)\epsilon_1 } \cdot \log(t) \lesssim t^{2\nu\; -}.
		\end{align*}
		Now estimating derivatives as indicated in the proof of Lemma \ref{lem:point-error} (differentiating through $ a = R (t \lambda)^{-1}$), we need to bound say in case all derivatives fall on the $ R -$ dependence
		\begin{align*}
			&\frac{t^{2\nu k}}{(t \lambda)^{2l}} R^{2k - i -j} \log^{s_{\ast}(l)}(R) \cdot a^{\beta_l} \log^{\tilde{s}_{\ast}(l)}(a)  \lesssim t^{2\nu} \cdot t^{(\nu - \epsilon_1)(i+j)},\\
			&t^{2\nu k} R^{2k-2 - i-j} \log(R)  \lesssim t^{2\nu}\cdot t^{(2k-2)\epsilon_1 } \cdot \log(t) \cdot t^{(\nu - \epsilon_1)(i+j)} \lesssim t^{2\nu},
		\end{align*}
		and similarly in the other cases of the Leibniz formula.
		Since  $ n_{\text{app},\mathcal{I}}^N - \lambda^2 W^2$  is a linear combination of terms in $ \lambda^2 (X_j)' $ where $ j \leq N $, we obtain the estimates  for the wave part by the same respective arguments (and only one order of derivatives below is allowed of course).\\[3pt]
		For the $L^2$-bounds, we then integrate as follows, where we restrict again to $ R \geq a \gg1 $ and $ R \geq R_0 \gg1 $. Thus we estimate similar as  before and for some $ 0 < \delta \ll1 $
		\begin{align*}
			&\frac{t^{2\nu k}}{(t \lambda)^{2l}}\cdot 	\bigg( \int_{R_0}^{\infty}  \big(x^{4k - 2i -2j }\log^{2s(l)}(x) \cdot \big(\frac{x}{(t \lambda)}\big)^{2\beta} (\log(x) - \log(t \lambda))^{2\tilde{s}(l)}\big)\cdot x^3\;dx\bigg)^{\f12}\\
			& \lesssim t^{2 \nu k} \cdot   t^{\nu(l +1) -l} \cdot t^{(\f12 - \epsilon_1)(l+1)}\cdot \log^{\tilde{\tilde{s}}(l)}(t)\bigg( \int_{R_0}^{\infty}  \big(x^{4k + 4 - 2i -2j  + \delta} \cdot x^{-1 - \delta}\;dx\bigg)^{\f12}\\
			& \lesssim t^{2 \nu k + 2 \nu}  \cdot  t^{(\epsilon_1 - \nu)(2k +2 + \frac{\delta}{2} -j -i)}  \log^{\tilde{\tilde{s}}(l)}(t)\bigg( \int_{R_0}^{\infty} x^{-1 - \delta}\;dx\bigg)^{\f12}\\
			& \lesssim t^{\epsilon_1(2k +2)} \cdot t^{(\nu - \epsilon_1)(j +i - \frac{\delta}{2})}  \log^{\tilde{\tilde{s}}(l)}(t).
		\end{align*}
		Similarly we have
		\begin{align*}
			t^{2\nu k}\cdot 	\bigg( \int_{R_0}^{\infty}  x^{4k -4- 2i -2j }\log^{2}(x) \cdot x^3\;dx\bigg)^{\f12}  &\lesssim t^{2 \nu k}  \cdot  t^{(\epsilon_1 - \nu)(2k + \frac{\delta}{2} -j -i)}  \log(t)\bigg( \int_{R_0}^{\infty} x^{-1 - \delta}\;dx\bigg)^{\f12}\\
			& \lesssim t^{2 \nu \epsilon_1}  \cdot  t^{(\nu - \epsilon_1)( - \frac{\delta}{2} + j + i)}  \log(t),
		\end{align*}
		and again the same argument for the wave part.  Now for estimating the $L^2$ norm of the remainder, we simply integrate the pointwise bounds in Lemma \ref{lem:point-error}. This hence requires, considering the Schr\"odinger part say, to show the bound
		\begin{align*}
			& t^{2\nu (N+1)}\bigg( \int_{R_0}^{\infty}  x^{4N - 2i -2j }\log^{2s(N)}(x) \big)\cdot x^3\;dx\bigg)^{\f12}\\
			&\lesssim  t^{2\nu (N+1)}\cdot x^{(\epsilon_1 - \nu)(2N + 2 +\frac{\delta}{2} - i - j) }\log^{2s(N)}(t) \bigg( \int_{R_0}^{\infty}  x^{-1 - \delta}\;dx\bigg)^{\f12}\\
			& \lesssim t^{2\epsilon_1(N+1)}\cdot x^{( \nu - \epsilon_1)(-\frac{\delta}{2} + i + j) }\log^{2s(N)}(t).
		\end{align*}
		Finally the estimate for the wave error is obtained analogously.
	\end{proof}
	
	\begin{Rem}\label{rem:modulatinglambda} All the results of this section still apply if we replace the scaling function $\lambda(t) = t^{-\frac12-\nu}$ by a function of the form 
		\[
		\lambda(t) = \tilde{\lambda}\cdot t^{-\frac12-\nu}
		\] 
		where $\tilde{\lambda}>0$ is a constant. Furthermore, interpreting $ u_{\text{app},\mathcal{I}}^N,\, \lambda^{-2}n_{\text{app},\mathcal{I}}^N$ with the preceding choice of $\lambda$ as functions of $\tilde{\lambda}$, the inequalities of Lemma~\ref{lem:estimates-inner} continue to hold if we replace 
		\[
		W,\,W^2
		\]
		by 
		\[
		\Lambda W: = \partial_{\tilde{\lambda}}\big(\tilde{\lambda}\cdot W(\tilde{\lambda}\cdot)\big)|_{\tilde{\lambda} = 1},\,\tilde{\Lambda} W^2: = \partial_{\tilde{\lambda}}\big(\tilde{\lambda}^2\cdot W^2(\tilde{\lambda}\cdot)\big)|_{\tilde{\lambda} = 1},
		\]
		respectively, as well as $u_{\text{app},\mathcal{I}}^N,\, \lambda^{-2}n_{\text{app},\mathcal{I}}^N$ by their partial derivative with respect to $\tilde{\lambda}$ evaluated at $\tilde{\lambda} = 1$. 
	\end{Rem}

	\section{The  self-similar parabolic  region $ t^{\frac12 + \epsilon_1} \lesssim r \lesssim t^{\frac12 - \epsilon_2}$}\label{sec:self}
	\;\\
	We now choose $ 0 < \epsilon_2 < 1 $ to be fixed below. In particular we recall the definition of the \emph{parabolic self similar region}
	\[
	\mathcal{S} : = \big\{  (t,r) ~|~ t^{\f12 + \epsilon_1} \frac{1}{c_2} \leq r \leq c_2 t^{\f12 - \epsilon_2}  \big\}
	\]
	where $c_2 > c_1^{-1} > 0$ is  a constant. In the present section, we like to change to the coordinates $ (t, y) = (t, t^{- \f12} r)$, which means $  t^{\epsilon_1} \lesssim y \lesssim t^{- \epsilon_2}$. We intend to  start with constructing an approximate solution in $\mathcal{S}$ with \emph{prescribed asymptotic data} where $  0 < y \ll1 $ and which is therefore \emph{close to} the previous correction of the  renormalized bulk solution in the overlap $\mathcal{I} \cap \mathcal{S}$.
	\begin{Rem}[Notation] We recall in this section we use $ \mathcal{O}(g(y))$ for representatives $f(y)$ of smooth functions with 
		\[
		f(y)\cdot g^{-1}(y) = \sum_{k \geq 0} c_k y^{2k},\;\;\; 0 < y \ll1,\;\; g(y) \neq 0.
		\]
		in an absolutely convergent sense.
	\end{Rem}
	The approximate solution in the preceding Section \ref{sec:inner} is 'well behaved' in the sense of Lemma \ref{lem:estimates-inner} in case  $(t,r) \in \mathcal{I}$, i.e. such that $r\lesssim t^{\frac12+\epsilon_1}$. To be more precise, we recall 
	$$ u_{\text{app},\mathcal{I}}^N = W + z_1 + z_2 + \dots + z_N,$$
	where $ z_{j} \in X_j$. Hence, according to Definition of the $X_j$ spaces and Lemma \ref{lem: induct},  we rewrite for each $ 1 \leq j \leq N$ 
	\[
	z_{j}(t,R,a) = \sum_{k + l \geq j} z_{k,l}(t,R,a),
	\]
	which has a generic  expansion as $ 1 \ll  a \leq R$ according to the asymptotic in the S-space \ref{defn:SQbetal}
	\;\\
	\boxalign{
		\begin{align} \label{gen-exp-self}
			z_{k,l}(t,R,a) =&\;\;  \frac{t^{2k \nu}}{(t \lambda)^{2l}}\sum_{p, r \geq 0 } (t \lambda)^{-2p} \sum_{s \geq 0} R^{2k - 2r} (\log(R))^s c_{k,r,s, p}(a)\\ \nonumber
			=&\;\; \frac{t^{2k \nu}}{(t \lambda)^{2l}} \sum_{r, q\geq 0 } (t \lambda)^{-2p}\sum_{s \geq 0} R^{2k - 2r} (\log(R))^s \sum_{\tilde{l} =0 }^{\sigma_2(l)} \sum_{ \tilde{k} \leq \sigma_1(l, \tilde{l})} a^{2\tilde{l}\cdot \nu + \tilde{k} } c^{k,r,s, p}_{\tilde{k},\tilde{l}}
		\end{align}
	} 
	with some fixed $ \sigma_1 \in \Z,\; \sigma_2 \in \Z_+$, the leading order terms are well controlled if $(t,r) \in \mathcal{I}$ since  $ R \lesssim t^{-\nu + \epsilon_1},\; a \lesssim t^{- \f12 + \epsilon_1} $. 
	\begin{Rem}In particular with the exact values $ \sigma_2(l) = l-1, \; \sigma_1(l, \tilde{l}) = 2l -4 - 3\tilde{l}$, we thus obtain in the `worst case' ($ \tilde{l} = l-1$)
		\begin{align*}
			t^{2\nu k}\cdot R^{2k} \leq C_k \cdot t^{2k\epsilon_1},\;\; \tfrac{1}{(t\lambda)^{2l}}\cdot a^{\beta_l} \leq C_l \cdot t^{c_1(\nu, \epsilon_1)l + c_2(l) \epsilon_1},~~~\;\;0 < t \ll1,
		\end{align*}
		where in fact we can take %c_{\nu, \epsilon_1} = (\nu -1)l + (\f12 - \epsilon_1)(l +1) $c_l = (2l -2) \nu$.\\[2pt]
		$c_1(\nu, \epsilon_1) = (\nu -1) + (\f12 - \epsilon_1)$ and $c_2(l)= (2l -2) \nu$.
	\end{Rem}
	%if $ 0 < t \ll 1 $ and where $y = r t^{-\frac12}$. 
	\tinysection[0]{Main idea}We now exchange $(t,R,a)$ for $ (t,y)$ in the above effective expansion \eqref{gen-exp-self}  where $ 1 \ll a \leq R $, which should (up to a fast decaying error) formally display an effective description for $ \mathcal{S} $ where $ 0 < y \ll1$.  Note we have the following variable changes
	\begin{align}\label{trans1}
		&R = (\lambda t) a,\;\; y = r t^{-\f12} = a t^{\f12} = R t^{\nu},
	\end{align}
	which then transforms the above asymptotic into
	\boxalign[13cm]{
		\begin{align}\label{trans2}
			&t^{2\nu k}\cdot R^{2k - 2r} = t^{2\nu r}y^{2k - 2r},\\
			& \tfrac{1}{(t\lambda)^{2l}}\cdot a^{\beta_l} = y^{\beta_l} t^{l(\nu - \f12) + \nu + \f12},\;\; \beta_l = (2l-2)\nu -l-1.
		\end{align}
	}
	Thus note the previous \emph{poles} $R^{2k}$ at $ R = \infty$ become \emph{summable} over $ k \in \Z_+$ in the $(t,y) $ coordinate frames and indicate a singular $y$-boundary with \emph{ poles} at  $ y = 0$ given by the now fixed $ r, l \in \Z_+$ exponents. This is of course a common matching argument, by which we mean the coefficients of \eqref{gen-exp-self} provide `asymptotic boundary values' for a uniquely determined   approximation procedure  in the self-similar $(t,y)$ limit system as $ 0 < y \ll1 $ matching \footnote{By \emph{matching} here we mean extracting the boundary data and the generic form of the expansion. Rigorously, the two approximations only coincide in $\mathcal{I}\cap \mathcal{S}$ up to a 'fast decaying error' provided by the \emph{tails} of the absolute expansions over $R$ and $y$ respectively. This is clarified below in a `consistency' Lemma \ref{lem:consistency-inner} and Lemma \ref{lem:consistency-self-sim-rem} } with the  $ R \gg1,~ a \gg1 $ limit of the previous $(t,R,a)$ asymptotic in the overlap $\mathcal{I}\cap \mathcal{S}$.
	In particular if we intersect  $\mathcal{S} \cap\mathcal{I}$, then there holds $ y \sim t^{\epsilon_1}$, whence  $ R  \sim t^{\epsilon_1 - \nu},~ a \sim t^{\epsilon_1 - \f12}$.\\[3pt]
	Describing the solution in the region  $\mathcal{S}$, we now pass from the representation 
	\[
	u = e^{i\alpha(t)}\lambda(t)(W + z),\,\,\,\Box n = \triangle(\lambda^2|W+z|^2)
	\]
	to the representation 
	\[
	u = e^{i\alpha(t)}t^{-\frac12}w(t, y),\,\,\,\Box n = \triangle (t^{-1}|w|^2). 
	\]
	This leads to the following system in $(t,y)$ variables
	
	\begin{align}
		\begin{cases} \label{schrod-wave-y}
			\,\, i t \partial_t  w + (\mathcal{L}_S - \alpha_0)w =  t \cdot n \cdot w, &~~\\[5pt] 
			\,\, \Box_S n = t^{-1}(\partial_y^2 + \frac{3}{y}\partial_y) \big(t^{-1}|w|^2\big), &~~
		\end{cases}
	\end{align}
	where we define the operators
	\begin{align}
		&\mathcal{L}_S = \partial_y^2 + \frac{3}{y}\partial_y - \frac{i}{2}\Lambda,\;\;\;\;\Lambda = \big(1 + y \partial_y\big),\\[4pt]
		&\Box_S = - \big(\partial_t - \f12t^{-1}y \partial_y\big)^2 + t^{-1}\big(\partial_y^2 + \frac{3}{y}\partial_y\big).
	\end{align}
	\tinysection[0]{Matched asymptotic heuristic}~ \; We calculate the expansion of the iterations in the inner region, i.e. we consider for $ N_1, \tilde{N}_1 \in \Z_+$
	\begin{align}
		&u^{\text{app}}_{N_1} = \lambda \big(W + z^*_{N_1}\big) = \lambda(W + \sum_{j=1}^{N_1}z_j),\\
		&n^{\text{app}}_{N_1, \tilde{N}_1} = \square_{\tilde{N}_1}^{-1}(| u^{\text{app}}_{N_1}|^2) =  | u^{\text{app}}_{N_1}|^2 + \sum_{j \leq \tilde{N}_1} n_j = : \lambda^2 W^2 + n^*_{N_1, \tilde{N}_1},
	\end{align}
	where $\square_{ \tilde{N}_1}^{-1}$ is the  linear wave parametrix in \ref{sec: the wave para gd} iterated to length $\tilde{N}_1$. Here we note
	\[
	\lambda(t) W(\lambda(t) r) = t^{- \f12 - \nu}\big(c_2 t^{2\nu}y^{-2} + c_4 t^{4\nu}y^{-4} + c_6 t^{6\nu}y^{-6} + \dots\big),
	\]
	which we denote by $ z_0$. Further, for the $z_j \in X_j, \;n_{j} \in \lambda^2 (X_j)'$  iterates we write according to approximation in the previous inner region 
	\begin{align*}
		& W(R) + \sum_{j = 0}^{N_1}z_j(t,R,a) = \sum_{j = 0}^{N_1} \sum_{ k + l \in I_j} z_{j,k,l}(t,R,a),
	\end{align*}
	and we either have 
	\begin{align} \label{first-form}
		&z_{j,k,l} \in \frac{t^{2\nu k}}{(t \lambda)^{2l}}S^m(R^{2k-2}\log^{s_{\ast}(l)}(R)),~~ l = 0,1,~k \in I_j,\;\;\;~~~~~~~~~~~~\text{or}\\ \label{second-form}
		& z_{j,k,l} \in  \frac{t^{2\nu k}}{(t \lambda)^{2l}}S^m(R^{2k}\log^{s_{\ast}(l)}(R), \mathcal{Q}^{\beta_l}),~~ k + l \in I_j,
	\end{align}
	where we can set $ s_{\ast}(l) = 2l +1$ in the former case \eqref{first-form} and $s_{\ast}(l) = 2l -1$ in the latter space \eqref{second-form}. 
	\begin{Rem}(i)\; Note the distinction between the terms in \eqref{first-form} and \eqref{second-form} can be considered technical, as we may embed the former in the latter spaces (with trivial $a-$dependence) if we  accept the (minor) loss of information this embedding comprises. This is meant in the sense that the expansions are of course consistent and in fact the leading order is better than stated (i.e. the relevant coefficients vanish). See also Remark \ref{Remark-beginning-self} below.
		%	(ii)\; Further, as seen in the previous Section ref?, if $ \nu > 0 $ is irrational the spaces $\tilde{Q}^{\beta}$  do not include $\log(a)$ poles at $ a = \infty$. In order to make it more accessible to possible extending the argument to rational $\nu >0$, we include this asymptotic in the matching heuristic. 
	\end{Rem}
	The generic form of the large $ R \geq a \gg1$ expansion for the latter type of functions in \eqref{second-form} we expect from the previous region
	\[
	z_{j,k,l}(t,R,a) = \frac{t^{2\nu k}}{(t \lambda)^{2l}} \sum_{r, p \geq 0} (t \lambda)^{- 2p}\sum_{s = 0}^{2r + s_*} R^{2k - 2r} \log^s(R)\sum_{\tilde{s}=0}^{\tilde{s}_*} \sum_{\substack{\tilde{l} \cdot \nu + \tilde{k} \leq \beta_l\\  \tilde{l} \geq 0 }} a^{\tilde{l}\cdot \nu + \tilde{k}} \log^{\tilde{s}}(a) ~c^{k,l,r,s}_{\tilde{k}, \tilde{l}, \tilde{s}}.
	\]
	in an absolute sense with $ \beta_l = 2 \nu(l-1)-l-1$. In the $X_j$ spaces of the previous Section \ref{sec:inner}, we have a slightly refined expansion and set 
	\[
	S_l : = \{  (\tilde{l}, \tilde{k}) \in \Z^2 \;|\; 0 \leq \tilde{l} \leq l-1 ,\;\; \tilde{k} \leq 2l -4 - 3 \tilde{l} \}.
	\]
	Then we expand in the region $ R \gg1, a \gg1 $ in an absolute sense 
	\begin{align}\label{no1}
		&z_{j,k,l}(t,R,a) = \frac{t^{2\nu k}}{(t \lambda)^{2l}} \sum_{r, p \geq 0}  (t \lambda)^{- 2p} \sum_{s = 0}^{2r + s_*} R^{2k - 2r} \log^s(R)\sum_{\tilde{s}=0}^{\tilde{s}_*} \sum_{ (\tilde{l}, \tilde{k}) \in S_l} a^{2\nu \tilde{l} + \tilde{k}} \log^{\tilde{s}}(a) ~c^{k,l,r,s}_{\tilde{k}, \tilde{l}, \tilde{s}, p}.\\ \label{no2}
		&z_{j,k,0}(t,R,a) = t^{2\nu k} \sum_{r \geq 0} \sum_{s = 0}^{2r + 1} R^{2k - 2r-2} \log^s(R)c^{k,l}_{r,s},~~l = 0.
	\end{align}
	Now for the wave part, we similarly have 
	\begin{align}\label{wave-part-terms}
		& \lambda^2(t)W^2(R) + n^*_{N_1, \tilde{N}_1}(t,R,a) =  \sum_{j =0}^{N(N_1)}\sum_{k,l} \tilde{z}_{j,k,l}(t,R,a) +  \sum_{j = 1}^{\tilde{N}_1} \sum_{ k,l }n_{j,k,l}(t,R,a),
	\end{align}
	where the terms in the left sum consist of the initial quadratic source of the $ \square_{KST}^{-1}$ approximation and $ \tilde{z}_{0,0,0}(R) = W^2(R)$. Further  each $j\in \Z_+$ in the latter sum incorporates one (the $j^{\text{th}}$) step in the $\square^{-1}$ parametrix applied to \emph{all} terms in the former sum. Clearly, up to the factor $\lambda^2(t) = t^{-1-2 \nu}$, the left side of \eqref{wave-part-terms} has the form \eqref{no1} and \eqref{no2} by Lemma \ref{lem:recoverz1} and Lemma \ref{lem:Lemma-n-Anfang-elliptic}. 
	\;\\
	%%%%%%%%%
	%	\textcolor{red}{The following remark should be deleted or in the previous Section}
	%%%%%%%%%%%
	\begin{Rem} \label{Remark-beginning-self}  First we recall $ s_* = s_*(l), \tilde{s}_* = \tilde{s}_* (l) $ and second we require
		$$
		c^{k,l,r,s}_{\tilde{k}, \tilde{l}, \tilde{s}, p} = 0,~~\;\;~c^{k,l}_{r,s} = 0
		$$
		except for finitely many $s \geq 0$ (respectively). Concerning the leading order $ \log^{s_*}(R) R^{2k}$, we should emphasize in order to have embeddings, the spaces in \eqref{first-form} and \eqref{second-form} do not require the asymptotic to be sharp, i.e. it may hold $c^{k,l,0,s}_{\tilde{k}, \tilde{l}, \tilde{s}, p}  = 0$ for some $ s \leq s_*$  and all choices of   $ (\tilde{k}, \tilde{l},\tilde{s})$.\\
		(ii)\;\; The summation over $s \in  \Z_{\geq 0}$  in \eqref{no1} and \eqref{no2}  may be better understood by dividing  the expansion into 'principal parts' and 'lower order terms' via 
		\begin{align*}
			\sum_{r \geq 0} \sum_{s = 0}^{2r + s_*}R^{2k - 2r} (\log(R))^s f^{k,l}_{r,s}(a)  =&  \sum_{s = 0}^{ s_*} \sum_{r \geq 0} R^{2k - 2r} (\log(R))^s  f^{k,l}_{r,s}(a)\\
			&\;\;+  \sum_{s \geq 1} \sum_{2r \geq s} R^{2k - 2r} (\log(R))^{s_* + s}  f^{k,l}_{r,s}(a).
		\end{align*}
		where the latter sum $ \sum_{s \geq 1}$ is finite (with implicit polynomial upper bounds depending on $k \in \Z_+$).
	\end{Rem}
	\;\\
	Considering the first line \eqref{no1}, by \eqref{trans1} we have
	\begin{align}
		R^{2k -2r} = (t^{\f12}\lambda(t))^{2k -2r} y^{2k -2r},\;\;\; a^{\tilde{l}\cdot \nu + \tilde{k}}  = t^{-\f12\big(\tilde{l}\cdot \nu + \tilde{k}\big)}y^{\tilde{l}\cdot \nu + \tilde{k}}.
	\end{align}
	Hence setting
	$$ \omega_{l,n,\tilde{l}, \tilde{k}, p}(t) = t^{2\nu(l + n + p) -l -p - \nu - \f12(2 \nu \cdot \tilde{l} + \tilde{k})}, $$
	we formally write for the following sum 
	\boxalign{
		\begin{align}\nonumber
			\lambda(t)& \sum_{j = 0}^{N_1} \sum_{\substack{k, l\\ l > 0}}   z_{j,k,l}(t,R,a)\\[4pt]  \nonumber
			&= t^{- \f12- \nu} \sum_{j,k,l}  \sum_{r, p \geq 0} \frac{t^{2\nu k}}{(t \lambda)^{2(l + p)}}\sum_{s = 0}^{2r + s_*} R^{2k - 2r}(\log(R))^s\sum_{\tilde{s} = 0}^{\tilde{s}_{\ast}} \sum_{(\tilde{l},\tilde{k}) \in S_l} (\log(a))^{\tilde{s}} a^{2 \nu \tilde{l} + \tilde{k}}~c^{k,l,s}_{\tilde{k}, \tilde{l}, \tilde{s}, r, p}\\[4pt]  \label{match}
			&= t^{- \f12 } \sum_{n, p \geq  0} \sum_l \sum_{ (\tilde{l}, \tilde{k}) \in S_l} \omega_{l,n,\tilde{l}, \tilde{k}, p}(t) \sum_{m = 0}^{s_{\star} + 2n}\sum_{\tilde{m} = 0}^{\tilde{s}_*} \times\\[4pt]  \nonumber
			& \hspace{4cm}\times\big(\log(y) - \nu \log(t) \big)^{ m}(\log(y))^{ \tilde{m}} \sum_{k \geq 0} y^{2k + 2\nu \cdot \tilde{l} + \tilde{k} - 2n } d^{k, l, n,m}_{\tilde{l}, \tilde{k}, \tilde{m}, p}
		\end{align}
	}
	\begin{comment}
		\begin{align*}
			\lambda(t)& \sum_{j = 0}^N \sum_{k, l}   z_{j,k,l}(t,R,a)\\
			&= t^{- \f12- \nu} \sum_{j,k,l} \frac{t^{2\nu k}}{(t \lambda)^{2l}} \sum_{r \geq 0} \sum_{s = 0}^{2r + s_*} R^{2k - 2r} \log^s(R)\sum_{\tilde{s}=0}^{\tilde{s}_*} \sum_{\substack{\tilde{l} \cdot \nu + \tilde{k} \leq \beta_l\\  \tilde{l} \geq 0 }} a^{\tilde{l}\cdot \nu + \tilde{k}} \log^{\tilde{s}}(a) ~c^{k,l,r,s}_{\tilde{k}, \tilde{l}, \tilde{s}}.\\
			&= t^{- \f12 } \sum_{n \geq  0} \sum_{\tilde{k} \in \Z } \sum_{\substack{l,\\add{text}tilde{l}\\ \tilde{l} \cdot \nu + \tilde{k} \leq \beta_l}} \omega_{l,n,\tilde{l}, \tilde{k}}(t) \sum_{m = 0}^{s_*}\sum_{\tilde{m} = 0}^{\tilde{s}_*} \big(\log(y) - \nu \log(t) \big)^{ m + \tilde{m}}\log^{\tilde{s}_* - \tilde{m}}(y) \sum_{k \geq 0} y^{2k + \nu \cdot \tilde{l} + \tilde{k} -n } d^{k, l, n,m}_{\tilde{l}, \tilde{k}, \tilde{m}}\\
			&~+~  t^{- \f12 } \sum_{n \geq  0} \sum_{\tilde{k} \in \Z } \sum_{\substack{l,\tilde{l}\\ \tilde{l} \cdot \nu + \tilde{k} \leq \beta_l}} \omega_{l,n,\tilde{l}, \tilde{k}}(t)  \sum_{m \leq n} \sum_{\tilde{m} = 0}^{\tilde{s}_*} \big(\log(y) - \nu \log(t) \big)^{s_* + m + \tilde{m}} \log^{\tilde{s}_* - \tilde{m}}(y) \sum_{k \geq 0} y^{2k + \nu \cdot \tilde{l} + \tilde{k} - n} \tilde{d}^{k, l, n,m}_{\tilde{l}, \tilde{k}, \tilde{m}},
		\end{align*}
	\end{comment}
	where $ s_{\star} = s_* + \tilde{s}_*$. In particular there holds 
	\[
	\omega_{l,n,\tilde{l}, \tilde{k}, p}(t) \leq t^{2\nu n + \f12 + l (\nu - \f12)},~~ 0 < t \leq 1,
	\]
	since  we have $  \tilde{l} \cdot \nu + \tilde{k} \leq \beta_l = 2\nu(l-1) -l-1 $. Therefore we have  $ \omega_{l,n,\tilde{l}, \tilde{k}, p}(t) = O(t^{2\nu n + \f12 + l (\nu - \f12)})$ uniformly over $ (\tilde{l}, \tilde{k})$ as $ t \to 0^+$ .  The above expansion in case $l = 0$ is now similar, in fact we write (this compares to the case of the cubic NLS in \cite{schmid})\\[6pt]
	
	\boxalign{
		\begin{align}\nonumber
			\lambda(t)& \sum_{j = 0}^{N_1} \sum_{k}   z_{j,k,0}(t,R)\\\label{match-ellnull}
			&= t^{- \f12- \nu} \sum_{j,k} t^{2\nu k} \sum_{r \geq 0} \sum_{s = 0}^{2r + 1} R^{2k -2 - 2r}\log^s(R)c^{k}_{r,s}\\  \nonumber
			&= t^{- \f12 } \sum_{n \geq  0}  t^{2\nu(n+\f12)} \sum_{m = 0}^{1 + 2n}\big(\log(y) - \nu \log(t) \big)^{ m}\sum_{k \geq 0} y^{2k -2 - 2n } c_{k, n, m}
		\end{align}
	}
	
	\begin{Rem}(a)\; Note the terms of the first kind in \eqref{first-form} with $l =1$ have a similar form as \eqref{match-ellnull} with factor $t^{2\nu (n + \f32) -1}$ instead of $t^{2\nu(n + \f12)}$. However, as in the previous Section \ref{sec:inner}, and in particular Section \ref{subsec:inductive-ite-z}, we regard them as a case of \eqref{no2} in the matching form \eqref{match}  for $ l =2$. In fact $ d^{k,l,m,n}_{\tilde{l},\tilde{k}, \tilde{m}, p} = 0$ unless $\tilde{l}  = \tilde{m} = p = 0$ in \eqref{match} and then
		$$c_{k,m,n } = d^{k,2,m,n}_{0,-2, 0, 0},\;\; \omega_{2,n,0,-2, 0}(t) = t^{2\nu (n + \f32) -1}.$$
		Infact, we may include also \eqref{match-ellnull} into the above representation \eqref{match} via the case of $l =1$ and $s_{\ast}(1) = 1$. We have  $ d^{k,l,m,n}_{\tilde{l},\tilde{k}, \tilde{m}} = 0$ unless
		$\tilde{l} = \tilde{m} = 0$ and $  \tilde{k} = -2$. Then
		\[
		c_{k,m,n } = d^{k,1,m,n}_{0,-2, 0, 0},\;\;\; t^{\nu(2n +1)} = t^{\nu(2n +2) - 1-\nu - \f12 (0 \cdot \nu -2)} = \omega_{1,n,0,-2}(t),\;\; n \geq 0.\; 
		\] 
		(b)\; The sum in \eqref{match} involves the finite summation over $ l = 1, \dots,\mathcal{N}$ with a fixed but variable upper bound for the length of all $\square^{-1}_{KST}$ approximations in all $j = 0, \dots, N_1 $ steps of the  previous iteration in the  region $\mathcal{I}$. %\\[5pt]
		%(c)\; From the previous inner region, c.f. Lemma ref?,  $0 \leq \tilde{l}$ in \eqref{match} can be restricted to \emph{even} integer and will not be arbitrary large. In fact we have $ \tilde{k} \leq -l-1,\; 0 \leq \tilde{l} \leq 2l -2$ instead of $ \tilde{l}\cdot \nu + \tilde{k} \leq \beta_l,\; 0 \leq\tilde{l}$ for a fixed finite range of $ 0 \leq l \leq N(\mathcal{I})$ in \eqref{match}.
	\end{Rem}
	\;\\
	\emph{The new `free' parameters} are thus considered 
	$$n, p \in \Z_{\geq 0},\; \tilde{k} \in \Z,\;\;\text{with}\;\;\; k \leq 2l -4 - 3\tilde{l},\;\; 0 \leq \tilde{l} \leq l-1,$$ 
	(given $l > 0$).
	We now turn to the the wave part $n$, for which we distinguish between the quadratic interactions we have in \eqref{wave-part-terms} instead of considering the generic form in  the $(X_j)' $ spaces. Therefore let us first consider terms of the form 
	\begin{align} 
		n(t,R,a) = 	\square^{-1}_{KST}(z \cdot \tilde{z})(t,R,a)
	\end{align}
	where $z$ and $ \tilde{z}$ are as in \eqref{no1} where we let $l \geq 1$. Then we infer the following expansion with Lemma \ref{lem: induct} \\[3pt]
	\boxalign{
		\begin{align}\nonumber
			&\lambda^2(t)\big(%W^2(R) +
			n^*_{N_1 \tilde{N}_1}(t,R,a)\big)\\[4pt] \nonumber
			&= t^{- 1- 2\nu} \sum_{j,k,l} \sum_{r, p \geq 0}\frac{t^{2\nu k}}{(t \lambda)^{2(l + p)}}  \sum_{s = 0}^{2r + s_*} R^{2k - 2r} (\log(R))^s\sum_{\tilde{s}=0}^{\tilde{s}_*} \sum_{ (\tilde{l}, \tilde{k}) \in S_l} a^{2 \nu \tilde{l} + \tilde{k}} (\log(a))^{\tilde{s} } ~c^{k,l,r,s}_{\tilde{k}, \tilde{l}, \tilde{s}, p}\\[4pt] \label{match-wave}
			&= \sum_{n, p \geq  0} \sum_l \sum_{(\tilde{l}, \tilde{k}) \in S_l} \tilde{\omega}_{l,n,\tilde{l}, \tilde{k}, p}(t) \sum_{m = 0}^{s_{\star} + 2n}\sum_{\tilde{m} = 0}^{\tilde{s}_*} \times\\[4pt] \nonumber
			&\hspace{4cm}\times\big(\log(y) - \nu \log(t) \big)^{ m}(\log(y))^{\tilde{m}} \sum_{k \geq 0} y^{2k + 2\nu \cdot \tilde{l} + \tilde{k} - 2n } d^{k, l, n,m}_{\tilde{l}, \tilde{k}, \tilde{m}, p},
		\end{align}
	}
	\;\\
	where we set 
	$$ \tilde{\omega}_{l,n,\tilde{l}, \tilde{k}, p}(t) = \omega_{l,n,\tilde{l}, \tilde{k}, p} \cdot t^{-\nu -1} =  t^{\nu(2n + 2l + 2p) - 2\nu -l -p -1 - \frac{2 \nu \tilde{l} + \tilde{k}}{2}}$$
	and thus similar as above 
	$
	\tilde{\omega}_{l,n,\tilde{l}, \tilde{k}, p}(t) \leq t^{\nu(2n-1) + l(\nu - \f12) - \f12}$ for $0 < t \leq 1$, which again gives a lower bound on the decay rate of this factor as $ t \to 0^+$.
	\begin{Rem}
		%(i)\; First we note the value of $ s _{\star} = s_{\star} (l) \in \Z_+$ is kept implicit in this notation, however it needs to be adapted and is of course not the same as in \eqref{match}.\\[5pt]
		(i) Taking products in \eqref{match} of the form $ t ^{-1}|w|^2$ (where $w$ is the ansatz in the latter line) shows this summation is (formally) consistent with the source in the second line of \eqref{schrod-wave-y}. We give more details on the expansion of the source terms in \eqref{schrod-wave-y} below.\\[5pt]
		(ii)\; As before, we in fact sum $ \tilde{k} \leq 2l - 4 - 3\tilde{l},\;  0 \leq \tilde{l} \leq l-1$.  Hence if we exchange variables %$ \underline{l} = l + 1,\; 
		$\underline{\tilde{k}} = \tilde{k}  +2,\; \underline{\tilde{l}} = \tilde{l} + 1$, then we have $  1 \leq \underline{\tilde{l}} \leq  l ,\; \underline{\tilde{k}} \leq 2l +1  -3\underline{\tilde{l}}  $ and there
		holds
		\[
		\tilde{\omega}_{l,n,\tilde{l}, \tilde{k}, p}(t) = t^{\nu(2n + 2 l + 2p)  - \nu - l -p - \frac{\underline{\tilde{l}} \cdot \nu + \underline{\tilde{k}}}{2}} = \omega_{l,n,\underline{\tilde{l}}, \underline{\tilde{k}}, p}(t). %  \cdot t^{-\nu}.
		\]
		Therefore the time factor in the ansatz for $n$ is that of \eqref{match}  with an extended range for $ \tilde{l}, \tilde{k}$ and in fact $ l \geq 2$.
	\end{Rem}
	\;\\
	In particular  \eqref{match} and \eqref{match-wave} now motivate the following ansatz for solving \eqref{schrod-wave-y}. Let  $ N_2^{(j)},\tilde{N}_2^{(j)} \in \Z_+,\; j = 1,2$ such that $\tilde{N}_2^{(j)} \gg \mathcal{N},\; j = 1,2,3$. Then we set
	\begin{align}\label{Schrod-ansatz1}
		&w_{N_2^{(1)}, N_2^{(2)}, N_2^{(3)}}(t,y) = \sum_{n = 0}^{N_2^{(1)}} \sum_{p= 0}^{N_2^{(3)}} \sum_{l, \tilde{l}} \sum_{\tilde{k} = 3 \tilde{l} + 4 - 2l}^{N_2^{(2)}} t^{2\nu(l + n + p) -l -p - \nu + \frac{\tilde{k} -2 \nu \cdot \tilde{l}}{2} } \times\\ \nonumber
		& \hspace{7cm}\times\sum_{m =0}^{s_{\star} + 2n} \big(\log(y) - \nu \log(t)\big)^m g_{n,m,l,\tilde{k}, \tilde{l}, p}(y),\\ \label{wave-ansatz1}
		&n_{\tilde{N}_2^{(1)}, \tilde{N}_2^{(2)}, \tilde{N}_2^{(3)}}(t,y) = \sum_{n = 0}^{\tilde{N}_2^{(1)}} \sum_{p = 0}^{\tilde{N}_2^{(3)}}\sum_{l, \tilde{l}} \sum_{\tilde{k} = 3 \tilde{l} - 1 -2l }^{\tilde{N}_2^{(2)}} t^{2\nu(l + n + p) -l -p - \nu  + \frac{\tilde{k} - 2\nu \cdot \tilde{l}}{2} } \times\\ \nonumber
		& \hspace{7cm}\times \sum_{m =0}^{s_{\star} + 2n} \big(\log(y) - \nu \log(t)\big)^m \tilde{g}_{n,m,l,\tilde{k}, \tilde{l}, p}(y),
	\end{align}
	where $ g_{n,m,l,\tilde{k}, \tilde{l}}(y),\; \tilde{g}_{n,m,l,\tilde{k}, \tilde{l}}(y)$ stand for the Schr\"odinger- and wave part respectively. The second sum $\sum_{l, \tilde{l}}$ is finite with fixed upper bounds for $l, \tilde{l} \in \Z_+$, i.e. $ 0 \leq l \leq \mathcal{N}$ and in the latter line \eqref{wave-ansatz1} we sum only $ l \geq 2$.  Further in the first line we observed $  0 \leq \tilde{l} \leq l-1 $ and in the second $  1 \leq \tilde{l} \leq l $ so far.\\[10pt]
	In fact we may also use just two parameters $ N_2, \tilde{N}_2 \in \Z_+$ and implicitly sum $ 0 \leq n \leq N_2^{(1)}$, as well as $3 \tilde{l} +4 -2l \leq \tilde{k} \leq N_2^{(2)} $ and $ 0 \leq p \leq N_2^{(3)}$ for any large integer 
	$$ N_2^{(1)} \geq N_2,\; N_2^{(2)} \geq N_2,\; N_2^{(3)} \geq N_2 \gg1,$$
	and similar in the second line with $\tilde{N}_2 \gg1$ (see below). %Moreover we note the ansatz for the case $ l = 0$  in \eqref{Schrod-ansatz1} is not explicitly stated and has to be adjusted to the form  of \eqref{match-ellnull}.  
	In addition we require absolutely convergent expansions of the form
	\begin{align}\label{first-ansatz-expansion1}
		&g_{n,m,l,\tilde{k}, \tilde{l}}(y) = \sum_{\tilde{m} = 0}^{\tilde{s}_*}\big(\log(y)\big)^{ \tilde{m}} \sum_{k \geq 0} y^{2k +2 \nu \cdot \tilde{l} - \tilde{k} - 2n } c^{k, l, n,m}_{\tilde{l}, \tilde{k}, \tilde{m},p},\;\; 0 < y \ll1,\\ \label{first-ansatz-expansion2}
		&\tilde{g}^{\text{w}}_{n,m, l,\tilde{k}, \tilde{l}}(y) = \sum_{\tilde{m} = 0}^{\tilde{s}_*}\big(\log(y)\big)^{ \tilde{m}} \sum_{k \geq 0} y^{2k + 2\nu \cdot \tilde{l} - \tilde{k} - 2n } \tilde{c}^{k, l, n,m}_{\tilde{l}, \tilde{k}, \tilde{m},p},\;\; 0 < y \ll1.
	\end{align}
	In the latter we should actually write $ y^{\nu \tilde{l} - 2 \nu -2 } $ due to the variable change instead of $y^{\nu \tilde{l}}$. \emph{We will make this precise below and it will further be convenient in the following to apply a change of variables for $ \tilde{k}, \tilde{l}$ one more time}, in order to unify the representation of $t-$powers in the form 
	$ t^{ \nu a + \f12 \tilde{a}},\;t^{ \nu a +  \tilde{a}} $
	where $ a \in\Z_+, \tilde{a} \in \Z$ in the $(t,y)$ separation (up to the logarithm).\\[4pt]
	Before we turn to this issue, let us first consider the remaining possible terms in the wave part, that is 
	\begin{align}
		n(t,R,a) =  z \cdot \tilde{z}(t,R,a)
	\end{align}
	where either $z$ or $ \tilde{z}$ are as in \eqref{no2}. In particular we consider the products of \eqref{no1} and \eqref{no2}. For such terms we obtain the expansion\\[3pt]
	\boxalign[11cm]{
		\begin{align}\nonumber
			&t^{- 1- 2\nu} \sum_{j,k,l}  \sum_{r, p \geq 0} \frac{t^{2\nu k}}{(t \lambda)^{2(l +p)}}\sum_{s = 0}^{2r + s_*} R^{2k - 2r -2} (\log(R))^s\sum_{\tilde{s}=0}^{\tilde{s}_*} \sum_{(\tilde{l}, \tilde{k}) \in S_l} a^{2 \nu \tilde{l}+ \tilde{k}} (\log(a))^{\tilde{s}}~c^{k,l,r,s}_{\tilde{k}, \tilde{l}, \tilde{s},p}\\[4pt] \label{match-wave2}
			&= \sum_{n, p \geq  0} \sum_{(\tilde{l}, \tilde{k}) \in S_l} \tilde{\omega}_{l,n,\tilde{l}, \tilde{k}, p}(t) \sum_{m = 0}^{s_{\star} + 2n}\sum_{\tilde{m} = 0}^{\tilde{s}_*} \times\\[4pt]\nonumber
			&\hspace{4cm}\times\big(\log(y) - \nu \log(t) \big)^{ m}(\log(y))^{\tilde{m}} \sum_{k \geq 0} y^{2k -2 + 2 \nu \cdot \tilde{l} + \tilde{k} - 2n } d^{k, l, n,m}_{\tilde{l}, \tilde{k}, \tilde{m}, p},
		\end{align}
	}
	where we set
	$$ \tilde{\omega}_{l,n,\tilde{l}, \tilde{k}}(t) = t^{2\nu(n + 1 + l + p ) - 2\nu  - p -l -1  - \frac{2 \nu \tilde{l}  + \tilde{k}}{2}}$$
	and $ l \geq 1$. If we change variables $ \underline{l} = l +1$  and $\underline{\tilde{l}} = \tilde{l} + 1$, then we infer %$ \underline{\tilde{k}} \leq -\underline{l} -1,\;\;
	$2 \leq  \tilde{l} \leq  \underline{l} - 1$ where $ \underline{l} \geq 2$ and we make an ansatz for $n(t,y)$ consistent with   \eqref{wave-ansatz1}.\\[10pt]
	Finally, considering the products of \eqref{no2} we have the expansion
	\boxalign{
		\begin{align}\nonumber
			&t^{- 1- 2\nu} \sum_{j,k} \sum_{r \geq 0} t^{2\nu k} \sum_{s = 0}^{2r + s_*} R^{2k - 2r -4} \log^s(R)~c_{k,r,s}\\\label{match-wave3}
			&= \sum_{n \geq  0}  t^{2\nu(n + 1)}\sum_{m = 0}^{s_{\star} + 2n} \big(\log(y) - \nu \log(t) \big)^{ m} \sum_{k \geq 0} y^{2k -4  - 2n } c_{k, n,m},
		\end{align}
	} This term can therefore seen as a special case of \eqref{wave-ansatz1} with $\tilde{k} = 4,\; \tilde{l} =1$ and $l = 2$. %we therefore change variables $ \underline{l} = l +2,\; \underline{\tilde{k}} = \tilde{k} -2$ and  $\underline{\tilde{l}} = \tilde{l} + 2$, then we infer $ \underline{\tilde{k}} \leq -\underline{l} -1,\;\; 2 \leq \underline{\tilde{l} } \leq 2 \underline{l} - 4 $ where $ \underline{l} \geq 2$ and we make an ansatz for $n(t,y)$ consistent with   \eqref{wave-ansatz1}.
	Finally we note taking the wave parametrix 
	\begin{align}
		n(t,R,a) =  \square^{-1}_{KST}(z \cdot \tilde{z})(t,R,a)
	\end{align}
	of such products always requires to calculate one (or two) elliptic steps as in Remark \ref{rem:elliptic-steps-first} and produces terms consistent with the type which we considered at first for the wave part.\\[4pt]
	\tinysection[10]{Form of the self-similar approximation}\;\\
	Now we like to find a more practical representation of \eqref{Schrod-ansatz1} and \eqref{wave-ansatz1}.
	Therefore we finally change 
	$ \underline{\tilde{l}} =  \tilde{l} +1,\; \underline{\tilde{k}} = \tilde{k} - 2l -2p$ for both the Schr\"odinger ansatz \eqref{Schrod-ansatz1} and 
	%$ \underline{\tilde{l}} =  \tilde{l} +4,\; \underline{\tilde{k}} = \tilde{k} - 2l$ 
	for the wave ansatz \eqref{wave-ansatz1}. In particular we then write these sums into
	\boxalign[13cm]{\begin{align}\label{Schrod-ansatz2}
			&w_{N_2^{(1)}, N_2^{(2)}, N_2^{(3)}}(t,y) = \sum_{n = 0}^{N_2^{(1)}} \sum_{p = 0}^{N_2^{(3)}} \sum_{l, \tilde{l}} \sum_{\tilde{k} =3 \tilde{l} -4l -2p +1}^{N_2^{(2)}} t^{\nu(2(l + n + p) - \tilde{l}) +  \frac{\tilde{k}}{2}}  \sum_{m =0}^{s_{\star} + 2n} \times\\[4pt] \nonumber
			& \hspace{9cm} \times\big(\log(y) - \nu \log(t)\big)^m g_{n,m,l,\tilde{k}, \tilde{l}, p}(y),\\[4pt] \label{wave-ansatz2}
			&n_{\tilde{N}_2^{(1)}, \tilde{N}_2^{(2)}, \tilde{N}_2^{(3)}}(t,y) = \sum_{n = 0}^{\tilde{N}_2^{(1)}} \sum_{p = 0}^{\tilde{N}_2^{(3)}}\sum_{l, \tilde{l}} \sum_{\tilde{k} = 3 \tilde{l} - 4l -2p -4}^{\tilde{N}_2^{(2)}}  t^{\nu(2(l + n + p) - \tilde{l}) +  \frac{\tilde{k}}{2}}  \sum_{m =0}^{s_{\star} + 2n} \times\\[4pt] \nonumber
			&\hspace{9cm}\times \big(\log(y) - \nu \log(t)\big)^m \tilde{g}_{n,m,l,\tilde{k}, \tilde{l}, p}(y),
		\end{align}
	}
	where $0 \leq l \leq \mathcal{N}$ is again an implicit but fixed sum (given in the previous Section \ref{sec:inner}) To be precise we have 
	\begin{align*} 
		&1 \leq \tilde{l} \leq l,\;\;\; 1 \leq l \leq \mathcal{N} \;\;\;\text{in\;the\;first\;line}\; \eqref{Schrod-ansatz2}.\\
		&2 \leq \tilde{l} \leq l +1,\;\;\; 2 \leq l \leq \mathcal{N}\;\;\; \text{in\;the\;second\;line}\; \eqref{wave-ansatz2}.
	\end{align*} %Note that we now sum over $ l \geq 2$ in the second line, which is consistent with the above heuristic derivation of \eqref{wave-ansatz1}.   
	Further, we require the absolute expansions where we now set $\tilde{s}_* = m_{\star} -m$
	\boxalign[14cm]{
		\begin{align}\label{exp1}
			&g_{n,m,l,\tilde{k}, \tilde{l}}(y) = \sum_{\tilde{m} = 0}^{m_{\star} -m}\big(\log(y)\big)^{ \tilde{m}} \sum_{k \geq 0} y^{2(k - l -n - p) + \nu \cdot (\tilde{l} -2) - \tilde{k} } \cdot c^{k, l, n,m}_{\tilde{l}, \tilde{k}, \tilde{m}},\;\; 0 < y \ll1,\\[4pt] \label{exp2}
			&\tilde{g}_{n,m, l,\tilde{k}, \tilde{l}}(y) = \sum_{\tilde{m}= 0}^{m_{\star} -m}\big(\log(y)\big)^{\tilde{m}} \sum_{k \geq 0} y^{2(k - l -1 -n -p) + \nu \cdot (\tilde{l} -4) - \tilde{k} } \cdot \tilde{c}^{k, l, n,m}_{\tilde{l}, \tilde{k}, \tilde{m}},\;\; 0 < y \ll1.
		\end{align}
	}
	We may slightly simplify the notation for \eqref{Schrod-ansatz2} and \eqref{wave-ansatz2}, i.e.  we can simply write
	\boxalign[15cm]{
		\begin{align}\label{Schrod-ansatz3}
			&w_{N_2}(t,y) = \sum_{n, l, \tilde{l}, p \geq 0} \sum_{\tilde{k} \geq  3\tilde{l} - 4l -2p +1} t^{\nu(2(l + n +p) - \frac{\tilde{l}}{2}) +  \frac{\tilde{k}}{2}}  \sum_{m =0}^{s_{\star} + 2n}\big(\log(y) - \nu \log(t)\big)^m g_{n,m,l,\tilde{k}, \tilde{l}, p}(y),\\[4pt] \label{wave-ansatz3}
			&n_{\tilde{N}_2}(t,y) =  \sum_{n, l, \tilde{l}, p \geq 0} \sum_{\tilde{k} \geq  3\tilde{l} - 4l -2p -4} t^{\nu(2(l + n +p) - \frac{\tilde{l}}{2}) +  \frac{\tilde{k}}{2}}  \sum_{m =0}^{s_{\star} + 2n} \big(\log(y) - \nu \log(t)\big)^m \tilde{g}_{n,m,l,\tilde{k}, \tilde{l}, p}(y),
		\end{align}
	}
	where we sum in the first, respectively the second line,
	\begin{align}\label{constraint-1}
		& n =  0, 1, \dots, N_2^{(1)},\; l =  1, \dots \mathcal{N}, \;\tilde{l} = 1, 2,3, \dots, l,\;\\ \label{constraint-2}
		& \tilde{k} = 3 \tilde{l} - 4l - 2p +1,\dots, N_2^{(2)},\;\; p = 0,1,2,\dots , N_2^{(3)},\\ \label{constraint-3}
		& n =  0, 1, \dots, \tilde{N}_2^{(1)}, \;l = 2,3, \dots \tilde{N}(\mathcal{I}), \;\tilde{l} = 2, 3, \dots, l +1,\\ \label{constraint-4}
		&\tilde{k} = 3 \tilde{l} - 4l - 2p -4 ,  \dots, \tilde{N}_2^{(2)},\;\;p = 0,1,2,\dots , \tilde{N}_2^{(3)},
	\end{align}
	and where we select any integers $ N_2^{(1)}, N_2^{(2)}, N_2^{(3)}  \geq N_1 \gg1,\;  \tilde{N}_2^{(1)}, \tilde{N}_2^{(2)} , \tilde{N}_2^{(2)} \geq \tilde{N}_2 \gg1 $.
	\;\\[5pt]
	The objective is now to construct a solution of \eqref{schrod-wave-y} under the  asymptotic condition of \eqref{exp1} and \eqref{exp2} with coefficients taken from the inner expansions in \eqref{match} and \eqref{match-wave} (as well as the subsequent special cases).  As explained above, we present  the procedure to do this  consistently (up to a fast decaying error for small $y \ll1 $) with the constructed approximation in Section \ref{sec:inner}.  In fact this will follow from a  `consistency' Lemma \ref{lem:consistency-inner} by interchanging finite sums at the expanse of separating tails of absolute sums (decaying as $ t \to 0^+$) in the region $ y \ll1,\; R \gg1$. \\[3pt]
	However, we first need to explain how to select the unique extension solution purely in terms of working with the $(t, y)$ coordinates, and how to then extend this solution to the %bigger 
	region $r \lesssim t^{\frac12 - \epsilon_2}$, or equivalently $y\lesssim t^{-\epsilon_2}$.\\[3pt] 
	The main novelty here compared to Perelman's prior framework, see \cite{OP}, \cite[Section 2.2]{Perelman}, comes from the second equation in \eqref{schrod-wave-y}, where we have to characterize the  extend of the contributions of \emph{approximate free wave solutions} along the iteration introduced by the wave interaction.\\[4pt]
	\tinysection[5]{Set-up for the iteration} To begin with, we have to consider the terms in the above expansions for $w, n$. Specifically, in the form above, we observe that among all the powers of $t > 0$ we can distinguish between
	\[
	\sum_{\alpha_1, \alpha_2}t^{\alpha_1\nu + \alpha_2} g_{\alpha_1, \alpha_2}(t,y),\; \sum_{\alpha_1, \alpha_2} t^{\alpha_1 \nu + \alpha_2+\tfrac12} g_{\alpha_1, \alpha_2}(t,y)\;\;
	\]
	where $\alpha_1 \in \Z_+, \alpha_2 \in \Z$  and note the latter is potentially negative as seen above. Further, the dependence of $g_{\alpha_1, \alpha_2}(t,y)$ on $ t > 0$ is logarithmic and clearly (comparing this to \eqref{Schrod-ansatz2}, \eqref{wave-ansatz2}) there holds in any case
	\[
	\alpha_1\nu + \alpha_2 \geq \nu,\;\; \alpha_1\nu + \alpha_2  + \tfrac12 \geq \nu.
	\]
	Correspondingly we first want to rewrite generically
	\begin{equation}\label{eq:wgeneric-1}\begin{split}
			w(t, y) &= \sum_{\alpha_1,\alpha_2}t^{\nu\alpha_1+\alpha_2} \sum_{m}(\log(y) - \nu\log(t))^m g^{(1)}_{m, \alpha_1,\alpha_2}(y)\\
			& \;\;+ \sum_{\alpha_1,\alpha_2}t^{\nu\alpha_1+\alpha_2 + \f12} \sum_{m}(\log(y) - \nu\log(t))^m g^{(2)}_{m, \alpha_1,\alpha_2}(y)
		\end{split}
	\end{equation}
	whence, considering  \eqref{Schrod-ansatz2}, we obtain
	\begin{equation}\label{eq:wgeneric-1b}\begin{split}
			w(t, y)	&= \sum_{\substack{n, l , \tilde{l}, \tilde{k}, p\\ \tilde{k}\; \text{even} }}t^{\nu(2(n + l + p) -\tilde{l})+ \frac{\tilde{k}}{2}} \sum_{m = 0}^{s_{\star} + 2n}(\log(y) - \nu\log(t))^m g^{(1)}_{m, n, l , \tilde{l}, \tilde{k}, p}(y)\\
			&\;\;+ \sum_{\substack{n, l , \tilde{l}, \tilde{k}, p\\ \tilde{k}\; \text{even}}}t^{\nu(2(n + l + p) -\tilde{l})+ \frac{\tilde{k}}{2} + \f12} \sum_{m = 0}^{s_{\star} + 2n}(\log(y) - \nu\log(t))^m g^{(2)}_{m, n, l , \tilde{l}, \tilde{k}, p}(y),
	\end{split}\end{equation}
	where the sum respects the above constraints in \eqref{Schrod-ansatz2}.
	We further assume throughout that $\nu$ is {\it{irrational}}, which implies that two terms in a different sum or in the same sum but with different $\alpha_1$ cannot differ by an integer. In particular we determine the coefficients $g^{(j)}$ via an  inductive procedure.\\[3pt] 
	More precisely, this induction %of course 
	has to be implemented jointly with the coefficients in $n$, which we write similarly as 
	\begin{equation}\label{eq:wgeneric-2}\begin{split}
			n(t, y) &= \sum_{\beta_1,\beta_2}t^{\nu \beta_1+\beta_2} \sum_{m}(\log(y) - \nu\log(t))^m  h^{(1)}_{m, \beta_1,\beta_2}(y)\\
			& \;\;+ \sum_{\beta_1,\beta_2}t^{\nu\beta_1+\beta_2 + \f12} \sum_{m}(\log(y) - \nu\log(t))^m  h^{(2)}_{m, \beta_1,\beta_2}(y),
		\end{split}
	\end{equation}
	whence, considering  \eqref{wave-ansatz2}, this reads
	\begin{equation}\label{eq:wgeneric-2b}\begin{split}
			n(t, y)	&= \sum_{\substack{n, l , \tilde{l}, \tilde{k}, p\\ \tilde{k}\; \text{even} }}t^{\nu(2(n + l + p) -\tilde{l})+ \frac{\tilde{k}}{2}} \sum_{m = 0}^{s_{\star} + 2n}(\log(y) - \nu\log(t))^m  h^{(1)}_{m, n, l , \tilde{l}, \tilde{k},p}(y)\\
			&\;\;+ \sum_{\substack{n, l , \tilde{l}, \tilde{k}, p\\ \tilde{k}\; \text{even}}}t^{\nu(2(n + l + p) -\tilde{l})+ \frac{\tilde{k}}{2} + \f12} \sum_{m = 0}^{s_{\star} + 2n} (\log(y) - \nu\log(t))^m  h^{(2)}_{m, n, l , \tilde{l}, \tilde{k},p}(y),
	\end{split}\end{equation}
	where the sum respects the above constraints in  \eqref{wave-ansatz2}.
	%Observe that we may as well include the factors $(\log y)^c$ in the functions $g^{(1)}_{a,c,\alpha_1,\alpha_2}(y)$ and so forth, provided we modify the expansions \eqref{eq:smallyexpansions} to also include logarithms. 
	%We shall henceforth do so, whence there is no more parameter $c$, except implicitly in the expansion of the $g^{(j)}_{a,c,\alpha_1,\alpha_2}(y)$ and so forth.
	%  Thus we shall set 
	%\[
	%w(t, y) = \sum'_{a,\alpha_1,\alpha_2}t^{\nu\alpha_1+\alpha_2} (\log y - \nu\log t)^a\cdot g^{(1)}_{a,\alpha_1,\alpha_2}(y) + \ldots. 
	%\]
	\;\\[3pt]
	Concerning \emph{the general strategy of the procedure}, c.f. \cite{Perelman}, \cite{schmid},  for determining the coefficient functions of \eqref{eq:wgeneric-1} and \eqref{eq:wgeneric-2}, i.e.
	\[
	g^{(j)}_{\alpha_1, \alpha_2, m}(y),\;\;h^{(j)}_{\beta_1, \beta_2,m}(y),\;\;j = 1,2,
	\]  
	we first note the following calculation, plugging \eqref{eq:wgeneric-1} into the first line of \eqref{schrod-wave-y}
	\boxalign[12cm]{
		\begin{align}\nonumber
			&t^{- (\nu\alpha_1+\alpha_2)}\big( i t \partial_t + \mathcal{L}_S - \alpha_0\big) \bigg( t^{\nu\alpha_1+\alpha_2}(\log(y) - \nu \log(t))^{m}g(y)  \bigg)\\[5pt] \label{clac1}
			& =  (\log(y) - \nu \log(t))^{m} (i (\nu\alpha_1+\alpha_2) + \mathcal{L}_S - \alpha_0 )g(y)\\[5pt]  \nonumber
			&\;\;\;\; + (\log(y) - \nu \log(t))^{m-1} m \frac{2}{y}\partial_yg(y) - (\log(y) - \nu \log(t))^{m-1}  i (\nu + \f12)m g(y)\\[5pt]  \nonumber
			&\;\;\;\; +\frac{2}{y^2} m (\log(y) - \nu \log(t))^{m-1}g(y) + \frac{1}{y^2}m (m -1)(\log(y) - \nu \log(t))^{m-2}g(y),
		\end{align}
	}
	where here $ \alpha_1 \in \Z_+$ and either $ \alpha_2 \in \Z $ or $ \alpha_2 \in \Z + \f12$. Now in each sum in \eqref{eq:wgeneric-1} and for each fixed $\alpha_1$, we select the maximal $m = m(\alpha_1)$ and minimal $\alpha_2 = \alpha_2(\alpha_1)$. Then if we let $\alpha_{1\star} = \min\{\alpha_1\}$ %and keeping in mind \eqref{schrod-wave-y}, the function 
	to be the minimum  of all (relevant) $\alpha_1 \in \Z_+$ , we expect the function  
	$$
	g(y) = g_{m(\alpha_{1\star}), \alpha_{1\star}, \alpha_2(\alpha_{1\star})}(y)
	$$
	to solve the self similar 'free equation'. More precisely, in this minimal case, we can deduce from  the leading order term in  \eqref{clac1} 
	\[
	(i (\nu \alpha_{1\star} + \alpha_2(\alpha_{1\star})) + \mathcal{L}_S - \alpha_0 )g(y) = 0
	\]
	and whence
	\begin{align}
		\big[i(\nu \alpha_{1\star} + \alpha_2(\alpha_{1\star})) - \alpha_0\big]g + (\partial_{y}^2 + \frac{3}{y}\partial_y)g - \frac{i}{2}\Lambda g= 0.
	\end{align}
	Near $ y = 0$, in the region $ 0 < y \ll1$, we solve this equation via  the following fundamental base. \footnote{this is of course the same reduction of the Schr\"odinger  operator considered in \cite[Section 2.2]{schmid} and the Lemma provided here is the same.}
	\begin{lem} \label{Lemma-FS-inner-small}Let $ \mu \in \C$. Then  the equation 
		$ (\mathcal{L}_S + \mu) g = 0$ 
		has fundamental  solutions 
		$$e_1^{(0)}(\mu, y),\; e_2^{(0)}(\mu, y)$$
		for $ y \in (0, \infty)$ such that there holds the following.   The function $ e_1^{(0)}(\mu, y)  = 1 + \mathcal{O}(y^2)$ is analytic in $(\mu, y)$ and has a series expression 
		\begin{align}\label{eq:e1}
			e_1^{(0)}(\mu, y) = 1 + \sum_{k \geq 1} c_k(\mu)y^{2k}%c_2y^2 + c_4y^4 + c_6 y^6 \ldots,\;\; 
		\end{align}
		which converges absolutely if $y \ll1$. % Clearly $e_1^{(0)}$ has a smooth even extension to $\R$.
		Further the function $ e_2^{(0)}(\mu, y)  $ satisfies
		\begin{align}\label{eq:e2}
			e_2^{(0)}(\mu, y) = \tilde{c}_1 y^{-2} +  \tilde{c}_2 \tilde{e}_1(\mu, y) +  \log(y) \cdot e^{(0)}_1(\mu, y),
		\end{align}
		where  $\tilde{e}_1(\mu, y) = \mathcal{O}(1)$ is analytic in $(y,\mu)$ and has an absolute expansion  as $ e^{(0)}_1(y)$ above.
	\end{lem}
	The proof quickly follows from a suitable absolutely convergent power series ansatz and  is provided in \cite[Lemma 2.20]{schmid}. %Note that both of these expressions are of the form \eqref{exp1}. 
	Given  a function $f(y)$ which is regular near $y = y_0 > 0$, the solution of the inhomogeneous equation
	\begin{align}
		(\mathcal{L}_S + \mu)w(y) = f(y)
	\end{align}
	has the form
	\begin{align}\label{Green-int1}
		&w(y) = c_1 e^{(1)}_0(\mu, y) + c_2 e^{(2)}_0(\mu,y) + \int_{y_0}^y w^{-1}(s) G_0(\mu, s, y) f(s) \;ds,\\[2pt]  \nonumber
		&G_0(\mu, s, y) =  e^{(1)}_0(\mu, y) e^{(2)}_0(\mu, s) - e^{(2)}_0(\mu, y) e^{(1)}_0(\mu, s),
	\end{align}
	and where $w(y)$ is the Wronskian with the expansion $w^{-1}(y) = \sum_{k \geq 0} y^{3 + 2k} d_k$ if $ 0 < y \ll1$. In fact using Lemma \ref{Lemma-FS-inner-small}, we may write the particular solution of \eqref{Green-int1} as follows if we assume $f(y)$ has an absolute expansion at $ y = 0$.
	\begin{align}\label{Green-int2}
		\tilde{w}(y) &= \int_{y_0}^y w^{-1}(s) G_0(\mu, s, y) f(s) \;ds,\\[2pt]  \nonumber
		&=  e^{(1)}_0(y) \int_{y_0}^y w^{-1}(s) \mathcal{O}(s^{-2}) f(s) \;ds - \mathcal{O}(y^{-2}) F(s) |_{y_0}^y\\ \nonumber
		& \;\; - e^{(1)}_0(y) \int_{y_0}^yF(s) s^{-1}\;ds +  e^{(1)}_0(y)\big( \log(s) F(s) |_{y_0}^y - \log(y) F(s) |_{y_0}^y\big),
	\end{align}
	where $ F(s) $ is a primitive of $ s \mapsto w^{-1}(s) e^{(1)}_0(s) f(s)$. Therefore if $ f(y) $ has \emph{controlled growth}, say $ f(y) = o( y^{-1 + \delta})$  as $ y \to 0^{+}$ for some $ 0 < \delta  <1$, then we can take $ y_0 = 0$ and the last two terms in \eqref{Green-int2} cancel.%\\[10pt]
	\tinysection[10]{Linear homogeneous systems}
	Now the coefficient $g_{m(\alpha_{1*}), \alpha_{1*}, \alpha_2(\alpha_1*)}(y)$ with minimal parameters  is a linear combination of 
	$$ e_1^{(0)}(y),\; e_2^{(0)}(y) $$
	and we see that this expression is consistent with the form \eqref{exp1}. Further, for the subsequent calculation of the coefficient functions
	\[
	g_{m_{\star} -1, \alpha_{1\star}, \alpha_2(\alpha_{1\star})}(y),\; g_{m_{\star} -2, \alpha_{1\star}, \alpha_2(\alpha_{1\star})}(y), \;g_{m_{\star} -3, \alpha_{1\star}, \alpha_2(\alpha_{1\star})}(y),\; \dots,\; g_{0, \alpha_{1\star}, \alpha_2(\alpha_{1\star})}(y)
	\]
	with $m_{\star} = m(\alpha_{1\star})$, i.e. the family of coefficient functions $  \{  g_{m, \alpha_{1\star}, \alpha_2(\alpha_{1\star})}  \}_{m =0}^{m_{\star}}$, 
	we need to consider the terms in \eqref{clac1} arising from the commutator 
	$$[i t \partial_t + \mathcal{L}_S, (\log t - \nu\log y)^{m}]g(y)$$
	and which are displayed in the third and fourth line of \eqref{clac1}.  Hence if we neglect the interaction terms (which will not be present if we select $\alpha_{1\star}, \alpha_2(\alpha_{1\star}), m(\alpha_{1\star})$) we obtain that the \emph{homogeneous equation}
	\begin{align}\label{hom-inner}
		\big( i t \partial_t + \mathcal{L}_S - \alpha_0\big)\big[t^{\nu\alpha_{1}+\alpha_{2}}\sum_{m = 0}^{m_{\star}}(\log y - \nu\log t)^m \cdot g_{m,\alpha_{1},\alpha_{2}}(y)\big] = 0,
	\end{align}
	where  as before $m_{\star} = m(\alpha_1)$ and  $ g$ either stands for $ g^{(1)}$ or $ g^{(2)}$, is solved via the following general \emph{homogeneous system}  for $\{ g_{m,\alpha_{1},\alpha_{2}} \}_{m}$
	\boxalign[11cm]{
		\begin{align}\label{inner-sys-first}
			&\big(\mathcal{L}_S + \mu_{\alpha_1, \alpha_2})g_{m_{\star}, \alpha_1,\alpha_2}(y) = 0\\[4pt]
			&\big(\mathcal{L}_S + \mu_{\alpha_1, \alpha_2})g_{m_{\star}-1, \alpha_1,\alpha_2 }(y) = - m_{\star} D_y g_{m_{\star}, \alpha_1,\alpha_2}(y),\\[4pt] \label{inner-sys-last}
			& \big(\mathcal{L}_S + \mu_{\alpha_1, \alpha_2})g_{m, \alpha_1,\alpha_2 }(y) = - (m+1) D_y g_{m +1, \alpha_1,\alpha_2}(y)\\[4pt] \nonumber
			& \hspace{4.2cm} - (m+2)(m+1) \frac{1}{y^2} g_{m +2, \alpha_1,\alpha_2}(y),\\[4pt] \nonumber
			& 0 \leq m \leq m_{\star} -2,
		\end{align}
	}
	where $ \mu_{\alpha_1, \alpha_2} := i(\nu \alpha_1 + \alpha_2) - \alpha_0 $ and we set (c.f. \cite{schmid})
	\[
	D_y : = \big(\tfrac{2}{y^2} - i(\tfrac12 + \nu)\big) + \tfrac{2}{y} \partial_y = \frac{2}{y^2}\Lambda- i(\tfrac12 + \nu)
	\]
	%we deduce an equation of the schematic form 
	%{align*}
	%	-\alpha_0 \tilde{g} - \frac{i}{2}\tilde{g} + i(\nu\alpha_{1*}\nu + \alpha(\alpha_{1*}))\tilde{g} - \frac{i}{2}y\tilde{g}_y + (\partial_{yy} + \frac{3}{y}\partial_y)\tilde{g}\sim y^{-1}\partial_y g + y^{-2}g + ig,
	%\end{align*}
	%and the unique solution is again obtained by matching with the corresponding function from the inner region $r\lesssim t^{\frac12+\epsilon}$. Proceeding inductively, we thereby determine all the coefficient functions 
	%\[
	%g_{a, \alpha_{1*}, \alpha_2(\alpha_1*)},\,0\leq a\leq a(\alpha_{1*}). 
	%\]
	For later reference, specifically the asymptotic expansion of the iterative corrections in the region of large $y\gg 1$, we want to make this precise. Therefore we  state the  following Lemma , for which a we again refer to \cite{schmid}.
	\begin{lem}\label{lem:smally1} Let $ \alpha_1 \in \Z$,\;$ \alpha_2 \in \Z$ or $ \alpha_2 \in \Z + \f12$. Further we let   $a_m, b_m \in \C$ for $ m = 0,1, \dots, m_{\star}$ and $ y_0 > 0$ be fix. Then the equation \eqref{hom-inner}, respectively the system \eqref{inner-sys-first} - \eqref{inner-sys-last},  admits a unique solution $\{g_{m,\alpha_1,\alpha_2}(y)\}_{m=0}^{m_{\star}}$ satisfying the conditions  
		\begin{align}\label{eq:goneconditions1}
			g_{m,\alpha_1,\alpha_2}(y_0) &= a_m,\;\;\partial_yg_{m,\alpha_1,\alpha_2}(y_0) = b_m,\,\, m = 0, 1,\ldots, m_{\star}
		\end{align}
		and such that the  functions $ g_{m,\alpha_1,\alpha_2}$ admit absolute expansions of the form 
		\begin{align}\label{exp-hom-schro}
			&g_{m_{\star} -m ,\alpha_{1},\alpha_{2}}(y) = \sum_{s = 0}^{m} \sum_{k \geq0} y^{2k -2} (\log(y))^s c^{(m)}_{k s} +  \sum_{k \geq0} y^{2k} (\log(y))^{m +1}   c^{(m)}_{k},\\[3pt] \nonumber
			&m = 0, 1, \dots, m_{\star},\;\; 0 < y \ll1,
		\end{align}
		Instead of \eqref{eq:goneconditions1} we may prescribe the coefficients $ c^{(m)}_{0 0}, c^{(m)}_{1 0}, m = 0,1,2,\dots, m_{\star}$, which uniquely determine the solution. 
		%where the $'$ indicates that the sum is over finitely many nonnegative $c$ and $k\geq -b$ for some positive $b$. 
	\end{lem}
	\begin{Rem} We note that the dependence of the coefficients $ c^{(m)}_{k s}, c^{(m)}_{k}$  in \eqref{exp-hom-schro} on $ c^{(m)}_{0 0}, c^{(m)}_{1 0} $ will be important below in order to satisfy the matching condition \eqref{exp1}, \eqref{exp2}. However it is more practical to give more details in the iteration procedure.
	\end{Rem}
	For such a solution we also refer to the \cite[Remark 2.29]{schmid} through the basis constructed in \cite[Lemma 2.27, Corollary 2.28]{schmid}. However, we can argue directly for the system \eqref{inner-sys-first} - \eqref{inner-sys-last} as follows. 
	\begin{proof}[Sketch of proof]
		The first line \eqref{inner-sys-first} implies the above expansion by Lemma \ref{Lemma-FS-inner-small} and the coefficients are selected according to \eqref{eq:goneconditions1}. For the right side of the line below \eqref{inner-sys-first}  we expand
		\[
		D_y g_{m_{\star}, \alpha_1, \alpha_2}(y) = \sum_{k \geq 0} y^{2k -4} \beta_{k,1}  + \sum_{k \geq 0} y^{2k -2} \beta_{k,2}  \log(y).
		\]
		Then calculating \eqref{Green-int1} using this expansion implies the claim for $ g_{m_{\star} -1, \alpha_1, \alpha_2}(y)$ where we again add a linear combination of $e^{(1)}_0,\; e^{(2)}_0$  in order to satisfy \eqref{eq:goneconditions1}.  We conclude by induction where 
		\[
		D_y g_{m_{\star} -m ,\alpha_{1},\alpha_{2}}(y) = \sum_{s = 0}^{m} \sum_{k \geq0} y^{2k -4} (\log(y))^s \beta^{(m)}_{k s}  +  \sum_{k \geq0} y^{2k-2} (\log(y))^{m +1}   \tilde{\beta}^{(m)}_{k},
		\]
		given $ 0 < y \ll 1 $ is small enough.
	\end{proof}
	Lemma \ref{lem:smally1} now uniquely determines the source term $ t^{-1}(\partial_{y}^2+ \frac{3}{y}\partial_y)(t^{-1}|w|^2)$  for the \emph{wave part} in the second line of  \eqref{schrod-wave-y} corresponding to the lowest power of $t > 0$, i.e.
	\[
	t^{2 \alpha_{1\star}\nu + 2\alpha_2(\alpha_{1 \star}) -2}
	\]
	Thus, in order to proceed, we need to consider the second  equation of \eqref{schrod-wave-y}. In particular we need  to understand \emph{homogeneous solutions} for  the operator $ \square_S$ and  corresponding %non-uniqueness 
	contributions within {\it{the class of functions admitting expansions of the form}}
	\[
	n_{\tilde{N}_2}(t,y) =  \sum_{n, l, \tilde{l}, p \geq 0} \sum_{\tilde{k} \geq  3 \tilde{l} - 4l -2p -4 } t^{\nu(2(l + n +p) - \tilde{l}) +  \frac{\tilde{k}}{2}}  \sum_{m =0}^{s_{\star} + 2n} \big(\log(y) - \nu \log(t)\big)^m \tilde{g}_{n,m,l,\tilde{k}, \tilde{l}, p}(y)
	\]
	derived above.  For now, we in fact consider the generic form \eqref{eq:wgeneric-2}, i.e.
	\begin{equation} \begin{split}
			n(t, y) &= \sum_{\beta_1,\beta_2}t^{\nu \beta_1+\beta_2} \sum_{m}(\log(y) - \nu\log(t))^m  h^{(1)}_{m, \beta_1,\beta_2}(y)\\
			& \;\;+ \sum_{\beta_1,\beta_2}t^{\nu\beta_1+\beta_2 + \f12} \sum_{m}(\log(y) - \nu\log(t))^m  h^{(2)}_{m, \beta_1,\beta_2}(y).
		\end{split}
	\end{equation}
	Therefore  we recall
	\[
	\square_S = - \big(\partial_t - \f12 t^{-1}y \partial_y\big)^2 + t^{-1}\big( \partial_y^2 + \frac{3}{y} \partial_y \big),
	\]
	hence we consider the \emph{homogeneous system} corresponding to 
	\begin{align} \label{hom-wave1}
		t^2 \cdot \square_S\bigg( \sum_{\beta_1,\beta_2}t^{\nu \beta_1+\beta_2} \sum_{m = 0}^{m_{\star}}(\log(y) - \nu\log(t))^m  h_{m, \beta_1,\beta_2}(y) \bigg) = 0,
	\end{align}
	where now $ \beta \in \Z_+$ is fixed, $ m_{\star} = m_{\star}(\beta_2) $, $h = h^{(1)}$ or $h = h^{(2)}$ and corresponding to the latter, $ \beta_2 \in \Z $ or $ \beta_2 \in \Z + \f12$. Further we note that $\beta_2 $ is not fixed,  but allowed to be variable with a lower bound depending only  on $ \beta_1 \in \Z_+$.
	\begin{Rem} As seen above and we make precise below, we have in fact 
		$$ \beta_1 = \beta_1(n, l, \tilde{l}, p),\; \beta_2 = \beta_2(\tilde{l}, l, p),\; m_{\star} = m_{\star}(l,n),$$
		and whether we consider the integer or half-integer case, depends only on  $ l \in 2 \Z$ and $ l \in 2\Z +1$. Here we recall  the size of $l \in \Z_+$ used to control the accuracy \footnote{in terms of decay of the remainder} of the wave-parametrix $\square^{-1}_{KST}$ in the interior iterations of the $(t,R,a)$ region.
	\end{Rem}
	\;\\
	First, in order to determine \eqref{hom-wave1},  it will be convenient to expand
	\begin{align*}
		&(\log(y) - \nu\log(t))^m\cdot  h_{m, \beta_1,\beta_2}(y)\\&= \big(-2\nu(\log(y) + \frac12 \log (t)) + (2\nu+1)\log(y)\big)^m\cdot h_{m, \beta_1,\beta_2}(y)\\
		& = \sum_{s =0}^m \left(\begin{array}{c}m\\s\end{array}\right)(-2\nu)^s(\log y + \frac12 \log t)^s (2\nu+1)^{m-s}(\log(y))^{m-s} \cdot h_{m,\beta_1,\beta_2}(y),
	\end{align*}
	thus we consider
	\begin{align} \nonumber
		0 &=	t^2 \square_S\bigg( \sum_{\beta_1,\beta_2}t^{\nu \beta_1+\beta_2} \sum_{m = 0}^{m_{\star}}(\log(y) + \f12\log(t))^m \big[ \sum_{s = 0}^{m_{\star} - m} (\log(y))^{s} c_{s, m} h_{s + m, \beta_1,\beta_2}(y)\big] \bigg)\\ \label{thisrealtion}
		& = \sum_{\beta_1,\beta_2} \sum_{m = 0}^{m_{\star}} %t^{\nu \beta_1+\beta_2}  H_{m, \beta_1, \beta_2}(t,y),
		t^2 \square_S \bigg( t^{\nu \beta_1+\beta_2} (\log(y) + \f12\log(t))^m  \tilde{h}_{m, \beta_1, \beta_2}(y) \bigg),
	\end{align}
	%where the functions $H_{m, \beta_1, \beta_2}(t,y)$ in the latter sum are equal  to
	with the coefficient functions
	\begin{align} \nonumber
		%& t^{-\nu \beta_1 - \beta_2} \cdot  t^2 \square_S \bigg( t^{\nu \beta_1+\beta_2} (\log(y) + \f12\log(t))^m  \tilde{h}_{m, \beta_1, \beta_2}(y) \bigg),\\ \label{that-relation}
		&\tilde{h}_{m, \beta_1, \beta_2}(y) :  =  \sum_{s = 0}^{m_{\star} - m} (\log(y))^{s} c_{s, m} h_{s + m, \beta_1,\beta_2}(y).
	\end{align} 
	Let us start by holding $ \beta_2,m $ fix and we calculate %the functions $ H_{m, \beta_1, \beta_2}(t,y)$, i.e. 
	\begin{align} \nonumber
		%	& t^{\nu \beta_1+\beta_2}  H_{m, \beta_1, \beta_2}(t,y) =\\
		t^2 \square_S &\bigg( t^{\nu \beta_1+\beta_2} (\log(y) + \f12\log(t))^m  \tilde{h}_{m, \beta_1, \beta_2}(y) \bigg) =\\ \label{thisthisrealtion}
		&  - (\log(y) + \f12\log(t))^m  t^2 \big(\partial_t - \f12 t^{-1}y \partial_y\big)^2 \bigg( t^{\nu \beta_1+\beta_2} \tilde{h}_{m, \beta_1, \beta_2}(y) \bigg)\\ \nonumber
		& +  t^{\nu \beta_1+\beta_2 + 1}\big( \partial_y^2 + \frac{3}{y} \partial_y \big)   \bigg( (\log(y) + \f12\log(t))^m \tilde{h}_{m, \beta_1, \beta_2}(y) \bigg),
	\end{align}
	where we used that changing to $ (t,r) $ coordinates shows the commutator
	\[
	[\big(\partial_t - \f12 t^{-1}y \partial_y\big) ,(\log(y) + \f12\log(t))] = 0.
	\]
	In terms of powers of $t > 0$, the last term in the above expression  lies in $ o(t^{\nu \beta_1+\beta_2})$, so it will contribute to an error (to be corrected) for coefficient functions of higher order. More precisely through $\tilde{h}_{\tilde{m}, \beta_1, \beta_2 +1}$ for some $ \tilde{m} \in \Z_+$. Hence we calculate the family $ \{ \tilde{h}_{m, \beta_1, \beta_2}\}_{m = 0}^{m_{\star} }$  via the first  expression, i.e. 
	\begin{align*}
		&\sum_{m = 0}^{m_{\star} } (\log(y) + \f12\log(t))^m t^2 \big(\partial_t - \f12 t^{-1}y \partial_y\big)^2 \big( t^{\nu \beta_1+\beta_2} \tilde{h}_{m, \beta_1,\beta_2}(y) \big)\\
		&=  t^{\nu \beta_1+\beta_2} \sum_{m = 0}^{m_{\star} } (\log(y) + \f12\log(t))^m \bigg( L_{\beta_1, \beta_2} \tilde{h}_{m, \beta_1,\beta_2}(y) \bigg),
	\end{align*}
	where we define the operator 
	\begin{align} \label{operator-wave1}
		L_{\beta_1, \beta_2}f(y) =&  \;\frac{1}{4} \big(y \partial_y\big)^2f(y) -  y \partial_y f(y) \cdot (\nu\beta_1+\beta_2 - \f12)\\[2pt] \nonumber
		&\;\;+ (\nu \beta_1+\beta_2) (\nu \beta_1+\beta_2 -1)\cdot  f(y).
	\end{align}
	Factoring off $y^2 $ we obtain the Wronskian $ w(y) \sim y^{4(\nu \beta_1 + \beta_2) -3}$ and the  simple fundamental base  $$\tilde{e}^{(1)}_0( y) = y^{2 (\nu \beta_1 + \beta_2)},\;  \tilde{e}^{(2)}_0( y) = y^{2 (\nu \beta_1 + \beta_2) -2},\;\; y \in (0, \infty).$$
	More generally for  $ \gamma \in \R$ fix,  the following equation
	\begin{align*}
		\big(y \partial_y\big)^2w(y)  -  (4 \gamma -2) \cdot y \partial_y w(y) + 4 \gamma (\gamma -1) \cdot w(y) = 0
	\end{align*}
	clearly has the fundamental solutions $ \tilde{e}^{(1)}_0(\gamma, y) = y^{2 \gamma},\;  \tilde{e}^{(2)}_0(\gamma, y) = y^{2 \gamma -2} $ with Wronskian $ w(y) \sim y^{4 \gamma -3}$. Therefore  the inhomogeneous equation 
	\begin{align}\label{inhom-wave-ode-scalar}
		L_{\beta_1 \beta_2} w(y) = f(y),\;\;y \in(0,y_1)
	\end{align} 
	has the following solutions for regular functions $f(y)$ and some $ y_1 > y_0 > 0$ 
	\begin{align}\label{part-sol-wave}
		w(y) =& \; c_0 y^{2(\beta_1\nu + \beta_2)} + c_1y^{2(\beta_1\nu + \beta_2)-2} + c \cdot y^{2(\beta_1\nu + \beta_2)} \int_{y_0}^y s^{-2(\beta_1 \nu + \beta_2) -1} f(s)\;ds\\[2pt] \nonumber
		&\;\;\hspace{2cm}- c \cdot y^{2(\beta_1\nu + \beta_2) -2} \int_{y_0}^y s^{-2(\beta_1 \nu + \beta_2) +1} f(s)\;ds
	\end{align} 
	Let us calculate the second expression on the right of  \eqref{thisthisrealtion}, i.e.
	\begin{align} \label{second-on-the-right-t-y}
		&\sum_{m=0}^{m_{\star}}\big( \partial_y^2 + \frac{3}{y} \partial_y \big)   \bigg( (\log(y) + \f12\log(t))^m \tilde{h}_{m, \beta_1, \beta_2}(y) \bigg)\\ \nonumber
		%	&= (\log(y) + \f12\log(t))^m \big( \partial_y^2 + \frac{3}{y} \partial_y \big) \tilde{h}_{m, \beta_1, \beta_2}(y) + [\big( \partial_y^2 + \frac{3}{y} \partial_y \big),  (\log(y) + \f12\log(t))^m]\tilde{h}_{m, \beta_1, \beta_2}(y)\\
		&= (\log(y) + \f12\log(t))^{m_{\star}} \big( \partial_y^2 + \frac{3}{y} \partial_y \big) \tilde{h}_{m_{\star}, \beta_1, \beta_2}(y)\\\nonumber
		&\hspace{.5cm} + (\log(y) + \f12\log(t))^{m_{\star}-1} \bigg(\big( \partial_y^2 + \frac{3}{y} \partial_y \big) \tilde{h}_{m_{\star}-1, \beta_1, \beta_2}(y) + m_{\star} \big(\frac{2}{y}\partial_y + \frac{2}{y^2}\big) \tilde{h}_{m_{\star}, \beta_1, \beta_2}(y)\bigg)\\ \nonumber
		&\hspace{.5cm} + \sum_{m = 0}^{m_{\star}-2}(\log(y) + \f12\log(t))^{m} \bigg(\big( \partial_y^2 + \frac{3}{y} \partial_y \big) \tilde{h}_{m, \beta_1, \beta_2}(y) + (m +1) \big(\frac{2}{y}\partial_y + \frac{2}{y^2}\big) \tilde{h}_{m+1, \beta_1, \beta_2}(y)\\ \nonumber
		&\hspace{5.5cm} + (m+2)(m+1) \frac{1}{y^2} \tilde{h}_{m+2, \beta_1, \beta_2}(y)\bigg).
	\end{align}
	Hence \eqref{thisrealtion} leads to the following system 
	\begin{align} \label{inner-wave-trival}
		&L_{\beta_1, \beta_2} \tilde{h}_{m, \beta_1,\beta_{2\star}}(y) = 0,\;\;0 \leq m \leq m_{\star},
	\end{align}
	in case $ \beta_{2\star} = \beta_{2}(\beta_1) \in \Z$ is minimal and 
	\boxalign[13cm]{
		\begin{align} \label{inner-wave-first}
			L_{\beta_1, \beta_2} \tilde{h}_{m_{\star}, \beta_1,\beta_2}(y) &=  \big( \partial_y^2 + \frac{3}{y} \partial_y \big) \tilde{h}_{m_{\star}, \beta_1, \beta_2-1}(y),\\ \label{inner-wave-second}
			L_{\beta_1, \beta_2} \tilde{h}_{m_{\star}-1, \beta_1,\beta_2}(y) &= \big( \partial_y^2 + \frac{3}{y} \partial_y \big) \tilde{h}_{m_{\star}-1, \beta_1, \beta_2-1}(y)\\ \nonumber
			& \;\;\;\; + m_{\star} \big(\frac{2}{y}\partial_y + \frac{2}{y^2}\big) \tilde{h}_{m_{\star}, \beta_1, \beta_2-1}(y),\\ \label{inner-wave-last}
			L_{\beta_1, \beta_2} \tilde{h}_{m, \beta_1,\beta_2}(y) &= \big( \partial_y^2 + \frac{3}{y} \partial_y \big) \tilde{h}_{m, \beta_1, \beta_2-1}(y)\\\nonumber
			&\;\;\;\; + (m +1) \big(\frac{2}{y}\partial_y + \frac{2}{y^2}\big) \tilde{h}_{m+1, \beta_1, \beta_2-1}(y)\\\nonumber
			&\;\;\;\;+ (m+2)(m+1) \frac{1}{y^2} \tilde{h}_{m+2, \beta_1, \beta_2-1}(y),\\ \nonumber
			0 \leq m \leq m_{\star}-2,&
		\end{align}
	}
	subsequently for \; $\beta_2 = \beta_{2\star} +1,\; \beta_{2\star} +2, \; \beta_{2\star} +3,\; \beta_{2\star} +4,\dots$.
	\begin{Rem}\;We iterate a finite length over $\beta_2 \in \Z$, therefore in contrast to the homogeneous system \eqref{inner-sys-first} - \eqref{inner-sys-last}, solving \eqref{inner-wave-first} - \eqref{inner-wave-last} only leads to a an `approximate' homogeneous solution. This is in particular the case because the `equal strength' derivatives in the $\square-$operator scale differently in this parabolic set-up.%\\[2pt]
		%(ii)\; The approximate constructions introduce an \emph{ambiguity for the approximation of the wave part}. This means there are free parameters (coefficients) which add wave approximations (iterated to high order) and lead to approximate solutions of the same type. We stress that $ t^{\epsilon_1} \lesssim y$ and if $ 0 <t \ll1$ is small enough then $ a \gtrsim t^{\epsilon_1- \f12 } \gg1 $. Thus in hyperbolic scaling  these approximations are far in the light cone exterior region.
	\end{Rem}
	We now have the following Lemma.
	\begin{lem}\label{lem:hom1} Let $n(t,y) $  have the expansion
		$$n(t,y) = \sum_{\beta_1,\beta_2}t^{\nu \beta_1+\beta_2} \sum_{m = 0}^{m_{\star}}(\log(y) - \nu\log(t))^m  h_{m, \beta_1,\beta_2}(y),$$
		where the sum is  finite over $ \beta_1 \in \Z_+$, absolutely convergent in $\beta_2 \geq \beta_{2\star}$ ( for $ 0 < y \ll1$), as well as derivatives up to second order term wise all absolutely convergent.
		We consider the homogeneous equation 
		\begin{align} \label{hom-wave2}
			& \square_S\bigg( \sum_{\beta_1,\beta_2}t^{\nu \beta_1+\beta_2} \sum_{m = 0}^{m_{\star}}(\log(y) - \nu\log(t))^m  h_{m, \beta_1,\beta_2}(y)\bigg) = 0,\\[3pt] \nonumber
			&\square_S = -(\partial_t -\frac12 t^{-1}y\partial_y)^2 + t^{-1}(\partial_{yy} + \frac{3}{y}\partial_y). 
		\end{align}
		Then there exist sequences of (real) coefficients for $ \ell \geq 0$ 
		$$ d^{(1)}_{\ell, k,m}, \; d^{(2)}_{\ell,k,m},\;\;0 \leq k \leq \ell, \; m = 0,\dots, m_{\star}$$ such that $ n(t,y) = n_1(t,y) + n_2(t,y)$ has the form
		\boxalign[13cm]{
			\begin{align} \label{expansion-n-1}
				n_1(t, y) &=  \sum_{\beta_1,\ell \geq 0} t^{\nu \beta_1 + \beta_{2\star} + \ell}\sum_{m = 0}^{m_{\star}} \sum_{k = 0}^{\ell} \times\\ \nonumber
				&\hspace{2cm} \times \big(\log(y) + \f12 \log(t)\big)^m d^{(1)}_{\ell, k ,m} y^{2(\nu\beta_1+ \beta_{2\star} + \ell -2k)},\\[5pt] \label{expansion-n-2}
				n_2(t, y) &= \sum_{\beta_1, \ell \geq 0} t^{\nu \beta_1 + \beta_{2\star} + \ell}\sum_{m = 0}^{m_{\star}} \sum_{k = 0}^{\ell}\times\\ \nonumber
				&\hspace{2cm} \times \big(\log(y) + \f12 \log(t)\big)^m d^{(2)}_{\ell, k ,m} y^{2(\nu\beta_1+ \beta_{2\star} + \ell -2k) -2},
			\end{align}
		}
		The coefficients $ d^{(1)}_{\ell, 0 ,m},\; d^{(2)}_{\ell, 0 ,m} $ with $ 0 \leq m \leq m_{\star}$ and $ 0 \leq \ell \leq L$  uniquely determine the coefficients 
		$$ d^{(1)}_{\tilde{l},k,\tilde{m}},\; d^{(1)}_{\tilde{l},k,\tilde{m}},\;\;1\leq \tilde{l}\leq L+1,\;\text{and}\; k > 0\;\text{if}\;\tilde{l} = L+1. $$
		In particular we define (iterating $L \in \Z_+$ steps)
		\begin{align}
			n_j^{(\leq L)} &= :\sum_{\beta_1}\sum_{\ell = 0}^L t^{\nu \beta_1 + \beta_{2\star} + \ell}\sum_{m = 0}^{m_{\star}} \sum_{k = 0}^{\ell} \times\\ \nonumber
			&\hspace{2cm} \times \big(\log(y) + \f12 \log(t)\big)^m d^{(j)}_{\ell, k ,m} y^{2(\nu\beta_1+ \beta_{2\star} + \ell -2k) -2(j-1)},\;\;\; 
		\end{align}
		for $j = 1,2$ and thus the functions $n^{(\leq L)} =:n_1^{(\leq L)} + n_2^{(\leq L)}$ satisfy
		\begin{align*}
			&\big(-(\partial_t -\frac12 t^{-1}y\partial_y)^2 + t^{-1}(\partial_{yy} + \frac{3}{y}\partial_y)\big) n^{(\leq L)} = \mathcal{O}\big(t^{\nu\beta_1+\beta_{2\star}+ L+ 1}),\;\text{as}\;t \to 0^+.
		\end{align*}
		We call $n^{(\leq L)}$ an \underline{almost %(structurally compatible)
			free wave} of $L$-th order and
		note the coefficients $d^{(1)}_{\ell, 0 ,m},\; d^{(2)}_{\ell, 0 ,m}$  can be freely chosen. 
	\end{lem}
	\begin{proof} We use the above observation that expanding the sum in \eqref{hom-wave2}
		\begin{align*}
			&(\log(y) - \nu\log(t))^m\cdot  h_{m, \beta_1,\beta_2}(y)\\&= \big(-2\nu(\log(y) + \frac12 \log (t)) + (2\nu+1)\log(y)\big)^m\cdot h_{m, \beta_1,\beta_2}(y)\\
			& = \sum_{s =0}^m \left(\begin{array}{c}m\\s\end{array}\right)(-2\nu)^s(\log y + \frac12 \log t)^s (2\nu+1)^{m-s}(\log(y))^{m-s} \cdot h_{m,\beta_1,\beta_2}(y),
		\end{align*}
		leads to the ansatz 
		\begin{align*}
			&n(t, y) = \sum_{\beta_1,\beta_2} t^{\nu\beta_1+\beta_2}\sum_{m = 0}^{m_{\star}}  (\log(y) + \f12\log(t))^m  \tilde{h}_{m, \beta_1, \beta_2}(y),\\
			&\tilde{h}_{m, \beta_1, \beta_2}(y) = \sum_{s = 0}^{m_{\star} - m} (\log(y))^{s} c_{s, m} h_{s + m, \beta_1,\beta_2}(y),
		\end{align*}
		hence we proceed as explained above. In particular, for each $\beta_1$, selecting $\beta_2  = \beta_{2\star}(\beta_1)$ minimal leads the coefficient functions $ \tilde{h}_{m, \beta_1, \beta_{2\star}}$ to be determined by \eqref{inner-wave-trival}. Thus we have 
		\[
		\tilde{h}_{m, \beta_1, \beta_{2\star}}(y) = c_{0,m}^{(1)}\cdot 	y^{2(\nu\beta_1+\beta_{2\star})} +  c_{0,m}^{(2)} \cdot y^{2(\nu\beta_1+\beta_{2\star})-2},\;\;  0 \leq m \leq m_{\star}.
		\]
		This precisely corresponds the part of the series with with $\ell = k = 0$.
		We prove by induction over $ \beta_{\star} + \ell$ and start with $ \beta_2 = \beta_{2\star} +1$. Hence the system \eqref{inner-wave-first}- \eqref{inner-wave-last} reads
		\begin{align*}
			L_{\beta_1, \beta_2} \tilde{h}_{m_{\star}, \beta_1,\beta_{2\star}+1}(y) &=\;  c_{0, m_{\star}}^{(1)}\zeta_0^{(1)}y^{2(\nu \beta_1 + \beta_{2\star}) -2}+ c_{0, m_{\star}}^{(2)}\zeta_1^{(1)} y^{2(\nu \beta_1 + \beta_{2\star}) -4},\\[8pt]
			L_{\beta_1, \beta_2} \tilde{h}_{m_{\star}-1, \beta_1,\beta_2}(y) &=\; c_{0, m_{\star}-1}^{(1)}\zeta_0^{(1)}y^{2(\nu \beta_1 + \beta_{2\star}) -2} + c_{0, m_{\star}-1}^{(2)}\zeta_1^{(1)} y^{2(\nu \beta_1 + \beta_{2\star}) -4},\\
			&\;\;\;\; +c_{0, m_{\star}}^{(1)} m_{\star}\zeta_0^{(2)}y^{2(\nu \beta_1 + \beta_{2\star}) -2}+ c_{0, m_{\star}}^{(2)} m_{\star}\zeta_1^{(2)} y^{2(\nu \beta_1 + \beta_{2\star}) -4},
		\end{align*}
		and for $ 0 \leq m \leq m_{\star -2}$
		\begin{align*}
			L_{\beta_1, \beta_2} \tilde{h}_{m, \beta_1,\beta_2}(y) &=\; c_{0, m}^{(1)}\zeta_0^{(1)}y^{2(\nu \beta_1 + \beta_{2\star}) -2}+ c_{0, m}^{(2)}\zeta_1^{(1)} y^{2(\nu \beta_1 + \beta_{2\star}) -4},\\
			&\;\;\;\; +c_{0, m +1}^{(1)} (m+1)\zeta_0^{(2)}y^{2(\nu \beta_1 + \beta_{2\star}) -2} + c_{0, m +1}^{(2)} (m+1)\zeta_1^{(2)} y^{2(\nu \beta_1 + \beta_{2\star}) -4},\\
			&\;\;\;\; + (m+2)(m+1) \big(c^{(1)}_{0, m+2} y^{2(\nu \beta_1 + \beta_{2\star}) -2}  +  c^{(2)}_{0, m+2} y^{2(\nu \beta_1 + \beta_{2\star}) -4}\big),
		\end{align*}
		where $ \zeta_{\ell}^{(k)}, k = 1,2 $ depend on $\beta_1, \beta_{2\star}$ and precisely we set
		\begin{align*}
			\zeta_{\ell}^{(1)} &: = 2(\nu \beta_1 + \beta_{2\star} - 2\ell) \big(2(\nu \beta_1 + \beta_{2\star}) - 2\ell + 2\big)\\
			\zeta_{\ell}^{(2)} &:= 2\big(2(\nu \beta_1 + \beta_{2\star}) -2\ell+1 \big)
		\end{align*}
		Therefore we have the following particular solutions
		$$ \tilde{h}_{m, \beta_1, \beta_{2\star} +1}(y) = d^{(1)}_{1,m} y^{2(\nu \beta_1 + \beta_{2\star}) -2} + d^{(2)}_{1,m}  y^{2(\nu \beta_1 + \beta_{2\star}) -4},$$
		and applying 
		$$ 4 L_{\beta_1, \beta_{2\star}} = y^2\partial_y^2 - (4\gamma -3)y \partial_y + 4\gamma(\gamma-1),\;\; \gamma = \beta_1 + \beta_{2\star}, $$
		shows the following explicit relations for the coefficients
		\begin{align*}
			&d^{(1)}_{1,m_{\star}} = \frac{1}{\mu_{1,1}} c_{0,m_{\star}}^{(1)}\zeta_{0}^{(1)},\;\;\;\; d^{(2)}_{1,m_{\star}} = \frac{1}{\mu_{1,2}}  c_{0,m_{\star}}^{(2)}\zeta_{1}^{(1)},\\[2pt]
			&d^{(1)}_{1,m_{\star}-1} = \frac{1}{\mu_{1,1}}\big(c_{0, m_{\star}-1}^{(1)}\zeta_{0}^{(1)} + c_{0, m_{\star}}^{(1)} m_{\star}\zeta_0^{(2)}\big),\;\;\;\;\; d^{(2)}_{1,m_{\star}-1} = \frac{1}{\mu_{1,2}}\big(c_{0, m_{\star}-1}^{(2)}\zeta_{1}^{(1)} + c_{0, m_{\star}}^{(2)} m_{\star}\zeta_1^{(2)}\big),\\[2pt]
			&d^{(1)}_{1,m} = \frac{1}{\mu_{1,1}}\big( c_{0, m}^{(1)}\zeta_{0}^{(1)} + c_{0, m +1}^{(1)} (m+1)\zeta_0^{(2)}+ (m+2)(m+1)c_{0, m+2}^{(1)} \big),\;\;0 \leq m \leq m_{\star}-2\\[2pt]
			&d^{(2)}_{1,m} =  \frac{1}{\mu_{1,2}}\big(  c_{0, m}^{(2)}\zeta_{1}^{(1)} + c_{0, m +1}^{(2)} (m+1)\zeta_1^{(2)} + (m+2)(m+1)c_{0, m+2}^{(2)}\big),\;\;0 \leq m \leq m_{\star}-2,
		\end{align*}
		where $ \mu_{j, \ell} $ depends on $\beta_1, \beta_{2\star}$. More precisely we set
		\begin{align*}
			\mu_{j,\ell} &:= \f14(2(\nu \beta_1 + \beta_{2\star}) - 2 \ell )(2(\nu \beta_1 + \beta_{2\star}) - 2 \ell -1)\\
			&\;\;\;\;\;\; - (\nu \beta_1 + \beta_{2\star} + j -3)(2(\nu \beta_1 + \beta_{2\star}) - 2 \ell )\\
			&\;\;\;\;\;\; + (\nu \beta_1 + \beta_{2\star} + j)(\nu \beta_1 + \beta_{2\star} + j-1).
		\end{align*}
		Then the solution has the form
		\begin{align*}
			\tilde{h}_{m, \beta_1, \beta_{2\star} +1}(y) &= d^{(1)}_{1,m} y^{2(\nu \beta_1 + \beta_{2\star} -1) } + d^{(2)}_{1,m}  y^{2(\nu \beta_1 + \beta_{2\star} -1) -2}\\
			&\;\;\;\;\;+ c_{1,m}^{(1)} y^{2(\nu \beta_1 + \beta_{2\star} +1) } + c_{1,m}^{(2)} y^{2(\nu \beta_1 + \beta_{2\star} +1) -2},
		\end{align*}
		where $ c_{1,m}^{(1)},c_{1,m}^{(2)}$ are free to choose. Considering the case $ \beta_2 = \beta_{2\star} +2$, the right side  of \eqref{inner-wave-first} - \eqref{inner-wave-last} clearly consists of linear combinations of the following terms,
		\[
		y^{2(\nu \beta_1 + \beta_{2\star} -2) },\;\; y^{2(\nu \beta_1 + \beta_{2\star} -3) },\; y^{2(\nu \beta_1 + \beta_{2\star}) },\;\;y^{2(\nu \beta_1 + \beta_{2\star} -1) },
		\]
		with coefficients depending on $ c_{1,m}^{(1)}, c_{1,m}^{(2)}$ and $ c_{0,m}^{(1)}, c_{0,m}^{(2)}$ (through $d_{1,m}^{(1)}, d_{1,m}^{(2)}$). Solving this system then will require the ansatz
		\begin{align*}
			\tilde{h}_{m, \beta_1, \beta_{2\star} +2}(y) &= d^{(1)}_{2,2,m} y^{2(\nu \beta_1 + \beta_{2\star} -2) } + d^{(2)}_{2,2,m}  y^{2(\nu \beta_1 + \beta_{2\star} -2) -2}\\
			&\;\;\;\;\;+ d^{(1)}_{2,1,m} y^{2(\nu \beta_1 + \beta_{2\star}) } + d^{(2)}_{2,1,m}  y^{2(\nu \beta_1 + \beta_{2\star}) -2}\\
			&\;\;\;\;\;+ c_{2,m}^{(1)} y^{2(\nu \beta_1 + \beta_{2\star} +2) } + c_{1,m}^{(2)} y^{2(\nu \beta_1 + \beta_{2\star} +2) -2}.
		\end{align*}
		We now make this precise in the inductive step and for convenience define
		\[
		d^{(j)}_{2,0,m} : = c_{2,m}^{(j)},\;\; d^{(j)}_{1,0,m} : = c_{1,m}^{(j)},\;\; d^{(j)}_{0,0,m} : = c_{0,m}^{(j)}.
		\] 
		Assume in the step: $\beta_2 + \ell $ the iteration has the claimed form \eqref{expansion-n-1} and \eqref{expansion-n-2}, that is for all $ 0 \leq \tilde{\ell} \leq \ell $ we have real coefficients 
		\[
		d^{(j)}_{\tilde{\ell},0,m},\;d^{(j)}_{\tilde{\ell},1,m},\;d^{(j)}_{\tilde{\ell},2,m},\;\dots,\; d^{(j)}_{\tilde{\ell},\tilde{\ell},m},\; 0 \leq m \leq m_{\star},\; j = 1,2,
		\]
		where $ d^{(j)}_{\tilde{\ell},0,m}$ was selected freely in the $\tilde{\ell}^{\text{th}}$-step and $ d^{(j)}_{\tilde{\ell},k,m}$ for $ k > 0$ depend only on 
		$$ d^{(j)}_{l,0,\tilde{m}},\; \zeta_{ l}^{(j)},\; \zeta_{ -l}^{(j)},\; \mu_{ l,  l+1},\;\mu_{ -l,  -l+1}$$ with $ 0 \leq l < \tilde{\ell}$ (iteratively through rational expressions of the above form). Then the right side of \eqref{inner-wave-first} - \eqref{inner-wave-last}, with $\beta_1, \beta_{2\star} + \ell +1$ on the left side, has the form
		\begin{align*}
			&\sum_{k = 0}^{\ell} \tilde{d}^{(1)}_{\ell, k, m} y^{2(\nu \beta_1 + \beta_{2\star} + \ell -2k) -2}+ \sum_{k = 0}^{\ell} \tilde{d}^{(2)}_{\ell, k, m} y^{2(\nu \beta_1 + \beta_{2\star} + \ell -2k) -4}\\
			&= \;\; \sum_{k = 1}^{\ell +1} \tilde{d}^{(1)}_{\ell, k -1, m} y^{2(\nu \beta_1 + \beta_{2\star} + \ell +1 -2k) }+ \sum_{k = 1}^{\ell +1} \tilde{d}^{(2)}_{\ell, k -1, m} y^{2(\nu \beta_1 + \beta_{2\star} + \ell +1 -2k) -2},
		\end{align*}
		where $ \tilde{d}^{(j+1)}_{\ell, k-1, m},\; j = 0,1$ depend on $ d^{(j+1)}_{\ell, k, m}$ through the same relations as above replacing $ \zeta^{(j)}_{j}$ by $ \zeta^{(j)}_{- \ell + 2k +j}$ respectively. Therefore the solution has the expansion
		\begin{align*}
			\tilde{h}_{m, \beta_1, \beta_{2 \star} + \ell +1}(y) =& \sum_{k = 0}^{\ell +1} \big(d^{(1)}_{\ell +1, k , m} y^{2(\nu \beta_1 + \beta_{2\star} + \ell +1 -2k) } + d^{(2)}_{\ell +1, k , m} y^{2(\nu \beta_1 + \beta_{2\star} + \ell +1 -2k) -2}\big),
		\end{align*}
		where, by calculating $ L_{\beta_1, \beta_{2\star} +\ell +1}$ falling on this expression, we have
		\begin{align*}
			&d^{(j+1)}_{\ell +1, k, m} = \frac{1}{\mu_{\ell+1, -(\ell +1) +2k +j}}\cdot \tilde{d}^{(j+1)}_{\ell, k-1, m},\; j = 0,1,\;\; k = 1, \dots, \ell +1,
		\end{align*}
		Further the coefficients $ d^{(j)}_{\ell +1, 0 , m} $ are again free to choose since they correspond to fundamental solutions of $L_{\beta_1, \beta_{2\star} +\ell +1}$.

	\end{proof}
	\begin{Rem}\label{Rem-after-wave-lem}(i)\;\; As seen in the proof, the coefficients in \eqref{expansion-n-1} and \eqref{expansion-n-2} (as well as $\beta_{2\star}$) depend of course also on the choice of $\beta_1$. Below we will make the exact dependence through the expansion in \eqref{wave-ansatz3} more precise.\\[3pt]
		(ii)\;\; We observe in fact the asymptotic $ n^{(\leq L)}(t,y) =\mathcal{O}(t^{\beta_1\nu + \beta_{2\star} + L + 1})$ is sufficiently regular. To be precise we obtain
		\begin{align*}
			&\partial_y^k \partial_{t}^j n^{(\leq L)}(t,y) = \mathcal{O}(t^{\beta_1\nu + \beta_{2\star} + L + 1 - j}),\;\;\;j < \beta_1\nu + \beta_{2\star} + L + 1,\\
			&k \leq 2(\beta_1\nu + \beta_{2\star}) - 2L -2.
		\end{align*}
		by  the proof of Lemma \ref{lem:hom1}.
	\end{Rem}
	Now Lemma \ref{lem:hom1} canonically determines the expansion \eqref{hom-wave1}, i.e. the equation
	\begin{align} \label{hom-wavee}
		&\square_S\big( \sum_{\beta_1,\beta_2}t^{\nu \beta_1+\beta_2} \sum_{m = 0}^{m_{\star}}(\log(y) - \nu\log(t))^m  h_{m, \beta_1,\beta_2}(y) \big) = 0,
		%&\tilde{h}_{m, \beta_1, \beta_2}(y) =  \sum_{s = 0}^{m_{\star} - m} (\log(y))^{s} c_{s, m} h_{s + m, \beta_1,\beta_2}(y),
	\end{align}
	which we sum up in the following Corollary.
	\begin{cor}\label{cor:hom-wave}The equation \eqref{hom-wavee} has a solutions (in the sense of Lemma \ref{lem:hom1})
		\boxalign{
			\begin{align}\label{hom-wave-formula}
				%& t^{-\nu \beta_1 - \beta_2} \cdot  t^2 \square_S \bigg( t^{\nu \beta_1+\beta_2} (\log(y) + \f12\log(t))^m  \tilde{h}_{m, \beta_1, \beta_2}(y) \bigg),\\ \label{that-relation}
				h_{m, \beta_1, \beta_{2}}(y)   &= \; \sum_{s = 0}^{m_{\star} - m} (\log(y))^s \tilde{h}_{s + m, \beta_1, \beta_2}(y)\\  \nonumber
				&= \; \sum_{s = 0}^{m_{\star} - m}(\log(y))^{s} c_{s,m} \sum_{k =0}^{\beta_2 - \beta_{2\star}} \big(d^{(1)}_{\beta_2 - \beta_{2\star}, k, s+m} y^{2(\beta_1 \nu + \beta_2 - 2k)}\\ \nonumber
				&\hspace{5cm} +   d^{(2)}_{\beta_2 - \beta_{2\star}, k, s+m} y^{2(\beta_1 \nu + \beta_{2} - 2k) -2} \big).
			\end{align} 
		}
		where  $c_{s,m}$ depends only on $ \nu > 0$. In particular we analogously define the  \underline{almost %(structurally compatible)
			free wave} of $L$-th order $n^{\leq L}(t,y)$ and note again  the coefficients $d^{(1)}_{\beta_2 - \beta_{2\star}, 0 ,m},\; d^{(2)}_{\beta_2 - \beta_{2\star}, 0 ,m}$  can be freely chosen.  
	\end{cor}
	%\tinysection[0]{General description: Linear inhomogeneous systems}
	\tinysection[0]{Linear inhomogeneous systems}
	We now derive the associated system for the linear operators of \eqref{schrod-wave-y}. Calculating the interaction terms in the iteration will reduce \eqref{schrod-wave-y} to such a system and therefore approximate solutions to \eqref{schrod-wave-y} by inductive construction of coefficient functions.\\[4pt]
	We first consider the equation
	$$ \square_S n(t,y) = \big(\partial_y^2 + \frac{3}{y}\partial_y \big)f(t,y),$$
	i.e. with the previous ansatz we write
	\begin{align} \label{inhom-wave1}
		&t^2\cdot \square_S\bigg( \sum_{\beta_2}t^{\nu \beta_1+\beta_{2}} \sum_{m = 0}^{m_{\star}}(\log(y) - \nu\log(t))^m  h_{m, \beta_1,\beta_2}(y) \bigg)  = \big( \partial_y^2 + \frac{3}{y} \partial_y \big)f(t,y)\\[5pt] \nonumber
		%	& \hspace{5cm} = \big( \partial_y^2 + \frac{3}{y} \partial_y \big)f(t,y)\\[5pt] 
		& f(t,y) = \sum_{\beta_1,\beta_2}t^{\nu \beta_1+\beta_2} \sum_{m = 0}^{m_{\star}}(\log(y) - \nu\log(t))^m  f_{m, \beta_1,\beta_2}(y),
	\end{align}
	where the sum over $ \beta_2 = \beta_{2\star} + \ell$ with $\ell \geq 0 $ finite. Here  we denote again $ \beta_{2\star} = \beta_{2}(\beta_1) \in \Z$ (or $ \Z + \f12 $) the minimal such (half-)integer over all $\beta_2$. Now with $\Lambda  = 1 + y \partial_y$ we calculate the commutator
	\begin{align*}
		[- t^2\big(\partial_t -& t^{-1} y \partial_y\big)^2, \big(\log(y) - \nu \log(t))^m]\\
		&= \;2m (\nu + \f12) \big(\log(y) - \nu \log(t))^{m-1} (t \partial_t  -\f12 \Lambda\big)\\
		&\;\;- m(m-1) (\f12 + \nu)^2 \big(\log(y) - \nu \log(t))^{m-2}.
	\end{align*}
	Combining this with the calculation \eqref{second-on-the-right-t-y} and comparing coefficients in  \eqref{inhom-wave1}, we obtain the system
	\boxalign[13cm]{
		\begin{align} \label{inhom-inner-wave-first}
			L_{\beta_1, \beta_2} h_{m_{\star}, \beta_1,\beta_2}(y) &=  \big( \partial_y^2 + \frac{3}{y} \partial_y \big) \big(h_{m_{\star}, \beta_1, \beta_2-1}(y) - f_{m_{\star}, \beta_1,\beta_2}(y)\big),\\ \label{inhom-inner-wave-second}
			L_{\beta_1, \beta_2} h_{m_{\star}-1, \beta_1,\beta_2}(y) &= \big( \partial_y^2 + \frac{3}{y} \partial_y \big) \big(h_{m_{\star}-1, \beta_1, \beta_2-1}(y) - f_{m_{\star} -1, \beta_1,\beta_2}(y)\big)\\ \nonumber
			&\;\;\;\; + 2m_{\star} (\nu + \f12)\big(\beta_1\nu + \beta_2 - \f12 \Lambda\big)h_{m_{\star}, \beta_1,\beta_2}(y)\\ \nonumber
			& \;\;\;\; + m_{\star} \big(\frac{2}{y}\partial_y + \frac{2}{y^2}\big) \big(h_{m_{\star}, \beta_1, \beta_2-1}(y) - f_{m_{\star}, \beta_1, \beta_2}(y)\big) ,\\ \nonumber
			L_{\beta_1, \beta_2} h_{m, \beta_1,\beta_2}(y) &= \big( \partial_y^2 + \frac{3}{y} \partial_y \big) \big(h_{m, \beta_1, \beta_2-1}(y) - f_{m, \beta_1, \beta_2}(y) \big)\\ \nonumber
			&\;\;\;\; + (m +1) \big(\frac{2}{y}\partial_y + \frac{2}{y^2}\big) \big(h_{m+1, \beta_1, \beta_2-1}(y) - f_{m+1, \beta_1, \beta_2}(y)\big)\\ \label{inhom-inner-wave-last}
			&\;\;\;\; +  2(m+1) (\nu + \f12)\big(\beta_1\nu + \beta_2 - \f12 \Lambda\big)h_{m +1, \beta_1,\beta_2}(y)\\ \nonumber
			&\;\;\;\;+ (m+2)(m+1) \frac{1}{y^2} \big(h_{m+2, \beta_1, \beta_2-1}(y) -  f_{m +2, \beta_1,\beta_2}(y)\big)\\ \nonumber
			&\;\;\;\; - (m+2)(m+1) (\f12 + \nu)^2 h_{m+2, \beta_1, \beta_2}(y) ,\\ \nonumber
			0 \leq m \leq m_{\star}-2.&
		\end{align}
	}
	\begin{Rem} The terms involving $\beta_2 -1$ on the right side are absent in the minimal case $ \beta_2 = \beta_{2\star}$.
	\end{Rem}
	For the solutions of the system \eqref{inhom-inner-wave-first} - \eqref{inhom-inner-wave-last}, we first note a simple Lemma.
	\begin{lem} \label{lem:das-einfache-welle-scalar-Lemma} Let $\beta_1 \in \Z_+$ and $\beta_2 \in \Z$ (or $\beta_2 \in \f12 + \Z$). We consider the equation
		\begin{align}\label{inhom-wave-ode-scalar2}
			L_{\beta_1, \beta_2} w(y) = f(y),\;\;y \in(0,\infty).
		\end{align} 
		Let $ f(y)$ be in $C^{\infty}((0, \infty))$ and such that in an absolute sense
		\begin{align}\label{f-linear-scalar-exp}
			f(y) = \sum_{s =0}^m \sum_{k \geq 0} \big(\log(y)\big)^{s} y^{2\nu r + 2k-l} c_{s,k},\;\; 0 < y \ll1,
			%f(y) = \sum_{s =0}^m \big(\log(y)\big)^{s} \mathcal{O}(y^{2\nu r -l}),\;\; 0 < y \ll1,
		\end{align}
		where $m , r,l \in \Z_{\geq 0}$ are fix. Then all solutions $w(y)$ of \eqref{inhom-wave-ode-scalar2} given by \eqref{part-sol-wave} satisfy
		$$ w(y) =  c_0 y^{2(\beta_1\nu + \beta_2)} + c_1y^{2(\beta_1\nu + \beta_2)-2} + \tilde{w}(y)$$
		where $\tilde{w}\in C^{\infty}((0,y))$ with an absolute expansion
		\begin{align}\label{this-exp-scalar-wave}
			\tilde{w}(y) = \sum_{s =0}^m \sum_{k \geq 0} \big(\log(y)\big)^{s} y^{2\nu r + 2k-l} \tilde{c}_{s,k},\;\; 0 < y \ll1,
			%\tilde{w}(y) = \sum_{s =0}^m \big(\log(y)\big)^{s} \mathcal{O}(y^{2\nu r -l}) ,\;\; 0 < y \ll1,
		\end{align}
		if $ r \neq \beta_1$. In case the latter is not satisfied, we have to replace $m $ by $m+1$ (increase the logarithmic power).
	\end{lem}
	\begin{proof} The particular solution is given by \eqref{part-sol-wave}, i.e. we set
		\begin{align*}
			\f14\tilde{w}(y) :=&\;  y^{2(\beta_1 \nu + \beta_2)} \int_{y_0}^y s^{-2(\beta_1 \nu + \beta_2) - 1} f(s)\;ds+ y^{2(\beta_1 \nu + \beta_2)-2} \int_{y_0}^y s^{-2(\beta_1 \nu + \beta_2) + 1} f(s)\;ds.
		\end{align*}
		We plug in the expansion \eqref{f-linear-scalar-exp}, for which integration by parts delivers \eqref{this-exp-scalar-wave} if $r \neq \beta_1$ (since $ \nu > 1$ is irrational). Otherwise the terms 
		\[
		s^{-2(\beta_1 \nu + \beta_2) - 1},\;\;s^{-2(\beta_1 \nu + \beta_2) +1},
		\]
		show that we obtain an additional  logarithmic power for the cases $ 2k - l -2\beta_2 \in \{-2, 0\}$ . Of course we also obtain \eqref{this-exp-scalar-wave} by comparing coefficients, i.e. we apply $L_{\beta_1, \beta_2}$ to the series \eqref{this-exp-scalar-wave} and thus we derive the formula
		\begin{align}
			&\Gamma_{k,l,r} \tilde{c}_{m,k} = c_{m,k},\\
			&\Gamma_{k,l,r} \tilde{c}_{m-1,k} + m(2(2k -l + 2\nu r) - (\beta_1 \nu + \beta_2 - \f12))\tilde{c}_{m,k}  = c_{m-1,k},\\
			&\Gamma_{k,l,r} \tilde{c}_{s,k} + s(2(2k -l + 2\nu r) - (\beta_1 \nu + \beta_2 - \f12))\tilde{c}_{s+1,k}\\ \nonumber
			&\hspace{2cm}  + \f14 (s+2)(s+1)\tilde{c}_{s+2,k}  = c_{s,k},\\
			& 0 \leq s \leq m-2,\nonumber
		\end{align}
		where
		\begin{align*}
			\Gamma_{k,l,r} =&\; \f14 (2k -l + 2\nu r)^2 - (2k -l + 2\nu r)(\beta_1\nu + \beta_2 - \f12)\\
			&\;\;\; + (\beta_1\nu + \beta_2 )(\beta_1\nu + \beta_2 -1).
		\end{align*}
		The latter quadratic term has roots $2k -l +2\nu r = 2(\beta_1 \nu + \beta_2) -1 \pm 1$, which are 'removed' by adding a logarithmic correction to \eqref{this-exp-scalar-wave}. 
	\end{proof}
	In particular we state the following for the system \eqref{inhom-inner-wave-first} - \eqref{inhom-inner-wave-last}.
	\begin{lem}\label{lem:inhom1} Let $\beta_1 \in \Z_+$ and $f(t, y)$ be a function admitting the representation of a finite sum
		\begin{align}
			f(t, y) = \sum_{\beta_2 \geq \beta_{2\star}}t^{\nu\beta_1 + \beta_2}\sum_{m = 0}^{m_{\star}}(\log (y) - \nu\log(t))^m f_{m, \beta_1, \beta_2}(y),\;\; 
		\end{align}
		where $m_{\star} = m_{\star}(\beta_1),\;\beta_{2\star} = \beta_{2\star}(\beta_1)$ and  $f_{m, \beta_1, \beta_2}(y),\; y \in (0, \infty)$ are smooth  coefficients. Further we require an absolute expansion
		\begin{align}\label{expansion-for-f}
			%	f_{m, \beta_1, \beta_2}(y) = \sum_{L = 0}^K\sum_{s = 0}^{m_{\star}-m} y^{2(L -K)} \big(\log(y)\big)^{s + L} \sum_{k \geq 0} y^{2k - (\gamma + 2 \beta_2)} \cdot y^{2\nu r} c^{m, L}_{k,s, \beta_1, \beta_2},\; 0 < y \ll 1,
			f_{m, \beta_1, \beta_2}(y) = \sum_{L = 0}^K\sum_{s = 0}^{m_{\star}-m} y^{2(L -K)} \big(\log(y)\big)^{s + L} \mathcal{O}(y^{- (\gamma + 2 \beta_2)} \cdot y^{2\nu r}),\; 0 < y \ll 1,
		\end{align} 
		where $ K, r \in \Z_{\geq 0},\; \gamma  \in \Z$ are fix depending on  $ \beta_1 \in \Z_{+}$.
		Then the inhomogeneous system \eqref{inhom-wave1}, i.e.
		\begin{equation}\label{eq:TheEquation1}
			\big(-t^ 2 \cdot (\partial_t -\frac12 t^{-1}y\partial_y)^2 + t \cdot (\partial_{y}^2 + \frac{3}{y}\partial_y)\big) n(t,y) = \big(\partial_y^2 + \frac{3}{y}\partial_y\big)f(t,y) 
		\end{equation}
		has a solution $n(t, y)$ of the form 
		\[
		n(t, y) = \sum_{\beta_2 \geq \beta_{2\star}} t^{\beta_1 \nu+\beta_2} \sum_{m = 0}^{m_{\star}}(\log(y) - \nu\log(t))^m h_{m,\beta_1,\beta_2}(y),
		\]
		such that $ h_{m,\beta_1,\beta_2}(y) =  h^{(hom)}_{m,\beta_1,\beta_2}(y) + h^{(part)}_{m,\beta_1,\beta_2}(y) $ are smooth functions with 
		\begin{align}\label{asympt-inhom-wave} \nonumber
			h^{(hom)}_{m,\beta_1,\beta_2}(y) =&\;\sum_{s = 0}^{m_{\star} - m}(\log(y))^{s}  \sum_{k =0}^{\beta_2 - \beta_{2\star}} \big(c^{(1)}_{\beta_2, k, s, s+m} y^{2(\beta_1 \nu + \beta_2 - 2k)}  + c^{(2)}_{\beta_2 , k,s,  s+m} y^{2(\beta_1 \nu + \beta_{2} - 2k) -2} \big).\\ 
			%h^{(part)}_{m,\beta_1,\beta_2}(y) =&\; \sum_{L = 0}^K\sum_{s = 0}^{m_{\star}-m} y^{2(L -K -1)} \big(\log(y)\big)^{s + L} \sum_{k \geq 0} y^{2k - (\gamma + 2 \beta_2)} \cdot y^{2\nu r} \tilde{c}^{m, L}_{k,s, \beta_1, \beta_2}, \;\; 0 < y \ll1.
			h^{(part)}_{m,\beta_1,\beta_2}(y) =&\; \sum_{L = 0}^K\sum_{s = 0}^{m_{\star}-m} y^{2(L -K -1)} \big(\log(y)\big)^{s + L} \mathcal{O}(y^{- (\gamma + 2 \beta_2)} \cdot y^{2\nu r}), \;\; 0 < y \ll1.
		\end{align}
		where the latter holds %in the absolute sense 
		if $ r \neq \beta_1 $ and if $ K >0$ or $ \gamma - 2\beta_2 \neq 0$ . If $ r = \beta_1 $ , then we need to replace $m_{\star}-m$ by $m_{\star} -m +1$ (in the second line) and in case $  \gamma + 2 \beta_2 = K =0$, then the factor $ y^{K -L -1}$ is replaced by a constant. Furthermore, the coefficients $ c^{(1)}_{\beta_2,0,0, m},\; c^{(2)}_{\beta_2, 0,0, m} $ are free to choose and uniquely determine the solution $n(t,y)$. %if each $i_{a,\gamma_1,\gamma_2}$ admits an absolutely convergent expansion (where the prime indicates that $\mu$ gets summed over a finite range of nonnegative values)
		%\[
		%i_{a,\gamma_1,\gamma_2}(y) = \sum_{\delta_1\geq 0, \delta_2>-\delta_{2*}}' c_{\delta_1\delta_2\mu}(\log y)^\mu y^{\delta_1\nu + \delta_2},
		%\]
		%then so does each $h_{a,\beta_1,\beta_2}$. 
	\end{lem}
	\begin{Rem}
		We write $ c^{(i)}_{\beta_2, k, s, s+m} $ instead of the coefficients  $ d^{(i)}_{\beta_2 - \beta_{2\star}, k, s+m} c_{s,m}$ displayed in \eqref{hom-wave-formula} in Corollary \ref{cor:hom-wave}. If we require 	$h_{m,\beta_1,\beta_2}(y)$ to have  the same type of expansion  as the source term with $y^{2\nu r},\; r \neq \beta_2$, then $  c^{(1)}_{\beta_2,0,0, m} = c^{(2)}_{\beta_2, 0,0, m}  = 0$ are fixed and there is a unique such solution.
	\end{Rem}
	\begin{proof} Solving \eqref{eq:TheEquation1} with the given form of $n(t,y)$ and $f(t,y)$ leads to a sequence of systems  \eqref{inhom-inner-wave-first} - \eqref{inhom-inner-wave-last} parameterized over $\beta_1, \beta_2 = \beta_{2\star} + \ell$.  In particular, starting with the case $ \beta_2 = \beta_{2\star}$ and $m = m_{\star}$ we have
		\[
		L_{\beta_1, \beta_{2\star}} h_{m_{\star}, \beta_1, \beta_{2\star}}(y) = - (\partial_y^2 + \frac{
			3}{y}\partial_y)f_{m_{\star}, \beta_1, \beta_{2\star}}(y).
		\]
		Now we use Lemma \ref{lem:das-einfache-welle-scalar-Lemma} with the given expansion \eqref{expansion-for-f}, which implies the claim. The second line \eqref{inhom-inner-wave-second} reads
		\begin{align*}
			L_{\beta_1, \beta_{2\star}} h_{m_{\star}-1, \beta_1, \beta_{2\star}}(y) =&\; - (\partial_y^2 + \frac{
				3}{y}\partial_y)f_{m_{\star}-1, \beta_1, \beta_{2\star}}(y) + m_{\star} \big(\frac{2}{y}\partial_y + \frac{2}{y^2}\big) h_{m_{\star}, \beta_1, \beta_{2\star}}(y)\\ \nonumber
			& \;\;\;\; + 2m_{\star} (\nu + \f12)\big(\beta_1\nu + \beta_{2\star} - \f12 \Lambda\big)h_{m_{\star}, \beta_1,\beta_{2\star}}(y),
		\end{align*}
		which involves two correction terms on the right. Clearly their expansion is consistent with the required  asymptotic expression \eqref{asympt-inhom-wave} and therefore an application of Lemma \ref{lem:das-einfache-welle-scalar-Lemma} delivers the claim. We note here by the special case $r = \beta_1$ in Lemma \ref{lem:das-einfache-welle-scalar-Lemma},  it is the operator
		\[
		2m_{\star} (\nu + \f12)\big(\beta_1\nu + \beta_{2\star}  - \f12 \Lambda\big),
		\]
		which introduces the logarithmic power in the homogeneous solution (as already remarked below Lemma \ref{lem:hom1}). The same holds true for the latter type of correction in the subsequent $(m-2)$ steps \eqref{inhom-inner-wave-last}. Now we iterate this over 
		$$ \beta_2 = \beta_{2\star}, \; \beta_{2\star} +1,\; \beta_{2\star} +2,\; \beta_{2\star} +3, \dots,\; \beta_{2\star} +N,$$
		by  simple induction steps and which requires to consider the new 
		$$ h_{m, \beta_1, \beta_2 -1}(y),\;\;  h_{m+1, \beta_1, \beta_2 -1}(y),\;\;  h_{m+2, \beta_1, \beta_2 -1}(y)$$ correction terms. Precisely, by induction we have
		\begin{align*}
			(\partial_y^2 + 2 y^{-1}\partial_y) h^{(part)}_{m,\beta_1,\beta_2-1}(y) =&\; \sum_{L = 0}^K\sum_{s = 0}^{m_{\star}-m} y^{2(L -K -2)} \big(\log(y)\big)^{s + L} \mathcal{O}(y^{2- (\gamma + 2 \beta_2)} \cdot y^{2\nu r}),\\
			2 y^{-2}\Lambda h^{(part)}_{m+1,\beta_1,\beta_2-1}(y) =&\; \sum_{L = 0}^K\sum_{s = 0}^{m_{\star}-m-1} y^{2(L -K -2)} \big(\log(y)\big)^{s + L} \mathcal{O}(y^{2- (\gamma + 2 \beta_2)} \cdot y^{2\nu r}),\\
			y^{-2} \cdot h^{(part)}_{m+2,\beta_1,\beta_2-1}(y) =&\; \sum_{L = 0}^K\sum_{s = 0}^{m_{\star}-m-2} y^{2(L -K -2)} \big(\log(y)\big)^{s + L} \mathcal{O}(y^{2- (\gamma + 2 \beta_2)} \cdot y^{2\nu r}).
		\end{align*}
		All of these expansions are included in the asymptotic form of $\triangle f_{m, \beta_1, \beta_2}(y)$ near $y =0$. Hence yet again Lemma \ref{lem:das-einfache-welle-scalar-Lemma} applies.\\[2pt]
		At last, we note the homogeneous part $h^{(hom)}_{m, \beta_1, \beta_2}(y)$ obtained along the induction is given by Lemma \ref{lem:hom1}. Further, similar to the proof of Lemma \ref{lem:das-einfache-welle-scalar-Lemma} and Lemma \ref{lem:hom1}, a direct comparison argument for the series representation also verifies the expansion of the particular part in \eqref{asympt-inhom-wave}.

	\end{proof}
	\begin{Rem} We note the particular `solutions' in the previous Lemma are of course also only approximate, to be precise truncating  $\sum_{\beta_2 = \beta_{2\star}}^L$ with some $L = L(\beta_1) \in \Z_{+}$ there holds
		\[
		t^2 \cdot \Box_S n_L(t,y) - \Delta_y f_L(t,y) = t^{\beta_1 \nu + L +1}\Delta_y \sum_{m = 0}^{m_{\star}} (\log(y) - \nu \log(t))^m h_{m, \beta_1, L+1}(y).
		\]
	\end{Rem}
	For a precise description of the solutions in Lemma \ref{lem:inhom1} and the %resulting 
	coefficients in  the case $ K = 0$, we  consider plugging in the ansatz
	\[
	h_{m, \beta_1, \beta_2} (y) = \sum_{s = 0}^{m_{\star} - m} \log^s(y)h_{s, m, \beta_1, \beta_2} (y),\;\;f_{m, \beta_1, \beta_2} (y) = \sum_{s = 0}^{m_{\star} - m} \log^s(y)f_{s, m, \beta_1, \beta_2} (y),
	\]
	into \eqref{inhom-inner-wave-first} - \eqref{inhom-inner-wave-last}, which reduces to the systems
	\boxalign[14cm]{
		\begin{align} \label{inhom-inner-wave-first-reduced}
			L_{\beta_1, \beta_2} h_{0, m_{\star}, \beta_1,\beta_2}(y) &=  \big( \partial_y^2 + \frac{3}{y} \partial_y \big) \big(h_{0, m_{\star}, \beta_1, \beta_2-1}(y) - f_{0, m_{\star}, \beta_1,\beta_2}(y)\big),\\ \label{inhom-inner-wave-second-reduced}
			L_{\beta_1, \beta_2} h_{1, m_{\star}-1, \beta_1,\beta_2}(y) &=  \big( \partial_y^2 + \frac{3}{y} \partial_y \big) \big(h_{1, m_{\star}-1, \beta_1, \beta_2-1}(y) - f_{1, m_{\star}-1, \beta_1,\beta_2}(y)\big),\\ 
			L_{\beta_1, \beta_2} h_{0, m_{\star}-1, \beta_1,\beta_2}(y) &= \big( \partial_y^2 + \frac{3}{y} \partial_y \big) \big(h_{0, m_{\star}-1, \beta_1, \beta_2-1}(y) - f_{0, m_{\star} -1, \beta_1,\beta_2}(y)\big)\\ \nonumber
			& \;\; + \widetilde{ L}_{\beta_1, \beta_2}  h_{1, m_{\star} -1, \beta_1, \beta_2}(y)\\ \nonumber
			& + \big(\frac{2}{y} \partial_y + \frac{2}{y^2}\big) \big(h_{1, m_{\star} -1, \beta_1, \beta_2-1}(y) - f_{1, m_{\star} -1, \beta_1, \beta_2}(y)\big)\\ \nonumber
			&\;\;\;\; + 2m_{\star} (\nu + \f12)\big(\beta_1\nu + \beta_2 - \f12 \Lambda\big)h_{0, m_{\star}, \beta_1,\beta_2}(y)\\ \nonumber
			& \;\;\;\; + m_{\star} \big(\frac{2}{y}\partial_y + \frac{2}{y^2}\big) \big(h_{0, m_{\star}, \beta_1, \beta_2-1}(y) - f_{0, m_{\star}, \beta_1, \beta_2}(y)\big)
	\end{align}}
	where $\widetilde{ L}_{\beta_1, \beta_2} = (\nu \beta_1 + \beta_2 - \f12\Lambda)$ and further 
	\boxalign[16cm]{
		\begin{align}\label{inhom-inner-wave-last-reduced}
			L_{\beta_1, \beta_2} h_{m_{\star} -m, m, \beta_1,\beta_2}(y) &= \big( \partial_y^2 + \frac{3}{y} \partial_y \big) \big(h_{m_{\star} -m, m, \beta_1, \beta_2-1}(y) - f_{m_{\star} -m, m, \beta_1, \beta_2}(y) \big)\\ 
			L_{\beta_1, \beta_2} h_{m_{\star} -m-1, m, \beta_1,\beta_2}(y)&= \big( \partial_y^2 + \frac{3}{y} \partial_y \big) \big(h_{m_{\star} -m-1, m, \beta_1, \beta_2-1}(y) - f_{m_{\star} -m-1, m, \beta_1, \beta_2}(y) \big)\\ \nonumber
			&\;\; + (m_{\star} -m)\widetilde{L}_{\beta_1, \beta_2} h_{m_{\star}-m, m,\beta_1, \beta_2}(y)\\\nonumber
			&\;\; +  (m_{\star} -m)\big(\frac{2}{y} \partial_y + \frac{2}{y^2}\big)\big(h_{m_{\star}-m, m,\beta_1, \beta_2-1}(y) - f_{m_{\star}-m, m,\beta_1, \beta_2}(y)\big)\\\nonumber
			&\;\; + (m+1) \big(\frac{2}{y} \partial_y + \frac{2}{y^2}\big)\big(h_{m_{\star}-m-1,m+1,\beta_1, \beta_2-1}(y) - f_{m_{\star}-m-1, m+1,\beta_1, \beta_2}(y)\big)\\ \nonumber
			&\;\;+ 2(m+1)(\nu + \f12)(\beta_1 \nu + \beta_2 - \f12 \Lambda) h_{m_{\star}-m-1,m+1,\beta_1, \beta_2}(y)\\[6pt] 
			L_{\beta_1, \beta_2} h_{s, m, \beta_1,\beta_2}(y) &= \big( \partial_y^2 + \frac{3}{y} \partial_y \big) \big(h_{s, m, \beta_1, \beta_2-1}(y) - f_{s, m, \beta_1, \beta_2}(y) \big)\\ \nonumber
			&\;\; + (s+1)\widetilde{L}_{\beta_1, \beta_2} h_{s+1, m,\beta_1, \beta_2}(y) + F_{s,m}(h,f)\\\nonumber
			&\;\; +  (s+1)\big(\frac{2}{y} \partial_y + \frac{2}{y^2}\big)\big(h_{s+1, m,\beta_1, \beta_2-1}(y) - f_{s+1 , m,\beta_1, \beta_2}(y)\big)\\\nonumber
			&\;\;+ (m+1) \big(\frac{2}{y} \partial_y + \frac{2}{y^2}\big)\big(h_{s,m+1,\beta_1, \beta_2-1}(y) - f_{s, m+1,\beta_1, \beta_2}(y)\big)\\\nonumber
			&\;\; + 2(m+1)(\nu + \f12)(\beta_1 \nu + \beta_2 - \f12 \Lambda) h_{s,m+1,\beta_1, \beta_2}(y)\\[15pt] \nonumber
			&\;\;0 \leq m \leq m_{\star}-2,\;\; 0 \leq s \leq m_{\star} -m -2.
		\end{align}
	}
	where
	\begin{align*}
		F_{s,m}(h,f) =& - \f14 (s+1)(s+2) h_{s+2, m, \beta_1, \beta_2}(y) +(s+1)(s+2) \frac{1}{y^2} \big(h_{s+2,m,\beta_1, \beta_2-1}(y) + f_{s+2, m,\beta_1, \beta_2}(y)\big)\\\nonumber
		&\;\; + \frac{s+1}{y^2} \big(h_{s+1,m+1,\beta_1, \beta_2-1}(y) + f_{s+1, m+1,\beta_1, \beta_2}(y)\big)\\\nonumber
		&\;\;- (s+1) h_{s+1, m+1,\beta_1, \beta_2}(y) + (m+2)(m+1) \frac{1}{y^2} \big( h_{s, m+2, \beta_1, \beta_2-1}(y) -  f_{s, m +2, \beta_1,\beta_2}(y)\big)\\ \nonumber
		&\;\;- (m+2)(m+1) (\f12 + \nu)^2 h_{s, m+2, \beta_1, \beta_2}(y).
	\end{align*}
	From this system, we consider in particular the case $ r =0$ and directly observe by induction the following Lemma, which becomes useful for describing pure Schr\"odinger approximations `passing through' the wave interaction in the second line of \eqref{schrod-wave-y}.
	\begin{lem}\label{lem:inhom1-update} Let $\beta_1 \in \Z_+$ and $f(t, y)$ be a function admitting the representation of a finite sum
		\begin{align}
			f(t, y) = \sum_{\beta_2 \geq \beta_{2\star}}t^{\nu\beta_1 + \beta_2}\sum_{m = 0}^{m_{\star}}(\log (y) - \nu\log(t))^m f_{m, \beta_1, \beta_2}(y),\;\; 
		\end{align}
		where $m_{\star} = m_{\star}(\beta_1),\;\beta_{2\star} = \beta_{2\star}(\beta_1)$ and  $f_{m, \beta_1, \beta_2}(y),\; y \in (0, \infty)$ are smooth  coefficients. Further we have
		\begin{align}\label{expansion-for-f-update}
			%	f_{m, \beta_1, \beta_2}(y) = \sum_{L = 0}^K\sum_{s = 0}^{m_{\star}-m} y^{2(L -K)} \big(\log(y)\big)^{s + L} \sum_{k \geq 0} y^{2k - (\gamma + 2 \beta_2)} \cdot y^{2\nu r} c^{m, L}_{k,s, \beta_1, \beta_2},\; 0 < y \ll 1,
			&f_{m, \beta_1, \beta_2}(y) = \sum_{s = 0}^{m_{\star}-m} \log^s(y) \sum_{k \geq 0}y^{2k - \gamma - 2\beta_2}c_{k,s,m} ,\; 0 < y \ll 1,\\ \nonumber
			&c_{k,m_{\star} -m,m} = 0,\;\;\; 2k <  \gamma + 2\beta_2 - 2\beta_{2\star},\;\;\;\; c_{k,m_{\star} -m -1,m} = 0,\;\;\; 2k <  \gamma + 2\beta_2 - 2\beta_{2\star}-2
		\end{align} 
		where $\gamma \in \Z_{\geq 0}$ is even and depending on  $ \beta_1 \in \Z_{+}$.
		Then the inhomogeneous system \eqref{inhom-wave1} has a unique solution $n(t, y)$ of the form 
		\[
		n(t, y) = \sum_{\beta_2 \geq \beta_{2\star}} t^{\beta_1 \nu+\beta_2} \sum_{m = 0}^{m_{\star}}(\log(y) - \nu\log(t))^m h_{m,\beta_1,\beta_2}(y),
		\]
		such that $ h_{m,\beta_1,\beta_2}(y) $ are smooth functions with 
		\begin{align}\label{asympt-inhom-wave-update} 
			%h^{(part)}_{m,\beta_1,\beta_2}(y) =&\; \sum_{L = 0}^K\sum_{s = 0}^{m_{\star}-m} y^{2(L -K -1)} \big(\log(y)\big)^{s + L} \sum_{k \geq 0} y^{2k - (\gamma + 2 \beta_2)} \cdot y^{2\nu r} \tilde{c}^{m, L}_{k,s, \beta_1, \beta_2}, \;\; 0 < y \ll1.
			h_{m,\beta_1,\beta_2}(y) =&\; \sum_{s = 0}^{m_{\star}-m} \log^s(y) \sum_{k \geq 0} y^{2k -\gamma - 2\beta_2} \tilde{c}_{k,s,m}, \;\; 0 < y \ll1.
		\end{align}
		where 
		\begin{align*}
			& \tilde{c}_{k,m_{\star} -m,m} = 0,\;\; 2k  < \gamma +2\beta_2 - 2\beta_{2\star},\;\;\;\;\tilde{c}_{k,m_{\star} -m-1,m} = 0,\;\; 2k  < \gamma +2\beta_2 - 2\beta_{2\star} -2,\\
			& \tilde{c}_{k,s,m} = 0,\;\; 2k  < -2,\;\; 0 \leq s \leq m_{\star} - m -2.
		\end{align*} %if each $i_{a,\gamma_1,\gamma_2}$ admits an absolutely convergent expansion (where the prime indicates that $\mu$ gets summed over a finite range of nonnegative values)
		%\[
		%i_{a,\gamma_1,\gamma_2}(y) = \sum_{\delta_1\geq 0, \delta_2>-\delta_{2*}}' c_{\delta_1\delta_2\mu}(\log y)^\mu y^{\delta_1\nu + \delta_2},
		%\]
		%then so does each $h_{a,\beta_1,\beta_2}$. 
	\end{lem}
	\begin{proof}
		We carry out a straight forward induction over $m_{\star} -m,\; 0 \leq s \leq m$ and for which we  construct unique particular solutions for each of the equations \eqref{inhom-inner-wave-first-reduced} - \eqref{inhom-inner-wave-last-reduced} given by the integral in \eqref{part-sol-wave}. Here of course the leading order of the expansion as $ y \to 0^+$ will remain as we have in the source terms. For the property $ h_{m_{\star}-m,m,\beta_1, \beta_2}(y) = \mathcal{O}(1),\;\; h_{m_{\star}-m-1,m,\beta_1, \beta_2}(y) = \mathcal{O}(y^{-2})$ we note that $ \{1, y^{-2} \}$ is a fundamental base for $ \Delta_y = \partial_y^2 + \frac{3}{y}\partial_y$.
	\end{proof}
	%We call the preceding solution of \eqref{eq:TheEquation1} the {\it{canonical solution}}. The ambiguity of this solution is then described in terms of Lemma~\ref{lem:hom}. 
	\;\;\\
	Similarly, we now obtain an analogous inhomogeneous system, see \cite{Perelman}, \cite{OP}  and \cite{schmid}, for the Schr\"odinger part
	\begin{align}\label{schrod-inhom}
		i t \partial_t w(t,y) + (\mathcal{L}_S - \alpha_0)w(t,y) = \tilde{f}(t,y),
	\end{align}
	i.e. considering
	\begin{align}\label{inhom-schrod-general}
		&(it \partial_t + \mathcal{L}_S - \alpha_0)\bigg(  \sum_{\alpha_1 \alpha_2}t^{\nu \alpha_1+\alpha_{2}} \sum_{m = 0}^{m_{\star}}(\log(y) - \nu\log(t))^m  g_{m, \alpha_1,\alpha_2}(y) \bigg) = \tilde{f}(t,y)\\[5pt] \nonumber
		&\tilde{f}(t,y) = \sum_{\alpha_1,\alpha_2}t^{\nu \alpha_1+\alpha_2} \sum_{m = 0}^{m_{\star}}(\log(y) - \nu\log(t))^m  \tilde{f}_{m, \alpha_1,\alpha_2}(y),
	\end{align}
	we have for the coefficients by comparison
	\boxalign[11cm]{
		\begin{align}\label{inhom-inner-sys-first}
			&\big(\mathcal{L}_S + \mu_{\alpha_1, \alpha_2})g_{m_{\star}, \alpha_1,\alpha_2}(y) = \tilde{f}_{m_{\star}, \alpha_1,\alpha_2}(y)\\[6pt] \label{inhom-inner-sys-second}
			&\big(\mathcal{L}_S + \mu_{\alpha_1, \alpha_2})g_{m_{\star}-1, \alpha_1,\alpha_2 }(y) = - m_{\star} D_y g_{m_{\star}, \alpha_1,\alpha_2}(y)\\ \nonumber
			& \hspace{4.2cm}\;\;\;\; + \tilde{f}_{m_{\star} -1, \alpha_1,\alpha_2}(y),\\[6pt] \label{inhom-inner-sys-last}
			& \big(\mathcal{L}_S + \mu_{\alpha_1, \alpha_2})g_{m, \alpha_1,\alpha_2 }(y) = - (m+1) D_y g_{m +1, \alpha_1,\alpha_2}(y)\\[6pt] \nonumber
			& \hspace{4.2cm} - (m+2)(m+1) \frac{1}{y^2} g_{m +2, \alpha_1,\alpha_2}(y)\\ \nonumber
			& \hspace{4.2cm} + \tilde{f}_{m, \alpha_1, \alpha_2}(y),\\[4pt] \nonumber
			& 0 \leq m \leq m_{\star} -2,
		\end{align}
	}
	and for which we recall 
	\[
	D_y = \big(\tfrac{2}{y^2} - i(\tfrac12 + \nu)\big) + \tfrac{2}{y} \partial_y  = \frac{2}{y^2}\Lambda- i(\tfrac12 + \nu),\;\;\;\;\mu_{\alpha_1, \alpha_2} = i(\nu \alpha_1 + \alpha_2) - \alpha_0.
	\]
	Similar to the above Lemma we infer the following. 
	\begin{lem}\label{lem:das-simple-lemma-schrodinger-skalar}
		Let $\mu \in \mathbf{C}$, then we consider the equation
		\begin{align}\label{inhom-schrod-ode-scalar}
			(\mathcal{L}_S + \mu)w(y) = \tilde{f}(y),\;\;y \in(0,\infty).
		\end{align} 
		Let $ \tilde{f}(y)$ be in $C^{\infty}((0, \infty))$ and such that in an absolute sense
		\begin{align}\label{f-linear-scalar-exp-schrod}
			\tilde{f}(y) = \sum_{s =0}^m \sum_{k \geq 0} \big(\log(y)\big)^{s} y^{2\nu r + 2k-l} c_{s,k},\;\; 0 < y \ll1,
			%f(y) = \sum_{s =0}^m \big(\log(y)\big)^{s} \mathcal{O}(y^{2\nu r -l}),\;\; 0 < y \ll1,
		\end{align}
		where $m , r,l \in \Z_{\geq 0}$ are fix. Then all solutions $w(y)$ of \eqref{inhom-schrod-ode-scalar} satisfy
		$$ w(y) =  c_0 e_{0}^{(1)}(y)  + c_1 e_{0}^{(2)}(y)  + \tilde{w}(y)$$
		where $\tilde{w}\in C^{\infty}((0,y))$ with an absolute expansion for $0 < y \ll1$
		\begin{align}\label{this-exp-scalar-schrod}
			\tilde{w}(y) = \sum_{s =0}^m &\big(\log(y)\big)^{s}\sum_{k \geq0}y^{2\nu r + 2k+2 -l}c^{(1)}_{k,s} + \sum_{k \geq 0}\big(\log(y)\big)^{m+1}\cdot y^{2\nu r + 2k+4 -l}c^{(2)}_{k}.
			%\tilde{w}(y) = \sum_{s =0}^m \big(\log(y)\big)^{s} \mathcal{O}(y^{2\nu r -l}) ,\;\; 0 < y \ll1,
		\end{align}
		if $ r \neq 0$. In case the latter is not satisfied, we need to replace 
		$$ \tilde{w}(y) = \big(\log(y)\big)^{m+2} \mathcal{O}(1) + \tilde{w}_0(y),$$
		where $ \tilde{w}_0$ has the expansion \eqref{this-exp-scalar-schrod}. The coefficients $c_0, c_1$ uniquely determine $w(y)$ and the dependence on $\mu \in \C$ is smooth.
	\end{lem}
	\begin{proof}
		Similar to the above calculation of the homogeneous solutions in Lemma \ref{lem:smally1}, the particular solutions are given by
		\[
		\tilde{w}(y) =  \int_{y_0}^y w^{-1}(s) G_0(s, y) f(s) \;ds,
		\]
		where $ G_0( s, y) =  e^{(1)}_0( y) e^{(2)}_0( s) - e^{(2)}_0( y) e^{(1)}_0( s)$ is the Greens function. Thus we expand
		\begin{align*}
			w(y) =  &(1 + \mathcal{O}(y^{2}))\int_{y_0}^y \mathcal{O}(s^3)(\mathcal{O}(s^{-2}) + \mathcal{O}(1) \log(s) )f(s) \;ds,\\
			& - (\mathcal{O}(y^{-2}) + \mathcal{O}(1) \log(y))\int_{y_0}^y \mathcal{O}(s^3)f(s) \;ds,\;\;0 < y_0 < y \ll1,
		\end{align*}
		by which we infer \eqref{this-exp-scalar-schrod} if $ r \neq 0$. If this is not the case we need to increase the logarithmic power, to be precise for coefficients with $2k = \gamma + 2\alpha_2 + 2(K-L) -4$. By comparison we also obtain \eqref{this-exp-scalar-schrod} via a recursion sequence for $ c^{(1)}_{k s}, c^{(2)}_k$.
		%\begin{align*}
		%\end{align*}
	\end{proof}
	For the solutions of  \eqref{inhom-schrod-general} with right side $ \tilde{f}(t,y)$ having a singular expansion at $ y = 0$, we find that solutions have the following general form.
	\begin{lem} \label{lem:inhom-schrod} Let $\alpha_1 \in \Z_+$ and $ \tilde{f}(t,y)$ be a function with a finite sum expression
		\[
		\tilde{f}(t,y) = \sum_{\alpha_2 \geq \alpha_{2\star}}t^{\nu \alpha_1+\alpha_2} \sum_{m = 0}^{m_{\star}}(\log(y) - \nu\log(t))^m  \tilde{f}_{m, \alpha_1,\alpha_2}(y),
		\]
		where $m_{\star} = m_{\star}(\alpha_1),\;\alpha_{2\star} = \alpha_{2\star}(\alpha_1)$ and  $\tilde{f}_{m, \alpha_1, \alpha_2}(y),\; y \in (0, \infty)$ are smooth  coefficients. Further we require an absolute expansion
		\begin{align}\label{expansion-for-f2}
			%\tilde{f}_{m, \alpha_1, \alpha_2}(y) = \sum_{L = 0}^K\sum_{s = 0}^{m_{\star}-m} y^{2(L -K)} \big(\log(y)\big)^{s + L} \sum_{k \geq 0} y^{2k - (\gamma + 2 \alpha_2)} \cdot y^{2\nu r} c^{m, L}_{k,s, \alpha_1, \alpha_2},\; 0 < y \ll 1,
			\tilde{f}_{m, \alpha_1, \alpha_2}(y) = \sum_{L = 0}^K\sum_{s = 0}^{m_{\star}-m} y^{2(L -K)} \big(\log(y)\big)^{s + L} \mathcal{O}(y^{- (\gamma + 2 \alpha_2)} \cdot y^{2\nu r}),\; 0 < y \ll 1,
		\end{align} 
		where $ K, r, \gamma \in \Z_{\geq 0}$.
		Then the inhomogeneous system \eqref{inhom-schrod-general}, i.e.
		\begin{equation}\label{eq:TheEquation2}
			it\partial_tw(t,y) - \big(\alpha_0 + \frac{i}{2}\Lambda \big)w(t,y) + (\partial_{y}^2 + \frac{3}{y}\partial_y)w(t,y) = f(t,y) 
		\end{equation}
		has a solution $w(t, y)$ of the form 
		\[
		w(t, y) = \sum_{\alpha_2 \geq \alpha_{2\star}} t^{\alpha_1 \nu+\alpha_2} \sum_{m = 0}^{m_{\star}}(\log(y) - \nu\log(t))^m g_{m,\alpha_1,\alpha_2}(y),
		\]
		such that $ g_{m,\alpha_1,\alpha_2}(y) =  g^{(hom)}_{m,\alpha_1,\alpha_2}(y) + g^{(part)}_{m,\alpha_1,\alpha_2}(y) $ are smooth functions and for $0 < y \ll1$
		\begin{align}\label{asympt-inhom-schrod}
			g^{(hom)}_{m,\alpha_1,\alpha_2}(y) =&\; \sum_{k \geq 0} \bigg(\sum_{s = 0}^{m_{\star} -m} y^{2k -2} (\log(y))^s c^{(m)}_{k,s, \alpha_1, \alpha_2} + \tilde{c}^{(m)}_{k,  \alpha_1, \alpha_2}y^{2k} (\log(y))^{m_{\star} -m +1}\bigg),\\ \label{asympt-inhom-schrod2}
			%	g^{(part)}_{m,\alpha_1,\alpha_2}(y) =&\; \sum_{L = 0}^K\sum_{s = 0}^{m_{\star}-m +1} y^{2(L -K +1)} \big(\log(y)\big)^{s + L} \sum_{k \geq 0} y^{2k - (\gamma + 2 \beta_2)} \cdot y^{2\nu r} \tilde{c}^{m, L}_{k,s, \alpha_1, \alpha_2}, \;\; 0 < y \ll1.
			g^{(part)}_{m,\alpha_1,\alpha_2}(y) =&\; \sum_{L = 0}^K y^{2(L -K)} \bigg( \sum_{s = 0}^{m_{\star}-m }  \big(\log(y)\big)^{s + L} \mathcal{O}(y^{2- (\gamma + 2 \beta_2)} \cdot y^{2\nu r})\\ \nonumber
			& \hspace{3cm}+ \big(\log(y)\big)^{m_{\star} -m +1 + L}  \cdot \mathcal{O}(y^{4- (\gamma + 2 \beta_2)} \cdot y^{2\nu r})\bigg),
		\end{align}
		where the latter holds %in the absolute sense 
		if $ r \neq 0 $. If this is false, then we need to replace 
		$$g^{(part)}_{m,\alpha_1,\alpha_2}(y) = \big(\log(y)\big)^{m_{\star} -m + 2 +K}\mathcal{O}(1) + \tilde{g}^{(part)}_{m,\alpha_1,\alpha_2}(y),$$
		where $\tilde{g}^{(part)}_{m,\alpha_1,\alpha_2}$ is as above in \eqref{asympt-inhom-schrod2}. 
		Furthermore, the coefficients $ c^{(m)}_{0,0,  \alpha_1, \alpha_2},\; c^{(m)}_{1,0,  \alpha_1, \alpha_2} $ are free to choose %(in fact they also change $c^{(m)}_{1,1,  \alpha_1, \alpha_2}$) 
		and uniquely determine the solution $w(t,y)$.	
		\begin{comment}
			and that
			\[
			f(t, y) = \sum'_{a,\gamma_1,\gamma_2}t^{\nu\gamma_1 + \gamma_2}(\log y - \nu\log t)^a\cdot i_{a,\gamma_1,\gamma_2}(y). 
			\]
			Then there exists 
			\[
			w(t, y) = \sum'_{a,\beta_1,\beta_2}t^{\nu\beta_1+\beta_2} (\log y - \nu\log t)^a\cdot h_{a,\beta_1,\beta_2}(y)
			\]
			with (we suppress the implicit dependance of the coefficients on $a,\beta_1, \beta_2$)
			\[
			h_{a,\beta_1,\beta_2}(y) =  \sum_{\delta_1\geq 0, \delta_2>-\delta_{2*}}' \tilde{c}_{\delta_1\delta_2\mu}(\log y)^\mu y^{\delta_1\nu + \delta_2}
			\]
			and such that 
			\begin{align*}
				-\alpha_0w - \frac{i}{2}w + itw_t - \frac{i}{2}yw_y + (\partial_{yy} + \frac{3}{y}\partial_y)w = f(t, y). 
			\end{align*}
			The coefficients $h_{a,\beta_1,\beta_2}(y)$ are unique up to a linear combination of two functions (depending on the parameters $a,\beta_1,\beta_2$) of the form \eqref{eq:e1}, \eqref{eq:e2}. 
		\end{comment}
	\end{lem}
	\begin{proof} 
		We start with the minimal case $\alpha_2 = \alpha_{2\star}$. The first line \eqref{inhom-inner-sys-first} implies the above expansion by Lemma \ref{lem:das-simple-lemma-schrodinger-skalar} and the coefficients are selected accordingly. Proceeding by induction  (let us assume for now $r \neq 0$) over $ 0 \leq m \leq m_{\star}$,  we need to expand the following terms on the right side of the lines below \eqref{inhom-inner-sys-first}
		\begin{align*}
			D_y g_{\tilde{m}+1, \alpha_1, \alpha_2}(y) =& \;\sum_{L = 0}^K y^{2(L -K)} \bigg( \sum_{s = 0}^{m_{\star}-(\tilde{m} +1) }  \big(\log(y)\big)^{s + L} \mathcal{O}(y^{- (\gamma + 2 \beta_2)} \cdot y^{2\nu r})\\ \nonumber
			& \hspace{3cm}+ \big(\log(y)\big)^{m_{\star} -\tilde{m} + L}  \cdot \mathcal{O}(y^{2- (\gamma + 2 \beta_2)} \cdot y^{2\nu r})\bigg),\\
			y^{-2}g_{\tilde{m}+2, \alpha_1, \alpha_2}(y) =& \;\sum_{L = 0}^K y^{2(L -K)} \bigg( \sum_{s = 0}^{m_{\star}-(\tilde{m}+2) }  \big(\log(y)\big)^{s + L} \mathcal{O}(y^{- (\gamma + 2 \beta_2)} \cdot y^{2\nu r})\\ \nonumber
			& \hspace{3cm}+ \big(\log(y)\big)^{m_{\star} -\tilde{m} -1 + L}  \cdot \mathcal{O}(y^{2- (\gamma + 2 \beta_2)} \cdot y^{2\nu r})\bigg).
		\end{align*}
		These asymptotic expressions are of course included in \eqref{asympt-inhom-schrod2}, precisely
		\[
		\tilde{f}_{\tilde{m}, \alpha_1, \alpha_2}(y) = \sum_{L = 0}^K\sum_{s = 0}^{m_{\star}-\tilde{m}} y^{2(L -K)} \big(\log(y)\big)^{s + L} \mathcal{O}(y^{- (\gamma + 2 \alpha_2)} \cdot y^{2\nu r}),\; 0 < y \ll 1,
		\]
		and hence we again obtain the claim by inverting the operator $\mathcal{L}_S + \mu_{\alpha_1, \alpha_{2\star}}$ via Lemma \ref{lem:das-simple-lemma-schrodinger-skalar}.\\[3pt]
		Further letting $ \alpha_2 > \alpha_{2 \star}$, we note these systems are solved \emph{independently} (unlike the wave system \eqref{inhom-inner-wave-first} - \eqref{inhom-inner-wave-last}). In fact we simply increase $\mu_{\alpha_1, \alpha_2 +1}$ which only modifies the $\mu$-dependence of the Greens function (respectively the fundamental base). Thus we observe the desired expansions and the case of $ r = 0$ follows similarly with the above logarithmic correction.
	\end{proof}
	\;\\
	Let us now consider more precisely how the coefficients in the above solutions are determined. In fact for $\tilde{s}_{\ast} = m_{\ast} -m$ in the conditions \eqref{exp1-new} and  \eqref{exp2-new}, we need to \emph{make sure the  matching coefficients from \eqref{exp1} and \eqref{exp2} are attained}, c.f. \cite{OP}, \cite[Lemma 2.21, Lemma 2.22]{schmid}, where \emph{no higher powers of  $\log(y)$ contribute  at $ y = 0$ to the solution than given by the source} expansion. Thus we state a more precise form of Lemma \ref{lem:inhom-schrod} in particular in the case of $ r = 0$.\\[4pt]
	We first reduce \eqref{inhom-inner-sys-first} - \eqref{inhom-inner-sys-last}  for  the ansatz
	\begin{align}
		\tilde{f}_{m, \alpha_1, \alpha_2}(y) &= \;\sum_{s = 0}^{m_{\star} -m} \log^s(y)\tilde{f}_{s, m, \alpha_1, \alpha_2}(y),\\
		g_{m,\alpha_1,\alpha_2}(y) &= \; \sum_{s = 0}^{m_{\star} -m} 	\log^s(y) g_{s, m,\alpha_1,\alpha_2}(y).
	\end{align}
	Thus, plugging this ansatz into  the system \eqref{inhom-inner-sys-first} - \eqref{inhom-inner-sys-last}, we infer by a straight forward calculation similar to the above the reduced system
	\boxalign[13cm]{
		\begin{align}\label{inhom-inner-sys-first-reduced}
			&\big(\mathcal{L}_S + \mu_{\alpha_1, \alpha_2})g_{0, m_{\star}, \alpha_1,\alpha_2}(y) = \tilde{f}_{0, m_{\star}, \alpha_1,\alpha_2}(y)\\[15pt] \label{inhom-inner-sys-second-reduced}
			&\big(\mathcal{L}_S + \mu_{\alpha_1, \alpha_2})g_{1, m_{\star}-1, \alpha_1,\alpha_2}(y) = \tilde{f}_{1, m_{\star}-1, \alpha_1,\alpha_2}(y),\\[4pt]\nonumber
			&\big(\mathcal{L}_S + \mu_{\alpha_1, \alpha_2})g_{0, m_{\star}-1, \alpha_1,\alpha_2 }(y) =   - \tilde{D}_y g_{1, m_{\star}-1, \alpha_1, \alpha_2} - m_{\star} D_y g_{0,m_{\star}, \alpha_1,\alpha_2}(y)\\ \nonumber
			& \hspace{4.2cm}\;\;\;\; + \tilde{f}_{0, m_{\star} -1, \alpha_1,\alpha_2}(y),
		\end{align}
	}
	for $ m= m_{\star}, m_{\star-1}$ as well as for the remaining cases
	\boxalign[13cm]{
		\begin{align} \label{inhom-inner-sys-last-reduced}
			& \big(\mathcal{L}_S + \mu_{\alpha_1, \alpha_2})g_{m_{\star} - m, m, \alpha_1,\alpha_2 }(y) =  \tilde{f}_{m_{\star} -m, m, \alpha_1, \alpha_2}(y),\\[4pt] \nonumber
			& \big(\mathcal{L}_S + \mu_{\alpha_1, \alpha_2})g_{m_{\star} - m -1, m, \alpha_1,\alpha_2 }(y) =   - (m_{\star} - m ) \tilde{D}_y g_{ m_{\star}-m,m, \alpha_1, \alpha_2}(y)\\ \nonumber
			& \hspace{5.2cm}\;\;\;\;- (m+1) D_y g_{m_{\star} - m -1,m+1,  \alpha_1,\alpha_2}(y)\\ \nonumber
			&\hspace{5.2cm}\;\;\;\; + \tilde{f}_{m_{\star} -m -1, m, \alpha_1, \alpha_2}(y)\\[4pt] \nonumber
			& \big(\mathcal{L}_S + \mu_{\alpha_1, \alpha_2})g_{s, m, \alpha_1,\alpha_2 }(y) =   - (s+1)\tilde{D}_y g_{ s+1,m, \alpha_1, \alpha_2}(y) - (m+1) D_y g_{s,m+1,  \alpha_1,\alpha_2}(y)\\ \nonumber
			& \hspace{4.2cm}\;\;\;\; - (m+1)(s+1) \frac{2}{y^2} g_{s+1 , m+1, \alpha_1, \alpha_2}(y)\\ \nonumber
			& \hspace{4.2cm}\;\;\;\;  - (s+2)(s+1) \frac{1}{y^2 } g_{s+2 , m, \alpha_1, \alpha_2}(y)\\ \nonumber
			&  \hspace{4.2cm}\;\;\;\;  - (m+2)(m+1) \frac{1}{y^2} g_{s , m+2, \alpha_1, \alpha_2}(y) + \tilde{f}_{s, m, \alpha_1, \alpha_2}(y),\\ \nonumber 
			& 0 \leq m \leq m_{\star} -2,\; 0 \leq s\leq m_{\star} - m -2,
		\end{align}
	}
	where we set
	\[
	\tilde{D}_y = \frac{2}{y } \partial_y + \frac{3}{y^2} - \frac{i}{2}.
	\]
	
	\begin{lem} \label{lem:inhom-schrod-practical} Let $\alpha_1 \in \Z_+$ and $ \tilde{f}(t,y)$ be a function with a finite sum expression
		\[
		\tilde{f}(t,y) = \sum_{\alpha_2 \geq \alpha_{2\star}}t^{\nu \alpha_1+\alpha_2} \sum_{m = 0}^{m_{\star}}(\log(y) - \nu\log(t))^m  \tilde{f}_{m, \alpha_1,\alpha_2}(y),
		\]
		where $m_{\star} = m_{\star}(\alpha_1),\;\alpha_{2\star} = \alpha_{2\star}(\alpha_1)$ and  $\tilde{f}_{m, \alpha_1, \alpha_2}(y),\; y \in (0, \infty)$ are smooth  coefficients. Further we require an absolute expansion
		\begin{align}\label{expansion-for-f2-2}
			%\tilde{f}_{m, \alpha_1, \alpha_2}(y) = \sum_{L = 0}^K\sum_{s = 0}^{m_{\star}-m} y^{2(L -K)} \big(\log(y)\big)^{s + L} \sum_{k \geq 0} y^{2k - (\gamma + 2 \alpha_2)} \cdot y^{2\nu r} c^{m, L}_{k,s, \alpha_1, \alpha_2},\; 0 < y \ll 1,
			\tilde{f}_{m, \alpha_1, \alpha_2}(y) =\sum_{s = 0}^{m_{\star} -m}\sum_{k \geq 0} \log^s(y)\cdot y^{2k - l- 2 \alpha_2} \cdot y^{2\nu r} c_{s, k,m} ,\; 0 < y \ll 1,
		\end{align} 
		where $ r, l \in \Z_{\geq 0}$. We assume $ l + 2\alpha_2 \geq 2m_{\star}$  and  $ c_{s, k,m} = 0$ if 
		$k < \tilde{l}+ \alpha_2 -m_{\star} +m + s$ in case $ r = 0,\;l = 2\tilde{l} \in 2\Z$ is even.
		Then the inhomogeneous system \eqref{inhom-schrod-general}, i.e.
		\begin{align}\label{eq:TheEquation2-2}
			&it\partial_tw(t,y) - \big(\alpha_0 + \frac{i}{2}\Lambda \big)w(t,y) + (\partial_{y}^2 + \frac{3}{y}\partial_y)w(t,y) = \tilde{f}(t,y),\\
			&w(t, y) = \sum_{\alpha_2 \geq \alpha_{2\star}} t^{\alpha_1 \nu+\alpha_2} \sum_{m = 0}^{m_{\star}}(\log(y) - \nu\log(t))^m g_{m,\alpha_1,\alpha_2}(y),
		\end{align}
		has a solution such that $ g_{m,\alpha_1,\alpha_2}(y) $ are smooth functions and for $0 < y \ll1$
		\begin{align}\label{asympt-inhom-schrod-2}
			g_{m,\alpha_1,\alpha_2}(y) =&\; \sum_{s = 0}^{m_{\star} -m}\sum_{k \geq 0} \log^s(y)\cdot y^{2k  -l -2 \alpha_2}\cdot y^{2\nu r} \tilde{c}_{s, k,m},
		\end{align}
		where  $\tilde{c}_{s,k,m} = 0$ if $ r = 0, \;l = 2\tilde{l}$ and $ k < \tilde{l}+ \alpha_2 -m_{\star} +m + s$. Further there holds the following:
		\begin{itemize}
			\item[$\bullet$] If $ l + 2\alpha_2 \geq 2m_{\star} + 2k_0$ for some $k_0 \in \Z_{\geq 0}$  and  $ c_{s, k,m} = 0$ with
			$k < \tilde{l}+ \alpha_2 -m_{\star} +m + s - k_0$ and $ s \leq m_{\star} - m -2$, then also $\tilde{c}_{s, k,m} = 0$ if $ k < \tilde{l} + \alpha_2 - m_{\star} +m + s - k_0 +1$.
			\item[$\bullet$] Such solutions satisfying an expansion of the form \eqref{asympt-inhom-schrod-2} are uniquely determined if $r \neq 0$ or $ l \in 2 \Z +1 $ is odd, or otherwise if we fix $ \tilde{c}_{s,  \tilde{l}+ \alpha_2 , 0} \in \mathbf{C},\; 0 \leq s \leq m_{\star} $ and $ \tilde{c}_{s, \tilde{l}+ \alpha_2 , 1}  \in \mathbf{C},\; 0 \leq s \leq m_{\star} -1$ (which are free to chose).
		\end{itemize}
		\begin{comment}
			and that
			\[
			f(t, y) = \sum'_{a,\gamma_1,\gamma_2}t^{\nu\gamma_1 + \gamma_2}(\log y - \nu\log t)^a\cdot i_{a,\gamma_1,\gamma_2}(y). 
			\]
			Then there exists 
			\[
			w(t, y) = \sum'_{a,\beta_1,\beta_2}t^{\nu\beta_1+\beta_2} (\log y - \nu\log t)^a\cdot h_{a,\beta_1,\beta_2}(y)
			\]
			with (we suppress the implicit dependance of the coefficients on $a,\beta_1, \beta_2$)
			\[
			h_{a,\beta_1,\beta_2}(y) =  \sum_{\delta_1\geq 0, \delta_2>-\delta_{2*}}' \tilde{c}_{\delta_1\delta_2\mu}(\log y)^\mu y^{\delta_1\nu + \delta_2}
			\]
			and such that 
			\begin{align*}
				-\alpha_0w - \frac{i}{2}w + itw_t - \frac{i}{2}yw_y + (\partial_{yy} + \frac{3}{y}\partial_y)w = f(t, y). 
			\end{align*}
			The coefficients $h_{a,\beta_1,\beta_2}(y)$ are unique up to a linear combination of two functions (depending on the parameters $a,\beta_1,\beta_2$) of the form \eqref{eq:e1}, \eqref{eq:e2}. 
		\end{comment}
	\end{lem}
	\begin{proof} 
		Let us consider the case where $ r \neq 0$ and we begin by solving the first line of \eqref{inhom-inner-sys-first-reduced}. Then we set $ \gamma = l + 2 \alpha_2 + 2\nu r$ and obtain the following recurrent formula by comparing coefficients
		\begin{align*}
			\begin{cases}
				\;\;\tilde{c}_{0, 0, m_{\star}} = 0&\\[3pt]
				\;\;(2k - \gamma) ( 2k -\gamma + 2) \tilde{c}_{0,k, m_{\star}} = c_{0, k-1, m_{\star}}&\\[6pt]
				\;\;\hspace{2cm} - \tilde{c}_{0, k-1, m_{\star}}(\mu - \frac{i}{2}\big(2k -1 -\gamma\big)),\;\; k \geq 1 &.
			\end{cases}
		\end{align*}
		We note in particular the factors on the right do not vanish (having $ r > 0$) since $\nu > 0$ is irrational. This is of course likewise the case if $ r = 0$ and $l \in 2 \Z + 1$ is an odd number (and technically always until $ 2k = \gamma -2$, which determines the first $ \lfloor  \gamma\slash 2 -1\rfloor$ coefficients).\\[3pt]
		We now obtain the expansion \eqref{asympt-inhom-schrod-2} with the coefficients  $\tilde{c}_{0, k, m_{\star}}$. We proceed by calculating
		\begin{align*}
			&D_y g_{s, m, \alpha_1, \alpha_2}(y) \sim  2(2k - \gamma +1) y^{2k-2 - \gamma} \tilde{c}_{s, k, m+1} - i(\f12 + \nu) \tilde{c}_{s,k, m+1} y^{2k - \gamma },\\
			&\tilde{D}_y g_{s, m, \alpha_1, \alpha_2}(y) \sim  (2(2k - \gamma) +3) y^{2k-2 - \gamma} \tilde{c}_{s+1, k, m} - \frac i2  \tilde{c}_{s+1,k, m} y^{2k - \gamma }.
		\end{align*}
		from which the subsequent recurrences from \eqref{inhom-inner-sys-second-reduced} are defined as follows
		\begin{align*}
			\begin{cases}
				\;\;\tilde{c}_{1, 0, m_{\star}-1} = \tilde{c}_{0, 0, m_{\star}-1} =  0,&\\[8pt]
				\;\;(2k - \gamma) ( 2k -\gamma + 2) \tilde{c}_{1, k, m_{\star}-1} = c_{1, k-1, m_{\star}-1}&\\[6pt]
				\;\;\hspace{2cm} - \tilde{c}_{1, k-1, m_{\star}-1}(\mu - \frac{i}{2}\big(2k -1 -\gamma\big)),\;\; k \geq 1&\\[8pt]
				\;\;(2k - \gamma) ( 2k -\gamma + 2) \tilde{c}_{0, k, m_{\star}-1} = c_{0, k-1, m_{\star}-1} - 2m_{\star} (2k -\gamma +1) \tilde{c}_{0, k, m_{\star}}&\\[6pt]
				\;\;\hspace{2cm} - \tilde{c}_{0, k-1, m_{\star}-1}(\mu - \frac{i}{2}\big(2k -1 -\gamma\big)) + m_{\star} i (\f12 + \nu) \tilde{c}_{0, k-1, m_{\star}},&\\[6pt]
				\;\;\hspace{2cm} - (2(2k-\gamma) +3)\tilde{c}_{1, k,m_{\star} -1} + \frac i 2 \tilde{c}_{1, k-1, m_{\star} -1}\;\; k \geq 1.&
			\end{cases}
		\end{align*}
		and for the general steps \eqref{inhom-inner-sys-last-reduced} we have
		\begin{align*}
			\begin{cases}
				\;\;\tilde{c}_{s, 0, m_{\star}-1} =  0,\;\; 0 \leq s \leq m_{\star} -m,&\\[8pt]
				\;\;(2k - \gamma) ( 2k -\gamma + 2) \tilde{c}_{m_{\star} -m, k, m}= c_{m_{\star} -m,k-1, m}  - \tilde{c}_{m_{\star} -m, k-1, m}(\mu - \frac{i}{2}\big(2k -1 -\gamma\big)), \;\; k \geq 1,&\\[8pt]
				\;\;(2k - \gamma) ( 2k -\gamma + 2) \tilde{c}_{m_{\star} -m-1, k, m}= c_{m_{\star} -m-1,k-1, m}  - \tilde{c}_{m_{\star} -m-1, k-1, m}(\mu - \frac{i}{2}\big(2k -1 -\gamma\big))&\\[6pt]
				\;\;\hspace{2cm} - 2(m+1) (2k -\gamma +1) \tilde{c}_{m_{\star} - m -1, k, m+1}  +  (m+1) i (\f12 + \nu) \tilde{c}_{m_{\star} - m -1, k-1, m+1}&\\[6pt]
				\;\;\hspace{2cm} - (2(2k -\gamma) + 3) \tilde{c}_{m_{\star} - m , k, m}  +  \frac i 2   \tilde{c}_{m_{\star} - m , k-1, m},\;\; k \geq 1, &\\[8pt]
				\;\;(2k - \gamma) ( 2k -\gamma + 2) \tilde{c}_{s, k, m}= c_{s,k-1, m}  - \tilde{c}_{s, k-1, m}(\mu - \frac{i}{2}\big(2k -1 -\gamma\big))&\\[6pt]
				\;\;\hspace{2cm} - 2(m+1) (2k -\gamma +1) \tilde{c}_{s, k, m+1}  +  (m+1) i (\f12 + \nu) \tilde{c}_{s, k-1, m+1}&\\[6pt]
				\;\;\hspace{2cm} - (2(2k -\gamma) + 3) \tilde{c}_{s+1, k, m}  +  \frac i 2   \tilde{c}_{s+1 , k-1, m}&\\[6pt]
				\;\;\hspace{2cm} - 2(m+1)(s+1) \tilde{c}_{s+1, k, m+1} ,&\\[6pt]
				\;\;\hspace{2cm} - (s+2)(s+1)\tilde{c}_{s+2, k, m}  - (m+2)(m+1)  \tilde{c}_{s , k, m+2},\;\; k\geq1,&\\[6pt]
				\;\;\hspace{2cm} 0 \leq m \leq m_{\star} - 2,\;\; 0 \leq s \leq m_{\star} - m -2 &.
			\end{cases}
		\end{align*}
		By comparing the polynomial contributions of $ k \in \Z_+$ on both sides, the series \eqref{asympt-inhom-schrod-2} converges absolutely at least on the domain of $\tilde{f}_{m, \alpha_1, \alpha_2}(y)$. Further, from the left side we obtain that the system does not degenerate  if $ r \neq 0$ or $ l \in \Z_+$ is odd. Now let us consider the case where $ r = 0$ and $ l = 2\tilde{l}$ is an even number. We  therefore recall the calculation in \eqref{Green-int2}  with $ y_0 = 0$
		\begin{align} \nonumber
			\tilde{w}(y)  =& \;\;e^{(1)}_0(y) \int_0^y w^{-1}(s) \mathcal{O}(s^{-2}) f(s) \;ds - \mathcal{O}(y^{-2}) \int_0^y w^{-1}(s)e^{(1)}_0(s) f(s) \;ds\\ \label{another}
			&\; - e^{(1)}_0(y) \int_{0}^y F(s) s^{-1}\;ds
		\end{align}
		where $ F(s)$ is a primitive of $ w^{-1}(s)e^{(1)}_0(s) f(s) = \mathcal{O}(s^3) f(s)$ as  $ 0 < s \ll1$. %For further reference we say coefficients $c_1, c_2,\dots, c_j$ are linearly dependent if $ c_j =  b + a_1 c_1 + a_2 c_2 + \dots + a_{j-1} c_{j-1}$ with $ a_1,\dots a_{j-1},b \in \mathbf{C}$ and $ a_k\neq 0$.  
		Starting with \eqref{inhom-inner-sys-first-reduced}, by assumption we have $ \tilde{f}_{0, m_{\star}, \alpha_1, \alpha_2}(y) = \mathcal{O}(1)$ as $y \to 0^+$ and 
		\begin{align}
			g_{0, m_{\star}, \alpha_1, \alpha_2}(y) = &\;\; d_{0,0, m_{\star}} e^{(1)}_0(y) + \int_{0}^y w^{-1}(s) G(s,y)  \tilde{f}_{0, m_{\star}, \alpha_1, \alpha_2}(s)\;ds\\
			& =  d_{0,0,  m_{\star}} \big(1 + \mathcal{O}(y^2)\big) + \mathcal{O}(y^2),
		\end{align}
		implied by the above \eqref{another} and where we choose any $d_{0,0, m_{\star}} \in \mathbf{C}$. Continuing with  \eqref{inhom-inner-sys-second-reduced}, we obtain  similarly  that $ \tilde{f}_{1, m_{\star}-1, \alpha_1, \alpha_2}(y) = \mathcal{O}(1)$ as $y \to 0^+$ and thus we have
		\begin{align}
			g_{0, m_{\star}, \alpha_1, \alpha_2}(y) = &\;\; d_{1,0, m_{\star}-1} e^{(1)}_0(y) + \int_{0}^y w^{-1}(s) G(s,y)  \tilde{f}_{1, m_{\star}-1, \alpha_1, \alpha_2}(s)\;ds\\
			& =  d_{1,0,  m_{\star}-1} \big(1 + \mathcal{O}(y^2)\big) + \mathcal{O}(y^2),\;\; d_{1,0,  m_{\star}-1}  \in \mathbf{C}.
		\end{align} Clearly this implies
		\begin{align*}
			&D_y g_{0, m_{\star}, \alpha_1, \alpha_2}(y) =  y^{-2} \cdot 2 d_{0,0, m_{\star} } + \mathcal{O}(1),\;\; 0 < y \ll1,\\
			&\tilde{D}_y g_{1, m_{\star}-1, \alpha_1, \alpha_2}(y) =  y^{-2} \cdot 3 d_{1,0,m_{\star} } + \mathcal{O}(1),\;\; 0 < y \ll1,\\
			&\tilde{f}_{0, m_{\star}-1, \alpha_1, \alpha_2}(y) - m_{\star}D_y g_{0, m_{\star}, \alpha_1, \alpha_2}(y) - \tilde{D}_y g_{1, m_{\star}-1, \alpha_1, \alpha_2}(y)\\
			&\hspace{3cm} = y^{-2} \cdot \gamma_{0,0, m_{\star}-1 } + \mathcal{O}(1),\;\; 0 < y \ll1,
		\end{align*}
		where the third line  is the right side of  \eqref{inhom-inner-sys-second-reduced} and $\gamma_{0,0,m_{\star} -1} + 2m_{\star}d_{0,0, m_{\star} } + 3 d_{1,0,m_{\star} -1}$ is  independent of $ d_{0,0, m_{\star} } ,\; d_{1,0,m_{\star}-1 }  $, in fact it only depends on $ c_{0, \tilde{l} + \alpha_2 -1, m_{\star}-1}$. Since $  y^{-2}$  is  the Greens kernel of $ \Delta_y = \partial_y^2 + 3 y^{-1} \partial_y $ we calculate (and hence set)
		\begin{align*}
			&(\mathcal{L}_S + \mu) y^{-2} = \big(- \frac{i}{2} \Lambda + \mu \big) y^{-2}  = (\frac{i}{2} + \mu) y^{-2},\\[3pt]
			&g_{0, m_{\star}-1, \alpha_1, \alpha_2}(y)  := \tilde{g}_{0, m_{\star}-1, \alpha_1, \alpha_2}(y)  + y^{-2} \cdot \tilde{\gamma}_{0,0, m_{\star}-1} ,\;\; \tilde{\gamma}_{0,0,  m_{\star}-1} :=  \frac{\gamma_{0,0, m_{\star} -1}}{(\frac{i}{2} + \mu)},
		\end{align*} 
		thus the solution of  the remaining problem
		$
		(\mathcal{L}_S + \mu) \tilde{g}_{0, m_{\star}-1, \alpha_1, \alpha_2}(y) = O(1),\; 0 < y \ll1 
		$
		has again the form
		\begin{align*}
			\tilde{g}_{0, m_{\star}-1, \alpha_1, \alpha_2}(y) =&\; d_{0, 0, m_{\star}-1}e^{(1)}_0(y) + \int_{0}^y w^{-1}(s) G(s,y) \mathcal{O}(1) \;ds\\
			=&  \;d_{0,0,  m_{\star}-1} + \mathcal{O}(y^2).
		\end{align*}
		where we choose any $d_{0, 0, m_{\star}-1} \in \mathbf{C}$. For $m = m_{\star} -2 $ in \eqref{inhom-inner-sys-last-reduced}, we first repeat the latter two steps with $ s = 2$ and $ s=1 $, i.e. we choose $ d_{2,0,  m_{\star}-2}, d_{1,0,  m_{\star}-2} \in \mathbf{C}$ freely and obtain
		\begin{align*}
			g_{2, m_{\star} -2, \alpha_1, \alpha_2}(y) = &\;\; d_{2,0, m_{\star}-2} e^{(1)}_0(y) + \int_{0}^y w^{-1}(s) G(s,y)  \tilde{f}_{2, m_{\star}-2, \alpha_1, \alpha_2}(s)\;ds\\
			& =  d_{2,0,  m_{\star}-2} + \mathcal{O}(y^2),\\
			g_{1, m_{\star}-2, \alpha_1, \alpha_2}(y)  = &\;\; d_{1,0,  m_{\star}-2} + \mathcal{O}(y^2) + y^{-2} \cdot \tilde{\gamma}_{1,0, m_{\star}-2} ,\;\; \tilde{\gamma}_{1,0,  m_{\star}-2} :=  \frac{\gamma_{1,0, m_{\star} -2}}{(\frac{i}{2} + \mu)},
		\end{align*} 
		where 
		\begin{align*}
			&D_y g_{1, m_{\star}-1, \alpha_1, \alpha_2}(y) =  y^{-2} \cdot 2 d_{1,0, m_{\star} -1} + \mathcal{O}(1),\;\; 0 < y \ll1,\\
			&\tilde{D}_y g_{2, m_{\star}-2, \alpha_1, \alpha_2}(y) =  y^{-2} \cdot 3 d_{2,0,m_{\star}-2 } + \mathcal{O}(1),\;\; 0 < y \ll1,\\
			&\tilde{f}_{1, m_{\star}-2, \alpha_1, \alpha_2}(y) - (m_{\star} -1)D_y g_{1, m_{\star}-1, \alpha_1, \alpha_2}(y) - 2\tilde{D}_y g_{2, m_{\star}-2, \alpha_1, \alpha_2}(y)\\
			&\hspace{3cm} = y^{-2} \cdot \gamma_{1,0, m_{\star}-2 } + \mathcal{O}(1),\;\; 0 < y \ll1,
		\end{align*}
		and $ \gamma_{1,0, m_{\star}-2 } + 6 d_{2, 0,m_{\star} -2} + 2(m_{\star} -1) d_{1, 0, m_{\star} -1}$ does neither depend on $ d_{2, 0,m_{\star} -2} $ nor on $ d_{1, 0,m_{\star} -1}$. Further we infer now for  $ 0 < y \ll1 $
		\begin{align*}
			&D_y g_{0, m_{\star}-1, \alpha_1, \alpha_2}(y) =  -2 \tilde{\gamma}_{0,0, m_{\star}-1}   \cdot y^{-4}  -   \big(i (\f12 + \nu) \tilde{\gamma}_{0,0 m_{\star}-1} - 2 d_{0,0, m_{\star} -1}\big)\cdot y^{-2}  + \mathcal{O}(1),\\
			&\tilde{D}_y g_{1, m_{\star}-2, \alpha_1, \alpha_2}(y) =  - \tilde{\gamma}_{1,0, m_{\star}-2}   \cdot y^{-4}  -   \big( \frac i 2 \tilde{\gamma}_{1,0 m_{\star}-2}  - 3 d_{1,0, m_{\star} -2}\big)\cdot y^{-2}  + \mathcal{O}(1),\;\;\text{and}\\[10pt]
			%& y^{-2} \big( 2(m_{\star} -1)g_{}  \big) \cdot g_{m_{\star}, \alpha_1, \alpha_2}(y) = d_{0, m_{\star}} \cdot y^{-2} +  \mathcal{O}(1),\\[4pt]
			&\tilde{f}_{0, m_{\star}-2, \alpha_1, \alpha_2}(y) - (m_{\star} -1)D_y g_{0, m_{\star}-1, \alpha_1, \alpha_2}(y) - \tilde{D}_y g_{1, m_{\star}-2, \alpha_1, \alpha_2}(y)\\
			&\hspace{2cm}\;\;   - (m_{\star} -1)\frac{2}{y^2} g_{1,  m_{\star} -1, \alpha_1, \alpha_2}(y) -  \frac{2}{y^2} g_{2,  m_{\star} -2, \alpha_1, \alpha_2}(y)\\
			&\hspace{2cm}\;\; - m_{\star} (m_{\star} -1)\frac{1}{y^2} g_{0,  m_{\star}, \alpha_1, \alpha_2}(y)\\
			&\hspace{2cm} = y^{-4} \cdot  \gamma_{0,0, m_{\star}-2} +  y^{-2} \cdot  \gamma_{0, 1,m_{\star}-2} + \mathcal{O}(1),
		\end{align*}
		where the latter corresponds to the right side of \eqref{inhom-inner-sys-last-reduced} in case $m = m_{\star} -2$. Further 
		$$  \gamma_{0,0, m_{\star}-2} + 2 (m_{\star} -1) \tilde{\gamma}_{0,0, m_{\star}-1} + \tilde{\gamma}_{1, 0 , m_{\star}-2} $$
		does not depend on $  \tilde{\gamma}_{0,0, m_{\star}-1}, \tilde{\gamma}_{1,0, m_{\star}-2} $ (and not on $d_{0,0, m_{\star}}$), as well as 
		\begin{align*}
			\gamma_{0, 1,m_{\star}-2} - & (m_{\star} -1)\big(i (\f12 + \nu) \tilde{\gamma}_{0,0 m_{\star}-1} - 2 d_{0,0, m_{\star} -1}\big) -  \big( \frac i 2 \tilde{\gamma}_{1,0 m_{\star}-2}  - 3 d_{1,0, m_{\star} -2}\big)\\
			&	+ (m_{\star} -1)2 d_{1, 0, m_{\star} -1}  +   2 d_{2, 0, m_{\star} -2} 
		\end{align*} 
		does not depend on  $ \tilde{\gamma}_{0,0, m_{\star}-1}, d_{0,0, m_{\star} -1}, \tilde{\gamma}_{1,0, m_{\star}-2},  d_{1,0, m_{\star} -2}$ or $ d_{2, 0, m_{\star} -2} $ (though on $d_{0,0, m_{\star}}$).\\[4pt]
		Now let us choose  $ d_{0, 0, m_{\star}}$  such that $ \gamma_{0,0,m_{\star}-2} = 0 $ (note the choice can be made through the dependence of  $ \tilde{\gamma}_{0,0, m_{\star}-1}, \tilde{\gamma}_{1,0, m_{\star}-2} $ on $ d_{0, 0, m_{\star}}$) and  hence we solve \eqref{inhom-inner-sys-last-reduced} (in case $m = m_{\star} -2$) via
		\begin{align*}
			g_{0, m_{\star}-2, \alpha_1, \alpha_2}(y) =&\;  y^{-2}\cdot \tilde{\gamma}_{0, 1,m_{\star}-2} + d_{0, 0, m_{\star}-2}e^{(1)}_0(y) + \int_{0}^y w^{-1}(s) G(s,y) \mathcal{O}(1) \;ds\\
			=&  y^{-2}\cdot \tilde{\gamma}_{0, 1,m_{\star}-2} + d_{0,0,  m_{\star}-2} + \mathcal{O}(1),\;\; \tilde{\gamma}_{1,m_{\star}-2} = \frac{ \gamma_{1,m_{\star}-2}}{(\frac{i}{2} + \mu)}
		\end{align*}
		In the subsequent step, we then fix $ d_{1,0, m_{\star} -1}, d_{1,0, m_{\star} -1} \in \mathbf{C}$. To be precise  
		we infer
		\begin{align}\nonumber
			&g_{3, m_{\star}-3, \alpha_1, \alpha_2}(y) = d_{3, 0, m_{\star}-3} + \mathcal{O}(y^2),\\ \nonumber
			&g_{2, m_{\star}-3, \alpha_1, \alpha_2}(y) = d_{2, 0, m_{\star}-3} + \mathcal{O}(y^2) + y^{-2} \tilde{\gamma}_{2, 0,m_{\star} -3},\;\;\;\tilde{\gamma}_{0,0,m_{\star}-3} = \frac{ \gamma_{0,0, m_{\star}-3}}{(\frac{i}{2} + \mu)},\\ \nonumber
			&g_{1, m_{\star}-3, \alpha_1, \alpha_2}(y) = d_{1, 0, m_{\star}-3} + \mathcal{O}(y^2) + y^{-2} \tilde{\gamma}_{1, 1,m_{\star} -3},
		\end{align}
		where for the second line we note
		\begin{align}
			f_{2, m_{\star}-3, \alpha_1, \alpha_2}(y)& - (m_{\star} -2) D_y g_{2, m_{\star} -2,\alpha_1, \alpha_2}(y) - 3\tilde{D}_y g_{3, m_{\star} -3,\alpha_1, \alpha_2}(y)\\ \nonumber
			&=  y^{-2}  \gamma_{2, 0,m_{\star} -3} + \mathcal{O}(1),
		\end{align}
		and $\gamma_{2, 0,m_{\star} -3}$ depends only on $d_{2, 0,m_{\star} -2},\; d_{3, 0, m_{\star} -3}$. For the third line we similarly have\\
		\begin{align*}
			&f_{1, m_{\star}-3, \alpha_1, \alpha_2}(y)- (m_{\star} -2) D_y g_{1, m_{\star} -2,\alpha_1, \alpha_2}(y) - 2\tilde{D}_y g_{2, m_{\star} -3,\alpha_1, \alpha_2}(y)\\[2pt]
			& - (m_{\star} -2)\frac{2}{y^2} g_{2, m_{\star} -2,\alpha_1, \alpha_2}(y) - \frac{6}{y^2} g_{3, m_{\star} -3,\alpha_1, \alpha_2}(y)- (m_{\star}-1)(m_{\star}-2) \frac{1}{y^2} g_{1, m_{\star} -1, \alpha_1, \alpha_2}(y)\\[2pt]
			& =  y^{-4}  \gamma_{1, 0,m_{\star} -3} + y^{-2}  \gamma_{1, 1, m_{\star} -3} +  \mathcal{O}(1).
		\end{align*}
		Here $  \gamma_{1, 0,m_{\star} -3}  $ depends only on $ \tilde{\gamma}_{1, 0, m_{\star} -2},  \tilde{\gamma}_{2, 0, m_{\star} -3}$ (and of course the corresponding source coefficient  for $ y^{-4}$). Considering the dependence of the $\gamma-$coefficients as calculated before, we choose $ d_{1, 0, m_{\star} -1} $ such that $  \gamma_{1, 0,m_{\star} -3}  = 0$, which allows the form of the above solution. Further $\gamma_{1,1,m_{\star}-3}$ depends only on 
		$
		\tilde{\gamma}_{1, 0,m_{\star}-2}, \tilde{\gamma}_{2, 0, m_{\star}-3}, d_{1, 0, m_{\star} -2}
		$ and $ d_{2, 0, m_{\star} -3}, d_{3, 0, m_{\star} -3}, d_{2, 0, m_{\star} -3}, d_{1, 0, m_{\star} -1}$. Now we claim $g_{0, m_{\star} -3, \alpha_1, \alpha_2}(y) $ has the form
		\[
		g_{0, m_{\star}-3, \alpha_1, \alpha_2}(y) = d_{0, 0, m_{\star}-3} + \mathcal{O}(y^2) + y^{-4} \tilde{c}_{0, \tilde{l} - 2\alpha_2 -2 ,m_{\star} -3} + y^{-2} \tilde{\gamma}_{0, 1,m_{\star} -3},
		\]
		where we in particular need
		\[
		(\mathcal{L}_S + \mu)  y^{-4} = 8 y^{-6} + i \tfrac 3 2 y^{-4},\;\;\;\ 8  \tilde{c}_{0, \tilde{l} - 2\alpha_2 -2 ,m_{\star} -3}  : = c_{0, \tilde{l} - 2\alpha_2 -3 ,m_{\star} -3},
		\]
		thus also
		\begin{align*}
			&	f_{0, m_{\star}-3, \alpha_1, \alpha_2}(y)+ i \tfrac 32 \tilde{c}_{0, \tilde{l} - 2\alpha_2 -2, m_{\star}-3} y^{-4}- (m_{\star} -2) D_y g_{0, m_{\star} -2,\alpha_1, \alpha_2}(y)\\[2pt]
			& \hspace{1cm}- 2\tilde{D}_y g_{1, m_{\star} -3,\alpha_1, \alpha_2}(y)  - (m_{\star} -2)\frac{2}{y^2} g_{1, m_{\star} -2,\alpha_1, \alpha_2}(y)\\[2pt]
			& \hspace{1cm}- \frac{6}{y^2} g_{2, m_{\star} -3,\alpha_1, \alpha_2}(y) - (m_{\star}-1)(m_{\star}-2) \frac{1}{y^2} g_{0, m_{\star} -1, \alpha_1, \alpha_2}(y)\\[2pt]
			& \hspace{1cm}\;\;  =  y^{-4}  \gamma_{0, 0,m_{\star} -3} + y^{-2}  \gamma_{1, 0, m_{\star} -3}  + \mathcal{O}(1).
		\end{align*}
		Hence again $\gamma_{0, 0,m_{\star} -3}$ depends only on the $y^{-4} $ coefficient of $ f_{0, m_{\star}-3, \alpha_1, \alpha_2}(y)$, $  \tilde{c}_{0, \tilde{l} - 2\alpha_2 -2, m_{\star}-3}$ and $ \tilde{\gamma}_{0, 1, m_{\star} -2},\; \tilde{\gamma}_{1, 1, m_{\star} -3}$. Thus we choose $ d_{0,0,m_{\star} -1} \in \mathbf{C}$ such that $ \gamma_{0,0,m_{\star} - 3} = 0$.\\[4pt]
		We continue by induction over $m \geq 0$ and assume $ g_{s, m_{\star} - k,\alpha_1, \alpha_2 }$ for $ 0 \leq k \leq m-1$ and $  0 \leq s \leq k$  are solutions of \eqref{inhom-inner-sys-first-reduced} - \eqref{inhom-inner-sys-last-reduced} and the first two equations of \eqref{inhom-inner-sys-last-reduced} for $ m_{\star} -m$ are as usual solved by
		\begin{align*}
			&g_{m, m_{\star} -m, \alpha_1, \alpha_2}(y) = d_{m, 0, m_{\star} -m} + \mathcal{O}(y^2),\\
			&g_{m-1, m_{\star} -m, \alpha_1, \alpha_2}(y) = d_{m-1, 0, m_{\star} -m} + \mathcal{O}(y^2) + y^{-2} \cdot \tilde{\gamma}_{m-1, 0, m_{\star} -m},
		\end{align*}
		where $ \gamma_{m-1, 0, m_{\star} -m}$ does only depend on $ d_{m-1,0, m_{\star} -m +1}, d_{m, 0, m_{\star} -m}$. The right side in the third equation of \eqref{inhom-inner-sys-last-reduced}  for $ m$ has the expansions\\[1pt]
		\begin{align}\label{induction-right-sinde-here}
			&\tilde{f}_{s, m_{\star} - m, \alpha_1, \alpha_2}(y) - (m_{\star} -m +1)D_y g_{s, m_{\star} - m +1, \alpha_1, \alpha_2}(y)  - (s+1)D_y g_{s+1, m_{\star} - m, \alpha_1, \alpha_2}(y)\\[2pt] \nonumber
			&- (m_{\star} -m +2) (m_{\star} -m +1)\frac{1}{y^2}g_{s, m_{\star} - m +2, \alpha_1, \alpha_2}(y)\\[2pt] \nonumber
			& - (s+2)(s+1) \frac{1}{y^2} g_{s+2, m_{\star} -m , \alpha_1, \alpha_2}(y) - (m_{\star} -m +1) (s+1) \frac{2}{y^2} g_{s+1, m_{\star} - m +1, \alpha_1, \alpha_2}(y)\\[3pt] \nonumber
			&\hspace{.5cm} = \sum_{l =2}^{ m -s}y^{-2l} \cdot  \gamma_{s, m -s -l,m_{\star} -m} +  y^{-2} \cdot  \gamma_{s, m -s -1,m_{\star} -m} + \mathcal{O}(1),\;\; 0 \leq s \leq m-2.
		\end{align}
		Likewise, for any $ 0 \leq k \leq m-1, 0 \leq s \leq k$, we define the $\gamma$-coefficients as follows.
		\begin{align}
			&\tilde{f}_{k-1, m_{\star} - k, \alpha_1, \alpha_2}(y) - (m_{\star} -k +1)D_y g_{k-1, m_{\star} - k +1, \alpha_1, \alpha_2}(y)  - kD_y g_{k, m_{\star} - k, \alpha_1, \alpha_2}(y)\\[2pt] \nonumber
			&\hspace{.5cm} =   y^{-2} \cdot  \gamma_{k,0, m} + \mathcal{O}(1),\;\; s = k -1,\\[6pt] \nonumber
			&\tilde{f}_{s, m_{\star} - k, \alpha_1, \alpha_2}(y) - (m_{\star} -k +1)D_y g_{s, m_{\star} - k +1, \alpha_1, \alpha_2}(y)  - (s+1)D_y g_{s+1, m_{\star} - k, \alpha_1, \alpha_2}(y)\\ \nonumber
			&- (m_{\star} -k +2) (m_{\star} -m +1)\frac{1}{y^2}g_{s, m_{\star} - k +2, \alpha_1, \alpha_2}(y)\\ \nonumber
			& - (s+2)(s+1) \frac{1}{y^2} g_{s+2, m_{\star} -k , \alpha_1, \alpha_2}(y) - (m_{\star} -k +1) (s+1) \frac{2}{y^2} g_{s+1, m_{\star} - k +1, \alpha_1, \alpha_2}(y)\\ \nonumber
			&\hspace{.5cm} = \sum_{l =2}^{ k -s}y^{-2l} \cdot  \gamma_{s, k -s -l,m_{\star} -k} +  y^{-2} \cdot  \gamma_{s, k -s -1, m_{\star} -k} + \mathcal{O}(1),\;\; 0 \leq s \leq k -2,\; k \geq 2.
		\end{align}
		Now we make the induction assumptions
		\begin{itemize} \setlength\itemsep{3pt}
			\item[$\bullet$] The coefficients $  \gamma_{s, k -s,m_{\star} -k}$ for $ s > 2 $ depend only on $m, \nu>0,\; \tilde{f}_{\tilde{s}, m_{\star} -\tilde{k}, \alpha_1, \alpha_2}(y)$ where $\tilde{k} \leq k$ with $ \tilde{s} > s$ if $ \tilde{k} =k$.
			\item[$\bullet$] The coefficients  $  \gamma_{s, k -s -2,m_{\star} -k} $ depend on $ \tilde{c}_{s, \tilde{l} + \alpha_2 -1, m_{\star} -k}$, $\tilde{f}_{s,  m_{\star} - k, \alpha_1, \alpha_2}(y),$ and\\ $ \gamma_{s+1,k -s -2, m_{\star} -k},  \gamma_{s,k -s -2, m_{\star} -k +1}$.
			\item[$\bullet$] The coefficients  $  \gamma_{s, k -s -1,m_{\star} -k} $ depend on $ \tilde{c}_{\tilde{s}, \tilde{l} + \alpha_2 , m_{\star} -\tilde{k}}$, $\tilde{f}_{\tilde{s},  m_{\star} - \tilde{k}, \alpha_1, \alpha_2}(y)$ and\\ $ \gamma_{\tilde{s},\tilde{k} - \tilde{s} - l, m_{\star} -\tilde{k}}$ where $ \tilde{k} \leq k$ and $\tilde{s} > s$ if $ k = \tilde{k}$.
			\item[$\bullet$] The coefficients $  \tilde{c}_{s, \tilde{l} + \alpha_2 , m_{\star} -k } \in \mathbf{C}$ are free to  choose for $ m-2 \leq  k \leq m-1,\; 0 \leq s \leq k$ and fixed via the above  arguments  for  $ 0 \leq k \leq m-3,\;\;  0 \leq s \leq k$.
		\end{itemize}
		Here `dependence' means linear combinations with coefficients depending on the index parameters as well as $  \nu, \mu$. Now from carrying out the $ m_{\star} -m $ step as above  we have the following conclusions.
		\begin{itemize} \setlength\itemsep{3pt}
			\item[$\bullet$] The coefficients $  \gamma_{s, m -s,m_{\star} -m}$ for $ s > 2 $ depend only on $m, \nu>0,\; \tilde{f}_{\tilde{s}, m_{\star} -\tilde{k}, \alpha_1, \alpha_2}(y)$ where $\tilde{k} \leq m$ with $ \tilde{s} > s$ if $ \tilde{k} =m$ (see below).
			\item[$\bullet$] The coefficients  $  \gamma_{s, m -s -2,m_{\star} -m} $ depend on $ \tilde{c}_{s, \tilde{l} + \alpha_2 -1, m_{\star} -m}$, $\tilde{f}_{s,  m_{\star} - m, \alpha_1, \alpha_2}(y),$ and\\ $ \gamma_{s+1,m -s -2, m_{\star} -m},  \gamma_{s,m -s -2, m_{\star} -m +1}$.
			\item[$\bullet$] The coefficients  $  \gamma_{s, m -s -1,m_{\star} -m} $ depend on $ \tilde{c}_{\tilde{s}, \tilde{l} + \alpha_2 , m_{\star} -\tilde{k}}$, $\tilde{f}_{\tilde{s},  m_{\star} - \tilde{k}, \alpha_1, \alpha_2}(y)$ and\\ $ \gamma_{\tilde{s},\tilde{k} - \tilde{s} - l, m_{\star} -\tilde{k}}$ where $ \tilde{k} \leq m$ and $\tilde{s} > s$ if $ m = \tilde{k}$.
			\item[$\bullet$] The coefficients $  \tilde{c}_{s, \tilde{l} + \alpha_2 , m_{\star} -k } \in \mathbf{C}$ are free to  choose for $ m-1 \leq  k \leq m,\; 0 \leq s \leq k$ and fixed via the above  arguments  for  $ 0 \leq k \leq m-2,\;\;  0 \leq s \leq k$. That is we select $ \tilde{c}_{s, \tilde{l} + \alpha_2 , m_{\star} - m} \in \mathbf{C}$ for $ 0 \leq s \leq m$ and  fix all $   \tilde{c}_{s, \tilde{l} + \alpha_2 , m_{\star} - m +2}$ in the $m_{\star} -m$ step.
		\end{itemize}
		In fact for the  induction step with $ 0 \leq s \leq m-3$, we obtain the first $m -s -3$ coefficients 
		$$ \tilde{c}_{s, \tilde{l} + \alpha_2 -m +s , m_{\star} -m},\; \tilde{c}_{s, \tilde{l} + \alpha_2 -m +s +1 , m_{\star} -m},\; \tilde{c}_{s, \tilde{l} + \alpha_2 -m +s +2 , m_{\star} -m},,\; \dots, \;  \tilde{c}_{s, \tilde{l} + \alpha_2 -3 , m_{\star} -m},\; $$
		from the right side \eqref{induction-right-sinde-here} via the recurrence above, i.e.
		\begin{align}
			\begin{cases}
				(\mathcal{L}_S + \mu) \big( \sum_{l =2}^{m -s}y^{2-2l} \cdot \tilde{c}_{s, \tilde{l} + \alpha_2 - l, m_{\star} -m} \big) =  \sum_{l =2}^{m -s}y^{-2l} \cdot  \gamma_{s, m - s -l,m_{\star} -m} &\\[10pt]
				(2-2l)(4-2l)\tilde{c}_{s, \tilde{l} + \alpha_2 - l, m_{\star} -m}  = \gamma_{s, m - s - l,m_{\star} -m},\;\;\; l = m-s,&\\[4pt]
				(2-2l)(4-2l)\tilde{c}_{s, \tilde{l} + \alpha_2 - l, m_{\star} -m}  = \gamma_{s, m - s -l,m_{\star} -m},&\\[4pt]
				\hspace{2cm} + \big( \frac{i}{2}( 1 -2l)-  \mu\big) \tilde{c}_{s, \tilde{l} + \alpha_2 - l -1, m_{\star} -m},\;\;\;\; 3 \leq l <  m -s&
			\end{cases}
		\end{align}
		The case where $ l =2$ is  degenerate and  we then fix $ \tilde{c}_{s, \tilde{l} + \alpha_2 - 1, m_{\star} -m +2}$ such that $ \gamma_{s, m - s - 2,m_{\star} -m} = 0$.  Further in case $ l=1$ we need to set again
		\begin{align}
			\tilde{c}_{s, \tilde{l} + \alpha_2 - 1, m_{\star} -m} = \tilde{\gamma}_{s, m - s - 1,m_{\star} -m}:= \frac{\gamma_{s, m - s - 1,m_{\star} -m}}{(\frac{i}{2} + \mu)},
		\end{align}
		The remaining part is 'regular' enough at $ y =0$ and simply adds a smooth function of order $ \mathcal{O}(y^2)$ for  $y \to 0^+$ as seen above. Now the case where $ k_0 > 0$ works similarly  starting with the above recurrence relation at $ s = m_{\star} - m -2$.%We conclude the claim by considering the expansion of the right side of \eqref{inhom-inner-sys-last-reduced} for $ m_{\star} - m - 1$. In fact by the recurrent system above, we verify our assumption, i.e. the above three $\bullet$ assertions hold true for the case $m+1$.  The induction shows in particular that $ \tilde{c}_{\tilde{l} +k-1,m}$ are fixed subsequently except for the last two $ \tilde{c}_{\tilde{l} +k-1,0},\;  \tilde{c}_{\tilde{l} +k-1,1}$ which uniquely determine the solution.
	\end{proof}
	\;\\
	\tinysection[0]{The Iteration in the self-similar region}
	We now aim to approximate solutions to \eqref{schrod-wave-y} consistent with the previous interior approximation. Therefore we recall the ansatz for $ w(t,y)$ in \eqref{eq:wgeneric-1b} which states ($ m_{\star} = s_{\star} + 2n,\; s_{\star} = 2l -1$)
	\begin{align}\label{hier-ansatz-w}
		w(t, y)	&=  \sum_{n,p}\sum_{l , \tilde{l}, \tilde{k}}t^{\nu(2(n + l +p) - \tilde{l})+ \tilde{k}} \sum_{m = 0}^{s_{\star} + 2n}(\log(y) - \nu\log(t))^m g^{(1)}_{m, n, l , \tilde{l}, \tilde{k},p}(y)\\ \nonumber
		&\;\;+ \sum_{n,p}\sum_{l , \tilde{l}, \tilde{k}}t^{\nu(2(n + l+ p) -\tilde{l})+ \tilde{k} + \f12} \sum_{m = 0}^{s_{\star} + 2n}(\log(y) - \nu\log(t))^m g^{(2)}_{m, n, l , \tilde{l}, \tilde{k},p}(y),
	\end{align}
	where the coefficients $g^{(i)}_{m, n, l , \tilde{l}, \tilde{k}, p}(y)$ are required to have an absolutely convergent expansion of the form
	\begin{align}\label{exp1-new}
		&g^{(i)}_{n,m,l,\tilde{k}, \tilde{l}, p}(y) = \sum_{\tilde{m} = 0}^{{m_{\star} -m}}\big(\log(y)\big)^{ \tilde{m}} \sum_{k \geq 0} y^{2(k - l -n - p- \tilde{k})  + 2\nu \cdot (\tilde{l} -1)} \cdot y^{- (i-1)}c^{k, l, n,m, p}_{\tilde{l}, \tilde{k}, \tilde{m},i},\;\;\;\\ \nonumber
		&0 < y \ll1.
	\end{align}
	Further the ansatz for the wave part $n(t,y) $ in  \eqref{eq:wgeneric-2b}  reads
	\begin{align} \label{hier-ansatz-n}
		n(t, y)	&=  \sum_{n,p} \sum_{l , \tilde{l}, \tilde{k}}t^{\nu(2(n + l + p) - \tilde{l})+ \tilde{k}} \sum_{m = 0}^{s_{\star} + 2n}(\log(y) - \nu\log(t))^m  h^{(1)}_{m, n, l , \tilde{l}, \tilde{k}, p}(y)\\ \nonumber
		&\;\;+ \sum_{n,p} \sum_{l , \tilde{l}, \tilde{k}}t^{\nu(2(n + l + p) -\tilde{l})+ \tilde{k} + \f12} \sum_{m = 0}^{s_{\star} + 2n} (\log(y) - \nu\log(t))^m  h^{(2)}_{m, n, l , \tilde{l}, \tilde{k}, p}(y),
	\end{align}
	for which we likewise require the expansions
	\begin{align} \label{exp2-new}
		&h^{(i)}_{n,m, l,\tilde{k}, \tilde{l}}(y) = \sum_{\tilde{m} = 0}^{m_{\star} -m}\big(\log(y)\big)^{ \tilde{m}} \sum_{k \geq 0} y^{2(k - l -1 -n - \tilde{k})  + 2\nu \cdot (\tilde{l} -2)  } \cdot y^{- (i-1)}\tilde{c}^{k, l, n,m}_{\tilde{l}, \tilde{k}, \tilde{m},i, p},\;\; \;\\ \nonumber
		&0 < y \ll1.
	\end{align}
	Both of the sums \eqref{hier-ansatz-w} and \eqref{hier-ansatz-n} need to respect the constraints in \eqref{Schrod-ansatz2} and \eqref{wave-ansatz2}. To be precise there should hold the following, derived from \eqref{constraint-1} - \eqref{constraint-4} and where $ i =1,2$ correspond to the (half)-integer cases with $g^{(i)},\; h^{(i)}$ coefficient functions.
	\begin{itemize}
		\item[($\ast$)]In the sum \eqref{hier-ansatz-w}: $l = 1,2,3, \dots \mathcal{N}$ has a fixed length determined in the inner region Section \ref{sec:inner} and we sum over a finite set of integers\\
		$ n = 0,1, \dots, N_2^{(1)}$,\;$p = 0,1, \dots, N_2^{(3)} $,\\
		$ 3 \tilde{l} - 4l - 2p +1 \leq 2\tilde{k} + (i-1) \leq 2N_2^{(2)}$ and where  $\tilde{l} = 1,2,3\dots, l$ are integers.\\
		\item[($\ast \ast$)]In the sum \eqref{hier-ansatz-n}: $l = 2,3,4, \dots \mathcal{N}$ has a fixed length determined in the inner region Section \ref{sec:inner} and we sum over a finite set of integers\\ $ n = 0,1, \dots, \tilde{N}_2^{(1)}$, $ p = 0,1, \dots, \tilde{N}_2^{(3)}$,\\
		$3 \tilde{l} - 4l - 2p -4 \leq 2\tilde{k} + (i-1) \leq 2\tilde{N}_2^{(2)}$ and  where  $\tilde{l} = 2,3,4, \dots, l+1$ are integers.
	\end{itemize}
	\begin{Rem} (i) The dependence on $ i =1,2$ is only a technical aspect of the iteration since we distinguish $\alpha_2 \in \f12 + \Z $ and $ \alpha_2 \in \Z$ for the terms $ t^{\alpha_1\nu + \alpha_2}$ respectively. Infact we could also run the iteration replacing \eqref{hier-ansatz-w} and \eqref{hier-ansatz-n} with $\sum_{\alpha_1, \alpha_2 \in \Z}t^{\alpha_1\nu + \f{\alpha_2}{2}}$.\\[3pt]
		(ii)\; More relevant concerning the required 'matching expansions' \eqref{exp1-new} and \eqref{exp2-new}, is understanding the contribution of the factors
		$ y^{2\nu (\tilde{l} -1)},\; y^{2\nu (\tilde{l} -2)}$,
		which are \emph{forced by the wave interaction} (and absent in \cite{schmid},\cite{Perelman}, \cite{OP},  see also the Remark below). %The most important aspect of the iteration is understanding the contribution of $ y^{2\nu(\tilde{l} -1)},\;y^{2\nu(\tilde{l} -2)}$ in the asymptotic expansions \eqref{exp1-new} and \eqref{exp2-new}.
	\end{Rem}
	\;\;\\
	%We now apply the preceding to the {\it{nonlinear system}} \eqref{schrod-wave-y}.\\[5pt]
	\underline{\emph{Description of the iteration}}.\;
	Since the {\it{nonlinear system}} system \eqref{schrod-wave-y} is a coupling of $n,w$ and the sums \eqref{hier-ansatz-w}, \eqref{hier-ansatz-n} have two free parameters $n \in \Z_+, \tilde{k} \in \Z$, we will construct the coefficients
	$$ g^{(1)}_{n,m,l, \tilde{l}, \tilde{k}, p}(y),\; g^{(2)}_{n,m,l, \tilde{l}, \tilde{k}, p}(y),\; h^{(1)}_{n,m,l, \tilde{l}, \tilde{k}, p}(y),\; h^{(2)}_{n,m,l, \tilde{l}, \tilde{k}, p}(y)$$
	simultaneously as follows. Let us first choose $ N_2^{(schr)}, N_2^{(wave)} $ in $ \Z_+$  arbitrarily but fixed. We further define  the \emph{parameter sets}
	\begin{align*}
		J^{(schr)}_{j} : = \{ (n,l,\tilde{l}, p) \in \Z_{\geq 0}^4\;|\;   2(n + l + p) - \tilde{l} \leq j \;\;\&\;\; 1 \leq l,\;\; 1 \leq \tilde{l} \leq l \; \; \},
	\end{align*}
	for $ j \in \Z_+$ fixed and likewise %we define the following  set 
	for the second line in \eqref{schrod-wave-y} 
	\begin{align*}
		J^{(wave)}_{j} : = \{ (n,l,\tilde{l}, p) \in \Z_{\geq 0}^4\;|\;   2(n + l + p) - \tilde{l} \leq j \;\;  \&\;\; 2 \leq l,\;\;2 \leq \tilde{l} \leq l+1 \; \}.
	\end{align*} 
	For convenience we set $ J_{0}^{(schr)}= J_{0}^{(wave)} = \emptyset$. Now we consider an inductive iteration over $ j \in \Z_+$. In each iteration step with $j \in \Z_+$ fixed, we then further iterate over $ \tilde{k} \leq N_2^{(schr)} $ in the first line, respectively $ \tilde{k} \leq N_2^{(wave)} $ in the second line of \eqref{schrod-wave-y}. To be precise we start with the coefficients having parameters  $  (n,l,\tilde{l},p)\in J^{(schr)}_{1}$ and $ (n,l,\tilde{l},p) \in J^{(wave)}_{1}$. Subsequently, we want to construct the coefficient functions in \eqref{hier-ansatz-w} and \eqref{hier-ansatz-n} corresponding to 
	\begin{align*}
		&(n,l,\tilde{l},p)\in J^{(schr)}_{2} \setminus J^{(schr)}_{1},\;\;\;\; 3 \tilde{l} - 4l -2p +1 \leq  2\tilde{k } + (i-1),\;\; \tilde{k}\leq N^{(schr)}_2,\\
		&(n,l,\tilde{l},p) \in J^{(wave)}_{2} \setminus J^{(wave)}_{1},\;\; \;\; 3 \tilde{l} - 4l -2p -4  \leq  2\tilde{k } + (i-1),\;\; \tilde{k}\leq N^{(wave)}_2,
	\end{align*}
	This of course means we fix
	$  2(n + l + p) - \tilde{l} = 2$ in both cases. Then we iterate finite sub-steps  over $\tilde{k}$ for each of such parameters $(n,l, \tilde{l}, p)$.\\[4pt]
	\emph{In general}, we assume by induction the coefficients  corresponding to $ J^{(schr)}_{j-1},\; J^{(wave)}_{j-1}$ are known and have an asymptotic expansion for $0 < y \ll1$ consistent with  \eqref{exp1-new} and \eqref{exp2-new}. We  now want to construct the coefficients (with a consistent expansion) for subsequent parameters
	\begin{align*}
		&(n,l,\tilde{l}, p)\in J^{(schr)}_{j} \setminus J^{(schr)}_{j-1},\;\;\;\; 3 \tilde{l} - 4l -2p +1 \leq 2\tilde{k} + (i-1),\;\; \tilde{k}\leq N^{(schr)}_2,\\
		&(n,l,\tilde{l}, p) \in J^{(wave)}_{j} \setminus J^{(wave)}_{j-1},\;\;\;\;3 \tilde{l} - 4l -2p -4  \leq  2\tilde{k } + (i-1),\;\; \tilde{k}\leq N^{(wave)}_2\;\;\;\;(\ast \ast \ast).
	\end{align*}
	Again this means we consider all $(n,l,\tilde{l},p)$ with  $2(n+l+p) - \tilde{l} = j$ in both the Schr\"odinger and the wave case. More concretely, we  introduce with $ m_{\star} = s_{\star} + 2n$
	\begin{align*}
		w_N^{\ast}(t,y) : &= \sum_{(n,l, \tilde{l}, p) \in J_{N}^{(schr)}} \sum_{ 3 \tilde{l} - 4l -2p +1\leq 2\tilde{k} \leq 2N^{(schr)}_2}t^{\nu(2(n + l+p) - \tilde{l})+ \tilde{k}} \sum_{m = 0}^{m_{\star} }(\log(y) - \nu\log(t))^m g^{(1)}_{m, n, l , \tilde{l}, \tilde{k}, p}(y)\\ \nonumber
		&\;\;+ \sum_{(n,l, \tilde{l}, p) \in J_{N}^{(schr)}} \sum_{ 3 \tilde{l} - 4l -2p  \leq 2\tilde{k}  \leq 2N^{(schr)}_2 }t^{\nu(2(n + l+p) -\tilde{l})+ \tilde{k} + \f12} \sum_{m = 0}^{m_{\star} }(\log(y) - \nu\log(t))^m g^{(2)}_{m, n, l , \tilde{l}, \tilde{k}, p}(y),\\
		n_N^{\ast}(t,y) : &=  \sum_{(n,l, \tilde{l}, p) \in J_{N}^{(wave)}}	\sum_{ 3 \tilde{l} - 4l -2p -4 \leq 2\tilde{k} \leq 2N^{(wave)}_2} t^{\nu(2(n + l+p) - \tilde{l})+ \tilde{k}} \sum_{m = 0}^{m_{\star} }(\log(y) - \nu\log(t))^m  h^{(1)}_{m, n, l , \tilde{l}, \tilde{k}, p}(y)\\ \nonumber
		&\;\;+ \sum_{(n,l, \tilde{l}, p) \in J_{N}^{(wave)}} \sum_{ 3 \tilde{l} - 4l -2p -5 \leq 2\tilde{k}  \leq 2N^{(wave)}_2}  t^{\nu(2(n + l+p) -\tilde{l})+ \tilde{k} + \f12} \sum_{m = 0}^{m_{\star} } (\log(y) - \nu\log(t))^m  h^{(2)}_{m, n, l , \tilde{l}, \tilde{k}, p}(y),
	\end{align*}
	for which we may write of course
	\begin{align*}
		&w^{\ast}_N(t,y) = w_1(t,y) + w_{2}(t,y) + \dots + w_{N}(t,y),\\
		&n^{\ast}_N(t,y) = n_1(t,y) + n_{2}(t,y) + \dots + n_{N}(t,y),
	\end{align*}
	where
	\begin{align*}
		&w_j(t,y) = \sum_{(n,l, \tilde{l}, p) \in J_{j}^{(schr)}\setminus J_{j-1}^{(schr)}} \sum_{ 3 \tilde{l} - 4l -2p +1\leq 2\tilde{k} \leq 2N^{(schr)}} t^{\nu(2(n + l +p) - \tilde{l})+ \tilde{k}} \sum_{m = 0}^{m_{\star} }(\log(y) - \nu\log(t))^m g^{(1)}_{m, n, l , \tilde{l}, \tilde{k}, p}(y)\\ \nonumber
		&\;\;+ \sum_{(n,l, \tilde{l}, p) \in J_{j}^{(schr)}\setminus J_{j-1}^{(schr)}}  \sum_{ 3 \tilde{l} - 4l -2p  \leq 2\tilde{k} \leq 2N^{(schr)}} t^{\nu(2(n + l +p) -\tilde{l})+ \tilde{k} + \f12} \sum_{m = 0}^{m_{\star} }(\log(y) - \nu\log(t))^m g^{(2)}_{m, n, l , \tilde{l}, \tilde{k}, p}(y),
	\end{align*}
	and likewise for $n_j(t,y)$.  Plugging $ w_j^{\ast}(t,y),\; n_j^{\ast}(t,y)$ into the equation \eqref{schrod-wave-y}, we thus determine the coefficients $g^{(j)}, h^{(j)}$ (after  $N$  iteration steps).  By comparison this subsequently leads to a family  of  systems \eqref{inhom-inner-sys-first} - \eqref{inhom-inner-sys-last} and \eqref{inhom-inner-wave-first} - \eqref{inhom-inner-wave-last}, where the interaction terms
	\[
	t \cdot n_{j}^{\ast}(t,y) \cdot w_{j}^{\ast}(t,y),\;\; (\partial_y^2 + \frac{3}{y}\partial_y)|w_{j}^{\ast}(t,y)|^2.
	\]
	(having a similar expansion as \eqref{hier-ansatz-w} and \eqref{hier-ansatz-n}) provide the inhomogeneity  $ f_{n,m,l,\tilde{l}, \tilde{k}, p}(y),\; \tilde{f}_{n,m,l,\tilde{l}, \tilde{k}, p}(y)$ with $(n, l,\tilde{l}),\; \tilde{k} $ as in $(\ast \ast \ast)$ on the right sides.%The source terms for the construction in $(\ast \ast \ast)$ are (compare the first and second line of \eqref{schrod-wave-y})
	%They have a similar form of the sum \eqref{hier-ansatz-w} and \eqref{hier-ansatz-n} which we state in the induction step below. The \emph{inhomogeneous data} for the Schr\"odinger system \eqref{inhom-inner-sys-first} - \eqref{inhom-inner-sys-last} and the wave system \eqref{inhom-inner-sys-first} - \eqref{inhom-inner-sys-last} with parameters $ (n,m,l,\tilde{l}, \tilde{k}) $ in $(\ast \ast \ast)$ \emph{are the corresponding coefficients in the source expansions}.\\[2pt]
	\;\;\\[20pt]
	We begin by describing the first steps of the iteration. Let us note here, in order to match the interior approximation, we rely on having $ \tilde{s}_{\ast} = m_{\star} -m$ in the expansions \eqref{exp1-new} and \eqref{exp2-new}. \\[10pt]
	%gwith for which we consider all $ (n,l,\tilde{l}) \in J^{(schr)}_{j}$ in the first line, as well as $ (n,l,\tilde{l}) \in J^{(wave)}_{j+1}$ in the second line of \eqref{schrod-wave-y}. We call this the $j^{\text{th}}$ \emph{basic Step}. Let us start with a fixed parameter triple $(n,l,\tilde{l}) \in J^{(schr)}_{j}$
	\underline{\emph{Start of the iteration}}.\; In the following \emph{Steps} 1,2 \& 3, we find the \emph{lowest order} terms, i.e. the correction $ w_1^*(t,y), \;n_1^*(t,y)$,  correspond to (unique) minimal parameter tuples in the sets 
	$$ J_1^{(schr)} = \{ (0,1,1, 0)\},\;\;J_1^{(wave)} = \{  (0,2,3, 0) \}.$$
	The associated coefficient functions are \emph{free solutions} of the linear Schr\"odinger \eqref{hom-inner} and wave equation \eqref{hom-wave1} respectively. We then calculate  subsequent higher order corrections indicating how the general induction should be obtained.\\[5pt]  %First we focus on corrections of the form \eqref{hier-ansatz-w} for the Schr\"odinger part.\\[2pt] %Therefore, if we sum over $ \tilde{k} \in 2 \Z$ even and since $ \tilde{k} \geq -l +1$ the above integer (first-line of \eqref{hier-ansatz-w}) and half-integer case (second line of \eqref{hier-ansatz-w})  
	\underline{\emph{Step 1}}.\;\emph{(Minimal parameter in $J_{1}^{(schr)}$)}.\;\; For the minimal coefficient  of the Schr\"odinger part
	$g_{m(\alpha_{1\star}),\alpha_{1\star},\alpha_{2\star}}(y)$, the following holds. Since we have the parameters 
	$$ \alpha_1 = 2n + 2l + 2p- \tilde{l},\;\; \alpha_2 = \tilde{k} $$ 
	the minimum is attained for $ n =p = 0, l =1, \tilde{l}= l = 1$ and thus $ 2\tilde{k} \geq  3 \tilde{l} - 4l -2p +1 = 0$ (and likewise for the half-integer case $2\tilde{k} +1 \geq 3 \tilde{l} - 4l -2p +2 = 1$). Hence the minimal coefficients correspond to 
	\begin{align*}
		&t^{\nu\alpha_1 + \alpha_2 } = t^{\nu(2 -1) + \tilde{k}} = t^{\nu + \tilde{k}}\;\;\;\text{integer\;case\;(first\;line\;of\;} \eqref{hier-ansatz-w}),\\
		&t^{\nu\alpha_1 + \alpha_2 } = t^{\nu(2 -1) +  \tilde{k} +\f12} = t^{\nu+  \tilde{k}+ \f12}\;\;\;\text{half-integer\;case\;(second\;line\;of\;} \eqref{hier-ansatz-w}),
	\end{align*}
	with $ \tilde{k} \geq 0$. The  case  of the pure Schr\"odinger interaction is of course included in this set up, as explained above.
	\begin{Rem}
		Infact, we may set $ n = 0$ in the ansatz of \eqref{match-ellnull}. This likewise implies $ t^{\nu\alpha_1 + \alpha_2 } = t^{2\nu( 0 + \f12)} = t^{\nu}$  and hence this coefficient corresponds to parameters in $ J_1^{(schr)}$.
	\end{Rem} %Note we can identify the expansion here with the previous case of $ l =1, n=0$ and further the maximal power $ m_{\star} = s_{\star}(l) + 2n$. 
	In particular by definition $ J_1^{(schr)} = \{ (0,1,1, 0) \}$ and thus
	$w_{\text{approx}}(t,y)$  consist of higher order perturbations (concerning powers of $ t> 0$) of the terms
	\begin{align}
		t^{\nu\alpha_{1\star}+\alpha_{2\star}}&(\log(y) - \nu\log(t))^{m_{\star}}\cdot g^{(1)}_{m_{\star}, \alpha_{1\star}, \alpha_{2\star}}(y)\\ \nonumber
		& = t^{\nu}(\log(y) - \nu\log(t))^{m_{\star}} \cdot g^{(1)}_{0,m_{\star}, 1,1, 0, 0}(y),\\[8pt]
		t^{\nu\alpha_{1\star}+\alpha_{2\star}}&(\log(y) - \nu\log(t))^{m_{\star}}\cdot g^{(2)}_{m_{\star}, \alpha_{1\star}, \alpha_{2\star}}(y)\\ \nonumber
		& = t^{\nu + \frac{1}{2}}(\log(y) - \nu\log(t))^{m_{\star}} \cdot g^{(2)}_{0,m_{\star},1,1,0, p}(y).
	\end{align}
	where $m_{\star} = s_{\star}(1) + 0 = 1 $. More precisely we set (let us drop the superscript of $N_2$ for now)
	\begin{align}
		w_1(t,y) = w_1^{\ast}(t,y) & = \sum_{\tilde{k} = 0}^{N_2}t^{\nu + \tilde{k}}\sum_{m = 0}^{m_{\star}}(\log(y) - \nu\log(t))^{m} \cdot g^{(1)}_{0,m,1, 1, \tilde{k}, 0}(y),\\[2pt] \nonumber
		&\;\;\; + \sum_{\tilde{k} = 0}^{N_2}t^{\nu + \tilde{k} +\frac{1}{2}}\sum_{m = 0}^{m_{\star}}(\log(y) - \nu\log(t))^{m} \cdot g^{(2)}_{0,m,1, 1, \tilde{k}, 0}(y)
	\end{align}
	for which we construct the coefficients below.
	\\[4pt]
	\underline{\emph{Step 2}}.\;\emph{(Minimal parameter in $J_{1}^{(wave)}$)}.\;\;By Step 1, we conclude the minimal coefficients in the wave source 
	\[
	(\partial_{yy} + \frac{3}{y}\partial_y)\big(|w|^2\big)
	\]
	of the second line of \eqref{schrod-wave-y} must correspond to the terms $ t^{2\nu},\; t^{2\nu + \f12}$. Precisely we have (let us drop the $p$-dependence here; which is $p = 0$)
	\begin{align*}
		|w_1(t,y)|^2 & = \sum_{\tilde{k} \geq 0}t^{2\nu + \tilde{k}}\sum_{m = 0}^{\tilde{m}_{\star}}(\log(y) - \nu\log(t))^{m} \cdot \sum_{\substack{k_1 + k_2 = m\\ \tilde{k} = \tilde{k}_1 + \tilde{k}_2}}c_{k_1, k_2, \tilde{k}_1, \tilde{k}_2}g^{(1)}_{0,k_1,1, 1, \tilde{k}_1}(y)\overline{g^{(1)}_{0,k_2,1, 1, \tilde{k}_2}(y)},\\[2pt] \nonumber
		&\;\;\; + \sum_{\tilde{k} \geq 0}t^{2\nu + \tilde{k} +1}\sum_{m = 0}^{\tilde{m}_{\star}}(\log(y) - \nu\log(t))^{m} \cdot \sum_{\substack{k_1 + k_2 = m\\ \tilde{k} = \tilde{k}_1 + \tilde{k}_2}}c_{k_1, k_2, \tilde{k}_1, \tilde{k}_2}g^{(2)}_{0,k_1,1, 1, \tilde{k}_1}(y)\overline{g^{(2)}_{0,k_2,1, 1, \tilde{k}_2}(y)},\\[2pt] \nonumber
		&\;\;\; + \sum_{\tilde{k} \geq 0} t^{2\nu +  \tilde{k} + \f12}\sum_{m = 0}^{\tilde{m}_{\star}}(\log(y) - \nu\log(t))^{m} \cdot \sum_{\substack{k_1 + k_2 = m\\ \tilde{k} = \tilde{k}_1 + \tilde{k}_2}}c_{k_1, k_2, \tilde{k}_1, \tilde{k}_2}g^{(1)}_{0,k_1,1, 1, \tilde{k}_1}(y)\overline{g^{(2)}_{0,k_2,1, 1, \tilde{k}_2}(y)},\\[2pt] \nonumber
		&\;\;\; + \sum_{\tilde{k} \geq 0} t^{2\nu +  \tilde{k} + \f12}\sum_{m = 0}^{\tilde{m}_{\star}}(\log(y) - \nu\log(t))^{m} \cdot \sum_{\substack{k_1 + k_2 = m\\ \tilde{k} = \tilde{k}_1 + \tilde{k}_2}}c_{k_1, k_2, \tilde{k}_1, \tilde{k}_2}g^{(2)}_{0,k_1,1, 1, \tilde{k}_1}(y)\overline{g^{(1)}_{0,k_2,1, 1, \tilde{k}_2}(y)},
	\end{align*}
	The minimal coefficient for the wave part $h_{m(\beta_{1\star}),\beta_{1\star},\beta_{2\star}}(y)$ on the other hand, has  the minimum parameter value $ \beta_1 = 2n + 2l + 2p - \tilde{l},\;\; \beta_2 = \tilde{k} $, whence $ n = p = 0, l = 2,\; \tilde{l}= l+1 = 3$ and thus $ 2\tilde{k} \geq 3 \tilde{l} - 4l -2p -4 = -3$ (and $ 2\tilde{k}  \geq 3 \tilde{l} - 4l -2p -5  = -4$ in the half-integer case). Hence
	the minimal coefficient corresponds to 
	\begin{align*}
		&t^{\nu\beta_1 + \beta_2 } = t^{\nu(4 -3) -1} = t^{\nu -1}\;\;\;\text{integer\;case\;(first\;line\;of\;}\eqref{hier-ansatz-n}),\\
		&t^{\nu\beta_1 + \beta_2 } = t^{\nu(4 -3)  -2 +\f12} = t^{\nu -\f32}\;\;\;\text{half-integer\;case\;(second\;line\;of\;}\eqref{hier-ansatz-n}).
	\end{align*}
	and we have $ J_1^{(wave)} = \{ (0,2,3, 0) \}$ which stands for the terms 
	$
	t^{\nu  + \tilde{k}},\;t^{\nu -\f12 + \tilde{k}}$ with $  \tilde{k} \geq -1  $.
	Therefore, the corrections
	\begin{align*} &t^{\nu  + \tilde{k}} \big(\sum_{m =0}^{m_{\star}}\big(\log(y) - \nu \log(t)\big)^m h^{(1)}_{0,m, 2,3,\tilde{k}, 0}(y)\big),\;\; \tilde{k} \geq -1,\\
		&t^{\nu + \f12 + \tilde{k}} \big(\sum_{m =0}^{m_{\star}}\big(\log(y) - \nu \log(t)\big)^m h^{(2)}_{0,m, 2,3,\tilde{k}, 0}(y)\big),\;\;\tilde{k} \geq -2,
	\end{align*} 
	must correspond to a \emph{free solution}, i.e. 
	they are given by the homogeneous analogue of \eqref{inhom-inner-wave-first} - \eqref{inhom-inner-wave-last}, which of course are provided in Lemma \ref{lem:hom1} and Corollary \ref{cor:hom-wave} via  the expansion \eqref{hom-wave-formula}. Precisely if $\tilde{k} \geq -1 $ is minimal,  we infer
	\[
	L_{2, - 1}h^{(1)}_{0,m, 2,3,-1,0}(y) = 0,\;\;L_{2, - \f32}h^{(2)}_{0,m, 2,3,-2, 0}(y) = 0,
	\]
	and thus we have
	\begin{align*}
		h^{(1)}_{0,m, 2,3,-1, 0}(y) = \sum_{s = 0}^{m_{\star} -m}\log^s(y)\big(c^{(1)}_{0,m, s ,2,3,-1}y^{2( \nu -1)} + d^{(1)}_{0,m, s, 2,3,-1} y^{2( \nu -1) -2}\big),\\
		h^{(2)}_{0,m, 2,3,-2, 0}(y) = \sum_{s = 0}^{m_{\star} -m}\log^s(y)\big(c^{(2)}_{0,m,s,2,3,-1}y^{2( \nu -\f32)} + d^{(2)}_{0,m,s 2,3,-2} y^{2( \nu -\f32) -2} \big).
	\end{align*}
	Subsequently, as seen in Lemma \ref{lem:hom1},  we obtain a homogeneous solution of \eqref{inhom-inner-wave-first} - \eqref{inhom-inner-wave-last} which has the form of Corollary \ref{cor:hom-wave}, i.e. with $ 0 \leq \tilde{k} \leq N_{2}^{(wave)}$ we have again 
	\begin{align} 
		%& t^{-\nu \beta_1 - \beta_2} \cdot  t^2 \square_S \bigg( t^{\nu \beta_1+\beta_2} (\log(y) + \f12\log(t))^m  \tilde{h}_{m, \beta_1, \beta_2}(y) \bigg),\\ \label{that-relation}
		h^{(1)}_{0, m, 2,3,\tilde{k}-1, 0}(y)   &= \sum_{s =0}^{m_{\star} -m}\log^s(y) \sum_{r =0}^{ \tilde{k}} \big( \tilde{c}^{(1)}_{m,s, \tilde{k}, r} y^{2 \nu - 2 + 2\tilde{k} - 4r} + \tilde{c}^{(2)}_{m, s,\tilde{k},r} y^{2 \nu - 4 + 2\tilde{k} - 4r}\big),\\ \label{diese-line}
		h^{(2)}_{0, m, 2,3,\tilde{k}-2, 0}(y)   &=  \sum_{s =0}^{m_{\star} -m} \log^s(y)\sum_{r =0}^{ \tilde{k}} \big( \tilde{d}^{(1)}_{m, s, \tilde{k}, r} y^{2 \nu - 3 + 2\tilde{k} - 4r} + \tilde{d}^{(2)}_{m, s, \tilde{k}, r} y^{2 \nu - 5 + 2\tilde{k} - 4r}\big),
	\end{align} 
	Hence comparing these to the expansion \eqref{exp2-new} required to match the interior approximation, we obtain consistency with the factor $y^{\nu \cdot b}$ allowed to contribute to the asymptotic.  The solutions are hence \emph{'matched' by proper selection of the free coefficients}  $ \tilde{c}^{(i)}_{m, 0, \tilde{k}, 0},\; \tilde{d}^{(i)}_{m, 0, \tilde{k}, 0} \in \mathbf{C}$. We note in particular from \eqref{exp2-new},  we need to match  $\tilde{d}^{(2)}_{m,0,0, 0} =0$ since with factor $t^{\nu - \frac{3}{2}}$ the matching expansion must be $ \mathcal{O}(y^{2\nu-3})$ 
	\begin{Rem}This is of course consistent at minimal $\tilde{k}$ with the growth of the free hyperbolic solutions $\mathcal{O}(a^{\beta_l}),\; \mathcal{O}(a^{\beta_l -1})$ in the $X_j$ spaces of Section \ref{sec:inner}.
	\end{Rem}
	\;\\
	\underline{\emph{Step 3}}.\;\emph{(Back to $ J_1^{(schr)}$ and extension to $ J_2^{(schr)}$)}.\;\; By Step 1 \& 2 the minimal coefficients of the Schr\"odinger source term (the first line of \eqref{schrod-wave-y}) are in the expansion of 
	\begin{align*}
		t \cdot n_1^*(t,y)\cdot w_1^*(t,y) = &\;\sum_{\tilde{k} \geq 0}\sum_{m= 0}^{m_{\star}}t^{2\nu + \tilde{k}} (\log(y) - \nu \log(t))^m \tilde{g}^{(1)}_{2, m,\tilde{k}, 0}\\
		&\;\; + \sum_{\tilde{k} \geq -1} \sum_{m= 0}^{m_{\star}}t^{2\nu + \tilde{k} + \f12} (\log(y) - \nu \log(t))^m \tilde{g}^{(2)}_{2, m,\tilde{k}, 0}.
	\end{align*}
	We therefore conclude the minimal coefficients considered in Step 1
	\[
	g^{(1)}_{0,m,1,1, \tilde{k}, 0}(y),\;\;g^{(2)}_{0,m,1,1, \tilde{k}, 0}(y),
	\] 
	corresponding to $ t^{\nu + \tilde{k}},\; t^{\nu + \f12 + \tilde{k}},\;J_1^{(schr)} = \{ (0,1,1, 0)\} $  must be \emph{free solutions}, i.e. solving the \emph{homogeneous system} \eqref{hom-inner} with $ m_{\star} = 2-1 + 0 =1$. Hence they are of the form in Lemma \ref{lem:inhom-schrod-practical} with a vanishing source term, i.e.  for all $ 0 \leq \tilde{k} \leq N_2^{(schr)}$ we have the expansion 
	%	\begin{align}
		%	g^{(i)}_{0,m_{\star} -m,1,1, \tilde{k}} (y) = \sum_{s=0}^m\mathcal{O}(y^{-2}) (\log(y))^s + \mathcal{O}(1) (\log(y))^{m +1},\;\; i = 1,2.
		%	\end{align}
	\begin{align}
		&g^{(1)}_{0,m_{\star} ,1,1, \tilde{k}, 0} (y) = c^{(0)}_{0, m_{\star},1,1,\tilde{k}} + \mathcal{O}(y^2),\;\; g^{(2)}_{0,m ,1,1, \tilde{k}, 0} (y)  = 0,\;\; 0 \leq m \leq m_{\star},\\
		&g^{(1)}_{0,m ,1,1, \tilde{k}, 0} (y)  = \sum_{s= 0}^{m_{\star} -m} \log^s(y)\big(c^{(0)}_{s, 0,m,1,1, \tilde{k}} \cdot y^{-2} + c^{(1)}_{s, 0,m,1,1, \tilde{k}} + \mathcal{O}(y^2)\big),\;\; 0 \leq m \leq m_{\star} -1,\\ \label{dieser-koeff}
		& c^{(0)}_{m_{\star} -m, 0,m,1,1, \tilde{k}} = 0.
	\end{align}
	where $ c^{(1)}_{s, 0,0,1,1, \tilde{k}}, \in \mathbf{C}$ for $ 0 \leq s \leq m_{\star}$ and $ c^{(1)}_{s, 0,1,1,1, \tilde{k}} \in \mathbf{C}$  for $ 0 \leq s \leq m_{\star} -1 $ are free to choose.
	%To be precise they are unique solutions (defined on $ y \in (0, \infty)$) such that for $ 0 < y \ll1$ 
	%	\[g^{(i)}_{0,m_{\star} -m,1,1, \tilde{k}} (y) = \sum_{s = 0}^{m} \sum_{k \geq 0} y^{2k -2} (\log(y))^s d^{(m)}_{k s}+  \sum_{k \geq0} y^{2k} (\log(y))^{m +1}   \tilde{d}^{(m)}_{k},\;\; i = 1,2,\; 
	%	\]
	%	in an absolute sense given $d^{(m)}_{0,0},\; d^{(m)}_{1,0} \in \mathbf{C}$ are prescribed and $d^{(m)}_{k s} = d^{(m)}_{k s}(n,l,\tilde{l}, \tilde{k},i)$. Thus we again infer the asymptotic  to be consistent with the expansion \eqref{exp1-new} (including the absence of a $y^{\nu \cdot b}$ contribution). In particular \emph{the solutions are matched by a proper choice of the above free coefficients}.
	The new parameter 4-tuple $(0,2,2, 0), $ in the subsequent iteration step
	\[
	J_2^{(schr)} = \{ (0,2,2, 0),(0,1,1, 0)\}
	\] 
	corresponds to $ t^{2\nu + \tilde{k}},\; t^{2\nu + \tilde{k} + \f12}$ where $\tilde{k} \geq 0$ in the former and $ \tilde{k} \geq -1$ in the latter case. Thus the coefficients 
	\[
	g^{(1)}_{0,m,2,2, \tilde{k}, 0}(y),\;\;g^{(2)}_{0,m,2,2, \tilde{k}, 0}(y),
	\] 
	solve the inhomogeneous system \eqref{inhom-inner-sys-first} - \eqref{inhom-inner-sys-last} with interaction source terms $\tilde{g}_{2, m, \tilde{k}}$ on the right.
	In particular, we now consider % the other terms $t^{2\nu + \tilde{k} },\;t^{2\nu + \tilde{k} + \f12}, \tilde{k} \geq 0$, i.e. (
	\begin{align} \label{wave-initial-iter}
		\big(i t \partial_t + \mathcal{L}_2  - \alpha_0\big)\bigg(&\sum_{\tilde{k} \geq 0} t^{2\nu + \tilde{k}}\sum_{m =0}^{m_{\star}}\big(\log(y) - \nu \log(t)\big)^m g^{(1)}_{0,m, 2,2,\tilde{k}, 0}(y)\bigg)\\ \nonumber
		& =  \sum_{\tilde{k} \geq 0} t^{2\nu + \tilde{k}}\sum_{m =0}^{m_{\star}}\big(\log(y) - \nu \log(t)\big)^m \tilde{g}^{(1)}_{0,m, 2,2,\tilde{k}}(y),\\  \label{wave-initial-iter2}
		\big(i t \partial_t + \mathcal{L}_2  - \alpha_0\big)\bigg(&\sum_{\tilde{k} \geq -1} t^{2\nu + \tilde{k} + \f12}\sum_{m =0}^{m_{\star}}\big(\log(y) - \nu \log(t)\big)^m g^{(2)}_{0,m, 2,2,\tilde{k}, 0}(y)\bigg)\\ \nonumber
		& =  \sum_{\tilde{k} \geq -1} t^{2\nu + \tilde{k} + \f12}\sum_{m =0}^{m_{\star}}\big(\log(y) - \nu \log(t)\big)^m \tilde{g}^{(2)}_{0,m, 2,2,\tilde{k}}(y),
	\end{align}
	where the coefficients on the right are corresponding to the above interaction terms. There holds from the previous steps (with $ \tilde{k} \geq 0$)
	\begin{align*}
		&\tilde{g}^{(1)}_{0,m_{\star}, 2,2,\tilde{k}, 0}(y) = \sum_{r =0}^{ \tilde{k}} \big( d^{(0)}_{ 2,m_{\star}, \tilde{k}}\mathcal{O}( y^{2 \nu - 2 + 2\tilde{k} - 4r}) + d^{(1)}_{2,m_{\star}, \tilde{k}} \mathcal{O}(y^{2 \nu - 4 + 2\tilde{k} - 4r})\big) = \mathcal{O}(y^{2 \nu}\cdot y^{ - 4 - 2\tilde{k} }),\\
		&\tilde{g}^{(1)}_{0,m, 2,2,\tilde{k}, 0}(y) =  \sum_{s = 0}^{m_{\star} -m} \log^s(y) \sum_{r =0}^{ \tilde{k}} \big( d^{(0)}_{s, 2,m, \tilde{k}}\mathcal{O}(y^{2 \nu - 4 + 2\tilde{k} - 4r}) + d^{(1)}_{s, 2,m, \tilde{k}}\mathcal{O}( y^{2 \nu - 6 + 2\tilde{k} - 4r})\big)\\
		&\hspace{2.2cm} = \sum_{s = 0}^{m_{\star} -m} \sum_{r =0}^{ \tilde{k}} \log^s(y) \cdot \tilde{d}^{(1)}_{s, 2,m, \tilde{k}} \mathcal{O}(y^{2 \nu - 6 + 2\tilde{k} - 4r}).
	\end{align*}
	Likewise we obtain in the half-integer case (with $ \tilde{k} \geq -1$)
	\begin{align*}
		&\tilde{g}^{(2)}_{0,m_{\star}, 2,2,\tilde{k}}(y)  =  \mathcal{O}(y^{2 \nu}\cdot y^{ - 5 - 2\tilde{k} }),\;\;\tilde{g}^{(2)}_{0,m, 2,2,\tilde{k}}(y)  =  \sum_{s = 0}^{m_{\star} -m} \log^s(y) \cdot \tilde{c}_{s, 2,m, \tilde{k}} \mathcal{O}(y^{2 \nu}\cdot y^{ - 7- 2\tilde{k} }).
	\end{align*}
	Thus by Lemma \ref{lem:inhom-schrod-practical}  the solutions   to  \eqref{wave-initial-iter} and \eqref{wave-initial-iter2} are uniquely determined and satisfy
	\begin{align*}
		&g^{(1)}_{0,m_{\star}, 2,2,\tilde{k}}(y) = \sum_{ k \geq 0} y^{2k -2 -2\tilde{k}} \cdot y^{2\nu} \;\tilde{c}^{0,2,2}_{k, m_{\star},\tilde{k}},\;\; 0 < y \ll1,\\
		&g^{(1)}_{0,m, 2,2,\tilde{k}}(y) =  \sum_{s = 0}^{m_{\star} -m} \log^s(y) \sum_{ k \geq 0} y^{2k -4 -2\tilde{k}} \cdot y^{2\nu} \;\tilde{c}^{0,2,2}_{s, k, m,\tilde{k}},\;\; 0 < y \ll1,\;\; 0 \leq m \leq m_{\star}-1,\\
		&g^{(2)}_{0,m_{\star}, 2,2,\tilde{k}}(y) =  \sum_{ k \geq 0} y^{2k -3 -2\tilde{k}} \cdot y^{2\nu} \;\tilde{c}^{0,2,2}_{k, m_{\star},\tilde{k}},\;\; 0 < y \ll1,\;\; \tilde{k}\geq -1,\\
		&g^{(2)}_{0,m, 2,2,\tilde{k}}(y) = \sum_{s = 0}^{m_{\star} -m} \log^s(y)  \sum_{ k \geq 0} y^{2k -5 -2\tilde{k}} \cdot y^{2\nu} \;\tilde{c}^{0,2,2}_{s, k, m,\tilde{k}},\;\; 0 < y \ll1,\;\; 0 \leq m \leq m_{\star}-1,\;\; \tilde{k}\geq -1.
	\end{align*}
	In fact since \eqref{dieser-koeff} holds, we actually  have
	$
	\tilde{c}^{0,2,2}_{s, 0, m,\tilde{k}}  = 0 $ in the cases where $ s  = m_{\star} -m $.\\[5pt]
	\underline{\emph{Step 4}}.\;\emph{(The parameter set $J_2^{(wave)}$)}.\;\;  Moving on to the second line of \eqref{schrod-wave-y} we have
	\begin{align*}
		&J_1^{(wave)} = \{  (0,2,3, 0) \},\\
		&J_2^{(wave)}  = \{   (0,2,3, 0), (0,2,2, 0), (0,3,4, 0)  \},
	\end{align*}
	and the two triples  $(0,2,2,0), (0,3,4, 0)  $  in $J_2^{(wave)}\setminus J_1^{(wave)}\  $ correspond to the coefficients  of
	\begin{align*}
		&(0,2,2,0)\;\;\;t^{2\nu + \tilde{k}},\; \tilde{k} \geq -3,\;\text{(integer\;case)},\;\;\; t^{2\nu + \tilde{k} + \f12},\; \tilde{k} \geq -3,\; \text{(half-integer\;case)},\\
		&(0,3,4,0)\;\;\;t^{2\nu + \tilde{k}},\; \tilde{k} \geq -2,\;\text{(integer\;case)},\;\;\; t^{2\nu + \tilde{k} + \f12},\; \tilde{k} \geq -2,\; \text{(half-integer\;case)}.
	\end{align*}
	Now we  recall the minimal terms on the right of \eqref{schrod-wave-y} involves the sum 
	\begin{align}
		&|w^{\ast}_1(t,y)|^2  = \sum_{\tilde{k} \geq 0}t^{2\nu + \tilde{k}}\sum_{m = 0}^{\tilde{m}_{\star}}(\log(y) - \nu\log(t))^{m} \cdot \tilde{h}^{(1)}_{0,m,2,2,\tilde{k}}(y)\\
		&\tilde{h}^{(1)}_{0,m,2,2,\tilde{k}, 0}(y) = \sum_{\substack{k_1 + k_2 = m\\ \tilde{k} = \tilde{k}_1 + \tilde{k}_2}}c_{k_1, k_2, \tilde{k}_1, \tilde{k}_2}g^{(1)}_{0,k_1,1, 1, \tilde{k}_1, 0}(y)\overline{g^{(1)}_{0,k_2,1, 1, \tilde{k}_2, 0}(y)},
	\end{align}
	with asymptotic for $ 0 < y \ll1$
	\begin{align}
		& \tilde{h}^{(1)}_{0,m,2,2,\tilde{k}, 0}(y) = \sum_{s=0}^{m_{\star} -m}\log^s(y) \big( c^{(1)}_{s,m,\tilde{k}} y^{-4} + c^{(2)}_{s,m,\tilde{k}} y^{-2} + \mathcal{O}(1)  \big),\;\; 0 \leq m \leq m_{\star} ,\\
		& c^{(1)}_{ m_{\star}-m, m,\tilde{k}} = c^{(2)}_{ m_{\star}-m, m,\tilde{k}} = c^{(1)}_{m_{\star} -m -1 ,m,\tilde{k}} = 0.
	\end{align}
	and hence by Lemma \ref{lem:inhom1-update} the solution has the expansion 
	\begin{align*}
		&h^{(1)}_{0,m_{\star},2,2,\tilde{k}}(y) = \mathcal{O}(1),\;\;  h^{(1)}_{0,k,m_{\star}-1,2,2,\tilde{k}}(y) = \mathcal{O}(1) + \log(y)\mathcal{O}(y^{-2}),\\
		& h^{(1)}_{0,m,2,2,\tilde{k}}(y) = \sum_{s = 0}^{m_{\star} -m}\sum_{k \geq 0} \log^s(y)\cdot y^{2k -6-2\tilde{k}} c_{k,s, \tilde{k}},\;\; 0 \leq m \leq m_{\star} -2,\;\;  h^{(2)}_{0,k, m,2,2,\tilde{k}}(y) = 0,
	\end{align*}
	where the absolute sum is of order $ \mathcal{O}(1) $ in case $ s = m_{\star} -m$ and $ \mathcal{O}(y^{-2})$ if $ s = m_{\star} -m -1$.\\[5pt]
	These are consistent with the matching condition \eqref{exp2-new} for $(0,2,2,0)$ and we trivially solve in the case of $ \tilde{k} = -3,-2,-1$ (the source terms require $ \tilde{k} \geq 0$).
	It remains to calculate the coefficients for $ (0,3,4, 0) $, which from the matching condition \eqref{exp2-new} and the source term must be an \emph{almost free wave } approximation. Hence, by Lemma \ref{lem:hom1}  and  Corollary \ref{cor:hom-wave} with $ 0 \leq \tilde{k} \leq N_{2}^{(wave)}$ the solution is given by
	\begin{align} 
		%& t^{-\nu \beta_1 - \beta_2} \cdot  t^2 \square_S \bigg( t^{\nu \beta_1+\beta_2} (\log(y) + \f12\log(t))^m  \tilde{h}_{m, \beta_1, \beta_2}(y) \bigg),\\ \label{that-relation}
		h^{(1)}_{0, m, 3,4,\tilde{k}-2, 0}(y)   &=  \sum_{s=0}^{m_{\star} -m}\log^s(y)\sum_{r =0}^{ \tilde{k}} \big( \tilde{c}^{(1)}_{s, m, 3,4,\tilde{k}, r} y^{4 \nu - 4 + 2\tilde{k} - 4r} + \tilde{c}^{(2)}_{s,m, 3,4,\tilde{k},r} y^{4 \nu - 6+ 2\tilde{k} - 4r}\big),\\
		h^{(2)}_{0, m, 3,4,\tilde{k}-2, 0}(y)   &=  \sum_{s=0}^{m_{\star} -m} \log^s(y) \sum_{r =0}^{ \tilde{k}} \big( \tilde{d}^{(1)}_{s,m, 3,4,\tilde{k}, r} y^{4 \nu - 3 + 2\tilde{k} - 4r} + \tilde{d}^{(2)}_{s,m, 3,4,\tilde{k}, r} y^{4\nu - 5 + 2\tilde{k} - 4r}\big),
	\end{align} 
	where again $ \tilde{c}^{(i)}_{s,m , 3,4,\tilde{k}, 0},\; \tilde{d}^{(i)}_{s,m,3,4,\tilde{k}, 0} \in \mathbf{C}$ are free to choose and need to be selected from  \eqref{exp2-new}.\\[3pt]
	\underline{\emph{Step 5}}.\;\emph{(Higher order corrections in  $ J_3^{(schr)} $ and $J_3^{(wave)}$)}.\;\;
	%and hence correspond to $ t^{3\nu + \tilde{k}},\; t^{3\nu + \tilde{k} + \f12}$ where $ \tilde{k} \geq 1$. In particular we conclude $ g^{(1)}_{n,m,l,\tilde{l}, \tilde{k}}(y),\;g^{(2)}_{n,m,l,\tilde{l}, \tilde{k}}(y)$ with $ (n,k,l)\in J_j^{(schr)}$ for $j = 1,2$ must be solutions of the \emph{homogeneous system} \eqref{hom-inner}. 
	We first turn to the Schr\"odinger case and note that we have 
	\begin{align*}
		&J_1^{(schr)} = \{ (0,1,1, 0)\},\;\;\; J_2^{(schr)} = \{ (0,2,2, 0),(0,1,1, 0)\}\\
		&J_3^{(schr)} = \{ (0,2,2, 0), (0,1,1, 0), (1,1,1, 0), (0,1,1, 1),  (0,2,1, 0), (0,3,3, 0)\},
	\end{align*}
	hence we determine the coefficients associated to  $J_{3}^{(schr)} \setminus J_{2}^{(schr)} $, i.e.
	$$\big\{ g^{(i)}_{1,m,1,1,\tilde{k}, 0}(y),\; g^{(i)}_{0,m,1,1,\tilde{k}, 1}(y),, \;g^{(i)}_{0,m,2,1,\tilde{k}, 0}(y),\;\;g^{(i)}_{0,m,3,3,\tilde{k}, 0}(y) \big\}.$$
	For the source terms we need to consider the right side of \eqref{schrod-wave-y}, i.e. the coefficients in the expansion of
	\begin{align*}
		t \cdot w_2^{\ast}(t,y) \cdot n_{2}^{\ast}(t,y)&  -   t \cdot w_1^{\ast}(t,y) \cdot n_{1}^{\ast}(t,y)\\
		=&\;\; \sum_{\tilde{k} }t^{3 \nu + \tilde{k}} \sum_{m = 0}^{m_{\star}} \big(\log(y) - \nu \log(t)\big)^m \cdot \tilde{g}^{(1)}_{m,\tilde{k}}(y)\\
		& + \sum_{\tilde{k} }t^{3 \nu + \tilde{k} + \f12} \sum_{m = 0}^{m_{\star}} \big(\log(y) -\nu \log(t)\big)^m \cdot \tilde{g}^{(2)}_{m,\tilde{k}}(y)\\
		&  + \sum_{\tilde{k}} t^{4 \nu + \tilde{k}} g^{(1)}_{\tilde{k}}(t,y) +  \sum_{\tilde{k}} t^{4 \nu + \tilde{k} + \f12} g^{(2)}_{\tilde{k}}(t,y),
	\end{align*}
	where terms above corresponding to $t^{3 \nu + \tilde{k}},\; t^{3 \nu + \tilde{k} + \f12}$ appear in the following cases\\[4pt]
	$\bullet$\;\;\; Case $J_2^{(wave)} \setminus J_1^{(wave)}  \times J_1^{(schr)} $:  Here we multiply the homogeneous solutions in $J_1^{(schr)}$ to: (i) The \emph{almost free wave} approximation (corresponding to $(0,3,4, 0)$)  in Step 4 or (ii)  the above inhomogeneous solution (corresponding to $(0,2,2, 0)$) in Step 4. Thus we obtain source terms 
	$$ \tilde{g}^{(1)}_{\tilde{k},m, (i)}(y),\; \tilde{g}^{(2)}_{\tilde{k},m, (i)}(y),\;\tilde{g}^{(1)}_{\tilde{k},m, (ii)}(y),\;\tilde{g}^{(2)}_{\tilde{k},m, (ii)}(y)$$ with  expansions
	\begin{align*}
		&(i) \;\;  \tilde{g}^{(1)}_{\tilde{k},m, (i)}(y) = \sum_{s=0}^{m_{\star} -m -1}\log^s(y) \mathcal{O}(y^{4 \nu - 8 - 2\tilde{k} }) + \log^{m_{\star} -m}(y) \mathcal{O}(y^{4 \nu - 6 - 2\tilde{k} })  \;\; \tilde{k} \geq -1,\\ %\;\;\;  \tilde{g}^{(1)}_{k,m, (i)}(y) = \mathcal{O}(y^{4 \nu - 10 - 2\tilde{k} }),\;\; \tilde{k} \geq -1 ,\\ 
		&\;\; \;\;\;\; \tilde{g}^{(2)}_{\tilde{k},m, (i)}(y) = \sum_{s=0}^{m_{\star} -m -1}\log^s(y)\mathcal{O}(y^{4 \nu - 7 - 2\tilde{k} }) + \log^{m_{\star} -m}(y) \mathcal{O}(y^{4 \nu - 5 - 2\tilde{k} }),\;\; \tilde{k} \geq -1,\\%, \;\;\;  \tilde{g}^{(2)}_{k,m, (i)}(y) = \mathcal{O}(y^{4 \nu - 9 - 2\tilde{k} }),\;\; \tilde{k} \geq -1 ,\\[6pt]
		&(ii) \;\;  \tilde{g}^{(1)}_{\tilde{k},m, (ii)}(y) = \sum_{s=0}^{m_{\star} -m -3}\log^s(y)  \mathcal{O}(y^{ -8 -2\tilde{k}}) +  \log^{m_{\star} -m -2}(y) \mathcal{O}(y^{-6 -2\tilde{k}})\\
		&\;\;\;\;\hspace{4cm} +  \log^{m_{\star} -m -1}(y)\mathcal{O}(y^{-2}) + \log^{m_{\star} -m}(y) \mathcal{O}(1)\;\; \tilde{k} \geq 0, %, \;\;\;  \tilde{g}^{(1)}_{k,m_{\star}-1, (ii)}(y) = \mathcal{O}(y^{-2}),\;\; \tilde{k} \geq 0 ,\\ 
		%&\;\; \;\; \tilde{g}^{(1)}_{k,m_{\star}-2, (ii)}(y) = \mathcal{O}(y^{-4}),\;\; \tilde{k} \geq 0 , \;\;\;  \tilde{g}^{(1)}_{k,m, (ii)}(y) = \mathcal{O}(y^{-6}),\;\; \tilde{k} \geq 0.
	\end{align*}
	where the corresponding sums are empty if the index is negative.\\[6pt]
	$\bullet$\;\;\; Case $J_2^{(schr)} \setminus J_1^{(schr)}  \times J_1^{(wave)} $: Here we multiply the \emph{almost free wave} approximation of $J_1^{(wave)}$ in Step 2 to  the above inhomogeneous solution (corresponding to $(0,2,2, 0)$) in Step 3. Thus we obtain source terms $ \tilde{g}^{(1)}_{\tilde{k},m}(y), \tilde{g}^{(2)}_{\tilde{k},m}(y)$ with expansions
	\begin{align*}
		&\;\;  \tilde{g}^{(1)}_{\tilde{k},m}(y) = \sum_{s=0}^{m_{\star} -m -1}\log^s(y)  \mathcal{O}(y^{4\nu -8-2\tilde{k}}) +  \log^{m_{\star} -m }(y) \mathcal{O}(y^{4\nu -6 -2\tilde{k}})\;\;\; \tilde{k} \geq 0,\\ 
		&\;\;  \tilde{g}^{(2)}_{\tilde{k},m}(y) = \sum_{s=0}^{m_{\star} -m -1}\log^s(y)  \mathcal{O}(y^{4\nu -9-2\tilde{k}}) +  \log^{m_{\star} -m }(y) \mathcal{O}(y^{4\nu -7 -2\tilde{k}})\;\;\; \tilde{k} \geq -1 .
	\end{align*}
	\;\\
	Now for the solution coefficients we have the following.\\[4pt]
	$\bullet$\;\;\; For $(1,1,1, 0), (0,1,1, 1), (0,2,1,0)$ we check the matching condition \eqref{exp1-new} and infer $ y^{2\nu (\tilde{l}-1)} = y^{0}$. Therefore the corresponding coefficients 
	$$ g^{(1)}_{1,m,1,1,\tilde{k}, 0}(y),\;g^{(1)}_{0,m,1,1,\tilde{k}, 1}(y),\; g^{(1)}_{0,m,2,1,\tilde{k}, 0}(y),$$
	`pick up' the source terms in (ii), i.e. given by $  \tilde{g}^{(1)}_{\tilde{k},m, (ii)} $. An application of Lemma \ref{lem:inhom-schrod-practical} implies the expansions (in all cases)
	\begin{align*}
		g^{(1)}_{m,\tilde{k}}(y) =  &\sum_{s=0}^{m_{\star} -m -2}\log^s(y)  \mathcal{O}(y^{ -6-2\tilde{k}}) +  \log^{m_{\star} -m -1}(y)\mathcal{O}(y^{-2}) + \log^{m_{\star} -m}(y) \mathcal{O}(1)\\
		&\; = \sum_{s=0}^{m_{\star} -m }\log^s(y) \sum_{k \geq 0} y^{ 2k-6-2\tilde{k}} c_{ s,m, k, \tilde{k}}\;\; \tilde{k} \geq 0,
	\end{align*}
	where $  c_{ s, 0,3 + \tilde{k} , \tilde{k}}\; c_{ s,1,3 + \tilde{k} , \tilde{k}} \in \mathbf{C}$ are free to choose and need to be selected according to \eqref{exp1-new}.\\[4pt]
	$\bullet$\;\;\; For  $(0,3,3, 0)$  we check the matching condition \eqref{exp1-new} and infer $ y^{2\nu (\tilde{l}-1)} = y^{4\nu}$. Therefore the corresponding coefficient
	$ g^{(1)}_{0,m,3,3,\tilde{k}}(y) $
	`picks up' the remaining source terms given by 
	$$  \tilde{g}^{(1)}_{k,m, (i)} ,\;  \tilde{g}^{(2)}_{k,m, (i)},\; \tilde{g}^{(1)}_{k,m}(y),\;  \tilde{g}^{(2)}_{k,m}(y).$$
	An application of Lemma \ref{lem:inhom-schrod-practical} implies the uniquely determined expansion (in both cases)
	\begin{align*}
		&\;\;  \tilde{g}^{(1)}_{\tilde{k},m}(y) = \sum_{s=0}^{m_{\star} -m -1}\log^s(y)  \mathcal{O}(y^{4\nu -6-2\tilde{k}}) +  \log^{m_{\star} -m }(y) \mathcal{O}(y^{4\nu -4 -2\tilde{k}})\;\;\; \tilde{k} \geq 0,\\ 
		&\;\;  \tilde{g}^{(2)}_{\tilde{k},m}(y) = \sum_{s=0}^{m_{\star} -m -1}\log^s(y)  \mathcal{O}(y^{4\nu -7-2\tilde{k}}) +  \log^{m_{\star} -m }(y) \mathcal{O}(y^{4\nu -5 -2\tilde{k}})\;\;\; \tilde{k} \geq -1.
	\end{align*}
	%	\begin{align*}
		%	&g^{(1)}_{0,m_{\star},3,3,\tilde{k}}(y) =\mathcal{O}(y^{4\nu - 6  - 2\tilde{k}}),\;\;\;\; g^{(1)}_{0,m,3,3,\tilde{k}}(y) =\mathcal{O}(y^{4\nu - 8  - 2\tilde{k}}),\;\;\tilde{k} \geq -1,\\
		%	&g^{(2)}_{0,m_{\star},3,3,\tilde{k}}(y) =\mathcal{O}(y^{4\nu - 7  - 2\tilde{k}}),\;\;\;\; g^{(2)}_{0,m,3,3,\tilde{k}}(y) =\mathcal{O}(y^{4\nu - 9  - 2\tilde{k}}),\;\;\tilde{k} \geq -1.
		%	\end{align*}
	\;\\[2pt]
	Let us turn to the wave case, i.e. the second line of \eqref{schrod-wave-y} and note 
	\begin{align*}
		&J_1^{(wave)} = \{ (0,2,3, 0)\},\;\;\; J_2^{(wave)} = \{ (0,2,3, 0),(0,2,2, 0), (0,3,4, 0)\}\\
		&J_3^{(wave)} = \{ (0,2,3, 0), (0,2,2, 0), (0,3,4, 0), (1,2,3, 0),(0,2,3, 1), (0,3,3, 0), (0,4,5, 0) \},
	\end{align*}
	hence we determine the coefficients associated to  $J_{3}^{(wave)} \setminus J_{2}^{(wave)} $, i.e.
	$$\big\{ h^{(i)}_{1,m,2,3,\tilde{k}, 0}(y),\;\;  h^{(i)}_{0,m,2,3,\tilde{k}, 1}(y),\;\;h^{(i)}_{0,m,3,3,\tilde{k}, 0}(y),\;\;h^{(i)}_{0,m,4,5,\tilde{k}, 0}(y) \big\}.$$
	The source terms are  derived from the right side of the second line of \eqref{schrod-wave-y}, i.e. we need to consider coefficients in the expansion of
	\begin{align*}
		|w_2^{\ast}(t,y)|^2 - |w_1^{\ast}(t,y)|^2 & = \sum_{\tilde{k} }t^{3 \nu + \tilde{k}} \sum_{m = 0}^{m_{\star}} \big(\log(y) - \nu \log(t)\big)^m \cdot \tilde{h}^{(1)}_{m,\tilde{k}}(y)\\
		&\;\;\; + \sum_{\tilde{k} }t^{3 \nu + \tilde{k} + \f12} \sum_{m = 0}^{m_{\star}} \big(\log(y) -\nu \log(t)\big)^m \cdot \tilde{h}^{(2)}_{m,\tilde{k}}(y)\\
		& \;\; + \sum_{\tilde{k}} t^{4 \nu + \tilde{k}} h^{(1)}_{\tilde{k}}(t,y) +  \sum_{\tilde{k}} t^{4 \nu + \tilde{k} + \f12} h^{(2)}_{\tilde{k}}(t,y),
	\end{align*}
	where the terms corresponding to $t^{3 \nu + \tilde{k}},\; t^{3 \nu + \tilde{k} + \f12}$ have the form 
	$$  \tilde{h}^{(i)}_{m,\tilde{k}}(y) =  \tilde{h}^{(i)}_{1, m,\tilde{k}}(y) \cdot \overline{\tilde{h}^{(i)}_{2, m,\tilde{k}}(y) },$$ 
	with either $ \tilde{h}^{(i)}_{1, m,\tilde{k}}(y) $ associated to 
	$ J_2^{(schr)} \setminus J_1^{(schr)}$ and $\tilde{h}^{(i)}_{2, m,\tilde{k}}(y) $ to  $J_1^{(schr)} $ or exactly reversed.\\[3pt]
	
	$\bullet$\;\;\;Right side $J_2^{(schr)} \setminus J_1^{(schr)}  \times J_1^{(schr)} $: Here we multiply the \emph{homogeneous} solution in  $J_1^{(schr)}$ to  the above inhomogeneous solution (corresponding to $(0,2,2,0)$) in Step 3. Thus we obtain source terms $ \tilde{h}^{(1)}_{k,m}(y), \tilde{h}^{(2)}_{k,m}(y)$ with expansions
	\begin{align*}
		&\;\;  \tilde{h}^{(1)}_{\tilde{k},m}(y) = \sum_{s=0}^{m_{\star} -m -1}\log^s(y)  \mathcal{O}(y^{2\nu -6-2\tilde{k}}) +  \log^{m_{\star} -m }(y) \mathcal{O}(y^{2\nu -2 -2\tilde{k}})\;\;\; \tilde{k} \geq 0,\\ 
		&\;\;  \tilde{h}^{(2)}_{\tilde{k},m}(y) = \sum_{s=0}^{m_{\star} -m -1}\log^s(y)  \mathcal{O}(y^{2\nu -7-2\tilde{k}}) +  \log^{m_{\star} -m }(y) \mathcal{O}(y^{2\nu -3 -2\tilde{k}})\;\;\; \tilde{k} \geq -1 .
	\end{align*}
	\;\\
	Now for the solution coefficients we have the following.\\[4pt]
	$\bullet$\;\;\; For $(1,2,3,0),(0,2,3,1), (0,3,3,0)$ we check the matching condition \eqref{exp2-new} and infer $ y^{2\nu (\tilde{l}-2)} = y^{2 \nu}$. Therefore the corresponding coefficients 
	$$ h^{(i)}_{1,m,2,3,\tilde{k}, 0}(y),\; h^{(i)}_{1,m,2,3,\tilde{k}, 0}(y),\; h^{(i)}_{0,m,3,3,\tilde{k}, 0}(y)$$
	`pick up' the above source terms on the right side , i.e. given by $  \tilde{h}^{(i)}_{\tilde{k},m, (ii)},\; i=1,2$. An application of Lemma \ref{lem:inhom1} implies a unique expansion (in all cases) of the form
	\begin{align*}
		&\;\;  h^{(1)}_{\tilde{k},m}(y) = \sum_{s=0}^{m_{\star} -m -1}\log^s(y)  \mathcal{O}(y^{2\nu - 8-2\tilde{k}}) +  \log^{m_{\star} -m }(y) \mathcal{O}(y^{2\nu -4 -2\tilde{k}})\;\;\; \tilde{k} \geq 0,\\ 
		&\;\;  h^{(2)}_{\tilde{k},m}(y) = \sum_{s=0}^{m_{\star} -m -1}\log^s(y)  \mathcal{O}(y^{2\nu -9-2\tilde{k}}) +  \log^{m_{\star} -m }(y) \mathcal{O}(y^{2\nu -5 -2\tilde{k}})\;\;\; \tilde{k} \geq -1 ,
	\end{align*}
	which is consistent with the matching condition \eqref{exp2-new}.\\[4pt]
	$\bullet$\;\;\; For the remaining  4-tuple  $(0,4,5, 0)$  we check the matching condition \eqref{exp2-new} and infer $ y^{2\nu (\tilde{l}-2)} = y^{6\nu}$. Therefore the corresponding coefficient
	$ h^{(i)}_{0,m,4,5,\tilde{k}, 0}(y) $
	does not `pick up' any source terms on the right, but corresponds to an \emph{almost free wave} approximation. Hence by Lemma \ref{lem:hom1} and Corollary \ref{cor:hom-wave} with $ 0 \leq \tilde{k} \leq N_{2}^{(wave)}$ the coefficient expansion is given by
	\begin{align} 
		%& t^{-\nu \beta_1 - \beta_2} \cdot  t^2 \square_S \bigg( t^{\nu \beta_1+\beta_2} (\log(y) + \f12\log(t))^m  \tilde{h}_{m, \beta_1, \beta_2}(y) \bigg),\\ \label{that-relation}
		h^{(1)}_{0, m, 4,5,\tilde{k}-2, 0}(y)   &=  \sum_{s=0}^{m_{\star} -m}\log^s(y)\sum_{r =0}^{ \tilde{k}} \big( \tilde{c}^{(1)}_{s, m, 4,5,\tilde{k}, r} y^{4 \nu - 4 + 2\tilde{k} - 4r} + \tilde{c}^{(2)}_{s,m, 4,5,\tilde{k},r} y^{4 \nu - 6+ 2\tilde{k} - 4r}\big),\\
		h^{(2)}_{0, m, 4,5,\tilde{k}-3, 0}(y)   &=  \sum_{s=0}^{m_{\star} -m} \log^s(y) \sum_{r =0}^{ \tilde{k}} \big( \tilde{d}^{(1)}_{s,m, 3,4,\tilde{k}, r} y^{4 \nu - 5 + 2\tilde{k} - 4r} + \tilde{d}^{(2)}_{s,m, 4,5,\tilde{k}, r} y^{4\nu - 7 + 2\tilde{k} - 4r}\big),
	\end{align} 
	where again $ \tilde{c}^{(i)}_{s,m , 4,5,\tilde{k}, 0},\; \tilde{d}^{(i)}_{s,m,4,5,\tilde{k}, 0} \in \mathbf{C}$ are free to choose and need to be selected from  \eqref{exp2-new}.
	\;\\[20pt]
	\underline{\emph{The general induction step}}. We now formulate the general iteration procedure as follows. The \emph{error functions} of the first and second line of \eqref{schrod-wave-y} after the $j^{\text{th}}$ step are defined to be
	\;\\
	\boxalign{
		\begin{align*}
			e_j^{(schr)}(t,y) : &= i t \partial_t  w_j^{\ast}(t,y) + (\mathcal{L}_S - \alpha_0)w_j^{\ast}(t,y) -  t n_j^{\ast}(t,y) \cdot w_j^{\ast}(t,y),\\
			e_j^{(wave)}(t,y) : &= \Box_S n_j^{\ast}(t,y) - t^{-1}(\partial_y^2 + \frac{3}{y}\partial_y)  \big(t^{-1}|w_j^{\ast}(t,y)|^2\big).
		\end{align*}
	}
	\;\\
	Let us further denote by 
	$e_{j, 0}^{(schr)}(t,y),\;e_{j, 0}^{(wave)}(t,y)$ the \emph{minimal terms with respect to $\nu >1$}. We define this to be the part of the error expansion with $t^{\alpha_1 \nu + \alpha_2},\; t^{\beta_1 \nu + \beta_2} $ having parameters in the set $J_{j+1} \setminus J_j$, i.e. such that $ \alpha_1 = j+1 $ and $\beta_1 = j+1 $ respectively.\\[5pt]
	Then the coefficient functions in the $j^{\text{th}}$ iteration Step are defined via the solutions of
	\;\\
	\boxalign{
		\begin{align} \label{iter-schro}
			&\big(i t \partial_t   + \partial_y^2 + \frac{3}{y}\partial_y - \frac{i}{2}\Lambda- \alpha_0\big)w_j(t,y) = - e_{j-1, 0}^{(schr)}(t,y),\\[4pt] \label{iter-wav}
			&-t^2 \cdot \big(\partial_t - \f12 t^{-1} y \partial_y\big)^2 n_j(t,y) + t \cdot \big(\partial_y^2 + \frac{3}{y}\partial_y\big) n_j(t,y)\\ \nonumber
			&\hspace{4cm}  = - t^2 e_{j-1, 0}^{(wave)}(t,y).
		\end{align}
	}
	where $w_j(t,y), n_j(t,y) $ are as given above, i.e. we recall (note $m_{\star} = s_{\star} + 2n$)
	\begin{align*}
		&w_j(t,y) = \sum_{(n,l, \tilde{l}, p) \in J_{j}^{(schr)}\setminus J_{j-1}^{(schr)}} \sum_{ 3\tilde{l} -4l -2p +1\leq 2\tilde{k} \leq 2N^{(schr)}} t^{\nu\cdot j+ \tilde{k}} \sum_{m = 0}^{m_{\star} }(\log(y) - \nu\log(t))^m g^{(1)}_{m, n, l , \tilde{l}, \tilde{k}, p}(y)\\ \nonumber
		&\hspace{1cm}+ \sum_{(n,l, \tilde{l}, p) \in J_{j}^{(schr)}\setminus J_{j-1}^{(schr)}}  \sum_{ 3\tilde{l} -4l -2p  \leq 2\tilde{k} \leq 2N^{(schr)}} t^{\nu\cdot j + \tilde{k} + \f12} \sum_{m = 0}^{m_{\star} }(\log(y) - \nu\log(t))^m g^{(2)}_{m, n, l , \tilde{l}, \tilde{k}, p}(y),\\[2pt]
		&n_j(t,y)  =  \sum_{(n,l, \tilde{l}, p) \in J_{j}^{(wave)} \setminus J_{j-1}^{(wave)}}	\sum_{ 3\tilde{l} -4l -2p -4 \leq 2\tilde{k} \leq 2N^{(wave)}_2} t^{\nu\cdot j + \tilde{k}} \sum_{m = 0}^{m_{\star} }(\log(y) - \nu\log(t))^m  h^{(1)}_{m, n, l , \tilde{l}, \tilde{k}, p}(y)\\ \nonumber
		&\hspace{1cm}+ \sum_{(n,l, \tilde{l}, p) \in J_{j}^{(wave)}\setminus J_{j-1}^{(wave)}} \sum_{ 3\tilde{l} -4l -2p -5 \leq 2\tilde{k}  \leq 2N^{(wave)}_2}  t^{\nu\cdot j + \tilde{k} + \f12} \sum_{m = 0}^{m_{\star} } (\log(y) - \nu\log(t))^m  h^{(2)}_{m, n, l , \tilde{l}, \tilde{k}, p}(y).
	\end{align*}
	\begin{Rem}\label{rem-ambiguity2}
		We may schematically write this into 
		\begin{align}
			&w_j(t,y) =  \sum_{ l_{\ast} +1\leq 2\tilde{k} \leq 2N^{(schr)}} t^{\nu\cdot j+ \tilde{k}} \sum_{m = 0}^{m_{\ast} }(\log(y) - \nu\log(t))^m \sum_{(n,l, \tilde{l}, p) \in J_{j}^{(schr)}\setminus J_{j-1}^{(schr)}} g^{(1)}_{m, n, l , \tilde{l}, \tilde{k}, p}(y)\\ \nonumber
			&\hspace{1cm}+ \sum_{ -l_{\ast}  \leq 2\tilde{k} \leq 2N^{(schr)}} t^{\nu\cdot j + \tilde{k} + \f12} \sum_{m = 0}^{m_{\ast} + 2n}(\log(y) - \nu\log(t))^m \sum_{(n,l, \tilde{l}, p) \in J_{j}^{(schr)}\setminus J_{j-1}^{(schr)}} g^{(2)}_{m, n, l , \tilde{l}, \tilde{k}, p}(y).
		\end{align}
		where $ l_{\ast}(j),\; m_{\ast}(j) = m(l_{\ast}) \in \Z_+$ are the unique maximal values in $J_{j}\setminus J_{j-1}$ and similar for $n_{j}(t,y)$. This leads us to expect a \emph{family of linear systems} for 
		\[
		\sum_{(n,l, \tilde{l}, p) \in J_{j}^{(schr)}\setminus J_{j-1}^{(schr)}} g^{(1)}_{m, n, l , \tilde{l}, \tilde{k}, p}(y),\;\; \sum_{(n,l, \tilde{l}, p) \in J_{j}^{(schr)}\setminus J_{j-1}^{(schr)}} g^{(2)}_{m, n, l , \tilde{l}, \tilde{k}, p}(y),
		\]
		as observed in the initial steps of the iteration. We clarify this below in the remarks (R.1)-(R.5).
	\end{Rem}
	\;\;\\
	Clearly the right sides of \eqref{iter-schro} and \eqref{iter-wav} contain the correct \emph{interaction terms}, i.e. they coincide with the terms in the expansion of 
	\[
	t \cdot \big( n_{j-1}^{\ast}(t,y) \cdot w_{j-1}^{\ast}(t,y)\big),\;\;(\partial_y^2 + \frac{3}{y}\partial_y)  \big(|w_{j-1}^{\ast}(t,y)|^2\big),
	\]
	corresponding to parameters in $J_{j} \setminus J_{j-1}$. A few remarks are in order.\\[4pt]
	\underline{(R.1)}\;\; As seen above for $ j =1 $ we set $  e_{0, 0}^{(schr)}(t,y) = e_{0, 0}^{(schr)}(t,y) =  0$ and thus $ w_1(t,y), n_1(t,y) $ consist of \emph{free solutions}. \\[5pt]
	\underline{(R.2)}\;\; It became clear above from the initial steps of the iteration, there is a certain (inherent) ambiguity concerning the sets $J_j$ (explained by Remark \ref{rem-ambiguity2}). Namely, $j \in \Z_+, \tilde{k}\in \Z$ \emph{completely determine the parameters} $\mu_{\alpha_1, \alpha_2} \in \mathbf{C}$ in \eqref{inhom-inner-sys-first} - \eqref{inhom-inner-sys-last} (and hence the fundamental solutions) as well as $(\beta_1, \beta_2)$ in \eqref{inhom-inner-wave-first} - \eqref{inhom-inner-wave-last}.\\[2pt]
	\emph{However, of course only the choice of a parameter tuple $(n,l,\tilde{l}, p)$} will \emph{uniquely determine} the asymptotic at $y =0$ according to \eqref{exp1-new} and \eqref{exp2-new}.\\[4pt]
	Since the systems are linear, we \emph{assign} interaction coefficients on the right of \eqref{iter-schro} and \eqref{iter-wav} to certain $g,h$ with parameters in $J_j$ \emph{guided by the required matched expansions} \eqref{exp1-new} and \eqref{exp2-new} (especially via the factor $y^{\nu \cdot b}$).\\[5pt]
	\underline{(R.3)}\;\; Given the above definition of $w_j, n_j$ and the explanation in (R.2), for each fixed $j \in \Z_+$ we thus subsequently solve a \emph{finite family of inhomogeneous linear systems parametrized by}
	\begin{align*}
		&(n,l,\tilde{l}, p) \in J_j^{(schr)},\;  3\tilde{l} - 4l -2p +1 \leq \tilde{k} \leq N_2^{(schr)},\\
		&(n,l,\tilde{l}, p) \in J_j^{(wave)},\;  3\tilde{l} - 4l -2p -4  \leq \tilde{k} \leq N_2^{(schr)},
	\end{align*}
	in the integer case (and similar the half-integer case).\\[5pt]
	\underline{(R.4)}\;\;Fixing $j \in \Z_+$ in the first line \eqref{iter-schro}, \emph{the above parameters do not change the order of the asymptotic of the fundamental solutions } at $y =0$, hence the asymptotic of the solution is individually determined by the source terms. According to the choice in (R.2), we solve each of the inhomogeneous Schr\"odinger systems  via Lemma \ref{lem:inhom-schrod} and running over $\tilde{k} \leq N_2^{(schr)}$.\\[5pt]
	\underline{(R.5)}\;\;Fixing $j \in \Z_+$ in the second line \eqref{iter-wav}, the parameters $j, \tilde{k} \in \Z$ \emph{change the asymptotic of the (fundamental) free wave approximations} in Lemma \ref{lem:hom1} rather directly. According to the choice in (R.2), we sub-iterate via Lemma \ref{lem:inhom1} running over $\tilde{k} \in \Z$ until $\tilde{k} \leq N_2^{(wave)}$.\\[8pt]
	We illustrate the general step of constructing the $j^{\text{th}}$ iterate via the following tables.
	\;\\[3pt]
	
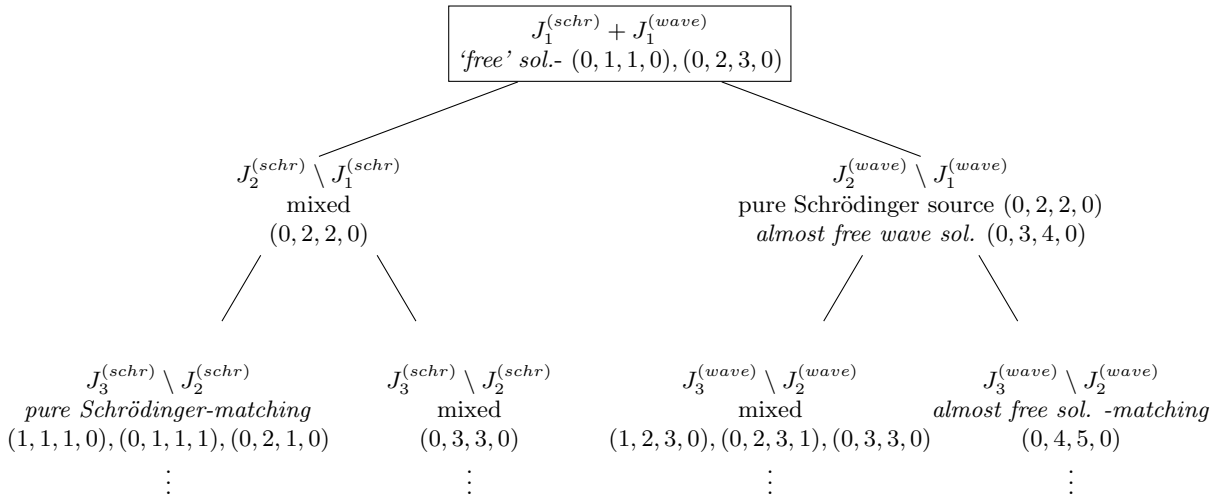
	\begin{table}[h] \footnotesize
		\centering
		\begin{tabular}{ |l|l|}
			\hline
			&\\[-6pt]
			\;\textbf{Step}  $J_j$, \;$j \in \Z_+$ & \;\textbf{Schr\"odinger system}  \eqref{iter-schro}\\[5pt] 	
			&\\[-6pt]
			\hline
			&\\[-6pt]
			\cell{$J_1^{(schr)}$}{pair $(0,1,1, 0)$} & \vtop{\hbox{\strut \emph{free solution} of \eqref{iter-schro}, given by Lemma \eqref{lem:smally1},}\hbox{\strut \emph{free choice of coefficients in expansion according to Lemma \ref{lem:inhom-schrod-practical}}}}\\[5pt] 
			&\\[-6pt]
			\hline
			&\\[-6pt]
			\cell{\cell{$J_j^{(schr)}$,\;$j \in \Z_+$ even}{pair $(n,l,\tilde{l}, p)$,\; $\tilde{l} = 2,4,6,\dots$ ($\leq l$)}}{\cell{$ 2(n+l + p) - \tilde{l} = j,\; 2\tilde{k} \geq 3 \tilde{l} -4l- 2p +1$}{($3 \tilde{l} -4l- 2p$ for half-integer case)}}&\cell{\emph{source terms \& matching cond.} \eqref{exp1-new} contribute $y^{2\nu r}$,\; $ r $ odd}{\cell{$g^{(i)}_{n,m,l,\tilde{l}, \tilde{k}, p}(y)$ \underline{`picks up' previous \emph{wave interactions}}}{\cell{Lemma \ref{lem:inhom-schrod}, Lemma \ref{lem:inhom-schrod-practical}: Assign and solve: \eqref{exp1-new} $\rightarrow$ source + bound. cond.}{b.c: \emph{no free choice of coeff. (fixed by absence of free asympt.)}}}}\\[5pt] 	
			&\\[-6pt]
			\hline
			&\\[-6pt]
			\cell{\cell{$J_j^{(schr)}$,\;$j \in \Z_+$ odd}{pair $(n,l,\tilde{l}, p)$,\; $\tilde{l} = 1,3,5,\dots$ ($\leq l$)}}{\cell{$ 2(n+l+p) - \tilde{l} = j,\; 2\tilde{k} \geq 3\tilde{l} -4l -2p +1$}{ ($3\tilde{l} -4l -2p +2$ for half-integer case)}}& \cell{\cell{\emph{source terms \& matching cond.} \eqref{exp1-new} contribute $y^{2\nu r}$,\; $ r $ even}{$g^{(i)}_{n,m,l,\tilde{l}, \tilde{k}}(y)$ \underline{`picks up' previous \emph{wave interactions} and}}}{\cell{for $ r=0$, $ n +l +p =\frac{j+1}{2}$ \underline{`picks up' pure \emph{Schr\"odinger interactions}}}{\cell{Lemma \ref{lem:inhom-schrod}: Assign and solve: \eqref{exp1-new} $\rightarrow$ source + bound. cond.}{\cell{b.c.: \emph{free choice of coefficients (corresp. to $r=0$)}}{b.c.: \emph{no free choice of coeff. (corresp. to $r>0$)}}}}}\\[5pt]
			&\\
			\hline
		\end{tabular}
	\end{table}
	\begin{figure}[h] \footnotesize
		\begin{tikzpicture}[level distance=15mm, sibling distance=80mm]
			%edge from parent path=
			%	{(\tikzparentnode.south) .. controls +(0,-1) and +(0,1)
				%	.. (\tikzchildnode.north)}]
			\node[draw, align=center]{$J_1^{(schr)} + J_1^{(wave)}$\\ \emph{`free' sol.}- $(0,1,1,0),(0,2,3,0)$}
			child { [sibling distance=40mm, level distance=34mm] node[ align=center]{\;\\\;\;\\\;\;\\\;\\$J_2^{(schr)}\setminus J_1^{(schr)}$\\ mixed\\ $(0,2,2,0)$ }
				child { node[ align=center]{\;\;\\\;\;\\$J_3^{(schr)}\setminus J_2^{(schr)}$\\ \emph{pure Schr\"odinger-matching}\\$(1,1,1,0), (0,1,1,1), (0, 2,1,0)$\\$\vdots$ }}
				child {node[ align=center]{\;\;\\\;\;\\$J_3^{(schr)}\setminus J_2^{(schr)}$\\ mixed\\$(0,3,3,0)$\\$\vdots$ }}
			}
			child {[sibling distance=40mm, level distance=34mm]node[ align=center]{\;\;\\\;\;\\\;\\\;\\$J_2^{(wave)}\setminus J_1^{(wave)}$\\ pure Schr\"odinger source $(0,2,2,0)$\\
					\emph{almost free wave sol.} $(0,3,4, 0)$}
				child {node[ align=center]{\;\;\\\;\;\\$J_3^{(wave)}\setminus J_2^{(wave)}$\\ mixed\\$(1,2,3,0), (0,2,3,1), (0, 3,3,0)$\\$\vdots$ }}
				child {node[ align=center]{\;\;\\\;\;\\$J_3^{(wave)}\setminus J_2^{(wave)}$\\ \emph{almost free sol. -matching}\\$ (0, 4,5,0)$\\$\vdots$ }}
			};
		\end{tikzpicture}
	\end{figure}
	%	\caption{xxx}
	%	\label{tab:singlebest}
	\begin{table} \footnotesize
		\begin{tabular}{ |l|l|}
			\hline
			&\\[-6pt]
			\;\textbf{Step}  $J_j$, \;$j \in \Z_+$ & \;\textbf{wave system} \eqref{iter-wav}\\[5pt] 	
			&\\[-6pt]
			\hline
			&\\[-6pt]
			\cell{$J_1^{(wave)}$}{pair $(0,2,3, 0)$} & \cell{\emph{almost free wave solution} of \eqref{iter-wav},}{ given by Lemma \ref{lem:hom1}, Corollary \ref{cor:hom-wave}, \emph{choice of coeff.}}\\[5pt] 
			&\\[-6pt]
			\hline
			&\\[-6pt]
			\cell{\cell{$J_j^{(wave)}$,\;$j \in \Z_+$ even}{pair $(n,l,\tilde{l}, p)$,\; $\tilde{l} = 2,4,\dots$ ($\leq l+1$)}}{\cell{$ 2(n+l+p) - \tilde{l} = j,\; 2\tilde{k} \geq 3 \tilde{l} -4l -2p -4$}{($3 \tilde{l} -4l -2p -5$ for half-integer case)}}& \cell{\cell{\cell{\emph{source terms \& matching cond.} \eqref{exp2-new} contribute $y^{2\nu r}$,\; $ r $ even}{$h^{(i)}_{n,m,l,\tilde{l}, \tilde{k}, p}(y)$ \underline{`picks up' previous \emph{wave interactions} and}}}{\cell{$ r=0$, $ n +l =\frac{j+2}{2}$ \underline{`picks up' pure \emph{Schr\"odinger interactions}}}{\cell{Assign and solve for $0 < r < j$: \eqref{exp2-new} $\rightarrow$ source + bound. cond.}{Lemma \ref{lem:inhom1}, Lemma \ref{lem:inhom1-update};\;b.c.: \emph{no free choice of coeff. (abs. free wave asympt.)}}}}}{\cell{Case $r =j$,\;$(0, j+1, j+2, 0)$: \emph{add almost free wave $\sim y^{2\nu j}$}}{Lemma \ref{lem:hom1}, Lemma \ref{lem:inhom-schrod-practical};\;b.c.: \emph{free choice of coeff. in hom. solution}}}\\[5pt] 	
			&\\[-6pt]
			\hline
			&\\[-6pt]
			\cell{\cell{$J_j^{(wave)}$,\;$j \in \Z_+$ odd}{pair $(n,l,\tilde{l}, p)$,\; $\tilde{l} = 3,5,\dots$ ($\leq l+1$)}}{\cell{$ 2(n+l+p) - \tilde{l} = j,\; 2\tilde{k} \geq 3 \tilde{l} - 4l -2p -4$}{ ($3 \tilde{l} - 4l -2p -5$ for half-integer case)}}& \cell{\emph{source terms \& matching cond.} \eqref{exp2-new} contribute $y^{2\nu r}$,\; $ r $ odd}{\cell{$h^{(i)}_{n,m,l,\tilde{l}, \tilde{k}, p}(y)$ \underline{`picks up' previous \emph{wave interactions}}}{\cell{Lemma \ref{lem:inhom1}, Lemma \ref{lem:inhom1-update}: Assign and solve: \eqref{exp2-new} $\rightarrow$ source + bound. cond.}{\cell{b.c. $r < j$: \emph{no choice of coeff. (fixed by absence of free asympt.)}}{b.c. $(0, j+1, j+2, 0)$: \emph{free choice of coeff. for almost free wave}}}}}\\[5pt]
			&\\
			\hline
		\end{tabular}
	\end{table}
	
	\begin{Rem}
		In the above tables, we mean by `boundary condition' and `free choice of coefficients' the question whether or not the coefficients (on the bottom) of Lemma \ref{lem:inhom1} (resp. Lemma \ref{lem:inhom1-update}) and Lemma \ref{lem:inhom-schrod} (resp. Lemma \ref{lem:inhom-schrod-practical}) can be \emph{freely selected in this iteration step (fulfilling \eqref{exp1-new} \and \eqref{exp2-new})} \emph{in order to remove the fundamental asymptotic}, or are implicitly determined through previous choice. The coefficients in the former case, collected along the iteration, uniquely determine the process and hence are selected to be those in \eqref{exp1-new} and \eqref{exp2-new} (\emph{matching}).
	\end{Rem}
	\;\\[10pt]
	\underline{\emph{Control of the error}}. From the definition, we conclude the following after the $j^{\text{th}}$ iteration step.
	\begin{align*}
		e_j^{(schr)}(t,y) &= e_{j-1}^{(schr)}(t,y) - e_{j-1,0}^{(schr)}(t,y) - t \cdot \big(\triangle_j(n(t,y)\cdot w(t,y))\big),\\
		e_N^{(wave)}(t,y) &= e_{j-1}^{(wave)}(t,y) - e_{j-1,0}^{(wave)}(t,y) -  t^{-2}\cdot\big(\partial_y^2 + \frac{3}{y}\partial_y\big) \big(\triangle_j|w(t,y)|^2\big),
	\end{align*}
	where
	\begin{align}
		&(\triangle_j F)(w,n) = F(n^{\ast}_{j}, w^{\ast}_{j}) - F(n^{\ast}_{j-1}, w^{\ast}_{j-1}),\\
		&(\triangle_j F)(w) = F(w^{\ast}_{j}) - F(w^{\ast}_{j-1}).
	\end{align}
	Therefore in the error expansions of the form 
	\begin{align*}
		&e_j^{(schr)}(t,y) = \sum_{\alpha_1, \alpha_2,m} t^{\alpha_1\nu + \alpha_2}(\log(y) - \nu \log(t))^{m} e^{(schr)}_{\alpha_1, \alpha_2}(y),\\
		&e_j^{(wave)}(t,y) = \sum_{\beta_1, \beta_2,m} t^{\beta_1\nu + \beta_2}(\log(y) - \nu \log(t))^{m} e^{(wave)}_{\beta_1, \beta_2}(y),
	\end{align*} 
	we sum $ \alpha_1, \alpha_2$ such that either the parameters $(n,l,\tilde{l}, p)$ with $ \alpha_1 = 2(n + l +p) - \tilde{l}$ either
	\begin{itemize}
		\item[(i)] belong to $  J_r^{(schr)}\setminus J_{r-1}^{(schr)}$ for some $r \geq j+1$, i.e. $\alpha_1 = r \geq j+1$,  \emph{or}
		\item[(ii)] they belong to $J_{j}^{(schr)}$ in which case  $\alpha_2 \geq  N_2^{(schr)} -2 $. 
	\end{itemize}
	A similar statement follows of course for $e^{(wave)}_j(t,y)$. The former error type (i) is generated by \emph{interaction terms} of previous iterates and the latter (ii) has two distinct sources.\\[3pt]
	Firstly, these coefficient functions are associated to interactions with parameters in $J_j$, which \emph{are not removed in the inner iterations} over $\tilde{k} \leq N_2$. Secondly, they contain the \emph{errors of the almost free wave approximations} in Lemma \ref{lem:hom1} contributing to the \eqref{iter-wav} wave corrections.\\[2pt]
	We thus control the error by the parameters $(j,N_2^{(schr)}, N_2^{(wave)})$ which will be made precise below.\\[8pt]
	\underline{\emph{Expansion of the error\;\slash\; interaction terms}}. Assume now we have the coefficients corresponding to $J_{j}^{(schr)},\; J_{j}^{(wave)}$ and they admit an expansion consistent with \eqref{exp1-new} and \eqref{exp2-new}. Then we obtain for the Schr\"odinger interaction term on the right side of \eqref{schrod-wave-y}
	\begin{align*}
		t \cdot \big( &n_{j}^{\ast}(t,y) \cdot w_{j}^{\ast}(t,y)\big)\\
		=& \sum_{r \geq 0} \sum_{(n,l,\tilde{l}, p)\in \tilde{J}_r^{(schr)}\setminus \tilde{J}_{r-1}^{(schr)}} \sum_{ 3 \tilde{l} - 4l -2p -1 \leq 2\tilde{k} \lesssim N_2^{(schr)}} t^{\nu(2(n+l+p) -\tilde{l}) + \tilde{k}} \sum_{m = 0}^{m_{\star}}\big(\log(y) - \nu \log(t))^m \tilde{g}^{(1)}_{n,m,l,\tilde{l}, \tilde{k}, p,r}(y)\\
		+& \;\;  \sum_{r \geq 0} \sum_{(n,l,\tilde{l}, p)\in \tilde{J}_r^{(schr)} \setminus \tilde{J}^{(schr)}_{r-1}} \sum_{ 3 \tilde{l} - 4l -2p -2 \leq 2\tilde{k} \lesssim N_2^{(schr)}} t^{\nu(2(n+l+p) -\tilde{l}) + \tilde{k} + \f12} \sum_{m = 0}^{m_{\star}}\big(\log(y) - \nu \log(t))^m \tilde{g}^{(2)}_{n,m,l,\tilde{l}, \tilde{k},p, r}(y)
	\end{align*}
	where the sum is finite (with upper bound $r \leq 2j$) and we set
	\[
	\tilde{J}_r^{(schr)} = \{  (n,l,\tilde{l}, p) \;|\; 2(n+l+ p) - \tilde{l} \leq r,\; l \geq 3,\; \tilde{l} = 3,4,5 \dots, l+1\},\;\;\;\tilde{J}_1 = \tilde{J}_0 = \emptyset.
	\]
	\begin{Rem}
		Concerning the lower bounds of $ \tilde{k} \in \Z$: In the formulation involving $ t^{\nu(2(n+l+p) -\tilde{l}) + \f{\tilde{k}}{2}}$ we calculate $ t \cdot t^{\f{\tilde{k}_1 + \tilde{k}_2}{2}} = t^{\f{\tilde{k}_1 + \tilde{k}_2 +2}{2}}$ where 
		\begin{align*}
			\tilde{k}_1 + \tilde{k}_2 +2 &\geq  3 \tilde{l}_1 - 4l_1 -2p_1 +1 +3 \tilde{l}_2 - 4l_2 -2p_2 -4 +2\\
			& = 3 (\tilde{l}_1 + \tilde{l}_2)- 4(l_1 + l_2)-2(p_1 + p_2) -1.
		\end{align*}
	\end{Rem}
	\;\\
	Further we obtain an absolute expansion of the form
	\begin{align}\label{exp1-new-int-schrod}
		\tilde{g}^{(i)}_{n,m,l,\tilde{l}, \tilde{k}, p}(y) =	\sum_{\tilde{m} = 0}^{m_{\star} -m} \big(\log(y)\big)^{\tilde{m}}\sum_{k \geq  0} y^{2(k - l -n -p -\tilde{k}) + 2\nu(\tilde{l}-3)} \cdot y^{- (i-1)} d^{k,l,n,m}_{\tilde{l},\tilde{k}, \tilde{m}, i, p}.
	\end{align}
	Similarly we obtain for the interaction of the wave part on the right of the lower line in \eqref{schrod-wave-y}
	\begin{align*}
		(\partial_y^2 + &\frac{3}{y}\partial_y) \big(|w_{j}^{\ast}(t,y)|^2\big)\\
		= &\; \sum_{r \geq 0} \sum_{(n,l,\tilde{l},p)\in \tilde{J}_r^{(wave)}\setminus \tilde{J}_{r-1}^{(wave)}} \sum_{ 3\tilde{l} - 4l -2p +2 \leq 2\tilde{k} \lesssim N_2^{(wave)}} t^{\nu(2(n+l+p) -\tilde{l}) + \tilde{k}} \sum_{m = 0}^{m_{\star}}\big(\log(y) - \nu \log(t))^m \tilde{h}^{(1)}_{n,m,l,\tilde{l}, \tilde{k}, p}(y)\\
		+& \;\;  \sum_{r \geq 0} \sum_{(n,l,\tilde{l}, p)\in \tilde{J}_r^{(wave)} \setminus \tilde{J}^{(wave)}_{r-1}} \sum_{ 3\tilde{l} - 4l -2p +1 \leq 2\tilde{k} \lesssim N_2^{(wave)}} t^{\nu(2(n+l+p) -\tilde{l}) + \tilde{k} + \f12} \sum_{m = 0}^{m_{\star}}\big(\log(y) - \nu \log(t))^m \tilde{h}^{(2)}_{n,m,l,\tilde{l}, \tilde{k},r, p}(y),
	\end{align*}
	where analogously we set
	\[
	\tilde{J}_r^{(wave)} = \{  (n,l,\tilde{l}, p) \;|\; 2(n+l+p) - \tilde{l} \leq r,\; l \geq 2,\; \tilde{l} = 2,3,4 \dots, l\},\;\;\;\tilde{J}_1 = \tilde{J}_0 = \emptyset,
	\]
	and we obtain an absolute expansion of the form
	\begin{align}\label{exp1-new-int-wave}
		\tilde{h}^{(i)}_{n,m,l,\tilde{l}, \tilde{k}, p}(y) =	\sum_{\tilde{m} = 0}^{m_{\star} -m} \big(\log(y)\big)^{\tilde{m}}\sum_{k \geq  0} y^{2(k -1- l -n -p-\tilde{k}) + 2\nu(\tilde{l}-2)} \cdot y^{- (i-1)} \tilde{d}^{k,l,n,m}_{\tilde{l},\tilde{k}, \tilde{m}, i, p}.
	\end{align}
	Therefore the error has expansions of the form (let us restrict to the Schr\"odinger integer case for simplicity)
	\begin{align} \label{error-exp}
		&e^{(schr)}_j(t,y)\\ \nonumber
		=& \sum_{r \geq j}  \sum_{(n,l,\tilde{l},p)\in \tilde{J}_{r+1}^{(schr)}\setminus \tilde{J}_{r}^{(schr)}} \sum_{ -3\tilde{l} - 4l -2p +2 \leq 2\tilde{k} \lesssim N_2^{(schr)}} t^{\nu(2(n+l+p) -\tilde{l}) + \tilde{k}} \sum_{m = 0}^{m_{\star}}\big(\log(y) - \nu \log(t))^m e^{(r-j)}_{n,m,l,\tilde{l}, \tilde{k}, p}(y)\\ \nonumber
		&+ \;\; \sum_{(n,l,\tilde{l}, p)\in \tilde{J}_{r}^{(schr)}} \sum_{ N_2^{(schr)} -2 \leq \tilde{k} \lesssim N_2^{(schr)}} t^{\nu(2(n+l+p) -\tilde{l}) + \tilde{k}} \sum_{m = 0}^{m_{\star}}\big(\log(y) - \nu \log(t))^m \tilde{e}_{n,m,l,\tilde{l}, \tilde{k}, p}(y),
	\end{align}
	and similar for the half-integer part as well as $e_j^{(wave)}(t,y)$. We note all coefficient functions have absolute expansions at $ y =0$ of the form \eqref{exp1-new-int-schrod} and \eqref{exp1-new-int-wave}. The coefficients $e^{(0)}_{n,m,k,l,\tilde{l}, \tilde{k}, p}(y)$ are those of the sum $e_{0,j}(t,y)$.
	We  now state the following Proposition finalizing the iteration process.
	\begin{prop}\label{prop:tysystemsolution2} Let $N_2, N_2^{(schr)}, N_2^{(wave)} \in \Z_+$. Then the iteration defined in \eqref{iter-schro} - \eqref{iter-wav} has a solution 
		\begin{align}
			w_{N_2}^{\ast}(t,y)  &= \sum_{j =1}^{N_2} w_j(t,y),\;\;n_{N_2}^{\ast}(t,y)  = \sum_{j =1}^{N_2} n_j(t,y) 
		\end{align}
		where as above
		\begin{align}
			w_j(t,y) &=  \sum_{(n,l, \tilde{l}, p) \in J_{j}^{(schr)}\setminus J_{j-1}^{(schr)}} \sum_{ 3 \tilde{l} -4l -2p +1\leq 2\tilde{k} \leq 2N^{(schr)}_2}t^{j \cdot \nu+ \tilde{k}} \sum_{m = 0}^{m_{\star} }(\log(y) - \nu\log(t))^m g^{(1)}_{m, n, l , \tilde{l}, \tilde{k}, p}(y)\\ \nonumber
			&\;\;+ \sum_{(n,l, \tilde{l}, p) \in J_{j}^{(schr)}\setminus J_{j-1}^{(schr)}} \sum_{ 3 \tilde{l} -4l -2p  \leq 2\tilde{k}  \leq 2N^{(schr)}_2 }t^{j \cdot \nu+ \tilde{k} + \f12} \sum_{m = 0}^{m_{\star}}(\log(y) - \nu\log(t))^m g^{(2)}_{m, n, l , \tilde{l}, \tilde{k}, p}(y),\\[3pt]
			n_{j}(t,y)  &= \sum_{(n,l, \tilde{l}, p) \in J_{j}^{(wave)}\setminus J_{j-1}^{(wave)}}	\sum_{ 3 \tilde{l} -4l -2p -4 \leq 2\tilde{k} \leq 2N^{(wave)}_2} t^{j \cdot \nu+ \tilde{k}} \sum_{m = 0}^{m_{\star}}(\log(y) - \nu\log(t))^m  h^{(1)}_{m, n, l , \tilde{l}, \tilde{k}, p}(y)\\ \nonumber
			&\;\;+ \sum_{(n,l, \tilde{l}, p) \in J_{j}^{(wave)}\setminus J_{j-1}^{(wave)}} \sum_{ 3 \tilde{l} -4l -2p -5 \leq 2\tilde{k}  \leq 2N^{(wave)}_2}  t^{j \cdot \nu+ \tilde{k} + \f12} \sum_{m = 0}^{m_{\star}} (\log(y) - \nu\log(t))^m  h^{(2)}_{m, n, l , \tilde{l}, \tilde{k}, p}(y),
		\end{align}
		(where $m_{\star}(n,l)$ is as above) given by smooth coefficient functions
		$$g^{(1)}_{n,m,l,\tilde{l}, \tilde{k}, p}(y),\; g^{(2)}_{n,m,l,\tilde{l}, \tilde{k}, p}(y),\;\;h^{(1)}_{n,m,l,\tilde{l}, \tilde{k}, p}(y),\; h^{(2)}_{n,m,l,\tilde{l}, \tilde{k}, p}(y), y \in (0, \infty)$$ 
		such that they have an expansion of form \eqref{exp1-new} and \eqref{exp2-new}, i.e. for $0 < y \ll1$ there holds 
		\begin{align}\label{exp-theorem1}
			g^{(i)}_{n,m,l,\tilde{l}, \tilde{k}, p}(y) &= \sum_{\tilde{m}=0}^{s_{\star}} \big(\log(y)\big)^{\tilde{m}} \sum_{k \geq 0} y^{2(k -n -p-l-\tilde{k}) + 2\nu (\tilde{l}-1)}\cdot y^{-(i-1)} c^{k,l,n,m}_{\tilde{k}, \tilde{l},\tilde{m}, i, p},\\ \label{exp-theorem2}
			h^{(i)}_{n,m,l,\tilde{l}, \tilde{k}, p}(y) &= \sum_{\tilde{m}=0}^{s_{\star}} \big(\log(y)\big)^{\tilde{m}} \sum_{k \geq 0} y^{2(k-1 -n -p-l-\tilde{k}) + 2\nu (\tilde{l}-2)}\cdot y^{-(i-1)} \tilde{c}^{k,l,n,m}_{\tilde{k}, \tilde{l},\tilde{m}, i, p},
		\end{align} 
		where $ s_{\star} = m_{\star} -m  $ is as above. In particular, the following coefficients are free to choose. For $ j \in \Z_+,\; j \leq N_2$ we consider the almost free wave for $ (0,j+1, j+2, 0) \in J^{(wave)}_j\setminus J^{(wave)}_{j-1}$ and the coefficients (for  $0 \leq m \leq m_{\star}, 0 \leq \tilde{m} \leq s_{\star}$) of the free solutions of $L_{j, \tilde{k}}$  for the integer \slash half-inters cases,
		\begin{align}\label{wave-coeff}
			&\tilde{c}^{2 \tilde{k},j+1,0,m}_{\tilde{k}, j+2,0,1, 0},\;\tilde{c}^{2 \tilde{k}-1,j+1,0,m}_{\tilde{k}, j+2,0,1, 0},\;\tilde{c}^{2 \tilde{k}+2,j+1,0,m}_{\tilde{k}, j+2,0,2, 0},\; \tilde{c}^{2 \tilde{k}+1,j+1,0,m}_{\tilde{k}, j+2,0,2, 0}  \in \mathbf{C},
		\end{align}
		where $2\tilde{k} \leq 2N_2^{(wave)}$ and $2 \tilde{k} \geq -j-2$ ($ 2 \tilde{k} \geq -j-3$ half-integer case). Further if $j \in 2\Z +1$ is \emph{odd}, then we consider $(n,l,\tilde{l}, p) \in J_j^{(schr)}$ with $n + l +p = \frac{j+1}{2}$, i.e. $(\tilde{l} = 1)$ and the coefficients (leading the  $ \mathcal{O}(1)$ part) 
		\begin{align}\label{schrod-coeff}
			c^{\tilde{k} + \frac{j+1}{2},n,l,0}_{\tilde{k},1,\tilde{m},1, p},\;c^{\tilde{k} + \frac{j-1}{2},n,l,1}_{\tilde{k},1,\tilde{m},1, p} \in \mathbf{C},\; 0 \leq \tilde{m} \leq m_{\star},\;\; (\text{resp.}\;m_{\star} -  1).
		\end{align}
		Prescribing these coefficients gives a unique solution assuming the asymptotic \eqref{exp-theorem1} \& \eqref{exp-theorem2} and further the dependence (of $g,h$ and its $y$-derivatives) on $(\nu, \alpha_0)$ is smooth and further the above error function have an expansion as in \eqref{error-exp}.
	\end{prop}
	\begin{proof} We set up the iteration as described above, starting with the cases of $ J_1^{(schr)}, J_1^{(wave)}$, which contain the parameters 
		\[
		(0,1,1,0), (0,2,3,0)
		\]
		defining \emph{free}, respectively approximately free, solutions of \eqref{inner-sys-first} - \eqref{inner-sys-last} for the Schr\"odinger part, as well as \eqref{inner-wave-first} - \eqref{inner-wave-last} for the wave part. We have already calculated the subsequent steps ( \emph{Step 1} - \emph{Step 5}) up to the coefficients in the iteration corresponding to $J_3^{(schr)}, J_3^{(wave)}$. This will verify the start of the  induction which we define now.\\[4pt]
		\emph{Assume the above  corrections of the iteration} $ n_{\tilde{j}}, w_{\tilde{j}}$ for $ \tilde{j} \leq j-1$ and thus the coefficients
		\[  
		\{ g^{(i)}_{n,m,l,\tilde{l}, \tilde{k}, p}(y), h^{(i)}_{n,m,l,\tilde{l}, \tilde{k}, p}(y) \}
		\]
		%for $ (n,l,\tilde{l}, p) \in J_{j-1}^{(schr)}$ ($ \tilde{k} \leq 2N_2^{(schr)}$) or $ (n,l,\tilde{l}, p) \in J_{j-1}^{(wave)}$ ($ \tilde{k} \leq 2N_2^{(wave)}$) 
		are constructed via coefficients of the matching condition \eqref{exp1-new} \& \eqref{exp2-new} up to some $ j-1 \in \Z_{+}$ and all $ (n,l, \tilde{l}, p) \in J_{j-1}^{(schr)},\; \tilde{k} \lesssim N_2^{(schr)},\; (n,l, \tilde{l}, p) \in J_{j-1}^{(wave)},\; \tilde{k} \lesssim N_2^{(wave)}$ as claimed in the Proposition. In order to solve for the coefficient functions
		\begin{align*}
			\{ g^{(i)}_{n,m,l,\tilde{l}, \tilde{k}, p}(y),\;\;|\;\;(n, l, \tilde{l}, p) \in J_j^{(schr)} \setminus J_{j-1}^{(schr)},\;\;\tilde{k} \lesssim N_2^{(schr)} \},\\
			\{ h^{(i)}_{n,m,l,\tilde{l}, \tilde{k}, p}(y),\;\;|\;\;(n, l, \tilde{l}, p) \in J_j^{(wave)} \setminus J_{j-1}^{(wave)},\;\;\tilde{k} \lesssim N_2^{(wave)} \},
		\end{align*}
		we now basically follow the heuristic argument in the above tables. Therefore let us first consider the source and calculate the lowest powers of $t> 0$ of the interaction terms,  i.e.  for the first (Schr\"odinger) line of \eqref{schrod-wave-y} we need to consider interaction terms in $ 	t \cdot n_{j-1}^* \cdot w_{j-1}^*$ of the form
		\begin{align}
			&t^{j \cdot \nu + \tilde{k} +1} \sum_{ m =0}^{m_{\star}} (\log(y) - \nu \log(t))^{m} \sum_{s = 0}^{m_{\star} -m}\log^s(y) \times\\ \nonumber
			& \hspace{1.5cm} \times \sum_{\substack{\tilde{k}_1 + \tilde{k}_1 = \tilde{k},\\ 1 \leq q  \leq j-1 }}  \sum_{\substack{m_1 + m_2 = m\\ s_1 + s_2 =s }} h_{\tilde{k}_1, j-q,s_1, m_1 }(y)\cdot g_{\tilde{k}_2, q,s_2, m_2 }(y) 
		\end{align}
		On the other hand, in the second (wave) line of \eqref{schrod-wave-y}, we consider term in $ w_{j-1}^* \cdot \overline{w_{j-1}^*}$ of the form
		\begin{align}
			&t^{j \cdot \nu + \tilde{k} } \sum_{ m =0}^{m_{\star}} (\log(y) - \nu \log(t))^{m} \sum_{s = 0}^{m_{\star} -m}\log^s(y) \times\\ \nonumber
			& \hspace{1.5cm} \times \sum_{\substack{\tilde{k}_1 + \tilde{k}_1 = \tilde{k},\\ 1 \leq q  \leq j-1 }}  \sum_{\substack{m_1 + m_2 = m\\ s_1 + s_2 =s }} g_{\tilde{k}_1, j-q,s_1, m_1 }(y)\cdot \overline{g_{\tilde{k}_2, q,s_2, m_2 }(y)}.
		\end{align}
		Of course if the left hand sides are of the integer case, we only need to multiply integer $\times$ integer or half-integer $\times$ half-integer coefficients. In the latter case, in fact, one gains a factor $ \cdot t$ on the left. In case we consider the half-integer terms on the left, we need to use $ t^{j \cdot \nu + \tilde{k} +\f32},\; t^{j \cdot \nu + \tilde{k} +\f12}$ and only multiply integer $\times$ half-integer terms.
		As indicated above, we can write these terms into sums
		\begin{align*}
			&(\text{S-source})\;\;\\
			&\sum_{(n,l,\tilde{l}, p)\in \tilde{J}_j^{(schr)}\setminus \tilde{J}_{j-1}^{(schr)}} \sum_{ 3 \tilde{l} - 4l -2p -1 \leq 2\tilde{k} \lesssim N_2^{(schr)}} t^{\nu(2(n+l+p) -\tilde{l}) + \tilde{k}} \sum_{m = 0}^{m_{\star}}\big(\log(y) - \nu \log(t))^m \tilde{g}^{(1)}_{n,m,l,\tilde{l}, \tilde{k}, p,j}(y)\\
			+& \;\; \sum_{(n,l,\tilde{l}, p)\in \tilde{J}_j^{(schr)} \setminus \tilde{J}^{(schr)}_{j-1}} \sum_{ 3 \tilde{l} - 4l -2p -2 \leq 2\tilde{k} \lesssim N_2^{(schr)}} t^{\nu(2(n+l+p) -\tilde{l}) + \tilde{k} + \f12} \sum_{m = 0}^{m_{\star}}\big(\log(y) - \nu \log(t))^m \tilde{g}^{(2)}_{n,m,l,\tilde{l}, \tilde{k},p, j}(y)\\[8pt]
			&(\text{W-source})\;\;\\
			&\;  \sum_{(n,l,\tilde{l},p)\in \tilde{J}_j^{(wave)}\setminus \tilde{J}_{j-1}^{(wave)}} \sum_{ 3\tilde{l} - 4l -2p +2 \leq 2\tilde{k} \lesssim N_2^{(wave)}} t^{\nu(2(n+l+p) -\tilde{l}) + \tilde{k}} \sum_{m = 0}^{m_{\star}}\big(\log(y) - \nu \log(t))^m \tilde{h}^{(1)}_{n,m,l,\tilde{l}, \tilde{k},  p, j}(y)\\
			+& \;\; \sum_{(n,l,\tilde{l}, p)\in\tilde{J}_j^{(wave)} \setminus \tilde{J}^{(wave)}_{j-1}} \sum_{ 3\tilde{l} - 4l -2p +1 \leq 2\tilde{k} \lesssim N_2^{(wave)}} t^{\nu(2(n+l+p) -\tilde{l}) + \tilde{k} + \f12} \sum_{m = 0}^{m_{\star}}\big(\log(y) - \nu \log(t))^m \tilde{h}^{(2)}_{n,m,l,\tilde{l}, \tilde{k}, p, j}(y),
		\end{align*}
		where for the sum over $(n, l, \tilde{l}, p)$ we use
		\begin{align}
			\tilde{J}_r^{(schr)} = \{  (n,l,\tilde{l}, p) \;|\; 2(n+l+ p) - \tilde{l} \leq r,\; l \geq 3,\; \tilde{l} = 3,4,5 \dots, l+1\},\\
			\tilde{J}_r^{(wave)} = \{  (n,l,\tilde{l}, p) \;|\; 2(n+l+p) - \tilde{l} \leq r,\; l \geq 2,\; \tilde{l} = 2,3,4 \dots, l\},
		\end{align}
		and there holds
		\begin{align}\label{exp1-new-int-schrod-here}
			\tilde{g}^{(i)}_{n,m,l,\tilde{l}, \tilde{k}, p,j}(y) =	\sum_{\tilde{m} = 0}^{m_{\star} -m} \big(\log(y)\big)^{\tilde{m}}\sum_{k \geq  0} y^{2(k - l -n -p -\tilde{k}) + 2\nu(\tilde{l}-3)} \cdot y^{- (i-1)} d^{k,l,n,m}_{\tilde{l},\tilde{k}, \tilde{m}, i, p},\\\label{exp1-new-int-wave-here}
			\tilde{h}^{(i)}_{n,m,l,\tilde{l}, \tilde{k}, p,j}(y) =	\sum_{\tilde{m} = 0}^{m_{\star} -m} \big(\log(y)\big)^{\tilde{m}}\sum_{k \geq  0} y^{2(k - l -n -p-\tilde{k}) + 2\nu(\tilde{l}-2)} \cdot y^{- (i-1)} \tilde{d}^{k,l,n,m}_{\tilde{l},\tilde{k}, \tilde{m}, i, p}.
		\end{align}
		Let us consider what we obtain in the matching conditions in order to solve for these right sides.\\[4pt]
		\emph{Case \rom{1}}. Assume $ j \in \Z $ is odd. Then we have of course for $ (n, l,  \tilde{l}, p) \in J^{(schr)}_{j}$,\;\;$\tilde{l}$\; odd,\; $ l \geq 1$ and 
		\[
		2(n + l +p ) - \tilde{l} = j,\;\; 2\tilde{k} \geq 3 \tilde{l} - 4l -2p +1,\;\;(2\tilde{k} \geq 3 \tilde{l} - 4l -2p ),
		\]
		hence we subsequently consider triples $(n, p,l)$ such that
		\[
		n + p + l \in \{ \frac{j + 1}{2},  \frac{j + 3}{2},  \frac{j + 5}{2}, \frac{j + 7}{2},\dots \}.
		\]
		Note then for such $(n_1, l_1, p_1)$ and $ \tilde{l}_1$ there holds
		\[
		n_1 + p_1 + l_1 = \frac{j + \tilde{l}_1}{2} =  \frac{j + \tilde{l}_1 + 2}{2} -1 =  n_2 + p_2 + l_2 -1, 
		\]
		where $ (n_2, l_2, \tilde{l}_1 +2, p_2) \in \tilde{J}_j^{(schr)}$.\\
		$\bullet$\;Thus we now \emph{assign} pairs on the right (which appear in the source) to corresponding pairs on the left. The identity shows that integrating the source expansion \eqref{exp1-new-int-schrod-here} for $ (n_2, l_2, \tilde{l}_1 +2, p_2) $ by Lemma \ref{lem:inhom-schrod}, Lemma \ref{lem:inhom-schrod-practical}, which is essentially multiplying by $ y^2 $, is consistent with  \eqref{exp1-new} for $ (n_1, l_1, \tilde{l}_1 , p_1)$.\\
		$\bullet$\; One special requirement, namely $ \mathcal{O}(1)$ for $ s = m_{\star} -m$ and $ \mathcal{O}(y^{-2})$ for $ s = m_{\star} - m -1$ arises in case $ \tilde{l}_1 =1$ (pure Schr\"odinger interaction)  such that we can solve via Lemma \ref{lem:inhom-schrod-practical} without extra  logarithmic powers `correcting'   degeneracy. This is inductively (see also the iteration start above) inferred for \eqref{exp1-new-int-schrod-here} by calculating the above product expressions. We obtain `free' coefficients as in Lemma \ref{lem:inhom-schrod-practical}, which are selected from \eqref{exp1-new}.\\[5pt]
		Now for $(n,l,\tilde{l}, p) \in J_{j}^{(wave)}$ we have 
		\[
		2(n + l +p ) - \tilde{l} = j,\;\; 2\tilde{k} \geq 3 \tilde{l} - 4l -2p -4,\;\;(2\tilde{k} \geq 3 \tilde{l} - 4l -2p -5 ),
		\]
		Note, similar to the above, for such $(n_1, l_1, p_1)$ and $ \tilde{l}_1 \geq 2$ there holds
		\[
		n_1 + p_1 + l_1 = \frac{j + \tilde{l}_1}{2}  =  n_2 + p_2 + l_2 , 
		\]
		where $ (n_2, l_2, \tilde{l}_1 , p_2) \in \tilde{J}_j^{(wave)}$.\\
		$\bullet$\;Thus we now \emph{assign} pairs on the right (which appear in the source) to corresponding pairs on the left. The identity shows that applying $ \partial_y^2 + 3 y^{-1} \partial_y$ as in \eqref{inhom-inner-wave-first} - \eqref{inhom-inner-wave-last}, then integrating this differentiated expansion  \eqref{exp1-new-int-wave-here} for $ (n_2, l_2, \tilde{l}_1, p_2) $ by Lemma \ref{lem:inhom1}, Lemma \ref{lem:inhom1-update}, which is essentially multiplying by $ y^{-2} $, is consistent with  \eqref{exp2-new} for $ (n_1, l_1, \tilde{l}_1 , p_1)$.\\
		$\bullet$\; Another (main) difference to the above $J_j^{(schr)}$ case is that we may choose $\tilde{l}_1 = l_1 +1 = j+2$, which is missed by the above relation. Then $ l_1 = j+1$ and in \eqref{exp2-new} we have $ \cdot y^{2j\cdot  \nu}$. This corresponds to an almost free wave as in Lemma \ref{lem:hom1} and Corollary \ref{cor:hom-wave}, where the coefficients are free to be selected from \eqref{exp2-new}.\\[4pt]
		\emph{Case \rom{2}}. Assume $ j \in \Z $ is even,  $ (n, l,  \tilde{l}, p) \in J^{(schr)}_{j}$,\;\;$\tilde{l}$\; even,\; $ l \geq 1$ and 
		\[
		2(n + l +p ) - \tilde{l} = j,\;\; 2\tilde{k} \geq 3 \tilde{l} - 4l -2p +1,\;\;(2\tilde{k} \geq 3 \tilde{l} - 4l -2p ),
		\]
		hence we have
		\[
		n + p + l \in \{ \frac{j + 2}{2},  \frac{j + 4}{2},  \frac{j + 6}{2}, \frac{j + 8}{2},\dots \}.
		\]
		Note that in particular the matching condition will always require a factor $ y^{2\nu r},\; r \neq 0$ and hence, similar as before from
		\[
		n_1 + p_1 + l_1 = \frac{j + \tilde{l}_1}{2} =  \frac{j + \tilde{l}_1 + 2}{2} -1 =  n_2 + p_2 + l_2 -1, 
		\]
		where $ (n_2, l_2, \tilde{l}_1 +2, p_2) \in \tilde{J}_j^{(schr)}$, we obtain a unique solution by Lemma \ref{lem:inhom-schrod}, consistent with \eqref{exp1-new}. Similar for  $(n,l,\tilde{l}, p) \in J_{j}^{(wave)}$ we have  again
		\[
		2(n + l +p ) - \tilde{l} = j,\;\; 2\tilde{k} \geq 3 \tilde{l} - 4l -2p -4,\;\;(2\tilde{k} \geq 3 \tilde{l} - 4l -2p -5 ),
		\]
		where $ \tilde{l} = 2,4,6,\dots$ and for such $(n_1, n_2, n_3)$ and $ \tilde{l}_1 \geq 2$ there holds
		\[
		n_1 + p_1 + l_1 = \frac{j + \tilde{l}_1}{2}  =  n_2 + p_2 + l_2 , 
		\]
		where $ (n_2, l_2, \tilde{l}_1 , p_2) \in \tilde{J}_j^{(wave)}$.\\
		$\bullet$\;Thus we now \emph{assign} pairs on the right (which appear in the source) to corresponding pairs on the left and solve \eqref{inhom-inner-wave-first} - \eqref{inhom-inner-wave-last} by Lemma \ref{lem:inhom1}, Lemma \ref{lem:inhom1-update}, uniquely as above. This again fails for the maximal $\tilde{l}_1 = j+2$, which corresponds to an \emph{almost free wave} by Lemma \ref{lem:hom1} and Corollary \ref{cor:hom-wave}.\\
		$\bullet$\; One special requirement, namely $ \mathcal{O}(1)$ for $ s = m_{\star} -m$ and $ \mathcal{O}(y^{-2})$ for $ s = m_{\star} - m -1$ arises in case $ \tilde{l}_1 =2$ (\emph{pure Schr\"odinger passes through wave interaction})  such that we obtain a solution as in  Lemma \ref{lem:inhom1-update} which is important for subsequent steps of odd Schr\"odinger type $J_j^{(schr)}$ as explained above. This is again inductively inferred for \eqref{exp1-new-int-wave-here} (in fact on has to use this and the above observation to deduce both of these special requirements inductively).
	\end{proof}
	\begin{Rem} 
		(i) \;The selection of the free coefficients in Proposition \ref{prop:tysystemsolution2} corresponds to selecting initial values for the $y-$dependent factors in the linear systems of Lemma \ref{lem:inhom1} and Lemma \ref{lem:inhom-schrod}, Lemma \ref{lem:inhom-schrod-practical}.\\ 	
		(ii)\;Note that for the Schr\"odinger coefficients \eqref{schrod-coeff}, only the integer case $ i =1$ is free and can in principle be matched to the homogeneous solutions in Lemma \ref{lem:inhom-schrod}, lemma \ref{lem:inhom-schrod-practical}, respectively Lemma \ref{lem:smally1} (since the expansion must be of even order).
	\end{Rem}
	\begin{cor}\label{Cor-matching-sol}
		There exists a unique solution as in Proposition \ref{prop:tysystemsolution2} matching the expansions  \eqref{exp1-new} and \eqref{exp2-new} with the coefficients taken from the interior $(t,a,R)$ expansion where $ 1 \ll a \ll R$. 
	\end{cor}
	\begin{Rem}
		Note running over $\tilde{k} \in \Z$ for the case of $ \tilde{l} =1$ in the matched solution of Corollary \ref{Cor-matching-sol} is only pathological, as the $a-$dependence of the interior expansion is constant (so only one $\tilde{k} \in \Z$ is non-trivial here). We however do not make this distinction for simplicity.
	\end{Rem}
	Let us remark that we can always write $ w_{N_2}^{\ast}(t,y) , n_{N_2}^{\ast}(t,y) $ into
	\begin{align*}
		w_{N_2}^{\ast}(t,y)  &= \sum_{(n,l, \tilde{l}, p) \in J_{N_2}^{(schr)}} \sum_{ 3\tilde{l} - 4l -2p +1\leq \tilde{k} \leq N^{(schr)}_2}t^{\nu(2(n + l + p) - \tilde{l})+ \frac{\tilde{k}}{2}} \sum_{m = 0}^{m_{\star} }(\log(y) - \nu\log(t))^m g_{m, n, l , \tilde{l}, \tilde{k}, p}(y)\\ 
		n_{N_2}^{\ast}(t,y)  &=  \sum_{(n,l, \tilde{l}, p) \in J_{N_2}^{(wave)}}	\sum_{ 3\tilde{l} - 4l -2p -4 \leq \tilde{k} \leq N^{(wave)}_2} t^{\nu(2(n + l + p) - \tilde{l})+ \frac{\tilde{k}}{2}} \sum_{m = 0}^{m_{\star} }(\log(y) - \nu\log(t))^m  h_{m, n, l , \tilde{l}, \tilde{k}, p}(y),
	\end{align*}
	where 
	\begin{align*}
		g_{m, n, l , \tilde{l}, \tilde{k}, p}(y) = 
		\begin{cases}
			g^{(1)}_{m, n, l , \tilde{l}, \frac{\tilde{k}}{2}, p}(y) & \tilde{k} \in 2\Z\\
			g^{(2)}_{m, n, l , \tilde{l}, \frac{\tilde{k}-1}{2}, p}(y) & \tilde{k} \in 2\Z +1\\
		\end{cases}
		\;\;\;
		h_{m, n, l , \tilde{l}, \tilde{k}, p}(y) = 
		\begin{cases}
			h^{(1)}_{m, n, l , \tilde{l}, \frac{\tilde{k}}{2}, p}(y) & \tilde{k} \in 2\Z\\
			h^{(2)}_{m, n, l , \tilde{l}, \frac{\tilde{k}-1}{2}, p}(y) & \tilde{k} \in 2\Z +1\\
		\end{cases}
	\end{align*}
	We now define the self-similar approximate solution, which we later may change to back to $(t,a,R)$ coordinates.
	\begin{Def}[Self-similar approximation]\label{defn:self-similar} Let $N_2, N_2^{(schr)}, N_2^{(wave)} \in \Z_+ $ be large. Then we set for $ t \in (0, t_0),\; y \in (0,\infty)$ with $ 0 < t_0 \ll1$
		\begin{align}
			&\tilde{u}_{S}^{N_2}(t,y) := 	w_{N_2}^{\ast}(t,y) = w_1(t,y) + w_2(t,y) + \dots w_{N_2}(t,y),\\
			&n_{S}^{N_2}(t,y) := 	n_{N_2}^{\ast}(t,y) =	n_1(t,y) + n_2(t,y) + \dots n_{N_2}(t,y),
		\end{align}
		where the functions are from Corollary \ref{Cor-matching-sol}. The corresponding error is defined as
		\begin{align}
			e_S^{w,N_2}(t,y) : &= i t \partial_t \tilde{u}_{S}^{N_2}(t,y) + (\partial_y^2 + \frac{3}{y}\partial_y - \frac{i}{2}\Lambda - \alpha_0)\tilde{u}_{S}^{N_2}(t,y)\\ \nonumber
			&\;\;\; -  t \cdot n_{N_2}^{\ast}(t,y) \cdot \tilde{u}_{S}^{N_2}(t,y),\\
			e_S^{n, N_2}(t,y) : &= -\big(\partial_t - \f12 t^{-1} y \partial_y\big)^2 n_{N_2}^{\ast}(t,y)  + t^{-1} \cdot \big(\partial_y^2 + \frac{3}{y}\partial_y\big) n_{N_2}^{\ast}(t,y)\\ \nonumber
			&\;\;\;- t^{-1}(\partial_y^2 + \frac{3}{y}\partial_y)  \big(t^{-1}|\tilde{u}_{S}^{N_2}(t,y)|^2\big).
		\end{align}
	\end{Def}
	We now want to compare these to the approximation of the interior region, i.e. we consider 
	\begin{align}
		&t^{-\f12}\tilde{u}_{S}^{N_2}(t,y) = t^{-\f12}\tilde{u}_{S}^{N_2}(t,t^{\nu}R) =t^{-\f12}\tilde{u}_{S}^{N_2}(t,t^{\f12}a) ,\;\; R = t^{-\f12 -\nu}r = t^{-\nu}y\\[2pt]
		&n_{S}^{N_2}(t,y) = n_{S}^{N_2}(t,R t^{\nu}) = n_{S}^{N_2}(t, t^{\f12}a),\;\; a = t^{-1}r = t^{-\f12}y,\\[2pt]
		&\lambda(t)u_{\text{app}, \mathcal{I}}^{N_1}(t,a, R) = \lambda(t)u_{\text{app}, \mathcal{I}}^{N_1}(t,y t^{-\f12}, y t^{-\nu}),\;\;\lambda^2(t)n_{\text{app}, \mathcal{I}}^{N_1}(t,a, R) = \lambda^2(t)n_{\text{app}, \mathcal{I}}^{N_1}(t,y t^{-\f12}, y t^{-\nu}),
	\end{align}
	and verify that these solutions only differ by a \emph{small error with fast decay rate} as $t \to 0^+$ in the \emph{overlap region} 
	$$ \mathcal{I}\cap \mathcal{S}= \{ (t,r) \;|\; c_2^{-1} t^{\f12 + \epsilon_1} \leq r \leq c_1 t^{\f12 + \epsilon_1}\}.$$
	In particular, we will choose the latter of the above representations (last line) and formulate
	\begin{lem}[Consistency in $\mathcal{I}\cap \mathcal{S}$]\label{lem:consistency-inner} Let $N \in \Z_+$, then there exists (by subsequent choice) $N_1^{(1)}, N_1^{(2)}, \mathcal{N}, N_2, N_2^{(schr)}, N_2^{(wave)} $ in $ \Z_+$ large such that for $m \in \Z_{\geq 0}$
		\begin{align*}
			&| \partial_R^{m}\big( t^{-\f12}\tilde{u}_{S}^{N_2}(t,t^{\nu}R) - \lambda(t)u_{\text{app}, \mathcal{I}}^{N_1}(t,R t^{-\f12 + \nu}, R)  \big) | \leq C_{m,N_1,\mathcal{N}, N_2} R^{-m} t^{\nu \cdot N},\\
			&| \partial_R^{m}\big( n_{S}^{N_2}(t,t^{\nu}R) - n_{\text{app}, \mathcal{I}}^{N_1}(t,R t^{-\f12 + \nu}, R)  \big) | \leq \tilde{C}_{m,N_1,\mathcal{N}, N_2} R^{-m} t^{\nu \cdot N}, %% oder brauche ich hier $ \lambda^2(t)n_{\text{app}, \mathcal{I}}^{N_1}(t,R t^{-\f12 + \nu}, R)$ statt dessen? --ich glaube nicht, lambda^2 is da schon drin
		\end{align*}
		for some $ C, \tilde{C} > 0$ and all $(t,R)$ with $ (t,r) \in \mathcal{I}\cap \mathcal{S}$ provided we have $ 0 < t \ll1$ small.
	\end{lem}
	\begin{proof} The proof follows from a uniqueness statement of the corresponding iteration in the $(t,y)$ coordinate frame. 
		By Lemma \ref{lem: induct}, Definition \ref{defn:X-space} and the asymptotic of the S-space in Definition \ref{defn:SQbetal}, for $ (t,R)$ such that $(t,r) \in \mathcal{I} \cap \mathcal{S}$, the function $ \lambda(t)u_{\text{app}, \mathcal{I}}^{N_1}(t,R t^{-\f12 + \nu}, R) $ has an expansion of the form ( $ a = \frac{r}{t}$)
		\begin{align*}
			& \lambda(t)u_{\text{app}, \mathcal{I}}^{N_1}(t,R t^{-\f12 + \nu}, R)\\
			&  = t^{- \f12 - \nu } \sum_{j = 1}^{N_1} \sum_{k + l \in I_j}\frac{t^{2\nu k}}{(t \lambda)^{2l}} \sum_{r, p \geq 0}  \sum_{s = 0}^{2r + s_*} \sum_{\tilde{s}=0}^{\tilde{s}_*} \sum_{ (\tilde{l}, \tilde{k}) \in S_l} (t \lambda)^{- 2p}R^{2k - 2r} \log^s(R) a^{2\nu \tilde{l} + \tilde{k}} \log^{\tilde{s}}(a) ~c^{k,l,r,s}_{\tilde{k}, \tilde{l}, \tilde{s}, p}\\
			& = t^{- \f12 } \sum_{j = 1}^{N_1} \sum_{k + l \in I_j}t^{2\nu(k + l + p - \f12) - l -p} \times\\
			&\hspace{2cm}\times \sum_{ (n, l , \tilde{l}, p) \in J_{\tilde{N}}^{(schr)}}  \sum_{s = 0}^{2n + s_*} \sum_{\tilde{s}=0}^{\tilde{s}_*} \sum_{ \tilde{k} \leq 4l -1 -3 \tilde{l} +2p} R^{2k - 2n} \log^s(R) a^{2\nu (\tilde{l}-1) + \tilde{k} - 2l -2p} \log^{\tilde{s}}(a) \cdot c^{k,l,r,s}_{\tilde{k}, \tilde{l}, \tilde{s}, p}\\
			&\;\;+ t^{- \f12 } \sum_{j = 1}^{N_1} \sum_{k + l \in I_j}t^{2\nu(k + l + p - \f12) - l -p} \times\\
			&\hspace{2cm}\times \sum_{ (n, l , \tilde{l}, p)\;:\; 2(n + l +p) - \tilde{l} > \tilde{N}}  \sum_{s = 0}^{2n + s_*} \sum_{\tilde{s}=0}^{\tilde{s}_*} \sum_{ \tilde{k} \leq 4l -1-3 \tilde{l} +2p} R^{2k - 2n} \log^s(R) a^{2\nu (\tilde{l}-1) + \tilde{k} - 2l -2p} \log^{\tilde{s}}(a) \cdot c^{k,l,r,s}_{\tilde{k}, \tilde{l}, \tilde{s}, p}\\
			& =:  t^{- \f12}\tilde{u}_{\tilde{N}, \mathcal{I}}^{N_1}(t,R t^{-\f12 + \nu}, R) + T^{\tilde{N}}_{\mathcal{I}, 1},
		\end{align*}
		in an absolute sense where we define the latter sum (tail) in the second last line as $ T^{\tilde{N}}_{\mathcal{I}, 1}$ and  we used
		\[
		S_l  = \{  (\tilde{l}, \tilde{k}) \in \Z^2 \;|\; 0 \leq \tilde{l} \leq l-1 ,\;\; \tilde{k} \leq 2l -4 - 3 \tilde{l} \}.
		\]
		Similarly we can write for the wave part
		\begin{align*}
			&n_{\text{app}, \mathcal{I}}^{N_1}(t,R t^{-\f12 + \nu}, R)\\
			& = \sum_{j = 1}^{N_1} \sum_{k + l \in I_j}t^{2\nu(k + l + p -1) - l -p - 1} \times\\
			&\hspace{2cm}\times \sum_{ (n, l , \tilde{l}, p) \in J_{\tilde{N}}^{(schr)}}  \sum_{s = 0}^{2n + s_*} \sum_{\tilde{s}=0}^{\tilde{s}_*} \sum_{ \tilde{k} \leq 4l +4 -3 \tilde{l} +2p} R^{2k - 2n} \log^s(R) a^{2\nu (\tilde{l}-1) + \tilde{k} - 2l -2p} \log^{\tilde{s}}(a) \cdot c^{k,l,r,s}_{\tilde{k}, \tilde{l}, \tilde{s}, p}\\
			&\;\;+  \sum_{j = 1}^{N_1} \sum_{k + l \in I_j}t^{2\nu(k + l + p - 1) - l -p-1} \times\\
			&\hspace{2cm}\times \sum_{ (n, l , \tilde{l}, p)\;:\; 2(n + l +p ) - \tilde{l} > \tilde{N}}  \sum_{s = 0}^{2n + s_*} \sum_{\tilde{s}=0}^{\tilde{s}_*} \sum_{ \tilde{k} \leq 4l -1-3 \tilde{l} +2p} R^{2k - 2n} \log^s(R) a^{2\nu (\tilde{l}-2) + \tilde{k} -2 - 2l -2p} \log^{\tilde{s}}(a) \cdot c^{k,l,r,s}_{\tilde{k}, \tilde{l}, \tilde{s}, p},
		\end{align*}
		where the right side will be denoted by $ \tilde{n}_{\text{app}, \mathcal{I}}^{N_1}(t,R t^{-\f12 + \nu}, R) + T^{\tilde{N}}_{\mathcal{I}, 2}$. Let us set 
		\[
		S(f,g) : = (i \partial_t f + \Delta f - t \cdot g \cdot f, \; \Box g - \Delta (|f|^2)).
		\]
		Then by definition we have
		\begin{align*}
			(\lambda^3 e^{N_1, (1)}_{\text{app}, \mathcal{I}}, e^{N_1, (2)}_{\text{app}, \mathcal{I}} ) &= S( \lambda(t)u_{\text{app}, \mathcal{I}}^{N_1}(t,R t^{-\f12 + \nu}, R),n_{\text{app}, \mathcal{I}}^{N_1}(t,R t^{-\f12 + \nu}, R))\\
			&  = S(t^{- \f12}\tilde{u}_{\tilde{N}, \mathcal{I}}^{N_1}+ T^{\tilde{N}}_{\mathcal{I}, 1},  \tilde{n}_{\text{app}, \mathcal{I}}^{N_1} + T^{\tilde{N}}_{\mathcal{I}, 2})\\
			& = S(t^{- \f12}\tilde{u}_{\tilde{N}, \mathcal{I}}^{N_1},  \tilde{n}_{\text{app}, \mathcal{I}}^{N_1} ) + \tilde{S}_1,
		\end{align*}
		where the latter is defined through this identity. Now we exchange for $ R = (t^{\f12}\lambda) y,\;\; a = t^{- \f12}y$, as in the matching argument, but  we can rigorously  commute the inner and outer sum 
		\begin{align*}
			&\tilde{u}_{\tilde{N}, \mathcal{I}}^{N_1} = \sum_{ (n, l , \tilde{l}, p) \in J_{\tilde{N}}^{(schr)}} t^{\nu(2n+ 2l + 2p - \tilde{l})  - \frac{\tilde{k}}{2}}   \sum_{s = 0}^{m_{\star}} (\log(y) - \nu \log(t))^{s} \sum_{\tilde{s}=0}^{m_{\star} -s} \log^{\tilde{s}}(y) \sum_{ \tilde{k} \leq 4l -1 -3 \tilde{l} +2p}\times\\
			& \hspace{2cm} \times    \sum_{k = 0}^{c N_1}y^{2k -2n-2l-2p + 2\nu (\tilde{l}-1) + \tilde{k}} \cdot \tilde{c}^{k,l,r,s}_{\tilde{k}, \tilde{l}, \tilde{s}, p},
		\end{align*}
		and similar for $ \tilde{n}_{\text{app}, \mathcal{I}}^{N_1} $. Note the proof of Lemma \ref{lem: induct} implicitly shows the above latter sum  converges absolutely over $ k \in \Z_+$ for $ 0 < y = t^{- \f12} r \ll1$, i.e. $ (t,r) \in \mathcal{I} \cap \mathcal{S}$. We thus complete the above expression to an infinite sum $\tilde{\tilde{u}}$ (over $ k \in \Z_{\geq 0}$) by subtracting the corresponding tail $ T^{cN_1}_{\mathcal{S},1}$ (and for $\tilde{n}$ we complete to $ \tilde{\tilde{n}} $ similarly with $ T^{\tilde{c}N_1}_{\mathcal{S},2}$). Thus
		\begin{align} \label{dies}
			(\lambda^3 e^{N_1, (1)}_{\text{app}, \mathcal{I}}, e^{N_1, (2)}_{\text{app}, \mathcal{I}} )  = S(t^{- \f12}\tilde{u}_{\tilde{N}, \mathcal{I}}^{N_1},  \tilde{n}_{\text{app}, \mathcal{I}}^{N_1} ) + \tilde{S}_1 = S(t^{- \f12}\tilde{\tilde{u}}_{\tilde{N}},  \tilde{\tilde{n}}_{\tilde{N}}) + \tilde{S}_1 + \tilde{S}_2,
		\end{align}
		where again the latter term $ S_2$ is defined over this identity. Note by $ R \sim t^{\epsilon_1 - \nu}, y \sim t^{\epsilon_1}$, we have 
		$
		T^{\tilde{N}}_{\mathcal{I},1} = O(t^{\nu \tilde{N}}),\; T^{\tilde{c}N_1}_{\mathcal{S},1} = O(t^{\epsilon_1 C N_1})
		$ and similar for the wave part. Hence also $ S_1 = O(t^{\nu C_1\tilde{N}}) $ and $ S_2 = O(t^{\epsilon_1 C_2 N_1})$.
		Changing the $S(f,g)$ operator into $(t,y)$ coordinates, we observe the set up of the above $(t,y)$ iteration for $ t^{- \f12} \tilde{\tilde{u}}$ and $ \tilde{\tilde{n}}$ for which in particular:\\[4pt]
		$\bullet$\; The additional error terms in \eqref{dies} are of high order and irrelevant for arbitrarily many (parameter controlled) steps of the iteration. Therefore, until  the error terms $(e^{(1)}, e^{(2)}), S_1, S_2 $ show up as coefficients of powers of $t> 0$, the systems to be solved are the same as above.\\[4pt]
		$\bullet$\; Hence, since the matching expansion converges here, it gives us the only approximate solution satisfying the matching coefficients. Thus it coincides with the $g$ coefficients in $ \mathcal{I} \cap \mathcal{S}$ by Proposition \ref{prop:tysystemsolution2}.\\[5pt]
		Now we are sure that \emph{all} of the $ \tilde{c}^{k,l,r,s}_{\tilde{k}, \tilde{l}, \tilde{s}, p}$  coefficients given by the \emph{inner expansion} must define the truncated sum of the $g$ (and $h$) functions in $\mathcal{I}\cap \mathcal{S}$. Hence
		\begin{align} \label{dies2}
			&t^{-\f12}\tilde{u}_{S}^{N_2}(t,t^{\nu}R) - \lambda(t)u_{\text{app}, \mathcal{I}}^{N_1}(t,R t^{-\f12 + \nu}, R)\\
			&= t^{-\f12}\tilde{u}_{S}^{N_2}(t,t^{\nu}R) - t^{-\f12} \tilde{\tilde{u}}_{\tilde{N}} + \tilde{T}^{\tilde{N}}_{\mathcal{I},1} + \tilde{T}^{c N_1}_{\mathcal{I},1},
		\end{align}
		for which the term $ t^{-\f12}\tilde{u}_{S}^{N_2}(t,t^{\nu}R) - t^{-\f12} \tilde{\tilde{u}}_{\tilde{N}}  = O(t^{\epsilon_1  C( \min\{N_1,\tilde{N}, N_2\} )})$ because of the above uniqueness statement. The claimed bound follows directly by differentiating the identity \eqref{dies2}.
	\end{proof}
	\;\\\;\\
	\subsection{Extension to the region $r \lesssim t^{\frac12-\epsilon_2}$:  Asymptotics for large $y \gg1$}
	\;\\\;\\[5pt]
	Previously we explained how to obtain effective approximate solutions of the system \eqref{schrod-wave-y} in the $(t, y)$-coordinate frame going beyond $r \lesssim t^{\frac12+\epsilon_1}$. In particular the new approximation is consistent with Section \ref{sec:inner} by Lemma \ref{lem:consistency-inner} where $ 0 < y \lesssim t^{\epsilon_1} \ll1$ and extends into the full self-similar region where $ y \in (0, \infty)$, i.e. into the set
	$$  \mathcal{S} = \{ (r,t)\;|\; t^{\frac12+\epsilon_1} \lesssim r \lesssim t^{\f12 - \epsilon_2}\}. $$
	\emph{Strategy}. We next consider the asymptotic behavior of this approximations near the corresponding opposite end, i.e. where $ (t,r) \in \mathcal{S}$ and in $(t,y)$ coordinates $1 \ll y \lesssim t^{-\epsilon_2}$. If we have a precise description of this \emph{oscillatory regime}, we will change back into $(t, r)$-coordinates for identifying the \emph{radiation field} separating at $t =0$.  More specifically, we need to give $ y \gg1$ \emph{asymptotic expansions} of the previous coefficient functions in Proposition \ref{prop:tyoscillatory}.\\[10pt]
	We first recall the approximate solutions constructed in the latter  Proposition, i.e.  they have the form
	\begin{align}
		w_j^*(t,y)	=&\;\sum_{(n,l, \tilde{l},p) \in J_{j}^{(schr)}} \sum_{ 3 \tilde{l} -4l -2p +1\leq 2\tilde{k} \leq 2N^{(schr)}} t^{\nu(2n +2l +2p- \tilde{l}) + \tilde{k}}g^{(1)}_{n, l , \tilde{l}, \tilde{k},p}(t, y)\\ \nonumber
		& \hspace{1cm}+  \sum_{(n,l, \tilde{l},p) \in J_{j}^{(schr)}}  \sum_{ 3 \tilde{l} -4l -2p   \leq 2\tilde{k} \leq 2N^{(schr)}} t^{\nu(2n +2l +2p- \tilde{l}) + \tilde{k} + \f12} g^{(2)}_{n, l , \tilde{l}, \tilde{k},p}(t, y),\\[4pt]
		n_j^*(t,y)	=&\;\sum_{(n,l, \tilde{l},p) \in J_{j}^{(wave)} }	\sum_{ 3 \tilde{l} -4l -2p -4 \leq 2\tilde{k} \leq 2N^{(wave)}_2} t^{\nu(2n +2l+2p - \tilde{l}) + \tilde{k}} h^{(1)}_{n, l , \tilde{l}, \tilde{k},p}(t, y) \\ \nonumber
		&\hspace{1cm}+ \sum_{(n,l, \tilde{l},p) \in J_{j}^{(wave)}} \sum_{ 3 \tilde{l} -4l -2p -5\leq 2\tilde{k}  \leq 2N^{(wave)}_2}  t^{\nu(2n +2l +2p- \tilde{l}) + \tilde{k} + \f12} h^{(2)}_{n, l , \tilde{l}, \tilde{k},p}(t, y).
	\end{align}
	Then, if we define the numbers\footnote{equality holds if $ k $ is odd and otherwise we have equality to $3-2k,\; -1+2k$}
	\begin{align*}
		&\ell^{(s)}_k  : = \min\{ 3 \tilde{l} -4l -2p +1\;|\; (n,l,\tilde{l},p) \in J_{k}^{(schr)} \setminus J_{k-1}^{(schr)}  \}  \geq 2 -2k,\\
		&\ell^{(w)}_k  : = \min\{ 3 \tilde{l} -4l -2p -4\;|\; (n,l,\tilde{l},p) \in J_{k}^{(wave)} \setminus J_{k-1}^{(wave)}  \}  \geq  -2-2k,
	\end{align*}
	we can rewrite (setting some of the $g^{(i)}_{n, l , \tilde{l}, \tilde{k},p}, h^{(i)}_{n, l , \tilde{l}, \tilde{k},p}$ functions to zero)
	\boxalign[15cm]{
		\begin{align}
			w_j^*(t,y)	=&\;\sum_{k=1}^j \sum_{\ell^{(s)}_k \leq 2\tilde{k} }t^{\nu \cdot k + \tilde{k}}  \sum_{(n,l, \tilde{l},p) \in J_{k}^{(schr)}\setminus J_{k-1}^{(schr)}} g^{(1)}_{n, l , \tilde{l}, \tilde{k},p}(t, y)\\ \nonumber
			& \hspace{1cm}+ \sum_{k=1}^j \sum_{\ell^{(s)}_k  \leq 2\tilde{k}+1 }t^{\nu \cdot k + \tilde{k} + \f12}  \sum_{(n,l, \tilde{l},p) \in J_{k}^{(schr)}\setminus J_{k-1}^{(schr)}} g^{(2)}_{n, l , \tilde{l}, \tilde{k},p}(t, y),\\[4pt]
			n_j^*(t,y)	=&\;\sum_{k=1}^j \sum_{\ell^{(w)}_k \leq 2\tilde{k} }t^{\nu \cdot k + \tilde{k}}  \sum_{(n,l, \tilde{l},p) \in J_{k}^{(wave)}\setminus J_{k-1}^{(wave)}} h^{(1)}_{n, l , \tilde{l}, \tilde{k},p}(t, y)\\ \nonumber
			&\hspace{1cm}+\sum_{k=1}^j \sum_{\ell^{(w)}_k \leq 2\tilde{k} +1 }t^{\nu \cdot k + \tilde{k}+ \f12}  \sum_{(n,l, \tilde{l},p) \in J_{k}^{(wave)}\setminus J_{k-1}^{(wave)}} h^{(2)}_{n, l , \tilde{l}, \tilde{k},p}(t, y).
		\end{align}
	}
	In particular, we now identify the terms 
	\begin{align}\label{identify-large}
		\sum_{(n,l, \tilde{l},p) \in J_{k}^{(schr)}\setminus J_{k-1}^{(schr)}} g^{(i)}_{n, l , \tilde{l}, \tilde{k},p}(t, y),\;\;\;\ \sum_{(n,l, \tilde{l},p) \in J_{k}^{(wave)}\setminus J_{k-1}^{(wave)}} h^{(i)}_{n, l , \tilde{l}, \tilde{k},p}(t, y),
	\end{align}
	to be the relevant coefficients (instead of the individual $(n,l,\tilde{l}, p)$ dependent functions) for the  asymptotic description as $ y\to \infty$ in the subsequent calculations. In fact we will now simply write
	\;\\
	\boxalign[15cm]{
		\begin{align}\label{diese-line-1}
			w_N^*(t,y)	=&\;\sum_{k=1}^N \sum_{\ell^{(s)}_k \leq 2\tilde{k} }t^{\nu \cdot k + \tilde{k}} g^{(1)}_{k, \tilde{k}}(t, y) + \sum_{k=1}^N \sum_{\ell^{(s)}_k  \leq 2\tilde{k} +1 }t^{\nu \cdot k + \tilde{k} + \f12}  g^{(2)}_{k, \tilde{k}}(t, y)\\[4pt] \label{diese-line2}
			n_N^*(t,y)	=&\;\sum_{k=1}^N \sum_{\ell^{(w)}_k \leq 2\tilde{k} }t^{\nu \cdot k + \tilde{k}}  h^{(1)}_{k, \tilde{k}}(t, y) +\sum_{k=1}^N \sum_{\ell^{(w)}_k \leq 2\tilde{k} +1 }t^{\nu \cdot k + \tilde{k}+ \f12} h^{(2)}_{k, \tilde{k}}(t, y).
		\end{align}
	}
	\;\\
	I order to proceed, it is convenient (c.f. \cite[Lemma 2.6]{Perelman} and e.g. \cite[Section 2.2.1]{schmid}) for the coefficients in \eqref{diese-line-1} to change to the  representation
	\begin{align*}
		g^{(i)}_{k, \tilde{k}}(t,y) = &\;\sum_{m =0}^{m_{\star}} \big(\log(t) - \nu \log(y)\big)^m g^{(i) }_{k, \tilde{k}, m}( y)\\
		=&\; \sum_{m =0}^{m_{\star}} \big(\log(t) - \f12 \log(y)\big)^m g^{(i), - }_{k, \tilde{k}, m}( y)  + \sum_{m =0}^{m_{\star}} \big(\log(t) + \f12 \log(y)\big)^m g^{(i), + }_{k, \tilde{k}, m}( y), 
	\end{align*}
	and similar for the wave part \eqref{diese-line2}, at least for the homogeneous solutions to \eqref{hom-inner}. Further below, in fact we change to the ansatz
	\;\\
	\boxalign[15cm]{
		\begin{align}
			&g^{(i) }_{k, \tilde{k}}(t, y) = \sum_{m_1 + m_2 \leq m_{\star}}(\log(y) +\f12\log(t))^{m_1} (\log(y) - \f12\log(t))^{m_2}g^{(i)}_{m_1, m_2, k, \tilde{k}}( y),\\[4pt]
			& h^{(i) }_{k, \tilde{k}}(t, y) = \sum_{m_1 + m_2 \leq m_{\star}}(\log(y) +\f12\log(t))^{m_1} (\log(y) - \f12\log(t))^{m_2} h^{(i)}_{m_1, m_2, k, \tilde{k}}( y).
		\end{align}
	}
	\;\\
	\tinysection[0]{The scalar equation}
	We start by revising the \emph{homogeneous solutions} of \eqref{schrod-wave-y} in the form \eqref{hom-inner} and \eqref{hom-wave1}. However, \emph{we now focus on the large $y \gg1$ asymptotic regime, which leads to oscillatory behavior}. Recall the operator
	\begin{align}
		\mathcal{L}_S + \mu_{\alpha_1, \alpha_2}=  \partial_{y}^2 + \frac{3}{y}\partial_y - \frac{i}{2}\Lambda +  \big(-\alpha_0  + i(\nu\alpha_1+\alpha_2)\big),
	\end{align}
	where $\Lambda = 1 + y \partial_y $. Then we refer to the description of a fundamental base in \cite{schmid}, which in fact is given by standard theory, c.f.  \cite[Theorem 2.1]{Olver},  of \emph{asymptotic solutions} for singular ODE. To be precise the ansatz
	\begin{align}
		&(\mathcal{L}_S + \mu)\big (e^{\frac{i}{4} y^2} f( \tfrac{i}{4} y^2) \big) =  0
	\end{align}
	reduces to solving Kummer's equation in $ z = \f{i}{4} y^2 $, i.e.
	\begin{align} \label{Kummer}
		&z \cdot \partial_z^2 f+ \big( 2 +z \big) \partial_z f + \big( \frac{3}{2} - i \mu\big)f = 0.
	\end{align}
	Hence by the application of \cite[Theorem 2.1]{Olver}, we obtain two well-known independent solutions
	%\begin{align}
	%&f^{(d)}_+(z) \sim z^{- i \mu - \frac{d+2}{4}} \big(\sum_{k \geq 0}a_k^{+}z^{-k}\big),~ %|\arg(z)| < \tfrac32 \pi - \delta,~ \delta > 0,\\
	%&f^{(d)}_-(z) \sim e^{ z}(-z)^{ i \mu - \frac{d-2}{4}} \big(\sum_{k \geq 0} a_k^{-}  z^{-k}\big),~~ z \to \infty,~~%|\arg(-z)| < \tfrac32 \pi - \delta,
	%\end{align}
	\begin{align} \label{ersteFL-abstrac}
		&f_+(z) = z^{ i \mu - \frac{3}{2}} \big(a_0^+ + O(z^{-1})\big),~~z \to \infty,\\[4pt] %|\arg(z)| < \tfrac32 \pi - \delta,~ \delta > 0,\\ 
		\label{zweiteFL-abstrac}
		&f_-(z) = e^{ -z}(-z)^{- i \mu - \frac{1}{2}} \big(a_0^- + O(z^{-1})\big)~~ z \to \infty,~~%|\arg(-z)| < \tfrac32 \pi - \delta,
	\end{align}
	\begin{Rem} Recall we use the $ O$-notation in the sense of  \emph{asymptotic series} expansion, i.e. $ f(z) = O(g(z))$ as $ z \to \infty$, if and only if $ |g(z)| > 0$ where $|z| \gg1$ and for all $ N \in \Z_{\geq 0}$ there exist $ a_0, a_1, a_2,\dots, a_N \in \mathbf{C}$ such that
		\begin{align}
			\forall\; j \in \Z_{\geq 0}\;\;\; \lim_{|z| \to \infty }z^{N + j} \partial^j_z\bigg(  f(z) \cdot g^{-1}(y) - \sum_{k =0}^N a_k z^{-k}\bigg) = 0.
		\end{align}
		Note that in the entire section such functions will depend analytically on the parameter $ \mu \in \mathbf{C}$. This follows since (1):
		All such expressions are obtained from algebraic transformations of \eqref{ersteFL-abstrac} \& \eqref{zweiteFL-abstrac} evaluated at $z = \f{i}{4} y^2$ and (2): The solutions to the equation \eqref{Kummer} have this property, c.f. \cite[Theorem 2.1, Theorem 3.1]{Olver} and \cite[Appendix A]{schmid}. 
	\end{Rem}
	Next, we sum this up in the following Lemma.
	\begin{Lemma}\label{FS-Inner-infty-y} Let  $ \mu \in \mathbf{C} $. Then $ \mathcal{L}_S + \mu$ has fundamental solutions  $ \phi_{\infty}(\mu,y),~ \psi_{\infty}(\mu,y)$ for $ y \in (0, \infty)$  smooth in $ \mu,y$ and such that
		\begin{align}
			&\phi_{\infty}(\mu,y) =  y^{-2i \mu - 1} \big(1 + O(y^{-2})\big),\\
			& \psi_{\infty}(\mu,y) = e^{ \frac{i}{4}y^2 } y^{2 i\mu- 3} \big(1 + O(y^{-2})\big).
		\end{align}
	\end{Lemma}
	\begin{proof} Follows directly from the above statement taken from \cite{Olver} setting $ z = i \frac{y^2}{4}$.  Infact it is easily obtained by iterating finite steps of an elementary recursion using  the ansatz	\begin{equation}\label{eq:gammasoln}
			w_1(y) = e^{i\frac{y^2}{4}}\cdot y^{\gamma}\cdot \big(1+c_{-2}y^{-2}+c_{-4}y^{-4}+\ldots\big)
		\end{equation}
		with 
		\[
		\big( \partial_{y}^2 - \frac{i}{2} y \partial_y\big)e^{i\frac{y^2}{4}} = O(1). 
		\]
		Hence this leads to the relation 
		\begin{align*}
			\big(\mu - \frac{i}{2}\Lambda + \partial_{y}^2 + \frac{3}{y}\partial_y\big)(e^{i\frac{y^2}{4}}\cdot y^{\gamma}) = O(y^{\gamma-2}).
		\end{align*}
		from which we infer $
		\mu - \frac{i}{2}(1 + \gamma) + i(2 + \gamma) = \mu  + \frac{3i}{2} + \frac{i}{2}\gamma = 0 $ 
		and so $\gamma = 2 i \mu  -3$. 
		For the higher order corrections, one inductively solves the equations
		\begin{align*}
			&\big(\mu  - \frac{i}{2}\Lambda + \partial_{y}^2 + \frac{3}{y}\partial_y\big)\big(c_l e^{i\frac{y^2}{4}}\cdot y^{\gamma-2l}\big) \\&= d_l e^{i\frac{y^2}{4}}\cdot y^{\gamma-2l} + O(y^{\gamma-2l-2}),\,d_{l+1} = \big[(\gamma - 2l)(\gamma-2l-1)+3(\gamma-2l)\big]c_l,
		\end{align*}
		whence $ 	(-il)c_l = d_l$. 
		Clearly the coefficients $c_l$ grow rapidly ($\sim C^l l!$), and so the expansion for $w_1$ is only asymptotic 
		\begin{align*}
			w_1(y) = e^{i\frac{y^2}{4}}\cdot y^{\gamma}\cdot (1+c_{-2}y^{-2}+c_{-4}y^{-4}+\ldots + c_{-2l}y^{-2l}) + O(y^{\gamma-2l-2}). 
		\end{align*}
		Next, in order to find the other solution toward $y = +\infty$, we make the ansatz
		\begin{equation}\label{eq:gammatildesoln}
			w_2(y) = y^{\tilde{\gamma}}(1+d_{-2}y^{-2} + d_{-4}y^{-4}+ \ldots), 
		\end{equation}
		where the sum is again asymptotic. Here we obtain the leading order condition 
		\begin{align*}
			\mu - \frac{i}{2} - \frac{i}{2}\tilde{\gamma} = 0, \;\;\text{thus}\;\;\;\;	\tilde{\gamma} =  - 2i \mu - 1. 
		\end{align*}
		%which results in 
		%	\begin{equation}\label{eq:tildegamma}
			%	\tilde{\gamma} = 2(\nu\alpha_1+ \alpha_2) - 1 + 2i\alpha_0. 
			%	\end{equation}
		The remaining coefficients $d_{-2l}$ are again determined inductively. 
	\end{proof}
	\begin{Rem}
		In fact, for any compact subset $ K \subset \C,\;k \in \Z_{\geq 0},~ n \in \Z_+$ there are positive constants  $ C > 0, y_0> 0$ (depending on $K,k,n$) with estimates
		\begin{align}\label{one}
			&\sup_{\mu \in K}\big | \partial_y^{l}\big(y^{ 2 i \mu} \phi_{\infty}(\mu,y) - \sum_{m = 0}^{n-1} c^{(1)}_m y^{-1 - 2m}\big)\big | \leq C y^{-1-2n-l},~~ y \geq y_0,\\ \label{two}
			&\sup_{\mu \in K}\big | \partial_y^{l}\big( y^{-2 i \mu} e^{ \frac{-i y^2}{4}} \psi_{\infty}(\mu,y) - \sum_{m = 0}^{n-1} c^{(2)}_my^{-3 - 2m}\big)\big | \leq C y^{-3 -2n-l},~~y \geq y_0,
		\end{align} 
		for $  l = 0,\dots, k$ and where $ c_0^{(i)} = 1 $.
	\end{Rem}
	%Considering the homogeneous part of the solutions of $ \mathit{L}_S \cdot \mathbf{A} = -\mathbf{e}^{0}$ in the  $ y \to \infty $ limit,  we now use a representation,  originally suggested in [P], which is more convenient for extending the approximate solution to the final region where $ r \gtrsim  (T-t)^{- \epsilon}$. 
	~~\\ 
	Clearly particular solutions of the linear equation $ (\mathcal{L}_S + \mu)u = f$  now have the form
	\begin{align}
		&u(y) = \beta^{(1)} \phi_{\infty}(y) + \beta^{(2)} \psi_{\infty}(y)  + \int_y^{\infty} G(y,s) w^{-1}(s) f(s)~ds,\\
		&G(y,s) = \phi_{\infty}(y) \psi_{\infty}(s) -  \phi_{\infty}(s) \psi_{\infty}(y)
	\end{align} 
	with $ w^{-1}(s)  \sim s^{3}e^{- i \frac{s^2}{4}}$ and in case, say $ f(y) \sim y^{- 3 -} (1 +O(y^{-2}))$ as  $ y \to \infty$. Otherwise we have to fix $ y_0 > 0$  and  consider 
	\begin{align}\label{general-sol-large-y-scal}
		\tilde{\beta}^{(1)} \phi_{\infty}(y) + \tilde{\beta}^{(2)}\psi_{\infty}(y)  +  \int_{y_0}^{y} G(y,s) w^{-1}(s) f(s)~ds.
	\end{align}
	\begin{Lemma} Let $\mu \in \mathbf{C}$. Then the problem 
		\[
		\big(\mathcal{L}_S + \mu\big)u(y) = e^{im \frac{y^2}{4}} y^{\zeta} \cdot  O(y^{-\beta}),\;\; y \gg1,
		\]
		where $ m,  \in \Z, \zeta \in \mathbf{C},\; \beta \in \R$ has a solution of the form \eqref{general-sol-large-y-scal} and such that
		\[
		u(y) = e^{im \frac{y^2}{4}} y^{\zeta} \cdot  O(y^{-\beta}),\;\;y \gg1,\;\;\;\text{if}\;\; \text{Re}(\zeta) - \beta \neq -1, -3,
		\]
		and where the coefficients depend on $\mu$.
	\end{Lemma}
	\tinysection[0]{Homogeneous systems} 
	Let us now return to the general schematic expansion
	\[
	w(t, y) = \sum_{\alpha_1}   \sum_{\alpha_2 \geq \alpha_{2\star}}\sum_{m = 0}^{m_{\star}} t^{\nu\alpha_1+\alpha_2} (\log(y) - \nu\log(t))^m\cdot g^{(i)}_{m, \alpha_1,\alpha_2}(y),
	\]
	where the sum over $ \alpha_1 \geq \alpha_{1 \star} $, $ \alpha_2 \in \Z$ (or $ \Z + \f12$, $i =2$), $ \alpha_{2} \geq \alpha_{2\star}(\alpha_1)$ are finite and $m_{\star} = m_{\star}(\alpha_1)$.
	In order to capture the asymptotic behavior near $y = +\infty$, we pass to the new equivalent formulation: 
	\begin{align}
		w(t, y) = &\sum_{\alpha_1} \sum_{\alpha_2 \geq \alpha_{2\star}}   t^{\nu\alpha_1+\alpha_2} \sum_{m = 0}^{m_{\star}} \big[ (\log(y) - \f12\log(t))^m\cdot g^{(i), -}_{m, \alpha_1,\alpha_2}(y)\\ \nonumber
		&\hspace{4cm} + (\log(y) + \f12\log(t))^m\cdot g^{(i), +}_{m, \alpha_1,\alpha_2}(y)\big].
	\end{align}
	The coefficients $g^{(i), \pm}_{m,\alpha_1,\alpha_2}(y)$ will  consist of a `linear' contribution and contributions coming from the interaction terms. In particular, we clearly use this representation in a case for which the  minimal coefficient $t^{\nu\alpha_{1\star} + \alpha_{2\star}}$ corresponds to a free solution. Thus for such $\alpha_1, \alpha_2$ we write
	\begin{align}\label{eq:g1plus}
		&g^{(1), +}_{m,\alpha_1,\alpha_2}(y) =y^{\tilde{\gamma}}( d_ {m}+d_{m,-2}y^{-2} + d_{m,-4}y^{-4}+\ldots) = y^{\tilde{\gamma}}( d_ {m}+ O(y^{-2})),\\
		\label{eq:g1minus}
		&g^{(1), -}_{m,\alpha_1,\alpha_2}(y) =  e^{i\frac{y^2}{4}}y^{\gamma}(c_m+ c_{m,-2}y^{-2} + c_{m,-4}y^{-4}+\ldots) = e^{i\frac{y^2}{4}}y^{\gamma}(c_m + O(y^{-2})),
	\end{align}
	where $\tilde{\gamma}, \gamma$ may be determined as in the proof of Lemma \ref{FS-Inner-infty-y}. We first consider the systems corresponding to 
	\begin{align*}
		&\big( it\partial_t + \mathcal{L}_S - \alpha_0\big)\bigg( t^{\nu\alpha_1+\alpha_2} \sum_{m =0}^{m_{\star}}(\log(y) \pm\frac12\log(t))^m\cdot g^{(i),\pm}_{m,\alpha_1,\alpha_2}(y)\bigg) = 0,
	\end{align*} 
	hence calculating the commutator 
	\begin{align*}
		&\bigg[ it\partial_t + \mathcal{L}_S - \alpha_0, t^{\nu\alpha_1+\alpha_2}(\log(y) \pm\frac12\log(t))^{m}\cdot g^{(i), \pm}_{m,\alpha_1,\alpha_2}(y) \bigg]\\
		& =  t^{\nu\alpha_1+\alpha_2}y^{-2}\big[m(m-1)(\log(y) \pm \frac12\log(t))^{m-2} + 2 m(\log(y) \pm\frac12\log(t))^{m-1}\big]g^{(i),\pm}_{m,\alpha_1,\alpha_2}(y)\\
		& \;\;\;+  t^{\nu\alpha_1+\alpha_2} (- \f i 2 \pm \f i 2 )m (\log(y) \pm \frac12\log(t))^{m-1} g^{(i), \pm }_{m,\alpha_1,\alpha_2}(y)\\
		& \;\;\;+ 2m t^{\nu\alpha_1+\alpha_2}y^{-1}(\log(y)  \pm \frac12\log(t))^{m-1}\partial_y g^{(i), \pm }_{m,\alpha_1,\alpha_2}(y), 
	\end{align*}
	the coefficients $ g^{(i),\pm}_{m,\alpha_1,\alpha_2}(y) $ must satisfy 
	\boxalign[14cm]{
		\begin{align}\label{new-first}
			(\mathcal{L}_S + \mu)g^{(i),\pm}_{m_{\star},\alpha_1,\alpha_2}(y) &=\; 0\\[4pt] \label{new-second}
			(\mathcal{L}_S +\mu)g^{(i),\pm}_{m_{\star }-1,\alpha_1,\alpha_2}(y)&= \; - m_{\star} D^{\pm}_yg^{(i),\pm}_{m_{\star},\alpha_1,\alpha_2}(y)\\[4pt] \label{new-last}
			(\mathcal{L}_S + \mu)g^{(i),\pm}_{m,\alpha_1,\alpha_2}(y) &=\;  - (m+1) D^{\pm}_y g^{(i),\pm}_{m+1,\alpha_1,\alpha_2}(y)\\[4pt] \nonumber 
			&\hspace{1.5cm}- (m+2)(m+1) \tfrac{1}{y^2} g^{(i),\pm}_{m+2,\alpha_1,\alpha_2}(y)\\[4pt] \nonumber
			0 \leq m \leq m_{\star} -2,~~~~~&
		\end{align}
	}
	where $ \mu = - \alpha_0 + i(\nu \alpha_1 + \alpha_2)$ and
	\[
	D^{\pm}_y =  - \frac{i}{2} \pm \f i 2 + \frac{2}{y^2} +  \frac{2}{y} \partial_y.
	\]
	Solving \eqref{new-first} - \eqref{new-last} simultaneously for $ \pm $ leads to the following, c.f. also \cite[Lemma 2.27]{schmid} and \cite[Lemma 2.6]{Perelman} (reference (2.61) -(2.64)),
	\begin{lem}\label{lem:yoscillatorylinear} Given the coefficients $c_m, d_m \in \mathbf{C}$ for  $0\leq m\leq m_{\star}$, there exists a unique solution of 
		%	\[
		%\big(-\alpha_0 - \frac{i}{2} + it\partial_t - \frac{i}{2}y\partial_y + \partial_{yy} + \frac{3}{y}\partial_y\big)w = 0
		%	\]
		%	with 
		\begin{align}
			&\big( it\partial_t + \mathcal{L}_S - \alpha_0\big)\bigg( t^{\nu\alpha_1+\alpha_2} \sum_{m =0}^{m_{\star}}\big[(\log(y) +\frac12\log(t))^m\cdot g^{(i),+}_{m,\alpha_1,\alpha_2}(y)\\ \nonumber
			&\hspace{5.5cm} +  (\log (y) -\frac12\log(t))^m\cdot g^{(i),-}_{am\alpha_1,\alpha_2}(y)\big] \bigg) = 0,
		\end{align}
		for which $g^{(i), \pm}_{m,\alpha_1,\alpha_2}(y)$ have the form  \eqref{eq:g1plus} and \eqref{eq:g1minus}, that is
		\begin{align}
			&g^{(i),+}_{m,\alpha_1,\alpha_2}(y) =  y^{-2i \mu -1}\big( d_{m} + O(y^{-2}) \big),\\
			&g^{(i),-}_{m,\alpha_1,\alpha_2}(y) = e^{\f i 4 y^2} y^{2i \mu -3}\big( c_{m} + O(y^{-2}) \big).
		\end{align}
	\end{lem}
	\begin{proof} We follow an induction for the systems \eqref{new-first} - \eqref{new-last} on $m$, starting with $m = m_{\star}$. Clearly in the first line \eqref{new-first} we pick \begin{align*}
			g^{(1),+}_{m_{\star},\alpha_1,\alpha_2}(y) =  d_{m_{\star}} \cdot \phi_{\infty}(y),\;\;g^{(1)-}_{m_{\star},\alpha_1,\alpha_2}(y) = c_{m_{\star}} \cdot \psi_{\infty}(y),
		\end{align*}
		to satisfy \eqref{eq:g1plus} and \eqref{eq:g1minus}. 
		\begin{comment}
			Then we formally obtain
			\begin{align*}
				&\big(-\alpha_0 - \frac{i}{2} + it\partial_t - \frac{i}{2}y\partial_y + \partial_{yy} + \frac{3}{y}\partial_y\big)\big[ t^{\nu\alpha_1+\alpha_2}(\log y +\frac12\log t)^{a_*}\cdot g^{(1+)}_{a_*,\alpha_1,\alpha_2}(y)\big]\\
				& =  t^{\nu\alpha_1+\alpha_2}(\log y +\frac12\log t)^{a_*}\mathcal{L}_{\nu\alpha_1+\alpha_2}g^{(1+)}_{a_*,\alpha_1,\alpha_2}(y)\\
				& +  t^{\nu\alpha_1+\alpha_2}y^{-2}\big[a_*(a_*-1)(\log y +\frac12\log t)^{a_*-2} + 3a_*(\log y +\frac12\log t)^{a_*-1}\big]g^{(1+)}_{a_*,\alpha_1,\alpha_2}(y)\\
				& + 2a_* t^{\nu\alpha_1+\alpha_2}y^{-1}(\log y +\frac12\log t)^{a_*-1}\partial_y g^{(1+)}_{a_*,\alpha_1,\alpha_2}(y)\\
			\end{align*}
			where we set 
			\[
			\mathcal{L}_{\nu\alpha_1+\alpha_2} = \alpha_0 - \frac{i}{2} + i(\nu\alpha_1+\alpha_2) - \frac{i}{2}y\partial_{y} + \partial_{yy} + \frac{3}{y}\partial_y
			\]
			whence $\mathcal{L}_{\nu\alpha_1+\alpha_2}g^{(1+)}_{a_*,\alpha_1,\alpha_2} =0$. 
			Now we add a correction of the form 
			\[
			t^{\nu\alpha_1+\alpha_2}(\log y +\frac12\log t)^{a_*-1}\cdot g^{(1+)}_{a_*-1,\alpha_1,\alpha_2}(y)
			\]
			to eliminate the error 
			\begin{align*}
				&t^{\nu\alpha_1+\alpha_2}y^{-2}\cdot  3a_*(\log y +\frac12\log t)^{a_*-1}g^{(1+)}_{a_*,\alpha_1,\alpha_2}(y)\\&\hspace{3cm} +  2 t^{\nu\alpha_1+\alpha_2}y^{-1}(\log y +\frac12\log t)^{a_*-1}\partial_y g^{(1+)}_{a_*,\alpha_1,\alpha_2}(y). 
			\end{align*}
		\end{comment}
		The error on the right of the second line \eqref{new-second} (for the $+$ case) then reads
		\begin{align*}
			- m_{\star} D^+_y g^{(i),+}_{m_{\star},\alpha_1,\alpha_2}(y) = &- m_{\star} y^{-2\mu i -3}\big(d_{m_{\star}} +  d_{m_{\star}, -2} y^{-2} + \dots \big)\\
			& +  m_{\star} (2\mu i + 1 )y^{-2\mu i -3}\big(d_{m_{\star}} +  d_{m_{\star}, -2} y^{-2} + \dots \big),
		\end{align*}
		where we note $ D^+_y f(y) = \frac{2}{y^2} \Lambda f(y)$. Therefore we pick 
		\begin{align*}
			g^{(i), +}_{m_{\star}-1,\alpha_1,\alpha_2}(y) =& d_{m_{\star}-1}\cdot \phi_{\infty}(y)  +  \tilde{d}_{m_{\star} -1} y^{- 2 \mu i -3}\big(1+\tilde{d}_{m_{\star}-1,-2}y^{-2} +\dots\big)\\
			& =:\tilde{g}^{(i), +}_{m_{\star}-1,\alpha_1,\alpha_2}(y) + \tilde{\tilde{g}}^{(i), +}_{m_{\star}-1,\alpha_1,\alpha_2}(y)
		\end{align*}
		where the coefficients $ \tilde{d}_{m_{\star}-1} , \tilde{d}_{m_{\star}-1,-2}, \dots$ are determined inductively such that 
		\begin{align*}
			\big(\mathcal{L}_S + \mu\big)\tilde{\tilde{g}}^{(i), +}_{m_{\star}-1,\alpha_1,\alpha_2} =  -m_{\star} (2\mu i) y^{-2\mu i -3}\big(d_{m_{\star}} +  d_{m_{\star}, -2} y^{-2} + \dots \big).
		\end{align*}
		Here we exploit that by definition of $\mu = - \alpha_0 + i(\nu \alpha_1 + \alpha_2)$ we have 
		\begin{align*}
			\big(\mathcal{L}_S + \mu\big)y^{- 2i \mu -1 -2l} = il y^{- 2i \mu -1 -2l} + O(y^{- 2i \mu -3 -2l}),  
		\end{align*}
		and hence the corresponding recursive system is non-degenerate. At this point we have calculate the right side of $\eqref{new-last}$ for $ m = m_{\star}-2$, hence
		\begin{align*}
			& - (m_{\star} -1)D^+_y g^{(i), +}_{m_{\star}-1,\alpha_1,\alpha_2}(y) - m_{\star}(m_{\star} -1) y^{-2} g^{(i), +}_{m_{\star}, \alpha_1,\alpha_2}(y) = O\big(y^{-2i \mu -3}\big),
		\end{align*}at which point we determine $g^{(1), +}_{m_{\star}-2,\alpha_1,\alpha_2}(y)$ as above, thus repeating this procedure  shows the inductive claim.
		\\
		The steps  for determining $g^{(1), -}_{m,\alpha_1,\alpha_2}(y)$ is similar. To begin with, we recall  $g^{(1), -}_{m_{\star},\alpha_1,\alpha_2}$ is the above fast oscillating part, and we note the cancellations
		\begin{align*}
			&D^-_y \big(e^{i\frac{y^2}{4}}f(y)\big) =  e^{i\frac{y^2}{4}} D^-_y f(y) + i e^{i\frac{y^2}{4}} f(y) = e^{i\frac{y^2}{4}} \frac{2}{y^2} \Lambda f(y),\\
			& \big(\mathcal{L}_S + \mu\big)  \big(e^{i\frac{y^2}{4}}f(y)\big)  =  e^{i\frac{y^2}{4}}  \big(\mathcal{L}_S + \mu\big) f(y) +   e^{i\frac{y^2}{4}}    i(1 + \Lambda)f(y).
		\end{align*}
		%	\begin{align*}
			%	e^{i\frac{y^2}{4}}y^{\gamma}(it\partial_t - \frac{i}{2}y\partial_y)(\log y -\frac12\log t) + 2y^{\gamma}\partial_y\big(e^{i\frac{y^2}{4}}\big)\cdot \partial_y(\log y -\frac12\log t) = 0
			%	\end{align*}
		Therefore we solve \eqref{new-first} - \eqref{new-last} inductively with the same ansatz for a recursion formula via 
		\begin{align*}
			g^{(i), -}_{m,\alpha_1,\alpha_2}(y) =& c_{m}\cdot \psi_{\infty}(y)  +  \tilde{c}_{m} e^{i\frac{y^2}{4}}  y^{ 2 \mu i -5 }\big(1+\tilde{c}_{m,-2}y^{-2} +\dots\big)\\
			& =:\tilde{g}^{(i), -}_{m,\alpha_1,\alpha_2}(y) + \tilde{\tilde{g}}^{(i), -}_{m,\alpha_1,\alpha_2}(y),
		\end{align*}
		where again $ \tilde{c}_{m} ,  \tilde{c}_{m, -2}, \dots $ in  the particular part $ \tilde{\tilde{g}}^{(i), -}_{m,\alpha_1,\alpha_2}(y) $ are chosen to remove the error on the right. Likewise this process proves the induction subsequently. 
		%\begin{align*}
		%	&\big(-\alpha_0 - \frac{i}{2} + it\partial_t - \frac{i}{2}y\partial_y + \partial_{yy} + \frac{3}{y}\partial_y\big)\big[t^{\nu\alpha_1+\alpha_2}(\log y -\frac12\log t)^{a_*}g^{(1-)}_{a_*,\alpha_1,\alpha_2}\big]\\
		%	& =  t^{\nu\alpha_1+\alpha_2}\cdot (\log y -\frac12\log t)^{a_*-1}\cdot e^{i\frac{y^2}{4}}y^{\gamma - 2}\cdot(e_0 + e_{-2}y^{-2}+\ldots)\\
		%	&+  t^{\nu\alpha_1+\alpha_2}\cdot (\log y -\frac12\log t)^{a_*-2}\cdot e^{i\frac{y^2}{4}}y^{\gamma - 2}\cdot(f_0 + f_{-2}y^{-2}+\ldots),\\
		%	\end{align*}
	%	and the first term on the right is eliminated by adding a suitable correction 
	%	\[
	%	t^{\nu\alpha_1+\alpha_2}(\log y -\frac12\log t)^{a_*-1}g^{(1-)}_{a_*-1,\alpha_1,\alpha_2}
	%	\]
	%	which will be constructed in analogy to $g^{(1+)}_{a_*-1,\alpha_1,\alpha_2}(y)$, after which the process continues inductively. 
\end{proof}

Next we need to relate the homogeneous solutions in the preceding Lemma \ref{lem:yoscillatorylinear} to the solutions constructed with an absolute expansion at  $ 0 < y \lesssim1 $ given by Lemma~\ref{lem:smally1}.  In particular we have the following.
\begin{lem}\label{lem:gluey} Let $ y_0 > 0$ and $\{a_m, b_m \}_{m=0}^{m_{\star}} $ complex numbers such that  $ \{ g^{(i)}_{m,\alpha_{1},\alpha_{2}}(y) \}$ satisfy \eqref{hom-inner}, that is
	\begin{align} \nonumber
		&\big(i t\partial_t + \mathcal{L}_S - \alpha_0\big)\bigg(\sum_{m =0}^{m_{\star}} (\log(y)  - \nu \log(t))^m \cdot g^{(i)}_{m,\alpha_{1},\alpha_{2}}(y)\bigg) = 0,\\ \label{condition-large-y}
		& g^{(i)}_{m,\alpha_{1},\alpha_{2}}(y_0) = a_m,\; \partial_y g^{(i)}_{m,\alpha_{1},\alpha_{2}}(y_0) = b_m,
	\end{align}
	which are the solutions in Lemma \ref{lem:smally1}. Then there exist unique numbers  $\{c_m, d_m\}_{m=0}^{m_{\star}}$ with 
	\begin{align*}
		&\sum_{m =0}^{m_{\star}}\big[(\log(y) +\frac12\log(t))^m \cdot g^{(i), +}_{m,\alpha_{1},\alpha_{2}}(y) +  (\log(y) -\frac12\log(t))^m \cdot g^{(1),-}_{m,\alpha_{1},\alpha_{2}}(y)\big]\\
		& = \sum_{m =0}^{m_{\star}} (\log(y)  - \nu \log(t))^m \cdot g^{(i)}_{m,\alpha_{1},\alpha_{2}}(y),
	\end{align*}
	for functions $ \{ g^{(i), \pm}_{m,\alpha_{1},\alpha_{2}}(y) \}$ as in Lemma \ref{lem:yoscillatorylinear}. 
\end{lem}
\begin{Rem} In the above Lemma \ref{lem:gluey} the reverse holds true for a unique choice of $\{a_m, b_m \}_{m=0}^{m_{\star}}$ and some $ y_0 > 0$. Further we may select the coefficients $c^{(m)}_{0,0},\; c^{(m)}_{1,0} \in \mathbf{C}$ in Lemma \ref{lem:gluey} instead of $a_m, b_m$ as suggested be Lemma \ref{lem:smally1}.
\end{Rem}
\begin{proof} We proceed again by induction over $m$, starting with $m= m_{\star}$. In fact, we start writing 
	\begin{align*}
		\sum_{m = 0}^{m_{\star}}(\log(y)& \pm\frac12\log(t))^{m} g^{(i), \pm}_{m,\alpha_{1},\alpha_{2}}(y)\\
		&  = \sum_{m=0}^{m_{\star}} (\log(y) - \nu\log(t))^m \big(  \sum_{\tilde{m} = 0}^{m_{\star} - m} \beta^{\pm}_{m, \tilde{m}} (\log(y))^{\tilde{m}} g^{(i), \pm}_{m + \tilde{m},\alpha_{1},\alpha_{2}}(y)\big)
	\end{align*}
	for coefficients $\beta_{m, \tilde{m}}^{\pm} = (\mp \frac{1}{2\nu})^{m + \tilde{m}} \binom{m + \tilde{m}}{\tilde{m}} (\mp 2\nu -1)^{\tilde{m}}$. %with $e_{a*a_*}{\pm} = (\mp \frac{\nu}{2})^{a_*}$, 
	Hence we infer the first two conditions 
	\begin{align*}
		&\beta_{m_{\star}, 0}^{+}\cdot g^{(i), +}_{m_{\star},\alpha_{1},\alpha_{2}}(y_0) + \beta_{m_{\star}, 0}^{-} \cdot g^{(i), -}_{m_{\star},\alpha_{1},\alpha_{2}}(y_0) = a_{m_{\star}},\\
		&\beta_{m_{\star}, 0}^{+}\cdot \partial_y g^{(i), +}_{m_{\star},\alpha_{1},\alpha_{2}}(y_0) +\beta_{m_{\star}, 0}^{-} \cdot \partial_yg^{(i), -}_{m_{\star},\alpha_{1},\alpha_{2}}(y_0) = b_{m_{\star}},
	\end{align*}
	which by linear independence of $g^{(i), \pm}_{m_{\star},\alpha_{1},\alpha_{2}}(y)$ uniquely determine $c_{m_{\star}}, d_{m_{\star}}$ inverting the $2\times 2$ system. \\
	Next we express $g^{(i)}_{m_{\star}-1,\alpha_{1},\alpha_{2}}(y)$ in terms of $ g^{(i), \pm}_{m_{\star}-1,\alpha_{1},\alpha_{2}}(y)$ and $ g^{(i), \pm}_{m_{\star},\alpha_{1},\alpha_{2}}(y)$. To be precise, evaluating \eqref{condition-large-y} for $m = m_{\star} -1$ gives the system
	\begin{align*}
		&\sum_{\pm}\big(\beta_{m_{\star}-1, 1}^{\pm} \log(y_0) g^{(i), \pm}_{m_{\star},\alpha_{1},\alpha_{2}}(y_0) + \beta_{m_{\star}-1, 0}^{\pm} \cdot g^{(i), \pm}_{m_{\star}-1,\alpha_{1},\alpha_{2}}(y_0) \big)= a_{m_{\star}-1},\\
		&   y_0^{-1} a_{m_{\star}} + \log(y_0)b_{m_{\star}} + \beta_{m_{\star}-1, 0}^{+}\cdot \partial_y g^{(i), +}_{m_{\star}-1,\alpha_{1},\alpha_{2}}(y_0) +\beta_{m_{\star}-1, 0}^{-} \cdot \partial_yg^{(i), -}_{m_{\star}-1,\alpha_{1},\alpha_{2}}(y_0) = b_{m_{\star}-1},
	\end{align*}
	which is again uniquely solved by independence of $ g^{(i), \pm}_{m_{\star}-1,\alpha_{1},\alpha_{2}}(y)$. Continuing this process concludes the claim inductively.
\end{proof}
\begin{Rem}
	For the latter construction in Lemma \ref{lem:gluey}, we also refer to \cite[Remark 2.29, Corollary 2.28]{schmid}.
\end{Rem}
\;\\
The preceding consideration determines homogeneous solutions 
\[
t^{\nu\alpha_{1} + \alpha_{2}} \sum_{m =0}^{m_{\star}}\big[(\log(y) +\frac12\log(t))^m\cdot g^{(i), +}_{m,\alpha_{1},\alpha_{2}}(y) +  (\log(y) +\frac12\log(t))^m\cdot g^{(i), -}_{m,\alpha_{1},\alpha_{2}}(y) \big],
\]
which by Lemma \ref{lem:yoscillatorylinear} describe the minimal coefficient corresponding to $ t^{\nu \alpha_{1\star} + \alpha_{2\star}}$ as a free solution in $w(t, y)$. To be precise it describes the \emph{homogeneous} part of the approximation in Proposition \ref{prop:tysystemsolution2}. The following is implied by Proposition \ref{prop:tysystemsolution2}, Lemma \ref{lem:yoscillatorylinear} and Lemma \ref{lem:gluey}.
\begin{Cor}\label{cor:hom-part-oscill-large-y} Let $\{ g^{(i)}_{k, \tilde{k}}(t,y) \} $ be the coefficients in \eqref{diese-line-1} given by Proposition \ref{prop:tysystemsolution2} and
	\begin{align*}
		g^{(i)}_{k, \tilde{k}}(t,y) = \sum_{m = 0}^{m_{\star}} \big(\log(y) - \nu \log(t)\big)^{m} g^{(i)}_{ k, \tilde{k}, m}(y).
	\end{align*}
	Then there exist unique coefficients $\{ c_m, d_m \} $ such that 
	\begin{align}
		&g^{(i)}_{k, \tilde{k}, m}(y) = \tilde{g}^{(i),+}_{ 1, \tilde{k}, m}(y) + \tilde{g}^{(i), -}_{ 1, \tilde{k}, m}(y),\\
		&g^{(i)}_{k, \tilde{k}, m}(y) = \tilde{g}^{(i), +}_{ k, \tilde{k}, m}(y) + \tilde{g}^{(i), -}_{k, \tilde{k}, m}(y) + g^{(i), \text{nl}}_{ k, \tilde{k}, m}(y),\;\; k \geq 2,\\
		&\sum_{m = 0}^{m_{\star}} \big(\log(y) - \nu \log(t)\big)^{m} \tilde{g}^{(i),\pm}_{ k, \tilde{k}, m}(y) = \sum_{m = 0}^{m_{\star}} \big(\log(y)  \pm \f12 \log(t)\big)^{m} g^{(i),\pm}_{ k, \tilde{k}, m}(y),
	\end{align}
	where $ \{g^{(i),\pm}_{k, \tilde{k}, m}(y) \}$ are the solutions in Lemma \ref{lem:yoscillatorylinear}  and $ g^{(i), \text{nl}}_{k, \tilde{k}, m}(y)$ correspond to  interaction terms (particular solutions of \eqref{inhom-inner-sys-first} - \eqref{inhom-inner-sys-last}).
\end{Cor}
In order to  \emph{inductively calculate the large $y$ asymptotic of higher order  interaction coefficients%determine the structure of the coefficients of the larger powers $t^{\nu\alpha_1+\alpha_2}$ in the expansion of $w$}
}, we turn to  the \emph{inhomogeneous systems  in the oscillatory regime $y\gg 1$},  in particular the  structure of $n(t,y)$ where $ y \gg1 $, which involves solving the wave equation %in a suitable function space of
via suitable oscillatory functions.
\tinysection[0]{Inhomogeneous systems}\;\\	 
Let $f(t,y), n(t,y) $ have %asymptotic 
expansions of the form 
\begin{align*}
f(t, y) &= \sum_{\beta_1, \beta_2} \sum_{s + \tilde{s} = 0}^{m_{\star}}  t^{\nu \beta_1+ \beta_2} (\log(y)+\frac12\log(t))^s\cdot(\log(y) - \frac12 \log(t))^{\tilde{s}}\cdot h^{(i)}_{s, \tilde{s},\beta_1,\beta_2}(y),\;\;\\
n(t, y) &= \sum_{\beta_1, \beta_2} \sum_{s + \tilde{s} = 0}^{\tilde{m}_{\star}}  t^{\nu \beta_1+ \beta_2} (\log(y)+\frac12\log(t))^s\cdot(\log(y) - \frac12 \log(t))^{\tilde{s}}\cdot n^{(i)}_{s, \tilde{s},\beta_1,\beta_2}(y),
\end{align*}
where we sum (finitely) over $ \beta_1 \geq \beta_{1 \star} $, $ \beta_2 \in \Z$ (or $ \Z + \f12$, $i =2$), $ \beta_{2} \geq \beta_{2\star}(\beta_1)$ and $m_{\star} = m_{\star}(\beta_1),\; \tilde{m}_{\star}\geq m_{\star}$. Then we consider the inhomogeneous wave equation
\begin{align}\label{eqn:wave-large-y}
& t^2 \cdot \Box_S n(t,y) = \big(\partial_y^2 + \frac{3}{y} \partial_y\big)f(t,y),\\
&\Box_S = -(\partial_t -\frac12 t^{-1}y\partial_y)^2 + t^{-1}(\partial_{y}^2 + \frac{3}{y}\partial_y).
\end{align}
Similar to previous calculations where $ 0 < y \lesssim 1$, we recall the commutator 
\begin{align*}
\bigg[ - t^2(\partial_t -\frac12 t^{-1}y\partial_y)^2,  \big(\log(y) +\f12 \log(t)\big)^{s} \bigg] = 0,
\end{align*}
hence there holds
\begin{align*}
t^2& \cdot \Box_S\bigg(t^{\beta_1 \cdot \nu + \beta_2} \big(\log(y) + \f12 \log(t)\big)^s \cdot  \big(\log(y) - \f12 \log(t)\big)^{\tilde{s}}n_{s, \tilde{s}, \beta_1, \beta_2}(y)\bigg)\\[4pt]
& = \big(\log(y) + \f12 \log(t)\big)^s (- t^2(\partial_t -\frac12 t^{-1}y\partial_y)^2 ) \bigg( t^{\beta_1 \cdot \nu + \beta_2} \big(\log(y) - \f12 \log(t)\big)^{\tilde{s}} h_{s, \tilde{s}, \beta_1, \beta_2}(y)\bigg)\\
&\;\;\;\;\ + t^{\beta_1 \cdot \nu + \beta_2+1} \big(\partial_y^2 + \frac{3}{y} \partial_y\big) \bigg( \big(\log(y) + \f12 \log(t)\big)^s \cdot  \big(\log(y) - \f12 \log(t)\big)^{\tilde{s}}h_{s, \tilde{s}, \beta_1, \beta_2}(y)\bigg)\\[4pt]
&=  -\big(\log(y) + \f12 \log(t)\big)^s \cdot  \big(\log(y) - \f12 \log(t)\big)^{\tilde{s}} t^{\beta_1 \cdot \nu + \beta_2} L_{\beta_1, \beta_2}n_{s, \tilde{s}, \beta_1, \beta_2}(y)\\
&\;\;\; + \big(\log(y) + \f12 \log(t)\big)^s \bigg[ - t^2(\partial_t -\frac12 t^{-1}y\partial_y)^2,  \big(\log(y) - \f12 \log(t)\big)^{\tilde{s}} \bigg]  t^{\beta_1 \cdot \nu + \beta_2}  n_{s, \tilde{s}, \beta_1, \beta_2}(y)\\
&\;\;\;\;\ + t^{\beta_1 \cdot \nu + \beta_2+1} \big(\partial_y^2 + \frac{3}{y} \partial_y\big) \bigg( \big(\log(y) + \f12 \log(t)\big)^s \cdot  \big(\log(y) - \f12 \log(t)\big)^{\tilde{s}}n_{s, \tilde{s}, \beta_1, \beta_2}(y)\bigg).
\end{align*}
Calculating the latter two lines  and setting $ \tilde{m}_{\star} = m_{\star}$ for now, further leads to the following systems for $ \{n_{s, \tilde{s}, \beta_1, \beta_2}\}$ parametrized over $ s = 0,\dots, m_{\star},\; \tilde{s} = 0,\dots, m_{\star} -s$.
\boxalign[14cm]{
\begin{align}
	L_{\beta_1, \beta_2} n_{m_{\star}, 0, \beta_1, \beta_2}(y) &= \big(\partial_y^2 + \frac{3}{y} \partial_y \big)\big(  n_{m_{\star}, 0, \beta_1, \beta_2-1}(y)  - h_{m_{\star}, 0, \beta_1, \beta_2}(y)\big)\\[6pt]
	L_{\beta_1, \beta_2} n_{m_{\star}-1, 1, \beta_1, \beta_2}(y) &= \big(\partial_y^2 + \frac{3}{y} \partial_y \big)\big(  n_{m_{\star}-1, 1, \beta_1, \beta_2-1}(y)  - h_{m_{\star}-1, 1, \beta_1, \beta_2}(y)\big)\\[6pt]
	L_{\beta_1, \beta_2} n_{m_{\star}-1, 0, \beta_1, \beta_2}(y) &= \big(\partial_y^2 + \frac{3}{y} \partial_y \big)\big(  n_{m_{\star}-1, 0, \beta_1, \beta_2-1}(y)  - h_{m_{\star}-1, 0, \beta_1, \beta_2}(y)\big)\\[3pt]  \nonumber
	&\;\; + (2 (\nu \beta_1 + \beta_2) - \Lambda)n_{m_{\star} -1, 1, \beta_1, \beta_2}(y)\\[3pt]  \nonumber
	&\;\; - F_{m_{\star} -1, 0}(n, f)(y)
\end{align}
}
where
\begin{align*}
F_{s, \tilde{s}}(n, f)(y) =&\;\;  \frac{2}{y^2} (s+1) \big(h_{s+1, \tilde{s}, \beta_1, \beta_2}(y) - n_{s+1, \tilde{s}, \beta_1, \beta_2-1}(y) \big)\\ \nonumber
&\;\;  + \frac{2}{y^2} (\tilde{s} +1)\big(h_{s, \tilde{s} +1, \beta_1, \beta_2}(y) - n_{s, \tilde{s} +1, \beta_1, \beta_2-1}(y)  \big)\\ \nonumber
&\;\; +  \frac{2}{y}  (s+1) \partial_y\big(h_{s+1, \tilde{s}, \beta_1, \beta_2}(y) - n_{s+1, \tilde{s}, \beta_1, \beta_2-1}(y) \big)\\ \nonumber
&\;\; +  \frac{2}{y} (\tilde{s} +1)\partial_y\big(h_{s, \tilde{s} +1, \beta_1, \beta_2}(y) - n_{s, \tilde{s} +1, \beta_1, \beta_2-1}(y)  \big)
\end{align*}
and further when $ 0 \leq s \leq m_{\star} -2$
\boxalign[16cm]{
\begin{align}
	L_{\beta_1, \beta_2} n_{s, m_{\star} -s, \beta_1, \beta_2}(y) &= \big(\partial_y^2 + \frac{3}{y} \partial_y \big)\big(  n_{s, m_{\star}-s, \beta_1, \beta_2-1}(y)  - h_{s, m_{\star} -s, \beta_1, \beta_2}(y)\big)\\[6pt]
	L_{\beta_1, \beta_2} n_{s, m_{\star} -s-1, \beta_1, \beta_2}(y) &= \big(\partial_y^2 + \frac{3}{y} \partial_y \big)\big(  n_{s, m_{\star} -s -1, \beta_1, \beta_2-1}(y)  - h_{s, m_{\star} -s -1, \beta_1, \beta_2}(y)\big)\\[3pt] \nonumber
	&\;\; + (m_{\star} -s) (2 (\nu \beta_1 + \beta_2) - \Lambda)n_{s, m_{\star} -s, \beta_1, \beta_2}(y)\\[3pt]  \nonumber
	&\;\; - F_{s, m_{\star} -s -1}(n, f)(y),\\[6pt] 
	L_{\beta_1, \beta_2} n_{s, \tilde{s}, \beta_1, \beta_2}(y) &= \big(\partial_y^2 + \frac{3}{y} \partial_y \big)\big(  n_{s, \tilde{s}, \beta_1, \beta_2-1}(y)  - h_{s, \tilde{s}, \beta_1, \beta_2}(y)\big)\\[3pt] \nonumber
	&\;\; + (\tilde{s} +1) ( 2 (\nu \beta_1 + \beta_2) - \Lambda )n_{s, \tilde{s} +1, \beta_1, \beta_2}(y)\\[3pt]  \nonumber
	&\;\; - F_{s, \tilde{s}}(n, f)(y) -  G_{s, \tilde{s}}(n, f)(y)\\[3pt]  \nonumber
	& \;\; - (\tilde{s} +2) (\tilde{s} +1) n_{s, \tilde{s}+2, \beta_1, \beta_2}(y)\\[3pt] \nonumber
	0 \leq \tilde{s} \leq m_{\star} - s -2.&
\end{align}
}
where similarly as above we set
\begin{align*}
G_{s, \tilde{s}}(n, f)(y) =& \;\; \frac{2(s+1) (\tilde{s} +1)}{y^2} \big(h_{s+1, \tilde{s}+1, \beta_1, \beta_2}(y) - n_{s+1, \tilde{s}+1, \beta_1, \beta_2-1}(y)\big)\\\nonumber
& \;\; + \frac{(s +2) (s +1)}{y^2} \big(h_{s+2, \tilde{s}, \beta_1, \beta_2}(y) - n_{s+2, \tilde{s}, \beta_1, \beta_2-1}(y)\big)\\\nonumber
&\;\; + \frac{(\tilde{s} +2) (\tilde{s} +1)}{y^2} \big(h_{s, \tilde{s} +2, \beta_1, \beta_2}(y) - n_{s, \tilde{s} +2, \beta_1, \beta_2-1}(y)\big).
\end{align*}
\begin{Rem}\label{das-rem-below-system-wave-large-y}
(i)\; As previously, we note  the coefficients involving $ \beta_2 -1$ are absent in case $ \beta_2 = \beta_{2\star}$ is minimal. (ii)\; Solving  the above system gives an expansion for $ n(t,y)$ with $ \tilde{m}_{\star} = m_{\star}$ involving  powers of  $ \log(y)$ corrections for the coefficients (if some term in the series is a homogeneous solution for  $L_{\beta_1, \beta_2}$). Thus  rewriting 
\[
\log(y) = \f12 \big[ \big(\log(y) + \f12 \log(t)\big) + \big(\log(y) - \f12 \log(t)\big)\big],
\]
we use may use a binomial expansion after the induction iteration in order to remove $\log(y)$ dependence.
\end{Rem} 
\begin{lem}\label{lem:inhomwaveeqnlargey} Given an %asymptotic 
expansion of the above form 
\begin{align*}
	h(t, y) = \sum_{\beta_1, \beta_2} \sum_{s + \tilde{s} = 0}^{m_{\star}}  t^{\nu \beta_1+ \beta_2} (\log(y)+\frac12\log(t))^s\cdot(\log(y) - \frac12 \log(t))^{\tilde{s}}\cdot h^{(i)}_{s, \tilde{s},\beta_1,\beta_2}(y)
\end{align*}
with 
\begin{align*}
	h_{s, \tilde{s},\beta_1,\beta_2}(y) &= \;\sum_{m, \zeta}e^{i\cdot m\frac{y^2}{4}}y^{\zeta}\big(d^{s, \tilde{s}}_{m, \zeta, \beta_1, \beta_2} + d^{s, \tilde{s}}_{m, \zeta, \beta_1, \beta_2, -2}y^{-2} + \dots\big)\\
	& = \;\sum_{m, \zeta}e^{i\cdot m\frac{y^2}{4}}y^{\zeta}\big(d^{s, \tilde{s}}_{m, \zeta, \beta_1, \beta_2} + O(y^{-2})\big),
\end{align*}
where $\sum_{m, \zeta}$ sums over finitely many  $m \in \Z$ and $\zeta \in \R$ (depending on $s,\tilde{s}, \beta_1, \beta_2$). Then the wave equation 
\begin{align}\label{eqn:nh}
	& t^2 \cdot \Box_S n(t,y) = \big(\partial_y^2 + \frac{3}{y} \partial_y\big)h(t,y),\\
	&\Box_S = \big(-(\partial_t -\frac12 t^{-1}y\partial_y)^2 + t^{-1}(\partial_{y}^2 + \frac{3}{y}\partial_y)\big),
\end{align}
has a solution of the form  (with $\tilde{m}_{\star} \geq m_{\star}$)
\begin{align*}
	&n(t, y) = \sum_{\beta_1, \beta_2} \sum_{s + \tilde{s} = 0}^{\tilde{m}_{\star}} t^{\nu \beta_1+ \beta_2} (\log(y)+\frac12\log(t))^s\cdot(\log(y) - \frac12 \log(t))^{\tilde{s}}\cdot n^{(i)}_{s, \tilde{s},\beta_1,\beta_2}(y),\\
	&n_{s, \tilde{s},\beta_1,\beta_2}(y) = \sum_{\tilde{m}, \tilde{\zeta}}e^{i\cdot \tilde{m}\frac{y^2}{4}}y^{\tilde{\zeta}} \cdot \big(c^{s, \tilde{s}}_{\tilde{m}, \tilde{\zeta}, \beta_1, \beta_2} + c^{s, \tilde{s}}_{\tilde{m}, \tilde{\zeta}, \beta_1, \beta_2, -2}y^{-2} + \dots\big)\\
	& \hspace{2cm}=\sum_{\tilde{m}, \tilde{\zeta}}e^{i\cdot \tilde{m}\frac{y^2}{4}}y^{\tilde{\zeta}} \cdot \big(c^{s, \tilde{s}}_{\tilde{m}, \tilde{\zeta}, \beta_1, \beta_2} + O(y^{-2})\big).
\end{align*}
\end{lem}
\begin{proof} We fix $\beta_1 \in \Z_{+}$ and prove by induction over $ \beta_2 \geq \beta_{2\star}$ and the logarithmic powers, i.e. $ s = 0, \dots, m_{\star}, \tilde{s} = 0, \dots, m_{\star} -s$. We begin with the lowest power of $t > 0$, i.e. $t^{\nu \beta_{1}+\beta_{2\star}}$. Thus, considering the source term 
\begin{align*}
	t^{\nu \beta_1 + \beta_{2\star}}\sum_{s, \tilde{s}}(\log(y)+\frac12\log(t))^s\cdot(\log(y)- \frac12 \log(t))^{\tilde{s}}\cdot h_{s, \tilde{s},\beta_1,\beta_{2\star}}(y),
\end{align*}
we solve \eqref{eqn:nh} in the first step up as explained above, which implies an error of order  $O(t^{\nu\beta_{1}+\beta_{2\star}+1})$,  i.e. in general iterating over $ \beta_2 \geq \beta_{\star}$, we pick
\begin{align*}
	n(t, y) = \sum_{\beta_2\geq \alpha_{2\star}} t^{\nu\beta_1 + \beta_2}(\log(y)+\frac12\log(t))^s\cdot(\log(y) - \frac12 \log(t))^{\tilde{s}}\cdot n_{s, \tilde{s},\beta_1,\beta_{2}}(y),
\end{align*}
where the coefficients $h_{s, \tilde{s},\beta_1,\beta_{2\star}}(y)$ solve the above system and according to Remark \ref{das-rem-below-system-wave-large-y} with vanishing $ \beta_{2\star} -1 $ contributions. Thus for $ s + \tilde{s} = m_{\star}$, we set 
\begin{align*}
	L_{\beta_1, \beta_2}n_{s,m_{\star} -s,\beta_{1},\beta_{2\star}}(y) & = \big( \beta_1 \nu + \beta_{2\star} - \tfrac12y \partial_y\big) \big( \beta_1 \nu + \beta_{2\star} -1 - \tfrac12y \partial_y\big)n_{s,m_{\star} -s,\beta_{1},\beta_{2\star}}(y)\\
	& =	\big(\f14 y^2\partial_y^2 - y \partial_y (\beta_1 \nu - \beta_{2\star} - \f 3 4) + (\beta_1 \nu + \beta_{2\star})(\beta_1 \nu + \beta_{2\star} -1)\big) n_{s,m_{\star} -s,\beta_{1},\beta_{2\star}}(y)\\
	& = - \big(\partial_y^2 + \frac{3}{y} \partial_y \big)h_{s,m_{\star} -s,\beta_{1},\beta_{2\star}}(y),
\end{align*}
for which we first note the identities
\begin{align}\label{das-hier-welle-first-line}
	& e^{- i m \frac{y^2}{ 4}}\big(L_{\beta_1, \beta_2} e^{ i m\frac{y^2}{ 4}} f(y)\big) = L_{\beta_1, \beta_2} f(y) + im \frac{y^2}{ 4} \big((\Lambda +1) - 2(\beta_1 \nu + \beta_2)\big) f(y) 
	\big)- m^2 \frac{y^4}{ 8} f(y),\\ \label{das-hier-welle-second-line}
	&  e^{- i m \frac{y^2}{ 4}}\big((\partial_y^2 + 3 y^{-1} \partial_y) e^{ i m\frac{y^2}{ 4}} f(y)\big) = (\partial_y^2 + 3 y^{-1} \partial_y)f(y) + i m  (\Lambda +1)f(y) - m^2 \frac{y^2}{ 4} f(y).
\end{align}
%which we solve by respecting the algebraic structure postulated in the lemma.
Therefore if  for some $ m\neq 0 $ and $ \zeta \in \R$
\[
h_{s,m_{\star} -s,\beta_{1},\beta_{2\star}}(y) =  e^{i\cdot m\frac{y^2}{4}}\cdot y^{\zeta} c_{s, m_{\star} -s, \beta_1, \beta_{2\star}} + e^{i\cdot m\frac{y^2}{4}}\cdot O(y^{\zeta-2} )
\]
we note  \eqref{das-hier-welle-first-line} - \eqref{das-hier-welle-second-line} suggest the ansatz
\begin{align} \label{hier-ansatz-welle1}
	& n_{s,m_{\star} -s,\beta_{1},\beta_{2\star}}(y) = e^{i\cdot m\frac{y^2}{4}} y^{\zeta -2} \cdot 2 c_{s, m_{\star} -s, \beta_1, \beta_{2\star}}  + e^{i\cdot m\frac{y^2}{4}}\cdot O(y^{\zeta-4}),
\end{align}
and further 
\begin{align}
	&\big((\zeta -2-l)\big[(\zeta -2-l-1) \f14 - (\beta_1 \nu +\beta_{2\star}) + \f34\big] + (\beta_1 \nu + \beta_{2\star}) (\beta_1 \nu + \beta_{2\star}-1) \big) y^{\zeta - 2-l} c_l\\ \nonumber
	& - \frac{m^2}{8} y^{\zeta -l} c_{l+2}+ \frac{im}{4} (\zeta -l+2 - 2(\beta_1 \nu + \beta_{2\star}))	y^{\zeta -l}c_l\\ \nonumber
	& = (\zeta -l) (\zeta -l+2) y^{\zeta -l-2} \tilde{c}_l + im(\zeta -l +2) y^{\zeta -l} \tilde{c}_l - \frac{m^2}{4} y^{\zeta -l} \tilde{c}_{l+2},	\;\; l \geq 0,\;\;\text{even},
	%&(\nu\alpha_{1*} + \alpha_{2*}+2 - y\partial_y)(\nu\alpha_{1*} + \alpha_{2*}+1 - y\partial_y)\big(-\frac{4}{m^2}e^{i\cdot m\frac{y^2}{4}}y^{\zeta-4}\big)\\
	%& = e^{i\cdot m\frac{y^2}{4}}y^{\zeta} + e^{i\cdot m\frac{y^2}{4}}\cdot O(y^{\zeta - 2}), 
\end{align}
where we set
$
c_0 = 2 c_{s, m_{\star} -s, \beta_1, \beta_{2\star}}, \tilde{c}_0 = c_{s, m_{\star} -s, \beta_1, \beta_{2\star}},
$
and the system determines the $ O(y^{\zeta -4})$ term in \eqref{hier-ansatz-welle1} via the asymptotic expansion
\begin{align}
	e^{- i\cdot m\frac{y^2}{4}}&n_{s,m_{\star} -s,\beta_{1},\beta_{2\star}}(y) -  y^{\zeta -2} \cdot 2 c_{s, m_{\star} -s, \beta_1, \beta_{2\star}}\\ \nonumber
	&  = c_2 y^{\zeta-4} + c_4 y^{\zeta-6}  + c_6 y^{\zeta-8}  + \dots,\;\;\; y \gg1.
\end{align}
%and eliminating the error inductively we obtain a formal solution 
%	\begin{align*}
	%	n_{a_*,a_*',\alpha_{1*},\alpha_{2*}+2}(y) = e^{i\cdot m\frac{y^2}{4}}y^{\zeta-4}(d_{m\zeta a_*,a'_*,\alpha_{1*}\alpha_{2*}} + d_{m\zeta a_*,a'_*,\alpha_{1*},\alpha_{2*},-2}y^{-2} + \ldots)
	%	\end{align*}
If on the other hand for a vanishing phase $m = 0$, we simply solve for $ n_{s,m_{\star} -s,\beta_{1},\beta_{2\star}}(y)  =  O(y^{\zeta-2})$ via 
\begin{align*}
	&n_{s,m_{\star} -s,\beta_{1},\beta_{2\star}}(y) = \sum_l y^{\zeta-2- 2l} (\zeta-2l)(\zeta-2l +2) \gamma_l^{-1}\tilde{c}_{2l} ,\\
	& \gamma_l = (\zeta -2l-2)\big[(\zeta -2l-3) \f14 - (\beta_1 \nu +\beta_{2\star}) +\f34\big] + (\beta_1 \nu + \beta_{2\star}) (\beta_1 \nu + \beta_{2\star}-1).
\end{align*}
as long as $\zeta - 2l -2 \notin \{2(\beta_1 \nu + \beta_{2\star}),2(\beta_1 \nu + \beta_{2\star})-2\}$, while if $\zeta -2l -2$ takes one of these values, we represent $n_{s, m_{\star}-s, ,\beta_{1},\beta_{2\star}}(y)$ as linear combination of terms $O(y^{\zeta-2})$ and $ y^{\tilde{\zeta}} \cdot \log(y)$ for $ \tilde{\zeta} = \zeta -2l-2$ with the relevant $l$. As explained above, in this case %The introduction of an extra logarithm does not violate the algebaric structrure of the terms in the lemma, since we can write 
we recover the claim of the theorem increasing $m_{\star}$ after the iteration via
\[
\log(y) = \frac{1}{2}[(\log(y) + \frac12\log(t)) + (\log(y)- \frac12 \log(t))].
\]
Having defined the coefficient $ n_{s, m_{\star}-s, ,\beta_{1},\beta_{2\star}}(y)$ corresponding to  $t^{\beta_1 \nu + \beta_{2\star}}$ %and setting for simplicity 
%	\[
%	\Box_{t,y} = -(\partial_t -\frac12 t^{-1}y\partial_y)^2 + t^{-1}(\partial_{yy} + \frac{3}{y}\partial_y),
%	\]
we determine the error to be 
\begin{align*}
	&\Box_{S}\big[ t^{\nu\beta_1 + \beta_{2\star}}(\log(y)+\frac12\log(t))^{s}\cdot(\log(y) - \frac12 \log(t))^{\tilde{s}}\cdot n_{s,m_{\star},\beta_{1},\beta_{2\star}}(y)\big]\\[4pt]
	& - t^{\nu\beta_1 + \beta_{2\star}}(\log(y)+\frac12\log(t))^{s}\cdot(\log(y)- \frac12 \log(t))^{\tilde{s}}\cdot h_{s,m_{\star} -s,\beta_{1},\beta_{2\star}}(y)
	%& = O\big( t^{\nu\alpha_{1*} + \alpha_{2*}}(\log y+\frac12\log t)^{a_*-\kappa_1}(\log y - \frac12 \log t)^{a'_*-\kappa_2}\big) + O( t^{\nu\alpha_{1*} + \alpha_{2*}+1}),
\end{align*}
and thus we move to the subsequent second lines in the system if $ s < m_{\star}$. Here we once again note the commutators
\begin{align*}
	&\big[ \Lambda, e^{i m \frac{y^{2}}{4}}\big]f(y) = i m \frac{y^2}{2} e^{i m \frac{y^{2}}{4}}f(y),\;\; \big[ y^{-1} \partial_y, e^{i m \frac{y^{2}}{4}}\big]f(y) = e^{i m \frac{y^{2}}{4}}i \frac{m}{2} f(y),
\end{align*}
and the contribution splits 
$$ n_{s, m_{\star}-s-1, ,\beta_{1},\beta_{2\star}}(y) = n'_{s, m_{\star}-s-1, ,\beta_{1},\beta_{2\star}}(y) + \tilde{n}_{s, m_{\star}-s-1, ,\beta_{1},\beta_{2\star}}(y), $$
where $ \tilde{n}$ solves for the source terms given by $h_{s, m_{\star}-s-1, ,\beta_{1},\beta_{2\star}}(y)$ (as above) and $ n'(y)$  for all other terms. In particular if $m \neq 0$
\[
(m_{\star} -s) (2 (\nu \beta_1 + \beta_2) - \Lambda)n_{s, m_{\star} -s, \beta_1, \beta_2}(y) = e^{im\frac{y^2}{4}} \cdot O(y^{\zeta}),\\
\]
which requires the ansatz $ n(y) = e^{i m  \frac{y^2}{4}} O(y^{\zeta-4}) + \dots$. The other parts in this identity are taken from the asymptotic expansion of $ 	F_{s, m_{\star} - s-1}(n, f)(y)$, which we explain below in the general case of the last five lines of the above system.  We note in case of vanishing phase, i.e. $m = 0$, solving for 
$$ (m_{\star} -s) (2 (\nu \beta_1 + \beta_2) - \Lambda)n_{s, m_{\star} -s, \beta_1, \beta_2}(y)$$ on right implies a $\log^2(y)$ contribution in case the above exceptional powers are attained, as seen from the variation of constants formula \eqref{part-sol-wave}. For the inductive process, given the solutions of the previous $ \beta_2 -1$ system, we conclude an induction (over $ s, \tilde{s}, \beta_2$) via the following schematic steps.\\[3pt]
$\bullet$ Given the asymptotic in the previous $ s+1, \tilde{s}+1$ steps, calculate  the asymptotic of $ F_{s, \tilde{s}}(n, f)(y)$ by noting, as $ y \gg1 $, the relations 
\begin{align*}
	&\frac{2}{y^2} (s+1) \big(h_{s+1, \tilde{s}, \beta_1, \beta_2}(y) - n_{s+1, \tilde{s}, \beta_1, \beta_2-1}(y) \big) \;\sim O(y^{-2})\cdot \big(h_{s+1, \tilde{s}, \beta_1, \beta_2}(y) - n_{s+1, \tilde{s}, \beta_1, \beta_2-1}(y) \big),\\ \nonumber
	& \frac{2}{y^2} (\tilde{s} +1)\big(h_{s, \tilde{s} +1, \beta_1, \beta_2}(y) - n_{s, \tilde{s} +1, \beta_1, \beta_2-1}(y)  \big) \;\sim O(y^{-2})\cdot\big(h_{s, \tilde{s} +1, \beta_1, \beta_2}(y) - n_{s, \tilde{s} +1, \beta_1, \beta_2-1}(y)  \big),\\ \nonumber
	& \frac{2}{y}  (s+1) \partial_y\big(h_{s+1, \tilde{s}, \beta_1, \beta_2}(y) - n_{s+1, \tilde{s}, \beta_1, \beta_2-1}(y) \big) \; \sim h_{s+1, \tilde{s}, \beta_1, \beta_2}(y) - n_{s+1, \tilde{s}, \beta_1, \beta_2-1}(y),\\ \nonumber
	& \frac{2}{y} (\tilde{s} +1)\partial_y\big(h_{s, \tilde{s} +1, \beta_1, \beta_2}(y) - n_{s, \tilde{s} +1, \beta_1, \beta_2-1}(y)  \big)\; \sim h_{s, \tilde{s} +1, \beta_1, \beta_2}(y) - n_{s, \tilde{s} +1, \beta_1, \beta_2-1}(y),
\end{align*}
where the latter two lines hold  (see the commutator above) in case of non-vanishing  phase $m \neq 0$. Otherwise we essentially again need $O(y^{-2}) $ times the expansions. In each of the cases multiplying with $y^{-2}$ we need to  check for degeneracy and $\log(y) $ correction (especially for the $h$ source terms).\\[3pt] 
$\bullet$ Note the same holds given the previous $ s+1, s+2, \tilde{s} +1, \tilde{s}+2$ steps for the asymptotic of $ G_{s, \tilde{s}}(n, f)(y)$, which is simply multiplication with $ y^{-2}$ (recall we start $ s = m_{\star}, m_{\star}-1 \dots, $ and $ \tilde{s} = m_{\star} -s, m_{\star} -s -1, \dots$).\\[3pt]
$\bullet$  Given the coefficients in the $\tilde{s}+1, \tilde{s}+2$ steps, consider and solve for the source terms
\[
(\tilde{s} +1) ( 2 (\nu \beta_1 + \beta_2) - \Lambda )n_{s, \tilde{s} +1, \beta_1, \beta_2}(y),\;\; (\tilde{s} +2) (\tilde{s} +1) n_{s, \tilde{s}+2, \beta_1, \beta_2}(y),
\]
Here we may need to increase the powers of the $\log(y)$ contribution depending on degeneracy in the previous $ \tilde{s} +1 $ line.\\[3pt]
$\bullet$  As a general rule, in case of non-vanishing phase $ m \neq 0$, we make a solution ansatz \emph{to leading order} by formally multiplying the \emph{source asymptotic} with $y^{-4}$. Thus for\\[2pt]
$\rhd$\; the \emph{ $\Lambda$-derivatives} or the \emph{$\Delta_y$-derivatives} in the source, i.e. 
$$
\Lambda e^{i m \frac{y^2}{4}} O(y^{\zeta}) = e^{i m \frac{y^2}{4}} O(y^{\zeta +2}) ,\;\;\Delta_y e^{i m \frac{y^2}{4}} O(y^{\zeta}) = e^{i m \frac{y^2}{4}} O(y^{\zeta+2}),
$$
the solution must be of order $  e^{i m \frac{y^2}{4}} O(y^{\zeta-2})$.\\[3pt]
$\rhd$\;  In $F$ the $y^{-1} \partial_y$ derivative terms give the leading order asymptotic, i.e. for
$$
y^{-1} \partial_y e^{i m \frac{y^2}{4}} O(y^{\zeta}) = e^{i m \frac{y^2}{4}} O(y^{\zeta }), 
$$
the solution must be of order $  e^{i m \frac{y^2}{4}} O(y^{\zeta-4})$. Likewise in $G$ the terms are multiplied by $y^{-2} $, i.e. 
given expansions $
y^{-2} \cdot e^{i m \frac{y^2}{4}} O(y^{\zeta}) =   e^{i m \frac{y^2}{4}} O(y^{\zeta-2}) $ on the right, the solution must be of order $  e^{i m \frac{y^2}{4}} O(y^{\zeta-6})$.\\[6pt]
Further if $ m = 0$, in all cases we retain the order of the source asymptotic and solve by direct coefficient comparison (up to a possible degeneracy which needs to be corrected).
This is enough to conclude the induction in general form stated in the Lemma, where however we obtain for each $ \beta_2 \geq \beta_{2\star}$ corrections up to
$ \log(y),\;\; \log^2(y),\dots \log^{m_{\star} -s}(y)$. These will then be rewritten as remarked above.
%	\[
%	F
%	\]
%	where $\kappa_j\in \{0,1\}$ and at least one is non-zero, and the notation is suggestive, meaning terms like in the expansion for $n$ but in the first $O(\ldots)$ term with lower %powers of the logarithmic terms, except that there may be additional logarithms $\log y$ implicit in the coefficient functions depending on $y$. 
%	\\
%	Applying the same method to the terms in 
%	\[
%	O\big( t^{\nu\alpha_{1*} + \alpha_{2*}}(\log y+\frac12\log t)^{a_*-\kappa_1}(\log y - \frac12 \log t)^{a'_*-\kappa_2}\big)
%	\]
%	and performing a finite induction on the exponents of the logarithmic terms $(\log y \pm\frac12 \log t)$ leads eventually to an error of the second type $O( t^{\nu\alpha_{1*} + \alpha_{2*}+1})$. At this point we apply the same procedure to this remaining error term. 
\end{proof}
Since the \emph{almost free wave} solutions (homogeneous) in Lemma \ref{lem:hom1} and Corollary \ref{cor:hom-wave} do not depend on a specific asymptotic region,  we  will now combine the preceding Lemma and proof with Lemma~\ref{lem:hom1}, as follows.
\begin{lem}\label{lem:tyoscinducprep} Given an %asymptotic 
expansion 
\begin{align*}
	h(t, y) = \sum_{\beta_1, \beta_2} \sum_{s + \tilde{s} = 0}^{m_{\star}}  t^{\nu \beta_1+ \beta_2} (\log(y)+\frac12\log(t))^s\cdot(\log(y) - \frac12 \log(t))^{\tilde{s}}\cdot h_{s, \tilde{s},\beta_1,\beta_2}(y)
\end{align*}
with $h_{s, \tilde{s},\beta_1,\beta_2}(y)$ as in Lemma \ref{lem:inhomwaveeqnlargey}. Then the general solution of the wave equation 
\begin{align}\label{eqn:nh2}
	&t^2 \cdot \big(-(\partial_t -\frac12 t^{-1}y\partial_y)^2 + t^{-1}(\partial_{y}^2 + \frac{3}{y}\partial_y)\big) n(t,y) = \big(\partial_y^2 + \frac{3}{y} \partial_y\big)h(t,y)
\end{align}
has the form  (with $\tilde{m}_{\star} \geq m_{\star}$)
\begin{align*}
	&n(t, y) = \sum_{\beta_1, \beta_2} \sum_{s + \tilde{s} = 0}^{\tilde{m}_{\star}} t^{\nu \beta_1+ \beta_2} (\log(y)+\frac12\log(t))^s\cdot(\log(y) - \frac12 \log(t))^{\tilde{s}}\cdot n_{s, \tilde{s},\beta_1,\beta_2}(y),\\
	&n_{s, \tilde{s},\beta_1,\beta_2}(y) = n^{(ih)}_{s, \tilde{s},\beta_1,\beta_2}(y)  + n^{(h)}_{s, \tilde{s},\beta_1,\beta_2}(y),
\end{align*}
where $n^{(ih)}_{s, \tilde{s},\beta_1,\beta_2}(y) $ is as in Lemma \ref{lem:inhomwaveeqnlargey} and $n^{(h)}_{s, \tilde{s},\beta_1,\beta_2}(y),\; \tilde{s} \neq 0$ satisfies
\begin{align*}
	\sum t^{\nu\beta_1 + \beta_2}&(\log(y)+\frac12\log(t))^s (\log(y)- \frac12\log(t))^{\tilde{s}} n^{(h)}_{s, \tilde{s},\beta_1,\beta_2}(y)\\
	&=\;  \sum t^{\nu\beta_1 + \beta_2}(\log(y)+\frac12\log(t))^{\tilde{\tilde{s}}}  \tilde{n}^{(h)}_{\tilde{\tilde{s}},\beta_1,\beta_2}(y),
\end{align*}
with the functions $ \tilde{n}^{(h)}_{\tilde{\tilde{s}},\beta_1,\beta_2}(y)$ from  Lemma~\ref{lem:hom1}. Finally, the coefficients $n_{s, \tilde{s},\beta_1,\beta_2}(y)$with $0\leq \beta_2\leq K+2$ only depend on the coefficients $h_{s,\tilde{s},\beta_1,\beta_2}$ with $0\leq \beta_2\leq K$. 
\end{lem}
We now have the tools to solve the combined system \eqref{schrod-wave-y} by using induction on the power of $t^{\alpha_1\nu + \alpha_2},\; t^{\beta_1 \nu + \beta_2}$. Before we turn to the expansions of $ w(t,y), n(t,y)$ in \eqref{diese-line-1} and \eqref{diese-line2}, we briefly outline how to obtain particular solutions of the Schr\"odinger part.\\[4pt]
Similar as before, we let $ \tilde{f}(t,y), w(t,y)$ have the expansions
\begin{align*}
\tilde{f}(t, y) &= \sum_{\alpha_1, \alpha_2} \sum_{s + \tilde{s} = 0}^{m_{\star}}  t^{\nu \alpha_1+ \alpha_2} (\log(y)+\frac12\log(t))^s\cdot(\log(y) - \frac12 \log(t))^{\tilde{s}}\cdot \tilde{f}^{(i)}_{s, \tilde{s},\alpha_1,\alpha_2}(y),\;\;\\
w(t, y) &=  \sum_{\alpha_1, \alpha_2}  \sum_{s + \tilde{s} = 0}^{\tilde{m}_{\star}}  t^{\nu \alpha_1+ \alpha_2} (\log(y)+\frac12\log(t))^s\cdot(\log(y) - \frac12 \log(t))^{\tilde{s}}\cdot g^{(i)}_{s, \tilde{s},\alpha_1,\alpha_2}(y),
\end{align*}
where we sum (finitely) over $ \alpha_1, \alpha_2$ with the analogous restrictions as in the wave part.  Then we consider the inhomogeneous Schr\"odinger system
\begin{align}\label{eqn:schrod-large-y}
&\big( i t \partial_t + \mathcal{L}_S - \alpha_0\big) w(t,y) = \tilde{f}(t,y),\\
&\mathcal{L}_S = \partial_y^2 + \frac{3}{y} \partial_y - \frac{i}{2}\Lambda.
\end{align}
Similar to the small $y$ region where $ 0 < y \lesssim 1$, setting $\tilde{m}_{\star} = m_{\star}$ we directly infer the following system for $\{ g^{(i)}_{s, \tilde{s}, \alpha_1, \alpha_2}(y)\}$ parametrized over $ s = 0,\dots, m_{\star},\; \tilde{s} = 0,\dots, m_{\star} -s$ (where as usual $ \mu = - \alpha_0 + i(\nu \alpha_1 + \alpha_2)$)
\boxalign[14cm]{
\begin{align}
	(\mathcal{L}_S + \mu) g^{(i)}_{m_{\star}, 0, \alpha_1, \alpha_2}(y) &= \;\tilde{f}_{m_{\star}, 0, \alpha_1, \alpha_2}(y)\\[6pt]
	(\mathcal{L}_S + \mu) g^{(i)}_{m_{\star}-1, 1,  \alpha_1, \alpha_2}(y) &=\;\tilde{f}_{m_{\star}-1, 1, 2, \alpha_1, \alpha_2}(y)\\[6pt] \nonumber
	(\mathcal{L}_S + \mu) g^{(i)}_{m_{\star}-1, 1,  \alpha_1, \alpha_2}(y) &=\; \tilde{f}_{m_{\star}-1, 0, \alpha_1, \alpha_2}(y) - m_{\star} D^+_y g^{(i)}_{m_{\star}, 0, \alpha_1, \alpha_2}(y)\\[5pt]
	&\;\;\;- D^{-}_y g^{(i)}_{m_{\star}-1, 1, \alpha_1, \alpha_2}(y),
\end{align}
}
and for $ 0 \leq s \leq m_{\star} -2$
\boxalign[14cm]{
\begin{align}
	(\mathcal{L}_S + \mu) g^{(i)}_{s,  m_{\star} -s,  \alpha_1, \alpha_2}(y) &=\;\tilde{f}_{s, m_{\star}  - s, 2, \alpha_1, \alpha_2}(y)\\[6pt]
	(\mathcal{L}_S + \mu) g^{(i)}_{s, m_{\star} -s-1,  \alpha_1, \alpha_2}(y) &= \;\tilde{f}_{s, m_{\star} -s-1, \alpha_1, \alpha_2}(y)\\[5pt] \nonumber
	&\;\;\; - (s+1) D^+_y g^{(i)}_{s+1, m_{\star} -s-1, \alpha_1, \alpha_2}(y)\\[5pt] \nonumber
	&\;\;\;- (m_{\star} -s)D^{-}_y g^{(i)}_{s,m_{\star}-s, 1, \alpha_1, \alpha_2}(y),\\[6pt]
	(\mathcal{L}_S + \mu) g^{(i)}_{s, \tilde{s},  \alpha_1, \alpha_2}(y) &= \;\tilde{f}_{s, \tilde{s}, \alpha_1, \alpha_2}(y) - (s+1) D^+_y g^{(i)}_{s+1, \tilde{s}, \alpha_1, \alpha_2}(y)\\[5pt] \nonumber
	&\;\;\;- (\tilde{s} +1)D^{-}_y g^{(i)}_{s,\tilde{s} +1, \alpha_1, \alpha_2}(y) - 2\frac{(s+1)(\tilde{s}+1)}{y^2} g^{(i)}_{s+1,\tilde{s} +1, \alpha_1, \alpha_2}(y),\\[5pt] \nonumber
	&\;\;\;- \frac{(s+2)(s +1)}{y^2}g^{(i)}_{s+2,\tilde{s} ,\alpha_1, \alpha_2}(y) - \frac{(\tilde{s}+2)(\tilde{s} +1)}{y^2}  g^{(i)}_{s,\tilde{s} +2, \alpha_1, \alpha_2}(y) ,\\[6pt] \nonumber
	0 \leq \tilde{s}\leq m_{\star} -s -2&
\end{align}
}
\;\\
\begin{lem}\label{lem:inhomschrodeqnlargey} Given $\tilde{f}(t,y)$ has an %asymptotic 
expansion of the above form, precisely 
\begin{align*}
	&\tilde{f}(t, y) = \sum_{\alpha_1, \alpha_2} \sum_{s + \tilde{s} = 0}^{m_{\star}}  t^{\nu \alpha_1+ \alpha_2} (\log(y)+\frac12\log(t))^s\cdot(\log(y) - \frac12 \log(t))^{\tilde{s}}\cdot \tilde{f}^{(i)}_{s, \tilde{s},\alpha_1,\alpha_2}(y),\;\;\\
	&\tilde{f}_{s, \tilde{s},\alpha_1,\alpha_2}(y) = \;\sum_{m, \zeta}e^{i\cdot m\frac{y^2}{4}}y^{\zeta}\big(c^{s, \tilde{s}}_{m, \zeta, \alpha_1, \alpha_2} + O(y^{-2})\big),
\end{align*}
where we sum finitely over  $m \in \Z$ and $\zeta \in \mathbf{C}$ (depending on $s,\tilde{s}, \alpha_1, \beta_2$). Then the equation
\begin{align}\label{eqn:schrodf}
	& \big( i t \partial_t + \mathcal{L}_S - \alpha_0\big) w(t,y) = \tilde{f}(t,y),
\end{align}
with $ \mathcal{L}_S = \partial_y^2 + \frac{3}{y} \partial_y -\frac{i}{2} - \frac{i}{2} y \partial_y $ has a solution of the form ($ \tilde{m}_{\star} \geq m_{\star}$)
\begin{align*}
	&w(t, y) =  \sum_{\alpha_1, \alpha_2}  \sum_{s + \tilde{s} = 0}^{\tilde{m}_{\star}}  t^{\nu \alpha_1+ \alpha_2} (\log(y)+\frac12\log(t))^s\cdot(\log(y) - \frac12 \log(t))^{\tilde{s}}\cdot g^{(i)}_{s, \tilde{s},\alpha_1,\alpha_2}(y),\\
	&g^{(i)}_{s, \tilde{s},\alpha_1\alpha_2}(y) = \sum_{\tilde{m}, \tilde{\zeta}}e^{i\cdot \tilde{m}\frac{y^2}{4}}y^{\tilde{\zeta}} \cdot \big(\tilde{c}^{s, \tilde{s}}_{\tilde{m}, \tilde{\zeta}, \alpha_1, \alpha_2} + O(y^{-2})\big).
\end{align*}
\end{lem}
\begin{proof} As for Lemma \ref{lem:inhomwaveeqnlargey} we outline the proof by induction over $ s, \tilde{s}, \alpha_1$. Recalling the calculation in the proof of Lemma \ref{lem:inhomwaveeqnlargey} we conclude
\begin{align}\label{das-hier-schrod-first-line}
	& e^{- i m \frac{y^2}{ 4}}\big(\mathcal{L}_Se^{ i m\frac{y^2}{ 4}} f(y)\big) = \mathcal{L}_S f(y) + im(\Lambda +1)f(y) - m(m-1) \frac{y^2}{4}f(y),\\ \label{das-hier-schrod-second-line}
	&  e^{- i m \frac{y^2}{ 4}}\big(D^{\pm}_ye^{ i m\frac{y^2}{ 4}} f(y)\big) = D^{\pm}_yf(y) + im f(y).
\end{align}
Assuming $ \tilde{f}_{s, m_{\star}  - s, 2, \alpha_1, \alpha_2}(y) = e^{i m \frac{y^2}{4}} \big(c_0 y^{\zeta} + O(y^{\zeta -2})\big)$ for phases $ m \neq 0, 1$, we then make the ansatz
\[
g_{s, m_{\star}  - s, 2, \alpha_1, \alpha_2}(y) = e^{i m \frac{y^2}{4}} \big(- \frac{4 c_0}{m (m-1)} y^{\zeta-2} + O(y^{\zeta -4})\big),
\]
where the $ O(y^{\zeta -4})$ terms are iteratively determined by
\begin{align*}
	&(\zeta -2l)(\zeta - 2l +2) y^{\zeta - 2l -4} \tilde{c}_l  + (- \frac{i}{2} (1 + \zeta -2l)  + \mu)y^{\zeta -2l -2}\tilde{c}_l + im(2 + \zeta -2l)y^{\zeta -2l -2}\tilde{c}_l\\
	& -m(m-1) y^{\zeta -2l -2}\tilde{c}_{l+2} = y^{\zeta -2l}c_l,\;\; l \geq 0,
\end{align*}
where $ \tilde{c}_0$ is the above modification of $ c_0$. In case $m=0,1 $ we have to consider the ansatz 
\[
g_{s, m_{\star}  - s, 2, \alpha_1, \alpha_2}(y) = e^{i  \frac{y^2}{4}} \cdot O(y^{\zeta }),\;\; g_{s, m_{\star}  - s, 2, \alpha_1, \alpha_2}(y) =  O(y^{\zeta }),\;\
\]
respectively. Then calculating the leading order terms in the corresponding recursion formula from \eqref{das-hier-schrod-first-line}, i.e. considering the operators
\[
- \frac{i}{2}\Lambda,\;\;(m =0),\;\;\; - \frac{i}{2} \Lambda + i(\Lambda + 2),\;\;(m =1),
\] 
we have degeneracy if for some $ l \geq 0$ there either holds $ \zeta - 2l = - 2\mu i -1 $ in case $ m = 0$ or $ \zeta - 2l =  2\mu i -3 $ in case $ m = 1$. This is of course consistent with a logarithmic leading order contribution in the variation of constants formula  \eqref{general-sol-large-y-scal}, i.e. if $ y \geq y_0 \gg1$ we solve via
\begin{align*}
	e^{i \frac{y^2}{4}}(1 + O(y^{-2}))& \int_{y_0}^y e^{- i \frac{s^2}{4}} s^{- 2 \mu i +2}(1 + O(s^{-2}))\; f(s) ds\\
	&-  y^{- 2\mu i -1} (1 + O(y^{-2}))\int_{y_0}^y s^{2 \mu i }(1 + O(s^{-2}))\; f(s) ds,
\end{align*}
hence we again need a $\log(y)$ correction. The general induction is then carried out similar to Lemma \ref{lem:inhomwaveeqnlargey}, that is we remark the following.\\[3pt]
$\bullet$ For any $\beta_2$ given the $s+1, \tilde{s}+1$ steps, we have the asymptotic (to leading order) as $ y \gg1 $
\begin{align}
	D^{+}_y g_{s,\tilde{s} +1, \alpha_1, \alpha_2}(y) \sim g_{s,\tilde{s} +1, \alpha_1, \alpha_2}(y) ,\;\; (m\neq 0),\;\;\; \sim y^{-2}\cdot g_{s,\tilde{s} +1, \alpha_1, \alpha_2}(y) ,\;\; (m = 0),\\
	D^{-}_y g_{s+1,\tilde{s} , \alpha_1, \alpha_2}(y) \sim g_{s+1,\tilde{s} , \alpha_1, \alpha_2}(y),\;\; (m\neq 1),\;\;\; \sim y^{-2}\cdot g_{s+1,\tilde{s} , \alpha_1, \alpha_2}(y) ,\;\; (m = 1),
\end{align}
with the need to check for degeneracy (and $\log(y)$ correction) if any of the cases $ m = 0,1$ occur.\\[3pt]
$\bullet$ Given the previous $ s+1, \tilde{s}+1, s+2, \tilde{s}+2$ steps, we  likewise consider the asymptotic of the remaining terms on the right of the system, i.e. 
\begin{align*}
	y^{-2}\cdot g_{s+1,\tilde{s} +1, \alpha_1, \alpha_2}(y),\;\;\; y^{-2}\cdot g_{s+2,\tilde{s} ,\alpha_1, \alpha_2}(y),\;\; y^{-2}\cdot  g_{s,\tilde{s} +2, \alpha_1, \alpha_2}(y), 
\end{align*}
$\bullet$ As a general rule if $ m \neq 0,1, $ we formally multiply the source asymptotic on the right by $ y^{-2}$ to find a solution ansatz, i.e. in the case of terms involving \emph{$D^{\pm}_y-$ derivatives}, we have 
\[
D^{\pm}_y e^{i m \frac{y^2}{4}} O(y^{\zeta}) = e^{i m \frac{y^2}{4}} O(y^{\zeta}),
\]
hence the solution has the asymptotic $e^{i m \frac{y^2}{4}} O(y^{\zeta-2})$ and for source terms with 
$$ y^{-2}\cdot e^{i m \frac{y^2}{4}} O(y^{\zeta}) = e^{i m \frac{y^2}{4}} O(y^{\zeta-2})$$
we of course solve by functions of the form $ e^{i m \frac{y^2}{4}} O(y^{\zeta-4})$.\\[3pt]
$\rhd$\; In case of $ m = 0$ or $m =1$, we solve again retaining the order from the source on the right.  
Here  we of course check  again for degeneracy with a possible correction by multiples of $ \phi(y) \cdot \log(y)$ where $ \phi(y)$ is one of the fundamental solutions for $ \mathcal{L}_S + \mu$. Thus we conclude the  induction claim as before in Lemma \ref{lem:inhomwaveeqnlargey}.
\end{proof}
\;\\
We can now clarify the structure of the coefficient functions $g^{(i)}_{\alpha_1,\alpha_2}(t,y), h^{(i)}_{\beta_1,\beta_2}(t,y)$ in the expansions for $w(t,y)$ and $n(t,y)$ of \eqref{diese-line-1} and \eqref{diese-line2}. In particular, we calculate the large $ y$ asymptotic for coefficients $g^{(i)}_{s, \tilde{s},\alpha_1,\alpha_2}(y), h^{(i)}_{s, \tilde{s},\beta_1,\beta_2}(y)$ as above with the following expansions. For \eqref{diese-line-1} we write
\begin{align}\label{diese-line-nun-1}
w_N^*(t,y)	=&\;\sum_{k=1}^N \sum_{\ell^{(s)}_k \leq 2\tilde{k} }t^{\nu \cdot k + \tilde{k}}  \sum_{s + \tilde{s} \leq \tilde{m}_{\star}}(\log(y) +\f12\log(t))^{s} (\log(y) - \f12\log(t))^{\tilde{s}}g^{(1)}_{s, \tilde{s}, k, \tilde{k}}( y)\\ \nonumber
&\; + \sum_{k=1}^N \sum_{\ell^{(s)}_k  \leq 2\tilde{k} +1 }t^{\nu \cdot k + \tilde{k} + \f12}   \sum_{s + \tilde{s} \leq \tilde{m}_{\star}}(\log(y) +\f12\log(t))^{s} (\log(y) - \f12\log(t))^{\tilde{s}}g^{(2)}_{s, \tilde{s}, k, \tilde{k}}( y),
\end{align}
as well as for  \eqref{diese-line2}
\begin{align} \label{diese-line-nun2}
n_N^*(t,y)	=&\;\sum_{k=1}^N \sum_{\ell^{(w)}_k \leq 2\tilde{k} }t^{\nu \cdot k + \tilde{k}} \sum_{s + \tilde{s} \leq \tilde{m}_{\star}}(\log(y) +\f12\log(t))^{s} (\log(y) - \f12\log(t))^{\tilde{s}} h^{(1)}_{s, \tilde{s}, k, \tilde{k}}( y)\\[4pt] \nonumber
&\;+\sum_{k=1}^N \sum_{\ell^{(w)}_k \leq 2\tilde{k} +1 }t^{\nu \cdot k + \tilde{k}+ \f12} \sum_{s + \tilde{s} \leq \tilde{m}_{\star}}(\log(y) +\f12\log(t))^{s} (\log(y) - \f12\log(t))^{\tilde{s}} h^{(2)}_{s, \tilde{s}, k, \tilde{k}}( y).
\end{align}
\;\\
In analogy to Proposition \ref{prop:tysystemsolution2}, we give the structure of the solution of the main equation  \eqref{schrod-wave-y} in the oscillatory region $y\gg 1$ via induction for the interaction coefficients. Therefore, let us change the notation in \eqref{diese-line-nun-1} - \eqref{diese-line-nun2} using $ \alpha_1, \alpha_2$ in the former and $\beta_1, \beta_2$ in the latter case. We note as usual $ \alpha_2 , \beta_2 \in \Z +\f12$ for the half-integer case $ i =2$.
\begin{prop}\label{prop:tyoscillatory} Let the functions in \eqref{diese-line-1} and \eqref{diese-line2} be the unique matched solutions of Proposition~\ref{prop:tysystemsolution2} defined via \eqref{identify-large} for $ y\in (0, \infty)$ . Then there exist unique numbers 
$$
c_{\ell, \alpha_1 ,\alpha_2}^{1, (i)}, d_{\ell, \alpha_1, \alpha_2}^{1, (i)} \in \mathbf{C},\;\;\;\; c_{\ell, \beta_1, \beta_2}^{2, (i)}, d_{\ell, \beta_1, \beta_2}^{2, (i)} \in \mathbb{R},$$
such that these functions admit representations of the form \eqref{diese-line-nun-1} and \eqref{diese-line-nun2} with coefficients $
\big \{g^{(i)}_{s, \tilde{s}, \alpha_1, \alpha_2}( y) \big\} $ and $ \big\{h^{(i)}_{s, \tilde{s}, \beta_1, \beta_2}( y)\big  \}$ defined for $ y \in (0, \infty)$
having the following asymptotic expansion as $ y \gg1 $.
\begin{align}
	&	\;g^{(i)}_{0, \tilde{s}, \alpha_1, \alpha_2}(y) = g^{(i), -}_{0, \tilde{s}, \alpha_1, \alpha_2}(y) + g^{(i), \text{nl}}_{0, \tilde{s}, \alpha_1, \alpha_2}(y),\;\;	g^{(i)}_{s, 0, \alpha_1, \alpha_2}(y) = g^{(i), +}_{s,0,\alpha_1, \alpha_2}(y) + g^{(i), \text{nl}}_{s ,0, \alpha_1, \alpha_2}(y),\\[4pt]
	&	\;g^{(i)}_{s, \tilde{s}, \alpha_1, \alpha_2}(y) =  g^{(i), \text{nl}}_{s ,\tilde{s}, \alpha_1, \alpha_2}(y),\; s\neq 0, \tilde{s} \neq 0,
\end{align}
where $ g^{(i), \text{nl}}_{s ,\tilde{s}, 1, \alpha_2}(y) = 0$ and for $ \alpha_1 = 2n-1,\; \alpha_1 =2n$
\begin{align}
	&g^{(i), \text{nl}}_{s ,\tilde{s}, \alpha_1, \alpha_2}(y) = \sum_{m = -n+1}^n e^{i m \frac{y^2}{4}} y^{2i \alpha_0(1 - 2m)} g^{(i), \text{nl}}_{s ,\tilde{s}, \alpha_1, \alpha_2, m}(y),\\
	& g^{(i), \text{nl}}_{s ,\tilde{s}, \alpha_1, \alpha_2,m}(y)  = \sum_{(\alpha_1', \alpha_2') \in \Omega_{m, \alpha_1, \alpha_2}} y^{2(\alpha_1\cdot \nu + \alpha_2) - 2(\alpha_1' \cdot \nu - \alpha_2')} \cdot O(y^{-3}),\;\; m \in \{0,1\}\\
	& g^{(i), \text{nl}}_{s ,\tilde{s}, \alpha_1, \alpha_2,m}(y)  = \sum_{(\alpha_1', \alpha_2') \in \Omega_{m, \alpha_1, \alpha_2}} y^{2(\alpha_1\cdot \nu + \alpha_2) - 2(\alpha_1' \cdot \nu - \alpha_2')} \cdot O(y^{-3 - 2|m|}),\;\; m \neq 0,1
\end{align}
where $ \Omega_{m, \alpha_1, \alpha_2}$ is a finite subset of all $ (\alpha_1', \alpha_2') \in \Z^2$ (or $\Z \times (\Z + \f12)) $ with $ \alpha_1' \cdot \nu + \alpha_2' \geq 0$ and $  \alpha_1' \cdot \nu + \alpha_2' \geq \nu -2 $ in case $ \tilde{s} \neq 0$. Further
\begin{align}
	&g^{(i), +}_{\ell, 0,  \alpha_1, \alpha_2}(y) = y^{-2i \mu_{ \alpha_1, \alpha_2} -1}\big( d_{\ell, \alpha_1, \alpha_2}^{1, (i)}  + O(y^{-2})\big),\\
	&g^{(i), -}_{0,\ell, \tilde{s},  \alpha_1, \alpha_2}(y) = e^{i \frac{y^2}{4}}y^{2i \mu_{ \alpha_1, \alpha_2} -3}\big( c_{\ell,  \alpha_1, \alpha_2}^{1, (i)}  + O(y^{-2})\big).
\end{align}
Likewise we have for the wave part
\begin{align}
	&\;h^{(i)}_{s, 0, \beta_1, \beta_2}(y) = h^{(i), \text{hom}}_{s,0, \beta_1, \beta_2}(y) + h^{(i), \text{nl}}_{s ,0,\beta_1, \beta_2}(y),\\
	&\;h^{(i)}_{s, \tilde{s}, \beta_1, \beta_2}(y) =  h^{(i), \text{nl}}_{s ,\tilde{s}, \beta_1, \beta_2}(y),\; \tilde{s} \neq 0,
\end{align}
where $ h^{(i), \text{nl}}_{s ,\tilde{s}, 1,  \beta_2}(y) = 0$ and for $ \beta_1 = 2n+1,\; \beta_1 =2n$
\begin{align}
	&h^{(i), \text{nl}}_{s ,\tilde{s}, \beta_1, \beta_2}(y) = \sum_{m = -n}^n e^{i m \frac{y^2}{4}} y^{-4m i \alpha_0} h^{(i), \text{nl}}_{s ,\tilde{s}, \beta_1, \beta_2, m}(y),\\
	& h^{(i), \text{nl}}_{s ,\tilde{s}, \beta_1, \beta_2,m}(y)  = \sum_{ (\ell, \beta_1', \beta_2') \in \tilde{\Omega}_{m, \beta_1, \beta_2}} y^{2(\beta_1 \cdot \nu + \beta_2) - 2(\beta_1' \cdot \nu + \beta_2') - \ell } \cdot O(y^{-3}),
\end{align}
where again $ \tilde{\Omega}_{m, \beta_1, \beta_2 }$ consists of $ (\ell, \beta_1', \beta_2') \in \Z^3$ (or $ \Z^2 \times(\Z + \f12)$) with $ \beta_1' \cdot \nu + \beta_2' \geq 0,\;\ell \geq 0$ and 
\begin{align}
	h^{(i), \text{hom}}_{\ell, 0, \beta_1, \beta_2}(y) =&\;\; y^{ 2(\beta_1 \cdot \nu +  \beta_2)}\big( d_{\ell, \beta_1, \beta_2}^{2, (i)}  + O(y^{-2})\big),\\
	&\;\; + y^{ 2(\beta_1 \cdot \nu +  \beta_2)-2}\big( c_{\ell,\beta_1, \beta_2}^{2, (i)}  + O(y^{-2})\big)
\end{align}
are almost free wave solutions.
\end{prop}
\begin{Rem} For the interaction parts $g^{(i), \text{nl}}_{s ,\tilde{s}, \alpha_1, \alpha_2,m}(y) , h^{(i), \text{nl}}_{s ,\tilde{s}, \beta_1, \beta_2,m}(y) $ where $ m \neq 0,1$ in the former and $ m \neq 0$ in the latter case, we can show the $m$-dependence in $ \Omega_m$ is as follows:  There exist $ \alpha_{1}^{j},\beta_1^{j}, j = 1, \dots, |m|$ with  finite sets of $\alpha_2^j s, \beta_2^j s$ such that
\begin{align*}
	& \alpha_{1}^{j} \nu + \alpha_2^{j}, \;\beta_{1}^{j} \nu + \beta_2^{j} \geq \nu -2,\;\; j = 1, 2, \dots, |m|,\\
	&g^{(i), \text{nl}}_{s ,\tilde{s}, \alpha_1, \alpha_2,m}(y) =y ^{2(\alpha_1\cdot \nu + \alpha_2)} \sum_{\alpha_2^1, \alpha_2^2,\dots, \alpha_2^{|m|}} y ^{-2(\alpha_1^{1}\cdot \nu + \alpha_2^{1}) -  \;\dots \;- 2(\alpha_1^{|m|} \nu + \alpha_2^{|m|})} \cdot O(y^{-3 -2|m|})
\end{align*}
and similar for $h^{(i), \text{nl}}_{s ,\tilde{s}, \beta_1, \beta_2,m}(y) $.
\end{Rem}
\begin{proof}
The proof concludes by induction over $t^{\nu\alpha_1+\alpha_2},\; t^{\nu \beta_1+\beta_2}$ occurring in $(w,n)$. We first note the homogeneous solutions $g^{(i), -}_{0, \tilde{s}, \alpha_1, \alpha_2}(y)$ and $ g^{(i), +}_{s, 0, \alpha_1, \alpha_2}(y) $  and hence the coefficients $ c_{\ell, \alpha_1 ,\alpha_2}^{1, (i)}, d_{\ell, \alpha_1, \alpha_2}^{1, (i)} $ of the Schr\"odinger part are  are uniquely determined by Lemma \ref{lem:gluey}, Corollary \ref{cor:hom-part-oscill-large-y} and in the induction through $\log(y)$ corrections (as in Lemma \ref{lem:inhomschrodeqnlargey}) in which case we need to  increase  $ \tilde{m}_{\star} \geq m_{\star}$. Similarly, the free wave parts $h^{(i), \text{hom}}_{\ell, 0, \beta_1, \beta_2}(y)$ and hence $ c_{\ell, \beta_1, \beta_2}^{2, (i)}, d_{\ell, \beta_1, \beta_2}^{2, (i)} $,   are uniquely determined  in the induction by Lemma \ref{lem:tyoscinducprep} which is the same data as in Proposition \ref{prop:tysystemsolution2} where $ 0 < y \lesssim 1$.
We start by calculation the beginning of the induction.\\[4pt]
{\bf{\emph{Step 1}} $\alpha_1 = \beta_1 =1$}:  For any $\alpha_2 \in \Z_+$ (and similar for $ \beta_1 \in \Z_+$), we henceforth denote by  $ \alpha_{2\star} = \alpha_{2\star}(\alpha_{1})\in \Z$ (or $ \Z +\f12$ in case $ i =2$) the corresponding minimal $\alpha_2$. The first step consists of picking the minimal $\alpha_1 = \alpha_{1\star}\in \Z_+$,  thus $ \alpha_1 = 1$ (and likewise $ \beta_1 = 1$). Assume we calculated  $N \in \Z_+$ iteration steps similar to the induction in the region where $ 0 < y \lesssim 1$. Then we may expand the source terms on the right of  \eqref{schrod-wave-y} as follows
\begin{align}
	&  t \cdot n_N^{\ast} (t,y) \cdot w_N^{\ast}(t,y)\\ \nonumber
	&= \;\sum_{\alpha_1 \geq 2,\; \alpha_2 \geq \alpha_{2\star}} t^{\alpha_1 \nu + \alpha_2 +1} \sum_{s + \tilde{s} \leq \tilde{m}_{\star}}(\log(y)+\frac12\log(t))^s\cdot(\log(y) - \frac12 \log(t))^{\tilde{s}}\cdot \tilde{g}^{(i)}_{s, \tilde{s},\alpha_1,\alpha_2+1}(y),\\[8pt]
	&|w_N^{\ast}(t,y)|^2\\\nonumber
	&=\; \sum_{\beta_1 \geq 2,\; \beta_2 \geq \beta_{2\star}} t^{\beta_1 \nu + \beta_2} \sum_{s + \tilde{s}  \leq \tilde{m}_{\star}} (\log(y)+\frac12\log(t))^s\cdot(\log(y) - \frac12 \log(t))^{\tilde{s}}\cdot \tilde{h}^{(i)}_{s, \tilde{s},\beta_1,\beta_2}(y),
\end{align}
for which, in case $N \geq 2$, the coefficients corresponding to 
$$ t^{2\nu + \alpha_2}, t^{3\nu + \alpha_2},\dots, t^{N\nu + \alpha_2},\;  t^{2\nu + \beta_2},t^{3\nu + \beta_2},\dots, t^{N\nu + \beta_2}$$
are 'removed' by the linear part of \eqref{schrod-wave-y} acting on  $ w_N^{\ast}, n_N^{\ast}$. Hence,  for the $N+1$-st iteration step we calculate  $w_{N+1}(t,y), n_{N+1}(t,y)$ by the minimal of the remaining contributions (in terms of $ t^{\alpha_1 \nu + \alpha_2},t^{\beta_1 \nu + \beta_2} $), which are of the form $ t^{(N+1)\nu + \alpha_{2} +1},\;t^{(N+1)\nu + \beta_{2} }$. In particular consider $ N =1$, we infer the minimal coefficients in this \emph{Step 1} are free solutions. Hence we obtain
\begin{align*}
	&g^{(i)}_{s, 0, 1, \alpha_2}(y) = y^{-2 i \mu_{1, \alpha_2} -1}(d^{1, (i)}_{s,1, \alpha_2} + O(y^{-2})) = y^{2i \alpha_0  + 2(\nu + \alpha_2)} (d^{1, (i)}_{s,1, \alpha_2} + O(y^{-2})),\\
	&g^{(i)}_{0, \tilde{s}, 1, \alpha_2}(y) = e^{i \frac{y^2}{4}}y^{2 i \mu_{1, \alpha_2} -3} (c^{1, (i)}_{\tilde{s},1, \alpha_2} + O(y^{-2})) = e^{i \frac{y^2}{4}}y^{- 2i \alpha_0  - 2(\nu + \alpha_2)} (d^{1, (i)}_{s,1, \alpha_2} + O(y^{-2})),
\end{align*}
as well as for the wave part
\begin{align*}
	h^{(i), \text{hom}}_{s, 0, 1, \beta_2}(y) = \sum_{\ell = 0}^{\beta_2 - \beta_{2\star}} \big(y^{2(\nu + \beta_2) - 4\ell}d_{\ell} +  y^{2(\nu + \beta_2)-2 - 4\ell} c_{\ell}\big),
\end{align*}
where $ d_0 = d^{2, (i)}_{s,1, \beta_2},\; c_0 = c^{2, (i)}_{s,1, \beta_2}$.\\[4pt]
{\bf{\emph{Step 2}} $\alpha_1 = \beta_1 =2$}:  Now we calculate the first source contributions corresponding to $ t^{2\nu + \beta_2}, t^{2\nu + \alpha_2}$ of the form
\begin{align}\label{this-one}
	&t \cdot n_1 (t,y) \cdot w_1(t,y)\\ \nonumber
	& =\; \sum_{\alpha_2 \geq \alpha_{2\star}} t^{2 \nu + \alpha_2 +1} \sum_{s + \tilde{s} = 0}^{\tilde{m}_{\star}} (\log(y)+\frac12\log(t))^s\cdot(\log(y) - \frac12 \log(t))^{\tilde{s}}\cdot \tilde{g}^{(i)}_{s, \tilde{s},2,\alpha_2+1}(y),\\[8pt] \label{this-two}
	&|w_1(t,y)|^2 = \sum_{\beta_2 \geq \beta_{2\star}} t^{2 \nu + \beta_2} \sum_{s + \tilde{s} = 0}^{\tilde{m}_{\star}}  (\log(y)+\frac12\log(t))^s\cdot(\log(y) - \frac12 \log(t))^{\tilde{s}}\cdot \tilde{h}^{(i)}_{s, \tilde{s},2,\beta_2}(y),\\ \nonumber
	& \tilde{g}^{(i)}_{s, \tilde{s},2,\alpha_2}(y) = \sum_{\alpha_2' + \tilde{\alpha}_2 = \alpha_2-1} \sum_{\substack{s_1 + s_2 = s\\ \tilde{s}_1 + \tilde{s}_2 = \tilde{s}}}g_{s_1, \tilde{s}_1,1,\alpha_2'}(y) h_{s_2, \tilde{s}_2,1,\tilde{\alpha}_2}(y),\\ \nonumber
	&\tilde{h}^{(i)}_{s, \tilde{s},2,\beta_2}(y) =  \sum_{\beta_2' + \tilde{\beta}_2 = \beta_2}\sum_{\substack{s_1 + s_2 = s\\ \tilde{s}_1 + \tilde{s}_2 = \tilde{s}}}  g_{s_1, \tilde{s}_1,1,\beta_2'}(y) \overline{g_{s_2, \tilde{s}_2,1,\tilde{\beta}_2}(y)}
\end{align}
where in case $i =1$ (for $\tilde{g}^{(i)}_{s, \tilde{s}}$ say) we actually need to sum the terms $ g^{(1)}_{s, \tilde{s}} \cdot h^{(1)}_{s, \tilde{s}},\; g^{(2)}_{s, \tilde{s}} \cdot h^{(2)}_{s, \tilde{s}}$  and  for $ i=2$ we sum  $ g^{(1)}_{s, \tilde{s}} \cdot h^{(2)}_{s, \tilde{s}},\; g^{(2)}_{s, \tilde{s}} \cdot h^{(1)}_{s, \tilde{s}}$ (and similar for the wave part). Note here since the asymptotic does not specifically depend on $ s, \tilde{s}$ we ignore this convolution product in the following. Thus considering \emph{Step 1}, we infer
\begin{align*}
	&\tilde{g}^{(i)}_{s, 0,2,\alpha_2}(y) = y^{-2i \mu_{1, \alpha_2'} -1} \cdot y^{2(\nu + \tilde{\alpha}_2)} (\tilde{c} + O(y^{-2})),\\
	&\tilde{g}^{(i)}_{s, \tilde{s},2,\alpha_2}(y) = e^{i \frac{y^2}{4}} y^{2i \mu_{1, \alpha_2'} -3} \cdot y^{2(\nu + \tilde{\alpha}_2)} (\tilde{c} + O(y^{-2})),\;\; \tilde{s}\neq 0,\;\; (s, \tilde{s}) = (0,0)
\end{align*}
whence we solve by Lemma \ref{lem:inhomschrodeqnlargey} and its proof
\begin{align}
	g^{(i), \text{nl}}_{s, 0,2,\alpha_2}(y)  =& \sum_{\alpha_2', \tilde{\alpha}_2}y^{2i \alpha_0 + 2(2\nu + \alpha_2'+  \tilde{\alpha}_2 +1 )}\cdot y^{-3} (\tilde{c} + O(y^{-2})) =  y^{2i \alpha_0 + 2(2\nu + \alpha_2 )}\cdot O(y^{-3}),\\
	g^{(i), \text{nl}}_{s, \tilde{s},2,\alpha_2}(y) =& \sum_{\beta_2', \tilde{\beta}_2}  e^{i \frac{y^2}{4}} y^{- 2i\alpha_0 + 2(\nu + \tilde{\beta}_2) - 2(\nu + \beta_2')-3} (\tilde{c} + O(y^{-2}))\\ \nonumber
	=& \sum_{\beta_2'}  e^{i \frac{y^2}{4}} y^{- 2i\alpha_0 + 2(2\nu + \beta_2) - 2(2\nu + 2\beta_2' )} \cdot O(y^{-3}).
\end{align}
Here we note $ \nu + \beta_2'  \geq \nu$ and in general at least $ \nu \beta_1' \nu + \beta_2' \geq \beta_1'(\nu -1) -1$ when the term is derived from the wave part or $  \nu \beta_1' \nu + \beta_2' \geq \beta_1'(\nu -1) $ when the term is derived from the Schr\"odinger part. We then add homogeneous solutions for the case $ \alpha_1 = 2$  according to Lemma \ref{lem:gluey} and Corollary \ref{cor:hom-part-oscill-large-y}. For the wave part \eqref{this-two}, we calculate the relevant source to be schematically a sum of the terms
\begin{align*}
	g^{(i)}_{s_1 , 0, 1, \beta_2'}(y) \cdot \overline{g^{(i)}_{s_2, 0, 1, \tilde{\beta}_2}(y)},\;\;\;\;g^{(i)}_{ 0,\tilde{s}_1, 1, \beta_2'}(y) \cdot \overline{g^{(i)}_{ 0,\tilde{s}_2, 1, \tilde{\beta}_2}(y)}, \;\;\;\;2\text{Re} \big( g^{(i)}_{s, 0, 1, \beta_2'}(y) \cdot \overline{g^{(i)}_{0, \tilde{s}, \beta_2'}(y)}\big).
\end{align*}
Thus they have the asymptotic expansion
\begin{align*}
	&\tilde{h}^{(i)}_{s, 0, 2, \beta_2}(y) = \sum_{\beta_2', \tilde{\beta}_2} y^{-2i \mu_{1, \beta_2'}-1}\cdot y^{2i \overline{\mu}_{1, \tilde{\beta}_2}-1}(\tilde{c} + O(y^{-2})),\\
	&\tilde{h}^{(i)}_{0, \tilde{s},  2, \beta_2}(y) = \sum_{\beta_2', \tilde{\beta}_2} y^{2i \mu_{1, \beta_2'}-3}\cdot y^{-2i \overline{\mu}_{1, \tilde{\beta}_2}-3}(\tilde{d} + O(y^{-2})),
\end{align*}
as well as 
\begin{align*}
	&\tilde{h}^{(i)}_{s, \tilde{s},  2, \beta_2}(y) = \sum_{\beta_2', \tilde{\beta}_2} \bigg(e^{- i \frac{y^2}{4}}y^{-2i \mu_{1, \beta_2'}-1}\cdot y^{-2i \overline{\mu}_{1, \tilde{\beta}_2}-3}(\tilde{e} + O(y^{-2}))\\
	&\;\hspace{2cm} + e^{ i \frac{y^2}{4}}y^{2i \overline{\mu}_{1, \beta_2'}-1}\cdot y^{2i \mu_{1, \tilde{\beta}_2}-3}(\tilde{\tilde{e}} + O(y^{-2}))\bigg).
\end{align*}
In particular solving by Lemma \ref{lem:inhomwaveeqnlargey} and Lemma \ref{lem:tyoscinducprep} (for the homogeneous part) we find (multiplying by $y^{-2}$)
\begin{align*}
	&h^{(i), \text{nl}}_{s, 0, 2, \beta_2}(y) =  y^{ 2(2\nu + \beta_2) -1} \cdot  O(y^{-3}),\;\;\;h^{(i), \text{nl}}_{0, \tilde{s},  2, \beta_2}(y) =  y^{-2(2\nu + \beta_2) -5} \cdot O(y^{-3}),\\ 
	&h^{(i), \text{nl}}_{s, \tilde{s},  2, \beta_2}(y) = \sum_{ \tilde{\beta}_2} \big[e^{- i \frac{y^2}{4}}y^{4i\alpha_0 + 2(2\nu + \beta_2)  - 2(2\nu + 2 \tilde{\beta}_2) - 3} \cdot O(y^{-3})\\
	&\;\hspace{4cm} + e^{i \frac{y^2}{4}}y^{- 4i\alpha_0 + 2(2\nu + \beta_2)  - 2(2\nu + 2 \tilde{\beta}_2) - 3} \cdot O(y^{-3})\big].
\end{align*}
where again $ \nu + \tilde{\beta}_2 \geq \nu $. The added homogeneous (approximate) solution is of course of the form
\[
t^{2\nu + \beta_2} \sum_{s\geq 0} \big(\log(y) + \f12 \log(t)\big)^s \big(y^{2(2\nu + \beta_2)}(\tilde{c} + O(y^{-2})) + y^{2(2\nu + \beta_2)-2}(\tilde{\tilde{c}} + O(y^{-2}))\big),
\]
where the leading coefficients are determined in Proposition \ref{prop:tysystemsolution2} in the small  $0 < y \lesssim 1 $  regime\\[4pt]
{\emph{\bf{Step 3} Higher order terms}}: Let us calculate the case where $ \alpha_1 = 3$ and $\beta_1 = 3$.  For the Schr\"odinger line in \eqref{schrod-wave-y}, we therefore consider the source terms
\begin{align} \label{source-schrod-large1}
	t \cdot n_1(t,y) \cdot w_2(t,y),\;\; t \cdot n_2(t,y) \cdot w_1(t,y).
\end{align}
For the first of these, the coefficients have the following asymptotic.
\begin{align*}
	\tilde{g}^{(i)}_{s, 0, 3, \alpha_2}(y) = \sum_{\alpha_2' , \tilde{\alpha}_2} y^{2 i \alpha_0 + 2(2\nu + \beta_2')-1} (\tilde{c} + O(y^{-2}))\cdot y^{2(\nu + \tilde{\beta})2)},\\
	\tilde{g}^{(i)}_{0, \tilde{s}, 3, \alpha_2}(y) = \sum_{ \tilde{\alpha}_2} e^{i \frac{y^2}{4}} y^{- 2i \alpha_0 + 2(3 \nu + \alpha_2) - 2(\nu + 2 \tilde{\alpha}_2) -2}\cdot O(y^{-3}),
\end{align*}
which we solve the interaction coefficients via Lemma \ref{lem:inhomschrodeqnlargey} as above where we use $\alpha_1' + \tilde{\alpha}_2 +1 = \alpha_2$. For the second of the above source terms in \eqref{source-schrod-large1}, we note the coefficients  have the form
\begin{align*}
	\tilde{g}^{(i)}_{s, 0, 3, \alpha_2}(y) =& \sum_{\alpha_2', \tilde{\alpha}_2} y^{- 2i \mu_{1, \alpha_2'} + 2(2\nu + \tilde{\alpha}_2)}(\tilde{c} + O(y^{-2})),\\
	\tilde{g}^{(i)}_{0,  \tilde{s}, 3, \alpha_2}(y) =& \sum_{\alpha_2', \tilde{\alpha}_2} e^{i \frac{y^2}{4}} y^{ 2i \mu_{1, \alpha_2'} - 2(2\nu + \tilde{\alpha}_2) -6} \cdot O(y^{-3})
\end{align*}
and further 
\begin{align*}
	\tilde{g}^{(i)}_{s,  \tilde{s}, 3, \alpha_2}(y) =& \sum_{\alpha_2', \tilde{\alpha}_2} \big(e^{i \frac{y^2}{4}} y^{ 2i \mu_{1, \alpha_2'} + 2(2\nu + \tilde{\alpha}_2) } (\tilde{c} + O(y^{-2})) + y^{-2i \mu_{1, \alpha_2'} -1 -2(2\nu + \tilde{\alpha}_2)-3} \cdot O(y^{-3})\big)\\
	& + \sum_{\alpha_2', \tilde{\alpha}_2}\sum_{\alpha_2''} e^{\pm i \frac{y^2}{4}} y^{\mp 4 i \alpha_0 - 2i \mu_{1, \alpha_2'} + 2(2\nu + \tilde{\alpha}_2)  - 2(2\nu + 2 \alpha_2'')-4}\cdot O(y^{-3})\\
	& + \sum_{\alpha_2', \tilde{\alpha}_2}\sum_{\alpha_2''} y^{4 i \alpha_0 + 2i \mu_{1, \alpha_2'} + 2(2\nu + \tilde{\alpha}_2)  - 2(2\nu + 2 \alpha_2'')-6}\cdot O(y^{-3})\\
	&+ \sum_{\alpha_2', \tilde{\alpha}_2}\sum_{\alpha_2''} e^{2 i \frac{y^2}{4}} y^{- 4 i \alpha_0 + 2i \mu_{1, \alpha_2'} + 2(2\nu + \tilde{\alpha}_2)  - 2(2\nu + 2 \alpha_2'')-2}\cdot O(y^{-3 - 4}).
\end{align*}
Hence we solve as before (note in the first line we exploit again that the extra factor $ t\cdot n \cdot w$ is present in the source term)
\begin{align}
	g^{(i), \text{nl}}_{s, 0, 3, \alpha_2}(y) =&  y^{ 2i \alpha_0 + 2(3\nu + \alpha_2) }\cdot O(y^{-3}),\\
	g^{(i), \text{nl}}_{0,  \tilde{s}, 3, \alpha_2}(y) =& e^{i \frac{y^2}{4}} y^{-2i \alpha_0 -2(3\nu + \alpha_2) -4 } \cdot O(y^{-3}),
\end{align}
and further there holds
\begin{align*}
	&g^{(i), \text{nl}}_{s,  \tilde{s}, 3, \alpha_2}(y)\\
	& = \sum_{\alpha_2'} e^{i \frac{y^2}{4}} y^{-2i \alpha_0 + 2(3\nu + \alpha_2) -2(3\nu + 2 \alpha_2')} \cdot O(y^{-3})+ \sum_{\tilde{\alpha}_2} y^{2i \alpha_0 + 2(4\nu + \alpha_2) -2(3\nu + 2 \tilde{\alpha}_2)-4} \cdot O(y^{-3})\\ \nonumber
	&\;\; + \sum_{\alpha_2''} e^{\pm i \frac{y^2}{4}} y^{\mp 4 i \alpha_0 - 2i \alpha_0 + 2(3 \nu + \alpha_2) - 2(2\nu + \alpha_2'' ) } \cdot O(y^{-7})\\\nonumber
	& \;\;+ \sum_{\alpha_2''}\sum_{\alpha_2'} y^{2 i \alpha_0 + 2(3 \nu + \alpha_2) -2(3\nu +2\alpha_2') -2(2\nu +2 \alpha_2'')}\cdot O(y^{-9})\\\nonumber
	& \;\;+ \sum_{\alpha_2''}\sum_{\alpha_2'} e^{2i \frac{y^2}{4}} y^{-6 i \alpha_0 + 2(3 \nu + \alpha_2) -2(3\nu +2\alpha_2') -2(2\nu +2 \alpha_2'')}\cdot O(y^{-9}).
\end{align*}
For the wave part, the relevant sources with $ \beta_1 = 3$ consist of
\[
w_1(t,y) \cdot \overline{w_2(t,y)},\;\; \;w_2(t,y) \cdot \overline{w_1(t,y)},
\]
which implies the interaction coefficients by Lemma \ref{lem:inhomwaveeqnlargey} have the asymptotic form
\begin{align}
	&h^{(i), \text{nl}}_{s, 0, 3, \beta_2}(y) = y^{2(3\nu + \beta_2) -1}\cdot O(y^{-3}),\\
	&h^{(i), \text{nl}}_{s, 0, 3, \beta_2}(y) = \sum_{\beta_2'} \sum_{\pm} e^{\pm i \frac{y^2}{4}} y^{\mp 4 i \alpha_0 + 2(3\nu + \beta_2) - 2(2\nu + 2\beta_2')}\cdot O(y^{-3})\\
	&\;\;\;\;\;\;\;\;\;\;\;\;\;\;\;\;+ \sum_{\beta_2''} \sum_{\pm} e^{\pm i \frac{y^2}{4}} y^{\mp 4 i \alpha_0 + 2(3\nu + \beta_2) - 2(2\nu + 2\beta_2'')}\cdot O(y^{-3})\\
	& \;\;\;\;\;\;\;\;\;\;\;\;\;\;\;\; +  \sum_{\beta_2''} \sum_{\beta_2'}  y^{ 2(3\nu + \beta_2) - 2(2\nu + 2\beta_2') - 2(2\nu +2\beta_2'')} \cdot O(y^{-3}) 
\end{align}
where $ \beta_2''$ is calculated from the decaying  factors in $w_2, \overline{w_2}$ and $ \beta_2'$ from the ones in $ w_1, \overline{w_1}$. We now outline the general induction step in order to conclude the proof of this proposition.\\[4pt]
{\emph{\bf{Step 4} General induction}}:   We proceed with the general induction steps over the powers 
$$ t^{\alpha_1 \nu + \alpha_2},\; t^{\beta_1 \nu + \beta_2}.$$
Assume the claims in the proposition are true for the calculation of all $ \alpha_1, \beta_1 = 1,2, \dots, N$ for some $ N\in \Z_+$ and further for all (but finite) 
$$ \alpha_2 \geq \alpha_{2\star}(\alpha_1),\; \beta_2 \geq \beta_{2\star}(\beta_1).$$
Then we calculate the relevant source terms on the right hand side of \eqref{schrod-wave-y} as in \emph{Step 1}, i.e. we need the asymptotic of the coefficients in  source terms
\begin{align}\label{this-one2}
	&\sum_{\alpha_2 \geq \alpha_{2\star}} t^{(N+1) \nu + \alpha_2 +1} \sum_{s + \tilde{s} = 0}^{\tilde{m}_{\star}} (\log(y)+\frac12\log(t))^s\cdot(\log(y) - \frac12 \log(t))^{\tilde{s}}\cdot \tilde{g}^{(i)}_{s, \tilde{s},N+1,\alpha_2+1}(y),\\[8pt] \label{this-two2}
	&\sum_{\beta_2 \geq \beta_{2\star}} t^{2 \nu + \beta_2} \sum_{s + \tilde{s} = 0}^{\tilde{m}_{\star}}  (\log(y)+\frac12\log(t))^s\cdot(\log(y) - \frac12 \log(t))^{\tilde{s}}\cdot \tilde{h}^{(i)}_{s, \tilde{s},N+1,\beta_2}(y),
\end{align}
where now $ \tilde{g}^{(i)}_{s, \tilde{s},N+1,\alpha_2}(y),  \tilde{h}^{(i)}_{s, \tilde{s},N+1,\beta_2}(y)$ are determined by sums of the form
\begin{align}\label{this-one3}
	&\sum_{\alpha_2' + \tilde{\alpha}_2 = \alpha_2-1} \sum_{\substack{s_1 + s_2 = s\\ \tilde{s}_1 + \tilde{s}_2 = \tilde{s}}}g_{s_1, \tilde{s}_1,\alpha_1',\alpha_2'}(y)\cdot  h_{s_2, \tilde{s}_2,\tilde{\alpha}_1,\tilde{\alpha}_2}(y),\;\;\;\alpha_1' + \tilde{\alpha}_1 = N+1,\\\label{this-two3}
	& \sum_{\beta_2' + \tilde{\beta}_2 = \beta_2}\sum_{\substack{s_1 + s_2 = s\\ \tilde{s}_1 + \tilde{s}_2 = \tilde{s}}}  g_{s_1, \tilde{s}_1,\beta_1',\beta_2'}(y) \cdot \overline{g_{s_2, \tilde{s}_2,\tilde{\beta}_1,\tilde{\beta}_2}(y)},\;\;\;	\beta_1' + \tilde{\beta}_1 = N+1.
\end{align}
Clearly we observe from the induction assumption\\[3pt]
$\bullet$\;\; The Schr\"odinger coefficient in \eqref{this-one3} in the region $ y \gg1 $,  involves a sum of terms of the form
\begin{align*}
	\sum_{\alpha_1^1, \alpha_1^2, \alpha_2^1, \alpha_2^2}	e^{i (m_1 + m_2) \frac{y^2}{4}} y^{2i \alpha_0(1 - 2(m_1+ m_2)) } \cdot y^{2(N+1  + \alpha_2) -  (\alpha_1^1 \nu + \alpha_2^1) -(\alpha_1^2\nu + \alpha_2^2) - \ell} \cdot O(y^{-3}),
\end{align*} 
where there are $ n_1, n_2$ with $-n_1 +1 \leq m_1 \leq n_1,\; -n_2 \leq m_2 \leq n_2$ and $ \ell \in \Z_{\geq 0}$. Thus applying  Lemma \ref{lem:inhomschrodeqnlargey}, its proof and Lemma \ref{lem:gluey} (Corollary \ref{cor:hom-part-oscill-large-y}) for the homogeneous parts, gives a consistent  form of the $N+1$-st coefficients.\\[3pt]
$\bullet$\;\; The wave coefficient in \eqref{this-two3} in the region $ y \gg1 $, involves a sum terms of the form
\begin{align*}
	&\sum_{\beta_1^j, \beta_2^j}	e^{i (m_1 - m_2) \frac{y^2}{4}} y^{2i \alpha_0(1 - 2m_1)  - 2i \alpha_0(1 - 2m_2)} \cdot y^{2(N+1  + \alpha_2) -  (\beta_1^1 \nu + \beta_2^1) -(\beta_1^2\nu + \beta_2^2) - \ell } \cdot O(y^{-3})\\
	&=  \sum_{\beta_1^j, \beta_2^j}	e^{i (m_1 - m_2) \frac{y^2}{4}} y^{- 4 i \alpha_0(m_1 - m_2)} \cdot y^{2(N+1  + \alpha_2) -  (\beta_1^1 \nu + \beta_2^1) -(\beta_1^2\nu + \beta_2^2) - \ell } \cdot O(y^{-3}),
\end{align*} 
where there are $ n_1, n_2$ with $-n_1 +1 \leq m_1 \leq n_1,\; -n_2 +1 \leq m_2 \leq n_2$ and $ \ell \in \Z_{\geq 0}$.
Then following the proof of Lemma \ref{lem:inhomwaveeqnlargey}, as well as applying Lemma \ref{lem:tyoscinducprep} for the homogeneous parts, gives a consistent  form of the $N+1$-st coefficients.\\[3pt]
Note we also need to include products with the homogeneous parts, however these terms are handled as above or as before in \emph{Step 1} - \emph{Step3}. We spare further details and conclude the induction.

\end{proof}
\begin{Rem}
The large-$y$ expansions of $ h^{(i), \text{hom}}_{s, 0, \beta_1, \beta_2}(y)$ and $h^{(i)}_{s, \tilde{s}, \beta_1, \beta_2}(y)$ are of course  always real valued, in fact from the form of wave source 
$$ (\partial_y^2 + 3 y^{-1}\partial_y)|w(t,y)|^2,$$
we calculate the \emph{leading contributions} for the wave part as $ y \gg1$ ($\beta_1 = 2n+1, \beta_1 = 2n$) to be
\begin{align*}
	\sum_{m = 0}^n &\big[e^{i m \frac{y^2}{4} - im4\alpha_0 \log(y)} \cdot  c_m y^{\tilde{c_m}} + e^{-i m \frac{y^2}{4} + im4\alpha_0 \log(y)} \cdot  \overline{c_m} y^{\tilde{c}_m}\big]\\
	&= \sum_{m = 0}^n y^{\tilde{c}_m} \big(2\cos\big(m \tfrac{y^2}{4} - 4m \alpha_0 \log(y)\big)\cdot \text{Re}(c_m)   - 2 \sin\big(m \tfrac{y^2}{4} - 4m \alpha_0 \log(y)\big) \cdot \text{Im}(c_m) \big).
\end{align*}
\end{Rem} 
\;\\
Let $ \chi_{\mathcal{S}} : (0, \infty) \times\R^4  \to [0,1] $ be a smooth cut-off function with support in the region where $ (t,|x|) \in \mathcal{S}$. As discussed for the definition $\tilde{u}_S^{n_2}, n_S^{N_2}$ of the self-similar approximation in  Definition \ref{defn:self-similar}, we now change to $ (t,R) (t, \lambda(t) r )$ coordinates. Thus we set
\begin{align*}
u_S^{N_2}(t,R) : = \lambda^{-1}(t) t^{- \f12} \tilde{u}_S^{N_2}n_2(t,R t^{\nu}) = t^{\nu} \tilde{u}^{N_2}_S(t,R t^{\nu}),\;\; (t,r) \in \mathcal{S}.
\end{align*}
Then we have the following as a consequence of Proposition \ref{prop:tysystemsolution2} and Proposition \ref{prop:tyoscillatory}. 
\begin{Cor}[Estimates in $\mathcal{S}$]\label{cor:estimates-in-y} Let $\alpha_0 \in \R, 0 < \epsilon_1 \ll1 $ be fix and $ 0 < \epsilon_2< \f12 $ arbitrary. Then there exists $ 0 < \tilde{t}_0(|\alpha_0|, \nu,  N_2, N_2^{(schr)}, N_2^{(wave)}, \mathcal{N}) \leq 1$ such that for any $ 0 < t_0 \leq \tilde{t_0}$ the self-similar approximations $u_S^{N_2}, n_S^{N_2}$  satisfy on $ t \in (0, t_0)$
\begin{align*}
	&\| \chi_{\mathcal{S}} \cdot R^{-j} \partial_R^i (W - u_S^{N_2})\|_{L^{\infty}_R} \leq C_{\nu, |\alpha_0|} t^{\nu},\;\; 0 \leq i + j \leq 2,\; i, j \geq 0,\\[3pt]
	& \| \chi_{\mathcal{S}} \cdot R^{-j} \partial_R^i (W^2- \lambda(t)^{-2} n_S^{N_2})\|_{L^{\infty}_R} \leq C_{\nu, |\alpha_0|} t^{2\nu+1},\;\; 0 \leq i + j \leq 2,\; i, j \geq 0\\[3pt]
	& \| \chi_{\mathcal{S}} \cdot R^{-j} \partial_R^i (W - u_S^{N_2})\|_{L^2_{R^3dR}} \leq  C_{\nu, |\alpha_0|} t^{c_1\nu(1 - 2\epsilon_2)},\;\; i+j \geq 0\\[3pt]
	& \| \chi_{\mathcal{S}} \cdot R^{-j} \partial_R^i (W^2- \lambda(t)^{-2} n_S^{N_2})\|_{L^2_{R^3dR}} \leq C_{\nu, |\alpha_0|} t^{c_1\nu(1 - 2\epsilon_2)},\;\; i+j \geq 0\\[3pt]
	&\| \chi_{\mathcal{S}} \cdot R^{-j} \partial_R^i e_S^{w, N_2}\|_{L^{\infty}_R} + \| \chi_{\mathcal{S}} \cdot R^{-j} \partial_R^i e_S^{w, N_2}\|_{L^2_{R^3dR}} \leq C_{\nu, |\alpha_0|} t^{(c_2 \tilde{N}_2(1-2\epsilon_2) - c_3) \nu},\;\;  i + j \geq 0,\\[3pt]
	&\| \chi_{\mathcal{S}} \cdot R^{-j} \partial_R^i e_S^{n, N_2}\|_{L^{\infty}_R} + \| \chi_{\mathcal{S}} \cdot R^{-j} \partial_R^i e_S^{n, N_2}\|_{L^2{R^3dR}} \leq C_{\nu, |\alpha_0|} t^{(c_2 \tilde{N}_2(1-2\epsilon_2) - c_3) \nu},\;\;   i + j \geq 0,
\end{align*}
where $ c_1, c_2 > 0$ are universal constants,  %(depending only on dimension $d=4$), 
$c_3 = c_3(\epsilon_1, \epsilon_2)> 0$ and $\tilde{N}_2 \in \Z,\; \tilde{N}_2 \gg1$ is such that
$ \tilde{N}_2 \leq \min\{ N_2, N_2^{(schr)}, N_2^{(wave)}\}$.
\end{Cor}
\begin{Rem} We recall $ \mathcal{N} \in \Z_{+}$ is the upper bound for $ l \in \Z_+$ in \eqref{Schrod-ansatz3} and \eqref{wave-ansatz3} counting $ l \geq 1$ (or $l \geq 2$) as in \eqref{constraint-1} - \eqref{constraint-4}. Hence $\mathcal{N}$ is determined in Section \ref{sec:inner}  by the accuracy of the $\Box^{-1}$ parametrix. In the above Corollary \ref{cor:estimates-in-y}, it suffices to fix $0 < \epsilon_1 < 10^{-1} $. The estimates are not sharp, in particular in the first two lines for example we obtain extra decay as $ t \to 0^+$ if $ i + j > 0$ . Lastly $c_1, c_2, c_3 > 0$ may be calculated explicitly, however we do not rely on more precise statements.
\end{Rem}

\begin{Rem}\label{rem:lambdamodulation2} The same observation as in Remark~\ref{rem:modulatinglambda} applies. 

\end{Rem}

\section{The remote region  $ r \gtrsim t^{\frac{1}{2} - \epsilon_2}$}\label{sec:remote}
Following Perelman's ansatz, see \cite[Section 2.4]{Perelman},  we now construct a new approximate solution in the restricted \emph{remote region}
\[
\mathcal{R}  = \{ (t,r)\;|\; r \geq c_2^{-1} t^{\f12 - \epsilon_2}  \}.
\]
We therefore proceed by interpreting  $(w, n)$ as an approximation in the overlap $ \mathcal{R} \cap \mathcal{S}$, where we change back to the $ (r,t)$ coordinate frame for the  large $ y \gg1 $ expansions. The new approximation, extending into the full $\mathcal{R}$ region, is obtained perturbatively around the \emph{time-independent function separated from the slow oscillatory  part as $ t \to 0^+$}.

% In order extract this radiation part, we carefully observe from Proposition 1.11 that the leading part of $w$ in the regime is given precisely by the contribution which doesn’t vanish when $t \to 0$, and which is given.
\;\\
In order to extract this radiation profile, we carefully observe from Proposition~\ref{prop:tyoscillatory} that the leading part of the Schr\"odinger  term $w(t,y)$ in the regime $r \gtrsim t^{\frac12-\epsilon_2}$ (whence $y\gtrsim t^{-\epsilon_2}$) is given by the contribution which does not vanish as $t\rightarrow 0$. To be precise, we schematically write for the expansions \eqref{diese-line-1} and \eqref{diese-line2} of Proposition \ref{prop:tyoscillatory} %via $y = t^{-\f12} r$
\begin{align}
w_N^*(t,t^{-\f12}r)	=&\;\sum_{\alpha_1=1}^N \sum_{\alpha_2 \geq \alpha_{2\star} }t^{\alpha_1 \nu + \alpha_2}  \sum_{s + \tilde{s} \leq \tilde{m}_{\star}}(\log(r))^{s} \cdot (\log(r \slash t)))^{\tilde{s}}\cdot g^{(1)}_{s, \tilde{s}, \alpha_1,\alpha_2}( t^{-\f12}r)\\ \nonumber
&\; + \sum_{\alpha_1=1}^N \sum_{\alpha_2 \geq \alpha_{2\star} }t^{\alpha_1 \nu + \alpha_2 }  \sum_{s + \tilde{s} \leq \tilde{m}_{\star}}(\log(r))^{s} \cdot (\log(r\slash t))^{\tilde{s}}\cdot g^{(2)}_{s, \tilde{s}, \alpha_1,\alpha_2}( t^{-\f12}r),\\[3pt]
n_N^*(t,t^{-\f12}r)	=&\;\sum_{\beta_1=1}^N \sum_{\beta_2 \geq \beta_{2\star} }t^{\beta_1\nu +\beta_2} \sum_{s + \tilde{s} \leq \tilde{m}_{\star}}(\log(r))^{s}\cdot  (\log(r\slash t))^{\tilde{s}}\cdot h^{(1)}_{s, \tilde{s}, \beta_1,\beta_2}( t^{-\f12}r)\\[4pt] \nonumber
&\;+\sum_{\beta_1=1}^N \sum_{\beta_2 \geq \beta_{2\star} }t^{\beta_1\nu +\beta_2} \sum_{s + \tilde{s} \leq \tilde{m}_{\star}}(\log(r))^{s}\cdot  (\log(r\slash t))^{\tilde{s}} \cdot h^{(2)}_{s, \tilde{s}, \beta_1,\beta_2}( t^{-\f12}r),
\end{align}
where in the respective second lines we sum $ \beta_2, \alpha_2 \in \Z + \f12$. Therefore isolating the large-$y$ leading terms, located  either where $ s = 0$ or $ \tilde{s} = 0$ for $w^{\ast}_N(t,y)$ and  where $ \tilde{s} = 0$ for $ n^{\ast}_N(t,y)$, we have the following expansions by Proposition \ref{prop:tyoscillatory}
\boxalign[15cm]{
\begin{align} \label{separation1}
	w_N^*(t,t^{-\f12}r)	=&\; t^{\f 1 2} e^{- i \alpha(t)} \sum_{\alpha_1=1}^N \sum_{\alpha_2 \geq \alpha_{2\star} }  \sum_{s  = 0}^{\tilde{m}_{\star}}(\log(r))^{s} \; r^{2 i \alpha_0 + 2(\alpha_1 \nu + \alpha_2) -1} \big( \tilde{c}_{s, \alpha_1, \alpha_2} + O(t \slash r^{2}) \big)\\ \nonumber
	&+  t^{\f 3 2} e^{- i \alpha(t)} e^{i \frac{r^2}{4t}}\sum_{\alpha_1=1}^N \sum_{\alpha_2 \geq \alpha_{2\star} }  \sum_{\tilde{s}  = 0}^{\tilde{m}_{\star}}(\log(r\slash t))^{\tilde{s}} \; r^{- 2 i \alpha_0 - 2(\alpha_1 \nu + \alpha_2) -3} \big( \tilde{d}_{\tilde{s},  \alpha_1, \alpha_2 } + O(t \slash r^{2}) \big)\\[4pt] \nonumber
	& + h.o.t.\\[3pt] \label{separation2}
	n_N^*(t,t^{-\f12}r)	=&\;\sum_{\beta_1=1}^N \sum_{\beta_2 \geq \beta_{2\star} } \sum_{s = 0}^{ \tilde{m}_{\star}}(\log(r))^{s} r^{2(\beta_1 \nu + \beta_2)} \big( c_{s, \beta_1, \beta_2} + O(t \slash r^{2})\big)\\[4pt] \nonumber
	& + h.o.t.
\end{align}
}
In particular the `higher order terms' and the product with $O(t \slash r^{2})$ vanish pointwise as $ t \to 0^{+}$ seen directly using $ t^{n \f12} \cdot r^{-n} \lesssim_n t^{n\epsilon_2 }$ for $ n \in \Z_+$. We further recall the homogeneous leading order wave part in the above expansion actually has the form
\[
\sum_{\beta_1=1}^N \sum_{\beta_2 \geq \beta_{2\star} } t^{\beta_1 \nu + \beta_2 }\sum_{s = 0}^{ \tilde{m}_{\star}}(\log(r))^{s} \big(y^{2(\beta_1 \nu + \beta_2)} c_{s}   + y^{2(\beta_1 \nu + \beta_2)-2} \big(d_{s} + O(y^{-2})\big)\big),
\]
where $ c_s , d_s$ are chosen where $ 0 < y \lesssim 1$ (satisfying the  matching condition). Thus applying $ \partial_t $, we easily infer the left terms $y^{2(\beta_1 \nu + \beta_2)} $ cancel and hence
\begin{align*}
\partial_tn_N^*(t,t^{-\f12}r)	=&\;\sum_{\beta_1=1}^N \sum_{\beta_2 \geq \beta_{2\star}} \sum_{s = 0}^{ \tilde{m}_{\star}}(\log(r))^{s} \big( (1 - \beta_1 \nu - \beta_2) r^{2(\beta_1 \nu + \beta_2 -1)} \big(d_{s} + O(t\slash r^2)\big) + h.o.t.
\end{align*}
This shows the radiation part of  the ion density velocity $\partial_t n$ is determined only by these second type of coefficients.
\begin{comment}
\begin{align*}
	w_{main}(t, y) : = &\sum'_{a,\alpha_1,\alpha_2}t^{\nu\alpha_1 + \alpha_2}\cdot(\log y + \frac12\log t)^a\cdot y^\zeta\cdot g_{0\zeta a0\alpha_1\alpha_2}\\
	& + \ldots,
\end{align*}
and where (see \eqref{eq:tildegamma})
\[
\zeta = 2i\alpha_0-1+2(\nu\alpha_1+\alpha_2)
\]
\end{comment}
\;\\[15pt]
Further, keeping in mind that we return to the variable $u(t,r)$ of the Schr\"odinger line via the formula 
\[
u(t,r) = e^{i\alpha(t)}t^{-\frac12}w(t, t^{-\f12}r),
\]
we infer that the main contribution to $u(t,r)$ in \eqref{separation1} only depends on  the radial variable $ r > 0$ while the strongly oscillating part decays.  In particular \eqref{separation1} and \eqref{separation2} (and the discussion above) suggest the following \emph{radiation profiles}
\begin{align}\label{rad1}
(schr)\;\;\;\;\;\;\;&f_0(r) : = \sum_{\alpha_1=1}^N \sum_{\alpha_2 \geq \alpha_{2\star} }  \sum_{s  = 0}^{\tilde{m}_{\star}}(\log(r))^{s} \; r^{2 i \alpha_0 + 2(\alpha_1 \nu + \alpha_2) -1}  \cdot \tilde{c}_{s, \alpha_1, \alpha_2},\\[4pt] \label{rad2}
(wave)\;\;\;\;\;\;\;& g_0(r) : = \sum_{\beta_1=1}^N \sum_{\beta_2 \geq \beta_{2\star} } \sum_{s = 0}^{ \tilde{m}_{\star}}(\log(r))^{s} r^{2(\beta_1 \nu + \beta_2)}  \cdot c_{s, \beta_1, \beta_2}\\ \label{rad3}
(wave)_t\;\;\;\;\;\;\;& g_1(r) : = \sum_{\beta_1=1}^N \sum_{\beta_2 \geq \beta_{2\star}} \sum_{s = 0}^{ \tilde{m}_{\star}}(\log(r))^{s} \big( (1 - \beta_1 \nu - \beta_2) r^{2(\beta_1 \nu + \beta_2 -1)} \big) \cdot d_{s, \beta_1, \beta_2}.
\end{align}
%is of the form 
%\begin{align*}
%u_{main}(t, r) = \sum'_{a,\alpha_1,\alpha_2} \log^ar\cdot r^{\nu\alpha_1+\alpha_2+2i\alpha_0}\cdot g_{a\alpha_1\alpha_2}, 
%\end{align*}
%and hence only depending on $r$. 
%\\
%Similarly, write the leading part of $n$ in the regime $r\gtrsim t^{\frac12-\epsilon}$ as 
%\begin{align*}
%n_{main}(t, r) = \sum'_{b,\beta_1,\beta_2} \log^br\cdot r^{\nu\beta_1+\beta_2}\cdot n_{b\beta_1\beta_2}
%\end{align*}
Following Perelman's method, we truncate these expressions to a region of the form $r \leq \delta$ for some small $\delta>0$, and then construct a solution as perturbation of this expression. Thus we set for  a  suitable cut-off function $ \chi \in C^{\infty}(\R_{\geq 0})$%shall set
\begin{align*}
&u_0(r) = \chi_{\lesssim \delta}(r) \cdot  f_0(r),\;\;n_0(r) = \chi_{\lesssim \delta}(r) \cdot  g_0(r),\;\;n_1(r) = \chi_{\lesssim \delta}(r) \cdot  g_1(r),
\end{align*}
and hence obtain for instance directly from \eqref{rad1} - \eqref{rad3}
\begin{align*}
& \| u_0 \|_{\dot{H}^k(\R^4)} \lesssim\delta^{2\nu +1-k} ,\;\; k \leq 2\nu,\;\;\; \| u_0 \|_{\dot{H}^k(\R^4)} \lesssim\delta^{2\nu - 1-k},\;\; k \leq 2\nu -\frac{3}{2}.
\end{align*}
We note clearly from the above expansions the sharp regularity (say in $H^s(\R^4)$) in all cases is of the form $H^{2\nu -c}(\R^4)$ for some $ c > 0$.
%Thus in particular
%\begin{align*}
%&u_0 = \chi_{r\lesssim \delta}\cdot\sum'_{a,\alpha_1,\alpha_2} \log^ar\cdot r^{\nu\alpha_1+\alpha_2+2i\alpha_0}\cdot g_{a\alpha_1\alpha_2}\\
%&n_0 = \chi_{r\lesssim \delta}\cdot\sum'_{b,\beta_1,\beta_2} \log^br\cdot r^{\nu\beta_1+\beta_2}\cdot n_{b\beta_1\beta_2}, 
%\end{align*}
and construct 
\begin{align}
&u(t,r) = u_0(r)+ \phi(t,r),\;\; \phi(0,r) =0,\;\;\;\;
\begin{cases} n(t,r) = n_0(r)+\psi(t,r),&\\[4pt]
	\psi(0,r) = 0,\;\;\partial_t \psi(0,r) = n_1(r).&
\end{cases}
\end{align}
%u(t,r) = u_0(r)+ \phi(t,r), \;\;n(t,r) = n_0(r)+\psi(t,r),
where the functions $\phi(t,r), \psi(t,r)$ are given in terms of expansions of the following kind: 
\begin{equation}\label{eq:phistructure}\begin{split}
	&\phi(t, r) =  \sum_{\alpha_1, \alpha_2} t^{\alpha_1 \nu + \alpha_2 }  \sum_{s, m} (\log(r) - \log(t))^s \cdot e^{-im\Phi(t,r)}\cdot G_{s, \alpha_1, \alpha_2, m}(r),\\&\Phi(t,r) = \frac{r^2}{4t} + 2\alpha_0\log(t),
	\end{split}\end{equation}
	where $  \sum_{\alpha_1,\alpha_2} $ is finite over $\alpha_1 \in \Z_{\geq 0}, \beta_2 \in \Z  $(or $\Z + \f12$) and for each fixed $\alpha_1,\alpha_2$ the sum $ \sum_{s,  m}  $ is finite over over $m \in \Z ,s \in \Z_{\geq 0} $ (the latter depending only on  $ \alpha_1$). Similarly  we have
	\begin{equation}\label{eq:psistructure}
\psi(t, r) = \sum_{\beta_1, \beta_2} t^{\beta_1\nu +\beta_2} \sum_{s, m} (\log(r) - \log(t))^{s} \cdot e^{-im\Phi(t,r)}\cdot N_{s, \beta_1, \beta_2, m}(r). 
\end{equation}
In the preceding sums, the $t$-powers are all positive, and more precisely, we can restrict to $\alpha_1\nu + \alpha_2\geq \nu$ if $\alpha_1>0$, while $\alpha_2\geq 1$ if $\alpha_1 = 0$, and similarly for $\beta_1, \beta_2$ where we assume $ \beta_1\nu + \beta_2 \geq \nu -2$.\\[4pt]
We observe right away that these expansions are compatible with Proposition~\ref{prop:tyoscillatory}, when translating the expressions there into the $(t, r)$-coordinate system, and using  that 
\[
\log(y) - \frac12\log(t) = \log(r) - \log(t),\,\log(y) + \frac12\log(t) = \log(r). 
\]
First, we derive the equations for $\phi(t,r), \psi(t,r)$, which are given as follows: 
\begin{align}
&(i\partial_t + \Delta)\phi = -\Delta u_0 + (n_0+\psi)(u_0+\phi)\label{eq:phi}\\[3pt]
&(-\partial_{t}^2 + \Delta)\psi = \Delta\big(|u_0+\phi|^2\big) - \Delta n_0\label{eq:psi}
\end{align}
Referring to \eqref{eq:phistructure}, \eqref{eq:psistructure}, we now use these equations to deduce a system of inductive relations which will allow us to determine the coefficient functions 
\[
G_{s, \alpha_1, \alpha_2,m}(r),\; N_{s, \beta_1, \beta_2, m}(r). 
\]
Specifically, keeping in mind the definition of the phase $\Phi(t,r)$, and calculating the coefficient corresponding to $t^{\nu\alpha_1 + \alpha_2 - 2}$, we deduce for the $\phi$-coefficients the recurrent relations
\boxalign[15cm]{
\begin{align}\label{eq:Gam}
	&-\frac{m(m+1)}{4}r^2\cdot G_{s, \alpha_1,\alpha_2,m} + i(\nu\alpha_1 + \alpha_2 - 1-2m - 2i\alpha_0 m)\cdot G_{s, \alpha_1,\alpha_2-1,m}\\[3pt] \nonumber
	&-i(m+1)(s+1)G_{s+1, \alpha_1,\alpha_2-1,m} - im\cdot r\partial_rG_{s, \alpha_1,\alpha_2-1,m} - n_0(r) \cdot G_{s, \alpha_1,\alpha_2-2,m}\\[3pt] \nonumber
	& - u_0(r) \cdot N_{s, \alpha_1\alpha_2-2, m}\\[3pt]  \nonumber
	& = \sum'_{\substack{\alpha_1'+\beta_1' = \alpha_1\\ \alpha_2' + \beta_2' = \alpha_2-2}} \sum_{\substack{s'+\tilde{s} = s\\m'+\tilde{m} = m}}G_{s', \alpha_1', \alpha_2', m'}\cdot  N_{\tilde{s}, \beta_1', \beta_2', \tilde{m} } + \delta_{m,0}\cdot\delta_{s,0}\cdot\delta_{\alpha_1,0}\cdot \delta_{\alpha_2,2}\cdot E_0(r),\\[12pt] \nonumber
	&E_0(r) = -\Delta u_0(r) + u_0(r)\cdot n_0(r).
\end{align}
}
Similarly, by calculating the coefficient of $t^{\nu\beta_1 + \beta_2 - 4}$, we then obtain for the $\psi$-coefficients the relations
\;\\
\boxalign[15cm]{
\begin{align}\label{eq:Nbm}
	&m^2r^4\cdot N_{s, \beta_1, \beta_2,m} + 2im(\nu\beta_1+\beta_2-1)\cdot r^2N_{s, \beta_1, \beta_2-1,m} \\[3pt]  \nonumber
	& + 2ims\cdot r^2N_{s+1,\beta_1\beta_2-1, m} - (\nu\beta_1+\beta_2-2)(\nu\beta_1+\beta_2-3)N_{s, \beta_1, \beta_2-2,m} \\[3pt]  \nonumber
	&+ 2(\nu\beta_1+\beta_2+1)\cdot s\cdot N_{s+1, \beta_1, \beta_2-2,m} - s(s-1)N_{s+2, \beta_1, \beta_2-2,m} \\[3pt]  \nonumber
	&- \frac{m}{4} r^2\cdot N_{s, \beta_1, \beta_2-2, m}-msi r \cdot N_{s+1, \beta_1, \beta_2-3,m} - im\cdot r\partial_rN_{s, \beta_1, \beta_2-3,m} \\[3pt]  \nonumber
	&-\frac{3m}{2} N_{s, \beta_1, \beta_2-3,m}  + \Delta \big(N_{s, \beta_1, \beta_2-4,m} \big) + 2s\cdot\partial_rN_{s+1, \beta_1, \beta_2-4,m}\\[3pt]  \nonumber
	& +\frac{s(s+2)}{r^2}\cdot N_{s+2, \beta_1, \beta_2-4,m}\\[3pt]  \nonumber
	& = Z(r,\{G_{s', \alpha_1', \alpha_2', m'}\}) +  \delta_{m,0}\cdot\delta_{s,0}\cdot\delta_{\alpha_1,0}\cdot\delta_{\alpha_2,4}\cdot F_0(r),\\[12pt] \nonumber
	&F_0(r)= \Delta (|u_0|^2) - \Delta n_0(r).  
\end{align}
}
\;\\
where the expression on the right is a function of $r> 0$ and a finite collection of the $G_{s, \alpha_1'\alpha_2'm'}$ obtained by determining the coefficient of $t^{\nu\alpha_1+\alpha_2-4}$ in the wave source term
\[
\Delta (|\phi|^2),
\]
which we calculate by using the expansion \eqref{eq:phistructure}. \;\\[4pt]
We start by making the following simple observations:
\begin{itemize} \setlength\itemsep{3pt}
\item The first equation \eqref{eq:Gam} becomes degenerate for the values $m = 0, -1$. The second equation \eqref{eq:Nbm} becomes degenerate only for $m = 0$. 
\item If $\alpha_1 = 0$ becomes minimal, then all indices $\alpha_1', \beta_1'$ in \eqref{eq:Gam},  as well as all indices $\alpha_1'$ in  \eqref{eq:Nbm} have to also vanish. This in turn implies that all indices $\alpha_2, \alpha_2', \beta_2'$ and so forth thus have to be positive. 
\item Due to the latter observation, we naturally start with  constructing $G_{s, 0, \alpha_2, m}(r), N_{s, 0, \beta_2, m}(r)$ and then continue with the coefficients for $\alpha_1, \beta_1 > 0$ by induction. 
\end{itemize}
{\bf{Step 1}}: {\it{The case $\alpha_1 = \beta_1 = 0$, thus $\alpha_2, \beta_2 >0$.}} Note that here $\alpha_1' = \beta_1' = 0$ in the source terms in \eqref{eq:Gam} and \eqref{eq:Nbm}, so that we can independently construct the coefficients $G_{s, 0, \alpha_2, m}(r), N_{s, 0, \beta_2, m}(r)$. We now specify the function spaces in which the iterates will  be captured. In  particular we use the following versions of $\mathcal{A}, \mathcal{B} $  spaces % `live in', we use the following, %entirely 
akin to Perelman's definition, c.f. \cite[Section 2.4]{Perelman} and \cite[Section 2.3]{schmid} for an analogous application:
\begin{Def}\label{def:outerregionspace} For $k \in \Z_+$ we let  $\mathcal{B}_k$ the space of $C^\infty(0,\infty)$-functions $f(r)$ such that for some small $ 0 <  \delta \ll1$
\begin{itemize}\setlength\itemsep{3pt}
	\item[(i)] we have $ f \in C^0([0, \infty)) $ and for $ r > 2 \delta$ the function $f(r)$  is a polynomial of degree $\leq k-1$
	\item[(ii)] $f(r) $ admits a Puiseux type expansion  in an absolute sense 
	\[
	f(r) = \sum_{\alpha_1,\alpha_2 } \sum_{\ell \geq 0} r^{2(\nu\alpha_1+\alpha_2) }(\log(r))^{\ell} a_{\ell, \alpha_1, \alpha_2},\;\; 0 < r < \delta,
	\]
	where  the second sums are finite and $ \alpha_1 \in \Z_{\geq 0}, \alpha_2 \in \Z$ (or $ \alpha_2 \in \Z + \f12$) and $  \alpha_1\nu + \alpha_2 \geq 0$ (and $ \geq \nu -2$ if positive).
\end{itemize}
Also we let $\mathcal{A}$ be the space of $C^\infty(0,\infty)$-functions  $ g(r)$ which are continuous on $ \R_{\geq 0}$, have support in $\{  r\leq 2\delta\}$ and admit an expansion at $ r=0$ similar as above in an absolute sense, i.e. 
\[
g(r) = \sum_{\alpha_1,\alpha_2} \sum_{\ell \geq 0} r^{\nu\alpha_1+\alpha_2}(\log(r))^{\ell} b_{\ell, \alpha_1, \alpha_2},\;\; 0 < r < \delta,
\]
where the second sum is again finite and $\alpha_1 \geq 0,\; \alpha_1 \nu + \alpha_2 $ are as above.
\end{Def} 
Then we have 
\begin{lem}\label{lem:outerregionalpha1=0} The systems \eqref{eq:Gam} and \eqref{eq:Nbm} restricted to the case $\alpha_1 = \beta_1= 0$ admit a unique solution compatible with the solution from Proposition~\ref{prop:tyoscillatory} (in the sense of expansions as in \eqref{separation1} and \eqref{separation2}) such that
\begin{align*}
	&G_{s,0, \alpha_2,m}\in r^{2i (1+2m)\alpha_0- 2\alpha_2}\mathcal{A},\,m\neq -1,\,\;G_{s, 0, \alpha_2, -1}\in r^{-2i \alpha_0- 2\alpha_2-1}\mathcal{B}_{\alpha_2},\\[3pt]
	&N_{s, 0, \beta_2, m}\in r^{4 i m\alpha_0 - 2\alpha_2}\mathcal{A},\,m\neq 0,\,\;N_{s, 0, \alpha_2, 0} \in r^{- 2\alpha_2-1}\mathcal{B}_{\alpha_2},
\end{align*}
where by $\mathcal{B}_{\alpha_2}$ we denote the space of polynomials of degree $\leq \frac32 \alpha_2 -\frac12$ here. 
\end{lem}
\begin{proof} We use an inductive procedure. First let $\alpha_{2\star}$ be the minimal index $\alpha_2$ occurring in both $G_{s, 0, \alpha_2, m}, N_{s, 0, \alpha_2, m}$. It is easily seen that necessarily $\alpha_{2\star} = 1$ for $m = 0$. To determine the function $G_{s, 0, 1, 0}(r)$, it suffices to analyze \eqref{eq:Gam} with $\alpha_2 = 2$ and by induction on $s \geq 0$, starting with the maximal value, we infer that necessarily $s = 0$, as well as correspondingly
\[
G_{0,0, 1, 0}(r)  = \big(\frac{-i}{\nu\alpha_1 + \alpha_2 - 1}\big)\big|_{\alpha_1=0,\,\alpha_2 = 2}\cdot E_0(r).
\]
Therefore the coefficient $ G_{0,0, 1, 0} $ is uniquely determined through $u_0(r), n_0(r)$ (see the above definition of $E_0(r)$), and certainly in $r^{2i \alpha_0-2}\mathcal{A}$. 
\\[4pt]
On the other hand, the function $N_{s, 0, 1, 0}(r)$ for $s = 0$ can be prescribed arbitrarily, which we have use to match the previous large-$y$ solution, while again by downward induction on $s \geq 0$ one infers that $N_{s, 0, 1,0}(r) = 0$  if  $s\neq 0$. To see this we  simply  put $\beta_2 = 3$ in \eqref{eq:Nbm}.
\\[4pt]
Let us next determine the remaining coefficient functions $G_{s, 0, 1,m}(r), N_{s, 0, 1, m}(r)$, starting with the exceptional value $m = -1$. Here we notice that in this case  \eqref{eq:Gam} with $\alpha_1 =0, \alpha_2 = 2$ simplifies to 
\begin{align*}
	(3+2i\alpha_0)\cdot G_{s,0, 1,-1}(r) + (r\partial_r)G_{s, 0, 1, -1}(r) = 0, 
\end{align*}
which implies 
\begin{align*}
	G_{s, 0, 1, -1}(r) = c\cdot r^{-3-2i\alpha_0},
\end{align*}
where $c \in \mathbf{C}$ is determined by matching (the coefficient) with the previous analogous large-$y$ expansion % the inner solution in the regime $r\lesssim t^{\frac12-\epsilon}$.
Note that this is compatible with the statement of the Lemma.
\\[4pt]
As for the functions $N_{s,0, 1, -1}(r)$, setting $\beta_1 = 0, \beta_2 = 1, m = -1$ in \eqref{eq:Nbm} implies 
\[
N_{s,0, 1,-1}(r) = 0\,\,\forall s \geq 0. 
\]
Next, we assume $m\notin\{0, 1\}$. Plugging $\alpha_1 = 0, \alpha_2 = 1$ in \eqref{eq:Gam},  but letting $ s, m $ range over arbitrary values, gives
\[
G_{s, 0, 1, m}(r) = 0,\,\; m\notin \{0, -1\}. 
\]
and similarly we infer $N_{s,0, 1, m}(r)= 0$ in all other non-exceptional cases $ m \neq 0$. 
\\[4pt]
We continue with $\alpha_2 = \beta_2 = 2$, which still has  an exceptional component in view of \eqref{eq:Nbm}. We again first treat the special values $m = 0, -1$, starting by setting $m= 0$ and considering the function $G_{s, 0, 2, 0}(r)$. If we set $\alpha_1 = 0, \alpha_2 = 3, m = 0$ in \eqref{eq:Gam},  we see from the preceding that the source terms on the right of  \eqref{eq:Gam} all vanish. Hence,  performing an induction over all possible values of $s \geq 0$, starting with the maximal integer, we infer 
\[
G_{s,0, 2, 0}(r)  = 0,\,\;\forall s\neq 0,\,\;\;G_{0,0,2,0}(r) = \frac{n_0(r)}{2i}\cdot G_{0, 0, 1,0}(r).  
\]
As for \eqref{eq:Nbm} with $\beta_1 = 0, \beta_2 = 4$, we do not observe any contribution from either the nonlinear interaction term, or the remaining source term. Thus the preceding implies 
\[
N_{s,0,2,0}(r) = 0,\,\; s\neq 0.
\]
On the other hand, if $s = 0$, we get the first nontrivial contribution to the $N$-coefficients by setting $\beta_1 = 0, \beta_2 = 4, m = 0, s = 0$ in \eqref{eq:Nbm}, which gives
\begin{align*}
	N_{0, 0,2,0}(r) = -\frac12 F_0(r),
\end{align*}
and therefore is uniquely determined through $ \Delta u_0(r), n_0(r)$ (see the definition of $F_0(r)$). For the remaining values of $m \in \Z$ and $\alpha_2 = 2$, using \eqref{eq:Gam} with $\alpha_2 = 2$ as well as $G_{0, 1, \alpha_2-1, m}(r) = 0,\, m\neq 0, -1$, and the vanishing of all the source terms in \eqref{eq:Gam} for these parameter values, we conclude 
\[
G_{s, 0, 2, m}(r) = 0,\,\; m \notin \{ 0,-1\}. 
\]
As for the $N-$coefficient functions $N_{s, 0, 2,m}(r)$ with $m\not 0$, we easily check from \eqref{eq:Nbm} that they all have to vanish.\\[8pt] 
In order to determine the coefficients $G_{s, 0, \alpha_2, m}(r),  N_{s, 0, \beta_2, m}(r)$ with $\alpha_2 > 2$, we use  induction. Thus assume these coefficients are known if $\alpha_2 = l \in \Z_{\geq 0}$ is replaced by $\alpha_2 -1 =l-1$. This implies all interaction terms on the right of  \eqref{eq:Gam} and \eqref{eq:Nbm} are completely determined.\\[3pt]
As before we first deal with the exceptional phase values $m \in \{0, -1\}$ or $ m =0$ (for the $N$-case). Let us  begin with $m = 0$ and  $\alpha_2 = l+1$ in \eqref{eq:Gam}, hence we perform an induction over  all possible logarithmic exponents $ s \geq 0$ and  start with the maximal value $s  = s_{\ast}$. Here from \eqref{eq:Gam} we obtain the equation
\begin{align*}
	il\cdot G_{s_{\ast},0, l, 0}(r) - n_0(r) \cdot G_{s_{\ast}, 0,l-1,0}(r) - u_0(r) \cdot N_{s_{\ast}, 0, l-1, 0}(r)  = \tilde{Z}(r),
\end{align*}
where the function on the right side is in fact in the following space:
\[
\tilde{Z}(r) \in r^{-2i p \alpha_0 - 2(l-2)}\mathcal{A},%%% Here it was written -2i\alpha_0p     (p -- I guess a typo ?--no)
\]
for some $ p \in \Z$. By induction we can show this to be of the claimed form in the Lemma. Now the remaining coefficients $G_{s,0,l, 0}(r)$ for $ s < s_{\ast}$ are obtained again by  induction (downwards for $ s_{\ast}, s_{\ast} -1, \dots$) leading to a similar structure with this argument.  
In fact, in the nonlinear interaction term on the right side of \eqref{eq:Gam}, when the $m$-coefficient of the wave-type factor satisfies  $\tilde{m} = 0$, then the other corresponding factor has of course $m' = 0$, and so we are in the compactly supported space. If on the other hand $ \tilde{m}\neq 0$, then according to the induction assumption the product is likewise  compactly supported , hence so are the two linear terms on the left. \\[3pt]
The compact support property of $ \mathcal{A}$ then implies that using the induction assumption we have 
\[
G_{s_{\star}, 0, l, 0}\in r^{-2i p \alpha_0 - 2l}\mathcal{A}.
\]
We proceed similarly for the wave part $N_{s, 0, l, 0}(r)$, where we set $\beta_2 = l+2$ in \eqref{eq:Nbm}. Here there is the possibility that two factors of type $G_{s, 0, \alpha_2', -1}(r)$ contribute to a \emph{non-compactly supported term},  since the phases $e^{\pm i\frac{r^2}{4t}}$ cancel in the product.  From its definition and our choice $m = 0$,  we have
\[
Z(r,\{G_{s, \alpha_1', \alpha_2', m'}\})\in \Delta \big(G_{s', \alpha_1', \alpha_2', m'} \cdot \overline{G_{\tilde{s}, \alpha_1'', \alpha_2'', \tilde{m}}}\big),\;\; \tilde{m} = - m', 
\]
and thus we  observe that 
\begin{align*}
	\big|Z(r,\{G_{s, \alpha_1'\alpha_2', m'}\})\big|&\lesssim r^{-2}\cdot r^{-\frac{\alpha_2'}{2}-\frac32}\cdot r^{-\frac{\alpha_2''}{2}-\frac32}\\
	&=r^{-\frac{l}{2}-4}
\end{align*}
since 
\[
\alpha_2' + \alpha_2'' = l-2. 
\]
Hence using the (downward) induction assumption on $s \geq 0$ to determine the  functions $N_{\tilde{s}, 0, l, 0}(r)$ from \eqref{eq:Nbm},  we infer even a better bound than needed. 
\\[4pt]
Next, consider the case $m = -1, \alpha_2 = l$. Setting $\alpha_2 = l+1$ in \eqref{eq:Gam} results in 
\begin{align*}
	(l+2+2i\alpha_0)\cdot G_{s, 0, l, -1}(r) + (r\partial_r)G_{s, 0, l,-1}(r) = \tilde{\tilde{Z}}(r), 
\end{align*}
where we have 
\begin{align*}
	\big| \tilde{\tilde{Z}}(r) \big|&\lesssim r^{-\frac{\alpha_2'}{2}-\frac32}\cdot r^{-\frac{\beta_2'}{2}-\frac32}\\
	& = r^{-\frac{l}{2}-\frac52}
\end{align*}
since $\alpha_2'+\beta_2' = l-1$. We therefore may solve the preceding equation via variation of constants, which yields 
\begin{align*}
	G_{s, 0, l, -1}(r) =-r^{-l-2-2i\alpha_0} \int_0^r s^{l+1+2i\alpha_0}\cdot \tilde{\tilde{Z}}(s)\,ds + d \cdot r^{-l-2-2i\alpha_0}, 
\end{align*}
where the constant $d$ is again uniquely determined by matching with the large-$y$ solution. In light of the above bound for $\tilde{\tilde{Z}}(r)$, and the fact that the algebraic structure is preserved under the above integration operation, we infer the desired structure for $G_{s, 0, l,-1}(r)$. 
\\[3pt]
Further, using $m = -1,\alpha_1 = 0,\alpha_2 = l$ in \eqref{eq:Nbm} and using the same bound for the nonlinear term as in the case $m = 0$, we infer the desired structure for $N_{s, 0, l, ,-1}(r)$ without integration, but by invoking the inductive hypothesis. In fact, the nonlinear source term, as well as all the other expressions, are compactly supported by the induction hypothesis and since not both factors in the nonlinear interaction term can have $m = -1$. 
\\[3pt]
Recovering the desired structure for the remaining coefficients $G_{s,0,l,m}(r), N_{s, 0, l,m}(r)$ with $m\notin \{0,\,-1\}$ is similar but simpler, since no more integration is required. The key point is here is that in both nonlinear interaction source terms, there is always at least one compactly supported factor, due to the induction hypothesis. 
\end{proof}
{\bf{Step 2:}} {\it{The case $\alpha_1>0$ and $ \beta_1 > 0$}}. This step follows exactly the same pattern of arguments as in the  preceding analysis,  using induction on $\alpha_1, \beta_2 $. In particular we observe that,  in both \eqref{eq:Gam} and \eqref{eq:Nbm}, the nonlinear interaction terms either have $\alpha_1' < \alpha_1$, or exactly one index $\alpha_1' = \alpha_1$ while the other $\beta_1' =0$ (or of course vice versa $\alpha_1'=0,\; \beta_1' < \beta_1$). In this case, we have that all  $\alpha_2' < \alpha_2$ where we note the latter case does of course not occur when $\alpha_2 = \alpha_{2\star}$ is minimal for given $\alpha_1$.\\[3pt]
Hence in order to determine all 
$$G_{s, \alpha_1, \alpha_2, m}(r) , N_{s, \beta_1, \beta_2, m}(r),$$
we use at first an (upward directed) induction on $\alpha_2, \beta_2$, starting with the minimal cases $\alpha_2 = \alpha_{2\star}, \beta_2 = \beta_{2\star}$ where one proceeds similar as previously shown. We spare more details. 
\\[6pt]
In fact we performing this induction, we infer the following.
\begin{lem}\label{lem:outerregionalpha1neq0} The systems \eqref{eq:Gam} and \eqref{eq:Nbm} for $\alpha_1\geq 0, \beta_1 > 0$ admit unique solutions compatible with the solution from Proposition~\ref{prop:tyoscillatory}, and such that 
\begin{align*}
	&G_{s, \alpha_1, \alpha_2, m}\in r^{2i (1 + 2m) \alpha_0- \nu \alpha_1 - 2\alpha_2}\mathcal{A},\,\; m\neq -1,\,\;G_{s, \alpha_1, \alpha_2, -1}\in r^{-2i \alpha_0- \nu\alpha_1 - 2\alpha_2-1}\mathcal{B}_{\alpha_2}\\[3pt]
	&N_{s, \beta_1, \beta_2, m}\in r^{-4i m \alpha_0- \nu\beta_1 - 2\beta_2}\mathcal{A},\,m\neq 0,\,\; N_{s, \beta_1, \beta_2, 0} \in r^{-\nu \alpha_1 - 2\alpha_2-1}\mathcal{B}_{\alpha_2},
\end{align*}
where $\mathcal{B}_{\alpha_2}$ is similar as before the space of polynomials of degree $\leq \max\{\frac32 \alpha_2 -\frac12, 0\}$ (where $ r > 2 \delta$). 
\end{lem}
\;\\
We now define the \emph{remote approximate solution}, which we may later change to back to $(t,a,R)$ coordinates. Thus let us first choose $ N_1 ,  \mathcal{N}, N_2,  N_2^{(schr)}, N_2^{(wave)} \in \Z_{+}$ sufficiently large as in the previous approximations in Section \ref{sec:inner} and Section \ref{sec:self}. This gives us the data in  $(t,r)$ coordinates from the large-$y$ expansion as in the beginning of this Section \ref{sec:remote}. Thus performing the above induction and invoking the Lemma \ref{lem:outerregionalpha1=0}, as well as Lemma \ref{lem:outerregionalpha1neq0},  we obtain the coefficient functions 
\[
\{ G_{s, \alpha_1, \alpha_2, m}(r),\; N_{s, \beta_1, \beta_2, m}(r) \}.
\]
\begin{Def}[Remote approximation] \label{defn:remo}  Let $N_3 \in \Z,\;N_3 \gg1 $. Then we define  for $(t,r) \in \mathcal{R}$ where $ r^{\f12 - \epsilon_2 } \lesssim r$ the functions
\begin{align*}
	&\tilde{u}^{N_3}_{R}(t,r) : =  u_0(r) +  \sum_{\alpha_1} \sum_{\alpha_2 \leq N_3} t^{\alpha_1 \nu + \alpha_2 }  \sum_{s, m} (\log(r) - \log(t))^s \cdot e^{-im\Phi(t,r)}\cdot G_{s, \alpha_1, \alpha_2, m}(r),\\[3pt]
	&n^{N_3}_{R}(t,r) : =  n_0(r) +  \sum_{\beta_1} \sum_{\beta_2 \leq N_3} t^{\beta_1\nu +\beta_2} \sum_{s, m} (\log(r) - \log(t))^{s} \cdot e^{-im\Phi(t,r)}\cdot N_{s, \beta_1, \beta_2, m}(r),\\[3pt]
	& e_R^{u, N_3}(t,r) : = i \partial_t \tilde{u}^{N_3}_{R}(t,r) + \Delta \tilde{u}^{N_3}_{R}(t,r)+ \tilde{u}^{N_3}_{R}(t,r) \cdot n^{N_3}_{R}(t,r),\\[3pt]
	& e_R^{n, N_3}(t,r) : = \big( - \partial_t^2 + \Delta\big)n^{N_3}_{R}(t,r) - \Delta\big(| \tilde{u}^{N_3}_{R}(t,r)|^2 \big).
\end{align*}
\end{Def}
We note here, as the above induction shows,  some of the $G,N$ coefficients here are determined to match the given large-$y$ data. The new approximation in Definition  \ref{defn:remo} is based on continuing this inductive iteration (according to Lemma \ref{lem:outerregionalpha1=0} and Lemma \ref{lem:outerregionalpha1neq0}) indefinitely (based on previous iteration steps) which we do here for $ \alpha_2, \beta_2$.\\[4pt] 
If we want to compare all the approximate solution constructed in each of the regions, we will need to  consider 
\begin{align}
& u_{app, \mathcal{I}}(t, R, t^{\nu - \f12}R),\;\;\; a = r \slash t = t^{\nu - \f12}R,\\[4pt]
&u_S(t,t^{\nu}R ) =  t^{\nu} \cdot \tilde{u}_S(t,t^{\nu}R ) =  \lambda^{-1}(t) t^{- \f12} \cdot \tilde{u}_S(t,t^{\nu}R ),\;\; y = r \slash t^{\f12} = t^{\nu}R,\\[4pt]
& e^{- i\alpha(t)} t^{\f12 + \nu}\cdot \tilde{u}_{R}(t,t^{\f12 + \nu}R) = e^{- i\alpha(t)}\lambda^{-1}(t)\cdot \tilde{u}_{R}(t,t^{\f12 + \nu}R),\;\; r =  t^{\f12 + \nu}R.
\end{align}
For the final approximation  defined in Section \ref{sec:finalsection}, we will essentially `glue' these approximations via cut-off functions and then multiply with $e^{i \alpha(t)} \lambda(t)$. 
Let us first note that we may compare the approximate solutions of the present section in Definition \ref{defn:remo} and the previous Section \ref{sec:self} in Definition \ref{defn:self-similar} and verify they only differ by a \emph{small error} (with fast decay rate) as $t \to 0^+$ in the \emph{overlap region}
$$ \mathcal{R}\cap \mathcal{S} = \{ (t,r) \;|\; c_2^{-1} t^{\f12 - \epsilon_2} \leq r \leq c_2ct^{\f12 - \epsilon_2}\}.$$
In particular, as a consequence of Proposition \eqref{prop:tyoscillatory} and Lemma \ref{lem:outerregionalpha1=0}, as well as Lemma \ref{lem:outerregionalpha1neq0}, we have the following.
\begin{lem}[Consistency in $ \mathcal{R}\cap \mathcal{S}$]\label{lem:consistency-self-sim-rem} Let $N \in \Z_+$, then there exists $N_1, \mathcal{N}, N_2, N_2^{(schr)}, N_2^{(wave)}, N_3 $ in $ \Z_+$ subsequently chosen and  large enough (depending on $ N, \epsilon_2$) such that for $ (t,r) \in \mathcal{S}\cap \mathcal{R}$ with $ y = t^{-\f12}r $
\begin{align*}
	&| y^{-l} \partial_y^{m}\big( \tilde{u}_{S}^{N_2}(t,y) - u_{R}^{N_3}(t, t^{\f12 } y)  \big) | \leq C_{N_1, \mathcal{N}, N_2, N_3, m,l} \cdot t^{c_1 \nu \cdot N - \tilde{\eta}_1},\;\; 0 \leq l + m \leq 2\\[3pt]
	&| y^{-l} \partial_y^{m}\big( n_{S}^{N_2}(t,y) - n_{R}^{N_3}(t,R t^{\f12} y)  \big) | \leq \tilde{C}_{N_1,\mathcal{N}, N_2, N_3, m, l}\cdot  t^{c_2 \nu \cdot N - \tilde{\eta}_2},\;\; 0 \leq l + m \leq 2.
\end{align*}
provided $ 0 < t \ll1$ is  small,  for some $ C, \tilde{C} > 0$ and  where $ c_1, c_2 > 0$, as well as $ \tilde{\eta}_j(\epsilon_2, \nu) > 0$, are fixed.
\end{lem}
%\begin{proof}
%\end{proof}
Let $ \chi_{\mathcal{R}} : (0, \infty) \times\R^4  \to [0,1] $ be a smooth cut-off function with support in the region where $ (t,r) \in \mathcal{R}$. We then consider likewise from Proposition \ref{prop:tyoscillatory} and Lemmas  \ref{lem:outerregionalpha1=0}, \ref{lem:outerregionalpha1neq0}
\begin{Cor}[Estimates in $\mathcal{R}$]\label{cor:estimates-in-(t,r)} (a)\; Let $\alpha_0 \in \R, 0 < \epsilon_1 \ll1 $  and now we fix $ 0 < \f12 - \epsilon_2 \ll 1 $. Then there exists $ 0 < t_0(|\alpha_0|, \nu,  N_j, j = 1,2,3, \mathcal{N}) \leq 1$ such that the approximations  $u_R^{N_3}, n_R^{N_3}$  satisfy on $ t \in (0, t_0)$
\begin{align}
	&\| \chi_{\mathcal{R}} \cdot r^{-l} \partial_r^k (u_0(r)- \tilde{u}_R^{N_3})\|_{L^{\infty}_r}\\[3pt] \nonumber
	&\hspace{1cm} +\| \chi_{\mathcal{R}} \cdot r^{-l} \partial_r^k (u_0(r)- \tilde{u}_R^{N_3})\|_{L^2_{r^3dr}}   \leq C_{\nu, |\alpha_0|} t^{c_1 \nu - (k+l)c_1'(1 -2\epsilon_2)},\;\; 0 \leq l + k \lesssim_{\epsilon_2} 1,\\[3pt]
	&\| \chi_{\mathcal{R}} \cdot r^{-l} \partial_r^k (n_0(r)- n_R^{N_3})\|_{L^{\infty}_r}\\[3pt] \nonumber
	& \hspace{1cm}+\| \chi_{\mathcal{R}} \cdot r^{-l} \partial_r^k (n_0(r)- n_R^{N_3})\|_{L^2_{r^3dr}}   \leq C_{\nu, |\alpha_0|} t^{c_2 \nu - (k+l)c_2'(1 -2\epsilon_2)},\;\; 0 \leq l + k \lesssim_{\epsilon_2} 1,
\end{align}
where $ c_1(\epsilon_2), c_2(\epsilon_2) > 0$ are bounded as $ \epsilon_2 \to \f12^-$  and $ c_1', c_2' > 0$ are universal constants. Further for the error
\begin{align}
	& \| \chi_{\mathcal{R}} \cdot r^{-l} \partial_r^k e_R^{u, N_3}\|_{L^{\infty}_r}  +  \| \chi_{\mathcal{R}} \cdot r^{-l} \partial_r^k e_R^{u, N_3}\|_{L^2_{r^3dr}} \leq C_{\nu, |\alpha_0|} t^{N_3 \tilde{\eta}_1}\\[3pt]
	&  \| \chi_{\mathcal{R}} \cdot r^{-l} \partial_r^k e_R^{n, N_3}\|_{L^{\infty}_r}  +  \| \chi_{\mathcal{R}} \cdot r^{-l} \partial_r^k e_R^{n, N_3}\|_{L^2_{r^3dr}} \leq C_{\nu, |\alpha_0|} t^{N_3 \tilde{\eta}_2}
\end{align}
where $ 0 \leq l + k \leq 2$ and $ 0 < \tilde{\eta}_1, \tilde{\eta}_2 \ll1, N_3 \gg1 $ depending on $ \nu , \epsilon_2$.\\[3pt]
(b)\; Considering the terms with their respective radiation parts, we control the norm via $ 0 < \delta \ll1 $.
\begin{align}
	& \| \chi_{\mathcal{R}} \cdot r^{-l} \partial_r^k  \tilde{u}_R^{N_3}\|_{L^{\infty}_r}  +  \| \chi_{\mathcal{R}} \cdot r^{-l} \partial_r^k  \tilde{u}_R^{N_3}\|_{L^2_{r^3dr}}\leq  C_{\nu, |\alpha_0|} \delta^{\tilde{c}_1\nu - \tilde{\tilde{c}}_1},\\
	& \| \chi_{\mathcal{R}} \cdot r^{-l} \partial_r^k  n_R^{N_3}\|_{L^{\infty}_r}  +  \| \chi_{\mathcal{R}} \cdot r^{-l} \partial_r^k  n_R^{N_3}\|_{L^2_{r^3dr}}\leq  C_{\nu, |\alpha_0|} \delta^{\tilde{c}_2\nu - \tilde{\tilde{c}}_2},
\end{align}
where $\tilde{c}_1 \nu - \tilde{\tilde{c}}_1 > 0$ and $ \tilde{c}_2 \nu - \tilde{\tilde{c}}_2 > 0$ for $ 0 \leq l + k \leq 2$.
\end{Cor}

\begin{Rem}\label{rem:lambdamodulation3} A similar observation as in Remark~\ref{rem:modulatinglambda} applies. In particular, in the remote region, assuming that $\lambda(t) = \tilde{\lambda}\cdot t^{-\frac12-\nu}$, we have the uniform estimates
\begin{align*}
	\big|\partial_{\tilde{\lambda}} \tilde{u}_R^{N_3}\big|_{\tilde{\lambda} = 1}\big|\ll_{\delta} 1,\,\big|\partial_{\tilde{\lambda}}n_R^{N_3}\big|_{\tilde{\lambda} = 1}\big|\ll_{\delta} 1
\end{align*}

\end{Rem}

%\bibliography{mybib}
%\bibliographystyle{alpha}

\section{Final approximation: Proof of Theorem \ref{main-theorem}}\label{sec:finalsection}
In this section, we define the final approximation function which is that of Theorem \ref{main-theorem} and  sketch how the relevant statements of the theorem follow  from Lemma \ref{lem: induct},  Proposition \ref{prop:tysystemsolution2}, Proposition \ref{prop:tyoscillatory} and the respective estimates Lemma \ref{lem:estimates-inner}, Corollary \ref{cor:estimates-in-y}, Corollary \ref{cor:estimates-in-(t,r)}, as well as the  consistency (`small error') statements in the overlaps, i.e. Lemma \ref{lem:consistency-inner} and Lemma \ref{lem:consistency-self-sim-rem}.\\[8pt]
If $\nu, \alpha_0$ are fixed we select $ 0 < \epsilon_1 \ll1 $ and $ \epsilon_2 \in (0, \tfrac{1}{2})$ such that $ \tfrac{1}{2} - \epsilon_2 \ll1 $. We then choose for given $ N \gg1 $ large subsequently $ N_1, \mathcal{N}, N_2, N_2^{(schr)}, N_2^{(wave)}, N_3 \in \Z_+$ large enough, constructing all approximate solutions in the above Propositions/Lemma and which,  according to Lemma \ref{lem:estimates-inner}, Corollary \ref{cor:estimates-in-y} and  Corollary \ref{cor:estimates-in-(t,r)}, provides us with some (small) $0 < t_0 \leq 1$ such that all the approximate functions  
\begin{align}
&u^{N_1}_{app, \mathcal{I}}(t,R,a),\; \tilde{u}^{N_2}_S(t,y),\; \tilde{u}^{N_3}_R(t, r),\\[3pt]
&n^{N_1}_{app, \mathcal{I}}(t,R,a),\; n_S^{N_2}(t,y),\; n_R^{N_3}(t, r),
\end{align}
exist for $ t\in (0, t_0)$ such that $ (t,r) \in \{  \mathcal{I}, \mathcal{S}, \mathcal{R}\}$ respectively and further all relevant estimates hold. Note in the above list of functions we suppress dependence on $ N_2^{(schr)}, N_2^{(wave)}, \mathcal{N}$.  Note further we choose all $N$ parameters (iteration indices) large enough so that all \emph{error functions}
\[
e^{z, N_1}_{ \mathcal{I}},\; e^{n, N_1}_{ \mathcal{I}}, \; e^{u, N_2}_{S},\; e^{n, N_2}_{S},\; e^{u, N_3}_{R},\; e^{n, N_3}_{R},
\]
as well as the bounds in the consistency Lemmas  \ref{lem:consistency-inner}, \ref{lem:consistency-self-sim-rem} are estimates of order $O(t^{c\cdot N })$ where $ c = c_{\epsilon_1, \epsilon_2, \nu} > 0$ is fixed. Now let us define $ (\psi_*, n_*)$. Therefore we let $ \Theta\in C^{\infty}$ be a suitable cut off, i.e. such that the following works\\[5pt]
$\bullet$\; The  decomposition below provides a partition of unity\\[3pt]
$\bullet$\; The function $\Theta = 1, 0$  is constant in a far inner\slash outer open set such that the derivatives have support on the overlap regions in the decomposition below.\\[4pt]
Then  we define in $ (t,R) = (t, r \lambda(t))$ coordinates
\begin{align}\label{final-def-1}
u_*^N(t,R) : = &\Theta(t^{- \nu - \epsilon_1} R) \cdot u^{N_1}_{app, \mathcal{I}}(t, t^{\nu - \f12}R, R)\\[3pt] \nonumber
& + (1 - \Theta(t^{\nu - \epsilon_1} R) ) \Theta(t^{\epsilon_2 - \nu}R) \cdot u_S^{N_2}(t, t^{\nu}R)\\[3pt] \nonumber
& + (1 - \Theta(t^{\epsilon_2 + \nu} R) ) e^{- i \alpha(t)} t^{\f12 + \nu} \cdot \tilde{u}_R^{N_3}(t, t^{\f12 + \nu}R),\\[5pt] \label{final-def-2}
\tilde{n}_*^N(t,R) : = & \lambda^{-2}(t)\Theta(t^{- \nu - \epsilon_1} R) \cdot n^{N_1}_{app, \mathcal{I}}(t, R, t^{\nu - \f12}R)\\[3pt] \nonumber
& +  \lambda^{-2}(t)(1 - \Theta(t^{\nu - \epsilon_1} R) ) \Theta(t^{\epsilon_2 - \nu}R) \cdot  n_S^{N_2}(t, t^{\nu}R)\\[3pt] \nonumber
& +  \lambda^{-2}(t)(1 - \Theta(t^{\epsilon_2 + \nu} R) ) \cdot   n_R^{N_3}(t, t^{\f12 + \nu}R).
\end{align} 
In particular we then further set
\begin{align}
&\psi_*(t,R) : = e^{i \alpha(t)} \lambda(t) \cdot u_*^N(t,R) = e^{i \alpha(t)} \lambda(t) \cdot \big(W(R) +   g_*^N(t,R)\big),\\[6pt]
&n_*(t,R) : = \lambda^2(t) \cdot \tilde{n}_*^N(t,R) =  \lambda^2(t) \cdot \big(W^2(R) +   h_*^N(t,R)\big),
\end{align} 
where of course
\[
g_*^N(t,R) = u_*^N(t,R) - W(R),\;\;h_*^N(t,R) = \tilde{n}_*^N(t,R) - W^2(R).
\]
Note that we can simply write $ u_*^N(t,r), \tilde{n}_*^N(t,r)$ for $u_*^N(t,t^{- \f12 -\nu} r) , \tilde{n}_*^N(t,t^{- \f12 -\nu} r)   $ and similar for $ g_*, h_*$.\\[4pt]
The estimate for the error functions in  the theorem follow from plugging \eqref{final-def-1} and \eqref{final-def-2} in the error definition and using Lemma \ref{lem:estimates-inner}, Corollary \ref{cor:estimates-in-y}, Corollary \ref{cor:estimates-in-(t,r)} and additionally  Lemma \ref{lem:consistency-inner} , as well as Lemma \ref{lem:consistency-self-sim-rem} in the overlap regions (where we have contributions from the cut off).  In particular this shows estimates of the form (actually better ones including derivatives)
\[
\|  e^N_{\psi_*}(t)\|_{L^2_{R^3dR}} + \|  e^N_{\phi_*}(t)\|_{L^2_{R^3dR}}  \lesssim N^{c_1 N - c_2}
\]
for some fixed constants $c_1, c_2> 0$. Provided we yet again take $N_1, N_2, N_3$ sufficiently large, we obtain the desired bound. Further note from Corollary  \ref{cor:estimates-in-(t,r)} changing back into $(t,r)$ coordinates, we directly obtain convergence to the \emph{radiation fields}
\begin{align*}
& \eta(t,x) : = e^{i \alpha(t)} \lambda(t) \cdot g_*^N(t,t^{-\f12 -\nu}|x|) = \chi_0(x) + o(1),\\
& \chi(t,x) : =  \lambda^2(t) \cdot h_*^N(t,t^{-\f12 -\nu}|x|) = \eta_0(x) + o(1),
\end{align*}
with radial profiles $ \chi_0(x) : = n_0(|x|) ,\; \eta_0(x) : = u_0(|x|) $ and convergence in
at least  $ \dot{H}^1 \cap\dot{H}^2(\R^4)$, respectively $H^2(\R^4)$ for the wave part.  The remaining estimates for $ g_*, h_*$ themselves follow likewise directly from the above Corollaries\slash Lemma and the definition of these functions.
\vspace{1cm}

\bibliographystyle{alpha}

\vspace{1cm}

\end{document}